\documentclass{memo-l}
\usepackage{amsmath, dsfont}
\usepackage{fnpct}
\usepackage[colorlinks]{hyperref}

\usepackage[usenames,dvipsnames,svgnames,table]{xcolor}

\usepackage[utf8]{inputenc}
\usepackage[OT2, T1]{fontenc}
\usepackage{ tipa }
\usepackage{helvet}

\usepackage{color}

\usepackage{mathrsfs}  
\usepackage{stmaryrd}  
\usepackage{amsmath,amsxtra,amsthm,amssymb,xr,enumerate,fullpage}
\usepackage[all]{xy}
\usepackage[bbgreekl]{mathbbol}
\usepackage{bbm}
\usepackage{tikz-cd}
\usepackage{mathtools}
\usepackage{enumitem}
\makeatletter
\newcommand{\mylabel}[2]{#2\def\@currentlabel{#2}\label{#1}}
\makeatother

\usepackage{hyperref}

\DeclareSymbolFont{cyrletters}{OT2}{wncyr}{m}{n}
\DeclareMathSymbol{\Sha}{\mathalpha}{cyrletters}{"58}

\DeclareMathSymbol{\invques}{\mathord}{operators}{`>}
\DeclareUnicodeCharacter{00BF}{\tmquestiondown}
\DeclareRobustCommand{\tmquestiondown}{%
  \ifmmode\invques\else\textquestiondown\fi
}
\usepackage{verbatim}

\numberwithin{equation}{chapter}

\newtheorem{theorem}{Theorem}[chapter]
\newtheorem{lemma}[theorem]{Lemma}
\newtheorem{conj}[theorem]{Conjecture}
\newtheorem{proposition}[theorem]{Proposition}
\newtheorem{corollary}[theorem]{Corollary}
\newtheorem{defn}[theorem]{Definition}
\newtheorem{example}[theorem]{Example}

\newtheorem{remark}[theorem]{Remark}

\newtheorem{question}[theorem]{Question}

\newtheorem{thmx}{Theorem}

\newtheorem{propx}[thmx]{Proposition}

\newtheorem{question*}[thmx]{Question}
\newtheorem{conj*}[thmx]{Conjecture}
\newtheorem{remarkx}[thmx]{Remark}

\newcommand{\be}{\begin{equation}}
\newcommand{\ee}{\end{equation}}

\newcommand{\pr}{\mathrm{pr}}

\newcommand{\Tr}{\operatorname{Tr}}
\newcommand{\Gal}{\operatorname{Gal}}
\newcommand{\Fil}{\operatorname{Fil}}

\newcommand{\DD}{\mathbb{D}}
\newcommand{\DDcc}{\,^{c}\mathbb{D}}

\newcommand{\NN}{\mathbb{N}}
\newcommand{\QQ}{\mathbb{Q}}
\newcommand{\Qp}{\mathbb{Q}_p}
\newcommand{\Zp}{\mathbb{Z}_p}
\newcommand{\ZZ}{\mathbb{Z}}

\newcommand{\FF}{\mathbb{F}}

\newcommand{\fz}{\mathfrak{z}}

\newcommand{\ord}{\mathrm{ord}}

\newcommand{\cL}{\mathcal{L}}

\newcommand{\cO}{\mathcal{O}}
\newcommand{\Iw}{\mathrm{Iw}}

\newcommand{\GL}{\mathrm{GL}}

\newcommand{\cyc}{\textup{cyc}}

\newcommand{\ff}{\mathfrak{f}}

\newcommand{\Hom}{\mathrm{Hom}}
\newcommand{\Sel}{\mathrm{Sel}}

\newcommand{\LL}{\Lambda}

\newcommand{\f}{\textup{\bf f}}

\newcommand{\lra}{\longrightarrow}
\newcommand{\ra}{\lra}
\newcommand{\res}{\textup{res}}

\newcommand{\ur}{\textup{ur}}

\newcommand{\Tw}{\mathrm{Tw}}

\newcommand{\id}{\mathrm{id}}

\newcommand{\p}{\mathfrak{p}}

\newcommand{\m}{\mathfrak{m}}

\newcommand{\cC}{\mathcal{C}}
\newcommand{\Frac}{\mathrm{Frac}}

\newcommand{\cW}{\mathcal{W}}
\newcommand{\cX}{\mathcal{X}}

\newcommand{\cE}{\mathcal{E}}
\newcommand{\cR}{\mathcal{R}}
\newcommand{\EXP}{\mathrm{EXP}}

\newcommand{\bz}{\mathbf{z}}

\newcommand{\Vcc}{{}^{\rm c}V}
\newcommand{\Dcc}{{}^{\rm c}\bD}

\newcommand{\BK}{{\rm BK}}
\newcommand{\tildeV}{\widetilde V}
\newcommand{\tildeVcc}{{}^{\rm c}\widetilde V}
\newcommand{\tildeDcc}{{}^{\rm c}\widetilde\bD}

\newcommand{\DdagrigA}{\mathbf D^\dagger_{\mathrm{rig},A}}
\newcommand{\DdagrigE}{\mathbf D^\dagger_{\mathrm{rig},E}}

\newcommand{\bD}{\mathbf{D}}
\newcommand{\bbD}{\mathbb{D}}
\newcommand{\RG}{\mathbf{R\Gamma}}
\newcommand{\CR}{\mathscr{R}}
\newcommand{\sel}{\mathrm{sel}}
\newcommand{\Spm}{\mathrm{Spm}}
\newcommand{\Dst}{{\mathbf D}_{\mathrm{st}}}
\newcommand{\Dc}{{\mathbf D}_{\mathrm{cris}}}
\newcommand{\DCc}{{\mathscr D}_{\mathrm{cris}}}
\newcommand{\Exp}{\mathrm{Exp}}
\newcommand{\CH}{\mathscr H}
\newcommand{\Log}{\mathfrak{L}}
\newcommand{\bfchi}{\bbchi}
\newcommand{\DdagrigX}{\mathbf{D}^\dagger_{\mathrm{rig},\cX}}
\newcommand{\fD}{\mathfrak D}

\newcommand{\Fr}{\mathrm{Fr}}
\newcommand{\bR}{\mathbf{R}}
\newcommand{\bExp}{\mathbf{Exp}}
\newcommand{\bexp}{\mathbf{exp}}
\newcommand{\rmE}{\mathrm{E}}
\newcommand{\crit}{\mathrm{cr}}
\newcommand{\tf}{\mathrm{tf}}
\newcommand{\Tam}{\mathrm{Tam}}

\definecolor{pinegreen}{rgb}{0.0, 0.47, 0.44}

 \definecolor{pAlgae}{RGB}{87,115,135}
\definecolor{airforceblue}{rgb}{0.36, 0.54, 0.66}
	\definecolor{bondiblue}{rgb}{0.0, 0.58, 0.71}
\definecolor{britishracinggreen}{rgb}{0.0, 0.26, 0.15}
\definecolor{camouflagegreen}{rgb}{0.47, 0.53, 0.42}
\definecolor{darkcyan}{rgb}{0.0, 0.55, 0.55}

\hypersetup{
    colorlinks = true,
    linkcolor=blue,
     filecolor=blue,
     citecolor = darkcyan,      
     urlcolor=cyan,
    linkbordercolor = {white},
}

\author{Denis Benois}
\address{Denis Benois\newline Institut de Math\'ematiques, Universit\'e de Bordeaux  \\ 351, Cours de la Lib\'eration 33405  \\ Talence, France
}
\email{denis.benois@math.u-bordeaux.fr}
\thanks{D.B. was partially supported by the  Agence National de Recherche (grant ANR-18-CE40-0029) in the framework of the ANR-FNR project ``Galois representations, automorphic forms and their $L$-functions''.}

\author{K\^az\i m B\"uy\"ukboduk}
\address{K\^az\i m B\"uy\"ukboduk\newline UCD School of Mathematics and Statistics\\ University College Dublin\\ Ireland}
\email{kazim.buyukboduk@ucd.ie}
\thanks{K.B.’s research conducted in this publication was funded by the Irish Research Council under grant number IRCLA/2023/849.}

\dedicatory{\textit{Dedicated to the memories of Jo\"el Bella\"iche and Jan Nekov\'a\v{r}}.}

\makeindex
\begin{document}
\frontmatter


\title{A\lowercase{rithmetic of} \lowercase{critical $p$-adic} $L$-\lowercase{functions}  }
\vspace{3 cm}

 \begin{abstract}
 Our objective in the present work is to develop a fairly complete arithmetic theory of critical $p$-adic $L$-functions on the eigencurve. To this end, we carry out the following tasks: 
 \begin{itemize}
 \item We give an ``\'etale'' construction of Bella\"iche's $p$-adic $L$-functions at a $\theta$-critical point on the cuspidal Coleman--Mazur--Buzzard eigencurve. 
\item  We introduce the algebraic counterparts of these objects (which arise as appropriately defined Selmer complexes) and develop Iwasawa theory in this context, including a definition of an Iwasawa theoretic $\mathscr L$-invariant $\mathscr{L}^{\rm cr}_\Iw$.
\item We formulate the (punctual) critical main conjectures, and study its relationship with its slope-zero counterpart. Along the way, we also develop descent theory (paralleling Perrin-Riou's work).
\item We introduce what we call \emph{thick} (Iwasawa theoretic) fundamental line and the \emph{thick} Selmer complex to counter Bella\"iche's secondary $p$-adic $L$-functions. This allows us to formulate an infinitesimal thickening of the Iwasawa main conjecture, and we observe that it implies both slope-zero and punctual critical main conjectures, but it seems stronger than both. 
\item We establish an $\cO_\cX$-adic leading term formula for the two-variable $p$-adic $L$-function over the affinoid neighborhood  $\cX={\rm Spm}(\cO_\cX)$ in the eigencurve about a $\theta$-critical point. Using this formula we prove, when the Hecke $L$-function of $f$ vanishes to order one at the central critical point, that the derivative of the secondary $p$-adic $L$-function can be computed in terms of the second order derivative of an $\cO_\cX$-adic regulator (rather than a regulator itself). 
 \end{itemize}
\end{abstract}
\maketitle



\tableofcontents

\chapter*{Preface}
Our objective in the present work is to develop a fairly complete arithmetic theory of critical $p$-adic $L$-functions on the eigencurve. 

As the first task, we give an ``\'etale'' construction of Bella\"iche's $p$-adic $L$-functions at a $\theta$-critical point on the cuspidal Coleman--Mazur--Buzzard eigencurve (corresponding to a critical $p$-stabilization of cuspidal newform $f$). The interpolation properties of our improved $p$-adic $L$-function are deduced from a novel ``eigenspace-transition via differentiation'' principle. 

As the second task, we introduce the algebraic counterparts of these objects (which arise as appropriately defined Selmer complexes) and develop Iwasawa theory in this context. These allow us to interpret the degeneracy properties of the critical $p$-adic $L$-functions as an extreme example of the exceptional zero phenomenon, and lead us naturally to the definition of an Iwasawa theoretic $\mathscr L$-invariant $\mathscr{L}^{\rm cr}_\Iw$, whose construction and properties should be of independent interest. We prove that the cyclotomic improvement of the critical $p$-adic $L$-function coincides with the slope-zero $p$-adic $L$-function, up to multiplication by $\mathscr{L}^{\rm cr}_\Iw$. In particular, the critical $p$-adic $L$-function is nonzero if and only if the Iwasawa theoretic $\mathscr L$-invariant is. It seems to us that there are no results paralleling these in the Betti counterpart of the theory.

Armed with these constructions, we then formulate the punctual critical main conjecture and show that it is equivalent to its slope-zero counterpart, subject to the non-vanishing of our Iwasawa theoretic $\mathscr L$-invariant. We establish an Iwasawa theoretic descent formalism (paralleling Perrin-Riou's theory in the non-critical scenarios), which in turn allow us to establish leading term formulae (within and off the critical range) for the module of critical $p$-adic $L$-functions, involving classical invariants.

To complete the picture, one needs a similar theory for Bella\"iche's secondary $p$-adic $L$-function, since its interpolation properties suggest that it also carries arithmetic information. Our key insight in the treatment of Bella\"iche's secondary $p$-adic $L$-function is that its arithmetic properties should be encoded in what we call \emph{thick} (Iwasawa theoretic) fundamental line and the \emph{thick} Selmer complex, which come attached to the infinitesimal deformation of Deligne's representation along the eigencurve. Based on this insight, we formulate an infinitesimal thickening of the Iwasawa main conjecture, and we observe that it implies both slope-zero and punctual critical main conjectures, but it seems stronger than both. We also develop a descent formalism for the \emph{thick} Iwasawa theoretic Selmer complexes away from the central critical point. The descent at the central critical point when the Hecke $L$-function of $f$ vanishes escapes our methods, for the underlying thick Selmer complex is not semi-simple in these scenarios. 

The discussion concerning the semi-simplicity of the thick Selmer complex raises the question of whether one should expect to interpret the derivative of the secondary $p$-adic $L$-function at the central critical point in terms of $p$-adic height pairing when the relevant $L$-value vanishes. To that end, we establish an $\cO_\cX$-adic leading term formula for the two-variable $p$-adic $L$-function over the affinoid neighborhood  $\cX={\rm Spm}(\cO_\cX)$ in the eigencurve about a $\theta$-critical point. Using this formula we prove, when the Hecke $L$-function of $f$ vanishes to order one at the central critical point, that the derivative of the secondary $p$-adic $L$-function can be computed in terms of the second order derivative of an $\cO_\cX$-adic regulator (rather than a regulator itself). The $\cO_\cX$-adic leading term formula also shows that the derivative of the (punctual) critical $p$-adic $L$-function at the central critical point can be computed in terms of heights on the thick Selmer complexes. It would be interesting to understand the meaning of this formula in terms of the more standard leading term formula involving heights associated to the slope-zero root of the Hecke polynomial of $f$.

\mainmatter

\chapter*{Introduction}

\section{Basic notation} 

\subsection{} 
We  fix forever an embedding 
$\iota_\infty: 
\overline{\QQ}\to \mathbb{C}.$ For any prime number $p,$ we fix an isomorphism 
$\iota_{\infty,p}: \mathbb{C} \xrightarrow{\sim}  \mathbb{C}_p$ and  put 
$\iota_p:=\iota_{\infty,p}\circ \iota_\infty\,:\,\overline{\QQ}\to \mathbb{C}_p.$  Using the  embeddings $\iota_\infty$ and $\iota_p$, we  consider  $\overline{\QQ}$ as a subset of $\mathbb{C}$ and $\mathbb{C}_p.$

\subsection{} 
For any prime $p$, we write  $v_p\,:\,\mathbb{C}_p
\rightarrow \mathbf{R}\cup \{+\infty\}$ for   the $p$-adic valuation on
$\mathbb{C}_p$ normalized by $v_p(p)=1.$ 
If $E$ is an extension of $\Qp$,  we denote by $\cO_E$ its ring of integers
and by $\CH_E$ the algebra 
of formal power series 
with coefficients in $E$ which converge on the $p$-adic open unit disc.
We will often omit $E$ from the notation and write, for example, $\CH$ in place of $\CH_E$. 

\subsection{}
\label{subsec_013_2023_07_05}
Let $\Gamma$ denote the Galois group of $\QQ (\zeta_{p^\infty})/\QQ.$ Then $\Gamma \simeq \Gamma_1\times \Delta,$ where $\Gamma_1=\Gal (\QQ (\zeta_{p^\infty})/\QQ (\zeta_p))$ and $\Delta=\Gal (\QQ (\zeta_p)/\QQ)$. For any $\ZZ_p[[\Gamma]]$-module $M$, we denote as usual by $M^\Gamma$ its $\Gamma$-invariants, and by $M_\Gamma:=M\otimes_{\ZZ_p[[\Gamma]]} \ZZ_p$ its $\Gamma$-coinvariants, where $\ZZ_p$ is endowed with the trivial $\Gamma$-action.

Fix a topological generator $\gamma_1\in \Gamma_1$ and define
$\CH (\Gamma_1):=\{h(\gamma_1-1) \mid h\in \CH \}$ and 
 $\CH (\Gamma):= \Qp [\Delta]\otimes_{\Qp}\CH (\Gamma_1).$ 
The algebra $\CH (\Gamma)$ is  equipped with the canonical
involution $\iota$ such that $\iota (\tau)=\tau^{-1}$
for all  $\tau\in \Gamma$ and the twist maps 
\[
\Tw_j \,:\,\CH (\Gamma) \rightarrow \CH(\Gamma), \qquad \Tw_j(f(\gamma-1))=
f(\chi^{-j}(\gamma)\gamma-1).
\] 
Note that   $\CH (\Gamma)$ is the large Iwasawa algebra introduced 
in \cite{perrinriou94}. It contains the usual Iwasawa algebra
$\LL :=\cO_E[\Delta]\otimes_{\cO_E} \cO_E[[\Gamma_1]]. $

For any affinoid algebra $\cX$ over $E,$  set $\CH_\cX (\Gamma)=
\CH (\Gamma)\widehat\otimes_{E} \cO_\cX.$ Let $X(\Gamma)$ denote the group of finite multiplicative characters of $\Gamma$ with values in $E^*.$

\subsection{}
\label{subsec: determinants} 
For any commutative ring $R,$ we denote by $\mathcal P(R)$  the category of finitely generated projective $R$-modules of rank one. These modules are exactly those that are invertible for the tensor product,  and therefore $\mathcal P(R)$ has the structure of a Picard category. A complex $C^\bullet =(C^i)_{i\in \ZZ}$ of $R$-modules is said to be {\it perfect} if there exists a bounded complex $P^\bullet=(P^i)_{i\in \ZZ}$ of finitely generated projective $R$-modules which is quasi-isomorphic to $C^\bullet$\,.  We denote by $\mathcal D^p(R)$ 
the derived category of perfect complexes of $R$-modules. Recall that the determinant functor of Knudsen--Mumford 
\[
{\det}_R\,:\,\mathcal D^p(R) \lra  \mathcal P(R)
\]
is defined as follows. For any projective  $R$-module $P$ of finite  rank $r,$ we set ${\det}_R(P)=\wedge^r P$ . The module ${\det}_R(P)$ is invertible and we denote by 
${\det}^{-1}_R(P)$ its $R$-dual. For any perfect complex $C^\bullet$,  we fix 
a quasi-isomorphic bounded complex $P^\bullet$ and  put ${\det}_R(C^\bullet)= \underset{i\in \ZZ}\otimes {\det}_R^{(-1)^i} (P^i).$ This functor is unique up to isomorphism 
of functors. We refer the reader to \cite{knudsen-mumford} for further details.

\section{$p$-adic $L$-functions of modular forms} 

\subsection{}
Let $f=\underset {n=1}{\overset{\infty}\sum} a_nq^n \in S_{k}(N,\varepsilon_f)$ be a normalized newform (for all Hecke operators $\{T_\ell\,:\, \ell\nmid N\}$ and $\{U_\ell, \langle \ell\rangle\,:\, \ell \mid N\}$) of weight $k\geq 2$ and level $N.$ 

Fix a prime number  $p\geqslant  3$ 
such that $(p,N)=1$, and denote by $\alpha$  and $\beta$ the roots of the 
polynomial $X^2+a_pX+\varepsilon_f (p) p^{k-1}.$
Let 
\[
f_\alpha, f_\beta \in S_{k}(\Gamma_p,\varepsilon_f), \qquad 
\Gamma_p=\Gamma_1 (N) \cap \Gamma_0(p)
\]
denote the  stabilizations (or refinements) of  $f$ attached to these roots.
We will always assume that $\alpha \neq \beta$ \footnote{This property is conjectured to hold always, and it is proved only for modular forms of weight two. It holds automatically in the critical-slope scenario, which is the main subject of the present paper.}.

\subsection{} We first review the case  $v_p\left(\alpha\right)<k-1$. The constructions that we recall below depend on the choice of a rational basis $\xi:=\{\xi^+,\xi^-\}$ of the Betti realization of the motive $\mathcal M(f)$ associated to $f$ (cf. \S\ref{subsubsec_2311_2023_07_05}). 
Using the theory of modular symbols,  Amice--V\'elu in \cite{AmiceVelu} and Vi\v{s}ik in \cite{Vishik} have given a construction of  $p$-adic distributions
\footnote{The subscripts ``S'' in the notation for distributions and $p$-adic $L$-functions stand for ``Stevens''.}
$\mu^\pm_{\mathrm{S},\xi}$ of growth order $\leqslant v_p(\alpha)$, which are characterized by the property that they interpolate the critical values of Hecke $L$-functions $L (f,\rho,s)$ attached to the twists  of $f$.
The Mellin transforms of $\mu^\pm_{\mathrm{S},\xi}$ give rise to an element
\[
L_{\mathrm{S},\alpha}(f,\xi) \in \CH (\Gamma)
\]
(where we recall that $ \CH (\Gamma)= \CH_E (\Gamma)$, with $E$ being a finite extension of $\Qp$ that contains the Hecke field of $f_\alpha$), which we  call the $p$-adic $L$-function attached to the 
$p$-stabilization $f_\alpha.$  The interpolation property of $L_{\mathrm{S},\alpha}(f,\xi)$
reads
\begin{equation}
\label{intro:classical interpolation property}
L_{\mathrm{S},\alpha} ( f,\xi; \rho \chi^j)
=(j-1)! \cdot e_{p,\alpha}(f,\rho,j) \cdot  \frac{L_\infty (f,\rho^{-1},j)}{(2\pi i)^{j+1-k}\cdot\Omega_f^\pm}
,\qquad 1\leqslant j\leqslant k-1,
\end{equation}
where $\rho \in X(\Gamma)$ and 
\begin{equation}
\nonumber
e_{p,\alpha}(f,\rho,j)=
\begin{cases}
\displaystyle\left (1-\frac{p^{j-1}}{\alpha}\right )\cdot \left (1-\frac{\beta}{p^{j}}\right )
&\textrm{if $\rho=\mathds{1}$},
\\
\displaystyle\frac{p^{nj}}{\alpha^n \tau (\rho^{-1})}
&\textrm{if $\rho$ has conductor $p^n.$}
\end{cases}
\end{equation}
Here $\Omega_f^\pm$ denote the periods attached to  
$\xi^\pm$,  and $\tau (\rho^{-1})$ is the Gauss sum.

\subsection{} Assume now that   $v_p(\alpha)=k-1$, in which case we say that $f_\alpha$ has critical slope. This case can be divided
in two subcases:
\begin{itemize}
\item[]{$\bullet$} Non-$\theta$-critical case. The $p$-stabilization $f_\alpha$ is not in the image of the $p$-adic $\theta$-operator 
$$\theta^{k-1}:=\left (\frac{d}{dq}\right )^{k-1}.$$
\item[]{$\bullet$} $\theta$-critical case.  The $p$-stabilization $f_\alpha$ belongs to the image of  $\theta^{k-1}.$
\end{itemize}

The non-$\theta$-critical case was established by Pollack--Stevens in \cite{PollackStevensJLMS}. It is worthwhile to note that the $p$-adic $L$-functions of Pollack--Stevens are characterized  via the properties of the $f_\alpha$-isotypic Hecke-eigenspace of the space of modular symbols and 
satisfy the classical interpolation property \eqref{intro:classical interpolation property}.

The $\theta$-critical case, which is the main subject of this paper,  is rather different and will be reviewed in \S\ref{sec_2023_01_03_0902}. 

\subsection{}  Let $\cC^{\rm cusp}$ denote the cuspidal Coleman--Mazur--Buzzard eigencurve and let $w:\cC^{\rm cusp}\to \mathcal{WS}$ denote the weight map.
We denote by  $x_0$  the point on the eigencurve that corresponds to the 
$p$-stabilization  $f_\alpha$ of $f$. As a result of Coleman's control theorem~\cite[Theorem 6.1]{Coleman1996}, one can deduce that the eigencurve is \'etale at $x_0$ over the weight space if 
$f_\alpha$ does not have critical slope.
Moreover, the refinement \cite[Theorem 8.1]{PollackStevensJLMS} of Coleman's control theorem shows that the eigencurve is \'etale over the weight space at all classical non-$\theta$-critical points; see~\cite[Lemma 2.8]{bellaiche2012}. 

When $f_\alpha$ is non-$\theta$-critical, a two-variable $p$-adic $L$-function $L_{\mathrm{S},\alpha}(\Phi_{\xi,\cX})\in \CH_\cX (\Gamma)$ was constructed by Stevens (unpublished, see also \cite{Panchishkin2003}). Here $\cX \subset \cC^{\rm cusp}$ is a sufficiently small  affinoid neighborhood of the point $x_0$,
and the notation  $\Phi_{\xi,\cX}$ corresponds to the choice of some  overconvergent modular symbol on $\cX.$ The specialization of 
$L_{\mathrm{S},\alpha}(\Phi_{\xi,\cX}) $   at any classical point in $\cX$ coincides, up to a nonzero constant, with the $p$-adic $L$-function of Amice--V\'elu and  Vi\v{s}ik
of the  modular form $f_x$ attached to $x$,  and this property characterizes $L_{\mathrm{S},\alpha}(\Phi_{\xi,\cX})$.

\subsection{} The  algebraic counterparts of $p$-adic $L$-functions are  the modules of algebraic $p$-adic $L$-functions constructed 
by Perrin-Riou in \cite{perrinriou95}. Namely, let $V_f$ denote the
two-dimensional  Galois representation with coefficients in $E$ attached to 
$f$ (cf. \cite{deligne69}). Consider the $k$th cyclotomic twist $V_f(k)$ of $V_f.$  The restriction of
$V_f(k)$ on the decomposition group at $p$ is crystalline, and the eigenvalues 
of the Frobenius operator $\varphi$ acting on the crystalline  Dieudonn\'e module $\Dc (V_f(k))$ are $\{\alpha/p^k,\beta/p^k\}$. Let 
$D^{(\alpha)}[k]$ and $D^{(\beta)}[k]$ denote  the corresponding $\varphi$-eigenspaces. Using these data, one can define an Iwasawa theoretic Selmer complex $\RG_\Iw \left (V_f(k),\alpha \right )$ together with a canonical trivialization 
\[
i_{\Iw, V_f(k)}^{(\alpha)} \,:\, {\det}^{-1}_{\CH (\Gamma)}\RG_\Iw \left (V_f(k),\alpha \right ) \xhookrightarrow{\quad} 
\CH (\Gamma). 
\]
Let $T_f(k)$ be a Galois stable lattice in  $V_f(k)$ and $N_\alpha[k]$ an $\cO_E$-lattice in $D^{(\alpha)}[k]$. This choice fixes a free $\LL$-module of rank one inside ${\det}^{-1}_{\CH (\Gamma)}\RG_\Iw \left (V_f(k),\alpha  \right )$, called the fundamental line,  and the module of algebraic $p$-adic $L$-functions $\mathbf{L}_{\Iw,\alpha}(T_f(k),N_\alpha[k])$ is defined as the  image of this $\LL$-module in $\CH (\Gamma)$ under the canonical trivialization $i_{\Iw, V_f(k)}^{(\alpha)}.$

Let $f^*=f\otimes \varepsilon_f^{-1}$ denote the dual form. Recall from  \S\ref{subsec_013_2023_07_05} that $\iota$ is the canonical involution on $\mathscr{H}(\Gamma)$. We put $\alpha^*:=p^{k-1}/\beta$ (which is one of the roots of the Hecke polynomial of the dual form $f^*$ at $p$), and let $\xi^*$ denote the basis of the Betti realization of the motive associated to $f^*$ that is dual to $\xi$ (cf. \S\ref{subsubsec_2411_2023_07_05_0731}). Perrin-Riou's formulation of the  Main Conjecture asserts that  
\[
\mathbf{L}_{\Iw,\alpha}\left (T_f(k),N_\alpha[k]\right )         
=
\left (L_{\mathrm{S},\alpha^*}(f^*,\xi^*)^\iota \right )
\]
as $\LL$-modules if the  lattices $T_f(k)$ and $N_\alpha[k]$ are compatible  with the period isomorphisms. We refer the reader to \cite{perrinriou95} and \cite{benoisextracris} for precise statements. Thanks to the fundamental works of Kato, Rubin, and Skinner--Urban,
this conjecture  is proved in a wide variety of cases, cf. \cite{kato04,skinnerurbanmainconj,xinwanwanhilbert}.

\subsection{} In \cite{perrinriou95}, Perrin-Riou built a very general machine that associates a $p$-adic $L$-function to any norm-compatible system of global cohomology classes. In \cite{kato04}, Kato constructed a canonical element  $\bz (f,\xi)$ in the first Iwasawa cohomology of $V_f$ and showed that the corresponding $p$-adic $L$-function\footnote{The subscript ``K'' in the notation of $p$-adic $L$-functions is the initial of ``Kato''.}
$L_{\mathrm{K},\alpha^*}(f^*,\xi^*)$ coincides with $L_{\mathrm{S},\alpha^*}(f^*,\xi^*).$

\subsection{}  Classical formulations of the Main Conjecture (especially if $f$ is ordinary at $p$) are stated in terms of the characteristic ideals of corresponding Selmer groups. These are equivalent to the formulation in the previous paragraph. However, since the formalism of Selmer complexes \cite{nekovar06} plays a crucial role in our study of the $\theta$-critical scenario, we adopt it from the beginning. The main reason is that, in our treatment, we will be naturally led to consider $p$-adic representations over various (simple) artinian rings $A$, such as $\cO_E[X]/(X^2)$.  The Galois cohomology modules over such rings  are not perfect objects in the derived  category of $A$-modules, which is remedied by working with the total derived functors $\RG.$

\section{The $\theta$-critical case}
\label{sec_2023_01_03_0902}
\subsection{} We turn our attention to the $\theta$-critical case. 
It is very different by nature from the classical case,   and we describe 
below some new phenomena which appear in this setting:  

\subsection{The eigencurve at a $\theta$-critical point}
If $f_\alpha$ is $\theta$-critical, then $\cC^{\rm cusp}$ is still smooth
at $x_0$, but it 
is no longer \'etale over $\mathcal{WS}$  \cite[Theorem~2.16 and Proposition~4.6]{bellaiche2012}. Let  $e>1$ denote the ramification 
index of the weight map at $x_0.$ We fix a finite extension $E$ of $\Qp$ that contains the image of the Hecke field of $f_\alpha.$ 
Then there exist affinoid neighborhoods 
$\mathcal{W}={\rm Spm}(\cO_\cW)$ and $\cX={\rm Spm}(\cO_\cX)$ of $k\in \mathcal{WS}(E)$ and $x_0\in \cC^{\rm cusp}(E)$ of the form
\[
\cO_\cW=E\left <Y/p^{2r}\right >, \qquad \qquad 
\cO_\cX=\cO_\cW[X]/(X^e-Y).
\]
Here, we use the standard notation 
$$E\langle Z\rangle:=\left\{\sum_{n\geq 0} c_nZ^n: c_n\in E, \lim_{n\to \infty} c_n =0\right\}$$ to denote the Tate-algebra over $E$ with parameter $Z$. We will adjust our choice of $\mathcal W$ on shrinking it as necessary (which amounts to increasing $r$) for our arguments.

\subsection{Overconvergent modular symbols}
The generalized $f_\alpha$-eigenspace $\mathrm{Symb}^\pm_{\Gamma_p}(\DD^\dagger_{k-2}(E))_{(f_\alpha)}$ the space of overconvergent modular symbols
  has dimension $e$ over $E.$ More precisely, the action of the Atkin--Lehner
  operator $U_p$  equips  $\mathrm{Symb}^\pm_{\Gamma_p}(\DD^\dagger_{k-2}(E))_{(f_\alpha)}$ with the  structure of an  $E[X]$-module, and 
\begin{equation}
\label{intro:structure of modular symbols}
\mathrm{Symb}^\pm_{\Gamma_p}(\DD^\dagger_{k-2}(E))_{(f_\alpha)}
\simeq E[X]/(X^e).
\end{equation}

\subsection{Local $p$-adic representation} The restriction of the representation $V_f$ on the decomposition group at $p$ splits into 
direct sum of two one-dimensional representations, namely
\begin{equation}
\label{intro:decomposition of local representation}
V_f = V_f^{(\alpha)}\oplus V_f^{(\beta)},
\end{equation}
where $V_f^{(\alpha)}$ and $V_f^{(\beta)}$ are crystalline of Hodge--Tate 
weights $1-k$ and $0$ respectively, cf. \cite{BreuilEmerton2010}. 

\subsection{Triangulation}  Shrinking $\cX$ if necessary, we can consider
the canonical (up to an isomorphism)  $p$-adic Galois  representation $V_\cX$ of rank two over $\cO_\cX$ interpolating Deligne's $p$-adic representations  
at classical points. Let $\DdagrigX (V_\cX)$ denote the $(\varphi,\Gamma)$-module attached to $V_\cX$. This module is equipped with 
a canonical submodule $\bD_\cX$ of rank one (triangulation of $\DdagrigX (V_\cX)$) afforded by the refinements of modular forms at classical points.   
As a result of the decomposition \eqref{intro:decomposition of local representation}, the specialization $\bD_{x_0}$ of $\bD_\cX$ at $x_0$
is a non-saturated submodule of $\DdagrigE (V_f)$. More precisely,
\begin{equation}
\label{intro:non saturated triangulation}
\bD_{x_0}=t^{k-1}\DdagrigE (V_f^{(\alpha)}),
\end{equation}
(cf. \cite{KPX2014}, \S6.4), where $t$ is Fontaine's $``2\pi i"$. 

\subsection{} Greenberg conjectured that the restriction of the  Galois representation $V_f$ to the decomposition group at $p$ splits only if 
$f$ has CM, cf. \cite{GhateGreenbergConjwt2}. It is therefore expected that the $\theta$-criticality occurs only in the CM setting.

In the  unpublished note \cite{CMLBellaiche}\footnote{We are grateful to J. Bergdall and R. Pollack for bringing Bella\"iche's note to our attention.}, Bella\"iche explains that a conjecture of Jannsen \cite{Ja89} combined with Greenberg's conjecture  and the decomposition \eqref{intro:decomposition of local representation} would imply $e=2$.

\subsection{Two-variable $p$-adic $L$-function} 
If $x_0$ is $\theta$-critical,  Bella\"iche \cite[\S4.4]{bellaiche2012} constructed
a two-variable $p$-adic  $L$-function  $L_{\mathrm S}( \widetilde \Phi_{\xi, \cX})$,
which interpolates  $p$-adic $L$ functions  of Amice--V\'elu and  Vi\v{s}ik
at classical points $x\in \cX\setminus\{ x_0\}.$ The behavior of $L_{\mathrm S}( \widetilde \Phi_{\xi, \cX})$ at $x_0$ is more complicated (as compared to its counterpart in the non-$\theta$-critical scenarios) and reflects the structure \eqref{intro:structure of modular symbols} of the  module of  overconvergent modular symbols.  Let $X$ denote the local parameter about $x_0\in \cX$, and consider 
$$L_{\mathrm{S},\alpha}^{[i]}(f,\xi):=\dfrac{\partial^{i} }{\partial X^{i}}L_{\mathrm{S}}(\widetilde \Phi_{\xi, \cX})\Big{|}_{X=0}\, \qquad\qquad i=0,\ldots, e-1\,.$$ 

The function $L_{\mathrm{S},\alpha}^{[0]}(f,\xi)$ is simply the specialization 
of $L_{\mathrm S}( \widetilde \Phi_{\xi, \cX})$ at $x_0.$ 
For $1\leqslant i \leqslant e-1$, the functions $L_{\mathrm{S},\alpha}^{[i]}(f,\xi)$ are  Bella\"iche's secondary $p$-adic $L$-functions. In \cite[\S4.4]{bellaiche2012}, Bella\"iche showed that if $i<e-1$, then 
\[
L_{\mathrm{S},\alpha}^{[i]}(f,\xi;\chi^j)=0
\]
 in the critical range $1\leqslant j\leqslant k-1$  and that $L_{\mathrm{S},\alpha}^{[e-1]}(f,\xi)$ verifies the expected interpolation property \eqref{intro:classical interpolation property}. In particular, $L_{\mathrm{S},\alpha}^{[e-1]}(f,\xi)$  is nonzero. 
We also remark that $L_{\mathrm{S},\alpha}^{[0]}(f,\xi)$ decomposes into the product
\[
L_{\mathrm{S},\alpha}^{[0]}(f,\xi)=\left (\underset{i=1-k}{\overset{-1}\prod} \ell_i^\iota
\right ) \cdot 
L_{\mathrm{S},\alpha}^{\textrm{imp}}(f,\xi),\qquad \qquad
\ell_i^\iota=i+\frac{\log (\gamma_1)}{\log\chi (\gamma_1)},
\]
where $L_{\mathrm{S},\alpha}^{\textrm{imp}}(f,\xi)$ is  bounded. We do not know whether or not $L_{\mathrm{S},\alpha}^{[0]}(f,\xi)$ is nonzero; see, however, Remark~\ref{remark_G_intro} for our results in this direction.
 
 We refer the reader to \S\ref{subsect: critical scenario} and \S\ref{subsec: infinitesimal Stevens functions} for a summary of these results.

\subsection{} Bella\"iche's constructions proceed by interpolating Betti cohomology. Our first objective in the present article is to recover his results via interpolation of \'etale cohomology. Throughout this paper, we assume that $e=2$, as this clarifies the exposition greatly (and it is always expected to hold). We shall give an alternative construction of Bella\"iche's $p$-adic $L$-functions in terms of Beilinson--Kato elements and the triangulation over the Coleman--Mazur--Buzzard eigencurve (see \S\ref{chapter:critical L-functions}). One of the consequences of our construction is the leading term formulae we prove (in \S\ref{chapter_critical_Selmer_padic_reg}) for the critical $p$-adic $L$-functions.

Another objective of the present article is to introduce the algebraic counterparts of critical $p$-adic $L$-functions and develop the tools for an Iwasawa theoretic treatment of critical $p$-adic $L$-functions (see \S\ref{chapter_main_conjectures} and \S\ref{chapter_main_conj_infinitesimal_deformation}). 
Our key insight in this treatment is that the Iwasawa theory of the $p$-adic 
$L$-function $L_{\mathrm{S},\alpha}^{[0]}(f,\xi)$ is somehow trivial, and that the Main Conjecture in the $\theta$-critical setting should relate the thick $p$-adic 
$L$-function 
\[
\widetilde L_{\mathrm{S},\alpha^*}(f^*,\xi^*)=L_{\mathrm{S},\alpha^*}^{[0]}(f^*,\xi^*)+XL_{\mathrm{S},\alpha^*}^{[1]}(f^*,\xi^*) 
\]
to Iwasawa theoretic invariants of the infinitesimal deformation 
$V_k=V_{\cX}\otimes_{\cO_\cX}\cO_{\cX}/(X^2)$ of $V_f$ along the eigencurve. 

We have organized our discussion to emphasize the underlying principles in an abstract setting. These suggest that our constructions in the present article can be immensely generalized. For example, we will build on our approach in the present article in future work to construct $p$-adic Rankin--Selberg $L$-functions in the $\theta$-critical case (where there are no prior constructions available). This is crucial if one desires to extend the results of \cite{BPSI} on $p$-adic Gross--Zagier formulae to treat neighborhoods of $\theta$-critical (not only critical slope, as in op. cit.) points. 

Before we explain our results in detail and outline our strategy, let us compare our approach via Beilinson--Kato elements and the triangulation over the eigencurve to the approach via the theory of modular symbols.

\subsubsection{}
Bella\"iche and Stevens construct the $p$-adic $L$-function $L_{\mathrm{S}}(\widetilde\Phi_{\cX,\xi})$ by evaluating a generator of the space of modular symbols (which they prove to be a locally free sheaf of rank one over the eigencurve) at the divisor $\{\infty\}-\{0\}$. One may think of the divisor $\{\infty\}-\{0\}$ as an analogue of the $p$-local Beilinson--Kato element as follows. By \cite[Corollary 4.5]{Emerton2005}, this divisor can be thought of as a functional on the completed cohomology of the classical modular curves. Combined with the local-global compatibility proved in \cite{Emerton2011Preprint} and Colmez $p$-adic local Langlands correspondence (see \cite{Colmez2010Asterisque330}), one may then construct an element in the $p$-local Iwasawa cohomology. One expects\footnote{Shortly after an initial version of this article was circulated as \cite{BB_CK1_A}, Colmez and Wang released a preprint~\cite{ColmezWang} emphasizing this viewpoint.} this class to be closely related to the $p$-local Beilinson--Kato element. 

Pairing  the $p$-local image of the Beilinson--Kato element with the exponential of a suitable differential which arises from the triangulation of the eigencurve about $x_0$, we  recover 
$L_{\mathrm{S}}(\widetilde\Phi_{\cX,\xi})$.
In Bella\"iche's construction,  the non-\'etaleness of the eigencurve at $x_0$ is reflected by the non-semisimplicity (as a Hecke module) of the space of modular symbols; cf. \eqref{intro:structure of modular symbols}. 
In our construction, it  exhibits itself as the failure \eqref{intro:non saturated triangulation} of the global triangulation over the eigencurve to restrict to a saturated triangulation at $x_0$.

 \section{Overview of constructions and results}

\subsection{Set-up} 
\label{subsec_intro_setup_2023_07_11_1413}
Let $x_0\in \cC^{\rm cusp}$ be the point that corresponds to a $\theta$-critical 
stabilization  $f_{\alpha}$ of a form $f$. The weight map $w$ is ramified at $x_0$ and we assume that the ramification degree equals $e=2$. We denote by $\cX$ and $\cW$
sufficiently small affinoid neighborhoods of $x_0$ and $k$ such that  
$\cO_\cX=\cO_\cW[X]/(X^2-Y).$


For any $E$-valued point $x\in \mathcal{X}(E)$ of classical weight $w(x) \in \ZZ_{\geq 2}$, we let $f_x \in S_{w(x)}(\Gamma_p)$ denote the corresponding $p$-stabilized eigenform and $f_x^\circ$ the newform associated to $f_x$.

Let $V_{f_x}$ denote the $f_x$-isotypic Hecke submodule of the \'etale cohomology 
\[
H^1_{\mathrm{\acute{e}t},c}(Y_{\overline{\QQ}}, \mathrm{Sym}^{w(x)-2}(R^1\lambda_*(\QQ_p))
\]
of the relevant open modular curve $Y$ equipped with the universal elliptic curve $\lambda \,:\, \mathcal{E} \rightarrow Y.$ Dually, we let $V_{f_x}'$  denote the $f$-isotypic 
quotient of $H^1_{\mathrm{\acute{e}t}}(Y_{\overline{\QQ}}, \mathrm{Sym}^{w(x)-2}(R^1\lambda_*(\QQ_p)^{\vee}))(1)$
\footnote{These are sometimes called the \emph{cohomological} and \emph{homological} normalizations of Deligne's Galois representation, respectively. We refer the reader to \cite[\S6.4.1]{BB_CK1_PR} for further details.}. Poincar\'e duality induces a canonical  non-degenerate pairing 
\[
V_{f_x}'\otimes_E V_{f_x} \lra E.
\]

There are natural free $\cO_\cX$-modules $V_\cX$ and $V_\mathcal{X}'$ of rank $2$, which are equipped with an $\cO_{\mathcal X}$-linear continuous Galois action such that $V_{\mathcal X}\otimes_{E,x}E\xrightarrow{\sim}V_{f_x}$ and $V_{\mathcal X}'\otimes_{E,x}E\xrightarrow{\sim}V_{f_x}'$; these modules can be realized in the cohomology of overconvergent \'etale sheaves \cite{andreattaiovitastevens2015} (see also \cite[\S\S6.3--6.4]{BB_CK1_PR} for a summary of results used in the present article). The action of Hecke operators on \'etale cohomology is compatible with the 
Poincar\'e duality and induces an $\cO_\cW$-linear  (but not an $\cO_\cX$-linear\footnote{This emphasized property is only pertinent to the $\theta$-critical scenario. Indeed, the weight map is \'etale at non-$\theta$-critical classical points. As a result, locally at a non-$\theta$-critical classical point, the weight map induces an isomorphism $\cO_\cW\xrightarrow{\sim}\cO_\cX$, and the pairing analogous to \eqref{eqn_big_Poincare_duality} can be upgraded to an $\cO_\cX$-linear pairing for free.}) 
pairing
\be
\label{eqn_big_Poincare_duality}
V_{\cX}' \otimes_{\cO_\cW}V_{\cX} \lra \cO_\cW.  
\ee
We refer the reader to Proposition~\ref{prop_big_Poincare_duality} where we discuss various properties of this pairing.

If $S$ is a finite set of prime numbers containing $p,$ we write  $G_{\QQ,S}$
for the Galois group of a maximal algebraic extension of $\QQ$ unramified outside
$S\cup \{\infty\}.$ For any topological $G_{\QQ,S}$-module $M$, we denote by
$H^i_S(M)$ and $H^i(\QQ_\ell, M)$ the continuous Galois cohomology of
$G_{\QQ,S}$ (respectively $G_{\QQ_\ell}$) with coefficients in $M.$
We will write $H^i_\Iw (M)$ and $H^i_\Iw (\QQ_\ell,M)$  for the global (respectively
local)  Iwasawa cohomology.

\subsection{$p$-adic $L$-functions}
\subsubsection{}
The starting point of our \'etale construction of critical $p$-adic $L$-functions is the interpolation of the Beilinson--Kato elements over  $\cX\subset\cC^{\rm cusp}$ to the (partially normalized) $\cO_{\mathcal{X}}$-adic Beilinson--Kato class $\bz(\cX,\xi)\in H^1_\Iw (V^\prime_{\mathcal X})$. See \S\ref{subsubsec_2221_18_11_2021} for a summary of the basic properties of this class. The construction of $\bz(\cX,\xi)$ is one of the main results of the prequel \cite[Theorem A]{BB_CK1_PR} to the present article, where it is denoted by ${\mathbb{BK}}^{[\mathcal X]}_{N} (j, \xi)$.

\begin{remarkx}
We briefly comment on the construction of the ``big'' Beilinson--Kato class $\bz(\cX,\xi)$ in \cite{BB_CK1_PR} and its comparison with related literature.

 The construction of this class in the context of slope-zero families is the subject of Ochiai's work~ \cite{ochiai_big_ES}. Hansen~\cite{HansenBKEigencurve} and Wang~\cite{Wang2021}, using different methods, give a construction of such a class under the condition that $\mathcal{X}$ is \'etale over $\cW$.  The approach in \cite{BB_CK1_PR}  (which, crucially for the present work, applies also when $\mathcal{X}$ is smooth but \emph{not} \'etale over $\cW$) is a synthesis of the techniques of \cite{bellaiche2012, LZ1, HansenBKEigencurve}, and relies on the overconvergent \'etale cohomology of Andreatta--Iovita--Stevens~\cite{andreattaiovitastevens2015} and Kings' theory of $\LL$-adic sheaves developed in \cite{Kings2015}. This approach requires checking (which we do, building on the results of Ash--Stevens and Pollack--Stevens) that the construction of the local pieces of the cuspidal eigencurve using different types of distribution spaces over the weight space, or compactly versus non-compactly supported cohomology produce the same end result. The interested reader may refer to \cite[\S5.3, \S5.1.5, \S6.3]{BB_CK1_PR} for details on this technical aspect.  In a different direction, Colmez and Wang interpolate Beilinson--Kato elements over the universal deformation space ~\cite{ColmezWang}.
\end{remarkx}

Fix a generator $\eta \in \DCc (\bD_\cX)$ such that its specialization 
at $x_0$ coincides with the canonical generator of $\Dc (V_f^{(\alpha)})$ (cf. \S\ref{subsec_slope_zero_padic_L}).  Our $p$-adic $L$-function $L_{\mathrm{K},\eta}(\cX,\xi )\in \mathscr{H}_\cX(\Gamma)$ is obtained on pairing  
$\bz(\cX,\xi)$ with the image of the element
\[
\bbeta =1\otimes (X\eta) + X\otimes\eta \in 
\cO_\cX \otimes_{\cO_\cW}
\DCc (\bD_\cX)
\]
under the large (Perrin-Riou) exponential map $\Exp_{\bD_\cX}$ associated 
to $\bD_\cX$  (cf. Definition~\ref{defn_fat_eta_etatilde} and \ref{def_two_var_padicL_function}).
 We refer the reader to \S\ref{subsec: large exp} for a very general (abstract) discussion of large exponential maps, and to \S\ref{subsec:triangulation} and \S\ref{subsubsec_221_04042022} for an overview of the Kedlaya--Pottharst--Xiao triangulation.  

The interpolation properties of $L_{\mathrm{K},\eta}(\cX,\xi )$ (and its improvements) are the content of Theorem~\ref{thmA} below. In analogy 
with Bella\"{\i}che's construction, we let 
$L^{[0]}_{\mathrm{K},\alpha}(f,\xi):=L_{\mathrm{K},\eta}(\cX, \xi; x_0)$
denote the specialization of $L_{\mathrm{K},\eta}(\cX,\xi )$ to $x_0$, and put
$$L_{\mathrm{K},\alpha}^{[1]}(f,\xi):=\displaystyle \left. \left (\frac{\partial}{\partial X}L_{\mathrm{K},\eta} (\cX,\xi)\right ) \right\vert_{X=0}.$$

We let $L^\pm_{\mathrm{K},\eta}(\cX,\xi )\in  \mathscr{H}_\cX(\Gamma)^\pm$ denote the $\pm1$-eigencomponent of $L_{\mathrm{K},\eta}(\cX,\xi )$ for the action of complex conjugation in $\Gamma$. We similarly define $L^{[0],\pm}_{\mathrm{K},\alpha}(f,\xi)$ and $L_{\mathrm{K},\alpha}^{[1],\pm}(f,\xi)$. For any continuous character
$\phi\,:\,\Gamma \rightarrow \cO_E^*$, we put $L(f,\xi;\phi):=\phi(L(f,\xi))$ where $L=L^{[0]}_{\mathrm{K},\alpha}, L^{[0],\pm}_{\mathrm{K},\alpha}, L^{[1],\pm}_{\mathrm{K},\alpha}$, etc.



\begin{thmx}[Theorem~\ref{thm_interpolative_properties}(i) and Equation \eqref{eqn:comparision with Manin-Vishik}]
\label{thmA} Assume that $e=2$. 
\item[i)] At any classical point $x\in \cX (E)\setminus\{x_0\}$,
the specialization $L_{\mathrm{K},\eta}(\cX,\xi; x)$ of $L_{\mathrm{K},\eta}(\cX,\xi)$ to $x$ agrees with the Manin--Vishik $p$-adic $L$-function 
$L_{\mathrm{S},\alpha(x)}( f_x^\circ,\xi_x)$ in the following sense: there exist
some nonzero constants $A_x^\pm \in E$, which depend on the choice of $\eta$, such that 
$$L_{\mathrm{K},\eta}^\pm(\cX,\xi; x)= A_x^\pm\,\cE_N(x)\,L_{\mathrm{S},\alpha (x) }^{\pm}( f_x^\circ,\xi_x),$$
where $\cE_N(x)$ is the product of Euler-like factors at bad primes defined
in \eqref{eqn:interpolation factor E}.

\item[ii)]  There exists a bounded $p$-adic $L$-function $L_{\mathrm{K},\alpha}^{[0],\mathrm{imp}}(f,\xi) \in \Lambda_E$ such that 
\begin{equation}
\nonumber
\begin{aligned}
&L_{\mathrm{K},\alpha}^{[0]}(f, \xi)= 
\left (\underset{i=1-k}{\overset{-1}\prod} \ell^\iota_i\right ) 
L_{\mathrm{K},\alpha}^{[0],\mathrm{imp}}(f,\xi)\,.
\end{aligned}
\end{equation}
In particular,  $L_{\mathrm{K},\alpha}^{[0]}(f, \xi; \rho\chi^j)=0$ for all integers $1\leqslant j\leqslant k-1$ and
 characters $\rho \in X(\Gamma)$ of finite order.
  \end{thmx}
  
Note that the factor $\cE_N(x)$  appears already in the interpolation of Beilinson--Kato elements; cf. the main results of \cite{BB_CK1_PR}.

We call $L_{\mathrm{K},\alpha}^{[0],\mathrm{imp}}(f,\xi)$ that appear in the statement of Theorem~\ref{thmA} the cyclotomic improvement of the critical $p$-adic $L$-function $L_{\mathrm{K},\alpha}^{[0]}(f,\xi)$.

\subsubsection{} We may in fact show that our two-variable function $L_{\mathrm{K},\eta}(\cX,\xi)$ agrees with Bella\"{\i}che's $p$-adic $L$-function. 
The next theorem  relies crucially on the validity of condition \eqref{item_C4}, which we formulate in Section~\ref{subsect:deformation}.  Roughly speaking, it  amounts to the requirement that the induced triangulation of the infinitesimal deformation is in general position with respect
to the Hodge--Tate filtration.

\begin{thmx}[Theorem~\ref{thm_interpolative_properties}(ii) and Theorem~\ref{thm:comparision with Bellaiche's construction}] 
\label{thm_comparision_with_Bellaiche_intro}
Assume that $e=2$ and condition \eqref{item_C4} holds true. 
Then 

\item[i)]  The secondary $p$-adic $L$-functions $L_{\mathrm{K},\alpha}^{[1],\pm}(f,\xi) $ verify the following interpolation property: 

For every character $\rho\in X(\Gamma)$ of finite order and integer $1\leqslant j 
\leqslant k-1$ we have
$$
L_{\mathrm{K},\alpha}^{[1],\pm}(f,\xi;\rho\chi^j)=C_{\mathrm K}\cdot 
{(j-1)!}\, e_{p,\alpha}(f,\rho,j)\,\mathcal{E}_N(f {;\rho\chi^j})\,\frac{L_\infty (f,\rho^{-1},j)}{(2\pi i)^{j+1-k}\Omega_{f}^\pm}\,, \qquad\qquad \textrm{\,if\, $\rho(-1)=\mp(-1)^{j}\,,$}
$$
where $e_{p,\alpha}(f,\rho,j)$ is given as in \eqref{eqn: interpolation factor e},  $C_{\mathrm{K}}$ as in \eqref{eqn: the constant a}, and ${\mathcal E}_N(f ;\rho\chi^j)$ is the value of ${\mathcal E}_N(f)$  at $\rho\chi^j$; cf. Equation \eqref{eqn: evaluation of euler-like factor EN}.

\item[ii)] There exist  functions $u^\pm\in \cO_{\mathcal X}^\times$
such that 
\[
L_{\mathrm{K},\eta}^{\pm}(\cX,\xi)=u^\pm\,{\mathcal E}_N \, L_{\mathrm{S}}^\pm(\widetilde\Phi^{\pm}_{\xi,\cX})\,.
\]
\end{thmx}

Part i) of Theorem~\ref{thm_comparision_with_Bellaiche_intro} is deduced from very general (abstract) results we obtain in \S\ref{sec_abstract_setting}; see in particular Proposition~\ref{prop_bellaiche_formal_step_1} and Theorem~\ref{prop_imoroved_padicL_vs_slope_zero_padic_L}. The crux of the proof of the latter is the ``eigenspace-transition via differentiation'' principle we establish in Proposition~\ref{prop: comparison of exponentials for different eigenvalues}, which we believe is of independent interest. Part ii) follows from  Part i) and   Theorem~\ref{thmA} via a continuity argument.

The proof of Theorem~\ref{thm_comparision_with_Bellaiche_intro} shows that, even without assuming \eqref{item_C4}, we have (shrinking $\cX$ as necessary)
\begin{equation}
\label{eqn_2023_01_04_1121}
L_{\mathrm{K},\eta}^{\pm}(\cX,\xi)=u^\pm\,{\mathcal E}_N \, L_{\mathrm{S}}^\pm(\widetilde\Phi^{\pm}_{\xi,\cX})
\end{equation}
for some $u^\pm \in \cO_{\cX}$.  If the condition \eqref{item_C4} does not hold, the proof of Theorem~\ref{thm_comparision_with_Bellaiche_intro} also shows that $L_{\mathrm{K},\alpha}^{[1]}(f,\xi;\rho\chi^{j})=0$ for all $\rho$ and $j$ as above. In fact, we have the following implications:
\begin{equation}
\label{eqn_2023_01_04_1109}
\textrm{\eqref{item_C4} fails}
\quad \iff \qquad \parbox[t]{4cm}{$L_{\mathrm{K},\alpha}^{[1],\pm}(f,\xi,\rho\chi^j)=0$\\on the central critical strip}
\quad \iff \qquad  u^\pm(x_0)=0
\quad 
\implies{} 
\quad L_{\mathrm{K},\alpha}^{[0],\pm}(f,\xi)=0.
\end{equation}
This is the motivation for our:
\begin{conj*}[see Conjecture~\ref{conjecture_GP}]
\label{conj: GP intro}  Condition \eqref{item_C4}
holds true for the infinitesimal deformation provided by the eigencurve
at a $\theta$-critical point. 
\end{conj*}

It follows from \eqref{eqn_2023_01_04_1109} that 
\begin{equation}
\label{eqn: intro non-vanishing of L^0 implies GP} 
L_{\mathrm{K},\alpha}^{[0],\pm}(f,\xi)\neq 0 \quad \implies \quad \mathrm{Conjecture~\ref{conj: GP intro}}. 
\end{equation}
We refer the reader to Remark~\ref{remark_G_intro}(iv) where we explain a proof of Conjecture~\ref{conj: GP intro} in many cases of interest. See also Remark~\ref{remark_22_2022_08_17_1607} and \S\ref{subsec_245_2022_05_11_0845_subsubsec1} for an elaboration on the hypothesis \eqref{item_C4}. 

\subsubsection{} Let $L_{\mathrm{K},\alpha}(f,\xi)$ denote the $p$-adic $L$-function associated to Kato's zeta element via Perrin-Riou's construction. It follows immediately from definitions that it coincides with $L_{\mathrm{K},\alpha}^{[0]}(f,\xi)$
up to the factor $\mathcal E_N(f).$ Therefore,  $L_{\mathrm{K},\alpha}(f,\xi)$
and $L^{[0]}_{\mathrm{S},\alpha}(f,\xi)$ are a pair of $p$-adic $L$-functions canonically attached 
to $f$ and $\xi$, and it follows from  \eqref{eqn_2023_01_04_1121} that 
\[
L_{\mathrm{K},\alpha}^{\pm} (f,\xi)=\lambda^\pm (f) L_{\mathrm{S},\alpha}^{[0],\pm}(f,\xi)
\]
for some canonically determined constants $u^\pm(x_0)=:\lambda^\pm (f)\in E$. According to \eqref{eqn_2023_01_04_1109}, the constants $\lambda^\pm (f)$ are nonzero if and only if the property \eqref{item_C4} holds true. 

\begin{question*} Can one compute the fudge factors $\lambda^\pm (f)$ explicitly?
\end{question*}

This question inquires about an explicit comparison of two natural constructions of $p$-adic $L$-functions, extending that in the non-$\theta$-critical scenario (cf. Proposition~\ref{prop_2_17_2022_05_11_0842} below, where both constructions yield the same $p$-adic $L$-function on the nose). The constants $\lambda^\pm (f)$ appear in the formulations of our  Main Conjectures for $\theta$-critical  forms (Conjectures~\ref{conj:punctual MC intro}  and \ref{conj: thick MC intro} below). The fact that the non-vanishing of $\lambda^\pm (f)$ is equivalent to a seemingly non-trivial property \eqref{item_C4} suggests that this is a subtle problem.

We refer the reader to \S\ref{remark_G_intro} below where we observe that the non-vanishing of $\lambda^\pm (f)$ can be explained, in the scenario when $f$ has CM, by the non-vanishing of the non-critical values of Katz' $p$-adic $L$-function.

\subsection{Improved $p$-adic $L$-function and punctual Main Conjectures}

\subsubsection{Critical $\mathscr L$-invariants}
\label{subsubsec_0431_2023_07_10_1201}
In Chapter~\ref{chapter_main_conjectures}, we address the arithmetic properties of our critical $p$-adic $L$-function $L_{\mathrm{K},\alpha}^{[0]}(f,\xi)$ as well as its cyclotomic improvement $L_{\mathrm{K},\alpha}^{[0],\mathrm{imp}}(f,\xi)$, much in the spirit of Perrin-Riou's theory. It follows from Kato's fundamental results \cite[Theorems~12.4 and 12.5]{kato04}  that for any $j\in \ZZ$, the first Iwasawa cohomology $H^1_\Iw (V_f(j))$ is a free $\LL_E$-module of rank one and   the localization map  
\begin{equation}
\label{intro:iwasawa theoretic restriction}
H^1_\Iw (V_f(j)) \lra  H^1_\Iw (\Qp, V_f^{(\alpha)}(j)) \oplus H^1_\Iw (\Qp, V_f^{(\beta)}(j))
\end{equation}
is injective (since   the image of the Beilinson--Kato element under this map is non-zero
by Kato's explicit reciprocity law). Here, $H^1_\Iw (\Qp, V_f^{(\alpha)}(j))$ and $H^1_\Iw (\Qp, V_f^{(\beta)}(j))$ are both free $\LL_E$-modules of rank one, and come equipped with canonical bases. This allows us to define the slope $\mathscr L_\Iw^{\mathrm{cr}}(V_f(j))\in 
\mathrm{Frac}(\LL_E)$
of the image of the map \eqref{intro:iwasawa theoretic restriction} with respect to these canonical bases, which we call the Iwasawa theoretic $\mathscr L$-invariant (see \S\ref{subsec:critical L-invariants}). Note that  $\mathscr L_\Iw^{\mathrm{cr}}(V_f(j))$ for different $j$ are related to one another via the twisting map in the Iwasawa algebra. Namely, we have 
$$\mathscr L_\Iw^{\mathrm{cr}}(V_f(j))=\Tw_{j} \left ( \mathscr L_\Iw^{\mathrm{cr}}(V_f)\right )\,.$$

The roots of the Hecke polynomial of the dual modular form $f^*$ are
$\alpha^*=p^{k-1}/\beta$ and $\beta^*=p^{k-1}/\alpha.$ 
In particular, the $\beta^*$-stabilization of $f^*$ is $p$-ordinary, and we can consider  the corresponding (ordinary)
$p$-adic $L$-function $L_{\mathrm{K},\beta^*}(f^*,\xi^*)$ (which coincides with
$L_{\mathrm{S},\beta^*}(f^*,\xi^*)$). 
The  next proposition  follows immediately from the definitions:

\begin{propx}[Proposition~\ref{prop: comparision p-adic L-functions for alpha and beta}] 
\label{prop_comparison_alpha_beta_padic_L_intro} Assume that $e=2$. 
Then,
\[
L_{\mathrm{K},\alpha^*}^{[0],\mathrm{imp}}(f^*,\xi^*)= (-1)^{k-1}\,\mathcal{E}_N(f^*) \,\mathscr L_{\Iw}^{\rm cr}(V_{f}(k)) \,  L_{\mathrm{K},\beta^*}(f^*, \xi^*)\,.
\]
\end{propx}

It seems that there is no direct way to prove the analogue of the comparison in Proposition~\ref{prop_comparison_alpha_beta_padic_L_intro} for the $p$-adic $L$-functions that are constructed by interpolating Betti cohomology (namely, without appealing to Theorem~\ref{thm_comparision_with_Bellaiche_intro}). This serves as an independent justification to pursue a Perrin-Riou style construction of critical $p$-adic $L$-functions. 
\begin{conj*}
\label{conj: L-inv intro} The Iwasawa theoretic $\mathscr L$-invariant $\mathscr L_\Iw^{\mathrm{cr}}(V_f)$ is a nonzero function.
\end{conj*} 
{We stress once again that it is a priori unclear that the $p$-adic $L$-function $L_{\mathrm{K},\alpha^*}^{[0]}(f^*,\xi^*)$, or Bella\"iche's $p$-adic $L$-function $L_{\mathrm{S},\alpha^*}^{[0]}(f^*,\xi^*)$, is nonzero (see also \cite[Remark 3.5]{LLZ_critical} for a relevant discussion). 
It follows from Proposition~\ref{prop_comparison_alpha_beta_padic_L_intro} that 
the non-vanishing of $L_{\mathrm{K},\alpha^*}^{[0]}(f^*,\xi^*)$
 is equivalent to the non-vanishing of  $\mathscr L_\Iw^{\mathrm{cr}}(V_f)$. We then infer from 
\eqref{eqn: intro non-vanishing of L^0 implies GP} that 
Conjecture~\ref{conj: L-inv intro} implies Conjecture~\ref{conj: GP intro}.

We refer the reader to \S\ref{subsubsec_0433_20223_07_10_1242} (which is a summary of our results in \S\ref{subsec:critical L-invariants}) for a detailed analysis of $\mathscr L_\Iw^{\mathrm{cr}}(V_f)$. In light of the discussion therein, we remark that proving its non-vanishing even at $\chi^{-j}$ for even a single integer $j$, let alone its explicit computation, seems to be a difficult problem. For those integers $j$ that are outside the critical range (e.g. $j\geq k$), Proposition~\ref{prop: regulators and L-invariant} expresses its value at $\chi^{-j}$ as the ratio of two $p$-adic regulators. When $f$ has CM, the non-vanishing of $\mathscr L_\Iw^{\mathrm{cr}}(V_f)$ at $\chi^{-j}$ can be related to the non-vanishing a value of Katz $p$-adic $L$-function outside its range of interpolation, cf. \eqref{item_Linv3} below.

\subsubsection{}
\label{remark_G_intro}
Throughout \S\ref{remark_G_intro}, we assume that $f$ has CM by an imaginary quadratic field $K$ where $p=\p\p^c$ necessarily splits. 

In \cite{CMLBellaiche}, Bella\"{\i}che related $L^{[0]}_{\mathrm{S},\alpha^*}(f^*,\xi^*)$ to the restriction of the Katz $p$-adic $L$-function to a suitable ``line'' of characters. In  \cite[Theorem~3.2]{LLZ_critical}, Lei--Loeffler--Zerbes proved  an analogue of this result for $L_{\mathrm{K},\alpha^*}^{[0]}(f^*,\xi^*)$. As a result, we infer that when $f$ has CM,  $L_{\mathrm{K},\alpha^*}^{[0]}(f^*,\xi^*)$  and $L^{[0]}_{\mathrm{S},\alpha^*}(f^*,\xi^*)$ coincide up to multiplication by a nonzero constant (cf. \cite{LLZ_critical}, Theorem~3.4).

 It follows from the discussion in the preceding paragraph that $L_{\mathrm{S},\alpha^*}^{[0]}(f^*,\xi^*)$  and $L_{\mathrm{K},\alpha^*}^{[0]}(f^*,\xi^*)$ are nonzero if and only if the restriction of a Katz $p$-adic $L$-function to a suitable range of characters (more precisely, to a range of characters whose Hodge--Tate weights at $\p$ and $\p^c$ are $1-k+j$ and $j$, respectively) is. Since the range of interpolation for the Katz $p$-adic $L$-function contains no such characters, the question as to whether $L_{\mathrm{S},\alpha^*}^{[0]}(f^*,\xi^*)$ and $L_{\mathrm{K},\alpha^*}^{[0]}(f^*,\xi^*)$ are nonzero remains open in general.

Our results show that $L_{\mathrm{K},\alpha^*}^{[0]}(f^*,\xi^*)= \lambda^\pm (f^*)\, \mathcal{E}_N(f^*) \,L^{[0]}_{\mathrm{S},\alpha^*}(f^*,\xi^*)$, or equivalently $L_{\mathrm{K},\alpha^*}(f^*,\xi^*)= \lambda^\pm (f^*) \,L^{[0]}_{\mathrm{S},\alpha^*}(f^*,\xi^*)$, without assuming $f$ is a CM form. If we assume in addition that $f$ is a CM form, we explain in \S\ref{subsubsec_3227_2022_04_29} (see also \eqref{item_Linv3} below for a summary) that $\mathscr L_\Iw^{\mathrm{cr}}(V_f)$ is nonzero if and only if the said restriction of the Katz $p$-adic $L$-function is nonzero, exploiting Rubin's proof of Iwasawa main conjectures over imaginary quadratic fields. Combined with the result of Bella\"{\i}che, this gives a new proof of the results of \cite{LLZ_critical}, exhibiting a consistency between the two approaches.

Let us assume in this paragraph that $f=f_A$ is the newform attached to an elliptic curve $A/\QQ$ (that necessarily has CM) with ${\rm ord}_{s=1}L(A,s)=1$. In this scenario, Rubin in \cite{Rubin1992PRConj, Rubin_RubinsFormula}
gives a construction of a rational point on $A$ and computed its $p$-adic height in terms of the value $L_p^{\rm Katz}({\psi}_{\p}\chi)$ of the Katz $p$-adic $L$-function at a character $\psi_\p\chi$ outside its range of interpolation; see \S\ref{subsubsec_3227_2022_04_29} for further details on the relevant Katz $p$-adic $L$-function and the description of the character ${\psi}_\p$. We briefly note that the character ${\psi}_\p\chi$ (resp. ${\psi}_\p^c\chi$) is the unique direct summand of $V_p(A)_{\vert_{G_K}}\simeq V_f(1)_{\vert_{G_K}}$ that has Hodge--Tate weights $(0,1)$ (resp., weights $(1,0)$) at $(\p,\p^c)$. As a result, ${\psi}_\p\chi$ coincides with the $p$-adic character that Rubin in \cite{Rubin1992PRConj} denotes by $\psi^*$.

In  Corollary~\ref{cor_Iw_crit_L_inv_non_trivial_analytic_rank_one} below, we show  unconditionally that both $L_{\mathrm{K},\alpha^*}^{[0]}(f^*,\xi^*)$  and $L_{\mathrm{S},\alpha^*}^{[0]}(f^*,\xi^*)$ are nonzero in the setting of the previous paragraph. 
This provides us with the first example where Conjecture~\ref{conj: L-inv intro} holds.   As a result, the property \eqref{item_C4} also holds true in this particular scenario.

\subsubsection{}
\label{remark_darmon_PR_Rubin}
Let us continue to assume that $f=f_A$ is the newform attached to a CM elliptic curve $A/\QQ$ with ${\rm ord}_{s=1}L(A,s)=1$. We briefly explain\footnote{We are grateful to Henri Darmon for encouraging us to make a note of this point of view.}\,, with the perspective offered by the present work, how the results of Perrin-Riou in \cite{PR93RubinsFormula} can be linked with Rubin's work~\cite{Rubin1992PRConj}, assuming that the Iwasawa theoretic $\mathscr{L}$-invariant $\mathscr{L}_{\rm Iw}^{\rm cr}$ has at most a simple pole\footnote{We prove that this is indeed the case in the present set-up \emph{starting off} with Rubin's result, cf. \eqref{item_Linv3} below.} at $\chi^{-1}$.

We have
\begin{align}
    \label{eqn_PR_implies_Rubin}
    \begin{aligned}
    L_p^{\rm Katz}({\psi}_\p\chi)&\,\dot{=}\, {\rm Res}_{s=0}\, \chi^{-1}\langle\chi\rangle^{-s}(\mathscr{L}_{\rm Iw}^{\rm cr}) \cdot L'_{\mathrm{K},\beta^*}(f^*,\xi^*,\chi)\\
    &\,\dot{=}\, {\rm Res}_{s=0}\,\chi^{-1}\langle\chi\rangle^{-s}(\mathscr{L}_{\rm Iw}^{\rm cr})\cdot \log_{\omega_f} (\res_p(\bz_0 (f^*,\xi^*)))\,,
\end{aligned}
\end{align}
where ``$\dot{=}"$ means equality up to explicit algebraic factors that we will not specify here, $\langle\chi\rangle$ is the projection of the cyclotomic character to $1+p\Zp$, the cohomology class $\bz_0 (f^*,\xi^*)\in H^1_{\rm f}(V_f'(1))$ is the image of the Beilinson--Kato element $\bz (f^*,\xi^*)\in H^1_{\rm Iw}(V_f')$ under the natural projection map, $\log_{\omega_f}$ is the Bloch--Kato logarithm, and the first equality follows from Proposition~\ref{prop_comparison_alpha_beta_padic_L_intro} combined with \cite[Theorem 3.4]{LLZ_critical} and the functional equation for Katz $p$-adic $L$-function, whereas the second from \cite[Proposition 2.2.2]{PR93RubinsFormula}. In view of \cite[\S13]{kato04} where Kato explains the relation between Beilinson--Kato elements and elliptic units, Equation \eqref{eqn_PR_implies_Rubin} is a variant of \cite[Corollary 10.3]{Rubin1992PRConj}. 


\subsubsection{} \label{subsubsec_0433_20223_07_10_1242} The discussion in \S\ref{subsubsec_0431_2023_07_10_1201} and \S\ref{remark_G_intro} shows that a fine understanding of the critical $p$-adic $L$-functions can be useful in the study of the $p$-adic Hodge theoretic properties of the eigencurve. This, among others to be recorded in \S\ref{subsubsec_summary_intro_Chapter_3} below, is one of the motivations for our detailed study in \S\ref{subsec:critical L-invariants} of the Iwasawa theoretic $\mathscr L$-invariant $\mathscr L_{\Iw}^{\rm cr}(V_{f})$. 

Based on our extensive analysis in \S\ref{subsec:critical L-invariants}, we propose the following refinement of Conjecture~\ref{conj: L-inv intro}:

\begin{conj*}
\label{conj: refined L-inv intro} The Iwasawa theoretic $\mathscr L$-invariant  $\mathscr L_\Iw^{\mathrm{cr}}(V_f)$ has neither a pole nor a zero at 
$\chi^{-j}$ if $j$ is an integer other than $\frac{k}{2}$, or else if $j=\frac{k}{2}$ and $L(f^*,\frac{k}{2})\neq 0.$
\end{conj*}

We close this subsection with a discussion on Conjecture~\ref{conj: refined L-inv intro}.

\begin{itemize}
\item[\mylabel{item_Linv1}{ $\mathscr L1$})] 
Let us consider the localization map
\begin{equation}
\label{intro:localization map}
H^1_{\{p\}}(V_f(j)) \lra H^1 (\Qp, V_f^{(\alpha)}(j))\oplus H^1 (\Qp, V_f^{(\beta)}(j)),
\end{equation}
where $H^1_{\{p\}}(V_f(j))$ denotes the Selmer group  of $V_f(j)$ with the relaxed
conditions at $p$. If the projection of $H^1_{\{p\}}(V_f(j))$ on $H^1 (\Qp, V_f^{(\alpha)}(j))$ is an isomorphism,
we define the critical $\mathscr L$-invariant 
$\mathscr L^{\mathrm{cr}}(V_f(j))$ as the slope of the map \eqref{intro:localization map}  with respect to the canonical bases of $H^1 (\Qp, V_f^{(\alpha)}(j))$ and $ H^1 (\Qp, V_f^{(\beta)}(j)).$ Then the  value of $\mathscr L_{\Iw}^{\mathrm{cr}}(V_f)$ at $\chi^{-j}$ coincides with $\mathscr L^{\mathrm{cr}}(V_f(j))$ up to an 
explicit Euler-like factor. Therefore we can formulate  Conjecture~\ref{conj: refined L-inv intro} in the following equivalent form: 
\medskip

The $\mathscr{L}$-invariant $\mathscr L^{\mathrm{cr}}(V_f(j))$  is well-defined and nonzero
if  $j$ is an integer other than $\frac{k}{2}$, or else if $j=\frac{k}{2}$ and $L(f^*,\frac{k}{2})\neq 0$.

\end{itemize}

\begin{itemize}
\item[\mylabel{item_Linv2}{ $\mathscr L2$})] 
It follows from Kato's results that $\mathscr L^{\mathrm{cr}}_\Iw(V_f)$ has no pole at $\chi^{-j}$   if  $1\leq j\leq k-1$ (where we assume that $L(f^*,\frac{k}{2})\neq 0$ if $j=\frac{k}{2}$). By analyticity, we deduce that this holds true for almost all $j$. Moreover,  for any integer $j\geq k$, the 
$\mathscr{L}$-invariant  $\mathscr L^{\mathrm{cr}}(V_f(j))$ is well-defined provided that the $p$-adic Beilinson conjecture \eqref{item_pB} for the slope-zero $p$-adic $L$-function for the dual form $f^*$ holds true at $j$. This is the content of Sections~\ref{subsec: analysis of selmer groups} and \ref{subsec:critical L-invariants}. 

\item[\mylabel{item_Linv3}{ $\mathscr L3$})]  
Assuming that $f$ has CM\footnote{This is so if $k=2$ and for general $k$, this still is expected always to be the case.}, the non-vanishing of $\mathscr L_{\Iw}^{\rm cr}(V_f)$ at $\chi^{-j}$   follows from
{the $p$-adic Beilinson conjecture for Katz $p$-adic $L$-function  \eqref{item_pBCM} at $j$.} 
\begin{itemize}
\item
Moreover, Proposition~\ref{prop_2022_04_26_13_00} and Proposition~\ref{prop: consequences of higher PR} establish very general criteria for the non-vanishing of $\mathscr L_{\Iw}^{\rm cr}(V_f)$ when $L(f^*,\frac{k}{2})=0$. These can be verified in many cases of interest, cf. Corollary~\ref{cor_2022_07_01_1502}.
\item Proposition~\ref{prop_2022_04_26_13_00} and Proposition~\ref{prop_22_04_26_1150} tell us further that, under suitable hypothesis (which we expect to hold almost always), $\mathscr L_{\Iw}^{\rm cr}(V_f)$ has a simple pole\footnote{We refer the reader to the final paragraph of \S\ref{subsubsec_summary_intro_Chapter_3} where we elaborate on this point. In very rough terms, when $L(f^*,\frac{k}{2})=0$, the Selmer complex associated with the $\theta$-critical $p$-stabilization of $f$ degenerates and produces a Selmer group that is expected to be strictly smaller than the Bloch--Kato Selmer group. The pole of $\mathscr L_{\Iw}^{\rm cr}(V_f)$ compensates this discrepancy.} at $\chi^{-\frac{k}{2}}$ whenever $L(f^*,\frac{k}{2})=0$.
\end{itemize}

\item[\mylabel{item_Linv4}{ $\mathscr L4$})]  Proposition~\ref{prop: regulators and L-invariant} expresses  the value of $\mathscr L_{\Iw}^{\rm cr}(V_f)$ at $\chi^{-j}$ (for $j\geq k$) as a ratio of two regulators. This is in line with Perrin-Riou's $p$-adic Beilinson philosophy, put together with Proposition~\ref{prop_comparison_alpha_beta_padic_L_intro}.
\end{itemize}


 \subsubsection{Punctual main conjecture for the improved $p$-adic $L$-function and Iwasawa descent}
\label{subsubsec_summary_intro_Chapter_3}

Recall that the $\beta^*$-stabilization of $f^*$ is $p$-ordinary, and  the classical ($p$-ordinary) Main Conjecture \ref{item_MCbeta} relates the $p$-adic $L$-function $L_{\mathrm{S},\beta^*}(f^*,\xi^*)$
to the module $\mathbf{L}_{\Iw,\beta} (T_f(k),N_\beta [k])$ of algebraic  $p$-adic  $L$-functions. On the other hand, if $\mathscr L_{\Iw}^{\rm cr}(V_f)$ is a nonzero function, Perrin-Riou's construction provides us with a well-defined
module of algebraic improved $p$-adic $L$-functions $\mathbf{L}_{\Iw,\alpha}^{\mathrm{imp}}(T_f(k), N_\alpha[k]).$ We can now formulate the punctual $\theta$-critical Main Conjecture \ref{item_MCalpha}:

\begin{conj*} 
\label{conj:punctual MC intro}
Conjecture \ref{conj: L-inv intro} holds and 
\\
\ref{item_MCalpha} \qquad\qquad\qquad 
$\mathbf{L}_{\Iw,\alpha}^{\mathrm{imp}}(T_f(k),N_\alpha[k])^\pm =\left(
\lambda^\pm (f^*) \cdot L_{\mathrm{S},\alpha^*}^{\mathrm{imp},\pm}(f^*,\xi^*)^\iota\right )$ as $\LL$-modules.
\end{conj*}

We immediately obtain the following from Proposition~\ref{prop_comparison_alpha_beta_padic_L_intro} and the known cases 
of the $p$-ordinary Main Conjecture:
\begin{thmx}
\label{thm_comparison_alpha_beta_MC_intro} 
Assume that $e=2$ and $\mathscr L_\Iw^{\mathrm{cr}}(V_f)$ is a nonzero function. 
Then the conjectures \ref{item_MCalpha}  and \ref{item_MCbeta} are equivalent.
In particular, \ref{item_MCalpha} holds assuming that  either the condition \eqref{item_SZ1} or else the condition \eqref{item_SZ2} below is valid.
\end{thmx} 

Modules of $p$-adic $L$-functions are defined in terms of an Iwasawa theoretic fundamental line given as in \S\ref{subsubsec_3312_20220505}, whose decent properties play a crucial role in our leading term formulae for the modules of $p$-adic $L$-functions in \S\ref{sec_Iwasawa_theoretic_descent}. Our main result in this vein is Theorem~\ref{thm:descent thm: noncentral case}, where we prove Birch and Swinnerton-Dyer-style formulae (both within and off the critical range) for the the module of $p$-adic $L$-functions  $\mathbf{L}^{\mathrm{imp}}_{\Iw,\alpha} (T_{f}(j),N_\alpha[j])$, involving the classical invariants associated to $f_\alpha$, which include the Tate--Shafarevich groups, $p$-adic regulators and Tamagawa numbers; see \S\ref{subsect: Tate-Shafarevich} for definitions of these objects.

Let us assume in this paragraph that $L(f^*,\frac{k}{2})=0$. We close \S\ref{subsubsec_0433_20223_07_10_1242} with an explanation of why it is conceivable in this scenario that $\mathscr L_{\Iw}^{\rm cr}(V_f)$ has a simple pole at $\chi^{-\frac{k}{2}}$. In view of Proposition~\ref{prop_comparison_alpha_beta_padic_L_intro}, this is equivalent to saying that the order of vanishing of $L_{\mathrm{K},\alpha^*}^{[0],\mathrm{imp}}(f^*,\xi^*)$ at $\chi^{-\frac{k}{2}}$ is one less than that of the slope-zero $p$-adic $L$-function $L_{\mathrm{K},\beta^*}(f^*, \xi^*)$. If \ref{item_MCalpha} and \ref{item_MCbeta} both hold true (cf. Theorem~\ref{thm_comparison_alpha_beta_MC_intro}), then this is equivalent to the assertion that the order of vanishing of the algebraic $p$-adic $L$-function $\mathbf{L}_{\Iw,\alpha}^{\mathrm{imp}}(T_f(k), N_\alpha[k])^\pm$ at $\chi^{-\frac{k}{2}}$ is one less than that of its slope-zero counterpart $\mathbf{L}_{\Iw,\beta}(T_f(k), N_\beta[k])^\pm$. In Theorem~\ref{thm: bockstein map in central critical case} below, we explain (under certain hypotheses which are expected to always hold true) that this amounts to a comparison of the dimensions of the fine Selmer group $H^1_0(V_f(\frac{k}{2}))$ and the Bloch--Kato Selmer group $H^1_{\rm f}(V_f(\frac{k}{2}))$, and indeed that the dimension of the former is one less than that of the latter.  


\subsection{Main Conjectures for the infinitesimal deformation}
\label{subsubsec_summary_intro_Chapter_4}
Theorem~\ref{thm_comparison_alpha_beta_MC_intro} shows that, once the  non-vanishing of $\mathscr L_\Iw^{\mathrm{cr}}(V_f)$ is known, the Main Conjectures \ref{item_MCalpha} and \ref{item_MCbeta} are equivalent. Let us set $\widetilde E=E[X]/(X^2)$ and $\widetilde{\mathscr{H}}(\Gamma)=\widetilde E\otimes_E
\CH (\Gamma).$
The interpolation formulae for  $L_{\mathrm{S},\alpha^*}^{[1]}(f^*,\xi^*)$ (compare with Theorem~\ref{thm_comparision_with_Bellaiche_intro}(i)) suggest that not only the critical $p$-adic $L$-function $L_{\mathrm{S},\alpha}^{[0]}(f^*,\xi^*)$ but the \emph{thick} $p$-adic $L$-function 
$$
\widetilde{L}_{\mathrm{S},\alpha^*}(f^*,\xi^*)=
{L}_{\mathrm{S},\alpha^*}^{[0]}(f^*,\xi^*)+
X\cdot {L}_{\mathrm{S},\alpha^*}^{[1]}(f^*,\xi^*) \in\widetilde{\mathscr{H}}(\Gamma),
$$
carries arithmetic information. Since the construction of Bella\"{\i}che's secondary $L$-function $L_{\mathrm{S},\alpha}^{[1]}(f,\xi)$ 
involves the infinitesimal deformation of $V_f$ along the eigencurve, it is natural to consider this deformation as a natural source of arithmetic invariants.   

This motivates us to define and study the properties of what we call the thick Selmer complex and (Iwasawa theoretic) fundamental line associated with the infinitesimal deformation $V_k$ of $V_f$ along the eigencurve. As before, we combine Perrin-Riou's approach in \cite{perrinriou95} and the formalism of Selmer complexes.  In the context of infinitesimal deformations studied in this book, it is crucial that we work in the derived categories of modules over appropriate Iwasawa algebras,
and the theory of Selmer complexes is indispensable even for the formulation of the Main Conjecture. The reason is that the infinitesimal deformations of classical Iwasawa algebras contain nilpotent elements and the theory of determinants over such rings is far more complicated than that over regular rings.  We refer the reader to Section~\ref{sec_fundamental_line_20220505} for further details and explanations. We also note that the exposition of classical Perrin-Riou's theory from the point of view of Selmer complexes can be found in \cite{benoisextracris}.

Let $\varpi_E$ denote a uniformizer of $E$ and let us put $\cO_E^{(n)}:= \cO_E+\varpi^{-n}\cO_EX\subset \widetilde E.$ It is not difficult to see (cf. Lemma~\ref{lemma_thick_lattice}) that for any $n\gg 0$, there exists  a Galois-stable $\cO_E^{(n)}$-lattice $T^{(n)}_{k}$ contained in the free $\widetilde{E}$-module $V_k$ of rank $2$, which can be viewed as an infinitesimal thickening of $T_f$, and which is uniquely determined in a suitable sense. Set $\LL^{(n)}=\cO_E^{(n)}\otimes_{\cO_E}\LL$ and $\widetilde \CH(\Gamma)=\widetilde{E}\otimes_E \CH (\Gamma).$ Fix a free $\cO_E^{(n)}$-submodule $N^{(n)}$
of $\DCc (\bD_\cX)/X^2 \DCc (\bD_\cX)$ such that $N^{(n)}\otimes_{\cO^{(n)}_E}\cO_E= N_\alpha.$


We consider the diagram 
$$
\xymatrix{
\RG_{{\Iw},S}( T_{k}^{ (n)}(j))\otimes_{\Lambda^{ (n)}}\widetilde \CH(\Gamma)
\ar[r] & \underset{\ell\in S}\bigoplus \RG_{\Iw} (\QQ_\ell, T_{k}^{ (n)}(j))\otimes_{\Lambda^{ (n)}} 
\widetilde\CH (\Gamma)\\
 & \left (\underset{\ell \in S}\bigoplus \RG_{\Iw}(\QQ_\ell,  T_{k}^{(n)}(j), 
 { N_{k}^{(n)}[j]})
\right ) \otimes_{\Lambda^{(n)}} \widetilde\CH (\Gamma)\,\,.
 \ar[u]
}
$$
Here  $\RG_{\Iw}(\QQ_\ell,  T_{k}^{(n)}(j), { N_{k}^{(n)}[j]})$ denote the local  condition at $\ell\in S.$
For $\ell \in S\setminus\{p\}$ we take the usual unramified local conditions, which  (see \cite{nekovar06}):
\[
\RG_{\Iw}(\QQ_\ell,  T_{k}^{(n)}(j), { N_{k}^{(n)}[j]}) :=
\left [(T_k^{ (n)}(j))^{I_\ell} \otimes \Lambda^{\iota} \xrightarrow{1-f_\ell} (T_k^{(n)}(j))^{I_\ell}\otimes \Lambda^{\iota} \right ],
\]
where $I_\ell$ is the inertia subgroup at $\ell$ and $f_\ell$ is the geometric Frobenius. In particular, it doesn't depend on $N_{k}^{(n)}.$
For $\ell=p,$ we take 
\[
\RG_{\Iw}(\Qp,  T_{k}^{(n)}(j), { N_{k}^{(n)}[j]}):=\left (N_{k}^{(n)}[j]\right ) [-1],
\]
and the vertical map is induced by the derived version of the large exponential map.
Let $\RG_\Iw (V_k,\alpha)$ denote the Selmer complex defined by the above diagram.
If Conjecture~\ref{conj: L-inv intro} holds (i.e. the Iwasawa theoretic $\mathscr L$-invariant $\mathscr L_\Iw^{\mathrm{cr}} (V_f)$ is non-zero), then the 
cohomology of this complex is concentrated in degree $2$, and the $\widetilde\CH (\Gamma)$-module $\mathbf{R}^2\boldsymbol{\Gamma}_\Iw (V_k,\alpha)$ has a free resolution of length $1$. This allows us to define a canonical trivialization
\[
\label{eqn_trivialization_j_intro}
i_{\Iw,V_{k}(j)}^{(\alpha)}\,:\, 
{\det}_{\widetilde\CH (\Gamma)}^{-1} \RG_\Iw \left ( V_{k}(j),\alpha\right )
\xhookrightarrow{\quad} \widetilde\CH(\Gamma)\,.
\] 
As in the classical case (cf. \cite{perrinriou95,benoisextracris}),
the Iwasawa-theoretic fundamental line $\Delta_{\Iw} \left ( T_{k}^{(n)}(j), { N_{k}^{(n)}}[j] \right )$ is defined as 
\begin{align*}
\Delta_{\Iw}& \left ( T_{k}^{ (n)}(j), { N_{k}^{(n)}[j]} \right ):=\\
&{\det}^{-1}_{\LL^{(n)}}\left (\RG_{\Iw,S}( T_{k}^{ (n)}(j)) \oplus \left (\underset{\ell\in S}\bigoplus
\RG_{\Iw} \left (\QQ_\ell, T_{k}^{ (n)}(j),{ N_{k}^{(n)}[j]} \right )\right )\right ) \otimes 
{\det}_{\LL^{ (n)}} \left (\underset{\ell\in S}\bigoplus
\RG_{\Iw} \left (\QQ_\ell, T_{k}^{(n)}(j) \right )\right ).
\end{align*}
Note that it follows from the general results of Flach \cite{flach} that the determinants in this definition are well-defined. The exact triangle \eqref{eqn_28_2021_06_02} provides us with  a natural injection
\begin{equation}
\label{eqn_2023_08_10_0844}
   \Delta_{\Iw} \left ( T_{k}^{(n)}(j), { N_{k}^{(n)}[j]} \right )
 \xhookrightarrow{\quad}  {\det}^{-1}_{\widetilde\CH (\Gamma)}\RG_{\Iw}\left (V_{k}(j),\alpha\right )\,. 
\end{equation}
Let us set $\cO_E^{(\infty)}:=\underset{n\geqslant 0}\cup \cO_E^{(n)}$ and put
\[
\Delta_{\Iw} \left ( T_{k}^{(\infty)}(j), N_k^{(\infty)}[j] \right ):=
\Delta_{\Iw} \left ( T_{k}^{(n)}(j), { N_{k}^{(n)}[j]} \right )\otimes_{\cO_E^{(n)}}\cO_E^{(\infty)}, \qquad n\gg 0.
\]
(One can show that this definition does not depend on $n$.) By linearity, the injection \eqref{eqn_2023_08_10_0844} extends to an injective map
\be\label{eqn_2023_08_14_0959}
 \Delta_{\Iw} \left ( T_{k}^{(\infty)}(j), { N_{k}^{(\infty)}[j]} \right )
 \xhookrightarrow{\quad}  {\det}^{-1}_{\widetilde\CH (\Gamma)}\RG_{\Iw}\left (V_{k}(j),\alpha\right )\,,
\ee
which allows to identify $\Delta_{\Iw} \left ( T_{k}^{(\infty)}(j), {N_{k}^{(\infty)}[j]} \right )$ with a projective $\Lambda^{(\infty)}$-submodule
of the $\widetilde{\CH}(\Gamma)$-module ${\det}^{-1}_{\widetilde\CH (\Gamma)}\RG_{\Iw}\left (V_{k}(j),\alpha\right ).$
Armed with these constructions, we define the infinitesimal thickening 
\linebreak
${\mathbf L}_{\Iw, \alpha } \left (T_{k}^{(\infty)}(j), N_k^{(\infty)}[j]\right )$ of the module of $p$-adic $L$-functions as the image of 
$\Delta_{\Iw} \left ( T_{k}^{(\infty)}(j), N_k^{(\infty)}[j] \right )$
under \eqref{eqn_2023_08_14_0959} composed with the canonical map $i_{\Iw,V_{k}(j)}^{(\alpha)}$.


We may now formulate the infinitesimal thickening \ref{item_MCinf} of the Main Conjecture:

\begin{conj*}
\label{conj: thick MC intro}
Conjecture \ref{conj: L-inv intro} holds true, and 
\\
\ref{item_MCinf} \qquad\qquad\qquad 
$\mathbf{L}_{\Iw,\alpha} \left (T_k^{[\infty)}(k),N^{(\infty)}_\alpha[k]\right )^\pm =\left(
\lambda^\pm (f^*) \cdot \widetilde{L}_{\mathrm{S},\alpha^*}^{\pm}(f^*,\xi^*)^\iota\right )$ \quad as $\LL^{(\infty),\pm}$-modules.
\end{conj*}

The infinitesimal thickening \ref{item_MCinf} of the Main Conjecture implies the punctual Main Conjecture~\ref{item_MCalpha} (and therefore, also the $p$-ordinary Main Conjecture \ref{item_MCbeta}). Moreover, based on the Iwasawa descent formalism we develop in \S\ref{subsec_425_2022_08_19_1548}, we prove Theorem~\ref{thm_414_2022_0810_1712_intro} below, which tells us that \ref{item_MCinf} is consistent with the combined conclusions of \ref{item_MCbeta} and Theorem~\ref{thmA}. 

\begin{thmx}[Theorem~\ref{thm_414_2022_0810_1712} and Corollary~\ref{cor: infinitesimal descent theorem}]
\label{thm_414_2022_0810_1712_intro}
Let $1\leqslant j\leqslant k-1$ be an integer such that either $j\neq \frac{k}{2}$, or else $j= \frac{k}{2}$ and $L(f^*,\frac{k}{2})\neq 0.$
Assume that $\mathscr L^{\mathrm{cr}}(V_f(k-j))\neq 0$ 
and \ref{item_MCinf} is valid. Then:
\[
 L^{[1],\pm}_{\mathrm{K},\alpha^*}(f^*, \xi^*,\chi^j)=u\cdot C_{\mathrm{K}}\cdot \mathcal E_N(f^*;\chi^j)\cdot \frac{e_{p,\alpha^*}(f^*\mathds 1,j)}{e_{p,\beta^*}(f^*\mathds 1,j)} \cdot \mathbf L_{\Iw,\beta}(T_{f}(k-j),N_\beta[k-j],\mathds 1,0)^\pm\,.
\]
\end{thmx}
Here, $u\in \cO_E^\times$ is a $p$-adic unit, $C_{\mathrm{K}}$ is  given as in \eqref{eqn: the constant a}, 
and $\mathbf L_{\Iw,\beta}(T_{f}(k-j),N_\beta[k-j],\mathds 1,0)^\pm$ is the leading term at the trivial character of the module of slope-zero $p$-adic $L$-functions.

We remark that the thick Selmer complexes $\RG_{\Iw} \left (V_{k}(j), \alpha\right )$ are not semi-simple. As a result, we can not make use of the standard descent techniques (cf. \cite{BurnsGreither2003Inventiones}, Lemma~8.1). Note that Theorem~\ref{thm_414_2022_0810_1712_intro} avoids the scenario when $L(f^*,\frac{k}{2})=0$, since our descent formalism fails in that case.



\subsubsection{Height formulae}
\label{subsubsec_summary_intro_Chapter_5}
In a complementary direction to our work we summarized in \S\ref{subsubsec_summary_intro_Chapter_3} and \S\ref{subsubsec_summary_intro_Chapter_4} (where our main tool was Iwasawa descent), we calculate the special values of our $p$-adic $L$-functions at the central critical point (and when $L(f,\frac{k}{2})=0$) in terms of $p$-adic regulators:

\begin{thmx}[Theorem~\ref{thm_A_adic_regulator_formula_eigencurve}]
\label{thm_A_adic_regulator_formula_eigencurve_intro}
Suppose that ${\rm ord}_{s=\frac{k}{2}}L(f,s)=1$ and Perrin-Riou's conjecture \eqref{item_PR1} as well as the condition \eqref{item_C4} hold true. 
\item[i)] ($\cO_\cX$-adic leading term formula)  For sufficiently small $\cX$, we have the following identity in $\cO_{\cX}$:
$$X\cdot {\rm LOG}_{\eta}(P) \cdot \frac{\partial}{\partial s}L_{\mathrm{K},\eta}^{\pm}(\cX,\xi )\ \Big{\vert}_{
{s=\frac{w(x)}{2}}
}=-  R_{\cX}\,.$$
Here, ${\rm LOG}_{\eta} $ is the big logarithm introduced in Definition~\ref{defn_big_LOG_X}\,, $P$ is the fixed generator of the big Selmer group $\mathbf{R}^1\boldsymbol{\Gamma}(\Vcc_\cX,\Dcc_\cX)$ as in \S\ref{subsubsec_5351_2022_12_19_1351}, and $R_{\cX}$ is the $\cO_\cX$-adic regulator given as in \eqref{eqn_2023_07_10_1022}\,.

\item[ii)] (Leading term formula for the thick $p$-adic $L$-function)
We have the following identity in $\cO_\cX/X^2\cO_\cX$:
\begin{align*}
\label{eqn_main_thm_dsdX_eigencurve_intro}
\begin{aligned}
\log_{\omega_{x_0}'}(P_{x_0})\,\frac{\partial}{\partial s} \widetilde{L}_{\mathrm{K},\alpha}^{\pm}&(f,\xi)\Big{\vert}_{s=\frac{k}{2}} \\
&= {(-1)^{\frac{k}{2}-1}  \,C_{\mathrm{K}}}\,{\Gamma\left(\frac{k}{2}\right) \left(1-\frac{p^{\frac{k}{2}-1}}{\alpha}\right)\left(1-\alpha p^{-\frac{k}{2}}\right)^{-1}}\left(\frac{1}{2}\left\langle P_{k}',P_{k}\right\rangle_{\Dcc_{k}',\Dcc_{k}} +\frac{X}{2}\frac{d^2 R_{\cX}}{dX^2}\big\vert_{X=0}\right).
\end{aligned}
\end{align*}
where $P_{x_0}\in \mathbf{R}^1\boldsymbol{\Gamma}(\Vcc_{x_0},\Dcc_{x_0})$ is the specialization of $P$ to the center $x_0$ of the affinoid $\cX$ (which corresponds to the $\theta$-critical eigenform $f_\alpha$), and 
$$\mathbf{R}^1\boldsymbol{\Gamma}(\Vcc_{k}',\Dcc_{k}')\otimes \mathbf{R}^1\boldsymbol{\Gamma}(\Vcc_{k},\Dcc_{k})\xrightarrow{\langle \,,\,\rangle_{\Dcc_{k}',\Dcc_{k}}} E$$ 
is the specialization of the $\cO_\cW$-valued height pairing $\langle \,,\,\rangle_{\Dcc',\Dcc}$ to weight $k$.
\end{thmx}
In particular, Theorem~\ref{thm_A_adic_regulator_formula_eigencurve_intro}(ii) tells us that the derivative of an $\cO_\cX$-adic regulator (rather than a regulator itself) computes the derivative $\frac{\partial}{\partial s} {L}_{\mathrm{K},\alpha}^{[1],\pm}(f,\xi)\Big{\vert}_{s=\frac{k}{2}} $ of the secondary $p$-adic $L$-function.


 \subsubsection{}
It is evident that our constructions in the context of the Buzzard--Coleman--Mazur eigencurve is a particular case of a very general phenomenon. It would be very interesting to study this in the context of general eigenvarieties. This includes the case of $(\GL2\times\GL2)_{/\QQ}$, which we plan to investigate in the near future. 


 \subsubsection{}
A natural question\footnote{We thank the anonymous referee for bringing this to our attention.} is what happens when the $\theta$-critical eigenform $f^\alpha$ is replaced with a $p$-regular theta series $g$ of a real quadratic field in which $p$ splits. By the work of Bella\"iche and Dimitrov~\cite[Theorem 1.1]{bellaichedimitrov}, the the eigencurve is smooth but not \'etale at the corresponding point $x_g$. Our construction of the $p$-adic $L$-function in \S\ref{chapter:critical L-functions} applies in this scenario. One important difference (between the case of weight-one form and the case of $\theta$-critical forms that we consider here) is that the Kedlaya--Pottharst--Xiao triangulation is saturated at $x_g$, cf. \cite[Proposition 6.4.5]{KPX2014}. As a result, the \emph{extreme exceptional zero phenomenon} (see Part (ii) of Theorem~\ref{thmA}) for the $p$-adic $L$-function associated to a $\theta$-critical eigenform does not seem to occur in this case. In the same vein, since new phenomena (as far as our constructions are concerned) emerge only when the triangulation fails to be saturated at the point of the eigencurve that we are interested in, our arithmetic considerations in Chapters~\ref{chapter_main_conjectures} and onward do not seem to apply in the case of weight-one points.


\subsection{Layout} We close our introduction with a layout of our article.   

Chapter~\ref{chapter: preliminaries} reviews general constructions that we rely on throughout the article. This involves a review of $p$-local Galois cohomology in \S\ref{subsec_cohom_2022_08_19_1628}, Perrin-Riou's large exponential maps in \S\ref{subsec: large exp} and the theory of exponential maps over affinoid algebras in \S\ref{subsec_exp_over_affinoid_algebras}. In \S\ref{sec_Sel_complex_and_duality}, we introduce Selmer complexes (\`a la Perrin-Riou, Pottharst, and the first named author) 
in great generality and review their basic properties. We review the formalism of $p$-adic heights in \S\ref{sec_padic_heights} and conclude this chapter proving in \S\ref{subsec_Rubin_style_formulae} a Rubin-style formulae for $p$-adic height pairings.

We begin Chapter~\ref{chapter:critical L-functions} establishing our ``eigenspace-transition by differentiation'' principle (Proposition~\ref{prop: comparison of exponentials for different eigenvalues}), which is of fundamental importance for our methods. In \S\ref{sec_abstract_setting}, we develop an abstract theory of Perrin-Riou $p$-adic $L$-functions in a scenario where the  underlying  triangulation is non-saturated. This is, of course, modelled on the behaviour of the relevant objects in the context of the Buzzard--Coleman--Mazur eigencurve, and inspired by the beautiful work of Bella\"iche~\cite{bellaiche2012}. The eigenspace-transition by differentiation principle is used in this portion to prove the interpolation properties of the secondary (abstract) $p$-adic $L$-function (cf. Theorem~\ref{prop_imoroved_padicL_vs_slope_zero_padic_L}). 

In \S\ref{sec_new_2_3_2022_03_14} and onwards, we focus our attention on the eigencurve and apply our general machinery in this context. We carefully review in \S\ref{sec_new_2_3_2022_03_14} the earlier constructions of $p$-adic $L$-functions attached to modular forms (and on the eigencurve) via modular symbols. We give an ``\'etale'' construction of critical $p$-adic $L$-functions in \S\ref{sec_2_4_2022_05_11_0809} (see \S\ref{subsec_defn_critical_padic_L_eigencurve} for the case of $\theta$-critical forms, which is the main focus of the present paper). We also explain in  \S\ref{subsec_defn_critical_padic_L_eigencurve}  and \S\ref{subsec_245_2022_05_11_0845} their relation with their counterparts constructed via modular symbols.

Chapter~\ref{chapter_main_conjectures} is devoted to an Iwasawa theoretic treatment of a $\theta$-critical form $f_\alpha$. After defining various Selmer groups attached to Deligne's Galois representation $V_f$ in \S\ref{sec_3_2_2022_08_19_1700}, we introduce critical $\cL$-invariants in \S\ref{subsec:critical L-invariants}, study their basic properties, and prove in \S\ref{subsubsec_3233_2022_05_06_1354} that they can be interpolated to an Iwasawa theoretic $\mathscr L$-invariant (which appears in the statement of Proposition~\ref{prop_comparison_alpha_beta_padic_L_intro}). A big bulk of \S\ref{subsec:critical L-invariants} is dedicated to establish explicit criteria for the non-vanishing of the critical $\cL$-invariant. In \S\ref{sec_modules_of_algebraic_padic_L_functions}, we introduce our modules of (punctual) algebraic $p$-adic $L$-functions (both for the critical $p$-stabilization $f_\alpha$ and the slope-zero $p$-stabilization $f_\beta$), in terms of Iwasawa theoretic fundamental lines and Selmer complexes we define therein. 

We formulate the punctual main conjectures \ref{item_MCalpha} and  \ref{item_MCbeta} in \S\ref{sect: main conjecture for critical forms} and study their relation with one another. 
For completeness,  in \S\ref{subsec_2022_08_09_1622}, we  prove Birch and Swinnerton-Dyer formulae for the module of algebraic $p$-adic $L$-functions within and off the critical range for $f$ (which are due to Perrin-Riou in the non-$\theta$-critical scenario). These, of course, can be translated to leading term formulae for the punctual critical $p$-adic $L$-function $L^{[0]}_{\mathrm{K},\alpha^*}(f^*, \xi^*)$ whenever the punctual critical main conjecture \ref{item_MCalpha} holds\footnote{We explained in Theorem~\ref{thm_comparison_alpha_beta_MC_intro} that \ref{item_MCalpha} holds assuming that $e=2$, the Iwasawa theoretic $\mathscr L$-invariant $\mathscr L_\Iw^{\rm cr} (V_f)$ is nonzero, and either the condition \eqref{item_SZ1} or else the condition \eqref{item_SZ2} is valid.}.  

In Chapter~\ref{chapter_main_conj_infinitesimal_deformation}, we carry out the key algebraic constructions to counter our thick $p$-adic $L$-functions. In order to establish an integral theory (which can in turn afford a non-trivial descent theory), we first prove in \S\ref{sec_fundamental_line_20220505} that the infinitesimal deformation $V_k$ contains a Galois stable lattice. Attached to this lattice, we introduce our thick fundamental line in \S\ref{subsec_complexes_412}, whose determinant naturally embeds into the determinant of our Iwasawa theoretic thick Selmer complex (cf. \eqref{eqn: injection for trivialization infinity}). We establish a general criteria to bound the projective dimensions of $\widetilde{\mathscr H}_E(\Gamma)$-modules in \S\ref{subsubsec_aux_trivialization}, which we apply with our Iwasawa theoretic thick Selmer complex to trivialize its determinant (assuming that the Iwasawa theoretic $\mathscr L$-invariant is nonzero), an in turn also our thick fundamental line. 

With these constructions at hand, we can define our infinitesimal thickening of the module of algebraic $p$-adic $L$-functions (cf. Definition~\ref{defn_eqn_trivialization_thick_module_of_padic_L}) and formulate, in \S\ref{subsec_IMC_thick}, the infinitesimal thickening \ref{item_MCinf} of the critical main conjectures, where we also prove that \ref{item_MCinf} is stronger than the punctual main conjecture \ref{item_MCalpha}. We establish a general descent formalism over the ring $\widetilde{\mathscr H}_E(\Gamma)$ in \S\ref{subsubsec_4232}, which we apply in \S\ref{subsec_thick_bockstein_424} with the thick Iwasawa theoretic Selmer complex to obtain an explicit interpretation of the Bockstein morphisms (which are relevant to the definitions of $p$-adic heights). Based on these, we establish an Iwasawa descent procedure in \S\ref{subsec_425_2022_08_19_1548} avoiding the central critical point if $L(f^*,\frac{k}{2})=0$. In this scenario, we prove a factorization formula for the specializations of the thick fundamental line in Theorem~\ref{thm: descent for tilt complex} in terms of the critical and slope-zero fundamental line, and a leading term formula for the secondary $p$-adic $L$-function in terms of the slope-zero fundamental line (cf. Theorem~\ref{thm_414_2022_0810_1712_intro} above). 

We supplement our discussion in Chapter~\ref{chapter_main_conj_infinitesimal_deformation} in \S\ref{subsec_thick_Selmer_groups_16_11} (cf. Proposition~\ref{prop_4_14_2022_08_19_1401}) with a discussion of the Iwasawa descent procedure at the central critical point in the scenario when ${\rm ord}_{s=\frac{k}{2}}L(f^*,s)=1$.

We employ in Chapter~\ref{chapter_critical_Selmer_padic_reg} our general Rubin-style formula (that we prove in \S\ref{subsec_Rubin_style_formulae}) to compute the leading terms of our $p$-adic $L$-functions and establish Theorem~\ref{thm_A_adic_regulator_formula_eigencurve_intro}. Along the way, we explain (in \S\ref{subsec_thick_Selmer_groups_16_11}) that the thick Selmer complexes produces Selmer groups that extend the Bloch--Kato Selmer groups (rather than the fine Selmer groups), and thereby reflect better the extremal exceptional zero phenomena for the critical $p$-adic $L$-functions.

\subsection*{Acknowledgements}
The first named author (D.B.) wishes  to thank the second named author (K.B.)  for his invitation to University College Dublin in April 2019 and Istanbul Centre for Mathematical Sciences (IMBM) in December 2019. This work started during these visits. The second named author (K.B.) thanks Universit\'e de Bordeaux for the warm hospitality, where many of the ideas in the present work were developed. K.B. also thanks M. Bertolini, M. Dimitrov, C.-Y. Hsu and V. Pa\v{s}k\={u}nas for stimulating conversations and comments on various portions of the present article. Special thanks are due to J. Bergdall for explaining to us that the infinitesimal deformation $V_k$ is never crystalline (cf. Remark~\ref{remark_bergdall}). The authors thank the anonymous referee for his/her thorough reading of our monograph and constructive suggestions, which helped us improve the exposition of our work.

It will be clear to the reader that this monograph draws greatly from the spectacular ideas of Jo\"el Bella\"iche and Jan Nekov\'a\v{r}. We dedicate this work to their memories.

\chapter{General constructions}
\label{chapter: preliminaries}
We begin Chapter~\ref{chapter: preliminaries} recording the basic notation and conventions that we adopt in this work. We then review various $p$-local constructions (e.g. the ``\emph{large exponential map} in \S\ref{subsec: large exp} and \S\ref{subsec_exp_over_affinoid_algebras}, including that for families of Galois representations; Tamagawa numbers in \S\ref{subsec_Tamagawa_numbers}). We introduce Selmer complexes in \S\ref{sec_Sel_complex_and_duality} and note for future use their 
fundamental properties (including key duality statements in \S\ref{subsec_global_duality}). We define $p$-adic height pairings on these in \S\ref{sec_padic_heights}, and conclude this chapter with the proof of Rubin-style formulae in \S\ref{subsec_Rubin_style_formulae}, which we shall utilize in \S\ref{sec_padic_regulators} when proving leading term formulae for critical $p$-adic $L$-functions.

\section{The large exponential map}
\label{sec_exponential_maps}
\subsection{Basic notation and conventions}
\label{Basic notation and conventions}
\subsubsection{}

Let $p$ be an odd prime and let $G_{\Qp}:=\Gal (\overline{\QQ}_p/\Qp).$
Fix a compatible system $\varepsilon =(\zeta_{p^n})_{n\geqslant 0}$ of $p^n$th roots of unity.
For $n\geqslant 1,$  set $K_n:=\Qp (\zeta_{p^n}).$ 
 Let $K_\infty=\underset{n\geqslant 1}\cup K_n$  and $\Gamma=\Gal (K_\infty/\Qp).$ 
For any $n\geqslant 1,$ set  $\Gamma_n:=\Gal (K_\infty/K_n)$ and  $G_n=\Gal (K_n/\Qp)$. 
We will denote indifferently by $\chi$ both the cyclotomic character $\chi \,:\,G_{\Qp}\rightarrow \Zp^*$ and its factorisation $\chi\,:\, \Gamma \rightarrow \Zp^*$ through $\Gamma$. 
If $n=1,$ we will write  $\Delta:=\Gal (K_1/\Qp)$ instead of $G_1$ and denote by $\omega\,:\, \Delta \rightarrow  \mu_{p-1}$ the restriction of the cyclotomic character $\chi$ on $\Delta .$ 
Fix a system of generators 
$\gamma_n$ of $\Gamma_n$ such that $\gamma_{n+1}=\gamma_n^p$ for all $n\geqslant 1.$

\subsubsection{} 
Fix a coefficient field $E/\Qp$ and denote by $\cO_E$ its ring of integers. We denote by  $\vert \,\cdot \,\vert$ the  absolute value on $E$ normalised so that $\vert p\vert =p^{-[E:\Qp]}$. Note that we have
\[
\vert a\vert^{-1} =  [\cO_E : a\cO_E], \qquad \forall a\in \cO_E.
\]

\subsubsection{} 
 Let $\LL :=\cO_E [\Delta]\otimes_{\cO_E}\cO_E [[\Gamma_1]]$ denote the Iwasawa algebra of $\Gamma$
over $\cO_E.$  We will also work with  Perrin-Riou's large Iwasawa algebra $\CH (\Gamma).$
Let $\CH $ denote  the algebra of formal power series $f(z)=\underset{j=0}{\overset{\infty}\sum} a_jz^j$ with coefficients in $E$ that converge on the open unit disc. Set
$$\CH (\Gamma_1):=\{f(\gamma_1-1) \mid f(z)\in \CH \}$$ 
and 
$$\CH (\Gamma)=\Zp [\Delta]\otimes_{\Zp} \CH (\Gamma_1).$$ 

\subsubsection{} \label{subsubsec_isotypic_components_Delta}
The elements 
\[
\delta_i=\frac{1}{\vert \Delta \vert} \underset{g\in \Delta}\sum \omega^{-i}(g) g, \qquad 0\leqslant i\leqslant p-2.
\]
are idempotents of $\Delta$ and one has the decompositions
\begin{equation}
\label{eqn: decomposition of LL}
\begin{aligned}
&\LL=\underset{i=0}{\overset{p-2}\oplus} \LL_{\omega^i}\,, &&  \textrm{ where $\LL_{\omega^i}=\cO_E [[\Gamma_1]] \delta_i$},\\
&\CH (\Gamma)=\underset{i=0}{\overset{p-2}\oplus} \CH (\Gamma)_{\omega^i}\,, &&
\textrm{where $\CH(\Gamma)_{\omega^i}=\CH  (\Gamma_1)\delta_i.$}
\end{aligned}
\end{equation}
The algebra  $\CH  (\Gamma)$ comes equipped with the twisting operators $\Tw_j$ given as follows (sic!):
\[
\Tw_j(f(\gamma-1)\delta_i)=f (\chi (\gamma_1)^{-j}\gamma_1-1)\delta_{i-j}, \qquad  j\in \ZZ.
\]

\subsubsection{} 
If $M$ is any $\Gamma$-module, we denote by  $\Tw_j\,:\, M\rightarrow M(j)$ the twist map
$\Tw_j(m)=m\otimes \varepsilon^{\otimes j}.$ In particular, if $M$ is a $\LL$-module, 
then the twisting morphism 
\[
\Tw_j\,:\, \CH (\Gamma)\otimes_{\LL} M \lra \CH (\Gamma)\otimes_{\LL} M (j)
\]
is explicitly given by the formula
\[
\Tw_j(f(\gamma-1) \otimes m)=\Tw_j (f(\gamma-1))\otimes \Tw_j(m)=f(\chi^{-j} (\gamma)\gamma-1) \otimes \Tw_j (m). 
\]
We denote by $\iota$ the canonical involution on $\Gamma$ given  by $\gamma\mapsto \gamma^{-1}$ and  will use the same notation for the induced involutions on $\LL$ and $\CH (\Gamma).$
If $M$ is a left $\Gamma$-module, we will denote by $M^{\iota}$ the structure of a right  $\Gamma$-module on $M$ defined by $(m,\gamma)\mapsto \iota(\gamma)m.$

\subsubsection{}
The ring of formal power series $\cO_E[[\pi]]$ will be equipped with the following structures:
\begin{itemize}
\item The action of $\Gamma$ defined by $\gamma (\pi)=(1+\pi)^{\chi (\gamma)}, \quad \gamma\in \Gamma$\,.
\item The Frobenius operator $\varphi$ defined by $\varphi (\pi)=(1+\pi)^p-1$\,.

\item The left inverse $\psi$ of $\varphi$ defined as
\be\label{eqn: definition of psi}
\psi (f(\pi)):=\underset{\zeta^p=1}\sum f(\zeta (1+\pi)-1).
\ee

\item The differential operator $\partial =(1+\pi)\frac{d}{d\pi}.$
\end{itemize}
Recall that  $\cO_E [[\pi]]^{\psi=0}$ is the free $\LL$-submodule of $\cO_E [[\pi]]$ generated by $(1+\pi).$ The operator $\partial$ is invertible on $\cO_E [[\pi]]^{\psi=0}.$

\subsubsection{} For any group $G$ and any left $G$-module $M,$ we denote by $M^\iota$ the right $G$-module whose underlying group
is $M$ and on which  $G$ acts by $m\cdot g=g^{-1}m.$

\subsubsection{} If $G$ is a finite abelian group, we denote by $X(G)$ the group of characters of 
$G$ with values in $E.$ We will always assume that $E$ is big enough to contain the values 
of all characters of $G.$  For any primitive character $\rho \in X(G)$,  we denote by 
\[
e_\rho=\frac{1}{\vert G\vert} \underset{g\in G}\sum \rho^{-1}(g) g
\] 
the corresponding indempotent of $E[G]$. For any $E[G]$-module $Y$ we denote 
by $Y^{(\rho)}=e_\rho Y$ its $\rho$-isotypic component, and   for any map $f\,:\, X\rightarrow Y$ we denote by $f^{(\rho)}$ 
the compositum  
$$f^{(\rho)}: X\xrightarrow{f} Y\xrightarrow{[e_\rho]} Y^{(\rho)}.$$ 

\subsubsection{} If $A$ is an affinoid algebra over $E,$ we set
$\mathcal {X}:=\Spm (A)$  and will often write $\cO_{\mathcal X}$ instead $A.$ Set  $\LL_{\mathcal{X}}=\cO_{\mathcal X}\widehat\otimes_{\cO_E}\LL$ and $\CH_{\mathcal X}(\Gamma):=\cO_{\mathcal X}\widehat\otimes_{E}\CH(\Gamma).$

\subsubsection{}  Let  $G$ be a topological group, 
and let $M$ be a module over a commutative ring  $A$ equipped with a continuous linear action of $G.$ 
We will denote by $C^\bullet (G,M)$ the complex of continuous cochains of $G$ with coefficients in $M$ and by   $\RG (G,M)$ the corresponding object of  the derived category $\mathscr D (A)$ of $\cO_\cX$-modules. If $G$ is the absolute Galois group of a local field $K,$ we will write
$\RG (K,M)$ instead $\RG (G,M).$

\subsection{Cohomology of $p$-adic representations}
\label{subsec_cohom_2022_08_19_1628}
\subsubsection{}
A $p$-adic representation $V$ of $G_{\Qp}$ over an affinoid $\mathcal X$ is a finitely generated projective $\cO_{\mathcal X}$-module equipped with a continuous linear action of $G_{\Qp}.$ The complex $\RG (K_n,V)$ computes 
the continuous Galois cohomology: 
\[
H^i(K_n,V):=\bR^i\mathbf\Gamma (\Qp,V).
\] 
We will also consider 
the Iwasawa cohomology $H^i_\Iw(\Qp,V)$ which can be defined as the cohomology of the complex
\begin{equation}
\nonumber
\RG_\Iw (\Qp,V):=\RG(\Qp, V\widehat\otimes \LL^\iota).
\end{equation}
Note that $H^i_\Iw(\Qp,V)=0$ for $i\neq 1,2.$ 
One has a Hochschild--Serre spectral sequence:
\begin{equation}
\label{eqn:hochschild-serre for representations}
H^i(\Gamma_n, H^j_{\Iw}(\Qp, V)) \,\implies\, H^{i+j}(K_n,V)
\end{equation}
which induces isomorphisms 
\begin{equation}
\nonumber
H^0(K_n,V)\simeq H^1_\Iw (\Qp,V)^{\Gamma_n}, \qquad
H^2(K_n,V) \simeq H^2_\Iw (\Qp,V)_{\Gamma_n}
\end{equation}
and an exact sequence
\begin{equation}
0\lra H^1_\Iw (K_n,V)_{\Gamma_n}\lra H^1(K_n,V)\lra H^1_\Iw (K_n,V)^{\Gamma_n}\lra 0. 
\end{equation}
The first map in this exact sequence is induced by the canonical projection
\[
\pr_n\,:\,H^1_\Iw (\Qp,V) \lra H^1(K_n,V).
\]
If 
\[
V'\times V \lra \cO_{\mathcal X}(\chi)
\]
is an $\cO_{\mathcal X }$-bilinear map, the local duality provides a  pairing
\begin{equation}
\label{eqn:pairing on local galois cohomology}
\begin{aligned}
&\RG (K_n,V')\times \RG (K_n,V) \lra \cO_{\mathcal X}[-2]\,.\\
\end{aligned}
\end{equation}
We denote by
\[
(\,,\,)_n \,:\, H^1(K_n,V')\times H^1(K_n,V) \lra \cO_\cW
\]
the induced pairing on the first cohomology.

\subsubsection{} There exists the Iwasawa-theoretic analog of this pairing
\begin{equation}
\label{eqn:Iwasawa pairing}
\left <\,,\,\right >\,:\,H^1_\Iw (\Qp,V')\times H^1_\Iw (\Qp,V)^\iota \lra \LL_{\mathcal X}
\end{equation}
which has the following explicit description (cf. \cite{perrinriou94}, \S3.6):  
\begin{equation}
\label{formula: Perrin-Riou pairing}
\left <x,y\right >\equiv \underset{\tau\in G_n}\sum 
\left (\tau^{-1}\pr_n^\prime (x)\,,\,\pr_n (y)\right )_n \tau \quad \mod{(\gamma_n-1)}\,.
\end{equation}
From this formula it follows  that  for any finite character $\rho \in X(\Gamma)$ of conductor $p^n$, we have
\begin{equation}
\label{eqn: rho component of Iwasawa pairing}
\rho \left (\left <x,y^\iota\right >\right )= \underset{\tau\in \Gamma/\Gamma_n}\sum 
\left (\tau^{-1}\pr_n^\prime (x),\pr_n (y)\right )_n \rho (\tau)=
\left (e_\rho(x), e_\rho(y^\iota) \right )_{\rho}\,,
\end{equation}
where $\left ( \cdot , \cdot \right )_{\rho}$ stands for the cup-product pairing
$$H^1(\Qp, V^\prime (\chi\rho^{-1}))\otimes_{\cO_\cW} H^1(\Qp, V(\rho)) \lra \cO_\cW.$$
We also recall the formula that describes the behavior of these pairings under  cyclotomic twists.

\subsubsection{}
\label{subsec:cohomology of phi-Gamma modules}
Let $\CR_E$ denote the Robba ring over $E,$ i.e. the ring of power series $f(\pi)=\underset{i\in \ZZ}
\sum a_n\pi^n$ with coefficients in $E$ that converge on some $p$-adic  annulus of the form $r(f)\leqslant \pi<1.$
Define the relative Robba ring over $\mathcal X$ as $\CR_{\mathcal X}:=\cO_{\mathcal X}\widehat\otimes_E\CR_E.$ We will use freely the theory of 
$(\varphi,\Gamma)$-modules  over $\CR_{\mathcal X}.$ If $\bD$ is a $(\varphi,\Gamma)$-module over $\CR_{\mathcal X},$ then,  restricting the action of $\Gamma,$  we can consider $\bD$ as a 
$(\varphi,\Gamma_n)$-module. Let $\RG (K_n,\bD)$ denote the Fontaine--Herr complex of $\bD$ 
in the category of $(\varphi,\Gamma_n)$-modules. We consider $\RG (K_n,\bD)$ as an object of $\mathscr D(\cO_{\mathcal X})$ and set 
\[
H^i(K_n,\bD):=\bR^i\Gamma (K_n,\bD), \qquad i\geqslant 0.
\]
The Iwasawa cohomology $H^i_\Iw (\bD)$  of $\bD$ is defined as the cohomology of the complex
\[
\RG_\Iw (\bD):=\left [\bD \xrightarrow{\psi-1} \bD \right ],
\]
concentrated in degrees $1$ and $2.$ Note that there exists a quasi-isomorphism in $\mathscr  D(\CH_{\mathcal X}(\Gamma)):$
\[
\RG_\Iw (\bD) \simeq \RG (K_n, \bD\widehat{\otimes}_E\CH (\Gamma)^\iota). 
\]
There also exists an analogue of the spectral sequence \eqref{eqn:hochschild-serre for representations} for the cohomology of $(\varphi,\Gamma)$-modules:
\begin{equation}
\nonumber
H^i(\Gamma_n, H^j_{\Iw}(\bD))\,\implies\, H^{i+j}(K_n,\bD).
\end{equation} 
As above, it provides us with isomorphisms
\begin{equation}
\nonumber
H^0(K_n,\bD)\simeq H^1_\Iw (\bD)^{\Gamma_n}, \qquad
H^2(K_n,\bD) \simeq H^2_\Iw (\bD)_{\Gamma_n}
\end{equation}
and an exact sequence
\begin{equation}
0\lra H^1_\Iw (\bD)_{\Gamma_n}\lra H^1(K_n,\bD)\lra H^1_\Iw (\bD)^{\Gamma_n}\lra 0\,. 
\end{equation}
The twist  $\CR_{\mathcal X}(\chi)$ of $\CR_{\mathcal X}$ by the cyclotomic character is a $(\varphi,\Gamma)$-module of rank one, and one has a canonical isomorphism
\[
H^2(K_n,\CR_{\mathcal X}(\chi))\simeq \cO_{\mathcal X}.
\]
For any bilinear map
\[
\bD'\times \bD \lra \CR_{\mathcal X}(\chi)
\]
the previous isomorphism induces  bilinear pairings
\begin{equation}
\nonumber
\begin{aligned}
&\RG (K_n,\bD')\times \RG(K_n,\bD) \lra \cO_{\mathcal X}[-2],\\
&\left <\,,\,\right >_\Iw \,:\,H^1_\Iw (\bD')\times H^1_\Iw (\bD)^\iota\lra \CH_{\mathcal X}(\Gamma),
\end{aligned}
\end{equation}
which generalize the pairings \eqref{eqn:pairing on local galois cohomology} and \eqref{eqn:Iwasawa pairing} in local Galois cohomology. Note an important formula
\begin{equation}
\label{eqn: twist of Iwasawa pairing}
\left <\Tw_{-j}(x), \Tw_j (y)\right >= \Tw_j\circ \left <x,y\right >.
\end{equation}

\subsubsection{} Let $\DdagrigX (V)$ be the $(\varphi,\Gamma)$-module associated to a $p$-adic representation $V.$ Then there exists canonical
quasi-isomorphisms
\begin{equation}
\label{eqn:quasi-iso between galois and phi-gamma cohomology}
\RG (K_n,\DdagrigX(V)) \simeq \RG (K_n,V),\qquad 
\RG_\Iw (\DdagrigX(V)) \simeq \RG_{\Iw} (\Qp,V)\otimes_{\LL_{\mathcal X}}\CH_{\mathcal X}(\Gamma),
\end{equation}
which are compatible with  the pairings and spectral sequences reviewed above. 

\subsection{Bloch--Kato exponential map} 

\subsubsection{}  Let $\bD$ be a $(\varphi,\Gamma)$-module over $E.$ Set
\[
\DCc (\bD):=\left (\bD [1/t]\right )^{\Gamma},
\]
where $t=\log [\varepsilon]$ is the ``$2\pi i$'' of Fontaine. The module 
$\DCc (\bD)$ is equipped with a canonical Frobenius operator $\varphi$ and 
an exhaustive decreasing filtration $\Fil^i\Dc (\bD).$  One says that $\bD$ is crystalline
if $\dim_E\DCc (\bD)=\mathrm{rk}_E (\bD).$

Assume that $\bD$ is crystalline. The elements of $H^1(K_n,\bD)$ classify the extensions 
of the trivial module $\CR_E$ by   $\bD$ in the category of $(\varphi,\Gamma_n)$-modules,
and for $x\in H^1(K_n,\bD)$ let $\bD_x$ denote the corresponding extension. 
Define: 
\begin{equation}
\nonumber
H^1_{\rm f}(K_n,\bD):=\left\{ x\in  H^1(K_n, \bD) \mid {\textrm{$\bD_x$ is  crystalline}}
\right\}
\end{equation}
(see \cite{Benois2011}, \cite{Nakamura2014JIMJ}).
The tangent space  of $\bD$ over $K_n$ is defined as follows:
\begin{equation}
\nonumber
t_{\bD}(K_n):=\left( \DCc (\bD)/\Fil^0\DCc(\bD)\right )\otimes_{\Qp}K_n\,.\\
\end{equation}
In the derived category of $E$-vector spaces, consider the complex
\[
\RG_f (K_n,\bD):= \left [ \Dc (\bD) \xrightarrow{f} \Dc (\bD) \oplus  t_{\bD}(K_n)\right ],
\]
where the terms are placed in degrees $0$ and $1$ and 
\[
f(x)=((1-\varphi)x, x\pmod{\Fil^0\DCc (\bD)\otimes K_n}).
\]
There exists a canonical morphism
\[
\bexp_{\bD,K_n}\,:\, \RG_f (K_n,\bD) \lra \RG (K_n,\bD),
\]
which induces isomorphisms
\[
\bR^0 \mathbf\Gamma (K_n,\bD)\simeq H^0(K_n,\bD), \qquad \bR^1 \mathbf\Gamma (K_n,\bD)\simeq H_f^1(K_n,\bD).
\]
This information can be encoded by the exact sequence 
\begin{equation}
\label{eqn:bloch-kato exact sequence for phi-gamma modules}
0\rightarrow H^0(K_n,\bD) \rightarrow \DCc (\bD) \xrightarrow{f} \DCc (\bD) \oplus  t_{\bD}(K_n)\xrightarrow{\exp_{\bD,K_n}} H^1_{\rm f}(K_n,\bD)
\rightarrow 0.
\end{equation}
Here the connecting map $\exp_{\bD,K_n}$ is  the exponential map defined in \cite{Benois2011},
\cite{Nakamura2014JIMJ}. 

The dual exponential map
\begin{equation}
\label{eqn:definition of dual exponential}
\exp^*_{\bD^*(\chi),K_n}\,:\, H^1(K_n,\bD (\chi))\rightarrow \Fil^0\DCc(\bD (\chi))
\end{equation}
is defined by the property
\[
(\exp_{\bD,K_n}(x),y)=\Tr_{K_n/\Qp}[x,\exp^*_{\bD^*(\chi),K_n}(y)],
\]
where 
\[
[\,,\,]\,:\, (\DCc(\bD)\otimes K_n)\times (\DCc (\bD (\chi))\otimes K_n)\lra K_n
\]
stands for the canonical duality.

\subsubsection{} If $\bD:=\DdagrigE (V)$ is the $(\varphi,\Gamma)$-module associated to
a crystalline $p$-adic representation $V,$ we recover the classical theory from \cite{blochkato}. 
Recall that the Bloch and Kato define the tangent space $t_V(K_n)$ of $V$ and the ``finite'' part of $H^1(K_n,V)$ as follows:
\begin{equation}
\nonumber
\begin{aligned}
&t_V(K_n):=\left( \Dc (V)/\Fil^0\Dc(V)\right )\otimes_{\Qp}K_n,\\
&H^1_{\rm f}(K_n,V):=\ker \left ( H^1(K_n, V)\rightarrow H^1(K_n,V\otimes_{\Qp}
\mathbf{B}_{\textrm{cris}})\right ),
\end{aligned}
\end{equation}
where $\mathbf{B}_{\textrm{cris}}$ is the ring of crystalline $p$-adic periods.
In the derived category of $E$-vector spaces, consider the complex
\[
\RG_f (K_n,V):= \left [ \Dc (V) \xrightarrow{f} \Dc (V) \oplus  t_V(K_n)\right ],
\]
where the terms are placed in degrees $0$ and $1$ and 
\[
f(x)=((1-\varphi)x, x\pmod{ \Fil^0\Dc (V)\otimes K_n}).
\]
There exists a canonical morphism
\[
\RG_f (K_n,V) \rightarrow \RG (K_n,V),
\]
which induces isomorphisms
\[
\bR^0 \mathbf\Gamma (K_n,V)\simeq H^0(K_n,V), \qquad \bR^1 \mathbf\Gamma (K_n,V)\simeq H_f^1(K_n,V)
\]
(see \cite[Section~2.1.3]{benoisextracris}).
Note that $\Dc (V)\simeq \DCc (\DdagrigE (V))$ by \cite{berger2008Ast}. Moreover we have the following result, which also follows directly from Berger's theory:

\begin{proposition} We have the following commutative diagram with canonical morphisms:
\[
\xymatrix{
\RG_f (K_n,V) \ar[d]_{\simeq} \ar[r]  &\RG (K_n,V)\ar[d]^{\simeq}\\
\RG_f (K_n,\DdagrigE (V))  \ar[r]  &\RG (K_n,\DdagrigE (V)).
}
\]
In particular,
\[
H^1_{\rm f}(K_n,V)\simeq H^1_{\rm f}(K_n,\DdagrigE (V)).
\]
\end{proposition}
\begin{proof} See \cite{Benois2011} and \cite{benoisextracris}.
\end{proof} 

\subsubsection{} 
If $\bD=\DdagrigE (V),$ the exact sequence (\ref{eqn:bloch-kato exact sequence for phi-gamma modules})  reads
 \[
0\rightarrow H^0(K_n,V) \rightarrow \Dc (V) \xrightarrow{f} \Dc (V) \oplus  t_V(K_n)\xrightarrow{\exp_{V,K_n}} H^1_{\rm f}(K_n,V)
\rightarrow 0.
\]
Here the coboundary map $\exp_{V,K_n}$ is the Bloch--Kato exponential map. 
 
In particular, if  $\Dc (V)^{\varphi=1}=0,$ one has an
isomorphism
\[
\exp_{V,K_n}\,:\, t_V(K_n) \rightarrow H^1_{\rm f}(K_n,V).
\]


\subsection{The large exponential map}
\label{subsec: large exp}
\subsubsection{} We review the main properties of the large exponential map in the 
setting of crystalline $(\varphi,\Gamma)$-modules. The results of this section can be either
proved by mimicking the arguments of Berger \cite{berger03}, or can be extracted from the general results
of Nakamura \cite{Nakamura2014JIMJ}, where the approach of Berger is extended to the case of general de Rham representations. However, in the general de Rham case, the relation of this construction to Iwasawa theory is not fully understood yet.

Note that, following the approach in \cite{benoisextracris},   we define the large exponential map as a morphism of complexes rather than that of cohomology groups.
This point of view has technical importance since we will study  Selmer complexes 
with local conditions provided by some large exponential maps in  Chapters~\ref{chapter_main_conjectures} and  \ref{chapter_main_conj_infinitesimal_deformation}. The general formalism of such complexes will be recalled in Section~\ref{subsec_124_21_11_1619} below.

\subsubsection{} Let $\bD$ be a crystalline $(\varphi,\Gamma)$-module over $\CR_E.$
Set
\[
\mathfrak D (\bD):= \cO_E [[\pi]]^{\psi=0}\otimes_{\cO_E}\DCc (\bD). 
\]

We define the maps 
\[
\widetilde{\Xi}_{\bD,n}\,:\,\mathfrak D (\bD)_{\Gamma_n} \rightarrow \DCc (\bD) \oplus 
t_{\bD}(K_n) 
\]
by the following explicit formulae. For any $\alpha (\pi)\in \mathfrak D (\bD)$,
\begin{equation}
\nonumber
\widetilde{\Xi}_{\bD,n}(\alpha(\pi))= \begin{cases}
p^{-n} \left (-\alpha (0), \underset{i=1}{\overset{n}\sum} (\sigma\otimes \varphi)^{-i}
\alpha (\zeta_{p^i}-1)\right )
& \text{if $n\geqslant 1$},\\
- ((1-p^{-1}\varphi^{-1})\alpha (0),0)
& \text{if $n=0.$}
\end{cases}
\end{equation}
These maps induce the morphisms of complexes
\begin{equation}
\nonumber
\begin{aligned}
&\widetilde{\mathbf{\Xi}}_{\bD,n}\,:\,\mathfrak D (\bD)_{\Gamma_n}[-1] \rightarrow \RG_f(K_n,\bD),\\
&\widetilde{\mathbf{\Xi}}_{\bD,n} (\alpha (\pi))=\bigl (0, \Xi_{\bD,n}(\alpha(\pi)) \mod{\Fil^0\DCc (\bD)\otimes_{\Qp}K_n}\bigr ).
\end{aligned}
\end{equation}
\subsubsection{}
Assume that the following condition holds:
\begin{itemize}
\item[\mylabel{item_LE}{\bf LE})] $\DCc (\bD)^{\varphi=1}=0.$ 
\end{itemize}
Then $\mathbf{R}^1\mathbf{\Gamma}_f(K_n,\bD) \simeq t_{\bD}(K_n)$, and $\widetilde{\mathbf{\Xi}}_{\bD,n}$ is homotopic to the map
\[
\Xi_{\bD,n}(\alpha(\pi))=\begin{cases}
p^{-n} \left (0, \underset{i=1}{\overset{n}\sum} (\sigma\otimes \varphi)^{-i}
\alpha (\zeta_{p^i}-1) + (1-\varphi)^{-1}\alpha (0) \right )
& \text{if $n\geqslant 1$},\\
\displaystyle \left (0,\left (\frac{1-p^{-1}\varphi^{-1}}{1-\varphi}\right ) \alpha (0)\right )
& \text{if $n=0.$}
\end{cases}
\]


\subsubsection{} 
\label{subsubsec_1143_2022_08_25_1053}
Fix an integer  $h\geqslant 1$ such that $\Fil^{-h}\DCc (\bD)=\DCc (\bD).$ 
Set
\begin{equation}
\label{eqn_PR_logarithmic_function_i}
\ell_i=i-\frac{\log (\gamma_1)}{\log \chi (\gamma_1)}, \qquad i\in \ZZ,
\end{equation}
where $\gamma_1$ is the fixed generator of $\Gamma_1.$ We remark that $\ell_i\in \CH (\Gamma).$

For any $\alpha(X)\in \mathfrak D (\bD),$
the equation 
\begin{equation}
\label{eqn: equation (phi-1)F=alpha}
(\varphi-1)(F)=\alpha (\pi)-\underset{i=0}{\overset{h-1}\sum}\frac{\partial^i\alpha (0)}{i!}\log^i(1+\pi)
\end{equation}
has a unique solution $F\in \CH  \otimes_E \DCc (\bD)$ (cf. \cite{perrinriou94}, Lemme~2.2.3). Let us set
\begin{equation}
\nonumber
\Exp_{\bD,h}(\alpha):= \frac{\log \chi(\gamma_1)}{p} \left (\underset{i=0}{\overset{h-1}\prod}\ell_i \right ) \circ F(\pi).
\end{equation}
The arguments of Berger show that 
$\Exp_{\bD,h}(\alpha)\in \bD^{\psi=1}.$ This provides us with  a map
\begin{equation}
\nonumber
\bExp_{\bD,h}\,:\,\mathfrak D(\bD)[-1] \rightarrow \RG_\Iw (\bD).
\end{equation}
We will denote by
\begin{equation}
\nonumber
\Exp_{\bD,h}\,:\,\mathfrak D(\bD) \rightarrow H^1_\Iw (\bD)
\end{equation}
the induced map on the first cohomology. 

\begin{theorem}
\label{thm:large exponential for phi-Gamma modules}
The map $\bExp_{\bD,h}$ satisfies the following properties: 
\begin{itemize}
\item[1)]{}For any $n\geqslant 0,$ one has a commutative diagram 
\begin{equation}
\label{eqn:Perrin-Riou commutative square_1st}
\begin{aligned}
\xymatrix{
\mathfrak D (\bD)   \ar[rr]^-{\bExp_{\bD,h}} 
\ar[d]_{\widetilde{\mathbf{\Xi}}_{V,n}}& & \RG_{\Iw}(\bD)
\ar[d]^{\pr_n}\\
\RG_f(K_n,\bD) \ar[rr]_-{(h-1)!\bexp_{\bD, K_n}} && \RG (K_n, \bD).
}
\end{aligned}
\end{equation}

\item[ii)] Let us denote by $d_{-1}:=\varepsilon^{-1}\otimes t$ the canonical generator of 
$\DCc (\CR_E(\chi^{-1}))$. Then 
\begin{equation}
\nonumber
\bExp_{\bD(\chi),h+1}=-\Tw_1 \circ \bExp_{\bD,h} \circ (\partial \otimes d_{-1}).
\end{equation}

\item[iii)] One has
\begin{equation}
\nonumber 
\bExp_{\bD,h+1}=\ell_{h}\bExp_{\bD,h},
\end{equation}
where 
\begin{equation}
\label{eqn:definition of ell}
\ell_h=h-\frac{\log (\gamma_1)}{\log \chi (\gamma_1)}.
\end{equation}
\end{itemize}
\end{theorem}
\begin{proof} The theorem is a reformulation of the results of Nakamura \cite{Nakamura2014JIMJ}
for  the case of crystalline $(\varphi,\Gamma)$-modules. Note that for crystalline $p$-adic representations it follows from the results of Berger \cite{berger03}. 
\end{proof}

\subsubsection{} Using Theorem~\ref{thm:large exponential for phi-Gamma modules} ii), one can 
define   $\Exp_{V,h}$ for all integers $h$ such that $\Fil^{-h}\DCc (\bD)=\DCc (\bD).$

\begin{remark} If \eqref{item_LE} holds, the diagram \eqref{eqn:Perrin-Riou commutative square_1st} induces a commutative diagram 
\[
\begin{aligned}
\xymatrix{
\mathfrak D (\bD)   \ar[rr]^-{\Exp_{\bD,h}} 
\ar[d]_{\widetilde{\mathbf{\Xi}}_{V,n}}& & H^1_{\Iw}(\bD)
\ar[d]^{\pr_n}\\
t_{\bD}(K_n) \ar[rr]_-{(h-1)!\bexp_{\bD, K_n}} && H^1 (K_n, \bD),
}
\end{aligned}
\]
where $\Exp_{\bD,h}$ is the usual large exponential map for $(\varphi,\Gamma)$-modules constructed in \cite{berger03,Nakamura2014JIMJ}.
\end{remark}

\subsection{Perrin-Riou's  large exponential map}
\subsubsection{}
\label{subsubsec_1151_2028_08_25_1200}
If $\bD$ is the  $(\varphi,\Gamma)$-module associated to a crystalline representation, the map $\Exp_{\bD,h}$ coincides with Perrin-Riou's large exponential map $\Exp_{V,h}$ \cite{perrinriou94}
(see \cite{berger03}). For the reader's convenience, we record below some of the previous statements in that particular case. 

Let $V$ be a crystalline representation of $G_{\Qp}$
with coefficients in $E.$ 
Set
\[
\mathfrak D (V):= \cO_E [[\pi]]^{\psi=0}\otimes_{\cO_E}\Dc (V). 
\]
Define the maps 
\[
\widetilde{\Xi}_{V,n}\,:\,\mathfrak D (V) \rightarrow \Dc (V)\oplus t_V(K_n)
\]
by
\begin{equation}
\nonumber
\widetilde\Xi_{V,n}(\alpha(\pi))= \begin{cases}
p^{-n} \left (-\alpha (0), \underset{i=1}{\overset{n}\sum} (\sigma\otimes \varphi)^{-i}
\alpha (\zeta_{p^i}-1)\right )
& \text{if $n\geqslant 1$},\\
- ((1-p^{-1}\varphi^{-1})\alpha (0),0)
& \text{if $n=0.$}
\end{cases}
\end{equation}

For any $h\geqslant 1$ such that $\Fil^{-h}\Dc (V)=\Dc (V),$ one has a  $\Gamma$-equivariant map
\begin{equation}
\nonumber
\Exp_{V,h}\,:\, 
\mathfrak D (V)
 \lra \CH (\Gamma)\otimes_{\LL} H^1_\Iw (\Qp,V).
\end{equation}
satisfying the following properties:
\begin{itemize}
\item[i)]{}For any $n\geqslant 0,$ one has a commutative diagram 
\begin{equation}
\label{eqn:Perrin-Riou commutative square}
\begin{aligned}
\xymatrix{
\mathfrak D (V)   \ar[rr]^-{\Exp_{V,h}} 
\ar[d]_{\widetilde{\Xi}_{V,n}}& & \CH (\Gamma)\otimes_{\Lambda} H^1_{\Iw}(\Qp,V)
\ar[d]^{\pr_n}\\
\Dc (V)\oplus t_V(K_n) \ar[rr]^(.53){(h-1)!\exp_{V, K_n}} && H^1(K_n, V).
}
\end{aligned}
\end{equation}

\item[ii)] Let us denote by $d_{-1}:=\varepsilon^{-1}\otimes t$ the canonical generator of $\Dc (\Qp (-1))$. Then 
\begin{equation}
\nonumber
\Exp_{V(1),h+1}=-\Tw_1 \circ \Exp_{V,h} \circ (\partial \otimes d_{-1}).
\end{equation}

\item[iii)] One has
\begin{equation}
\label{eqn: relation Exp for h and h+1}
\Exp_{V,h+1}=\ell_{h}\Exp_{V,h},
\end{equation}
where $\ell_h$ is given as in \eqref{eqn:definition of ell}.
\end{itemize}

\subsubsection{}
\label{subsubsection: isotypic components}
 
Recall that we set  $G_n=\Gal (K_n/K).$ Shapiro's lemma gives an isomorphism of $E[G_n]$-modules
\[
H^1(K_n, V)\stackrel{\sim}{\lra} H^1(\Qp, V\otimes_E E[G_n]^{\iota}).
\]
On taking the $\rho$-isotypic components, we obtain isomorphisms
\begin{equation}
\nonumber
H^1(K_n, V)^{(\rho)}  \stackrel{\sim}{\lra}       H^1(\Qp, V(\rho^{-1})),\qquad\qquad \rho \in X(G_n).
\end{equation}

\subsubsection{}
Assume that $V$ satisfies \eqref{item_LE}.
Then in the diagram  \eqref{eqn:Perrin-Riou commutative square}, $\widetilde{\mathbf{\Xi}}_{V,n}$ 
can be replaced by the map
\[
\Xi_{V,n}\,:\,\mathfrak D(V) \lra t_V(K_n)
\]
which is defined as
\[
\Xi_{V,n}(\alpha(\pi))=\begin{cases}
p^{-n} \left (0, \underset{i=1}{\overset{n}\sum} (\sigma\otimes \varphi)^{-i}
\alpha (\zeta_{p^i}-1) + (1-\varphi)^{-1}\alpha (0) \right )
& \text{if $n\geqslant 1$},\\
\displaystyle \left (0,\left (\frac{1-p^{-1}\varphi^{-1}}{1-\varphi}\right ) \alpha (0)\right )
& \text{if $n=0.$}
\end{cases}
\]

\begin{remark}  
Assume that  $V$ is the $p$-adic realization of a pure motive $\mathcal M/\QQ$ of motivic weight $w:=\mathrm{wt} (\mathcal M)$ which has good reduction at $p$. It follows from the definition of  Hodge structures that the tangent space $t_V (K_n)$ is nonzero if and only if $w\leqslant -1.$ On the other hand, the condition $w\leqslant -1$ implies that   $\Dc (V)^{\varphi=1}=0$ (see \cite{KatzMessing}, see also  discussion in \cite{benois-selmer}, \S3.2.1).  Therefore,  $1-\varphi$ is an isomorphism on $\Dc (V),$ and the maps $\Xi_{V,n}$ are well defined if the representation $V$ has a motivic origin.
We also remark that the operator  $1-p^{-1}\varphi^{-1}$ can have a nontrivial kernel and therefore the map $\Xi_{V,0}$ is not necessarily injective. The non-vanishing of $\Dc (V)^{\varphi=p^{-1}}$ is closely related to the phenomenon of exceptional zeros of $p$-adic $L$-functions attached to $\mathcal M.$ We refer the reader to \cite{benoisextracris,benois-selmer} for further details. 
\end{remark}

We record the  following formula giving the $\rho$-isotypic component of the map $\Xi_{V,n}$:
\begin{lemma}
\label{lemma from BenoisBerger2008}
Assume that $\rho \in X(G_n)$ is a primitive character. Then
for any $\alpha =f(\pi)\otimes d \in \mathfrak D (V)$
\begin{equation}
\nonumber
\Xi_{V,n}^{(\rho)}(\alpha)=
\begin{cases} e_{\rho}(f(\zeta_{p^n}-1))\otimes p^{-n}\varphi^{-n}(d)
& \textrm{if $n\geqslant 1$},\\
(1-p^{-1}\varphi^{-1})(1-\varphi)^{-1} (d)
 &\textrm{if $n=0.$}
 \end{cases}
 \end{equation}
\end{lemma}
\begin{proof} See \cite[Lemma~4.10]{BenoisBerger2008}.
\end{proof}

\subsubsection{} 
We will identify $\mathfrak D(V)$ with $\LL\otimes_{\Zp}\Dc (V)$ 
using the canonical isomorphism $\Zp [[X]]^{\psi=0}\simeq \LL.$ 
The  duality $\Dc (V^*(1))\times \Dc (V)\xrightarrow{[\,,\,]} \Qp$ extends 
to a $\LL[1/p]$-bilinear perfect pairing
\[
\star \,:\,\mathfrak D(V^*(1)) \times \mathfrak D(V) \lra \LL[1/p].
\]

The explicit reciprocity law for the large exponential map 
(cf. \cite{benois00,colmez98,berger04}) states that
\begin{equation} 
\label{eqn:explicit reciprocity}
\left <\Exp_{V^*(1),1-h}(x),c\circ \Exp_{V,h}(y)\right >=(-1)^h (x\star y),
\qquad x\in \mathfrak D(V^*(1)),\quad y\in \mathfrak D(V).
\end{equation}
It means that the large exponential map $\Exp_{V,h}$  interpolates also the inverses of the dual
exponential maps $\exp_{V(j)}^*$  for $j\ll 0.$ We record it in the following form:
\begin{equation}
\label{eqn:interpolation at negative twists}
\Exp_{V^*(1),1-h}(x)=
\frac{(-1)^h}{(h-1)!} (\exp_{V^*(1)}^*)^{-1} ((1-p^{-1}\varphi^{-1})(1-\varphi)^{-1} x(0)).
\end{equation}

Let $\mathscr{K}(\Gamma)$ denote the total ring of fractions of $\CH (\Gamma).$ Extending scalars from $\LL$ to $\mathscr{K}(\Gamma)$, we obtain an isomorphism
\[
\Exp_{V,h}\,:\, \mathscr{K}(\Gamma) \otimes_{\LL}\Dc (V)
\lra \mathscr{K}(\Gamma) \otimes_{\LL} H^1_\Iw (\Qp, V)\,,
\]
and its inverse 
\[
\mathrm{Log}_{V,h}\,:\, H^1_\Iw (\Qp, V) \lra \mathscr{K}(\Gamma) \otimes_{\LL}\Dc (V)\,.
\]
Perrin-Riou's approach to the construction of $p$-adic $L$-functions postulates, roughly speaking, that when $\bz\in H^1_\Iw (\Qp, V^*(1))$ is the $p$-local image of the initial term of an Euler system for a global geometric Galois representation $V^*(1)$ and $y$ is a generator of a regular submodule 
$D$ of $\Dc (V),$ the element $\mathrm{Log}_{V^*(1),h} (\bz)\star y$ of $\mathscr{K}(\Gamma)$  coincides with the $p$-adic $L$-function attached to the data $(V,D)$; cf. \cite{perrinriou95,benoisextracris}. 
The explicit reciprocity law \eqref{eqn:explicit reciprocity} can be  recast in the form
\[
\left <\bz,c\circ \Exp_{V,h}(y)\right >=(-1)^h (\mathrm{Log}_{V^*(1),h} (\bz)\star y),\qquad \bz\in H^1_\Iw (\Qp, V^*(1)) ,\quad y\in \mathscr{K}(\Gamma) \otimes_{\LL}\Dc (V)\,,
\]
and therefore the $p$-adic $L$-function can be alternatively constructed using the map $\Exp_{V,h}.$ In the present work, we will take the latter approach and use the map $\Exp_{V,h}$ as it is better adapted to the study of Selmer complexes arising in Perrin-Riou's theory.

\subsection{The case of one-dimensional representations} 
\label{subsec:one-dimensional representations}
\subsubsection{} 
In this section, we specialize  Perrin-Riou's theory  to the case of one-dimensional crystalline representations. These representations are obtained by twisting unramified characters by a power of the cyclotomic character. The large exponential map for these representations is closely related to the theory of Coleman and Yager and  can be constructed on the integral level. 

\subsubsection{} 
Let $\nu \,:\,G_{\Qp} \rightarrow \cO_E^*$ be a continuous unramified character
and let $E(\nu)$ denote the associated $p$-adic representation. Then $\cO_E(\nu)$ is a canonical 
lattice in $E(\nu).$ The crystalline module $\Dc (E(\nu))$ contains the  canonical lattice 
\[
\Dc (\cO_E(\nu)):= (\cO_E (\nu)\otimes \widehat{\Zp^{\mathrm ur}})^{G_{\Qp}},
\]
where $\widehat{\Zp^{\mathrm ur}}$ denotes the completion of  the maximal unramified extension of 
$\Zp.$ The Frobenius operator $\varphi$ acts on $\Dc (E(\nu))$ as the multiplication by
$\alpha:=\nu (\Fr_p),$ where $\Fr_p$ denotes the {\it geometric} Frobenius. 

For any $j\in \ZZ,$ we denote by $E(\nu \chi^j)$ the $j$th cyclotomic twist of $E(\nu).$
The  module $\Dc (E(\nu \chi^j))$ has the canonical lattice 
\[
\Dc (\cO_E(\nu \chi^j)):=\Dc (\cO_E(\nu))\otimes {d[j]},
\]
where $d[j]$ denotes the canonical generator of $\Dc (\Qp (j)).$

\subsubsection{} 
\label{subsubsec_1163_2022_04_29}
Let us set
$\mathfrak D(\cO_E(\nu \chi^j)):=\cO_E[[\pi]]^{\psi=0}\otimes_{\Zp} \Dc(\cO_E(\nu \chi^j)).$ 
There exist isomorphisms of free $\Lambda$-modules of rank one  
\begin{equation}
\label{eqn: integral exponentials}
\Exp_{\cO_E(\nu\chi^j),j}\,:\, \mathfrak D(\cO_E(\nu \chi^j))\xrightarrow{\sim} H^1_\Iw (\Qp, \cO_E(\nu \chi^j) )
\end{equation}
which coincide with the corresponding  Perrin-Riou's exponential maps after extension of scalars\footnote{Here we use the fact that $\nu$ is a nontrivial character.}.
Let
\begin{equation}
\nonumber
\Log_{\cO_E(\nu\chi^j),j}\,:\,H^1_\Iw (\Qp, \cO_E(\nu \chi^j) )\lra  \mathfrak D(\cO_E(\nu \chi^j))
\end{equation} 
denote its inverse. 
We record two formulae, which follow from the general interpolation properties 
of the large exponential map:
\begin{equation}
\label{eqn:specialization of PR formulae}
\begin{aligned}
&\pr_0\circ \Exp_{\cO_E(\nu\chi^j),j}(x)= (j-1)! \exp_{E(\nu \chi^j)}
\left ((1-\alpha^{-1} p^{j-1})(1-\alpha p^{-j})^{-1}x\right ), && j\geqslant 1\\
&\pr_0\circ \Exp_{\cO_E(\nu \chi^j),j}(x)= \frac{(-1)^{j-1}}{(-j)!} (\exp^*_{E(\nu \chi^j)})^{-1} \left ((1-\alpha^{-1}p^{j-1})(1-\alpha p^{-j})^{-1}x\right ), && j\leqslant 0.
\end{aligned}
\end{equation}
\begin{lemma} 
\label{lemma: computation tamagawa}
Let $\nu \,:\,G_{\Qp} \rightarrow \cO_E^*$ be a nontrivial
unramified character. Then:

\item[i)] For each $j\geqslant 1,$ one has
\[
\left [H^1(\Qp, \cO_E(\nu\chi^j)): \exp \bigl (\Dc (\cO_E(\nu\chi^j))\bigr )\right ]= 
\left \vert (j-1)!\right \vert \cdot \left \vert
\frac{1-\alpha^{-1} p^{j-1}}{1-\alpha p^{-j}}\right \vert
\cdot \# H^0(\Qp, E/\cO_E (\nu^{-1}\chi^{1-j})).
\]

\item[ii)]  For each $j\leqslant 0,$ one has
\[
\left [H^1(\Qp, \cO_E(\nu\chi^j)): \left (\exp^*\right )^{-1} \bigl (\Dc (\cO_E(\nu\chi^j))\bigr )\right ]= 
\left \vert (-j)!\right \vert^{-1}  \cdot \left \vert
\frac{1-\alpha^{-1} p^{j-1}}{1-\alpha p^{-j}}\right \vert
\cdot \# H^0(\Qp, E/\cO_E (\nu^{-1}\chi^{1-j})).
\]
\end{lemma}
\begin{proof} These formulae follow immediately from \eqref{eqn:specialization of PR formulae} together with \eqref{eqn:hochschild-serre for representations}.
\end{proof}

\subsubsection{} Fix a generator $d$ of $\Dc (\cO_E(\nu \chi^j))$ and denote by $d^*\in \Dc (\cO_E(\nu^{-1} \chi^{1-j}))$ its dual.  Then we have an  integral version of the map
$\mathfrak L_{V,1-h,d},$ namely:
\begin{equation}
\nonumber 
\Log_{\cO_E(\nu\chi^j),1-j,d^*} \,:\,H^1_{\Iw} (\Qp,\cO_E(\nu^{-1}\chi^{1-j}) )  \lra \LL,
\qquad \Log_{\cO_E(\nu\chi^j),j,d^*}(x):=\left <x, c\circ \Exp_{\cO_E(\nu\chi^{j}),j} 
({d}\otimes (1+\pi))^\iota \right >.
\end{equation}
The explicit reciprocity law says:
\begin{equation}
\nonumber
\mathrm{Log}_{\cO_E(\nu^{-1}\chi^{1-j}),1-j}(x)=(-1)^j\Log_{\cO_E(\nu\chi^j),1-j,d^*}(x)  d^*, \qquad
x\in H^1_{\Iw} (\Qp,\cO_E(\nu^{-1}\chi^{1-j}) ).
\end{equation}

\subsection{Tamagawa numbers} 
\label{subsec_Tamagawa_numbers}
\subsubsection{}
For a crystalline representation $V$ of $G_{\Qp}$ with coefficients in $E$, consider the Bloch--Kato exact sequence
\begin{equation}
\label{eqn: Bloch-Kato sequence for V(j)}
0\lra H^0(\Qp, V)\lra \Dc (V) \xrightarrow{(1-\varphi, \pr)} 
\Dc (V)\oplus t_{V} \xrightarrow{\exp} H^1_{\rm f}(\Qp,V)\lra 0, 
\end{equation}
where $t_V:=\Dc (V)/{\rm Fil}^0\Dc (V)$ denotes the tangent space of $V$. Let $T\subset V$ be a $G_{\QQ_p}$-stable $\cO_E$-lattice and let us put 
\[
\begin{aligned}
&L_{\rm f}(V):= {\det}_{E} H^0(\Qp,V)\otimes {\det}_{E}^{-1}H^1(\Qp,V)\\
&L_{\rm f}(T):= {\det}_{\cO_E} H^0(\Qp,T)\otimes {\det}_{\cO_E}^{-1}H^1(\Qp,T)\,.
\end{aligned}
\]
The exact sequence \eqref{eqn: Bloch-Kato sequence for V(j)} gives rise to an isomorphism
\[
\alpha_V\,:\, L_{\rm f}(V) \simeq {\det}^{-1}_E t_V\,.
\]
Let us fix a nonzero element  $\omega \in {\det}_E\, t_V$. The \emph{Tamagawa number} associated to the pair $(T,\omega)$ is defined on setting 
\[
\Tam^0_{p,\omega} (T):= \left \vert  a  \right \vert^{-1}, \qquad
\textrm{where $\alpha_V(L_{\rm f}(T))=a\omega^{-1}$}. 
\] 
If $\Dc (V)^{\varphi=1}=0,$ then  $H^0(\Qp, V)=0,$ and we set
\[
\Tam_{p,\omega}(T):= \left [ H^1_{\rm f}(\Qp, T) : \exp (L)\right ]\,,
\]
where $L\subset t_V$ is an $\cO_E$-lattice such that ${\det}_{\cO_E}L=\cO_E \omega$\,.
It follows directly from these definitions that
\[
\Tam^0_{p,\omega}(T)=\bigl \vert {\det}_{E}(1-\varphi \,\vert \, \Dc (V))\bigr  \vert 
\cdot \Tam_{p,\omega}(T)\,.
\]

\subsubsection{} 
\label{subsubsec_1172_2023_07_07_0859}
Let $T=\cO_E(\nu \chi^j),$ where $j\in \ZZ.$  We will denote by $\Tam_{p,\mathrm{can}}(\cO_E(\nu \chi^j))$ the Tamagawa number  $\Tam_{p,\omega}(\cO_E(\nu \chi^j)),$ where
$\omega$ is an arbitrary generator of $\Dc (\cO_E(\nu \chi^j)).$
If $N\subset \Dc (E(\nu \chi^j))$ is the $\cO_E$-lattice generated
by some $\eta \in \Dc (\cO_E(\nu \chi^j)),$ we set
\begin{equation}
\label{eqn: definition of the period Omega}
\Omega_p (\cO_E(\nu\chi^j),N)=\vert b\vert, \quad \text{\rm where $\Dc (\cO_E(\nu \chi^j))=bN.$}
\end{equation}
We remark that 
\begin{equation}
\label{eqn: comparision of Tamagawa numbers}
\Tam^0_{p,\eta}(\cO_E(\nu \chi^j))=\Tam^0_{p,\mathrm{can}}(\cO_E(\nu \chi^j))\cdot \Omega_p (\cO_E(\nu \chi^j),N).
\end{equation}
We deduce from Lemma~\ref{lemma: computation tamagawa} the following well-known formulae for the Tamagawa numbers of one-dimensional representations  (cf. \cite{perrinriou94}):
\begin{equation}
\label{eqn: formula for Tamagawa numbers}
\Tam^0_{p,\eta}(\cO_E(\nu \chi^j))=
\left \vert (j-1)!\right \vert 
\cdot \# H^0(\Qp, E/\cO_E (\nu^{-1}\chi^{1-j}))
\cdot \Omega_p (\cO_E(\nu \chi^j),N)\,.
\end{equation}

\subsection{Exponential map over affinoid algebras}
\label{subsec_exp_over_affinoid_algebras}
\subsubsection{}
In this section, we record  the ``family-version'' of the constructions of  Section~\ref{subsec:one-dimensional representations}. 

For any continuous character 
$\delta \,:\,\Qp^\times \rightarrow \cO_{\mathcal X}^\times$ let  $\bD_\delta=\CR_{\mathcal X}e_\delta$ denote  the $(\varphi,\Gamma)$-module of rank one over  $\CR_{\mathcal X}$, generated by $e_\delta$ and such that
\begin{equation}
\nonumber
\varphi (e_\delta)=\delta (p)e_\delta \quad, \quad 
\gamma (e_\delta)= \delta (\chi (\gamma)) e_\delta \qquad \qquad \gamma \in \Gamma.
\end{equation}
For any $x\in \mathcal X(E)$, we denote by  $\bD_{\delta,x}$  the  specialization 
of $\bD_\delta$ at $x$:
\[
\bD_{\delta,x}:=\bD\otimes_{\cO_{\mathcal X},x}E.
\] 
Let $\CR_{\mathcal X}^+\subset \CR_{\mathcal X}$ denote the subring of analytic functions on
the open unit disc $\vert X\vert_p<1.$  Set $\bD_\delta^+:=\CR_A^+e_\delta.$
We say that $\bD_\delta$ is of Hodge--Tate weight $m\in \ZZ$  if 
\begin{equation}
\nonumber 
\delta (u)=u^{m}, \qquad\qquad \forall u\in \Zp^\times.
\end{equation}
If $\bD_\delta$ is of Hodge--Tate weight $m,$ then  $\DCc (\bD_\delta):=\left (\bD_{\delta}[1/t]\right )^{\Gamma}$ is the free $\cO_{\mathcal X}$-module of rank one generated 
by $d_\delta=t^{-m}e_\delta,$ where $t=\log (1+\pi)$ is the ``$p$-adic $2\pi i$'' of Fontaine. 
Moreover $\DCc (\bD_\delta)$  has the unique filtration break at $-m,$ and $\varphi$ acts on $\DCc (\bD_\delta)$ as the multiplication by $p^{-m}\delta (p)$ map. Set
$\mathfrak D(\bD_\delta)=\cO_E[[X]]^{\psi=0}\otimes_E \DCc (\bD_\delta).$

\subsubsection{} First assume that $m\geqslant 1.$ Mimicking the constructions in Section~\ref{subsec: large exp}, for any $\alpha (\pi)\in \mathfrak D(\bD_\delta)=\cO_E[[\pi]]^{\psi=0}$
we take a solution of the equation
\[
(\varphi-1)F=\alpha (\pi)-\underset{i=0}{\overset{h-1}\sum} \frac{\partial^i\alpha (0)}{i!}
\log^i(1+\pi).
\]
Then 
\begin{equation}
\nonumber
\left (\underset{i=0}{\overset{h-1} \prod}\right )\ell_i \circ F(\pi)=
(-1)^{h-1}t^h\partial^h(F)\in \left (t^{h-m} \bD_\delta^+ \right )^{\psi=1}\qquad \forall h\geqslant m.
\end{equation}
This allows to define the map
\[
\bExp_{\bD_\delta,h}\,:\, \mathfrak D(\bD_\delta) [-1]\rightarrow \RG_\Iw (\bD_\delta),\qquad
\bExp_{\bD_\delta,h}(\alpha)=\frac{\chi (\gamma_1)}{p} \left (\underset{i=0}{\overset{h-1} \prod}\ell_i \right )\circ F(\pi).
\]
The following proposition is clear from construction:

\begin{proposition} 
\label{prop: large exponential map}
The following assertions hold true:

\item{i)} The map $\bExp_{\bD_\delta,h}$ interpolates the  exponentials 
$\bExp_{\bD_{\delta,x}}$ for $x\in \mathcal X.$

\item{ii)} For any  $h\geqslant m$ one has
\[
\bExp_{\bD_\delta, h+1}=\ell_h \bExp_{\bD_\delta, h}.
\]

\item{iii)} For any $h\geqslant m$ one has
\[
\bExp_{\bD_{\delta\chi}, h+1}=-\Tw_1 \circ \bExp_{\bD_{\delta}}\circ (\partial \otimes d_1).
\]

\item{iv)} Using property iii), one can extend the definition of $\bExp_{\bD_\delta,h}$
to $(\varphi,\Gamma)$-modules $\bD_\delta$ of arbitrary Hodge--Tate weight $m\in\ZZ$
and $h\geqslant m.$
\end{proposition}

\section{Selmer complexes}
\label{sec_Sel_complex_and_duality}


\subsection{Local conditions}
\subsubsection{Unramified local conditions}
\label{subsubsec_1211_2023_07_06_1328}
In this section, we revisit the construction of Selmer complexes following \cite{nekovar06, benoisheights} and introduce various Selmer groups that are utilized in this paper. 

Let $p$ be a fixed prime. Fix a finite set of primes  $S$ containing $p$ and  denote by $G_{\QQ,S}$ the Galois group of the maximal algebraic extension of $\QQ$ unramified outside $S\cup\{\infty\}$. For any topological $G_{\QQ,S}$-module
$M$ we denote by $H^i(M):=H^i(G_{\QQ,S},M)$ the continuous cohomology
of $G_{\QQ,S}$ with coefficients in $M$.

We maintain the notation of Section~\ref{sec_exponential_maps}. In particular, we consider an affinoid $\mathcal X$ over $E.$  Let $\cO_{\mathcal X}^\circ$ denote the unit ball (with respect to the Gauss norm) of $\cO_{\mathcal X}$. Fix a $p$-adic representation $V$  of $G_{\QQ,S}$ with coefficients in $\cO_{\mathcal X}$. We will assume that there exists a free $\cO_{\mathcal X}^\circ$-lattice $T$ in $V$ such that $V=T[1/p]$. 

Fix $\ell \in S\setminus \{p\}.$ Let $I_\ell$ denote the inertia subgroup of $G_{\QQ_\ell}.$
 We will work with the following versions of unramified local conditions:

\begin{itemize}
\item{\it Nekov\'a\v r's unramified local conditions.} For $M=V,T,$ or $V/T,$ consider the complex
\[
C_{\ur}^\bullet (G_{\QQ_\ell}, M)= \left [ M^{I_\ell}\xrightarrow{\Fr_\ell-1} M^{I_\ell}\right ],
\]
where the terms are concentrated in degrees $0$ and $1.$ Let $\RG_{\ur}(\QQ_\ell, M)$ denote  the 
corresponding complex in the derived categories  $\mathscr D_{\textrm{ft}}(\cO_{\mathcal X})$,
$\mathscr D_{\textrm{ft}}(\cO^\circ_{\mathcal X})$ and $\mathscr D_{\textrm{coft}}(\cO^\circ_{\mathcal X})$ respectively. This complex is equipped with a natural map 
\[
g_\ell\,:\, C_{\ur}^\bullet (G_{\QQ_\ell}, M) \lra C^\bullet (G_{\QQ_\ell}, M).
\]
One has
\[
\bR^0\mathbf \Gamma_{\ur}(\QQ_\ell, M)\simeq H^0(\QQ_\ell,M), \qquad \bR^1\mathbf \Gamma_{\ur}(\QQ_\ell, M)\simeq
\ker \left (H^1(\QQ_\ell, M) \rightarrow H^1(I_\ell,M) \right ).
\]
On the level of cohomology, the map $g_\ell$ induces the isomorphism $\bR^0\mathbf \Gamma_{\ur}(\QQ_\ell, M)\simeq H^0(\QQ_\ell,M)$ and the natural inclusion $\bR^1\mathbf \Gamma_{\ur}(\QQ_\ell, M) \hookrightarrow H^1(\QQ_\ell,M).$

\item{\it Bloch--Kato's unramified local conditions $H^1_{\rm f}.$}
Here $\mathcal X=\Spm (E),$ 
$\cO_{\mathcal X}=E$ and $\cO^\circ_{\mathcal X}=\cO_E.$
 Bloch--Kato's local conditions $H^1_{\rm f}$   are defined for the  first cohomology group $H^1(\QQ_\ell, M):$ 
\begin{equation}
\nonumber
H^1_{\rm f}(\QQ_\ell,M):= 
\begin{cases} \ker \left (H^1(\QQ_\ell, V) \lra H^1(I_\ell,V)\right ), &\textrm{if $M=V$},\\
\mathrm{im} \left (H^1_{\rm f}(\QQ_\ell, V) \lra H^1(\QQ_\ell,V/T)\right ),  &\textrm{if $M=V/T$},\\
\textrm{$\ker\left(H^1(\QQ_\ell, T) \lra \dfrac{H^1(\QQ_\ell,V)}{H^1_{\rm f}(\QQ_\ell,V)}\right)$},  &\textrm{if $M=T$}.
\end{cases}
\end{equation}
\end{itemize}
We tautologically have $ H^1_{\rm f}(\QQ_\ell,V)=\bR^1\mathbf \Gamma_{\ur}(\QQ_\ell, V)$. However
$H^1_{\rm f}(\QQ_\ell,M)$ need not agree with $\bR^1\mathbf \Gamma_{\ur}(\QQ_\ell, M)$ for $M=T, V/T$: When $\cX={\rm Spm}(E)$ is a point, the difference is accounted by Tamagawa numbers.

\subsubsection{Greenberg's local conditions}
Let $\ell=p.$ We will work with Greenberg-style local conditions in the following settings:

\begin{itemize}
\item{\it Classical Greenberg's conditions.} $M=V,$ $T$ or $V/T,$ and $Z$ is a $G_{\Qp}$-submodule of $M.$ Greenberg's local conditions are defined  as the natural map
\[
g_p\,:\, C^\bullet (G_{\Qp},Z) \rightarrow C^\bullet (G_{\Qp},M).
\]
On cohomology, $g_p$ induces natural maps
\[
H^i(\Qp,Z) \rightarrow H^i(\Qp, M).
\]

\item{\it Greenberg's conditions arising from $(\varphi,\Gamma)$-modules.} Here $M=V.$   Let $C^\bullet_{\varphi,\gamma}(\DdagrigX (V))$
denote the Fontaine--Herr complex of $\DdagrigX (V).$ Recall that there exists a complex $K^\bullet_p(V)$ of $\cO_{\mathcal X}$-modules 
together with quasi-isomorphisms of complexes
\begin{equation}
\label{eqn_Fontaine_Herr_enhanced}
C^{\bullet}(G_{\QQ_p},V) \xrightarrow{\xi_p} K^\bullet_p(V)
\xleftarrow{\alpha_p}  C^\bullet_{\varphi,\gamma}(\DdagrigX (V))\,
\end{equation}
(cf. \cite{benoisheights}, Section~2.5). Note that these quasi-isomorphisms
induce  the  quasi-isomorphism (\ref{eqn:quasi-iso between galois and phi-gamma cohomology})
between  $\RG (\Qp,\DdagrigX (V))$ and $\RG (\Qp, V).$ If $Z^\bullet$ is a complex 
of $\cO_{\mathcal X}$-modules equipped with a morphism $Z^\bullet \rightarrow  \RG (\Qp,\DdagrigX (V)),$ 
it defines a local condition
\[
g_p\,:\, Z^\bullet \rightarrow K^{\bullet}_p(V).
\]
It gives rise to a map $Z^\bullet \rightarrow \RG (\Qp,V)$ in the derived category.

Notice the following important case. Let $\bD \subset \DdagrigX (V)$ be a $(\varphi,\Gamma)$-submodule (not necessarily saturated).  The choice of $\bD$ then defines the local condition
\[
g_p\,:\, C^\bullet_{\varphi ,\gamma}(\bD) \stackrel{\iota_{\bD}}{\lra} 
C^\bullet_{\varphi,\gamma}(\DdagrigX (V))
\xrightarrow{\alpha_p} K^\bullet_p(V).
\]
On the level of cohomology, $g_p$ induces natural maps
\[
H^i(\Qp,\bD) \lra H^i(\Qp,\DdagrigX (V))\simeq H^i(\Qp,V).
\]

\item{\it Greenberg's conditions on $H^1$.}
The conditions introduced here
are defined only on the level of  the first cohomology group $H^1(\Qp,M).$
Assume that  $\mathcal X=\Spm (E),$ 
$\cO_{\mathcal X}=E$ and $\cO^\circ_{\mathcal X}=\cO_E.$ 
Let $\bD$ be a submodule of $\DdagrigE (V).$
Whenever the  map $H^1(\Qp,\bD)\rightarrow H^1(\Qp,V)$ is injective, we denote by $H^1_{\bD}(\Qp,V)$ 
its image and set
\begin{equation}
\nonumber
H^1_{\bD}(\Qp,M):= 
\begin{cases}
\mathrm{im} \left (H^1_{\bD}(\Qp, V) \rightarrow H^1(\Qp,V/T)\right ),  &\textrm{if $M=V/T$},\\
\textrm{$\ker\left(H^1(\QQ_\ell, T) \lra \dfrac{H^1(\QQ_\ell,V)}{H^1_{\bD}(\QQ_\ell,V)}\right)$},  &\textrm{if $M=T$}.
\end{cases}
\end{equation}

\end{itemize}

\subsubsection{Bloch--Kato's local conditions at $p$.} 
 Here we assume that  $\mathcal X=\Spm (E).$ 
Following Bloch and Kato, we put
\begin{equation}
\nonumber
H^1_{\rm f}(\Qp,M):= 
\begin{cases} \ker \left (H^1(\Qp, V) \rightarrow H^1(\Qp,V\otimes_{\Qp}\mathbf{B}_{\textrm{cris}})\right ), &\textrm{if $M=V$},\\
\mathrm{im} \left (H^1_{\rm f}(\Qp, V) \rightarrow H^1(\Qp,V/T)\right ),  &\textrm{if $M=V/T$},\\
\textrm{$\ker\left(H^1(\Qp, T) \lra \dfrac{H^1(\Qp,V)}{H^1_{\rm f}(\Qp,V)}\right)$},  &\textrm{if $M=T$}.
\end{cases}
\end{equation}

\subsection{Selmer complexes}
\label{subsec_selmer_complexes}

\subsubsection{}  In this paper, we will work with Selmer complexes with coefficients in $V$
defined by Greenberg local conditions arising from subrepresentations of $V$ or, more generally,
submodules of $\DdagrigX (V).$  Therefore our setting differs slightly  from the point of view of  \cite{nekovar06}, where Selmer complexes are defined on the integral level, and Pontryagin duality plays a central role. 

We review  general constructions following  \cite{benoisheights}. 

\subsubsection{} Let $V$ be a $p$-adic representation of $G_{\QQ,S}.$
We will write $\RG (V)$ (resp. $\RG (\QQ_\ell,V)$) for the complexes of continuous cochains
of $G_{\QQ,S}$ (resp. of $G_{\QQ_\ell}$) with coefficients in $V$. 
Let 
\begin{equation}
\nonumber
\res_S:=(\res_\ell)_{\ell\in S} \,:\,C^\bullet (G_{\QQ,S},V) \xrightarrow{(\res_\ell)_{\ell\in S}} \bigoplus_{\ell\in S} C^\bullet (G_{\QQ_\ell},V)
\end{equation}
denote the restriction map. 
For any $\ell\in S\setminus\{p\},$  set $K^\bullet_\ell (V)=C^\bullet (G_{\QQ_\ell},V)$ and define $K^\bullet (V):=\underset{\ell\in S}\bigoplus K^\bullet_\ell (V)$. For each $\ell\in S$, let $f_\ell$  denote the map given by 
\[
f_\ell=\begin{cases}
\res_\ell &\text{if $\ell\in S\setminus \{p\}$},\\
\xi_p\circ \res_p &\text{if $\ell=p,$}
\end{cases}
\]
where $\xi_p$ is the map from \eqref{eqn_Fontaine_Herr_enhanced}. We have a morphism of complexes:
\[
f_S=(f_\ell)_{\ell\in S} \,:\,C^\bullet (G_{\QQ,S},V) \lra K^\bullet (V).
\]

\subsubsection{} 
\label{subsubsec_1223_12_11}
Let $\bD$ be a $(\varphi,\Gamma)$-submodule of $\DdagrigX (V)$ (not necessarily saturated). We associate to $\bD$  the  collection $\left\{U_\ell^\bullet(V,\bD), g_\ell\right\}_{\ell\in S}$ of local conditions  given as follows: 
\begin{equation}
\nonumber
U_\ell^\bullet (V,\bD):= 
\begin{cases}
C^{\bullet}_{\textrm{ur}}(G_{\QQ_\ell}, V), &\forall\,\ell\in S\setminus \{p\},\\
C^\bullet_{\varphi,\gamma}(\bD), &\ell=p.
\end{cases}
\end{equation}
This  data  can be summarized with the diagram
\begin{equation}
\label{eqn:first diagram selmer}
\xymatrix{
C^\bullet (G_{\QQ,S},V) \ar[r]^(.6){f_S} & K^\bullet (V)\\
& U^\bullet (V,\bD), \ar[u]^{g_S}
}
\end{equation}
where $U^\bullet (V,\bD):=\oplus_{\ell\in S} U^\bullet_l (V,\bD)$ and $g_S:=(g_\ell)_{\ell\in S}$. 
Let $\RG (\QQ_\ell,V,\bD)=[U^\bullet (\QQ_\ell,V,\bD)].$ Then in the derived category 
of $\cO_{\cX}$-modules, the diagram (\ref{eqn:first diagram selmer}) takes the form
\begin{equation}
\nonumber 
\xymatrix{
\RG (V) \ar[r]^(.45){f_S} & \underset{\ell\in S}\oplus \RG (\QQ_\ell,V)\\
& \underset{\ell\in S}\oplus \RG (\QQ_\ell,V,\bD). \ar[u]^(.4){g_S}
}
\end{equation}

The Selmer complex associated to $\left\{U_\ell^\bullet(V,\bD), g_\ell\right\}_{\ell\in S}$ is defined as 
\begin{equation}
\nonumber
S^\bullet(V,\bD):=\mathrm{cone}
\left [C^\bullet (G_{\QQ,S},V) \oplus U^\bullet (V,\bD)
\xrightarrow{f_S-g_S} K^{\bullet}(V) \right ] [-1].
\end{equation}

We denote by  $\RG (V,\bD)= [S^\bullet (V,\bD)]$  the corresponding object of the category of bounded perfect complexes of $\cO_{\mathcal X}$-modules (see \cite[Chapter~3]{benoisheights} for detail) and write 
$\mathbf{R}^i\boldsymbol{\Gamma}(V,\bD)$ for the cohomology of 
$S^\bullet (V,\bD).$ Note that $H^i(V,\bD)=0$  for $i\geqslant 4.$
Each cohomology class $[z_{\rm f}]\in H^i(V,\bD)$ can be represented by a triple
\begin{equation}
\nonumber
z_{\rm f}=(z,(z_\ell^+)_{\ell\in S}, (\lambda_\ell)_{\ell\in S}),
\end{equation}
where 
\begin{equation}
\begin{aligned}
&z\in C^i(G_{\QQ,S},V), && z_\ell^+\in U^i_\ell(V,\bD), &&&\lambda_\ell\in K_\ell^{i-1}(V),\\
&d(z)=0, && d(z_\ell^+)=0, &&& f_\ell(z)=g_\ell(z_\ell^+)-d (\lambda_\ell).
\end{aligned}
\end{equation}
We will often write  $(z,(z_\ell^+), (\lambda_\ell))$ instead of 
$(z,(z_\ell^+)_{\ell\in S}, (\lambda_\ell)_{\ell\in S})$. 

\subsubsection{} 
In the setting of \S\ref{subsubsec_1223_12_11}, we have a distinguished triangle
\begin{equation}
\nonumber
K^\bullet (V) \lra S^\bullet (V,\bD)[1] \lra 
\bigl ( C^\bullet (G_{\QQ,S},V)\oplus U^\bullet (V,\bD)\bigr ) [1]\lra 
K^\bullet (V)[1]\,,
\end{equation}
which induces the long exact sequence 
\begin{multline}
\label{exact sequence to compute Selmer groups}
\begin{aligned}
0   \lra& H^0(V,\bD)\lra 
H^0(\QQ,V) \oplus \left (\underset{\ell\in S\setminus \{p\}} \bigoplus 
H^0(\QQ_\ell, V) \oplus H^0(\Qp, \bD)\right ) \lra \underset{\ell\in S} \bigoplus 
H^0(\QQ_\ell, V) \\
&\lra H^1(V,\bD) \lra 
H^1(G_{\QQ,S},V) \oplus \left (\underset{\ell\in S\setminus \{p\}} \bigoplus 
H^1_{\rm f}(\QQ_\ell, V)\oplus H^1(\Qp, \bD) \right )\lra \underset{\ell\in S} \bigoplus 
H^1(\QQ_\ell, V)\lra \cdots ,
\end{aligned}
\end{multline}


\subsection{Selmer groups and Tate--Shafarevich groups}
\label{subsec_123_21_11_1607}
Throughout \S\ref{subsec_123_21_11_1607}, we assume that $\cX$ is a point and $\cO_\cX=E$. Suppose that $V$ is a $p$-adic representation of $G_{\QQ,S}$ with coefficients in $E$.
\subsubsection{}  Let 
\[
U:=\{U_\ell(V)\subset H^1(\QQ_\ell, V), \ell \in S\}
\]
be a family of local conditions defined on the level of cohomology groups in degree $1$. 
Let us set $H^1_S(V):=H^1(G_{\QQ,S},V)$ and put
\[
H^1_U(V):=\ker \left (H^1(V)\rightarrow \underset{\ell\in S}\bigoplus 
\frac{H^1(\QQ_\ell,V)}{U_\ell(V)}\right ).
\]
If  $T$ be a fixed lattice of $V,$ we consider the local conditions on $T$ and $V/T$ given by
\[
\begin{aligned}
\nonumber 
&U_\ell (T):= \ker\left(H^1(\QQ_\ell,T) \to 
\dfrac{H^1(\QQ_\ell,V)}{U_{\ell}(V)}\right),\\
&U_\ell (V/T):= \textrm{im} \left ( U_\ell(V) \rightarrow H^1(\QQ_\ell,V/T)\right ).
\end{aligned}
\]
Using this data, we define the associated Selmer group and Tate--Shafarevich group on setting
\begin{equation}
\nonumber
\begin{aligned}
&H^1_U(V/T):= \ker \left (H^1(V/T)\rightarrow \underset{\ell\in S}\bigoplus 
\frac{H^1(\QQ_\ell,V/T)}{U_\ell(V/T)}\right ),\\
&\Sha_U(T):= \mathrm{coker} \left (H^1_U(V) \rightarrow H^1_U(V/T) \right ).
\end{aligned}
\end{equation}
Here we record explicitly some important special cases of interest:
\begin{itemize}
\item{\it Bloch--Kato Selmer groups.} Here $U_\ell (V)=H^1_{\rm f}(\QQ_\ell,V)$ for all $\ell \in S.$
The groups $H^1_U(V),$ $H^1_U(V/T)$ and $\Sha_U(T)$ coincide with $H^1_{\rm f}(V),$ $H^1_{\rm f}(V/T)$ and $\Sha (T)$ as defined by Bloch and Kato \cite{blochkato}. 

\item{\it Greenberg-style Selmer groups.} In this case, $U_\ell (V)=H^1_{\rm f}(\QQ_\ell,V)$ for all $\ell \in 
S\setminus\{p\}$ and 
\[
U_p(V)= \textrm{im} \left ( H^1 (\Qp, Z) \rightarrow H^1(\Qp,V) \right ),
\]
where $Z$ is either a $G_{\Qp}$-sub-representation of $V$ or (more generally) a $(\varphi,\Gamma)$-submodule
of $\DdagrigE (V)$.
\end {itemize}

\subsubsection{} 
\label{subsubsec_defn_fine_selmer_relaxed_Selmer}
Replacing the local condition $U_p$ at $p$ with the entire local cohomology group at $p$, we obtain the \emph{relaxed Selmer groups} $H^1_{U,\{p\}}(V),$ $H^1_{U,\{p\}}(T)$ and  $H^1_{U,\{p\}}(V/T)$. Replacing the local condition at $p$ with the zero subgroup, we obtain the \emph{strict Selmer groups} 
 $H^1_{U,0}(V),$ $H^1_{U,0}(T)$ and  $H^1_{U,0}(V/T).$ 
 
 When $U$ is the Bloch--Kato local conditions on $V$ (so that $U_\ell (V)=H^1_{\rm f}(\QQ_\ell,V)$ for all $\ell \in S$), we write $H^1_{{\rm f},\{p\}}(V)$ and $H^1_{0}(V)$ to denote the corresponding relaxed and strict Selmer groups.

\subsubsection{} For each $\ell\in S,$ let $U_\ell^\perp (V^*(1))\subset H^1(\QQ_\ell,V^*(1))$ denote the orthogonal complement of $U_\ell (V)$ under the local duality. The following proposition follows directly from the definitions:

\begin{proposition}
\label{prop: about general tate-shafarevich}
The following statements hold true:
 
\item[i)]{} The Poitou--Tate exact sequence induces the exact sequences
\[
0\lra H^1_U(T) \lra H^1_{U,\{p\}}(T) \lra
\frac{H^1(\Qp,T)}{H^1_U(\Qp,T)} \lra
\left (\frac{H^1_{U^\perp}(V^*/T^*(1))}{H^1_{U^\perp,0}(V^*/T^*(1))}\right )^\vee
\lra 0\,,
\] 
\[
0\lra H^1_U(V) \lra H^1_{U,\{p\}}(V) \lra
\frac{H^1(\Qp,V)}{H^1_U(\Qp,V)} \lra
\left (\frac{H^1_{U^\perp}(V^*(1))}{H^1_{U^\perp,0}(V^*(1))}\right )^\vee
\lra 0\,,
\]

\item[ii)]{} The Tate-Shafarevich group $\Sha_U (T)$ is finite.

\end{proposition}
\begin{proof} The exact sequences in (i) follow from the Poitou--Tate exact sequence. 
See \cite[Theorem 2.3.4]{mr02} or \cite[Theorem I.7.3]{rubin00} for details.
The finiteness of $\Sha_U (T)$ is proved, for example, in \cite[Section~5.3.5]{fontaine-pr94}.
\end{proof} 

\subsection{Selmer complexes in Perrin-Riou's theory} 
\label{subsec_124_21_11_1619}
Let $V$ be a $p$-adic representation of $G_{\QQ,S}$ with coefficients in $E$. We will assume that $V$ is semistable at $p$ and denote by $\Dst (V)$ and $\Dc (V)=\Dst (V)^{\varphi=N}$ Fontaine's modules associated to $V$. Let $T$ be a $\cO_E$-lattice of $V$ stable under the Galois action.  

We shall work in \S\ref{subsec_124_21_11_1619} with Galois cohomology with coefficients in $V\otimes_{E}\CH (\Gamma)^{\iota}$, which can be viewed as the Iwasawa cohomology with coefficients in $V$. Set 
\begin{equation}
\nonumber
\RG_{\Iw}(T):= \RG (G_S, T\otimes_{\cO_E}\LL^{\iota}), 
\qquad \RG_{\Iw}(V):= \RG_{\Iw}(T)\otimes_{\cO_E}E.
\end{equation}
The cohomology of these complexes coincides with the classical Iwasawa cohomology
\begin{equation}
\label{eqn: global Iwasawa cohomology}
H^i_{\Iw}(T):= \underset{n}\varprojlim H^i(G_{\QQ (\zeta_{p^n})},T),
\qquad H^i_{\Iw}(V):=H^i_{\Iw}(T)\otimes_{\cO_E}E.
\end{equation}

\subsubsection{}
Let $D$ be a $\varphi$-submodule of $\Dc (V)$. Let us set
 \[
\RG (\QQ_\ell, D,V):= D\otimes_{\cO_E}\LL  [-1]
\] 
and consider it together with the restriction of Perrin-Riou's exponential map on $D\otimes_{\cO_E}\LL$
(where we identify $\LL$ with  $\cO [[X]]^{\psi=0}$):
\[
\bExp_{V,h}\,:\,D\otimes_{\cO_E}\LL  [-1] \ra \RG_{\Iw}(\Qp,T)\otimes^{\mathbf L}_{\LL}\CH (\Gamma).
\]
This datum does not give rise to a local condition in the proper sense, but the induced $\CH (\Gamma)$-linear map 
\[
\bExp_{V,h}\,:\,D\otimes_{E}\CH (\Gamma)  [-1] \ra \RG_{\Iw}(\Qp,T)\otimes^{\mathbf L}_{\LL}\CH (\Gamma)
\]
defines a local condition at $p.$ If $N\subset D$ is an $\cO_E$-lattice, then 
 \[
\RG (\QQ_\ell, N,T):= N\otimes_{\cO_E}\LL  [-1]
\] 
is a $\LL$-lattice inside $\RG (\QQ_\ell, D,V).$

For all $\ell \in S\setminus\{p\}$, take  
\begin{equation}
\label{eq: Iwasawa unramified local condition}
\RG (\QQ_\ell, N,T):= \left [ T\otimes_{\cO_E}\LL^{\iota} \xrightarrow{\Fr_\ell-1}
 T\otimes_{\cO_E}\LL^{\iota}\right ].
\end{equation}
The canonical map $\RG (\QQ_\ell, N,T)\rightarrow \RG_\Iw (\QQ_\ell,T)$ induces a map
\[
\RG (\QQ_\ell, N,T)\otimes_{\LL}^{\mathbf L}\CH (\Gamma) \rightarrow  \RG_{\Iw}(\QQ_\ell,T)\otimes^{\mathbf L}_{\LL}\CH (\Gamma),
\]
which we consider as  a local condition at $\ell.$ These data can be summarized 
by the diagram:
\begin{equation}
\nonumber
\xymatrix{
\RG_{\Iw}(T)\otimes^{\mathbf L}_{\LL}\CH(\Gamma) \ar[r] &\underset{\ell\in S}\oplus \RG_{\Iw} (\QQ_\ell, T)\otimes_{\LL}^{\mathbf L}\CH(\Gamma)  \\ 
&\left (\underset{\ell \in S}\oplus  \RG_\Iw (\QQ_\ell, N,T)\right )\otimes_{\LL}^{\mathbf L}
\CH (\Gamma)\,. \ar[u]
}
\end{equation}
The associated Selmer complex
\begin{multline}
\nonumber
\RG_{\Iw,h}(V,D):=
\textrm{cone} \left [
\left (\RG_{\Iw}(T) \oplus \left (\underset{\ell \in S}\oplus  \RG_\Iw (\QQ_\ell, N,T)\right )
\right )\otimes_{\LL}^{\mathbf L} \CH(\Gamma) \right. \\
\left.
\ra 
\underset{\ell\in S}\oplus \RG_{\Iw} (\QQ_\ell, T)\otimes_{\LL}^{\mathbf L}\CH(\Gamma)
\right ] [-1]
\end{multline}
plays a central role in Perrin-Riou's theory.

\subsection{Pairings}
\label{subsec_global_duality}

\subsubsection{}
\label{subsubsec_global_duality_1} 
We consider the following setting, which is a slight generalization of 
the setting from \cite[Chapter 3]{benoisheights}. Let $\mathcal X \xrightarrow{w} \cW$  
be a morphism of affinoids and let $\cO_{\cW} \rightarrow \cO_{\cX}$ the associated 
morphism of affinoid algebras.  Let $V$ and $V'$ be two Galois representations with coefficients in 
$\cO_{\cX}$ equipped with 
an $\cO_{\cW}$-bilinear pairing 
\begin{equation}
\label{pairing between V and V' } 
\left (\,,\,\right )\,\,:\,\,V'\times V\lra \cO_{\cW}\,.
\end{equation}
 The pairing \eqref{pairing between V and V' } induces a $\cO_{\cW}$-linear pairing
\begin{equation}
\nonumber 
(\,,\,) \,\,:\,\,\DdagrigX (V') \times \DdagrigX (V) \lra \cO_{\cW}\,.
\end{equation}
We say that $(\varphi, \Gamma)$-submodules $\bD \subseteq \DdagrigX (V)$ and 
$\bD' \subseteq  \DdagrigX (V')$ over $\CR_\cX$ are orthogonal if 
\[
(x,y)=0, \qquad \forall\,\,  x\in \bD',\quad\forall\,\,  y\in \bD.
\]

\subsubsection{Semi-local invariant maps} 
Recall the construction of semi-local invariant maps. For the proofs of the following facts and further details, we refer the reader to \cite[Section~5.4.1]{nekovar06} (see also \cite{benoisheights}, Theorem~3.1.5). Consider the composition
\[ 
\tau_{\geqslant 2}C^\bullet (G_{\QQ,S},\cO_\cW (1))\xrightarrow{\res_S} 
\bigoplus_{\ell\in S} \tau_{\geqslant 2} C^\bullet (G_{\QQ_\ell},\cO_\cW (1)) \xrightarrow{f_S} \tau_{\geqslant 2} K^\bullet (\cO_\cW(1)),
\]
where $\tau_{\geqslant 2}$ denote the truncation at degree $2,$
and define
\[
E^\bullet =\mathrm{cone} \bigl [\tau_{\geqslant 2}C^\bullet (G_{\QQ,S},\cO_\cW (1))
\lra \tau_{\geqslant 2} K^\bullet (\cO_\cW (1))\bigr ][-1]\,.
\]
The cohomology of $E^\bullet$ is concentrated at degree $2.$ More precisely, local class field theory gives rise to canonical isomorphisms $i_\ell\,:\, \cO_\cW \simeq H^2(\QQ_\ell, \cO_\cW(1))$ for each $\ell \in S$  and we have a commutative diagram
\begin{equation}
\nonumber 
\xymatrix{
\tau_{\geqslant 2}C^\bullet (G_{\QQ,S},\cO_\cW (1)) \ar[r]^-{f_S}
&\bigoplus_{\ell\in S} \tau_{\geqslant 2} C^\bullet (G_{\QQ_\ell},\cO_\cW (1))
\ar[r]^-{j} &E^\bullet [1]\\
& \bigoplus_{\ell\in S} \cO_\cW [-2] \ar[r]_-{\sum} \ar[u]^{(i_\ell)_{\ell \in S}} & \cO_\cW [-2] \ar[u]^{i_S}
}
\end{equation}
where $\sum$ denotes the summation over $\ell\in S,$ and $i_S=j\circ i_{\ell_0}$ for some fixed $\ell_0\in S.$ By global class field theory, $i_S$ is a quasi-isomorphism, and there exists a homotopy inverse 
\[
r_S\,:\, E^\bullet\rightarrow  \cO_\cW [-3]
\]
of $i_S,$ which is unique up to homotopy. 

\subsubsection{}
Assume that $\bD$ and $\bD'$ are orthogonal. We then have a pairing
\begin{equation}
\nonumber
\cup^{r_S}_{\bD',\bD} \,:\, S^\bullet(V'(\chi), \bD'(\chi)) \otimes_A S^\bullet (V,\bD) \lra \cO_\cW [-3]
\end{equation}
given as follows: If
\[
y_{\rm f}=(y, (y_\ell^+), (\lambda_\ell))\in S^i(V',\bD), 
\quad z_{\rm f}=(z, (z_\ell^+), (\mu_\ell))\in S^{3-i}(V,\bD),
\]
then 
\begin{equation}
\nonumber
\left (y_{\rm f} \right ) \cup_{\bD',\bD}^{r_S} 
\left (z_{\rm f} \right )=r_S \left (y\cup z, \lambda_\ell\cup f_\ell(z)+(-1)^ig_\ell(y_\ell^+)\cup \mu_\ell \right ).
\end{equation}
Since the homotopy class of $r_S$ does not depend on the choice of $r_S,$ 
it induces  a canonical pairing 
\begin{equation}
\nonumber 
\cup_{\bD',\bD}\,:\, \RG (V'(\chi), \bD'(\chi)) \otimes_{\cO_\cW} ^{\mathbf{L}} \RG (V,\bD) \lra \cO_{\cW} [-3]\,.
\end{equation}
On the level of cohomology groups, this gives rise to the $\cO_\cW$-linear pairings
\begin{equation}
\nonumber
\cup_{\bD',\bD} \,:\, {\mathbf{R}^i\boldsymbol{\Gamma}(V'(\chi), \bD'(\chi)) \otimes \mathbf{R}^{3-i}\boldsymbol{\Gamma}(V, \bD)} \lra \cO_\cW\,\,, \qquad 0\leqslant i\leqslant 3.
\end{equation}


\section{$p$-adic heights}
\label{sec_padic_heights}

\subsection{Heights afforded by cyclotomic variation}
\label{subsec_height_cyclotomic} 
In this \S\ref{subsec_height_cyclotomic}, we review the construction of $p$-adic heights. We refer the reader to \cite{nekovar06,benoisheights} for further details. 
\subsubsection{}
The cyclotomic Galois group $\Gamma$ has a canonical decomposition into the direct product $\Gamma =\Delta \times \Gamma_1.$  Put $\Lambda (\Gamma_1):=\cO_E [[\Gamma_1]]$ and fix  a generator $\gamma_1$ of $\Gamma_1$. Recall that  $\Lambda (\Gamma_1)$ is equipped with the $\cO_E$-linear  involution $\iota$ given by $\iota (\gamma)=\gamma^{-1},$ $\gamma\in \Gamma_1.$ Let  $ \cO_E^{\cyc}=
\Lambda (\Gamma_1)^\iota/J^2,$ where $J:=\ker\left(\LL (\Gamma_1)\xrightarrow{\gamma_1\mapsto 1} O_E\right)$ denotes the augmentation ideal of $\Lambda (\Gamma_1),$ and let  $E^{\cyc}:= \cO_E^{\cyc} [1/p].$ There exists  an exact sequence 
\begin{equation}
\label{eqn: infinitesimal deformation of E}
0\lra E \stackrel{\mathfrak{g}}{\lra} E^{\cyc} \lra E\lra 0
\end{equation}
(which does not depend on the choice of $\gamma_1$), where $\mathfrak{g} (a)=a\tau $ and  $\tau=\dfrac{\gamma_1-1}{\log\chi(\gamma_1)}$.

\subsubsection{}
\label{subsubsec_312_11_11}
Let $V$ be a $p$-adic Galois representation with coefficients in $\cO_\cX$ and let 
$\bD$ be a $(\varphi,\Gamma)$-submodule of $\DdagrigX (V).$ 
Set $V^\cyc:=E^\cyc \widehat\otimes_E V$ and $\bD^\cyc=  E^\cyc \widehat\otimes_E \bD$.
The exact sequence \eqref{eqn: infinitesimal deformation of E} induces exact sequences 
\begin{equation}
\nonumber
0\lra V\lra V^\cyc \lra V\lra 0\qquad,\qquad
0\lra \bD \lra \bD^\cyc \lra \bD\lra 0.
\end{equation}
This gives rise to a distinguished  triangle of Selmer complexes
\begin{equation}
\label{eqn_bocstein_sequence_cyclo} 
\RG (V,\bD)\lra \RG (V^\cyc,\bD^\cyc) \lra \RG (V,\bD)\xrightarrow{\beta_{\bD'}^\cyc} \RG (V,\bD) [1] \,,
\end{equation}
see \cite[Section~3.2]{benoisheights} for a detailed exposition. The morphism $\beta_{\bD'}^{\cyc}$ is called the Bockstein morphism (for the infinitesimal cyclotomic variation).


\subsubsection{}
\label{subsec_height_cyclotomic_def}
We place ourselves in the scenario of Section~\ref{subsec_global_duality}.  
Let $(\,\,,\,\,)\,: \, V'\times V \rightarrow \cO_\cW$ be an $\cO_\cW$-bilinear pairing 
and $\bD \subseteq \DdagrigA (V)$ and $\bD' \subseteq \DdagrigA (V')$ two
orthogonal $(\varphi, \Gamma)$-submodules. 
The $p$-adic height pairing associated to these data is defined as the morphism
\begin{equation}
\label{eqn_defn_height_cyclo}
h^\cyc_{\bD',\bD}\,:\,\RG (V'(\chi), \bD'(\chi)) \otimes_{\cO_\cW}^{\mathbf L}\RG (V,\bD) \xrightarrow{\beta_{\bD'}^\cyc \otimes \id} 
 \RG (V'(\chi), \bD'(\chi))[1] \otimes_{\cO_\cW}^{\mathbf L}\RG (V,\bD)
\xrightarrow{\cup_{\bD',\bD}} \cO_\cW [-2].
\end{equation}
It induces an $\cO_\cW$-linear pairing on cohomology
\begin{equation}
\nonumber
\left < \,,\,\right>^{\cyc}_{\bD',\bD}\,:\, { \mathbf{R}^1\boldsymbol\Gamma (V'(\chi), \bD'(\chi)) \otimes \mathbf{R}^1\boldsymbol\Gamma (V, \bD)} \lra \cO_\cW.
\end{equation}

\subsubsection{} 
\label{subsubsec_314_10_11_2021}

For any $\kappa \in \cW (E),$ let  $\mathfrak m_\kappa$ denote  the maximal ideal of $\cO_\cW$ associated to $\kappa.$ Set $E_\kappa:=\cO_\cW/\mathfrak m_{\kappa}$
and $\widetilde E_\kappa:=\cO_\cX/\mathfrak m_{\kappa}\cO_{\cX}.$ 
Set $V_\kappa=V\otimes_{\cO_\cW} E_\kappa$ and $\bD_\kappa=\bD\otimes_{\cO_\cW} E_\kappa.$
Note that  $V_\kappa$ (respectively $\bD_\kappa$) has a natural structure of $\widetilde E_\kappa$-module  (respectively $(\varphi,\Gamma)$-module over $\cR_{\widetilde E_\kappa}$). 

We record some functorial properties of Selmer complexes and  $p$-adic heights in the following  proposition: 

\begin{proposition} 
\label{proposition_general_properties_of_heights}
\item[i)] For any $\kappa \in \mathcal W$, one has
\[
\RG (V_\kappa'(\chi),\bD_\kappa'(\chi)) =\RG (V'(\chi), \bD'(\chi))\otimes_{\cO_\cX}^{\mathbf L} \widetilde E_\kappa,
\]
\item[ii)]  The following diagram commutes:
\[
\xymatrix{
\RG (V'(\chi), \bD'(\chi)) \otimes_{\cO_\cW}^{\mathbf L}\RG (V,\bD) \ar[d]^{\otimes^{\mathbf L}_{\cO_{\cW}}
E_\kappa}
 \ar[r]^-{\cup_{\bD',\bD}}
&\cO_\cW [-3] \ar[d]^{\otimes^{\mathbf L}_{\cO_\cW}E_\kappa}\\ 
\RG (V'_\kappa(\chi),\bD'_\kappa(\chi)) \otimes_{E_\kappa}^{\mathbf L}\RG (V_\kappa,\bD_\kappa)  \ar[r]^-{\cup_{\bD_\kappa,\bD'_\kappa}}
&E_\kappa [-3].
}
\]
\item[iii)] Assume that $\bD_1' \subseteq \bD_2' \subseteq \DdagrigX (V')$ are   $(\varphi,\Gamma)$-submodules of $\DdagrigX (V').$ Then the canonical map 
$C^{\bullet}_{\varphi,\gamma}(\bD_1'(\chi)) \rightarrow C^{\bullet}_{\varphi,\gamma}(\bD_2'(\chi))
$ induces a canonical map
\begin{equation}
\nonumber
\RG (V'(\chi),\bD_1'(\chi)) \lra \RG (V'(\chi),\bD_2'(\chi)).
\end{equation}
If $\bD_2'$ and $\bD\subseteq \DdagrigX (V)$ are orthogonal, then $\bD_1'$ and $\bD$ are orthogonal and 
the following diagrams commute:
\[
\xymatrix{
\RG (V'(\chi),\bD_1'(\chi)) \otimes_{\cO_\cW}^{\mathbf L}\RG (V,\bD) \ar[d] \ar[r]^-{\cup_{\bD_1,\bD'}}
&\cO_\cW[-3] \ar@{=}[d]\\ 
\RG (V'(\chi),\bD_2'(\chi)) \otimes_{\cO_\cW}^{\mathbf L}\RG (V,\bD)  \ar[r]^-{\cup_{\bD_2,\bD'}}
&\cO_\cW[-3]
}
\qquad
\xymatrix{
\RG (V'(\chi),\bD_1'(\chi)) \otimes_{\cO_\cW}^{\mathbf L}\RG (V,\bD) \ar[d] \ar[r]^-{h^\cyc_{\bD_1,\bD'}}
&\cO_\cW [-2] \ar@{=}[d]\\\ 
\RG (V'(\chi),\bD_2'(\chi)) \otimes_{\cO_\cW}^{\mathbf L}\RG (V,\bD)  \ar[r]^-{h^\cyc_{\bD_2,\bD'}}
&\cO_\cW [-2].
}
\] 
\end{proposition}
\begin{proof} The proof is purely formal and is omitted here.
\end{proof}

\subsection{Heights afforded by weight variation}
\label{subsec_height_weight} 

{The constructions in \S\ref{subsec_height_weight} will not be used in our calculations with what we call the thick $p$-adic $L$-function in \S\ref{chapter_critical_Selmer_padic_reg}. However, one can prove that the $p$-adic heights afforded by weight variation (``weight-heights'') do appear in the formula for the second order derivative with respect to $X$ of a two-variable $p$-adic $L$-function, whenever the relevant Bloch--Kato Selmer group is $1$-dimensional. We record the definition of the weight-height and review its basic properties for the sake of completeness.}

We place ourselves in the situation of \S\ref{subsubsec_global_duality_1} and \S\ref{subsubsec_314_10_11_2021}. Let $\varpi_\kappa\in \m_\kappa$ be any generator of the maximal ideal of $\cO_\cW$ associated to $\kappa\in \cW(E)$.

The short exact sequence
\be
\label{eqn_bockstein_weight_A}
0\lra \cO_\cW/\m_\kappa\xrightarrow{[\varpi_{\kappa}]} \cO_\cW/\m_\kappa^2\lra \cO_\cW/\m_\kappa\lra 0
\ee
induces the short exact sequences
\be
\label{eqn_bockstein_weight_A_M}
0\lra M_{\kappa} \stackrel{[\varpi_\kappa]}{\lra} M/\m_\kappa^2 M\lra  M_{\kappa}\lra 0\,,\qquad M=V'(\chi),\bD'(\chi)\,.
\ee
As in \S\ref{subsec_height_cyclotomic}, these in turn give rise to a distinguished triangle
\begin{align}
\label{eqn_bocstein_sequence_weight} 
\begin{aligned}
\RG (V_{\kappa}'(\chi),\bD_{\kappa}')\lra \RG (V'(\chi)/\m_\kappa^2V'(\chi),\bD'(\chi)/\m_\kappa^2\bD'(\chi))
\lra \RG& (V_{\kappa}'(\chi),\bD_{\kappa}'(\chi))\\
&\xrightarrow{\beta_{\bD'}^{\rm wt}} \RG (V_{\kappa}'(\chi),\bD_{\kappa}'(\chi)) [1]
\end{aligned}
\end{align}
where we call $\beta_{\bD'}^{\rm wt}$ the Bockstein morphism (for the infinitesimal weight variation). With this at hand, we define $h^{\rm wt}_{\bD',\bD}$ as the compositum of the arrows
\begin{align}
\begin{aligned}
\label{eqn_defn_height_weight}
h^{\rm wt}_{\bD',\bD}\,:\,\RG (V_{\kappa}'(\chi),\bD_{\kappa}'(\chi)) \otimes_{\cO_\cW}^{\mathbf L}\RG (V_{\kappa},\bD_{\kappa}) \xrightarrow{\beta_{\bD'}^{\rm wt} \otimes \id} 
 \RG (V_{\kappa}'(\chi),\bD_{\kappa}'(\chi))[1] \otimes_{\cO_\cW}^{\mathbf L}&\RG (V_{\kappa},\bD_{\kappa})
\\&\quad\xrightarrow{\cup_{\bD_{k},\bD'_{k}}} E[-2]
\end{aligned}
\end{align}
mimicking the construction of $h^{\rm cyc}_{\bD',\bD}$ in \S\ref{subsec_height_cyclotomic_def} and denote by 
\begin{equation}
\nonumber
\left < \,,\,\right>^{\rm wt}_{\bD',\bD}\,:\, { \mathbf{R}^1\boldsymbol\Gamma (V_\kappa'(\chi),\bD_\kappa'(\chi)) \times  \mathbf{R}^1\boldsymbol\Gamma (V_\kappa, \bD_\kappa) }\lra E
\end{equation}
the induced $E$-linear pairing.

\subsection{Rubin-style formulae}
\label{subsec_Rubin_style_formulae}
We compute (under a set of natural hypotheses) the height pairings we have introduced in \S\ref{subsec_height_cyclotomic} and \S\ref{subsec_height_weight} in terms of local duality pairings. This computation will play a crucial role in our leading term formulae in \S\ref{sec_padic_regulators}. Our computations are the generalizations of those Rubin proved in his much-studied works~\cite{Rubin1992PRConj, Rubin_RubinsFormula} and our treatment closely follows \cite{nekovar06, kbbMTT, benoisbuyukboduk1}.

In this subsection, let us denote by $W$ (resp. ${\mathbb D}$) either $V$ (resp. either $\bD$) or $V'(\chi)$ (resp. or $\bD'(\chi)$) to lighten our notation. Consider the exact sequence 
\be\label{eqn_deform_Wk}
0\lra W_\kappa \xrightarrow{[\varpi_\kappa]} W_{\epsilon} \xrightarrow{\pi_\kappa} W_\kappa\lra 0
\ee
where $ W_{\epsilon}:=W/\m_\kappa^2 W$\,. We similarly put ${\mathbb D}_\epsilon:={\mathbb D}/\m_\kappa^2{\mathbb D}$ and define for $(Z,D)=(W_\kappa,{\mathbb D}_k)$ or $(W_\epsilon, {\mathbb D}_\epsilon)$
$$ \widetilde{U}^\bullet(Z,{D})=\oplus_{\ell\in S} \widetilde{U}^\bullet_\ell(Z, {D}):={\rm cone}[U^\bullet(Z,D)\xrightarrow{-g_S}K^\bullet(Z)]\,.$$ 

\subsubsection{Selmer complexes and singular quotients}
For $(Z,D)=(W_\kappa,{\mathbb D}_k)$ or $(W_\epsilon, {\mathbb D}_\epsilon)$, we have the following tautological exact triangle in the derived category:
\be\label{eqn:selmersequence}
 \widetilde{U}^\bullet(Z,{D})[-1]\lra\RG (Z, D)\lra 
 {\RG(Z)} \stackrel{\widetilde{\textup{res}}_S}{\lra}  \widetilde{U}^\bullet(Z, {D})\,,
\ee
where $\RG(Z)$ (resp. $\RG (Z, D)$) denotes the image of the complex $C^\bullet(G_{\QQ,S},Z)$ (resp. the complex $S^\bullet(Z, D)$ given as in \S\ref{subsec_selmer_complexes}) in the derived category, cf. \cite[Section~6.1.3]{nekovar06}.  Moreover, for each $\ell \in S$, we have the tautological exact triangle
\be\label{eqn:localsequence}
{U}_\ell^\bullet (Z,{D})\lra K_\ell^\bullet(Z)\lra \widetilde{U}^\bullet_\ell(Z, {D})\lra {U}_\ell^\bullet(Z,{D})[1]
\ee
in the derived category.

Set $\widetilde{D}:= \DdagrigE (Z)/D$. Recall that  $\RG(\Qp,{D})$ (resp. $\RG(\Qp,\widetilde{{D}})$) denotes the class of the complex $C_{\varphi,\gamma}^{\bullet}(D)$ 
(resp. $C_{\varphi,\gamma}^{\bullet}(\widetilde{D})$) in the corresponding derived category ($D={\mathbb{D}}_?$ with $?=\epsilon, \kappa$). We have a distinguished triangle
\be
\label{eqn:localsequencetilde1}
\RG({ \Qp},{D})\lra \RG({ \Qp},\DdagrigE (Z))\stackrel{\mathfrak{s}}{\lra} \RG({ \Qp},\widetilde{{D}})\lra\RG({ \Qp},{D})[1]\,
\ee
The quasi-isomorphism $\alpha$ in \eqref{eqn_Fontaine_Herr_enhanced} together with the sequences \eqref{eqn:localsequence} and \eqref{eqn:localsequencetilde1} induce a functorial quasi-isomorphism 
\be\label{eqn:identifyingsungularquotients}\RG({ \Qp},\widetilde{{D}})\stackrel{qis}{\lra}  
\widetilde{U}_p^{ \bullet}(Z, { D})\,.\ee


The composition of \eqref{eqn:selmersequence} and 
\eqref{eqn:identifyingsungularquotients} induces a map
\be
\label{eqn:definitionofcoboundary}
\partial_0\,:\,H^0(\QQ_p,\widetilde{{D}})\lra \mathbf R^1\Gamma (Z, D)
\ee
which is compatible with all the duality statements in the obvious sense.

\subsubsection{Bockstein-normalized derivative} The short exact sequence \eqref{eqn_deform_Wk} induces a commutative diagram of complexes
$$\xymatrix@R=.6cm{
& 0\ar[d]&0\ar[d]& 0\ar[d]&\\
0\ar[r]&S^\bullet(W_\kappa,{ \mathbb{D}_\kappa}) \ar[d]\ar[r]&{C}^\bullet(W_\kappa) \ar[r]\ar[d]& \widetilde{U}^{ \bullet}(W_\kappa, { \mathbb{D}_\kappa}) \ar[d]\ar[r]& 0\\
0\ar[r]& S^\bullet(W_\epsilon,{\mathbb{D_\epsilon}}) \ar[d]\ar[r]&{C}^\bullet(W_\epsilon) \ar[d]\ar[r]& \widetilde{U}^{ \bullet} (W_\epsilon, { \mathbb{D}_\epsilon})\ar[d] \ar[r]& 0\\
0\ar[r]&S^\bullet(W_\kappa,{ \mathbb{D}_\kappa})\ar[d] \ar[r]&{C}^\bullet(W_\kappa) \ar[d]\ar[r]& \widetilde{U}^{ \bullet}(W_\kappa, { \mathbb{D}_\kappa}) \ar[d]\ar[r]& 0\\
& 0&0& 0& 
}$$
where ${C}^\bullet(Z)$ is a shorthand for ${C}^\bullet(G_{\QQ,S},Z)$ for $(Z,D)=(W_\kappa, \mathbb{D}_k)$ or $(W_\epsilon, \mathbb{D}_\epsilon)$. The horizontal (exact) lines are deduced from the defining properties of the Selmer complexes as mapping cones. On the level cohomology, this gives rise to the following commutative diagram with exact rows and columns:
\be
\begin{aligned}
\label{eqn_huge_bocksteindiagram_kappa}
\xymatrix@R=2pc @C=1.3pc {&&&&H^0(\widetilde{U}^{ \bullet}(W_\kappa, { \mathbb{D}_\kappa}))\ar[d]^{\beta^0}\\
&&&&H^1(\widetilde{U}^{ \bullet}(W_\kappa, { \mathbb{D}_\kappa}))\ar[d]^{i}\\
&&&{ H^1(W/\m_\kappa^2W)}\ar[d]^{\pi_\kappa}\ar[r]^-{\widetilde{\res}}&H^1(\widetilde{U}^{ \bullet}(W/\m_\kappa^2W, { \mathbb{D}/\m_\kappa^2\mathbb{D}}))\ar[d]^{\pi_\kappa}\\
&H^0(\widetilde{U}^{ \bullet}(W_\kappa, { \mathbb{D}_\kappa}))\ar[d]_{\beta^0}\ar[r]&{ \mathbf{R}^1\boldsymbol{\Gamma}}(W_\kappa,\mathbb D_\kappa)\ar[d]^{\beta^1}\ar[r]&{ H^1(W_\kappa)}\ar[r]_-{\widetilde{\res}}\ar[d]&H^1(\widetilde{U}^{ \bullet}(W_\kappa, { \mathbb{D}_\kappa}))\ar[d]\\
&H^1(\widetilde{U}^{ \bullet}(W_\kappa, { \mathbb{D}_\kappa}))\ar[r]_(.5){\partial}&{\mathbf{R}^2
\boldsymbol{\Gamma}}(W_\kappa,\mathbb D_\kappa)\ar[r]&{ H^2(W_\kappa)}\ar[r]&H^2(\widetilde{U}^{ \bullet}(W_\kappa,
{ \mathbb{D}_\kappa}))
}
\end{aligned}
\ee

\begin{lemma}
\label{lem_bocksteinnormalizedderivative}
Suppose that we are given a class $[z_{\rm f}]=[(z,(z_\ell^+)_{\ell \in S},(\lambda_\ell)_{\ell \in S})] \in { \mathbf{R}^1\boldsymbol{\Gamma}}(W_\kappa,\mathbb D_\kappa)$ such that 
$$\pi_\kappa\left([\widetilde{z}]\right)=[z] \in {  H^1(W_\kappa)}$$
for some $[\widetilde{z}]\in { H^1_S(W)}$.
Then there exists a class 
$$[D\widetilde{z}]=\left([D\widetilde{z}]_{\ell}\right)_{\ell \in S} \in \bigoplus_{\ell\in S} H^1(\widetilde{U}_\ell^{ \bullet}(W_\kappa, { \mathbb{D}_\kappa}))= H^1(\widetilde{U}^{ \bullet}(W_\kappa, { \mathbb{D}_\kappa}))$$
 such that 
\item[i)] $i([D\widetilde{z}])=\widetilde{\res}\left([\widetilde{z}]  \mod \m_\kappa^2\right)\,.$\\
\item[ii)] $\beta^1([z_{\rm f}])=-\partial([D\widetilde{z}])$\,.
\end{lemma}
\begin{proof}
Since we have 
$$\pi_\kappa\circ\widetilde{\res}\left([\widetilde{z}] \mod \m_\kappa^2\right)=\widetilde{\res}([z])=0$$ 
it follows, arguing as in Lemma~1.2.19 of \cite{nekovar06}, that there exists a class $d \in H^1(\widetilde{U}^{ \bullet}(W_\kappa, { \mathbb{D}_\kappa}))$ that verifies 
\begin{itemize}
\item $i(d)=\widetilde{\res}\left([\widetilde{z}]  \mod \m_\kappa^2\right)$\,,\\
\item $\beta^1([z_{\rm f}])+\partial(d)-\partial\circ\beta^0(t)=0$ for some $t\in H^0(\widetilde{U}^{ \bullet}(W_\kappa, { \mathbb{D}_\kappa}))$\,.
\end{itemize}
Set $[D\widetilde{z}]=d-\beta^0(t)$.
\end{proof}

\begin{defn}
\label{def_bocksteinder} 
In the situation of Lemma~\ref{lem_bocksteinnormalizedderivative}, we let  $\mathfrak{d}\left[\widetilde{z}\right] \in H^1(\QQ_p,\widetilde{\mathbb{D}}_\kappa)$ denote the image of $[{D}\widetilde{z}]_p$ under the natural isomorphism $H^1(\widetilde{U}^\bullet_p(W_\kappa,\mathbb D_\kappa))\xrightarrow{\eqref{eqn:identifyingsungularquotients}} H^1(\QQ_p,\widetilde {\mathbb D}_\kappa)$.
\end{defn}

 We will call the class $\mathfrak{d}\left[\widetilde{z}\right]$ the \emph{Bockstein-normalized derived singular projection} of $\widetilde{z}$. 

\subsubsection{Rubin-style formula for weight-heights}
\label{subsubsec_333_11_11}
The main goal in this subsection is to give a proof of the following Theorem, which we refer to as the \emph{Rubin-style-formula}, where we calculate (under suitable hypotheses) the weight-height pairing
$$\langle \,, \,\rangle_{\kappa}^{\rm wt}: \qquad { \mathbf{R}^1\boldsymbol\Gamma (V_\kappa'(\chi),\bD_\kappa'(\chi)) \otimes \mathbf{R}^1\boldsymbol\Gamma  (V_\kappa,\bD_\kappa) } \xrightarrow{\langle \,, \,\rangle_{\bD',\bD}^{\rm wt}} E$$
in terms of local duality pairings.
\begin{theorem}
\label{thm_RSformula}
Suppose that $H^1_{\rm f}(\QQ_\ell,V_\kappa)=0$ or, equivalently,
that $H^0(\QQ_\ell,V_\kappa)=0$ for all $\ell \in S\setminus\{p\}$. Let $[y_{\rm f}]=[(y,(y_\ell^+)_{\ell \in S},(\mu_\ell)_{\ell \in S})] \in  { \mathbf{R}^1\boldsymbol\Gamma} (V_\kappa,\bD_\kappa)$ be any class.  For $[z_{\rm f}]$ and $[\widetilde{z}]$ as in the statement of Lemma~\ref{lem_bocksteinnormalizedderivative} (with $W=V'(\chi)$) we have
\begin{align*}
\langle [z_{\rm f}], [y_{\rm f}]\rangle_{\kappa}^{\rm wt}&=-\textup{inv}_p\left((D\widetilde{z})_p\cup g_p(y_p^+)\right)\\
&=-\left<\mathfrak{d}\left[\widetilde{z}\right], \res_p([y])\right>
\end{align*}
where $D\widetilde{z}\in { \widetilde{U}^1}(V'_\kappa { (\chi)}, { \mathbb{D}_\kappa' (\chi)})$ is any cocycle representing $[D\widetilde{z}]$.
\end{theorem}
See Lemma~\ref{lemma_appendix_LGC_bis} below where we give a criterion for the vanishing of $H^1_{\rm f}(\QQ_\ell,V_\kappa)$, which is useful in our applications in the context of the eigencurve.

\begin{remark}
As in \cite[Remark A.3]{kbbMTT} (appealing to the discussion \cite[Section 1.2.1]{benoisheights} in place of \cite{nekovar06}, Sections 1.3.1 and 6.2.2), one may check for each coboundary $d(u,v) \in d\,\widetilde{U}_\ell^0(V_\kappa' { (\chi)}, { \mathbb{D}_\kappa' (\chi)})$ that 
$$\sum_{\ell\in S}\textup{inv}_\ell\left(d(u,v)\cup { g_\ell (y_{\ell}^+)}\right)=0$$
(even without assuming that $H^1_{\rm f}(\QQ_\ell,V_\kappa)=0$)  and therefore verify that the right side of the asserted equality in Theorem~\ref{thm_RSformula} is well-defined (despite the fact that $(D\widetilde{z})_\ell$ is defined only up to a coboundary).
\end{remark}

\begin{proof}[Proof of Theorem~\ref{thm_RSformula}]
This is entirely formal and follows the proof of \cite[Proposition 11.3.15]{nekovar06} with obvious modifications.
Recall the complex
\[
E^{\bullet}=\mathrm{cone} \left (\tau_{\geqslant 2} C^{\bullet}(G_{\QQ,S},
E(1))\xrightarrow{\res_S} \tau_{\geqslant 2} K^{\bullet}_S(E(1))\right ) [-1].
\]
Global class field theory gives rise to the following diagram with exact rows:
$$\xymatrix{{ H^2(\QQ_p(1))}\ar[r]^(.48){\textup{res}_S}\ar@2{-}[d]&\oplus_{\ell \in S}{ H^2(\QQ_\ell,\QQ_p(1))} \ar[r]^(.52){}\ar@2{-}[d]& H^3(E^{\bullet})\ar[r]\ar@{.>}[d]^{\textup{inv}_S}& 0\\
{ H^2(\QQ_p(1))}\ar[r]^(.48){\textup{res}_S}&\oplus_{\ell \in S}{ H^2(\QQ_\ell,\QQ_p(1))} \ar[r]^(.7){\sum_{\ell\in S}\textup{inv}_\ell}& E \ar[r]& 0}$$
Suppose $[a_{\textup{f}}]=\left[(a,a_S^+,\omega_S)\right] \in {  \mathbf{R}^2\boldsymbol\Gamma }(V_\kappa'(\chi),\bD_\kappa'(\chi))$ and $[y_{\textup{f}}]=\left[(y, y_S^+,\mu_S)\right]\in  { \mathbf{R}^1\boldsymbol\Gamma} (V_\kappa,\bD_\kappa)$, where we have set  $a_S^+=(a_\ell^+)$, $\omega_S=(\omega_\ell)$ $y_S^+=(y_\ell^+)$ and $\mu_S=(\mu_\ell)$ to ease our notation. The formula 
$$a_\textup{f}\,\widetilde\cup\, y_\textup{f} :=(a\cup y, \omega_S\cup \textup{res}_S(y) + g_S(a_S^+)\cup \mu_S) \in E^3$$
defines a cup product pairing
$$\widetilde\cup: S^2(V_\kappa'(\chi), \bD_\kappa'(\chi))\otimes S^1(V_\kappa, \bD_\kappa)\lra E^3\,.$$
This pairing is homotopic to the cup-product pairing given as in \S\ref{subsec_global_duality} thanks to \cite[Proposition~1.3.2]{nekovar06}. Therefore, the pairing on the level of cohomology
$$\langle \,,\,\rangle_{\textup{PT}} \,:\,{ \mathbf{R}^2\boldsymbol\Gamma}(V_\kappa'(\chi),\bD_\kappa'(\chi)) \otimes { \mathbf{R}^1\boldsymbol\Gamma} (V_\kappa, \bD_\kappa)\lra H^3(E^{\bullet})
\xrightarrow{\textup{inv}_S}E $$
(given by the discussion in \S\ref{subsec_global_duality})  can be computed by the formula
\[
\langle [a_{\textup f}],[y_{\textup f}]\rangle_{\textup{PT}} = \textup{inv}_S(a_\textup{f}\,\widetilde\cup\, y_\textup{f}).
\]
Since the cohomological dimension of $G_S$ equals to $2$, $a\cup y$ is a coboundary 
and there exists a cochain $C$ such that $dC=a\cup y$. One may therefore compute $\textup{inv}_S([a_\textup{f}\,\widetilde\cup\, y_\textup{f}])$ to be 
\[\textup{inv}_S([a_\textup{f}\,\widetilde\cup\, y_\textup{f}])=\sum_{\ell\in S}\textup{inv}_{\ell}\left(\omega_\ell\cup \textup{res}_\ell(y) +g_p(a_\ell^+)\cup \mu_\ell + \textup{res}_{\ell}(C)\right).\]
Let now  $[a_{\textup{f}}]=\beta^1([z_{\textup{f}}]),$ where $[z_{\textup{f}}]\in H^1(V_\kappa'(\chi), \bD_\kappa'(\chi))$. Since  $\pi_\kappa([\widetilde{z}])=[z],$ it follows from the exact sequence
\[
{ H^1(V'_{\epsilon}  (\chi))\rightarrow  H^1 (V'_\kappa (\chi)) \xrightarrow{\beta^1} H^2(V'_\kappa  (\chi))}
\] 
that $a$ is a coboundary. Write $a=dB$ for some $B\in C^1 (G_{\QQ,S},V'_\kappa)$, one may take $C=B\cup y$ above. Then $g_\ell(a_{\ell}^+)=d(\omega_{\ell}+\res_{\ell}(B))$
for each $\ell\in S$ and it is easy to check that
\[
\beta^1\left([z_{\textup{f}}]\right)=\partial \circ\widetilde{\res}_S\left(\left[\left( \omega_{\ell}+\res_{\ell}(B)\right)_{\ell}\right]\right).
\]
Thence,
\begin{align*}
\langle [z_{\rm f}], [y_{\rm f}]\rangle_{\kappa}^{\rm wt}&=\sum_{\ell\in S}\textup{inv}_{\ell}\left(\omega_\ell\cup \textup{res}_\ell(y) + g_\ell(a_\ell^+)\cup \mu_\ell + \textup{res}_{\ell}(B)\cup \res_{\ell}(y)\right)\\
&=\sum_{\ell\in S}\textup{inv}_{\ell}\left((\omega_\ell+\res_{\ell}(B))\cup \textup{res}_\ell(y) + d(\omega_{\ell}+\res_{\ell}(B))\cup \mu_\ell)\right)\\
&=\sum_{\ell\in S}\textup{inv}_{\ell}\left((\omega_\ell+\res_{\ell}(B))\cup (\textup{res}_\ell(y) + d \mu_\ell)\right)\\
&=\sum_{\ell\in S}\textup{inv}_{\ell}\left((\omega_\ell+\res_{\ell}(B))\cup g_\ell(y_{\ell}^+)\right)\\
&=\sum_{\ell\in S}\textup{inv}_{\ell}\left(\widetilde{\res}_{\ell}(\omega_\ell+\res_{\ell}(B))\cup g_\ell(y_{\ell}^+)\right ).
\end{align*}
We remark that $\widetilde{\res}_{\ell}(\omega_\ell+\res_{\ell}(B))$ and  $g_\ell(y_{\ell}^+)$ are cocycles for every $\ell$. Furthermore, 
\[
\textup{inv}_{\ell}\left(\widetilde{\res}_{\ell}(\omega_\ell+\res_{\ell}(B))\cup g_\ell(y_{\ell}^+)\right )=0, \qquad \ell \neq p,
\]
because  $[g_\ell(y_\ell^+)]\in H^1_{\textup{f}}(\QQ_\ell,V_\kappa)=0$ for $\ell \neq p$ thanks to our running hypothesis. It follows from the definition of $(D\widetilde{z})_p$ that
\[
\widetilde{\res}_{p}(\omega_p+\res_{p}(B))=-(D\widetilde{z})_p+\widetilde{\res}_p(b)
\]
for some $b\in C^1(G_{\QQ,S},V'_\kappa).$ An easy computation shows that  
\[
\textup{inv}_{p} (\widetilde{\res}_p(b)\cup g_p(y_{p}))=
\sum_{\ell\in S}\textup{inv}_{\ell} (\widetilde{\res}_{\ell}(b)\cup g_\ell(y_{\ell}))=0\,;
\]
cf. the proof of \cite[Proposition 11.3.15]{nekovar06}. We now conclude that
\[
\langle [z_{\rm f}], [y_{\rm f}]\rangle_{\kappa}^{\rm wt}=-\textup{inv}_{p}\left((D\widetilde{z})_p\cup g_p(y_p^+) \right )
\]
and the proof of our theorem is now complete.
\end{proof}

We close \S\ref{subsubsec_333_11_11} with a criterion for the vanishing of $H^1_{\rm f}(\QQ_\ell,V_\kappa)$ for $\ell \in S\setminus\{p\}$. For any $x\in \cX(E)$, let us put $V_x:=V \otimes_{\cO_\cX,x}E$ and denote by $\m_x\subset \cO_\cX$ the maximal ideal corresponding $x$. For each such $x$, let us fix a generator $\varpi_x$ of $\m_x$.

\begin{lemma}
\label{lemma_appendix_LGC_bis} Let $\ell \in S\setminus\{p\}$ be a prime and $\kappa\in \cW(E)$. Suppose that $H^1_{\rm f}(\QQ_\ell,V_x)=0$ for every $x\in \cX(E)$ that lies above $\kappa$. Then also $H^1_{\rm f}(\QQ_\ell,V_\kappa)=0.$
\end{lemma}

\begin{proof}
Let $\{x_1,\cdots,x_s\}\in \cX(E)$ denote the points above $\kappa$, so that we have
$$\m_\kappa\cO_\cX=\m_{x_1}^{e_1}\cdots \m_{x_s}^{e_s}\,,\qquad e_i\in \NN\,.$$
Considering the $G_{\QQ_\ell}$-cohomology of the exact sequence
$$0\lra V/\m_{x_i}^{j}V\xrightarrow{{[\varpi_{x_i}]}} V/\m_{x_i}^{j+1}V \lra V_{x_i}\lra 0$$
inductively (for $1\leq j\leq e_i-1$), together with our running assumption that
$$\dim H^0(\QQ_\ell,V_{x_i})=\dim H^1_{\rm f}(\QQ_\ell,V_{x_i})=0\,,\qquad\forall i=1,\cdots,s\,,$$ 
we infer that
\begin{equation}
\label{eqn_lemma_36_11_11}
\dim H^0(\QQ_\ell,V/\m_{x_i}^{e_i}V)=\dim H^1_{\rm f}(\QQ_\ell,V/\m_{x_i}^{e_i}V)=0\,,\qquad\forall i=1,\cdots,s\,.
\end{equation}
Considering the $G_{\QQ_\ell}$-cohomology of the exact sequence 
$$0\lra V/\m_{x_1}^{e_1}\cdots \m_{x_{j-1}}^{e_{j-1}}V\xrightarrow{[\varpi_{x_j}^{e_j}]}V/\m_{x_1}^{e_1}\cdots \m_{x_{j-1}}^{e_{j}}V\lra V/ \m_{x_j}^{e_{j}}V\lra 0$$
inductively (for $2\leq j\leq s$) together with \eqref{eqn_lemma_36_11_11}, we deduce that
 $$0=\dim H^0(\QQ_\ell,V_\kappa)=\dim H^1_{\rm f}(\QQ_\ell,V_\kappa)$$ 
 as required.
\end{proof}

\subsubsection{Rubin-style formula for cyclotomic heights}
\label{subsubsec_334_11_11}
We record in this subsection the suitable variants of Theorem~\ref{thm_RSformula} for the height pairing $\left < \,,\,\right>^{\rm cyc}_{\bD',\bD}$. The notation of \S\ref{subsubsec_312_11_11} is in effect and we recall that $W$ (resp. ${\mathbb D}$) denotes either $V$ (resp. either $\bD$) or $V'(\chi)$ (respectively, or $\bD'(\chi)$). Analogous to \eqref{eqn_huge_bocksteindiagram_kappa}, we have the following commutative diagram with exact rows and columns:
\be
\begin{aligned}
\label{eqn_huge_bocksteindiagram_cyc}
\xymatrix@R=2pc @C=1.3pc {&&&&H^0(\widetilde{U}^{\bullet}(W, { \mathbb{D}}))\ar[d]^-{\beta^0}\\
&&&&H^1(\widetilde{U}^{ \bullet}(W, { \mathbb{D}}))\ar[d]^-{i}\\
&&&{ { H^1(W^\cyc)}}\ar[d]^-{\pi_\cyc}\ar[r]^-{\widetilde{\res}}&H^1(\widetilde{U}^{ \bullet}(W^\cyc, { \mathbb{D}^{\cyc}}))\ar[d]^-{\pi_\cyc}\\
&H^0(\widetilde{U}^{ \bullet}(W, { \mathbb{D}}))\ar[d]_-{\beta^0}\ar[r]&{ \mathbf{R}^1\boldsymbol\Gamma}(W,{\mathbb D})\ar[d]^-{\beta^1}\ar[r]&{ { H^1(W)}}\ar[r]_-{\widetilde{\res}}\ar[d]&H^1(\widetilde{U}^{\bullet}(W, { \mathbb{D}}))\ar[d]\\
&H^1(\widetilde{U}^{ \bullet}(W, { \mathbb{D}}))\ar[r]_(.5){\partial}&{ \mathbf{R}^2\boldsymbol\Gamma} (W, \mathbb D)\ar[r]&{ { H^2(W})}\ar[r]&H^2(\widetilde{U}^{ \bullet}(W, { \mathbb{D}}))\,\,.
}
\end{aligned}
\ee

The proof of the first portion of the following theorem is essentially identical to that of Lemma~\ref{lem_bocksteinnormalizedderivative}, whereas its second part is formally identical to the proof of Theorem~\ref{thm_RSformula}.

\begin{theorem}
\label{thm_RSformula_cyclo_height}
Suppose that $[z_{\rm f}]=[(z,(z_\ell^+)_{\ell \in S},(\lambda_\ell)_{\ell \in S})] \in { \mathbf{R}^1\boldsymbol\Gamma} (V'(\chi), \bD'(\chi))$ verifies the property that 
$$\pi_\cyc\left([\widetilde{z}]\right)=[z] \in  { H^1(V'(\chi))}$$
for some $[\widetilde{z}]\in { H^1(V'(\chi)^\cyc)}$.
\item[i)] There exists a class 
$$[D_\cyc\widetilde{z}]=\left([D_\cyc\widetilde{z}]_{\ell}\right)_{\ell \in S} \in \bigoplus_{\ell\in S} H^1(\widetilde{U}_\ell^{ \bullet}(V'(\chi), { \bD'(\chi)})= H^1(\widetilde{U}(V'(\chi), { \bD'(\chi)}))$$
 such that 
 \begin{itemize}
\item $i([D_\cyc\widetilde{z}])=\widetilde{\res}\left([\widetilde{z}]\right)\,.$\\
\item $\beta^1([z_{\rm f}])=-\partial([D_\cyc\widetilde{z}])$\,.
\end{itemize}

\item[ii)] Let $D_\cyc\widetilde{z}\in \widetilde{U}_\ell^1(V'(\chi),\bD'(\chi))$ be any cocycle representing $[D_\cyc\widetilde{z}]$ and let $[y_{\rm f}]=[(y,(y_\ell^+)_{\ell \in S},(\mu_\ell)_{\ell \in S})] \in  { \mathbf{R}^1\boldsymbol\Gamma} (V, \bD)$ be any class. Then,
\begin{align*}
\langle [z_{\rm f}], [y_{\rm f}]\rangle_{\bD, \bD'}^{\rm cyc}&=-\sum_{\ell \in S}\textup{inv}_\ell\left((D_\cyc\widetilde{z})_\ell\cup g_\ell(y_\ell^+)\right)
\end{align*}
In particular, if $H^1_{\rm f}(\QQ_\ell,V)=0$ for $\ell\neq p$, then we have
$$\langle [z_{\rm f}], [y_{\rm f}]\rangle_{\bD, \bD'}^{\rm cyc}=-\left<\mathfrak{d}_\cyc\left[\widetilde{z}\right], \res_p([y])\right>$$ 
where $\mathfrak{d}_\cyc\left[\widetilde{z}\right] \in H^1(\QQ_p,\widetilde{\bD'}(\chi))$ is the image of $[{D}_\cyc\widetilde{z}]_p$ under the isomorphism 
$$H^1(\widetilde{U}_p(V'(\chi), \bD'(\chi)))\xrightarrow{\eqref{eqn:identifyingsungularquotients}} H^1(\QQ_p,\widetilde {\bD'}(\chi)).$$
\end{theorem}

 \begin{remark}
    The weight heights and cyclotomic heights are defined in an analogous way, in terms of weight (respectively cyclotomic) deformations of the given $p$-adic representation. They serve different purposes on the analytic side, in the sense that they are expected to explain derivatives of multivariate $p$-adic $L$-functions in different directions. This comment also applies to the Rubin-style formulae we prove in Theorem~\ref{thm_RSformula} and Theorem~\ref{thm_RSformula_cyclo_height}(ii) for both height pairings: The derived cohomology classes $\mathfrak{d}[\widetilde{z}]$ and $\mathfrak{d}_\cyc[\widetilde{z}]$ (which appear in the Rubin-style formulae) also share a similar definition, but the key input in their definition (as in the definition of Bockstein morphisms) is infinitesimal variation in two different directions (namely, the weight versus the cyclotomic direction).
\end{remark}


\chapter{Critical $p$-adic $L$-functions}
\label{chapter:critical L-functions}
The main goal of Chapter~\ref{chapter:critical L-functions} is to give an ``\'etale construction'' of critical $p$-adic $L$-functions. 

The first part of this chapter (\S\ref{sect:eigenspace-transition}) is devoted to the development of the required tools in the context of $p$-adic Hodge theory in a rather formal framework. The key result in this portion is what we call the \emph{eigenspace-transition principle} (cf. Proposition~\ref{prop: comparison of exponentials for different eigenvalues}), which allows us to prove the interpolation formulae for the critical $p$-adic $L$-functions we define in \S\ref{sec_abstract_setting} (still working in an abstract setting). 

The main results in \S\ref{sec_abstract_setting} are Proposition~\ref{prop_bellaiche_formal_step_1} and Theorem~\ref{prop_imoroved_padicL_vs_slope_zero_padic_L}, which outline the interpolation properties of the (abstract) $p$-adic $L$-functions we introduce in Definition~\ref{defn_fat_eta_etatilde}. 

In \S\ref{sec_new_2_3_2022_03_14}, which is independent from the rest of this chapter, we review the construction of $p$-adic $L$-functions attached to elliptic modular forms via the theory of modular symbols, following the works of Stevens, Pollack--Stevens and Bella\"iche. 

In the final portion of this chapter (\S\ref{sec_2_4_2022_05_11_0809}) we apply, in the context of the eigencurve, the formalism we have developed in \S\ref{sec_abstract_setting}, and introduce $p$-adic $L$-functions over the neighborhoods of the eigencurve about $\theta$-critical points. The main results in this portion (and also of this chapter) are Theorem~\ref{thm_interpolative_properties} (where we state and prove the interpolation properties of our critical $p$-adic $L$-functions), Theorem~\ref{thm:comparision with Bellaiche's construction} and its Corollary~\ref{cor_thm:comparision with Bellaiche's construction} (which establishes a comparison between our construction and the one in \S\ref{sec_new_2_3_2022_03_14} using modular symbols), as well as Proposition~\ref{prop: comparision p-adic L-functions for alpha and beta} (which relates the cyclotomic improvement of our critical $p$-adic $L$-function to a slope-zero $p$-adic $L$-function, where the comparison involves our Iwasawa theoretic $\cL$-invariant that we introduce in the next chapter).

\section{The eigenspace-transition principle}
\label{sect:eigenspace-transition}

\subsection{Deformations of a split representation} 
\label{subsect:deformation}
\subsubsection{} In this section, we consider a two dimensional  $p$-adic representation
of the Galois group $G_{\Qp}$   with coefficients in a finite extension $E$ of $\Qp$
satisfying the following condition:
\begin{itemize}
\item{} $V=V^{(\alpha)}\oplus V^{(\beta)},$ where $V^{(\alpha)}$ and  $V^{(\beta)}$ are one-dimensional crystalline representations with Hodge--Tate weights $-(k-1)$ ($k\geqslant 2$) and $0$ respectively. We denote by $\alpha$ and $\beta$ the eigenvalues of the Frobenius map 
$\varphi$ on $\Dc (V^{(\alpha)})$ and  $\Dc (V^{(\beta)}).$
\end{itemize} 


Let us set $\bD^{(\alpha)}=\DdagrigE (V^{(\alpha)})$ and $\bD^{(\beta)}=\DdagrigE (V^{(\beta)})$. These $(\varphi,\Gamma)$-modules can be explicitly described as follows:
\begin{equation}
\nonumber
\bD^{(\alpha)}=\CR_E e^{(\alpha)}, \qquad \bD^{(\beta)}=\CR_E e^{(\beta)},
\end{equation}
where
\begin{equation}
\label{formula: formulae for bD_x and bD_x^sat}
\begin{aligned}
&\gamma (e^{(\alpha)})=\chi^{1-{ k}}(\gamma)e^{(\alpha)},
&&\varphi (e^{(\alpha)})= \alpha p^{1-{k}}e^{(\alpha)}, &&&\\
&\gamma (e^{(\beta)})=e^{(\beta)},
&&\varphi (e^{(\beta)})= \beta e^{(\beta)} &&& \gamma\in \Gamma.
\end{aligned}
\end{equation}

\subsubsection{} Set $\widetilde{E}=E[X]/(X^2).$ Let 
\[
\bD=t^{k-1}\bD^{(\alpha)}, 
\] 
where $t=\log (1+\pi).$ 
We remark that $\DCc (\bD)$ and $\Dc (V^{(\alpha)})$ are isomorphic as
$\varphi$-modules but not as filtered modules: They have  Hodge--Tate weights
$0$ and $-(k-1)$, respectively.

Until the end of this section, we assume that we are given a
$G_{\Qp}$-representation $\widetilde{V}$ of rank two over $\widetilde{E}$ 
together with a crystalline submodule $\widetilde \bD \subset \bD^{\dagger}_{\mathrm{rig},\widetilde{E}}(\widetilde V)$ of rank one over $\mathscr{R}_{\widetilde{E}}$
such that
\begin{itemize}
\item{} $\widetilde V/X\widetilde V \simeq V.$

\item{} $\widetilde{\bD}/X\widetilde{\bD} \simeq  \bD.$
\end{itemize}
We can consider $(\widetilde{V}, \widetilde \bD)$ as a deformation of the data 
$(V, \bD).$ 
Since $X\widetilde V \simeq V$  as Galois representations, 
we have a tautological exact sequence
\begin{equation}
\label{exact_sequence_definition_of_widetilde_V}
0\lra V \lra \widetilde V \xrightarrow{\pi} V \lra 0\,.
\end{equation}
To simplify notation, we will often write $XV$ instead of $X\widetilde{V}$ when we wish to consider $V$ as a submodule of $\widetilde{V}$. 

\subsubsection{} 
\label{subsubsec_2113_2022_08_24_1732}
The representation $\widetilde V$ is not necessarily crystalline\footnote{See Remark~\ref{remark_bergdall} where we explain that it never is in the context of Coleman--Mazur eigencurve.}. On the other hand, from a technical perspective, it is convenient to work with the crystalline portion of $\widetilde V$ which should ``contain'' all the interesting crystalline sub-objects. To serve that purpose, let us define the $G_{\QQ_p}$-representation ${\widetilde V^{(\alpha)}}:=\pi^{-1}(V^{(\alpha)})$ and set $\widetilde \bD^{(\alpha)}:=\DdagrigE \bigl (\widetilde V^{(\alpha)}\bigr )$. We then have an exact sequence
\begin{equation}
\label{eqn_2022_06_13_0652} 
0\lra  XV \lra \widetilde V^{(\alpha)} \xrightarrow{\pi}  V^{(\alpha)} \lra 0
\end{equation}
of $G_{\QQ_p}$-representations, which gives rise to the following exact sequence of $(\varphi,\Gamma)$-modules:
\begin{equation}
\nonumber
0\lra \DdagrigE (XV) \lra 
\widetilde \bD^{(\alpha)}  \xrightarrow{\pi} \bD^{(\alpha)} \lra 0\,.
\end{equation}
Notice that we have $\DdagrigE (XV)=X\DdagrigE (V)$ as well as the identification 
\[
\DdagrigE (XV)=X\bD^{(\alpha)}\oplus X\bD^{(\beta)}\,.
\]

This discussion can be summarized by the commutative diagram
\begin{equation}
\nonumber
\xymatrix{
0\ar[r] &X\bD \ar[r] \ar@{^{(}->}[d] & \widetilde\bD  \ar[r] \ar@{^{(}->}[d] &\bD \ar[r] \ar@{^{(}->}[d] &0\\
0\ar[r] &X\DdagrigE (V) \ar[r] &
\widetilde \bD^{(\alpha)}  \ar[r] & \bD^{(\alpha)} \ar[r] & 0
}
\end{equation}
with exact rows. Applying the functor $\DCc$ and  taking into account
that $\DCc (\bD)$ and $\Dc (V^{(\alpha)})$ are isomorphic as $\varphi$-modules, we obtain the diagram
\begin{equation}
\label{diagram: properties of widetilde V}
\begin{aligned}
\xymatrix{
0\ar[r] &X\DCc (\bD) \ar[r] \ar@{^{(}->}[d] & \DCc (\widetilde \bD) \ar[r] \ar@{^{(}->}[d] &\DCc (\bD) \ar[r] \ar[d]^{\simeq} &0\\
0\ar[r] & X\Dc (V^{(\alpha)})\oplus X\Dc (V^{(\beta)}) \ar[r] &
\Dc (\widetilde V^{(\alpha)}) \ar[r] & \Dc (V^{(\alpha)})\ar[r] & 0.
}
\end{aligned}
\end{equation}

\begin{proposition}
\label{prop: properties of widetilde V}
\item[i)] The rows of the diagram \eqref{diagram: properties of widetilde V} are exact.
\item[ii)] $\widetilde V^{(\alpha)}$ is a crystalline representation with Hodge--Tate weights 
$(0, -(k-1),-(k-1))$. In particular, 
\begin{equation}
\nonumber
\dim_E\left (\Fil^{ k-1}\Dc (\widetilde V^{(\alpha)})\right )=2,
\qquad 
\dim_E \left (\Dc (\widetilde V^{(\alpha)})/\Fil^{  k-1}\Dc (\widetilde V^{(\alpha)})\right)=1.
\end{equation}
\item[iii)] One has 
$
X\Dc (V^{(\beta )}) \cap \Fil^{k-1}\Dc (\widetilde V^{(\alpha)})=\{0\}
$
and $X\DCc (\bD) \subset  \Fil^{k-1}\Dc (\widetilde V^{(\alpha)}).$

\item[iv)] One has 
\[
\widetilde\bD^{(\alpha)}=\left (\widetilde\bD+X\DdagrigE (V)\right )^{\mathrm{sat}}\,,
\]
where the superscript ``${\rm sat}$'' stands for the saturation of the indicated $(\varphi,\Gamma)$-module.

\end{proposition}

\begin{proof} 
\item[i)] It follows from its definition that $\widetilde\bD$ is crystalline of 
rank $2$ over $\CR_E.$ Since $\bD$ is crystalline of rank one
and the functor $\DCc$ is left exact, this implies the exactness of the upper row
of \eqref{diagram: properties of widetilde V}. 
The exactness of the bottom row follows from the left exactness of 
the functor $\Dc$ together with the exactness of the upper row.
 
\item[ii)] We deduce from the first portion that $\dim_E \Dc (\widetilde V^{(\alpha)})=\dim_E \widetilde V^{(\alpha)},$  so $\widetilde V^{(\alpha)}$ is indeed crystalline. Since the Hodge--Tate weights of $V^{(\alpha)}$ and $V^{(\beta)}$ are $-({k}-1)$ and $0$
respectively, the Hodge--Tate weights of $\widetilde V^{(\alpha)}$ are 
$(0,  -(k-1), -(k-1))$.
\item[iii)] The first assertion in this portion is clear since $X\Dc (V^{(\beta )})$ has Hodge--Tate weight $0$. To verify the second assertion, we note that the image of 
$X\DCc (\bD)$ under the vertical map in \eqref{diagram: properties of widetilde V} is  $X\Dc (V^{(\alpha)})$, since 
$\DCc (\bD)=\DCc (\bD^{\mathrm{sat}}) \simeq \Dc (V^{(\alpha)}).$  

\item[iv)] We have an inclusion of $(\varphi,\Gamma)$-modules of the same rank:
\[
\widetilde\bD+\DdagrigE (XV)\subset \widetilde \bD^{(\alpha)}.
\]
Since $\widetilde \bD^{(\alpha)}$ is saturated, this proves the assertion in this part.
 \end{proof}

\begin{remark}
\label{remark_bergdall}
Before moving ahead, we pause to elucidate that the 4-dimensional Galois representation $V_k$ is never crystalline at $p$. We learned the argument we record here from John Bergdall and we are grateful to him for explaining this fact to us. 

The infinitesimal deformation $V_k$ corresponds to a nonzero vector in the tangent space $T$ at the point $x_0$ of the Coleman--Mazur eigencurve. The tangent space $T$ injects into $H^1(\QQ_S/\QQ,{\rm ad}^0 V_{f})$. If $V_k$ were crystalline, the corresponding tangent vector would be crystalline at $p$, i.e. it would produce a \emph{nonzero} cohomology class in the Bloch--Kato Selmer group $H^1_{\rm f}(\QQ, {\rm ad}^0 V_{f})$. The main theorem of \cite{NewtonThorne2020} tells us that $H^1_{\rm f}(\QQ, ad^0 V_{f})=\{0\}$, and this proves that $V_k$ cannot be crystalline at $p$.
\end{remark}

\subsubsection{}
\label{subsubsec_412}
We shall assume the validity of the following condition:
\begin{itemize}
\item[\mylabel{item_C4}{\bf GP})] $\DCc (\widetilde\bD ) \not\subset\Fil^{k-1}\Dc (\widetilde V^{(\alpha) })$.
\end{itemize}
The condition \eqref{item_C4} means that the submodule $\widetilde\bD \subset \DdagrigE (\widetilde V^{(\alpha)})$ is in \emph{general position} relative to the canonical filtration on 
$\Dc (\widetilde V^{(\alpha)})$.  
\begin{remark}
\label{remark_22_2022_08_17_1607}
Suppose in this remark that the property \eqref{item_C4} fails. In this scenario, $\Dc(\widetilde\bD)$ would be weakly admissible and therefore $\widetilde V^{(\alpha)}$ would contain a sub-representation $W$ with Hodge--Tate weights $(-(k-1),-(k-1))$ such that $\Dc (W)=\Dc(\widetilde\bD)$. This in turn would mean that the exact sequence \eqref{eqn_2022_06_13_0652} splits and we would have 
\be\label{eqn_2022_06_13_06_58}  
\widetilde V^{(\alpha)} \simeq V^{(\beta)}\oplus V^{(\alpha)} \oplus V^{(\alpha)} \subset \widetilde V 
\ee
as $G_{\QQ_p}$-representations. 
\end{remark}
See \S\ref{subsec_245_2022_05_11_0845_subsubsec1} for an indirect reason why we expect \eqref{item_C4} to hold true in the context of the Coleman--Mazur eigencurve. 

\begin{proposition} 
\label{prop:definition of kappa}
Suppose that the condition \eqref{item_C4} holds and let $j$ be a fixed integer. 

\item[i)] For any $d \in \DCc (\widetilde\bD(\chi^j))$ there exists
a unique $\kappa_j (d)\in \Dc (V^{(\beta )} (j))$ such that 
\begin{equation}
\nonumber
d-X\kappa_j (d) \in \Fil^{k-1-j}\Dc (\widetilde V^{(\alpha)}(j)).
\end{equation}

\item[ii)] The map $d\mapsto \kappa_j(d)$ is $E$-linear and induces 
an isomorphism of one-dimensional filtered $E$-vector spaces:
\[
\overline \kappa_j \,: \,
\DCc (\bD (\chi^j)) \xrightarrow{\,\,\sim\,\,} \Dc (V^{(\beta)}(j)).
\]

\item[iii)] The diagram 
\begin{equation}
\nonumber
\xymatrix{
\DCc (\bD) \ar[d] \ar[r]^-{\overline \kappa_0} &\Dc (V^{(\beta)}) \ar[d]\\
\DCc (\bD (\chi^j))  \ar[r]_-{\overline \kappa_j} &\Dc (V^{(\beta)}(j)),
}
\end{equation}
where the vertical arrows are given by the canonical maps
$x\mapsto x\otimes d_j$, 
 commutes. 

\end{proposition}
\begin{remark} Since $\DCc (\bD) \simeq \Dc (V^{(\alpha)}),$ the proposition 
gives us isomorphisms (which we denote again by $\overline{\kappa}_j)$
\[
\overline \kappa_j \,: \,
 \Dc (V^{(\alpha)}(j)) \xrightarrow{\,\,\sim\,\,} \Dc (V^{(\beta)}(j)).
\]

\end{remark}

\begin{proof} 
The first assertion follows from 
Proposition~\ref{prop: properties of widetilde V} (iii) combined with \eqref{diagram: properties of widetilde V}.

The linearity of $\kappa_j$ is clear. 
The condition \eqref{item_C4} together with the containment $X\DCc (\bD) \subset  \Fil^{{k}-1}\Dc (\widetilde V^{(\alpha)})$ of Proposition~\ref{prop: properties of widetilde V}(iii) show that 
\be
\label{eqn_171_11_11_31}
\ker (\kappa_j)= X\DCc (\bD (\chi^j)).
\ee
This concludes the proof of the second assertion. The final portion follows immediately from the definitions. 
\end{proof}

\subsection{Deformation of the large exponential map}
\label{subsec: defortmation of exponential map}

\subsubsection{} In this section, we  relate different  large exponential maps associated  to the subspaces $V^{(\alpha)}$  and  $V^{(\beta)}$ using the infinitesimal deformation
$\widetilde V.$ Since, in general, $\widetilde V$ is not crystalline, we replace it with the smaller
representation $\widetilde V^{(\alpha)}$, which is crystalline. To compare the exponential maps $\Exp_{V^{(\beta)}(j),j}$ and $\Exp_{\widetilde\bD(\chi^j),j}$, we will use the canonical embeddings of their targets into the module $\CH (\Gamma)\otimes_{\Lambda_E}H^1_{\Iw}(\Qp, \widetilde V^{(\alpha)}(j)).$ 
The notation for the resulting composed maps will carry a `tilde' to stress that they take values in the cohomology of the deformation $\widetilde{V}^{(\alpha)}$.

The discussion below will be crucially used in the proof of Proposition~\ref{prop: comparison of exponentials for different eigenvalues}. Consider the large  exponential map for  $V^{(\beta)}$
\[
\Exp_{V^{(\beta)}(j),j}\,:\,
\mathfrak D (V^{(\beta)}(j)) 
\lra  \CH (\Gamma)\otimes_{\Lambda_E}H^1_{\Iw}(\Qp, V^{(\beta)}(j)),
\]
which we will denote simply by $\Exp_{\beta,j}.$
Let   
\begin{equation}
\label{eqn:definition of widetilde E beta}
\widetilde \Exp_{\beta,j}\,:\,
\mathfrak D (V^{(\beta)}(j)) 
\lra  \CH (\Gamma)\otimes_{\Lambda_E}H^1_{\Iw}(\Qp, \widetilde V^{(\alpha)}(j))
\end{equation}
denote the composition of this map with the injection
\[
\CH (\Gamma)\otimes_{\Lambda_E}H^1_{\Iw}(\Qp, V^{(\beta)}(j))
\hookrightarrow
\CH (\Gamma)\otimes_{\Lambda_E}H^1_{\Iw}(\Qp, \widetilde V^{(\alpha)}(j)),
\]
induced by the canonical inclusion  $$V^{(\beta)}\xrightarrow[{[X]}]{\sim} XV^{(\beta)} \subset XV \hookrightarrow\widetilde V^{(\alpha)}.$$

\subsubsection{} 
Analogously,   the large exponential map 
\begin{equation}
\nonumber
\Exp_{\widetilde\bD(\chi^j),j}\,:\, \fD (\widetilde\bD (\chi^j)) \rightarrow H^1_\Iw (\Qp,
\widetilde \bD (\chi^j))
\end{equation} 
will be denoted simply as $\Exp_{\widetilde\bD(\chi^j)}.$
Let us denote by 
\[
\widetilde \Exp_{\widetilde\bD (\chi^j)}\,:\, \mathfrak D(\widetilde\bD (\chi^j))\lra  \CH (\Gamma)\otimes_{\Lambda_E}H^1_{\Iw}(\Qp, \widetilde V^{(\alpha)}(j)) 
\]
the composition of $\Exp_{\widetilde\bD(\chi^j)}$ with the natural injection
\[
H^1_{\Iw}(\Qp, \widetilde\bD(\chi^j)) \hookrightarrow H^1_{\Iw}(\Qp,\widetilde\bD^{(\alpha)}
(\chi^j)) \stackrel{\sim}{\lra} \CH (\Gamma)\otimes_{\Lambda_E}H^1_{\Iw}(\Qp, \widetilde V^{(\alpha)}(j)).
\]  
\subsubsection{}
For any $n\in \mathbb{N}$, we let  $\widetilde \Exp_{\widetilde\bD(\chi^j),n}$
(resp.,  $\widetilde \Exp_{\beta,j,n}$) denote the composition of  $\widetilde \Exp_{\widetilde\bD (\chi^j)}$ (resp., $\widetilde \Exp_{\beta,j}$) with the Iwasawa theoretic projection 
$$\CH (\Gamma)\otimes_{\Lambda_E}H^1_{\Iw}(\Qp, \widetilde V^{(\alpha)}(j))
\lra H^1(K_n,\widetilde V^{(\alpha)}(j)).$$ 
These maps sit in the commutative diagram
\begin{equation}
\label{diagram: comparision of exponentials for different slopes}
\begin{aligned}
\xymatrix{
\mathfrak D (\widetilde V^{(\alpha)}(j)) 
\ar[ddrr]^{ \Exp_{\widetilde V^{(\alpha)}(j),n}} & &\,\,\,\mathfrak D (\widetilde\bD  (\chi^j)) 
\ar[dd]^{\widetilde \Exp_{\widetilde\bD (\chi^j),n}}  \ar@{_(->}[ll] \\
& & &\\
\mathfrak D (V^{(\beta)}(j)) 
\ar[rr]^-{\widetilde \Exp_{\beta, j,n}}
\ar@{^(->}[uu]&& H^1 (K_n, \widetilde V^{(\alpha)}(j)),
}
\end{aligned}
\end{equation}
where the diagonal map is induced by the large exponential map for $\widetilde V^{(\alpha)}(j).$

\subsubsection{}
Throughout the remainder of \S\ref{sect:eigenspace-transition}, we assume that $\alpha\neq p^{k-1}$ and $\beta\neq 1.$ In other words, we assume that $V^{(\alpha)}\neq E(1-k)$ and
$V^{(\beta)}\neq E(0).$
Note that this property holds automatically if $V$ is the $p$-adic realization of a pure motive of motivic weight $w\notin \{0, 2(k-1)\}$ that has good reduction at $p$.

Let $\rho$ be a finite character of $\Gamma$ of conductor $p^n.$ 
We treat  $\rho$ as a primitive character of $G_n.$ 
Define 
\begin{equation}
\label{eqn: definition of euler-like factors}
\begin{aligned}
&a(\rho,j)=
\begin{cases}
\displaystyle\left (1- \frac{p^{j-1}}{\beta }\right ) \left (1-\frac{\beta }{p^j}\right )^{-1}
, &\textrm{if $\rho=1$},\\
\alpha^n &\textrm{if $\rho\neq 1$},
\end{cases}
\\
&b(\rho,j)=
\begin{cases}
\displaystyle\left (1- \frac{p^{j-1}}{\alpha}\right ) \left (1-\frac{\alpha }{p^j}\right )^{-1}
, &\textrm{if $\rho=1$},\\
\beta^n &\textrm{if $\rho\neq 1$}.
\end{cases}
\end{aligned}
\end{equation}
Consider the maps 
\begin{equation}
\nonumber
\begin{aligned}
&\widetilde{\rmE}^{(\rho)}_{\alpha, j}\,:\,\mathfrak D (\widetilde\bD (\chi^j))
\lra H^1(\Qp, \widetilde V^{(\alpha)}(\rho^{-1}\chi^j)),\\
&\widetilde{\rmE}^{(\rho)}_{\alpha, j} (z)=
a(\rho, j) \,  \widetilde\Exp^{(\rho)}_{\widetilde\bD(\chi^j),n}(z)
\end{aligned}
\end{equation}
and 
\begin{equation}
\nonumber
\begin{aligned}
\nonumber
&\widetilde{\rmE}^{(\rho)}_{\beta, j}\,:\,\mathfrak D (V^{(\beta )}(j))
\lra H^1(\Qp, \widetilde V^{(\alpha)}(\rho^{-1}\chi^j)),\\
&\widetilde \rmE^{(\rho)}_{\beta, j} (z)=
b(\rho,j)\, \widetilde\Exp^{(\rho)}_{\beta,j,n}(z),
\end{aligned}
\end{equation}
where $\widetilde \Exp^{(\rho)}_{\widetilde\bD(\chi^j),n}$ and $\widetilde \Exp^{(\rho)}_{\beta, j,n}$ denote the  $\rho$-isotypic components of the underlying exponential maps (cf. \S\ref{Basic notation and conventions}).

 The following ``eigenspace-transition by differentiation'' principle is key to our main calculations.

\begin{proposition} 
\label{prop: comparison of exponentials for different eigenvalues}
 Let $(\widetilde V,\widetilde\bD)$ be a deformation of $(V,\bD)$  for which the condition \eqref{item_C4} is satisfied. Assume in addition that $\alpha \neq p^{k-1}$ and $\beta \neq 1.$
Then for any finite character $\rho$ and any    integer $1\leqslant j\leqslant  k-1$  the following diagram commutes:
\begin{equation}
\nonumber
\xymatrix{
& &
& \mathfrak D (\widetilde \bD (\chi^j))
\ar[d]^-{\widetilde\rmE^{(\rho)}_{\alpha, j}} 
\ar[dlll]_{  \kappa_j\otimes \mathrm{id}}
\\
\mathfrak D (V^{(\beta )}(j))
 \ar[rrr]_-{\widetilde\rmE^{(\rho)}_{\beta,j}}
& & & H^1(\Qp, \widetilde V^{(\alpha)}(\rho^{-1}\chi^j)).
}
\end{equation}

\end{proposition} 

\begin{proof} Let $z =f(\pi)\otimes d_{\delta\chi^j}\in \mathfrak D(\widetilde\bD (\chi^j))$
where $d_{\delta\chi^j}$ is a fixed generator of $\DCc (\widetilde \bD (\chi^j)).$
By Lemma~\ref{lemma from BenoisBerger2008}, 
\begin{equation}
\label{eqn: comutation of Xi-map}
\Xi_{\widetilde \bD (\chi^j),n}^{(\rho)}((\id\otimes \varphi)^n(z))=
\begin{cases}
d_{\delta\chi^j}\otimes f^{(\rho)}(\zeta_{p^n}-1)
&\text{if $n\geqslant 1$},\\
(1-p^{-1}\varphi^{-1})(1-\varphi)^{-1}(d_{\delta\chi^j}) 
&\text{if $n=0.$}
\end{cases}
\end{equation}
and analogously 
\begin{equation}
\nonumber
\Xi_{V^{(\beta)}(j),n}^{(\rho)}((\id\otimes \varphi)^n( \overline\kappa_j (z)))=
\begin{cases}
{\overline\kappa_j} (d_{\delta\chi^j})\otimes f^{(\rho)}(\zeta_{p^n}-1)
&\text{if $n\geqslant 1$},\\
(1-p^{-1}\varphi^{-1})(1-\varphi)^{-1} ({\overline\kappa_j}(d_{\delta\chi^j})) 
&\text{if $n=0.$}
\end{cases}
\end{equation}
 First, we consider the case $n\geqslant 1.$ Then formula \eqref{eqn: comutation of Xi-map}
gives:
\[
\Xi_{\widetilde\bD (\chi^j),n}^{(\rho)}((\id\otimes\varphi)^n(z))\,\,\equiv \,\,
X \cdot \Xi_{V^{(\beta)}(j),n}^{(\rho)}((\id\otimes\varphi)^n({{ \kappa_j}}(z))) 
\pmod{\Fil^{ k-1-j}\Dc (\widetilde V^{(\alpha)})\otimes K_n}\,.
\]
For $0\leqslant j\leqslant  k-1$, we have $\Fil^{ k-1-j}\Dc (\widetilde V^{(\alpha)}(j))
\subset  \Fil^{0}\Dc (\widetilde V^{(\alpha)}(j)).$  Since $\Fil^{0}\Dc (\widetilde V^{(\alpha)}(j))\otimes K_n$ is in the kernel of the exponential map, it follows that
\begin{equation}
\label{formula: equality for exponentials}
\exp_{\widetilde V^{(\alpha)}(j),K_n}\left (\Xi_{\widetilde\bD (\chi^j),n}^{(\rho)}((\id\otimes\varphi)^n(z))\right )=\exp_{\widetilde V^{(\alpha)}(j),K_n}\left (X\cdot \Xi_{{ }V^{(\beta)}(j),n}^{(\rho)}((\id\otimes \varphi)^n( { {\kappa_j}}(z)))\right ).
\end{equation}
Since the exponential map is $G_n$-equivariant, we deduce using 
 the diagrams  \eqref{eqn:Perrin-Riou commutative square} and  \eqref{diagram: comparision of exponentials for different slopes} that
\begin{multline}
\nonumber
\widetilde\alpha^n \widetilde \Exp^{(\rho)}_{\widetilde\bD (\chi^j),n} (z) {=}  \,  \widetilde \Exp^{(\rho)}_{\widetilde\bD (\chi^j),n} ((\id\otimes \varphi)^n (z)) \stackrel{\rm \eqref{diagram: comparision of exponentials for different slopes}}{=}
\Exp^{(\rho)}_{\widetilde V^{(\alpha)}(j),n}((\id\otimes\varphi)^n (z))\\
=
(j-1)! \exp_{\widetilde V^{(\alpha)}(j),K_n} (\Xi_{\widetilde\bD (\chi^j),n}^{(\rho)}((\id\otimes\varphi)^n(z))  ), \qquad n\geqslant 0,
\end{multline}
where $\widetilde \alpha \in \widetilde E$ is the eigenvalue of $\varphi$ on 
$\DCc (\widetilde \bD)$. Likewise,
\begin{multline}
\nonumber
\beta^n\, \widetilde \Exp^{(\rho)}_{\beta,j,n} ({ \kappa_j}(z))=   
 \widetilde \Exp^{(\rho)}_{\beta,j,n} ((\id\otimes\varphi)^n ({ \kappa_j}(z)))
\stackrel{\rm \eqref{diagram: comparision of exponentials for different slopes}}{=}
\Exp^{(\rho)}_{\widetilde V^{(\alpha)}(j),n}((\id\otimes\varphi)^n ({X {\kappa_j} }(z)))\\
= 
(j-1)! \exp_{\widetilde V^{(\alpha)}(j),K_n} (X\cdot \Xi_{V^{(\beta)}(j),n}^{(\rho)}((\id\otimes\varphi)^n({ { \kappa_j}}(z)))    ), \qquad n\geqslant 0.
\end{multline}

These together with \eqref{formula: equality for exponentials} yield
\begin{equation}
\label{formula: equality for exponentials first technical}
\widetilde\alpha^n\, \widetilde \Exp^{(\rho)}_{\widetilde\bD (\chi^j),n} (z)=
\beta^n\, \widetilde \Exp^{(\rho)}_{\beta,j,n} ({ \kappa_j} (z)).
\end{equation}
Since we have  $\widetilde\alpha (0)=\alpha \neq 0$, it follows that
\[
\widetilde \Exp^{(\rho)}_{\widetilde\bD (\chi^j),n} (z)\in X\cdot H^1(\Qp, \widetilde V^{(\alpha)}(\rho^{-1}\chi^j))
\]
as well. Since we have $\widetilde \alpha -\alpha \in XE[X]/(X^2)$ and $X^2$ annihilates $\widetilde{V}^{(\alpha)}$, this in turn shows that
\begin{equation}
\label{formula: equality for exponentials second technical}
\widetilde \alpha^n \widetilde \Exp^{(\rho)}_{\widetilde\bD (\chi^j),n} (z)=
\alpha^n \widetilde \Exp^{(\rho)}_{\widetilde\bD (\chi^j),n} (z).
\end{equation}
Now the case $\rho \neq \mathds{1}$ of the proposition follows on combining \eqref{formula: equality for exponentials first technical} and \eqref{formula: equality for exponentials second technical}.

 It remains to consider the case $\rho=\mathds{1}$. Formula \eqref{eqn: comutation of Xi-map} reads
\[
\begin{aligned}
&\Xi_{\widetilde\bD (\chi^j),0}^{(\mathds{1})}(z)= b(\mathds{1},j)\, d_{\delta \chi^j},\\ 
&\Xi_{V^{(\beta)}(j),0}^{(\mathds{1})}(\overline\kappa_j (z))=
a(\mathds{1},j) \, \overline\kappa_j (d_{\delta \chi^j}).
\end{aligned} 
\]
Therefore 
\[
a(\mathds{1},j)\, \Xi_{\widetilde\bD (\chi^j),0}^{(\mathds{1})}(z) \equiv 
X b(\mathds{1},j) \, \Xi_{V^{(\beta)}(j),0}^{(\mathds{1})}
\left (\overline\kappa_j (d_{\delta \chi^j})\right )
\pmod{\Fil^{ k-1-j}\Dc (\widetilde V^{(\alpha)})\otimes K_n},
\]
and 
\[
a(\mathds{1},j) \, \exp_{\widetilde V^{(\alpha)}(j),K_n}\left (\Xi_{\widetilde\bD (\chi^j),0}^{(\mathds{1})}(z)\right )
=
b(\mathds{1},j)\, \exp_{\widetilde V^{(\alpha)}(j),K_n}\left (X\cdot \Xi_{V^{(\beta)}(j),0}^{(\mathds{1})}
\left (\overline\kappa_j (d_{\delta \chi^j})\right )\right )
\]
for $0\leqslant j\leqslant k-1.$ From diagram \eqref{eqn:Perrin-Riou commutative square} we obtain that 
\[
a(\mathds{1},j) \, \widetilde\Exp_{\widetilde\bD (\chi^j),0}(z)=
 b(\mathds{1},j)\, \widetilde\Exp_{\beta, j,0}(z).
\]
The proposition is proved.

\end{proof}

\begin{corollary} 
We have a commutative diagram
\begin{equation}
\nonumber
\xymatrix{
& &
& \mathfrak D (V^{(\alpha)}(j))
\ar[d]
\ar[dlll]_{  \overline\kappa_j\otimes \mathrm{id}}
\\
\mathfrak D (V^{(\beta )}(j))
 \ar[rrr]_-{\widetilde\rmE^{(\rho)}_{\beta,j}}
& & & H^1(\Qp, \widetilde V^{(\alpha)}(\rho^{-1}\chi^j)),
}
\end{equation}
where the vertical map is induced by $\widetilde\rmE^{(\rho)}_{\alpha, j}.$
\end{corollary}

\begin{corollary}
\label{cor: transition for exponentials}
Under the assumptions of Proposition~\ref{prop: comparison of exponentials for different eigenvalues}, we have
\[
\mathrm{im} \left (\widetilde\rmE^{(\mathds{1})}_{\alpha, j}\right )=
H^1_{\rm f}(\Qp, X\cdot \widetilde V (\chi^j))=
H^1(\Qp, X  V^{(\beta)}(\chi^j))
.
\]
\end{corollary}
\begin{proof} This follows from  Proposition~\ref{prop: comparison of exponentials for different eigenvalues} and the fact that 
\[
\mathrm{im} \left (\widetilde\rmE^{(\mathds{1})}_{\beta, j}\right )=
H^1(\Qp, X V^{(\beta)}(\chi^j))=H^1_{\rm f}(\Qp, X\cdot \widetilde V (\chi^j)).
\]
\end{proof}

\section{Abstract critical $p$-adic $L$-functions}
\label{sec_abstract_setting}
\subsection{Algebraic preliminaries}
\label{subsec_alg_prelim}
\subsubsection{} Let $E$ be a finite extension of $\Qp$. Fix an integer  $k\geqslant 2$. For any auxiliary  fixed integer $r\geqslant 0$, we consider  
the affinoid $\cW=\Spm \left (E\left <Y/p^{2r}\right >\right )$ as the closed disc  about 
$k$  and radius $1/p^{2r}$. In more precise terms, if we let  $D(k,p^{m})=k+p^{m}\Zp$ denote the closed disc with center $k$ and radius $1/p^{m}$, we may then identify each  $y\in D(k,p^{2r})$  with the maximal ideal 
\begin{equation}
\nonumber
\mathfrak m_y= \left ((1+Y)- (1+p)^{ y-k} \right ).
\end{equation}
In particular, it allows to identify  the integers $k+p^{2r}\ZZ$ with a subset
of $\cW (E).$ Morally, $\cW$ should be viewed as  the weight space. 

We analogously consider the affinoid  $\cX=\Spm \left (E\left <X/p^{r}\right >\right )$ 
as a  closed disc of radius $p^r$ centered at some point  $x_0$. The map 
\[
\cO_{\cW}\lra \cO_{\cX}, \qquad Y\mapsto X^2
\]
induces the map $w\,:\,\cX \lra \cW$, which we refer as the weight map.  Note that  $w(x_0)=
k$. The weight map  is ramified at $x_0$ with ramification index $2$. The set
\[
\mathcal X^{\mathrm{cl}}(E):=\{x\in \mathcal X(E)  \mid w(x)\in \ZZ, {w(x) \geqslant 2}\}
\]
will be  called the set of classical points in $\mathcal X$.

\subsubsection{} Suppose that $M_\cX$ is a finitely generated free $\cO_\cX$-module of rank $d.$ We may consider $M_\cX$ also as a $\cO_\cW$-module, via the morphism $\cO_\cW\rightarrow \cO_\cX$.  For any  $x\in \cX (E)$ and $y\in \cW (E),$  we set 
\[
M_{x}=M_\cX\otimes_{\cO_\cX}\cO_\cX/\mathfrak m_{x}, \qquad 
M_{y}=M_\cX\otimes_{\cO_\cW,w }\cO_\cW/\mathfrak m_{y}.
\]
Note that  $M_{x}$ and $M_y$ are finitely dimensional $E$-vector spaces of dimensions $d$  
and $2d$, respectively. 
We denote by 
$\pi_x\,:\, M_{w(x)}\rightarrow  M_x$  the natural projection. 

Consider $X$ as a function on $\cX.$ For each $x\in \cX (E),$ let $X(x)$ denote the value
of $X$ at $x.$ Set
\[
M[x]:=\ker \left (X-X(x)\,:\, M_{w(x)} \rightarrow M_{w(x)} \right ).
\] 
Note that $M[x]$ is an $E$-vector space of dimension $d.$ The multiplication by $X+X(x)$
gives an isomorphism
\[
M_{x} \xrightarrow{\sim} M[x]. 
\]
In particular, for $x=x_0$ one has $M_{x_0}=M_\cX/XM_\cX,$ $M_{k}=M_\cX/X^2M_\cX$ and 
$M[x_0]=XM_\cX/X^2M_\cX.$ 

\subsubsection{}
\label{subsubsec_2113_12_11}
 Suppose that $M_\cX$ and $M_\cX^\prime$ are both free $\cO_\cX$-modules of 
finite rank equipped with an $\cO_\cW$-linear (sic!) pairing
\begin{equation}
\nonumber 
(\,,\,): M_\cX^\prime\otimes_{\cO_\cW} M_\cX\lra \cO_{\cW}
\end{equation}
satisfying the following condition: 
\begin{itemize}
\item[\mylabel{item_Adj}{\bf Adj})] For every $m^\prime\in M_\cX^\prime$ and $m\in M_\cX$, one has 
\begin{equation}
\label{eqn:condition Adj}
(X m^\prime,m)=(m^\prime,Xm).
\end{equation}
\end{itemize}
For any $x\in \cX (E),$ this pairing induces a  pairing
\begin{equation}
\nonumber 
(\,,\,)_{w(x)} : M_{w(x)}^\prime\otimes M_{w(x)}\lra E.  
\end{equation}
On restricting the second argument to $M[x]$ in this pairing, we obtain a unique pairing 
\[
(\,,\,)_{x} \,:\, M^\prime_{x}\otimes_E M[x]\lra E 
\]
such that the restriction of  $(\,,\,)_{w(x)} $ to  the subspace 
$M_{w (x)}^\prime\otimes_E M[x]$ factors as 
\begin{equation}
\label{eqn:factorization of pairing}
\begin{aligned}
\xymatrix@R=.4cm{M'_{w (x)}\otimes_E M [x]\ar[dr]_{\pi_{x}\otimes {\rm id}\,\,
}\ar[rr]^(.6){(\,,\,)_{w (x)}}&&E\\
& M^\prime_{x}\otimes_E M[x]\ar[ur]_(.6){(\,,\,)_{x}}}.
\end{aligned}
\end{equation}
We refer the reader to \cite[\S2.4]{BB_CK1_PR} for a proof and further details.

\subsubsection{}
Let $M_\cX$ be a $\cO_\cX$-module and let us put $M^\circ_\cX:=\cO_\cX \otimes_{\cO_\cW} M_\cX $,  endowing it with the structure of an $\cO_\cX$-module through the $\cO_\cX$-action on the first factor. We have 
a natural specialization map
\begin{equation}
\label{eqn: abstract specialization}
\mathrm{sp}_x \,:\, M^\circ_\cX \lra
M_{w(x)}\,
\end{equation}
which is defined as the composition
\[
\cO_\cX \otimes_{\cO_\cW} M\xrightarrow
{x \otimes {\rm id}} \cO_\cX/\mathfrak{m}_x\otimes_{\cO_\cW} M \simeq \cO_\cW/\mathfrak{m}_{w(x)}\otimes_{\cO_\cW} M= M_{w(x)}\,.
\]
In particular, we have $\left ( M^\circ_{\cX}\right )_x \simeq M_{w(x)}$.

We also recall the following result:

\begin{lemma} 
\label{lemma: abstract bellaiche lemma}
Assume that $M_\cX$ is a free  $\cO_\cX$-module of  rank one. 
Let $m$ denote a generator of $M_\cX$ and let 
\[
\Phi= 1\otimes Xm + X\otimes m \in M_\cX.
\]
Then: 
\begin{itemize}
\item[i)] We have $(X\otimes 1) \Phi= (1\otimes X) \Phi.$

\item[ii)]{} For any $x\in \cX (E),$ the element ${\rm sp}_x(\Phi)$ generates the $E$-vector 
space $M[x].$

\item[iii)] We have ${\rm sp}_{x_0}(\Phi)=Xm_{x_0}.$
\end{itemize}
\end{lemma}
\begin{proof} This is an abstract version of \cite[Lemma~4.13 \& Proposition 4.14]{bellaiche2012}. See also \cite[\S2.3]{BB_CK1_PR}. 
\end{proof}

\subsection{Triangulations}
\label{subsec:triangulation}
\subsubsection{} 
\label{subsubsec_1_2_1_2021_06_02}
We maintain  the notations of Chapter~\ref{chapter: preliminaries}. Let $V_\cX$  be a free $\cO_\cX$-module of rank $2$  endowed with a continuous action of $G_{\QQ,S}$.
According to our conventions in \S\ref{subsec_alg_prelim}, for any  $x\in \mathcal X(E)$, we set 
$V_{w(x)}=V_\cX/\mathfrak m_{w(x)}V_\cX$
and  $V_{x}=V_\cX/\mathfrak m_x V_\cX .$ 
Recall that   $V_k=V_\cX/X^{2}V_\cX,$  $V_{x_0}=V_\cX/XV_\cX,$  and $V[x_0]:=XV_k$. We shall assume that the Galois representation $V_\cX$ has the following properties:

\begin{itemize}
 \item[\mylabel{item_C1}{\bf C1})]
 For each $x\in \mathcal X^{\mathrm{cl}}(E),$  the restriction of $V_x$ on the decomposition group at $p$ is semistable of Hodge--Tate weights $(0,1-w (x)).$

 \item[\mylabel{item_C2}{\bf C2})]
 There exists $\alpha \in \cO_\cX$ such that for all $x\in \mathcal X^{\mathrm{cl}}(E)$ the eigenspace  $\Dst (V_x)^{\varphi=\alpha (x)}$ is one dimensional. 

 \item[\mylabel{item_C3}{\bf C3})]
 For each $x\in \mathcal X^{\mathrm{cl}}(E)-\{x_0\}$ 
\begin{equation}
\nonumber
\Dst (V_x)^{\varphi=\alpha (x)}\cap \Fil^{w (x)-1}\Dst (V_{x_0})=0\,.
\end{equation}
and
\begin{equation}
\nonumber\Dst (V_{x_0})^{\varphi=\alpha (x_0)}= 
\Fil^{k-1}\Dst (V_{x_0})\,.
\end{equation} 
\end{itemize}
 The prototypical example of $V_\cX$ will arise as the Galois representation afforded by the eigencurve in a neighborhood of a $\theta$-critical point.



\subsubsection{}
Let $\bD^\dagger_\cX (V)$ denote the $(\varphi,\Gamma)$-module
over $\CR_{\cX}$ associated to the restriction of $V$ on $G_{\Qp}.$ For each $x\in \cX (E),$
the specialization map identifies $\bD^\dagger_\cX (V)\otimes_{\cO_\cX,x}E$ with the 
$(\varphi,\Gamma)$-module $\DdagrigE (V_x)$ associated to the Galois representation $V_x.$

On shrinking $\mathcal{X}$ as necessary, it follows from  \cite{KPX2014} that one can construct a unique $(\varphi,\Gamma)$-submodule $\bD_\cX\subseteq \DdagrigX (V)$ of rank one satisfying the following properties:

\begin{itemize}
 \item[\mylabel{item_phiGamma1}{$\varphi\Gamma_1$)}]
 $\bD_\cX=\bD_\delta$ with $\delta\,:\,\Qp^*\rightarrow \cO_{\cX}^*$ such that $\delta (p)=\alpha$ and $\left.\delta \right \vert_{\Zp^*}=1.$ 

 \item[\mylabel{item_phiGamma2}{$\varphi\Gamma_2$)}]
 $\DCc (\bD_x)=\Dst (V_x)^{\varphi =\alpha (x)}$ for each $x\in \mathcal X^{\textrm{cl}}(E)$, where $\bD_{x}:=\bD_\cX\otimes_{\cO_\cX,x}E$.

 \item[\mylabel{item_phiGamma3}{$\varphi\Gamma_3$)}]
 For each $x\in \cX^{\textrm{cl}}(E)-\{x_0\},$  the $(\varphi,\Gamma )$-module $\bD_x$ is saturated in $\DdagrigE (V_x).$
\end{itemize}
Let $\bD_{x_0}^{\mathrm{sat}}$ denote the saturation of the specialization $\bD_{x_0}$  of   the $(\varphi,\Gamma )$-module $\bD_\cX$ at $x_0$. Then, it follows from the discussion in the final section of \cite{KPX2014} and in view of our running hypothesis \eqref{item_C3},  we have $\bD_{x_0}=t^{k-1}\bD_{x_0}^{\mathrm{sat}}$  and 
$ \DCc (\bD_{x_0}^{\mathrm{sat}})   = \Fil^{k-1}\Dst (V_{x_0})$.
Set $V:=V_{x_0},$ $\bD:=\bD_{x_0},$  $\widetilde V:=V_{k}$ and $\widetilde \bD:=\bD_k.$
Then $(\widetilde V, \widetilde \bD)$ is a deformation of 
$(V,\bD)$ in the sense of Section~\ref{subsect:deformation}.

\subsection{Abstract $p$-adic $L$-functions in families}
\label{subsec_214_18_11}
\subsubsection{}
\label{subsubsec_2141_12_11}
Suppose that we are given another  free $\cO_{\cX}$-module $V_\cX^\prime$ of rank two which is equipped with a continuous $G_{\QQ,S}$-action, together with a Galois equivariant  $\cO_{\cW}$-linear pairing 
\begin{equation}
\label{pairing V and V'}
(\,,\,): V_\cX^\prime\otimes V_\cX\lra \cO_{\cW}
\end{equation}
satisfying  condition \eqref{item_Adj}, namely that for every $v^\prime\in V_\cX^\prime$ and 
$v\in V_\cX$, we have $(Xv^\prime,v)=(v^\prime,Xv)$. 
On the level of Iwasawa  cohomology, it induces the $\LL_{\cW}$-bilinear pairing
(\ref{eqn:Iwasawa pairing}):
\begin{equation}
\label{eqn_pairing_Iw}
\left <\,\,,\,\,\right > \,:\, H^1_{\Iw}(\Qp,V_\cX^\prime (1)) 
\otimes_{\LL_\cW}  H^1_{\Iw}(\Qp,V_\cX)^\iota \longrightarrow \LL_{\cW}.
\end{equation}
(Here we consider $V_\cX'$  and $V_\cX$ as $\cO_{\cW}$-modules). Tensoring the first factor
with $\CH_\cW (\Gamma)$ and the second one with $\CH_\cX(\Gamma)$ (sic!), we obtain a
$\CH_\cW (\Gamma)$-bilinear map
\begin{equation}
\label{eqn_PR_pairing_Iw_A}
\left <\,\,,\,\,\right >_{\cX} \,:\,\left(\CH_{\cW}(\Gamma)\otimes_{\LL_\cW} 
H^1_{\Iw}(\Qp,V_\cX^\prime (1))\right)
\otimes  \left(\CH_{\cX}(\Gamma)\otimes_{\LL_\cW} H^1_{\Iw}(\Qp,V_\cX)^\iota\right) \longrightarrow 
\CH_{\cX}(\Gamma)\,.
\end{equation}

\subsubsection{}
 In this section we introduce  the formalism of Perrin-Riou's $p$-adic $L$-functions following \cite[\S4.2]{BB_CK1_PR}.  
We note that our treatment here concerns the technically more challenging $\theta$-critical case, whereas in op. cit. we have studied the non-$\theta$-critical scenario.
The following construction is inspired by Bella\"{\i}che's construction of secondary
$p$-adic $L$-functions using modular symbols  \cite{bellaiche2012}.
 
Recall that we set $\mathfrak D(\bD_\cX):= \cO_{E}[[\pi]]^{\psi=0}\otimes_{\cO_E}\DCc (\bD_\cX).$ 

Although $\mathfrak D(\bD_\cX)$ has a natural structure of $\cO_\cX$-module, we consider it as 
an $\cO_\cW$-module by restriction of scalars and set
\[
\mathfrak D(\bD_\cX)^{\circ}:=\cO_\cX \otimes_{\cO_{\cW}}\mathfrak D(\bD_\cX).
\]
Then $\mathfrak D(\bD)_{\cX}$ has a  natural structure of a $\cO_\cX$-module given
by the multiplication on the second factor. 
More explicitly, one can write:
\[
\mathfrak D(\bD_{\cX})^\circ\simeq
\cO_{\cX}[[\pi]]^{\psi=0}\otimes_{\cO_\cW}\DCc (\bD_\cX) \,.
 \]
 The large exponential map $\Exp_{\bD_\cX,h}$ can be extended by linearity to an $\cO_\cX$-linear map
\begin{equation}
\nonumber
\Exp_{\bD_\cX,h}^\circ\,:\,\mathfrak D(\bD_\cX )^\circ
\rightarrow 
\cO_{\cX}\otimes_{\cO_\cW} H^1_{\Iw}(\Qp, \bD_\cX).
\end{equation}
Composing this map with the injection 
\begin{equation}
\nonumber
\cO_{\cX} \otimes_{\cO_\cW}  H^1_{\Iw}(\Qp, \bD_\cX) \lra \cO_\cX \otimes_{\cO_\cW} \left(
\CH_\cW (\Gamma)\otimes_{\LL_\cW}  H^1_{\Iw}(\Qp,V_\cX)\right)= \CH_{\cX} (\Gamma)\otimes_{\LL_\cW}
H^1_{\Iw}(\Qp,V_\cX),
\end{equation}
we obtain a map (which we denote  by  the same symbol)
\begin{equation}
\label{eqn:Exponential from D to V}
\Exp_{\bD_\cX,h}^{\circ}\,:\,\mathfrak D(\bD_\cX)^\circ 
\lra 
 \CH_{\cX} (\Gamma)\otimes_{\LL_\cW}
H^1_{\Iw}(\Qp,V_\cX)
\,.
\end{equation}

\subsubsection{} 
\label{subsubsec_2143_18_11}
For any  $\cO_\cX$-module generator  $\eta \in \DCc (\bD_\cX),$  set
\begin{equation}
\begin{aligned}
\label{eqn_bbeta_def}&\bbeta:=X\otimes \eta +1\otimes (X\eta)
\in \DCc (\bD_\cX)^\circ,\\
&\widetilde \bbeta:=(1+\pi) \otimes  \bbeta 
\in \mathfrak D(\bD_{\cX})^\circ. 
\end{aligned}
\end{equation}
The reader is invited to compare this with the definition of the modular symbol $\Phi$ in \cite[\S4.3.3]{bellaiche2012}.

Let $c$ denote the unique element of  $\Gamma$ of order $2$, so that we have $c(\zeta_{p^n})=\zeta_{p^n}^{-1}$ for every natural number $n$.

\begin{defn}
\label{defn_fat_eta_etatilde}
\item[i)]
For each $ h\geqslant 0$ and cohomology class $\bz\in H^1_{\Iw}(\Qp, V_\cX^\prime (1))$,  we define
\begin{equation}
\nonumber
\begin{aligned}
&\Log_{\bD_\cX,\eta,1-h}\,:\,H^1_\Iw (\Qp, V_\cX^\prime (1))\lra \CH_\cX (\Gamma)\,,
\\
&\Log_{\bD_\cX,\eta,1-h} (\bz):=\left <\bz, c \circ \Exp_{\bD_\cX,h}(\widetilde\bbeta)^\iota \right >_{\cX}.
\end{aligned}
\end{equation}

\item[ii)] We define the two-variable $p$-adic $L$-function associated to $\bz$  setting
\[
L_{p,\eta}(\bz)=\Log_{\bD_\cX,\bbeta,1}(\bz) \in \CH_\cX (\Gamma).
\]
\end{defn}

\subsubsection{} In this subsection, we shall study the specializations of $L_{p,\eta}(\bz)$.
For any $x\in \cX (E)$, let us set 
\[
L_{p,\eta}(\bz,x):=L_{p,\eta}(\bz)_x \in \mathscr H(\Gamma)\,.
\]
By \cite[Lemma~2.3]{BB_CK1_PR}, the specialization $\bbeta_x$ of  $\bbeta$ is a generator of 
$\DCc (\bD_\cX)[x].$ 
The large exponential map given as the compositum
\[
\Exp_{\bD [x],h} \,:\,\mathfrak D(\bD [x])  
\lra 
H^1_{\Iw}(\Qp, \bD[x])\lra \CH (\Gamma)\otimes_{\LL_E} H^1_{\Iw}(\Qp, V[x])
\]
is compatible with \eqref{eqn:Exponential from D to V} in the evident sense, so that we have
\begin{equation}
\label{eqn_24_2021_06_02}
\Exp_{\bD_\cX,h}(\widetilde\bbeta)_x= \Exp_{\bD [x],h}(\widetilde\bbeta_x)\,. 
\end{equation}
Recall that for any $x\in \cX (E),$ the specialization of (\ref{pairing V and V'}) 
gives rise to  a pairing
$V'_{x}\otimes_E V[x]  \xrightarrow{(\,\,,\,\,)_x}  E.$ Therefore one has  a  pairing
on Iwasawa cohomology: 
\begin{equation}
\label{Iwasawa duality in families}
\begin{aligned}
&\left <\,\,,\,\,\right >_{x} \,:\, 
\CH (\Gamma)\otimes_{\LL_E} H^1_{\Iw}(\Qp,V_x^\prime (1)) \otimes  \left(\CH (\Gamma) \otimes_{\LL_E} H^1_{\Iw}(\Qp,V[x])^\iota\right)\lra  \CH (\Gamma).
\end{aligned}
\end{equation}
Since the pairing $(\,\,,\,\,)_x$ is induced by the pairing $(\,\,,\,\,)_{w(x)}\,:\, 
V_{w(x)}'\otimes V_{w(x)} \rightarrow E,$ the pairing \eqref{Iwasawa duality in families} is induced by the pairing
\begin{equation}
\label{eqn: Iwasawa pairing in w(x)}
\left <\,,\,\right >_{w(x)}\,:\, \left (\CH (\Gamma)\otimes_{\LL_E} H^1_\Iw (\Qp, V'_{w(x)}(1))
\right )
\otimes  
\left (\CH (\Gamma)\otimes_{\LL_E} H^1_\Iw (\Qp, V_{w(x)})\right )^\iota \lra \CH (\Gamma).
\end{equation}
These facts will be used repeatedly.

\subsubsection{}
\label{subsubsec_2235_1204}
We have a well defined map
\begin{equation}
\nonumber
\begin{aligned}
&\Log_{\bD_x,\eta_x,1-h}\,:\,H^1_\Iw (\Qp, V_{x}^\prime (1))\lra \CH_E (\Gamma)\,,\\
&\Log_{\bD_x,\eta_x,1-h} (z):=\left <z, c \circ \Exp_{\bD[x],h}(\widetilde\bbeta_x)^\iota \right >_{x}.
\end{aligned}
\end{equation}
For any $\bz \in H^1_\Iw (\Qp, V_\cX^\prime (1))$, set
\[
L_{p,\eta_x}(\bz_{x}):=\Log_{\bD_x,\eta_x,1}(\bz_x),
\]
where $\bz_x\in H^1_\Iw (\Qp, V_{x}^\prime (1))$ is the specialization of $\bz.$
The  proposition below  follows immediately from the above discussion and the functoriality 
of cup-products:

\begin{proposition} 
\label{prop_2_5_18_11}
For any $x\in \cX(E),$ one has
\begin{equation}
\nonumber
L_{p,\eta}(\bz,x)=L_{p,\eta_x}(\bz_{x}).
\end{equation}
\end{proposition}

 


\subsection{Specialization at $x_0$}

\subsubsection{}
For any continuous  character $\rho\,:\,\Gamma \rightarrow E^*$, let
\[
L_{p,\eta}(\bz, x, \rho ):= \rho \circ L_{p,\eta}(\bz,x).
\]
denote the value of $L_{p,\eta}(\bz,x)$ at $\rho.$
We also define  analytic $p$-adic $L$-functions
\[
L_{p,\eta}(\bz, x, \rho, s):= L_{p,\eta}(\bz, x, \rho \left <\chi\right >^s)\,.
\]
In this subsection, we are interested in the specialization  of $L_{p,\eta}(\bz)$ at $x_0.$

From   definition,
it follows that 
$$ L_{p,\eta}(\bz)=\left <\bz, c \circ \Exp_{\bD,0}(X\widetilde\eta)^\iota \right >_{\cX} +X\left <\bz, c \circ \Exp_{\bD,0}(\widetilde\eta)^\iota \right >_{\cX}.
$$
Hence, 
\[
L_{p,\eta}(\bz, x_0)=\left <\bz_{x_0}, c \circ \Exp_{\bD[x_0],0}(X\widetilde\eta_{x_0})^\iota \right >_{x_0}\,.
\]
Note that $\bD[x_0]$ is canonically isomorphic to the $(\varphi,\Gamma)$-module
associated to $V^{(\alpha)}_{x_0},$ and therefore 
\[
L_{p,\eta}(\bz, x_0)=\left <\bz_{x_0}, X \left ( c \circ \Exp_{V_{x_0}^{(\alpha)},0}(\widetilde\eta_{x_0})^\iota \right )\right >_{x_0}.
\]




\subsubsection{}

\begin{proposition}
\label{prop_bellaiche_formal_step_1}
\item[i)] The $p$-adic $L$-function $L_{p,\eta}(\bz, x_0)$
can be written in the form
\[
L_{p,\eta}(\bz, x_0)= \left (\underset{i=1-k}{\overset{-1}\prod} \ell_{i}^{\iota}\right ) \cdot L_{p,\eta}^{\mathrm{imp}}(\bz_{x_0}),
\]
where $L_{p,\eta}^{\mathrm{imp}}(\bz_{x_0})\in \Lambda_E$ and $\ell_i^{\iota}=i+\log (\gamma_1)/\log\chi (\gamma_1).$ 

\item[ii)] For any finite character $\rho \in X(\Gamma),$ the analytic $p$-adic  function 
$L_{p,\eta}(\bz, x_0,\rho,s) $ can be written in the form
\[
L_{p,\eta}(\bz, x_0,\rho,s) =\underset{i=1}{\overset{k-1}\prod} (s-i) \cdot
L_{p,\eta}^{\mathrm{imp}}(\bz, x_0,\rho,s),
\]
where the function  $L_{p,\eta}^{\mathrm{imp}}(\bz, x_0,\rho,s) $ is bounded
on the open unit disc.
\end{proposition}
\begin{proof} 
Since $V^{(\alpha)}$ is a one-dimensional representation with Hodge--Tate weight $-(k-1)$, 
the Perrin-Riou exponential map is defined over $\LL_E$; cf. \eqref{eqn: integral exponentials}:
\[
\Exp_{V_{x_0}^{(\alpha)},1-k}\,:\, \mathfrak D(V_{x_0}^{(\alpha)}) \lra H^1_\Iw (\Qp, V_{x_0}^{(\alpha)}).
\]
It follows from the fundamental relation \eqref{eqn: relation Exp for h and h+1} that
\[
\Exp_{V_{x_0}^{(\alpha)},0}= \left (\underset{i=1-k}{\overset{-1}\prod} \ell_{i}\right )
\Exp_{V_{x_0}^{(\alpha)},1-k}.
\]
Let us set
\[
L_{p,\eta}^{\mathrm{imp}}(\bz_{x_0}):=
 \left <\bz_{x_0}, X\left (c\circ \Exp_{V_{x_0}^{(\alpha)},1-k}(\widetilde\eta_{x_0})^\iota\right )\right >_{x_0}\,.
\]
Then $L_{p,\eta}^{\mathrm{imp}}(\bz_{x_0})\in \Lambda_E $,
and the first assertion follows from the bilinearity of $\left <\,,\,\right >_{x_0}.$ The second assertion is its immediate consequence. 
\end{proof}

\begin{corollary}
For any integer $1\leqslant j\leqslant k-1$ and any character $\rho\in X(\Gamma)$ of finite order,  
\[
L_{p,\eta}(\bz, x_0,\rho\chi^j)=0.
\]
\end{corollary}

\subsection{Secondary $p$-adic $L$-function}
\label{subsec: secondary function}
\subsubsection{}
In the spirit of Bella\"{\i}che's approach   and in view of Proposition~\ref{prop_bellaiche_formal_step_1}(ii), we introduce  the following secondary $p$-adic $L$-function.
\begin{defn} 
\label{defn_secondary_padicL_abstract}
The function 
\begin{equation}
\nonumber
L^{[1]}_{p,\eta}(\bz,x_0):= \left.\left( \frac{\partial}{\partial X} L_{p,\eta}(\bz)\right)\right\vert_{X=0}
\end{equation}
is called the secondary $p$-adic $L$-function attached to $\bz.$
\end{defn} 
By definition,
\begin{equation}
\label{eqn_Lpimproved_recast}
L^{[1]}_{p,\eta}(\bz,x_0)=\left.\left <\bz, c \circ \Exp_{\bD,0}(\widetilde\eta)^\iota \right >\right\vert_{X=0}\,.
\end{equation}

\subsubsection{} We explain how the constructions above are related to  the scenario studied in Section~\ref{sect:eigenspace-transition}. Recall that we have set $V:=V_{x_0},$ $\bD:=\bD_{x_0},$   $\widetilde V:= V_k,$ $\widetilde \bD :=\bD_k$, and considered 
$(\widetilde V, \widetilde \bD )$ as a deformation of  $(V,\bD).$ Let us also set $V':= V'_{x_0}$. With these identifications in mind, the pairing \eqref{eqn: Iwasawa pairing in w(x)} then reads
\begin{equation}
\label{eqn: Iwasawa pairing for widetilde V}
\left <\,,\,\right >_{k}\,:\, \left (\CH (\Gamma)\otimes_{\LL_E} H^1_\Iw (\Qp, \widetilde V'(1))
\right )
\otimes  
\left (\CH (\Gamma)\otimes_{\LL_E} H^1_\Iw (\Qp, \widetilde V)\right )^\iota \lra \CH (\Gamma)\,.
\end{equation}
Then,
\[
L^{[1]}_{p,\eta}(\bz,x_0)=\left <\bz_k, c \circ \Exp_{\widetilde\bD,0}(\eta_k)^\iota \right >_k\,.
\]

\subsubsection{}\label{subsubsec_2022_05_10_1535} 
Our main result concerning the secondary $p$-adic $L$-function $L^{[1]}_{p,\eta}(\bz,x_0)$ relates its special values in the range $1\leqslant j\leqslant k-1$ to the special values of the $p$-adic $L$-function associated to the slope-$0$ eigenvalue $\beta$. Throughout \S\ref{subsubsec_2022_05_10_1535}, we assume that $(\widetilde V,\widetilde D)$ satisfies the condition \eqref{item_C4}.

\begin{defn}
We define the slope zero one variable  $p$-adic $L$-function $L_{p,\kappa_0(\eta)}(\bz_{x_0})$ as
\begin{equation}
\label{eqn_slope_zero_padic_L}
L_{p,\kappa_0(\eta)}(\bz_{x_0}):=\left <\bz_{x_0}, c\circ \widetilde \Exp_{\beta,0}
\left (\kappa_0 (\eta_{k})\otimes (1+\pi)\right )^{\iota}\right >_{x_0}
\end{equation}
where the map $\kappa_0$ is given as in Proposition~\ref{prop:definition of kappa} and 
the map $\widetilde \Exp_{\beta,0}$ is defined by (\ref{eqn:definition of widetilde E beta}).  
\end{defn}

As usual, for a character $\rho  \in X(\Gamma)$ of finite order and $j\in \ZZ,$ we put
\[
L^{[1]}_{p,\eta}(\bz,x_0,\rho \chi^j):=\rho \chi^j \circ L^{[1]}_{p,\eta}(\bz,x_0)\,, \qquad 
L_{p,\kappa_0(\eta)}(\bz_{x_0},\rho \chi^j ):= \rho \chi^j\circ L_{p,\kappa_0(\eta)}(\bz_{x_0}).
\]
The following assertion should be compared to the discussion in \cite[\S4.4]{bellaiche2012}.

\begin{theorem}
\label{prop_imoroved_padicL_vs_slope_zero_padic_L}
 Let $\rho\in X(\Gamma)$ be a finite character of
conductor $p^n.$ Then for each integer $1\leqslant j\leqslant  k-1$ we have
\begin{equation}
\nonumber
 L^{[1]}_{p,\eta}(\bz,x_0,\rho \chi^j)=
\left (\frac{b(\rho,j)}{a(\rho,j)}\right )^{n} L_{p,\kappa_0(\eta)}(\bz_{x_0},
\rho\chi^j),
\end{equation}
where $a(\rho,j)$ and $b(\rho,j)$ are given in  \eqref{eqn: definition of euler-like factors}.
\end{theorem}
\begin{proof}
For any finite character $\rho$ and $j\in\ZZ$, we have a canonical pairing
\[
\left (\,,\,\right )_{\rho \chi^j}\,:\, H^1 (\Qp, \widetilde V'(\rho^{-1}\chi^{1-j}))
\otimes  
 H^1 (\Qp, \widetilde V(\rho \chi^j))\lra E\,.
\]
Let us set $\eta [j]:=t^{-j}\eta \otimes \chi^j  \in \DCc (\bD (\chi^j))$ and $\widetilde \eta =\eta_{k} \otimes (1+\pi).$ Formula \eqref{eqn: rho component of Iwasawa pairing} yields
\begin{equation}
\nonumber
L^{[1]}_{p,\eta}(\bz,x_0,\rho \chi^j)=
(\rho \chi^j) \circ
\left <\bz_{k}, c\circ \Exp_{\widetilde\bD,0}({ \widetilde \eta_{k}})^{\iota}\right >_{k}=
(-1)^j\left  ( \Tw_{-j} (\bz_{k})^{(\rho)}, c\circ \widetilde\Exp^{(\rho^{-1})}_{\widetilde\bD (\chi^j)}\bigl ({ \widetilde \eta_{k}[j]}\bigr  )^{\iota}\right )_{\rho \chi^j}\,,
\end{equation}
and 
\begin{align*}
L_{p,\kappa_0(\eta)}(\bz_{x_0},\rho \chi^j)&=(\rho \chi^j) \circ  \left <\bz_{x_0}, c\circ \widetilde\Exp_{\beta,0}
\left ({\kappa_0 (\widetilde \eta_{k}})\right )^{\iota}\right >_{x_0}
\\
&=
(-1)^j \left <\Tw_{-j}(\bz_{x_0})^{(\rho)}, c\circ 
\widetilde\Exp^{(\rho^{-1})}_{\beta, j,n}
\left ({\kappa_j \bigl (\widetilde \eta_{k}[j]} \bigr )\right )^{\iota}\right >_{\rho\chi^j}.
\end{align*}
By Proposition~\ref{prop: comparison of exponentials for different eigenvalues}, 
\[
a(\rho,j)^n \, \widetilde\Exp^{(\rho^{-1})}_{\widetilde \bD (\chi^j)}({ \widetilde \eta_{k}[j]})=
b(\rho,j)^n\,  \widetilde\Exp^{(\rho^{-1})}_{\beta, j,n}
\left ({\kappa_j (\widetilde \eta_{k}[j]})\right ).
\]
The proof of our proposition follows from these identities.
\end{proof}

\section{Modular symbols and $p$-adic $L$-functions }
\label{sec_new_2_3_2022_03_14}
In the present subsection, we review the constructions of $p$-adic $L$-functions attached to eigenforms via the theory of modular symbols. 

For any module $M$ over an unspecified Hecke algebra $\mathbb T$ and any eigenform $f$, we denote by $M[f]$ (resp. by $M_{(f)}$) the $f$-isotypic eigenspace (resp. the generalized $f$-eigenspace) in $M$. 


\subsection{The motive attached to $f$}
\label{subsect: the motive attached to f}

\subsubsection{} Let $Y_1(N)$ denote the modular curve associated to the congruence subgroup $\Gamma_1(N).$
We denote by $\lambda\,:\,\mathcal E_N\rightarrow Y_1(N)$ the universal elliptic curve over $Y_1(N)$ and set $\omega:=\lambda_*\Omega^1_{\mathcal E_N/Y_1(N)}.$ 
Let $f=\underset{n=0}{\overset{\infty}\sum} a_nq^n$ be a modular form of weight $k$, level $N$ and  nebentypus $\varepsilon_f.$
The motive $\mathcal M(f)$ associated to $f$ has the following explicit description in terms 
of realizations:

\begin{itemize}
\item{\it The Betti realization} $\mathcal M_{\mathrm B}(f)$ of $\mathcal M(f)$ is the  $f$-isotypic eigenspace 
\[
H^1\left (Y_1(N)(\mathbb C), \mathrm{Sym}^{k-2} \left (R^1\lambda_*(\ZZ)_{\QQ}\right ) \right )[f]\subset H^1\left (Y_1(N)(\mathbb C), \mathrm{Sym}^{k-2} \left (R^1\lambda_*(\ZZ)_{\QQ}\right ) \right ).
\]

\item{\it The de Rham realization} $\mathcal M_{\mathrm{dR}}(f)$ of $\mathcal M(f)$ is the  $f$-isotypic eigenspace in 
\[
H^0\left (Y_1(N), \Omega^1_{Y_1(N)} \otimes \omega^{k-2}\right )\oplus 
\overline{H^0\left (Y_1(N), \Omega^1_{Y_1(N)} \otimes \omega^{k-2}\right )}.
\]

\item{\it  For each prime number $p$,  the $p$-adic realization} 
$\mathcal M_{p}(f)$ of $\mathcal M(f)$ is the  $f$-isotypic component  
\[
H^1_{\mathrm{ {\acute e}t,c}}\left (Y_1(N)_{\overline \QQ}, \mathrm{Sym}^{k-2}
\left (R^1\lambda_*(\Zp)_{\Qp}\right )  \right )[f]\subset H^1_{\mathrm{ {\acute e}t,c}}\left (Y_1(N)_{\overline \QQ}, \mathrm{Sym}^{k-2}
\left (R^1\lambda_*(\Zp)_{\Qp}\right )  \right ).
\]
\end{itemize}
The space of cuspidal forms $S_k(N)$ injects into $H^0(Y_1(N), \Omega^1_{Y_1(N)} \otimes \omega^{k-2})$ via the Eichler--Shimura map, and we consider $S_k(N) [f]$ as a subspace
of $\mathcal M_{\mathrm{dR}}(f)$ through this injection.

\subsubsection{} 
\label{subsubsec_2311_2023_07_05}
Let $f^*:=\underset{n=1}{\overset{\infty}\sum} \overline a_nq^n$ denote the dual modular form. There exists a canonical pairing
\[
\mathcal M(f^*) \otimes \mathcal M (f) \lra \QQ (k-1)
\]
which induces pairings between the realizations of $\mathcal M(f^*)$ and $\mathcal M(f).$ The complex conjugation acts on $\mathcal M_{\mathrm B}(f)$, and we fix an eigenbasis
\[
\xi:=\{\xi_+,\xi_-\},\qquad \xi_{\pm}\in \mathcal  M_{\mathrm B}(f)^\pm
\]
for its action.

Let $\xi^*:=\{\xi^*_+,\xi^*_-\}$ denote the dual basis of $\mathcal M_{\mathrm{B}}(f^*).$ 
The comparison isomorphism
\[
\mathrm{comp}_{\mathrm{dR},\infty}\,:\, M_{\mathrm{dR}}(f^*)\otimes \mathbb{C}
\simeq M_{\mathrm{B}}(f^*)\otimes \mathbb{C}
\]
gives rise to  complex periods $\Omega^\pm_{f}$ 
defined by
\[
\mathrm{comp}_{\mathrm{dR},\infty} (f^*)= \Omega_{f}^+\xi^*_+ +\Omega_{f}^-\xi^*_-.
\]
We remark that this normalization of $\Omega^\pm_{f}$ agrees with \cite{kato04}.

\subsection{$p$-adic $L$-functions}
\label{subsec_modular_symbols_2022_03_23_16_13}

\subsubsection{}
Recall that we have fixed a prime number $p \nmid N.$  The Hecke polynomial 
$X^2-a_pX+\varepsilon_f(p)p^{k-1}$ has two nonzero roots which we shall denote by $\alpha$
and $\beta$. We will assume that $\alpha\neq \beta$ (this conjecturally always holds). Without loss of generality, we may and will assume that the roots are ordered in such a way that 
\[
v_p(\alpha)\geqslant v_p(\beta).
\]
Denote by  
\[
f_\alpha(q):= f(q)-\beta f(q^p), \qquad f_\beta(q):= f(q)-\alpha f(q^p)
\]
the stabilizations of $f$ with respect to $\alpha$ and $\beta.$

Put $\Gamma_p:=\Gamma_1 (N) \cap\Gamma_0(p)$. For a field $F$, we let $\mathscr P_{k-2}(F)$ be the space of polynomials over $F$ of degree $\leqslant k-2$ equipped with the canonical left action of $\mathrm{GL}_2(\ZZ)$. We denote by $\mathscr V_{k-2}(F)$ the $F$-linear dual of $\mathscr P_{k-2}(F)$. Recall  the space of complex-valued modular symbols
\[
\mathrm{Symb}_{\Gamma_p}(\mathscr V_{k-2}(\mathbb C)):=
{\Hom}_{\Gamma_p}(\mathrm{Div}^0(\mathbb P^1(\QQ)), \mathscr V_{k-2}(\mathbb C))\,. 
\]
The complex conjugation acts on this space, and we will denote by $\mathrm{Symb}_{\Gamma_p}^\pm(\mathscr V_{k-2}(\mathbb C))$ the eigenspaces for this action.

Recall that the Eichler--Shimura integrals of $f_\alpha$ \footnote{The modular symbols below depend on the choice of the root $\alpha$, but we suppress that dependence from our notation.}
give rise to a pair of complex modular symbols
\[
\varphi_{\infty}^\pm \in \mathrm{Symb}^\pm_{\Gamma_p}(\mathscr V_{k-2}(\mathbb C)) [f_\alpha]. 
\]
It follows from the theorem of Shimura and Manin that the modular symbols  
\begin{equation}
\nonumber
\begin{aligned}
&\varphi_\xi^+:=  \frac{(2\pi i)^{k-1}}{2 \Omega_f^{(-1)^k}} \cdot {\varphi_{\infty}^+},\\
&\varphi_\xi^-:= - \frac{(2\pi i)^{k-1}}{2 \Omega_f^{(-1)^{k-1}}} \cdot {\varphi_{\infty}^-}
\end{aligned}
\end{equation}
are rational and therefore, can be considered as elements of $\mathrm{Symb}^\pm_{\Gamma_p}(\mathscr V_{k-2}(E))[f_\alpha ]$ for a sufficiently large extension $E$ of $\Qp$. 

Let $\DD^\dagger_{k-2}(E)$ denote the space of overconvergent distributions of weight $k-2$
(cf. \cite{PollackStevensJLMS, bellaiche2012}).
The evaluation of distributions on polynomials of degree $\leqslant k-2$ induces 
a Hecke equivariant  projection map
\begin{equation}
\label{eqn_2022_03_23_15_51}
   \rho_{k-2}\,:\, \mathrm{Symb}_{\Gamma_p} (\DD^\dagger_{k-2}(E)) \lra 
\mathrm{Symb}_{\Gamma_p}(\mathscr V_{k-2}(E))\,.
\end{equation}

\subsubsection{} Let us first assume that $f_\alpha$ does not have critical slope, i.e. $v_p(\alpha)<k-1$. By Stevens' control theorem \cite[Theorem~5.4]{PollackStevensJLMS}, the map \eqref{eqn_2022_03_23_15_51} induces an isomorphism on the slope $<k-1$ subspaces
\[
\rho_{k-2}\,:\, \mathrm{Symb}_{\Gamma_p} (\DD^\dagger_{k-2}(E))^{<k-1} \xrightarrow{\sim} 
\mathrm{Symb}_{\Gamma_p}(\mathscr V_{k-2}(E))^{<k-1}.
\]
Therefore, there exists a unique pair of overconvergent modular symbols $\Phi_\xi^\pm\in \mathrm{Symb}_{\Gamma_p} (\DD_{k-2}^\dagger (E))^{<k-1}$  such that 
\[
\rho_{k-2}(\Phi_\xi^\pm)=\varphi_\xi^\pm.
\]
We define the $p$-adic distributions 
\[
\mu^\pm_{\mathrm{S},\xi} := \Phi_\xi^\pm ((\infty)-(0)).
\]
We recall that to each distribution $\mu$ on $\Zp^*$, one can associate its Amice transform in the ring $\mathscr R_E^+$ of power series in $\pi$ converging on the open unit disc:
\[
\mathscr A_{\mu}(\pi):=\int (1+\pi)^x d \mu (x) \in (\mathscr R_E^+)^{\psi=0},
\]
where $\psi$ denotes the operator \eqref{eqn: definition of psi}. Since $(\mathscr R_E^+)^{\psi=0}=\mathscr H(\Gamma) \cdot (1+\pi),$ there exists 
a unique element $\mathrm{Mel}(\mu)\in \mathscr H(\Gamma)$ such that 
\[
\mathscr A_{\mu}(\pi)=\mathrm{Mel}(\mu )\cdot  (1+\pi).
\] 
We  call $\mathrm{Mel}(\mu)$ the Mellin transform of $\mu .$

The $p$-adic $L$-function associated to the distribution  
\begin{equation}
\label{eqn: definition of Stevens distribution}
\mu_{\mathrm{S},\xi}:=\mu^+_{\mathrm{S},\xi}+\mu^-_{\mathrm{S},\xi}
\end{equation}
is defined as the Mellin transform of $x^{-1}\mu_{\mathrm{S},\xi}$:
\begin{equation}
\label{eqn: definition Stevens L}
L_{\mathrm{S},\alpha} ( f,\xi) :=\mathrm{Mel}(x^{-1}\mu_{\mathrm{S},\xi}).
\end{equation}
A standard computation with modular symbols shows that $L_{\mathrm{S},\alpha} ( f,\xi)$
coincides with the classical $p$-adic $L$-function of Manin--Vishik. In particular, it satisfies the following interpolation property: for any
finite character $\rho \in X(\Gamma)$, we have
\begin{equation}
\label{eqn: interpolation property in non critical case}
L_{\mathrm{S},\alpha} ( f,\xi, \rho \chi^j)
=(j-1)! \cdot e_{p,\alpha}(f,\rho,j) \cdot  \frac{L (f,\rho^{-1},j)}{(2\pi i)^{j+1-k}\cdot\Omega_f^\pm}
,\qquad 1\leqslant j\leqslant k-1,
\end{equation}
where 
\begin{equation}
\label{eqn: interpolation factor e}
e_{p,\alpha}(f,\rho,j)=
\begin{cases}
\left (1-\frac{p^{j-1}}{\alpha}\right )\cdot \left (1-\frac{\beta}{p^{j}}\right )
&\textrm{if $\rho=\mathds{1}$},
\\
\frac{p^{nj}}{\alpha^n \tau (\rho^{-1})}
&\textrm{if $\rho$ has conductor $p^n.$}
\end{cases}
\end{equation}

\subsubsection{} Let us now assume that $v_p(\alpha)=k-1$ and that $f_\alpha$ is non-$\theta$-critical, namely that it is not in the image of the $\theta$-operator $\theta_k:= \left (\frac{d}{dq}\right )^{k-1}$. Pollack and Stevens \cite{PollackStevensJLMS}  proved that in this case the generalized eigenspace  $\mathrm{Symb}_{\Gamma_p} (\DD^\dagger_{k-2}(E))_{(f_\alpha)}$ is one-dimensional, and therefore the projection map 
\[
\rho_{k-2}\,:\,\mathrm{Symb}_{\Gamma_p} (\DD^\dagger_{k-2}(E))_{(f_\alpha)} \lra 
\mathrm{Symb}_{\Gamma_p}(\mathscr V_{k-2}(E))_{(f_\alpha)}
\]
is an isomorphism. This shows, as in the non-critical-slope case, that there exist unique lifts $\Phi_\xi^\pm$ of the pair of modular symbols $\phi^\pm_\infty$.  The distribution $\mu_{S,\xi}$ and the $p$-adic $L$-function $L_{{\mathrm S},\alpha}(f,\xi)$ are then defined via the equations \eqref{eqn: definition of Stevens distribution} and \eqref{eqn: definition Stevens L}, respectively.

\subsection{Modular symbols in families and $p$-adic $L$-functions on the eigencurve}
\label{subsect: modular symbols in families}
\subsubsection{} Let $\mathcal C^{\mathrm{cusp}}$ denote the cuspidal eigencurve of level $N$
constructed with the Hecke operators $U_p$, $T_{\ell}$ and $\left <\ell\right >$ for 
$(\ell, Np)=1.$ Let $x_0\in \mathcal C^{\mathrm{cusp}}$ denote the point on
$\mathcal C^{\mathrm{cusp}}$ which corresponds to the stabilisation  $f_\alpha^*$ of $f^*.$ 
We will denote by $\cX$ a sufficiently small affinoid disc centered at $x_0$
and by $\cW$ an affinoid disc centered at $k$ in the weight space. Shrinking these discs, one can always assume that the weight map $w\,:\,\cX \rightarrow \cW$ is surjective. Moreover it is \'etale if $x_0$ is not $\theta$-critical (cf. \cite[Proposition 2.11]{bellaiche2012}). For any $x\in \cX^{\mathrm{cl}}(E),$ we denote by 
$f_x$ the associated eigenform and define the function $\alpha (x)\in \cO_\cX$ by $U_p(f_x)=\alpha (x) f_x.$ In particular, $f_{x_0}=f_\alpha$ and $\alpha (x_0)=\alpha.$

\subsubsection{} 
Let $\mathrm{Symb}_{\Gamma_p}(\bbD^\dagger (\cW))$ denote the module of 
$\cO_\cW$-valued overconvergent modular symbols \cite{bellaiche2012} (see also \cite[\S5]{BB_CK1_PR}). The operator $U_p$  acts on $\mathrm{Symb}_{\Gamma_p}(\bbD^\dagger (\cW))$
and we denote by $\mathrm{Symb}_{\Gamma_p}(\bbD^\dagger (\cW))^{\leqslant \nu}$ the submodule
on which $U_p$ acts with slope at most $\nu\in \mathbb{R}$. The complex conjugation acts on 
$\mathrm{Symb}_{\Gamma_p}(\bbD^\dagger (\cW))^{\leqslant \nu}$ and we have a canonical decomposition into the direct sum of the corresponding eigenspaces:
\[
\mathrm{Symb}_{\Gamma_p}(\bbD^\dagger (\cW))^{\leqslant \nu} =
\mathrm{Symb}^+_{\Gamma_p}(\bbD^\dagger (\cW))^{\leqslant \nu} \oplus
\mathrm{Symb}^-_{\Gamma_p}(\bbD^\dagger (\cW))^{\leqslant \nu}. 
\]
Fix a real number $\nu >k_0-1$. Then for a sufficiently small $\cX$ we have well defined $\cO_\cX$-modules
\[
\mathrm{Symb}^\pm_{\Gamma_p}(\cX)^{\leqslant \nu}:=
\mathrm{Symb}_{\Gamma_p}^\pm(\bbD^\dagger (\cW))^{\leqslant \nu}
\otimes_{\mathbb T^\pm_{\cW,\nu}}\cO_\cX,
\]
where $\mathbb T^\pm_{\cW,\nu}$ denote the Hecke algebra acting on $\mathrm{Symb}_{\Gamma_p}^\pm(\bbD^\dagger (\cW))^{\leqslant \nu}.$ 
Then:
\begin{itemize}
\item{} $\mathrm{Symb}^\pm_{\Gamma_p}(\cX)^{\leqslant \nu}$ is a free $\cO_\cX$-module
of rank one (see \cite[\S4.2.1]{bellaiche2012}).

\item{} The natural projection 
\begin{equation}
\label{eqn: natural projection on the generalized eigenspace}
\mathrm{Symb}^\pm_{\Gamma_p}(\cX)^{\leqslant \nu} \lra
\mathrm{Symb}^\pm_{\Gamma_p}(\mathbb D^\dagger_{k-2}(E))_{(f_\alpha)}
\end{equation}
 is surjective. See \cite[Theorem 3.10]{bellaiche2012} for a proof of this claim. When $k=2$, one further needs to input from Theorem~3.30(iii) in op. cit., which tells us that the support of the cokernel of the map \eqref{eqn: natural projection on the generalized eigenspace} is necessarily Eisenstein (therefore trivial, since $f_\alpha$ is cuspidal).
\end{itemize}

\subsubsection{} Let us first assume that $x_0$ is not $\theta$-critical. Choose any lifts
 $\Phi_{\xi,\cX}^\pm\in \mathrm{Symb}^\pm_{\Gamma_p}(\cX)^{\leqslant \nu}$ of 
 $\Phi_\xi^\pm \in \mathrm{Symb}^\pm_{\Gamma_p}(\mathbb D^\dagger_{k_0-2}(E))_{(f_\alpha)}$
and set 
\[
\mu^\pm_{\mathrm{S},\xi,\cX}:=\Phi_{\xi,\cX}^\pm((\infty)-(0)), \qquad
\mu_{\mathrm{S},\xi,\cX}:=\mu_{\mathrm{S},\xi,\cX}^+ +\mu^-_{\mathrm{S},\xi,\cX}.
\]
Stevens' two-variable $p$-adic $L$-function is defined on setting 
\[
L_{\mathrm{S},\alpha}(\Phi_{\xi,\cX}) :=\mathrm{Mel}
\left (x^{-1}\mu^\pm_{\mathrm{S},\xi,\cX}\right ) \in \mathscr H_{\cX}(\Gamma).
\]
For any $x\in \cX (E),$ we denote by $L_{\mathrm{S},\alpha}(\Phi_{\xi,\cX},x)$
the specialization of $L_{\mathrm{S},\alpha}(\Phi_{\xi,\cX})$ at $x.$ In particular,
\[
L_{\mathrm{S},\alpha}(\Phi_{\xi,\cX}, x_0)=L_{\mathrm{S},\alpha}(f,\xi).
\]
More generally, let $x\in \cX^{\mathrm{cl}} (E)$ be a classical point corresponding to the stabilization  of a cuspidal eigenform $f^\circ_x.$ Let us choose nonzero elements 
$\xi_{x,\pm} \in \mathcal M_{\mathrm{B}}(f_x^\circ)^\pm$. Then, 
\[
L_{\mathrm{S},\alpha}(\Phi_{\xi,\cX}, x) =C_x \cdot L_{\mathrm{S},\alpha}(f_x^\circ,\xi_{x})
\]
for some $C_x\in E$.

\subsection{The $\theta$-critical scenario}
\label{subsect: critical scenario}

\subsubsection{} 
\label{subsubsec_2341_2022_08_17_1038}
We now discuss the $\theta$-critical case, which is the main focus of the present article. To that end, we assume that $f_\alpha$ belongs to the image of $\theta_k$.  Let us set $X:=U_p-\alpha.$ Then we can consider $X$ as a function on $\cX$, 
and the generalized eigenspace $\mathrm{Symb}_{\Gamma_p} (\DD^\dagger_{k-2}(E))_{(f_\alpha)}$ 
carries a natural $E[X]$-module structure. In \cite{bellaiche2012}, Bella\"{\i}che proves the following results: 

\begin{itemize}
\item[1.]{} The weight map $w\,:\, \cX \rightarrow \cW$ is ramified at $x_0.$ Namely, 
shrinking $\mathcal X$ if necessary, one can assume that $\cO_{\mathcal X}=\cO_\cW[X]/(X^e-Y)$ with $e\geqslant 2$. See \cite[Theorem~4]{bellaiche2012}.

\item[2.] The generalized eigenspace $\mathrm{Symb}_{\Gamma_p}^\pm (\DD^\dagger_{k-2}(E))_{(f_\alpha)}$ is isomorphic to $E[X]/(X^e)$ as $E[X]$-modules.

\item[3.]{} The $E$-vector space $\mathrm{Symb}^\pm_{\Gamma_p} (\DD^\dagger_{k-2}(E))[f_\alpha]
\subset \mathrm{Symb}_{\Gamma_p} (\DD^\dagger_{k-2}(E))_{(f_\alpha)}$ is 1-dimensional.

\item[4.] The map $\rho_{k-2}$ given as in \eqref{eqn_2022_03_23_15_51} is surjective and 
\[
\rho_{k-2} \left (\mathrm{Symb}^\pm_{\Gamma_p} (\DD^\dagger_{k-2}(E))[f_\alpha]\right )=0.
\]
\end{itemize}

Until the end of this paper, we shall assume that $e=2$, so that 
\begin{equation}
\label{eqn:e=2}
\cO_{\mathcal X}=\cO_\cW[X]/(X^2-Y).
\end{equation}
We remark that a well-known conjecture of Greenberg asserts that  $f$ has CM if and only if the associated $p$-adic representation of the local Galois group $G_{\Qp}$ decomposes into the direct sum of two characters.  This conjecture implies, thanks to a deep result due to Breuil and Emerton~\cite[Theorem 1.1.3]{BreuilEmerton2010}, that any $\theta$-critical eigenform has CM. Moreover,  if the critical point $x_0$ has CM (in the sense that the corresponding eigenform $f_\alpha$ has CM), then it  follows from a conjecture of Jannsen formulated in \cite{Ja89} that $e=2$ (cf. \cite{CMLBellaiche}).   Therefore, it is expected that \eqref{eqn:e=2} always holds true.

\subsubsection{}  We are in the scenario of Section~\ref{subsec_alg_prelim} since we assume $e=2$. In the what follows, we fix 
\label{subsubsec_2342_2022_08_17_1038}
\begin{itemize}
\item[$\bullet$]{} modular symbols $\Phi^\pm_\xi\in \mathrm{Symb}^\pm_{\Gamma_p} (\DD^\dagger_{k-2}(E))_{(f_\alpha)}$ such that $\rho_{k-2}(\Phi^\pm_\xi)=\varphi^\pm_\xi$ (where $\varphi^\pm_\xi$ are given as in \S\ref{subsec_modular_symbols_2022_03_23_16_13})\,,

\item[$\bullet$]{} modular symbols  
$\Phi^\pm_{\xi, \cX}\in \mathrm{Symb}^\pm_{\Gamma_p} (\cX)^{\leqslant \nu}$
which are lifts of $\Phi^\pm_\xi$ under the map (\ref{eqn: natural projection on the generalized eigenspace}). 
\end{itemize}

Let us put $\Phi_{\xi, \cX}:=\Phi^+_{\xi, \cX}+\Phi^-_{\xi, \cX}$ and consider  
\[
\widetilde \Phi_{\xi, \cX} :=1\otimes X\Phi_{\xi, \cX} + X\otimes \Phi_{\xi,\cX}
\in \cO_\cX \otimes_{\cO_\cW} \mathrm{Symb}_{\Gamma_p} (\cX)^{\leqslant \nu}.
\]
Following Bella\"{\i}che, we define the two-variable $p$-adic  $L$-function 
$L_{\mathrm{S}}( \widetilde \Phi_{\xi, \cX}) \in \mathscr H_\cX (\Gamma)$ as 
\[
L_{\mathrm S}( \widetilde \Phi_{\xi, \cX})=
\mathrm{Mel} (x^{-1}\widetilde \mu_{\xi,\cX}), \qquad   \widetilde \mu_{\xi, \cX} =\widetilde \Phi_\xi ((\infty)-(0)).
\] 
As in the non-critical case, we  denote by $L_{\mathrm S}( \widetilde \Phi_{\xi, \cX},x)$ the specialization of  $L_{\mathrm S}( \widetilde \Phi_{\xi, \cX})$ at $x\in \cX (E).$

\subsubsection{} 
\label{subsubsec_2343_2022_08_17_1038}
Suppose that $x\in \cX^{\mathrm{cl}}(E)$ corresponds to the $p$-stabilization $f_x$ of some eigenform $f_x^\circ$ of level $N$. It follows from Lemma~\ref{lemma: abstract bellaiche lemma} that the specialization map \eqref{eqn: abstract specialization} sends  $\widetilde \Phi_{\xi, \cX}$ to a modular symbol 
\[
\widetilde \Phi_{\xi, x}\in \mathrm{Symb}_{\Gamma_p}(\DD_{w(x)-2}^\dagger (E) )[f_x].
\]
This implies that $L_{\mathrm S}( \widetilde \Phi_{\xi, \cX})$ interpolates 
$p$-adic $L$-functions of one variable on $\cX (E) \setminus \{x_0\}$. In more precise terms, we have
\[
L_{\mathrm S}( \widetilde \Phi_{\xi, \cX},x) =C_x \cdot L_{\mathrm{S},\alpha (x)}(f_x^\circ, \xi_x), 
\qquad x\neq x_0,
\]
for some constant $C_x.$

\subsubsection{} 
\label{subsubsec_2344_2022_08_17_1038}
Let us set
\[
L_{\mathrm S,\alpha}^{[0]}( f, \xi):=L_{\mathrm S}( \widetilde \Phi_{\xi, \cX},x_0) \in 
\mathscr H (\Gamma).
\] 

\begin{proposition} 
\label{prop:Bellaiche improved}
The following assertions are valid. 
\item[i)]{} The function  $L_{\mathrm S,\alpha}^{[0]}( f, \xi)$ does not depend on the choice of 
the lift $\widetilde \Phi_{\xi, \cX}.$ Moreover 
\[
L_{\mathrm S,\alpha}^{[0]}( f, \xi)= \mathrm{Mel} (x^{-1}\mu_{\mathrm S,\xi}^{[0]}),
\qquad 
\mu_{\mathrm S,\xi}^{[0]}:=\Phi^{[0]}_\xi ((\infty)-(0)),
\]
where  $\Phi^{[0]}_\xi:=X\Phi_\xi$ and $\Phi_\xi=\Phi_\xi^++\Phi_\xi^-$ with $\Phi_\xi^\pm$ as above.

\item[ii)]{} The function
$L^{[0]}_{\mathrm{S},\alpha}(f,\xi)$ can be written in the form
\[
L^{[0]}_{\mathrm{S},\alpha}(f,\xi)=\left (\underset{i=1-k}{\overset{-1}\prod}\ell_i^\iota \right )
\cdot L_{\mathrm{S},\alpha}^{\mathrm{imp}}(f, \xi),
\]
where $L_{\mathrm{S},\alpha}^{\mathrm{imp}}(f, \xi)\in \LL_E$ is bounded. 
In particular,  for any finite character $\rho \in X(\Gamma)$ we have
\[
L^{[0]}_{\mathrm{S},\alpha}(f,\xi; \rho\chi^r)=0, \qquad 1\leqslant r\leqslant k-1.
\]
\end{proposition}
\begin{proof} The proof of the first assertion is straightforward and is omitted here. 
The second assertion has been noted in \cite[Section~1.4.3]{bellaiche2012}, we briefly record an explanation here. According to a classical result due to Vishik~\cite{Vishik}, the growth order of a distribution associated to an overconvergent modular eigensymbol on which $U_p$ acts with slope $\nu$ is at most $\nu$. This shows that the distribution $\mu_{\mathrm S,\xi}^{[0]}$ has growth order bounded from above by $k-1$. Since each logarithmic factor $\ell_i$ is the Mellin transform of a distribution with growth order $1$, it follows that $L_{\mathrm{S},\alpha}^{\mathrm{imp}}(f, \xi)$ is indeed bounded. 
\end{proof}

\subsection{The infinitesimal thickening of critical $p$-adic $L$-functions}
\label{subsec: infinitesimal Stevens functions}

\subsubsection{}
 We continue to work in the $\theta$-critical scenario (as in \S\ref{subsect: critical scenario}) and maintain previous notation and conventions. Let us set
\begin{align}
\label{eqn_2022_05_13_1041}
\begin{aligned}
\widetilde \Phi_\xi &:=1\otimes X\Phi_\xi + X\otimes \Phi_\xi\\
&= 
1\otimes \Phi^{[0]}_\xi + X\otimes \Phi_\xi \in E[X]/(X^2)\otimes_E \mathrm{Symb}^\pm_{\Gamma_p} (\DD^\dagger_{k-2}(E))
\end{aligned}
\end{align}
and 
\[
L_{\mathrm S}(\widetilde \Phi_\xi):=
\mathrm{Mel} (x^{-1}\widetilde \mu_\xi), \qquad  \widetilde \mu_\xi=\widetilde \Phi_\xi ((\infty)-(0))\,.
\] 
We remark that 
\[
\widetilde L_{\mathrm S,\alpha}(\widetilde \Phi_{\xi}) =L_{\mathrm S}( \widetilde \Phi_{\xi, \cX})
\pmod{X^2}\,.
\]

\begin{lemma} Suppose that $\Psi_\xi^\pm$ are another pair of lifts of the modular symbols $\varphi^\pm_\xi$ and let us put $\Psi_\xi=\Psi^+_\xi+\Psi^-_\xi$\,. Then, 
\[
 L^\pm_{\mathrm S}(\widetilde \Psi_\xi)= u^\pm(X)\,\widetilde L^\pm_{\mathrm S}(\widetilde\Phi_\xi)
\] 
for some $u^\pm(X)\in 1+EX$. 
\end{lemma}
\begin{proof} For $\star\in \{+,-\},$
the symbol  $X\Phi^\star_\xi$ generates
$\mathrm{Symb}^\star_{\Gamma_p} (\DD^\dagger_{k-2}(E))[f_\alpha],$ and therefore 
\[
\Psi^\star_\xi=(1+a_\star X)\Phi^\star_\xi
\] 
for some $a_\star\in E$.
A direct computation then shows that 
\[
\widetilde \Psi^\star_\xi =(1+a_\star X\otimes 1) \widetilde \Phi^\star_\xi.
\] 
This concludes the proof of the lemma.
\end{proof}

When the choice of $\Phi_\xi$ (which in turn determines $\widetilde\Phi_\xi$, cf. \eqref{eqn_2022_05_13_1041}) is not specified, we will write $ \widetilde L_{\mathrm S,\alpha}(f, \xi)$ in place of  $L_{\mathrm S}( \widetilde \Phi_{\xi})$. 

Observe that we can write $\widetilde L_{\mathrm S,\alpha}(f, \xi)$ in the form
\be\label{eqn_2022_05_11_1256}
\widetilde L_{\mathrm S,\alpha}(f, \xi)=L^{[0]}_{\mathrm S,\alpha}(f, \xi) +
X L^{[1]}_{\mathrm S,\alpha}(f, \xi), \qquad  L^{[1]}_{\mathrm S,\alpha}(f, \xi) \in 
\mathscr H(\Gamma)\,,
\ee
and for that reason, we refer to $\widetilde L_{\mathrm S,\alpha}(f, \xi)$ as the infinitesimal thickening of $L^{[0]}_{\mathrm S,\alpha}(f, \xi)$.

\begin{proposition}
\label{prop: Bellaiche secondary}
 The function $L^{[1]}_{\mathrm S,\alpha}(f, \xi)$ verifies the following interpolation property: There exists a constant $C_{\mathrm{S}}\in E^\times$ such that 
$$
L_{\mathrm{S},\alpha}^{[1]}(f, \xi, \rho\chi^j)= C_{\mathrm{S}}\cdot 
{(j-1)!}\, e_{p,\alpha}(f,\rho,j)\,\,\frac{L (f,\rho^{-1},j)}{(2\pi i)^{j+1-k}\Omega_{f}^\pm}
\qquad\qquad \textrm{\,if\, $\rho(-1)=\mp(-1)^{j}\,,$}
$$
where $e_{p,\alpha}(f,\rho,j)$ is given as in \eqref{eqn: interpolation factor e}. 

\end{proposition}
\begin{proof} See \cite{bellaiche2012}.
\end{proof}


\section{Euler systems and $p$-adic $L$-functions} 
\label{sec_2_4_2022_05_11_0809}
In this section, we review the Perrin-Riou-style construction of $p$-adic $L$-functions attached to $p$-stabilized normalized eigenforms and compare these $p$-adic $L$-functions to those constructed using modular symbols (cf. \S\ref{sec_new_2_3_2022_03_14}). The scenario when the slope of the $p$-stabilized eigenform is non-critical, which we review in \S\ref{subsubsec_2413_20220511_0821} and \S\ref{subsec_slope_zero_padic_L}, is covered by Kato's fundamental work~\cite{kato04}. The sought-after comparison in the setting when the $p$-stabilized eigenform is of critical slope but non-$\theta$-critical is established in \cite[Proposition 7.3]{BB_CK1_PR} and recorded as part of Proposition~\ref{prop_2_17_2022_05_11_0842} below. 

The scenario when the $p$-stabilized eigenform is $\theta$-critical is much more involved and the comparison of two different constructions of the $p$-adic $L$-functions in this set-up is one of the main results of the present article (cf. \S\ref{subsec_defn_critical_padic_L_eigencurve} and \S\ref{subsec_245_2022_05_11_0845}).  

\subsection{Kato's $p$-adic $L$-functions}
\label{subsec: Kato's p-adic L-functions}
\subsubsection{}
\label{subsubsec_2411_2023_07_05_0731}
In this section, we maintain the notation and conventions of \S\ref{sec_abstract_setting} and \S\ref{sec_new_2_3_2022_03_14}. Let $f$ denote a normalized newform of weight $k$ and level $N$. Fix an odd prime $p$ such that $p\nmid N.$ To simplify notation, we set
\[
V_f:=\mathcal M_p(f) \qquad \textrm{(the $p$-adic realization of $\mathcal M(f)$).} 
\]
Since $p\nmid N,$ the $p$-adic Galois representation $V_f$ is crystalline at $p$. We have the canonical pairings
\begin{equation}
\label{eqn_2411_2022_05_11_0827}
\begin{aligned}
&V_{f^*} \otimes_E V_f \lra E(\chi^{1-k}), \\
[\,,\,]\,:\,\,\, 
&\Dc (V_{f^*})\times \Dc (V_{f}) \lra 
\Dc (E (\chi^{1-k}))\simeq E
\end{aligned}
\end{equation}
induced from Poincar\'e duality. The embedding of $S_k(N)[f]$ in $M_{\mathrm{dR}}(f)$ fixes a canonical element 
$\omega_f\in \Fil^{k-1}\Dc (V_f)$ via the comparison isomorphisms.

In \cite{kato04}, Kato gave the construction of the Beilinson--Kato Euler systems, which live in the first Galois cohomology of $V_{f^*}$.  Let us fix a basis $\xi$ of $\mathcal{M}_{\mathrm{B}}(f)$ and denote by $\xi^*$
the dual basis of $\mathcal{M}_{\mathrm{B}}(f^*).$
 We denote by 
$$\bz (f^*,\xi^*):=\bz^+(f^*,\xi^*)+\bz^-(f^*,\xi^*) \in H^1_\Iw ( V_{f^*}(k))$$ 
the Beilinson--Kato element, which comes naturally associated to the class $\xi^*$. We remark that our $\bz (f^*,\xi^*)$  is denoted by $\bz_\gamma^{(p)}(f^*)(k)$ in \cite{kato04}. The Beilinson--Kato elements  play  a central role in \S\ref{sec_2_4_2022_05_11_0809}.  

Recall that $\alpha$ and $\beta$ stands for the Hecke eigenvalues of $f$ at $p$. Then, $\alpha^*=p^{k-1}/\beta $ and  $\beta^*=p^{k-1}/\alpha$ are the Hecke eigenvalues 
of $f^*$ at $p$. We recall that we assume $v_p(\beta) \leqslant v_p(\alpha)$ and therefore $v_p(\beta^*) \leqslant v_p(\alpha^*)$. 

\subsubsection{} \label{subsubsec_2222_2022_04_27} 
Until the end of \S\ref{subsubsec_2413_20220511_0821}, we assume that the $p$-stabilization  $f_\alpha$ is not $\theta$-critical, so that $\Fil^{k-1}\Dc (V_f) \cap \Dc (V_f)^{\varphi=\alpha}=\{0\}$. This is clear in the non-critical-slope scenario (i.e. when $v_p(\alpha)<k-1$) and it follows from a deep result of Breuil and Emerton  \cite{BreuilEmerton2010} in the scenario when $v_p(\alpha)=k-1$ but $f_\alpha$ is not $\theta$-critical. 

As a result, there exists a unique element $\eta_{f}^\alpha \in \Dc (V_{f})^{\varphi=\alpha}$ such that 
\[
\left [\omega_{f^*} ,\eta_{f}^{\alpha} \right ]=1. 
\]
 
\subsubsection{}\label{subsubsec_2413_20220511_0821} The first pairing in \eqref{eqn_2411_2022_05_11_0827} gives rise to the pairing 
\[
V_{f^*}(k)\otimes V_f \lra E(\chi)\,,
\]
and we have the following induced pairings on the $p$-local Iwasawa cohomology:
\[
\bigl < \,,\,\bigr >_{V_f} \,:\, \left (\mathscr H(\Gamma)\otimes_{\Lambda [1/p]} H^1_{\Iw}(\Qp,V_{f^*}(k))\right ) \times \left (\mathscr H(\Gamma)\otimes_{\Lambda [1/p]} H^1_{\Iw}(\Qp,V_{f})^\iota 
\right )\lra \mathscr H(\Gamma).
\]
Let us consider the Perrin-Riou-style $p$-adic $L$-function 
\[
L_{\mathrm{K},\alpha} ( f,\xi):=
\left < \res_p (\bz (f^*,\xi^*)), c\circ \Exp_{V_f,0}\left (\widetilde{\eta_{f}^{\alpha}}
\right )^\iota \right >_{V_f} \in  \mathscr H(\Gamma), \qquad 
\widetilde{\eta_{f}^{\alpha}}=\eta_{f}^{\alpha}\otimes (1+\pi),
\]
which can be realized as the Mellin transform of a distribution $\mu_{\mathrm{K},\xi}$ with growth order $v_p(\alpha)$. It follows from Kato's explicit reciprocity laws that 
$L_{\mathrm{K},\alpha} ( f,\xi)$ also verifies the interpolation property \eqref{eqn: interpolation property in non critical case}.

\begin{proposition}
\label{prop_2_17_2022_05_11_0842}
If $f_\alpha$ is not $\theta$-critical, then 
\[
L_{\mathrm{K},\alpha} ( f,\xi)= L_{\mathrm{S},\alpha} ( f,\xi)\,.
\]
Equivalently, 
\[
\mu_{\mathrm{K},\xi}=x^{-1}\mu_{\mathrm{S},\xi}.
\]
\end{proposition}

\begin{proof} In the case $v_p(\alpha)<k-1$ (or equivalently, when $v_p(\alpha^*)<k-1$), this proposition is proved in \cite{kato04}. Namely, since  both distributions  $\mu_{\mathrm{K},\xi}$ and $\mu_{\mathrm{S},\xi}$ have growth order strictly smaller than $k-1$, this  follows from the fact that
\[
L_{\mathrm{K},\alpha} ( f,\xi, \rho \chi^r)= L_{\mathrm{S},\alpha} ( f,\xi, \rho \chi^r)
\]
for all $1\leqslant r\leqslant k-1$ and $\rho \in X(\Gamma)$ of finite order.

The case $v_p(\alpha)=k-1$ follows from \cite[Proposition~7.3]{BB_CK1_PR}. The proof relies on the interpolation of Beilinson-Kato elements in families. More precisely, let us put
\be\label{eqn_20220511_1024}
\mathcal E_N (f):= \underset{\ell \mid N}\prod (1-a_\ell \sigma_\ell^{-1})\in \LL, 
\ee 
where $\sigma_\ell\in \Gamma$ is the unique element such that $\chi (\sigma_\ell)=\ell$. Set $\bz:= \mathcal E_N (f) \cdot \bz (f^*,\xi^*)$ and define
\[
L(\bz):=\left < \res_p (\bz), c\circ \Exp_{V_f,0}\left (\widetilde{\eta_{f}^{\alpha}}\right )\right >_{V_f}.
\] 
Proposition~7.3 and  Formula (68)  in op. cit. imply that 
\[
L(\bz) =\mathcal E_N (f) L_{\mathrm{S},\alpha} ( f,\xi).
\]
Since  $L(\bz)=\mathcal E_N (f) L_{\mathrm{K},\alpha} ( f,\xi),$ this concludes the proof of our proposition. 
\end{proof}

\subsubsection{} 
Let us set $\mathcal E_N(f;\rho \chi^r):=\rho \chi^r \circ  \mathcal E_N(f)$. Then,
\begin{equation}
\label{eqn: evaluation of euler-like factor EN}
\mathcal E_N(f;\rho \chi^r)=\underset{\ell \mid N}\prod (1-a_\ell \rho (\sigma_\ell)^{-1} \ell^{-r}).
\end{equation}

\subsection{Critical points on the eigencurve}

\subsubsection{}
\label{subsubsec_221_04042022}
In this section, we assume that $f_\alpha$ is $\theta$-critical. The $p$-stabilization $f_\alpha$ of $f$ with respect to $\alpha$ corresponds to a point on 
the cuspidal eigencurve $\cC^{\mathrm{cusp}}$, which we denote by $x_0$ and will also say that $x_0$ is $\theta$-critical (or sometimes simply \emph{critical}). Our objective  is to use the $\cO_{\mathcal X}$-adic Beilinson--Kato element we have introduced in \cite[Definition 6.14]{BB_CK1_PR} to give an ``\'etale'' construction of Bella\"iche's $p$-adic $L$-function.

We fix a sufficiently small  neighborhood  $\mathcal X$ of $x_0$ and denote by $w\,:\, \cX \rightarrow \cW$ the weight map. As in Section~\ref{subsect: critical scenario}, we assume 
that $e=2$ and therefore $\cO_\cX \simeq \cO_\cW [X]/(X^2-Y).$
For each classical crystalline point 
$x\in \cX (E),$ we denote by $f_{x}$ the specialization of $\mathbf{f}$ at $x$ and 
by $f_x^\circ$  the newform of level prime to $p$ whose $p$-stabilization is $f_x$. In particular, $f_{x_0}^\circ =f$.

\subsubsection{} 
\label{subsubsec_221_24082022}
Recall that the overconvergent \'etale cohomology provides us with a family of 
$p$-adic Galois representations $V_{\mathcal X}$ and $V_{\mathcal X}'$ of rank $2$ over $\cO_\cX$
(given as in \cite{BB_CK1_PR}, Equation (51)), that interpolate Deligne's representations (homological and cohomological, respectively)  at classical points. In particular, $V_{x_0}\simeq V_f$.

\begin{proposition} 
\label{prop_big_Poincare_duality}
\item[i)] The $G_{\QQ_p}$-representation $V_{\cX}$ is equipped with a triangulation satisfying the conditions \eqref{item_C1}--\eqref{item_C3}. 
\item[ii)] Recall that $e=2$. Then the pairing \eqref{eqn_big_Poincare_duality}
\[
(\,,\,)\,:\,V_{\mathcal X}'\otimes V_{\mathcal X} \lra {\cO_{\mathcal W}}
\]
satisfies the condition \eqref{item_Adj}.
\item[iii)] The pairing $(\,,\,)$ is perfect.
\end{proposition}
 We remark that we do not need (iii) in the present work, but we have included it for the sake of completeness.
\begin{proof} 
\item[i)] The properties \eqref{item_C1} and \eqref{item_C2} follow from the results of Faltings and Saito on the $p$-local properties of $p$-adic representations arising from modular forms. 

For the first half of the property \eqref{item_C3}, we note that $\Fil^{w(x)-1}\Dst (V_x)\cap \Dst (V_x)^{\varphi=\alpha (x)}\neq \{0\}$ at any non-critical classical point $x$, as a result of the weak admissibility of $\Dst (V_{x})$. 
If $x=x_0$, the restriction of the  representation $V_{x_0}=V_f$ to the decomposition group at $p$ is reducible (cf. \cite{BreuilEmerton2010}, Theorem~1.1.3). Therefore, $V_{x_0}=V_{x_0}^{(\alpha)}\oplus V_{x_0}^{(\beta)},$ where $V_{x_0}^{(\alpha)}$ and  $V_{x_0}^{(\beta)}$ have Hodge--Tate weights $-(k-1)$ and $0$, respectively. This implies that 
$\Fil^{k-1}\Dst (V_{x_0})= \Dst (V_{x_0})^{\varphi=\alpha}$, which is the second part of the property \eqref{item_C3}.

\item[ii)] This portion  follows from the standard properties of the Atkin--Lehner operator $U_p$, cf. \cite[Proposition~6.6(iii)]{BB_CK1_PR}.
\item[iii)] Consider the natural morphism 
$$
\mathscr{F}:\,V_\cX'\lra {\rm Hom}_{\cO_\cW}(V_\cX,\cO_\cW)\,,\qquad v\xrightarrow{\,\mathscr{F}\,} (w\mapsto (v,w))\,, \quad v\in V_\cX', w\in V_\cX\,.
$$
Note that both the source and the target of the morphism $\mathscr{F}$ are free $\cO_\cW$-modules of rank $4$. 

Since $\cO_\cW$ is a PID, it follows that $\ker(\mathscr{F})$ is a free $\cO_\cW$-module. Moreover, for any classical weight $w\in \cW^{\rm cl}(E)$, we have an exact sequence
$$\ker(\mathscr F)_w \lra V_w'\xrightarrow{\,\mathscr{F}_w} {\rm Hom}_{E}(V_w,E)\,,$$
where for any $\cO_\cW$-module $M$, we write as before $M_w$ to denote $M \otimes_{\cO_\cW, w}E$\,. When $w>k$, it follows from control theorems utilized as in the proof of \cite[Theorem 5.6]{BB_CK1_PR} and the perfectness of Poincar\'e duality that the map $\mathscr{F}_w$ is an isomorphism; in particular, $\ker(\mathscr F)_w=\{0\}$. Since $\ker(\mathscr{F})$ is a free $\cO_\cW$-module, we conclude that $\ker(\mathscr F)=0$. This shows that the pairing  $(\,,\,)$ is non-degenerate on the left, and a similar argument it is also non-degenerate on the right (therefore, non-degenerate). 

We also infer that ${\rm coker}(\mathscr F)$ is a torsion $\cO_\cW$-module. Our argument in the preceding paragraph shows that ${\rm coker}(\mathscr F)$ has no support at any classical point $w>k$. As a result, on shrinking $\cW$ as necessary, we may assume with out loss of generality that ${\rm coker}(\mathscr F)$ has no support away from $k$. To prove that $\mathscr F$ is an isomorphism (and therefore, that the pairing $(\,,\,)$ is perfect as required), it suffices to show that ${\rm coker}(\mathscr F)_k=\{0\}$. 

Let us consider the exact sequence 
$$ V_k'\xrightarrow{\,\mathscr{F}_k} {\rm Hom}_{E}(V_k,E)\lra {\rm coker}(\mathscr F)_k\lra 0\,,$$
and choose an $E[X]/(X^2)$-module bases $\{E_1,E_2\}$ and $\{E_1',E_2'\}$ of $V_k$ and $V_k'$, respectively. Let us denote by $e_i'\in V_{x_0}=V'_k/(X)$ the image of $E_i'$, so that $\{e_1',e_2'\}$ is an $E$-basis of $V_{x_0}'\simeq V_f'$. Let us similarly define the $E$-basis $\{e_1,e_2\}$ of $V_{x_0}\simeq V_f$.  Recall that $XV_k'\simeq V_f'$ and $XV_k\simeq V_f$, and we shall treat $\{XE_1',XE_2'\}$ and  $\{XE_1,XE_2\}$ as $E$-bases of $V_f'$ and $V_f$, respectively, via these isomorphism. 

It follows from Part (ii) that $(E_i', XE_j)_k=(e_i',XE_j)_f$ and $(XE_i', E_j)_k=(XE_i',e_j)_f$ can be computed in terms of the Poincar\'e duality pairing $V_f'\otimes V_f\xrightarrow{(\,,\,)_f}E$, and that $(XE_i',XE_j)=0$. As a result, the matrix of $\mathscr F_k$ (in the ordered $E$-bases $\{XE_1', XE_2', E_1',E_2'\}$ and $\{E_1,E_2, XE_1, XE_2\}$ of $V_k'$ and $V_k$, respectively) is block-diagonal:
$$\mathscr{F}_k \sim \begin{pmatrix}
    A & \star\\
    0 & B
\end{pmatrix}\,,\qquad 
A=\begin{pmatrix}
    (XE_1', e_1)_f& (XE_2', e_1)_f\\
    (XE_1', e_2)_f & (XE_2', e_2)_f\end{pmatrix}\,,
B=\begin{pmatrix}
    (e_1', XE_1)_f& (e_2', XE_1)_f\\
    (e_1', XE_2)_f & (e_2', XE_2)_f
\end{pmatrix}\,.$$
Since $\det A\neq 0\neq \det B$ by the perfectness of the Poincar\'e duality pairing $(\,,\,)_f$, it follows that $\mathscr{F}_k$ is an isomorphism, as desired.
\end{proof}

\subsection{Slope-zero $p$-adic $L$-function}
\label{subsec_slope_zero_padic_L}
As a reflection of the \emph{eigenspace-transition by differentiation principle} (cf. \S\ref{sect:eigenspace-transition}), one expects to relate the improved critical $p$-adic $L$-function to the slope-zero $p$-adic $L$-function. We indeed show that this is the case as part of Proposition~\ref{prop: comparision p-adic L-functions for alpha and beta} below. 

With that purpose in mind, we review the Perrin-Riou style construction of the slope-zero $p$-adic $L$-function in \S\ref{subsubsec_2434_2022_05_11_1018}, in a manner that is suitable to our purposes.

\subsubsection{}
\label{subsubsec_2221_18_11_2021}
In \cite[Section~6.6.6]{BB_CK1_PR}, we gave a construction of an element\footnote{This element is denoted by  ${\mathbb{BK}}^{[\mathcal{X}]}_{N} (j, \xi)$ in \cite[Section~6.6.6]{BB_CK1_PR}.}
\[
\bz (\cX,\xi) \in  H^1(V^{\prime}_{\mathcal X}\widehat{\otimes}\,
\Lambda^\iota(1))
\]
for each basis $\xi=\{\xi_+,\xi_-\}$ of $\mathcal M_{\mathrm{B}}(f)^\pm$, interpolating the classical Beilinson--Kato zeta elements that were introduced in \cite{kato04}. To be more precise, let us denote by $\bz (x,\xi)$ the specialization of $\bz (\cX,\xi)$ at $x\in \cX (E)$. Let us write
$
f_x:=\underset{n=1}{\overset{\infty}\sum} a_n(x)q^n
$
and define
\begin{equation}
\label{eqn:interpolation factor E} 
\mathcal{E}_N\,:=\,\underset{\ell \mid  N}\prod (1- a_\ell(x) \sigma_\ell^{-1})  \in \LL_\chi.
\end{equation}
Note that $\mathcal{E}_N(x_0)=\mathcal{E}_N(f)$ (cf. \eqref{eqn_20220511_1024}).

The element $\bz (\cX,\xi)$ has the following interpolation properties: 

\begin{itemize}
\item{} We have
\be
\label{eqn: comparision of kato elements at 0}
\bz (x_0,\xi)=\mathcal{E}_N(x_0)\, \bz (f^*,\xi^*).
\ee

\item{} For any classical point $x\in \cX^{\mathrm{cl}}(E)$ and any 
basis $\xi_x$ of $\mathcal M_{\mathrm{B}}(f^\circ_x)$, we have 
\[
\bz (x,\xi)^\pm= C_x^\pm \cdot \mathcal{E}_N(x)^\pm \cdot \bz ((f_x^\circ)^*,\xi^*_x)^\pm , \qquad 
\textrm{for some $C_x^\pm \in E^*,$}
\]
where $\bz ((f_x^\circ)^*,\xi^*_x) \in H^1_{\Iw}(V_x^\prime (1))$ denotes the Beilinson--Kato zeta element associated to the pair $((f_x^\circ)^*, \xi^*_x)$.  Here, we recall from \S\ref{subsubsec_221_04042022} that $f_x^\circ$ denotes the newform attached to the $p$-old form $f_x$.   We remark that we have used the canonical isomorphism
\[
V_x^\prime (1) \simeq V_{(f_x^\circ)^*}(k)\,.
\]
\end{itemize}

\subsubsection{} We shall fix canonical bases of $\Dc (V_f)$ and $\Dc (V_{f^*})$ in the $\theta$-critical scenario. 
 The differential forms associated to $f$ and $f^*$ fix  bases $\{\eta_{f}^\alpha\}$ of $\Fil^{k-1}\Dc (V_{f})=\Dc (V_f^{(\alpha)})$ and $\{\eta_{f^*}^\alpha\}$ of  $\Fil^{k-1}\Dc (V_{f^*})=\Dc (V_{f^*}^{(\alpha^*)})$. Note then that $\eta_{f}^\alpha=\omega_f$ and $\eta_{f^*}^\alpha=\omega_{f^*}$.
We choose $\eta_{f}^\beta\in \Dc (V_f^{(\beta)})$ and $\eta_{f^*}^\beta\in \Dc (V_{f^*}^{(\beta)})$
in such a way that 
$$[\eta_f^\beta, \eta^\alpha_{f^*}]=1=[\eta_{f}^\alpha, \eta^\beta_{f^*}]\,,$$ 
$$[\eta_f^\beta, \eta^\beta_{f^*}]=0=[\eta_{f}^\alpha, \eta^\alpha_{f^*}]\,.$$

\subsubsection{} Since $v_p(\beta)=0,$ the $p$-adic $L$-function  
$L_{\mathrm{S},\beta} (f,\xi)\in \LL_E$ coincides with  the classical  Manin--Vi\v{s}ik $p$-adic $L$-function associated to the triple $(f,\beta,\xi)$.

\subsubsection{}
\label{subsubsec_2434_2022_05_11_1018}
We retain the notation of \S\ref{sect:eigenspace-transition}. Until the end of \S\ref{subsubsec_2434_2022_05_11_1018}, we assume that condition \eqref{item_C4} holds true. 
Recall the transition map $\kappa_0 \,:\, \DCc (\bD_k) \rightarrow \Dc (V_{x_0}^{(\beta)})$ given as in Proposition~\ref{prop:definition of kappa}. We fix  a generator $\eta \in \DCc (\bD_\cX)$
such that 
\[
\eta_{x_0}=\eta^{\alpha}_f.
\]
To simplify our notation, we will write $\kappa_0 (\eta)$ in place of $\kappa_0(\eta_k)$. We then have
\begin{equation}
\label{eqn: the constant a}
\kappa_0 (\eta)=C_{\mathrm{K}} \eta^{\beta}_f \quad \textrm{for some $C_{\mathrm{K}}\in E^\times.$}
\end{equation}
Let us put
\begin{align*}
\begin{aligned}
L_{\mathrm{K}, {\kappa}_0(\eta) }(f,\xi)&:= L_{p,{\kappa}_0(\eta)} (\bz (x_0,\xi))\\
&=
 \left <\res_p (\bz (x_0,\xi)),c\circ \Exp_{V_{x_0},0}(\kappa_0(\widetilde \eta))^\iota \right >_{x_0}
\,.
\end{aligned}
\end{align*}
We deduce on combining \eqref{eqn: comparision of kato elements at 0} and \eqref{eqn: the constant a} that
\begin{equation}
\label{eqn:condition for thm_interpolative_properties}
L_{\mathrm{K}, {\kappa}_0(\eta) }(f,\xi)= C_{\mathrm{K}}\,\cE_N (f) L_{\mathrm{K},\beta}(f,\xi)=C_{\mathrm{K}}\,\cE_N (f) L_{\mathrm{S},\beta}(f,\xi)\,. 
\end{equation}




\subsection{$\theta$-critical $p$-adic $L$-functions}
\label{subsec_defn_critical_padic_L_eigencurve}
We introduce the Perrin-Riou-style critical $p$-adic $L$-functions attached to $\theta$-critical forms. 
\subsubsection{}
Recall that we fixed a  generator $\eta$ of $\DCc(\bD_\cX)$ such that $\eta_{x_0}=\eta_f^{\alpha}$. Recall also from \eqref{eqn_bbeta_def} the element
$${\bbeta}:=1 \otimes X\eta+ X\otimes \eta\in \cO_{\mathcal X} \otimes_{\cO_\cW} \DCc(\bD_\cX) \,.$$

\begin{defn}[Arithmetic $\theta$-critical $p$-adic $L$-function in two variables]
\label{def_two_var_padicL_function} 
 We call the element
\[
\begin{aligned}
&L_{\mathrm{K},\eta}(\cX,\xi ):=L_{p,\eta}\left (\res_p\left ( \bz (\cX,\xi)\right )\right )\in \mathscr{H}_{\mathcal X}(\Gamma),
\end{aligned}
\]
where $L_{p,\eta}$ is the abstract $p$-adic $L$-function 
from Definition~\ref{defn_fat_eta_etatilde}(ii), the  arithmetic $p$-adic $L$-function (in two variables) on the connected component $\mathcal{X}$ of the eigencurve.
\end{defn}

Recall that for each classical point $x\in \cX^{\mathrm{cl}}(E) \setminus \{x_0\}$ and a basis $\xi_x$ of $\mathcal M_{\mathrm{B}}(f_x^\circ)$, we denote by $L_{\mathrm{K},\alpha (x)}(f_x^\circ, \xi_x)$ the   Manin--Vi{\v{s}}ik $p$-adic $L$-function associated to the stabilization $f_x$ of $f_x^\circ$.
It follows from Proposition~\ref{prop_2_5_18_11} and Kato's results reviewed 
in  \S\ref{subsec: Kato's p-adic L-functions} that  
\begin{equation}
\label{eqn:comparision with Manin-Vishik}
L_{\mathrm{K},\eta}^\pm(\cX,\xi, x)= A_x^\pm\,\cE_N(x)\, L_{\mathrm{K},\alpha(x)}^{\pm}( f_x^\circ,\xi_x)\quad
\textrm{for some nonzero $A_x^\pm \in E.$}
\end{equation}
Our first result concerns the behavior  of  $L_{p,\eta}$  at the critical point $x_0.$  
Paralleling the notation of \S\ref{subsec: secondary function}, we set
\[
\begin{aligned}
&L^{[0]}_{\mathrm{K},\alpha}(f,\xi):=L_{\mathrm{K},\eta}(\cX, \xi, x_0),\\
&L_{\mathrm{K},\alpha}^{[1]}(f,\xi):=\left. \left (\frac{d}{dX}L_{\mathrm{K},\eta} (\cX,\xi)\right ) \right\vert_{X=0}\,.
\end{aligned}
\]
Evidently, $L^{[0]}_{\mathrm{K},\alpha}(f,\xi)$ depends only on the choice of $\xi$ and it can be written in the form
\[
L^{[0]}_{\mathrm{K},\alpha}(f,\xi)={\mathcal E_N(f)\,}\left <\bz (f^*,\xi^*), c\circ \Exp_{V_f,0}
\left (\widetilde {\eta_f^{\alpha}}
\right )^\iota
\right >_{V_f}.
\]
Therefore,
\begin{equation}
\label{eqn: comparision LK and L[0]K}
L^{[0]}_{\mathrm{K},\alpha}(f,\xi)={\mathcal E_N(f)\,} 
L_{\mathrm{K},\alpha}(f,\xi).
\end{equation}
Here, we note that we have identified the pairings $\left <\,,\,\right >_{V_f}$ and $\left <\,,\,\right >_{x_0}$. On the other hand, $L_{\mathrm{K},\alpha}^{[1]}(f,\xi)$
depends also on the choice of the lift $\eta_k\in \bD_k$ of $\eta_f^{\alpha}$. We will write $L_{\mathrm{K},\eta}^{[1]}(f, \xi)$ in place of $L_{\mathrm{K},\alpha}^{[1]}(f,\xi)$ whenever we want to stress this dependence.

The following result should be compared with Proposition~\ref{prop:Bellaiche improved} and Proposition~\ref{prop: Bellaiche secondary}.

\begin{theorem}
\label{thm_interpolative_properties} Assume that $e=2$ and that the condition \eqref{item_C4} holds true.

\item[i)] There exist bounded $p$-adic $L$-functions 
$L_{\mathrm{K},\alpha}^{\mathrm{imp}}(f,\xi)$ and  
$L_{\mathrm{K},\alpha}^{[0],\mathrm{imp}}(f,\xi) \in \Lambda_E$ such that 
\begin{equation}
\nonumber
\begin{aligned}
&L_{\mathrm{K},\alpha}(f, \xi)= 
\left (\underset{i=1-k}{\overset{-1}\prod} \ell^\iota_i\right ) 
L_{\mathrm{K},\alpha}^{\mathrm{imp}}(f,\xi)\,,\\
&L_{\mathrm{K},\alpha}^{[0]}(f, \xi)= 
\left (\underset{i=1-k}{\overset{-1}\prod} \ell^\iota_i\right ) 
L_{\mathrm{K},\alpha}^{[0],\mathrm{imp}}(f,\xi)\,.
\end{aligned}
\end{equation}
In particular, 
$L_{p,\alpha}^{[0]}(f, \xi; \rho\chi^j)=0$ for all integers $1\leqslant j\leqslant k-1$ and
 characters $\rho \in X(\Gamma)$ of finite order.

\item[ii)] The secondary $p$-adic $L$-functions $L_{\mathrm{K},\alpha}^{[1],\pm}(f,\xi) $ verify the following interpolation property: For every character $\rho\in X(\Gamma)$ of finite order and integer $1\leq j \leq k-1$ we have
$$
L_{\mathrm{K},\alpha}^{[1],\pm}(f,\xi;\rho\chi^j)=C_{\mathrm K}\cdot 
{(j-1)!}\, e_{p,\alpha}(f,\rho,j)\,
\mathcal{E}_N(f;\rho \chi^j)\,\frac{L (f,\rho^{-1},j)}{(2\pi i)^{j+1-k}\Omega_{f}^\pm}\,,
\qquad\qquad \textrm{\,if\, $\rho(-1)=\mp(-1)^{j}\,,$}
$$
where $e_{p,\alpha}(f,\rho,j)$ is given as in \eqref{eqn: interpolation factor e} and 
$C_{\mathrm{K}}$ as in \eqref{eqn: the constant a}.
\end{theorem}

We call $L_{\mathrm{K},\alpha}^{\mathrm{imp}}(f,\xi)$ and 
$L_{\mathrm{K},\alpha}^{[0],\mathrm{imp}}(f,\xi)$
that appear in the statement of Theorem~\ref{thm_interpolative_properties}(i) the cyclotomic improvement of the critical $p$-adic $L$-functions
$L_{\mathrm{K},\alpha} (f,\xi)$ and 
$L_{\mathrm{K},\alpha}^{[0]}(f,\xi)$, respectively.

\begin{proof}
\item[i)] This portion is a restatement of Proposition~\ref{prop_bellaiche_formal_step_1}(ii) in the present setting. 

\item[ii)] Suppose  that the condition \eqref{item_C4} holds true. Then Theorem~\ref{prop_imoroved_padicL_vs_slope_zero_padic_L} yields
\begin{equation}
\label{formula:main theorem1}
L^{[1]}_{\mathrm{K},\eta}(f, \xi;  \rho \chi^j) ={\frac{b(\rho,j)}{a(\rho,j)} }
L_{\mathrm{K},\kappa_0(\eta)} (f,\xi ;\rho\chi^j)
\end{equation}
for all $1\leq j\leq k-1$ and every finite-order character $\rho$ with $p$-power conductor. We note that $a(\rho,j)$ and $b(\rho,j)$ in Equation \eqref{formula:main theorem1} are defined in \eqref{eqn: definition of euler-like factors}. On the other hand, it follows from \eqref{eqn:condition for thm_interpolative_properties} and the interpolation properties of the slope-zero $p$-adic $L$-function $L_{\mathrm{K},\beta}(f,\xi)$ that 
\begin{equation}
\label{formula:main theorem3}
L_{\mathrm{K}, { \kappa_0(\eta)}}(f, \xi; \rho\chi^j)={ C_K \cdot}
 (j-1)!\,e_{p,\beta}(f,\rho,j)\,\mathcal{E}_N(f{; \rho\chi^j})\,\frac{L (f,\rho^{-1},j)}{(2\pi i)^{j+1-k}\Omega_{f}^\pm}
\end{equation}
where 
\[
e_{p,\beta}(f,\rho,j)=\begin{cases}
\displaystyle\left (1-\frac{p^{j-1}}{\beta }\right ) \left (1-\frac{\alpha}{p^j}\right ) & \hbox{ if } \rho=\mathds{1},\\
\displaystyle\frac{p^{nj}}{{\beta^n}\tau (\rho^{-1})}
& \hbox{ if } \rho \neq\mathds{1}.
\end{cases}
\]
An easy computation shows that
\be\label{eqn_2022_05_11_1231}
\frac{b(\rho,j)}{a (\rho,j)}e_{p,\beta}(f,\rho,j)=e_{p,\alpha}(f,\rho,j).
\ee
 The assertion (ii) follows on combining \eqref{formula:main theorem1}
 with \eqref{formula:main theorem3} and \eqref{eqn_2022_05_11_1231},
\end{proof}

\subsubsection{} As our second result, we prove that our two-variable function $L_{\mathrm{K},\eta}(\cX,\xi)$ agrees with Bella\"{\i}che's $p$-adic $L$-function $L_{\mathrm{S}}(\widetilde{\Phi}_{\xi,\cX})$:

\begin{theorem} 
\label{thm:comparision with Bellaiche's construction}
Assume that $e=2$ and condition \eqref{item_C4} holds true. 
Then there exist  functions $u^\pm\in \cO_{\mathcal X}^\times$
such that 
\[
L_{\mathrm{K},\eta}^{\pm}(\cX,\xi)=u^\pm\,{\mathcal E}_N \,
L_{\mathrm{S}}^\pm(\widetilde\Phi^{\pm}_{\xi,\cX})\,,
\]
where ${\mathcal E}_N$ is defined by  (\ref{eqn:interpolation factor E}).  
\end{theorem}
\begin{proof}
We have
\[
L_{\mathrm{K},\alpha(x)}^\pm(f_x^\circ,\xi_x)= B_x^\pm  L_{\mathrm{S}}(\widetilde\Phi_{\xi,\cX}^\pm,x)\, \,\,,\qquad \textrm{for some  $B_x^\pm\in E^\times\,.$}
\]
for each classical non-critical point $x$. Comparing this fact with  \eqref{eqn:comparision with Manin-Vishik}, we infer that 
\[
L_{\mathrm{K},\eta}^\pm(\cX,\xi)= \lambda_x^\pm \,\cE_N(x) \, L_{\mathrm{S}}(\widetilde\Phi^\pm_{\xi,\cX},x), \qquad 
\lambda_x^\pm:=A_x^\pm B_x^\pm\in E.
\]
for each $x\in \mathcal X^{\rm cl}(E)$. On writing $L_{\mathrm{K},\eta}^\pm(\cX,\xi)$ and 
$L_{\mathrm{S}}^\pm(\widetilde\Phi^\pm_{\xi,\cX}) $ as power series with coefficients in $O_{\mathcal X}$, we see that there exist  meromorphic functions $\lambda^\pm\in {\rm Frac}(O_{\mathcal X})$ such that
\begin{equation}
\label{eqn_31_2021_06_03}
L_{\mathrm{K},\eta}^\pm(\cX,\xi)  = \lambda^\pm\,\cE_N \,\,L_{\mathrm{S}}^\pm(\widetilde\Phi^\pm_{\xi,\cX}). 
\end{equation}
We wish to prove that $\lambda^\pm(x_0)\neq 0$. Evaluating \eqref{eqn_31_2021_06_03} at the character $\rho\chi^r$
(where $1\leqslant r\leqslant k-1$ and $\rho$ has finite order) we have
\[
L_{\mathrm{K},\eta}^\pm(\cX,\xi; \rho\chi^r )  = \lambda^\pm\,\cE_N \,\,L_{\mathrm{S}}^\pm(\widetilde\Phi^\pm_{\xi,\cX}; \rho\chi^r).
\]
It follows from Proposition~\ref{prop:Bellaiche improved} and Theorem~\ref{thm_interpolative_properties} that the functions 
\[
L_{\mathrm{K},\eta}^\pm(\cX,\xi; \rho\chi^r),\, L_{\mathrm{S}}^\pm(\widetilde\Phi^\pm_{\xi,\cX}; \rho\chi^r) \in \cO_\cX
\]
both vanish at $x_0$. Therefore, considering the derivatives of both sides with respect to $X$ at $x_0$, we obtain that 
\[
L_{\mathrm{K},\alpha}^{[1],\pm}(f,\xi; \rho\chi^r)= \lambda^\pm (x_0)\, 
\cE_N(x_0;\rho \chi^r)\,  L_{\mathrm{S},\alpha}^{[1],\pm}(f,\xi; \rho\chi^r).
\] 
Using the non-vanishing results of Rohrlich in \cite{rohrlich88Annalen} if $k=2$, we choose $\rho$ and $r$ such that $L  (f,\rho^{-1},r)\neq 0$ and $\cE_N(x_0;\rho \chi^r)^\pm$. It follows from Proposition~\ref{prop: Bellaiche secondary} and Theorem~\ref{thm_interpolative_properties}  that 
$L_{\mathrm{K},\alpha}^{[1],\pm}(f,\xi; \rho\chi^r)$ and 
$L_{\mathrm{S},\alpha}^{[1],\pm}(f,\xi; \rho\chi^r)$ do not vanish. Therefore $\lambda^\pm(x_0)\neq 0$, and the theorem is proved. 
\end{proof}

\begin{corollary}
\label{cor_thm:comparision with Bellaiche's construction}
Assume that $e=2$ and condition \eqref{item_C4} holds true. Then,
\be\label{eqn_2022_05_11_1312}
L_{\mathrm{K},\alpha}^{[0],\pm}(f,\xi)\,= {\mathcal E_N(f)}\,\lambda^\pm (f)\, L_{\mathrm{S},\alpha}^{[0],\pm}(f,\xi)\,,
\ee
where $\lambda^\pm (f):= \lambda^\pm (x_0) \neq 0$. 
Equivalently,
\be\label{eqn: constant lammbda}
L_{\mathrm{K},\alpha}^{\pm}(f,\xi)\,= \,\lambda^\pm (f)\, L_{\mathrm{S},\alpha}^{[0],\pm}(f,\xi)\,.
\ee
\end{corollary}

\begin{remark} The interpolation properties for the $p$-adic function 
$L_{\mathrm{S},\alpha}^{[1]}(f,\xi)$ was  proved by Bella\"{\i}che using the theory of modular symbols. In the present work, we recover it using an entirely different method, assuming the validity of \eqref{item_C4}.
\end{remark}

\subsubsection{}
We continue to assume that the eigenform $f_\alpha$ is $\theta$-critical. Recall that $\beta$ denotes the other root of the Hecke polynomial of $f$ at $p$, so that $v_p(\beta)=0$ and the $p$-stabilized eigenform $f_\beta$ has slope zero.

As another incarnation of the eigenspace transition by differentiation principle (cf. \S\ref{sect:eigenspace-transition}), one can explicitly compare the improved $p$-adic $L$-function $L_{\mathrm{K},\alpha}^{\mathrm{imp}}(f,\xi)$ and the slope-zero $p$-adic $L$-function $L_{\mathrm{K},\beta}(f, \xi)$. This comparison involves an Iwasawa theoretic $\mathscr L$-invariant $\mathscr L_{\Iw}^{\rm cr}$ (cf. Proposition~\ref{prop: comparision p-adic L-functions for alpha and beta} below), whose appearance is a reflection of the extreme exceptional zero phenomena the $\theta$-critical $p$-adic $L$-functions exhibit. The construction and the properties of $\mathscr L_{\Iw}^{\rm cr}$\,, which we believe are of independent interest, will be given in \S\ref{subsec:critical L-invariants}. 

To motivate for the latter portions of our paper (and for the completeness of our discussion of $p$-adic $L$-functions), we record this comparison of the improved $p$-adic $L$-function $L_{\mathrm{K},\alpha}^{\mathrm{imp}}(f,\xi)$ and the slope-zero $p$-adic $L$-function $L_{\mathrm{K},\beta}(f, \xi)$ here, deferring its proof to \S\ref{subsubsec_proof_prop_2_16}.

\begin{proposition} 
\label{prop: comparision p-adic L-functions for alpha and beta}
The $p$-adic $L$ functions $L_{\mathrm{K},\alpha}^{\mathrm{imp}}(f,\xi)$ and 
 $L_{\mathrm{K},\beta}(f, \xi)$ are related via the identity
\[
L_{\mathrm{K},\alpha}^{\mathrm{imp}}(f,\xi)= (-1)^{k-1}\, \mathscr L_{\Iw}^{\rm cr} (V_{f^*}(k))\,  L_{\mathrm{K},\beta}(f, \xi)\,,
\]
where $\mathscr L_{\Iw}^{\rm cr} (V_{f^*}(k))$ is the Iwasawa theoretic $\mathscr L$-invariant defined in \S\ref{subsec:critical L-invariants}.
\end{proposition} 
It seems that there is no direct way to prove the analogue of the comparison in Proposition~\ref{prop: comparision p-adic L-functions for alpha and beta} for the $p$-adic $L$-functions $L_{\mathrm{S},\alpha}^{\mathrm{imp}}(f,\xi)$ and $L_{\mathrm{S},\beta}(f, \xi)$ that are constructed by interpolating Betti cohomology (i.e. without appealing to \eqref{eqn_2022_05_11_1312} and \eqref{eqn:condition for thm_interpolative_properties}). This serves as an independent justification to pursue a Perrin-Riou style construction of critical $p$-adic $L$-functions.

\subsection{Infinitesimal thickening of $p$-adic $L$-functions revisited}\label{subsec_245_2022_05_11_0845}
We will next discuss an extension of the comparison in Theorem~\ref{thm:comparision with Bellaiche's construction} in an infinitesimal neighborhood of a $\theta$-critical point $x_0$ (cf. Proposition~\ref{prop_1_19_2022_16_03} below). This refinement is relevant to the Iwasawa theoretic study of these $p$-adic $L$-functions in \S\ref{chapter_main_conjectures}.
Throughout \S\ref{subsec_245_2022_05_11_0845}, we assume that $f_\alpha$ is a $\theta$-critical normalized eigenform and retain the conventions of \S\ref{subsec_defn_critical_padic_L_eigencurve}. 
\subsubsection{}
\label{subsec_245_2022_05_11_0845_subsubsec1}
If the condition \eqref{item_C4} does not hold, then $L_{\mathrm{K},\alpha}^{[1]}(f,\xi;\rho\chi^{j})=0$ for all $\rho$ and $j$ as above. This can be proved  in a very similar manner to the proof of Theorem~\ref{thm_interpolative_properties} and we omit details. This in comparison with Bella\"iche's interpolation result (cf. Proposition~\ref{prop: Bellaiche secondary} and Theorem~\ref{thm:comparision with Bellaiche's construction}) leads us to the following:
\begin{conj} 
\label{conjecture_GP}
 The condition \eqref{item_C4} holds at any $\theta$-critical point.
\end{conj}

\subsubsection{} Recall that we assume $e=2$ and that the  condition \eqref{item_C4} holds. 
Then Corollary~\ref{cor_thm:comparision with Bellaiche's construction} tells us that $$L_{\mathrm{K},\alpha}^{\pm}(f,\xi)= \lambda^\pm (f)\, L_{\mathrm{S},\alpha}^{[0],\pm}(f,\xi)\,,\qquad \lambda^\pm (f)\in E^\times\,.$$ 
The following is another important open problem in the theory of critical $p$-adic $L$-functions: 

\begin{conj} The functions $L_{\mathrm{S},\alpha}^{[0]}(f,\xi)$ and $L_{\mathrm{K},\alpha} (f,\xi)$ are nonzero. 
\end{conj} 
Recall that $L_{\mathrm{S},\alpha}^{[0]}(f,\xi)$ and $ L_{\mathrm{K},\alpha}(f,\xi)$ both depend only on the choice of $\xi$. It is also natural to raise the following
\begin{question}
\label{conj:equality of two p-adic functions}
 Can we explicitly determine $\lambda^\pm (f)$? 
\end{question} 

\subsubsection{}
\label{subsubsec_4_5_1_1_16_03_2022}
Let us put $\widetilde{E}:= E[X]/(X^2)$ and fix a uniformizer $\varpi_E$ of $\cO_E$. Let us define for each natural number $n$ the ring
{ 
$$\cO_E^{(n)}= \cO_E+\varpi_E^{-n} \cO_E X \,\subset \widetilde{E}.$$
}
It is easy to see that  $ \cO_E^{(n)}$ is a unital\footnote{An algebra $A$ is called unital if it contains a multiplicative unit $1_A$. A subalgebra $B$ of such $A$ is called a \emph{unital subalgebra} if $1_A\in B$.} $\cO_E$-subalgebra of $\widetilde E$. 
\begin{lemma} 
Let $A$ be a unital $\cO_E$-subalgebra of $\widetilde E$ 
such that the natural projection $\widetilde E\rightarrow E$ induces a surjective homomorphism of rings $A\rightarrow \cO_E$.  Then either  $A$ is isomorphic to $\cO_E^{(n)}$ for some $n\geqslant 0$, or 
$A=\cO_E+EX$.
\end{lemma}
\begin{proof} Let $I$ denote the kernel of the map $A\rightarrow \cO_E.$ 
Then $I$ is an $\cO_E$-submodule of $EX.$ Hence either $I=\varpi_E^{-n} \cO_E X$ or 
$I=EX.$ Since $\cO_E\subset A$ (recall that $\cO_E$-algebra $A$ is unital by assumption) the proof of our lemma follows. 
\end{proof}

\subsubsection{} Set  $\widetilde{\mathscr{H}} (\Gamma):=\widetilde{E}\otimes_E \mathscr{H} (\Gamma).$ 
Recall from \S\ref{subsec: infinitesimal Stevens functions} the infinitesimal thickening of the critical $p$-adic $L$-function 
\[
\widetilde L_{\mathrm{S},\alpha}(f,\xi):=
L_{\mathrm{S}} (\widetilde{\Phi}_{\xi,\cX}) \pmod{X^2}\quad  \in \widetilde{\mathscr{H}} (\Gamma).
\]
As we have noted in \eqref{eqn_2022_05_11_1256}, it can be written in the form
\[
\widetilde L_{\mathrm{S},\alpha}(f,\xi)=L_{\mathrm{S},\alpha}^{[0]}(f,\xi)+X L_{\mathrm{S},\alpha}^{[1]}(f,\xi).
\]
Analogously, we define the infinitesimal thickening 
\[
\widetilde L_{\mathrm{K},\alpha}(f,\xi):=L_{\mathrm{K},\alpha}^{[0]}(f,\xi)+X L_{\mathrm{K},\alpha}^{[1]}(f,\xi) \in \widetilde{\mathscr{H}} (\Gamma)\,.
\]
of the arithmetic critical $p$-adic $L$-function. Let us set $\LL^{(n)}:=\cO_E^{(n)}\otimes_{\cO_E}\LL$ and 
$\LL^{(\infty)}:=\varinjlim_n \LL^{(n)}.$
Proposition~\ref{prop_1_19_2022_16_03} below, where we compare the two infinitesimal thickenings of critical $p$-adic $L$-functions ($\widetilde L_{\mathrm{K},\alpha}(f,\xi)$ and $\widetilde L_{\mathrm{S},\alpha}(f,\xi)$) in  $\LL^{(\infty)}$  is an infinitesimal  refinement of the comparison in \eqref{eqn_2022_05_11_1312}. 
Let $\widetilde{\mathcal E}_N \in \widetilde{\Lambda}_E=\LL_E\otimes_{E} E[X]/(X^2)$ denote the image of ${\mathcal E}_N \in \LL_\cX$ given as in \eqref{eqn_20220511_1024}.

\begin{proposition}
\label{prop_1_19_2022_16_03}
Assume that $e=2$ and the condition \eqref{item_C4}  holds. Then  the $ \LL^{(\infty),\pm}$-free modules generated by ${\widetilde L}^\pm_{\mathrm{K},\alpha}(f, \xi)$ and $\lambda^\pm (f)\cdot \widetilde{\mathcal E}_N\cdot
{\widetilde L}^\pm_{\mathrm S,\alpha}(f, \xi)$ coincide.
\end{proposition}
\begin{proof} By Theorem~\ref{thm:comparision with Bellaiche's construction}, we have
\[
{\widetilde L}^\pm_{\mathrm{K},\alpha}(f, \xi)=u^\pm(X) \cdot 
\lambda^\pm (f)\cdot \widetilde{\mathcal E}_N\cdot
{\widetilde L}^\pm_{\mathrm S,\alpha}(f, \xi),
\]
with $u^\pm (X)\in 1+EX.$ Let us choose $n\gg 0$ so that $u^\pm (X)\in \LL^{(n)}$.  Then $u^\pm (X)$ is invertible in $\LL^{(n)}$ and the proof of our proposition follows. 
\end{proof}

\chapter{Main Conjectures for the punctual eigencurve and the improved critical $p$-adic $L$-functions}
\label{chapter_main_conjectures}
In this chapter, we will introduce and study the algebraic counterparts of $\theta$-critical $p$-adic $L$-functions, with our goal to extend Perrin-Riou's theory to this particular setting. 

Along the way, we introduce our \emph{critical $\cL$-invariants} in \S\ref{subsec:critical L-invariants}, study its basic properties in this portion with the aid of the earlier portions of \S\ref{sec_3_2_2022_08_19_1700}, and present in \S\ref{subsubsec_proof_prop_2_16} a proof of Proposition~\ref{prop: comparision p-adic L-functions for alpha and beta} which we stated in the previous chapter. 

We study various Selmer groups (relevant to our discussion)  in \S\ref{subsect: Tate-Shafarevich} in terms of objects of arithmetic interest (e.g. Tate--Shafarevich groups and regulators), which we also define in this portion. 

In \S\ref{sec_modules_of_algebraic_padic_L_functions}, we introduce modules of algebraic $p$-adic $L$-functions building on Perrin-Riou's philosophy and formulate in \S\ref{sect: main conjecture for critical forms} the Iwasawa main conjectures for $\theta$-critical forms in terms of these. Our main result in this portion, Proposition~\ref{prop: equivalence of main conjectures}, tells us that the Iwasawa main conjectures for either of the $p$-stabilizations, \ref{item_MCalpha} and \ref{item_MCbeta}, are equivalent to one another provided that our Iwasawa theoretic $\cL$-invariant does not vanish. 

The nomenclature ``the module of $p$-adic $L$-functions'' is due to Perrin-Riou (cf. \cite{perrinriou95}, \S2). This terminology is motivated by the expectation that this module should be generated by an appropriate (analytic) $p$-adic $L$-function (cf. \cite[\S4.2.2]{perrinriou95}; see also \ref{item_MCalpha} and \ref{item_MCbeta} below). This expectation, in fact, is a form of the relevant Iwasawa main conjecture, which generalizes the classical Main Conjectures of Mazur and Greenberg to the non-critical non-ordinary setting. In this sense, the module of $p$-adic $L$-functions generalizes the notion of the characteristic polynomial of the Iwasawa theoretic Selmer group of an elliptic curve at an ordinary prime.

A further justification is the fact that they recover, via Iwasawa descent, the conjectural leading terms of the $p$-adic $L$-functions that the $p$-adic variants of the conjectures of Beilinson--Bloch--Kato predict (see \cite[Theorem 3.5.2]{perrinriou95}; see also Theorems~\ref{thm:descent thm: noncentral case} and \ref{thm_331_2022_04_29_1629} below). Note that the theory of Selmer complexes provides a natural framework for Iwasawa descent, and the modules of $p$-adic $L$-functions are, in a certain sense, more natural objects than characteristic polynomials of classical Iwasawa modules. This point of view was stressed in \cite{nekovar06} and  \cite{benoisextracris}.  

The final section in this chapter (\S\ref{sec_Iwasawa_theoretic_descent}), which is also the lengthiest, is where we develop Iwasawa descent for our modules of $p$-adic $L$-functions, which are then employed to prove our leading term formulae for these. Our main results in this portion are Theorem~\ref{thm:descent thm: noncentral case} (which concerns both non-critical and non-central critical values) and Theorem~\ref{thm_331_2022_04_29_1629} (which concerns central critical values). We would like to draw the attention of the reader also to Theorem~\ref{thm: bockstein map in central critical case}, which illustrates a key discrepancy between modules of $p$-adic $L$-functions attached to the $\theta$-critical $p$-stabilization and that to the slope-zero $p$-stabilization, exhibiting itself in the difference between their orders of vanishing at the central critical point. In view of the Iwasawa main conjectures \ref{item_MCalpha} and \ref{item_MCbeta} combined with Proposition~\ref{prop: comparision p-adic L-functions for alpha and beta}, this difference is accounted by the Iwasawa theoretic $\cL$-invariant, which is expected to acquire a pole at the central critical point to compensate this discrepancy.

\section{Iwasawa theory for $\theta$-critical forms: Preliminaries}
\subsection{} 
\label{subsec311_2022_08_24_0934}
Throughout this chapter we use the following notation and conventions. 
Let  $f\in S_k(\Gamma_1(N), \varepsilon_f)$ be a newform. We fix a prime number $p\geqslant 3$ and 
denote by $\alpha$ and $\beta$  the roots of its Hecke polynomial $X^2-a_p X+p^{k-1}\epsilon_f(p)$ of $f$ at $p$.  Let $\Gamma_p:=\Gamma_1(N)\cap \Gamma_0(p)$ and let 
$f_\alpha, f_\beta \in S_k(\Gamma_p, \varepsilon_f)$ be the $p$-stabilizations of $f$ with respect to $\alpha$ and $\beta$ respectively.  
We will always assume, as in \S\ref{subsubsec_221_04042022}, that $v_p(\alpha)=k-1$ and that  $f_\alpha$ is 
a $\theta$-critical newform of weight $k$ and nebentypus $\varepsilon_f$.
We let $V_f$ denote Deligne's $p$-adic Galois representation attached to $f$, so that we have $V_f\simeq V_{x_0} \simeq V_{f_\alpha}$ in the notation of Chapter~\ref{chapter:critical L-functions}. 
We denote by $f^*$  the complex conjugate of $f$ (see Section~\ref{subsect: the motive attached to f}). We recall from \eqref{eqn_2411_2022_05_11_0827} the canonical 
non-degenerate pairings
\begin{equation}
\nonumber
\begin{aligned}
&V_{f^*} \otimes_E V_f \lra E(\chi^{1-k}), \\
[\,,\,]\,:\,\,\, 
&\Dc (V_{f^*})\times \Dc (V_{f}) \lra 
\Dc (E (\chi^{1-k}))\simeq E\,.
\end{aligned}
\end{equation}

\subsection{}
\label{subsec: duality of punctual representations} 

Recall that the restriction of $V_f$ on the decomposition group at $p$ decomposes into a direct sum of $p$-adic representations
\begin{equation}
\nonumber
V_f=V_f^{(\alpha)}\oplus V_f^{(\beta)},
\end{equation}
where $V_f^{(\alpha)}$ and $V_f^{(\beta)}$ are crystalline one-dimensional representations of  Hodge--Tate weights $1-k$ and $0$ respectively. 
More precisely,
\[
V_f^{(\alpha)}=E(\nu_\alpha\chi^{1-k}), \qquad V_f^{(\beta)}=E(\nu_\beta),
\]
where $\nu_\alpha$ and $\nu_\beta$ are non-trivial unramified characters
such that $\nu_\alpha \nu_\beta=\varepsilon_f.$
In particular,  
\[
\Dc (V_f^{(\alpha)})=\Fil^{k-1}\Dc (V_f).
\]
This property has important consequences on the structures of various Selmer groups, which we discuss in the present Chapter. 

The roots of the Hecke polynomial of the dual form $f^*$ at $p$ are
$$\alpha^*:=p^{k-1}/\beta=\varepsilon_f^{-1}(p)\alpha\,, \qquad \beta^*=p^{k-1}/\alpha:=\varepsilon_f^{-1}(p)\beta\,.$$
To ease our notation, we will write  $f^*_{\alpha}$ and $f^*_{\beta}$ to denote the stabilizations $f^*_{\alpha^*}$ and $f^*_{\beta^*}$. 

The eigenform $f^*_{\alpha}$ is also $\theta$-critical.
As a reflection of this fact,  the restriction of $V_{f^*}$ to the decomposition group at $p$ factors as 
\[
V_{f^*}=V_{f^*}^{(\alpha)}\oplus V_{f^*}^{(\beta)}\,,
\]
where  $\Dc \left (V_{f^*}^{(\alpha)}\right )=\Fil^{k-1}\Dc (V_{f^*}).$ The duality \eqref{eqn_2411_2022_05_11_0827}
gives rise to the isomorphisms
\begin{equation}
\label{eqn: local duality for critical case}
V_{f^*}^{(\alpha)} \simeq \left (V_f^{(\beta)}\right)^*(1-k)\,, \qquad 
V_{f^*}^{(\beta)} \simeq \left (V_f^{(\alpha)}\right)^*(1-k).
\end{equation}
We recall that the differential forms associated to $f$ and $f^*$ fix  bases $\{\eta_{f}^\alpha\}$ of $\Fil^{k-1}\Dc (V_{f})=\Dc \left (V_f^{(\alpha)}\right )$ and $\{\eta_{f^*}^\alpha\}$ of  $\Fil^{k-1}\Dc (V_{f^*})$, respectively. As in Section~\ref{subsec_slope_zero_padic_L}, we choose $\eta_{f}^\beta\in \Dc (V_f^{(\beta)})$ and $\eta_{f^*}^\beta\in \Dc (V_{f^*}^{(\beta)})$
in such a way that 
$$[\eta_f^\beta, \eta^\alpha_{f^*}]=1=[\eta_{f}^\alpha, \eta^\beta_{f^*}]\,,$$ 
$$[\eta_f^\beta, \eta^\beta_{f^*}]=0=[\eta_{f}^\alpha, \eta^\alpha_{f^*}]\,.$$ 
For any $\eta \in \Dc (V_f)$ and  $j\in \ZZ,$ we denote by $\eta [j]\in \Dc (V_f(j))$ the twist of $\eta$ by the canonical generator $d_j:=t^{-j}\otimes \varepsilon^{\otimes j}$ of $\Dc (\Qp (j))$;  cf. Section~\ref{subsubsec_1151_2028_08_25_1200}.

\subsection{} We denote by $H^1_\Iw (V_f)$ the Iwasawa cohomology
with coefficients in $V_f$, namely,
\[
H^1_\Iw (V_f):=\left(\underset{n}{\varprojlim} H^1(G_{\QQ(\zeta_{p^n})}, T_f)\right)\otimes_{\cO_E}E\,,
\]
where $T_f\subset V_f$ is any $G_\QQ$-stable $\cO_E$-lattice; cf. \eqref{eqn: global Iwasawa cohomology}. Recall that the restriction map
\[
H^1_\Iw (V_f) \lra H^1_\Iw (\Qp, V_f)
\]
is injective thanks to Kato~\cite{kato04} and we may identify $H^1_\Iw (V_f)$ with its image in $H^1_\Iw (\Qp, V_f)$. We let 
 $$\res_{p,\star}\,:\, H^1_\Iw (V_f(k)) \lra H^1_\Iw (\Qp, V_f^{(\star)}(k))\,,\qquad \star\in\{\alpha,\beta\}$$
denote the composition of the restriction map $H^1_\Iw (V_f(k)) \xrightarrow{\res_p} H^1_\Iw (\Qp, V_f(k))$ with the projection onto $H^1_\Iw (\Qp, V_f^{(\star)}(k)).$ Recall the Beilinson--Kato 
elements $\bz (f,\xi) \in H^1_\Iw (V_f(k)).$ It follows from \eqref{eqn: local duality for critical case} that Kato's $p$-adic $L$-functions of $f^*$  can be written in the form
\begin{equation}
\label{eqn: Kato p-adic L-functions in critical case}
\begin{aligned}
&L_{\mathrm{K},\beta^*} ( f^*,\xi^*)=
\left < \res_{p,\alpha} (\bz (f,\xi)), c\circ \Exp_{V_{f^*}^{(\beta)},0}\left (\widetilde{\eta}_{f^*}^{\beta}
\right )^\iota \right >_{V_{f^*}^{(\beta)}} \in  \mathscr H(\Gamma), \qquad 
\widetilde{\eta}_{f^*}^{\beta}=\eta_{f^*}^{\beta}\otimes (1+\pi)\,,\\
&L_{\mathrm{K},\alpha^*}^{[0]} ( f^*,\xi^*)=
\left < \res_{p,\beta} (\bz (f,\xi)), c\circ \Exp_{V_{f^*}^{(\alpha)},0}\left (\widetilde{\eta}_{f^*}^{\alpha}
\right )^\iota \right >_{V_{f^*}^{(\alpha)}} \in  \mathscr H(\Gamma), \qquad 
\widetilde{\eta}_{f^*}^{\alpha}=\eta_{f^*}^{\alpha}\otimes (1+\pi)\,.
\end{aligned}
\end{equation}
Let us put
\[
z (f,\xi,j):=\pr_0 \left (\Tw_{j-k} (\bz (f,\xi))\right )\in H^1 (V_f(j))\,,
\]
where $\pr_0\,:\, H^1_\Iw (V_{f}(j))\rightarrow  H^1(V_{f}(j))$
denotes the canonical projection.

The following set of statements will be used repeatedly.
\begin{proposition}
\label{prop:projection of beilinson-kato element} 

\item[i)]{} We have
\[
H^1_\Iw (V_f) \cap H^1_\Iw (\Qp, V_f^{(\beta)})=\{0\}. 
\]

\item[ii)]{} For any $j\in \ZZ$, 
\[
\res_p \bigl (z (f,\xi,j)\bigr )\in H^1(\Qp, V_f^{(\beta)}(j))\, \iff\, L_{\mathrm{S},\beta^*} (f^*,\xi^*;\chi^{k-j})=0.
\]

\item[iii)]{} For  all but finitely many $j\in \ZZ$ we have
\[
\res_p \bigl (z (f,\xi,j)\bigr )\notin H^1(\Qp, V_f^{(\beta)}(j)).
\]
\end{proposition}
\begin{proof} \item[i)] This assertion follows from \eqref{eqn: Kato p-adic L-functions in critical case} together with the fact that $L_{\mathrm{K},\beta^*}=L_{\mathrm{S},\beta^*}$ are non-zero.

\item[ii)] We have
\[
L_{\mathrm{K},\beta^*} ( f^*,\xi^*;\chi^{k-j} )=
\left (\res_{p,\alpha} \bigl (z (f,\xi,j)\bigr ) , c\circ \pr_0 \left ( \Tw_{k-j}\circ \Exp_{V_{f^*}^{(\beta)},0}\left (\widetilde{\eta}_{f^*}^{\beta}
\right )^\iota \right ) \right  )_{V_{f^*}^{(\beta)}(k-j)},
\] 
where $(\,,\, )_{V_{f^*}^{(\beta)}(k-j)}$ denotes the local duality.
By Theorem~\ref{thm:large exponential for phi-Gamma modules}(ii), we have
\[
\pr_0 \left ( \Tw_{k-j}\circ \Exp_{V_{f^*}^{(\beta)},0}\right )=(-1)^{k-j}
\pr_0 \left (\Exp_{V_{f^*}^{(\beta)}(k-j),k-j}(\partial^{j-k}\otimes d_{k-j})\right )\,.
\] 
Together with the explicit formulae in \eqref{eqn:specialization of PR formulae}, this 
implies that 
\[
\pr_0 \left ( \Tw_{k-j}\circ \Exp_{V_{f^*}^{(\beta)},0}\left (\widetilde{\eta}_{f^*}^{\beta}
\right )^\iota \right )
\]
does not vanish, since the Euler-like factor
\[
(1-\beta^{-1}p^{k-j-1})(1-\beta p^{j-k})^{-1}
\]
does not vanish for any $j\in \ZZ$ (note that $v_p(\beta)=0$ and $\beta\neq 1$). 
Therefore,
\[
\res_p \bigl (z (f,\xi,j)\bigr )\in H^1(\Qp, V_f^{(\beta)}(j))\,\iff\, 
L_{\mathrm{K},\beta^*} (f^*,\xi^*;\chi^{k-j})=0\,
\iff\, L_{\mathrm{S},\beta^*} (f^*,\xi^*;\chi^{k-j})=0.
\]
This proves the second portion of our proposition.

\item[iii)] It follows from Kato's fundamental work (cf. Proposition~\ref{prop_2_17_2022_05_11_0842}) that
$L_{\mathrm{K},\beta^*}$  coincides with $L_{\mathrm{S},\beta^*}$, which is non-zero and bounded. Therefore, $L_{\mathrm{K},\beta^*}$
can have only finitely many zeroes. Part (iii) of the proposition now follows from (ii). 
\end{proof}

\section{Selmer groups and critical ${\mathscr L}$-invariants}
\label{sec_3_2_2022_08_19_1700}
\subsection{} According to the general philosophy (which manifests itself in $p$-adic Bloch--Kato--Beilinson conjectures in concrete form), the $p$-adic $L$-function associated to $f^*_{\alpha^*}$ should carry arithmetic information that is conjecturally encoded by the Selmer group attached to the Galois representation $V_{f}$. Fix an $\cO_E$-lattice  $T_f$ of $V_f$ stable under the Galois action. In what follows, we shall introduce Selmer groups associated to various local conditions. At  $\ell \in S\setminus\{p\}$ (where we recall that $S$ is the set of prime divisors of $Np$) we will always take the unramified local condition 
$$H^1_{\rm f}(\QQ_\ell,V_f(j)):=\ker\left(H^1(\QQ_\ell,V_f(j))\lra H^1(\QQ_\ell^{\rm ur},V_f(j)) \right)$$ 
as the local conditions.  We will consider the following local conditions at $p:$
\begin{itemize}
\item[$\bullet$]{} Bloch--Kato condition: $H^1_{\rm f}(\Qp, V_f(j))$\,.
\item[$\bullet$]{} $H^1_\alpha(\Qp, V_f(j)):=H^1(\Qp, V^{(\alpha)}_f(j))$\,.
\item[$\bullet$]{} $H^1_\beta(\Qp, V_f(j)):=H^1(\Qp, V_f^{(\beta)}(j))$\,.
\item[$\bullet$]{} Relaxed local condition: $H^1_{\{p\}}(\Qp, V_f(j)):=H^1(\Qp, V_f(j))$ \,.
\item[$\bullet$]{} Strict local condition: $H^1_{0}(\Qp, V_f(j)):=0.$
\end{itemize}




\begin{defn}
\label{defn_3_1_2022_29_04_11_09}
For $\star\in \{{\rm f},\alpha,\beta, \{p\},0\},$ we denote by 
 $H^1_\star (V_f(j))$, $H^1_\star (T_f(j))$ and  $H^1_\star (V_f/T_f(j))$ 
 the Selmer groups associated to local conditions 
 $\{H^1_{\rm f}(\QQ_\ell,V_f(j))\,: p\neq \ell\in S\}$ and $H^1_\star (\Qp, V_f(j))$; cf. Section~\ref{subsec_123_21_11_1607}. 
\end{defn}

We remark that 
\begin{equation}
\label{eqn: critical selmer}
H^1_\star (V_f(j))=\ker \left (H^1_{\{p\}}(V_f(j))\lra
\frac{H^1(\Qp, V_{f}(j))}{H^1_\star (\Qp, V_f(j))}\right ),
\end{equation}
where
\[
H^1_{\{p\}}(V_{f}(j)):=\ker \left (
{ H^1_S(V_f(j))} \lra \underset{\ell \in S\setminus \{p\}}\bigoplus 
\frac{H^1(\QQ_\ell, V_f(j))}{H^1_{\rm f}(\QQ_\ell, V_f(j))} \right )
\]
is the relaxed Selmer group. 

\subsubsection{}
By Proposition~\ref{prop: about general tate-shafarevich} (Poitou--Tate global duality), the Bloch--Kato Selmer group $H^1_{\rm f}(V_f(j))$ sits in the exact sequence 
\[
0\lra H^1_{\rm f}(V_{f}(j)) \lra H^1_{\{p\}}(V_{f}(j))
\lra \frac{H^1(\Qp, V_{f}(j))}{H^1_{\rm f}(\Qp, V_{f}(j))}\lra \left (\frac{H^1_{\rm f}(V_{f^*}(k-j))}{H^1_0(V_{f^*}(k-j))}\right )^*
\lra 0\,. 
\]
Moreover, one has
\begin{equation}
\label{eqn:dimension of Selmer groups of Tate duals}
\dim_E H^1_{\rm f}(V_f(j))-\dim_E H^1_{\rm f}(V_{f^*}(k-j))=\dim_E t_{V_f(j)} -1,
\end{equation} 
where $t_{V_f(j)}$ denotes the Bloch--Kato tangent space of $V_f(j)$. We remark that \eqref{eqn:dimension of Selmer groups of Tate duals} is a particular case of the general formula proved in  \cite[Section~II,2.2.2]{fontaine-pr94}. 

\subsection{}
\label{subsec: analysis of selmer groups} We outline below what is already known to the experts and what is expected concerning the dimensions of various Selmer groups.
\label{subsec_322_2022_04_08}

\subsubsection{Non-central critical case: $1\leqslant j\leqslant k-1$ and $j\neq \frac{k}{2}$} \label{subsubsec_3221_2022_04_08}
In this scenario, we have $H^1_{\rm f} (V_f(j))=0$ and $\dim_E  H^1_{\{p\}}(V_f(j))=1$, cf. \cite{kato04}.

\subsubsection{Non-critical case $j\geqslant k$}
\label{subsubsec_3222_2022_04_28_1537}
In this case,  $H^1_{\rm f}(\Qp, V_f(j))=H^1 (\Qp, V_f(j))$  and therefore  $H^1_{\rm f}(V(j))=H^1_{\{p\}}(V(j))$. The Beilinson--Bloch--Kato conjecture predicts  that 
\[
\dim_EH^1_{\rm f}(V(j))=\dim_E H^1_{\{p\}}(V(j))=1.
\]
We remark that this agrees with the $p$-adic variants of this conjecture. More precisely, let us consider the condition
\begin{itemize}
\item[\mylabel{item_pB}{\bf $p$Bei})] $L_{{\mathrm S}, \beta^*} (f^*,{\xi^*; \chi^j})\neq 0$. 
\end{itemize}
This  condition holds for all but finitely many $j$, since the slope-zero $p$-adic $L$-function $L_{{\mathrm S}, \beta^*} (f^*,{\xi^*; \chi^j})$ is  bounded. The $p$-adic Beilinson conjecture of Perrin-Riou predicts that \eqref{item_pB} 
holds for all  $j\in \ZZ\setminus\{\frac{k}{2}\}$ (in particular, for all $j\geq k$, which is the case of interest in \S\ref{subsubsec_3222_2022_04_28_1537}). 

\begin{lemma}
\label{lemma_padicBeilinson_implies_B_J}
If \eqref{item_pB} holds true at  $k-j$, then so do the conditions \eqref{item_BK_j} and \eqref{item_J} below:
\begin{itemize}
\item[\mylabel{item_BK_j}{\bf B})] 
$H^1_{\rm f}(V_{f^*}(k-j))=0$ 
and 
$\dim_E H^1_{\rm f}(V_f(j))=1$ (Beilinson--Bloch--Kato conjecture)\,.
\item[\mylabel{item_J}{\bf J})] The restriction map 
$$\res_p\,:\,H^1_{\mathrm{f}}(V_f(j)) \lra H^1(\Qp,V_f(j))$$
is injective (Jannsen's conjecture, cf. \cite{Ja89})\,.
\end{itemize}
We remark that the conclusion of \eqref{item_J} is equivalent to the vanishing of the strict Selmer group $H^1_0(V_f(j))$.
\end{lemma}

\begin{proof}
The fact that $H^1_{0}(V_{f^*}(k-j))=0$ follows if one can show the existence of an Euler system 
$\{c_{\QQ(\mu_n)}\}_{n\nmid N}$ for the dual $V_f(j)$ of the Galois representation $V_{f^*}(k-j)$, for which we have $c_{\QQ}\neq 0$ \cite[Theorem~II.2.3]{rubin00}.  If the eigenform form $f$ does not have CM, then thanks to Kato's explicit reciprocity laws (cf. \cite{kato04}, Theorem 16.6), the Beilinson--Kato element Euler system verifies this property whenever $L_{{\mathrm S}, \beta^*} (f^*, \xi^*; \chi^{k-j})\neq 0$. If the eigenform $f$ has CM, then one utilizes elliptic units in place of the Beilinson--Kato elements and proceeds in an identical fashion. Since $H^1_{\rm f}(\Qp, V_{f^*}(k-j))=0$ for $j\geqslant k,$  we deduce that $H^1_{\rm f}(V_{f^*}(k-j))=H^1_{0}(V_{f^*}(k-j))=0$. Moreover, since $\dim_E t_{V_f(j)}=2$ for $j\geq k$, it follows from \eqref{eqn:dimension of Selmer groups of Tate duals} that $\dim_E H^1_{{\rm f}}(V_f(j))=1$. This discussion shows that the condition \eqref{item_BK_j} holds true.

In order to verify \eqref{item_J}, one only needs to show that the map $\res_p$ is not identically zero, as we have already explained that $H^1_{\rm{f}}(V_f(j))$ is an $E$-vector space of dimension one under our running assumption that \eqref{item_pB} holds true for $k-j$. In other words, one needs to exhibit a single element  $c\in H^1_{\mathrm{f}}(V_f(j))$ such that $\res_p(c)\neq 0$. It follows from Proposition~\ref{prop:projection of beilinson-kato element}
that we can take $c=z(f,\xi,j)$ and this concludes our proof.
\end{proof}
\subsubsection{Non-critical case $j\leqslant 0$}
In this case, the Beilinson--Bloch--Kato conjecture predicts that
\begin{equation}
\label{eqn: selmer groups the case j<0}
H^1_{\rm f}(V(j))=\{0\}, \quad \dim_E H^1_{\{p\}}(V(j))=1.
\end{equation}
Using  \eqref{eqn:dimension of Selmer groups of Tate duals}, we deduce  that the conclusion of \eqref{item_BK_j} for the {\it dual} eigenform $f^*$ implies (\ref{eqn: selmer groups the case j<0}). In particular, \eqref{eqn: selmer groups the case j<0} follows from  the non-vanishing of  $L_{{ \mathrm S}, \beta}(f, \xi; \chi^{k-j}).$

\subsubsection{} In summary, it is known or expected that  $H^1_{\{p\}}(V_f(j))$ is a one-dimensional vector space which injects into $ H^1(\Qp,V_f(j))$ whenever $j\neq \frac{k}{2}$.

\subsubsection{}
We turn our attention to the groups $H^1_\alpha (V_f(j))$ and $H^1_\beta (V_f(j))$. Our analysis of these Selmer groups will play a crucial role in \S\ref{subsec:critical L-invariants}, where we introduce our critical $\mathscr{L}$-invariants. 

Since we have
\[
H^1_{\rm f}(\Qp,V_f(j))=H^1 (\Qp, V_f^{(\beta)}(j))\,,\qquad \textrm{if  $1\leqslant j\leqslant k-1,$}
\]
we also have
\[
H^1_\beta(V_f(j))=H^1_{\rm f} (V_f(j))\,,\qquad \textrm{if  $1\leqslant j\leqslant k-1$\,.}
\]
In particular, $H^1_\beta(V_f(j))=0$ if $1\leqslant j\leqslant k-1$ and $j\neq \frac{k}{2}.$

\subsubsection{} 
\label{subsec:vanishing of Hbeta}
If  $j\geqslant k$,  Lemma~\ref{lemma_padicBeilinson_implies_B_J} together with 
Proposition~\ref{prop:projection of beilinson-kato element} imply that
\begin{equation}
\label{lemma_H1beta_zero_under_pB}
L_{\mathrm{S},\beta^*}(f^*,\xi^*,\chi^{k-j})\neq 0\,\, \implies\,\, H^1_\beta(V_f(j))=0\,,
\end{equation}
and therefore 
\be \label{eqn_3_4_2022_0404}
H^1_{\rm f}(\Qp, V_f(j)) = \res_p\left(H^1_{\rm f}(V_f(j))\right)\oplus H^1 (\Qp, V_f^{(\beta)}(j))
\,.
\ee
Thanks to the implication \eqref{lemma_H1beta_zero_under_pB}, we deduce that $H^1_\beta(V_f(j))=0$ for all but finitely many $j\geq k$.

We remark that the condition \eqref{eqn_3_4_2022_0404} is equivalent to the regularity of the submodule 
$$\Dc (V_f^{(\beta)}(j)) \subset \Dc (V_f(j))$$ 
in the sense of \cite{benoisextracris}.

In summary, it is expected that $H^1_\beta (V_f(j))=0$ for all $j\neq \frac{k}{2}$.

\subsubsection{} 
\label{subsubsec_3227_2022_04_29}
To analyze the Selmer group $H^1_\alpha(V_f(j))$, we will assume that $f$ has CM by an imaginary quadratic field $K$. More precisely, $f=\theta(\psi)$ is the theta-series of a Hecke character $\psi: K^\times\textbackslash \mathbb{A}_K^\times \rightarrow L^\times \xhookrightarrow{\iota_\infty}\mathbb{C}^\times$ of an imaginary quadratic field $K$ with infinity type $(1-k,0)$ and conductor coprime to $p$, where $L$ is the number field generated by the image of $\psi$ on the finite id\`eles $\mathbb{A}_{K}^{\infty,\times}$ over $K$. Since we work in the setting that $f$ admits a $\theta$-critical $p$-stabilization $f_{\alpha}$, it follows that the prime $p=\p\p^c$ splits in $K/\QQ$, where $c$ denotes the unique non-trivial element of $\Gal(K/\QQ)$. Recall also the embedding $\iota_p: \overline{\QQ}\to \mathbb{C}_p$ we have fixed and assume that the prime $\p$ is induced from this embedding. Note also that the Hecke field of $f$ is contained in $L$, and we enlarge $E/\QQ_p$ if necessary so that it contains $\iota_p(L)$.

In this paragraph, we closely follow the discussion in \cite[\S15.8]{kato04} and adopt the normalizations therein; see also \cite[\S1.3]{LLZ_critical} and \cite[\S5.2]{LLZ2}.  Let $\psi_\p$ denote the $p$-adic avatar of $\psi$, determined by the embedding $\iota_p$. Namely, for $x\in \mathbb{A}_{K}^{\infty,\times}$, we put $\psi_\p(x):=x_\p^{1-k}\cdot \iota_p\circ \psi(x)$. In particular, $\psi_\p$ is ramified at $\p$. We denote by $\psi_\p$ also the associated Galois character via the geometrically-normalized Artin map $ K^\times\textbackslash\mathbb{A}_{K}^{\infty,\times}\xrightarrow{\mathfrak{a}} {\rm Gal}(K^{\rm ab}/K)$ of class field theory. In more precise terms, we define $\psi_\p(\mathfrak{a}(x)):=\psi_\p(x)^{-1}$ for $x\in \mathbb{A}_{K}^{\infty,\times}$. We then have 
$$V_f\simeq {\rm Ind}_{K/\QQ}\psi_{\p}\,.$$
Note that $\iota_p$ induces an identification $G_{K_\p}\xrightarrow{\sim} G_{\QQ_p}$. Since $\psi_\p$ is ramified at $\p$, it follows that $G_{\QQ_p}$ acts on $V_{f}^{(\alpha)}$ via $\psi_\p$, and in turn also that it acts on $V_{f}^{(\beta)}$ via $\psi_\p^c$. In particular, $\psi_\p(\p^c)=\beta$.
Shapiro's lemma allows us to identify $H^1_\alpha(V_f(j))$ canonically with the Selmer group
$$H^1_{\emptyset,0}(\psi_\p\chi^j):=\ker\left(H^1_{\{p\}}(\psi_\p\chi^j)\lra H^1(K_{\p^c},\psi_\p\chi^j) \right)\,.$$
Let us also define the dual Selmer group
\be\label{Selmer_Kazt_true_Greenberg_dual}
H^1_{0,\emptyset}(\psi_\p^{-1}\chi^{1-j}):=\ker\left(H^1_{\{p\}}(\psi_\p^{-1}\chi^{1-j})\lra H^1(K_{\p},\psi_\p^{-1}\chi^{1-j}) \right)\,.
\ee

Let $\mathfrak{f}$ denote the conductor of $\psi$ and let 
$L_{\p,\ff}^{\rm Katz}\in {W}(\overline{\FF}_p)[[H_{\ff p^\infty}]]$ denote the Katz $p$-adic $L$-function (which depends on the choice of the embedding $\iota_p$) of tame conductor $\ff$, where  ${W}(\overline{\FF}_p)$ is the Witt vectors of $\overline{\mathbb{F}}_p$ (which coincides with the $p$-adic completion of the maximal unramified extension $\Zp^{\mathrm{ur}}$ of $\ZZ_p$), $H_{\ff p^\infty}:=\varprojlim H_{\ff p^n}$ and $H_{\ff p^n}$ is the ray class group of $K$ of conductor $\ff p^n$.

\begin{remark} 
\label{remark_two_Greenberg_Selmer_groups_comparison_MC_Rubin}
At first sight, the definition \eqref{Selmer_Kazt_true_Greenberg_dual} may seem unnatural. We discuss its significance in this remark. 

Note that whenever $1\leq j\leq k-1$, so that 
$$L(f^*,j)=L(\overline{\psi},j)=L(\psi^{-1}\mathbf{N}_{K}^{1-j},0)$$
is a critical value (where $\mathbf{N}_{K}:=\mathbf{N}\circ \mathbb{N}_{K/\QQ}$ is the adelic norm character), the Selmer group $H^1_{0,\emptyset}(\psi_\p\chi^{j})$ coincides with the Bloch--Kato Selmer group. More generally, if $\nu$ is a Hecke character of $K$ such that $L(\nu^{-1},0)=L(\overline{\nu},1)$ is critical and $\nu_\p$ denotes its $p$-adic avatar induced by $\iota_p$, then the Selmer group  $H^1_{0,\emptyset}(\nu_\p\chi)$ defined in a manner analogous to \eqref{Selmer_Kazt_true_Greenberg_dual} is critical. In other words, the Iwasawa theoretic Selmer group 
$$ H^1_{0,\emptyset}(\ZZ_p(1)\otimes_{\ZZ_p}\ZZ_p[[H_{\ff p^\infty}]]):=\ker\left(H^1_{\{p\}}(\ZZ_p(1)\otimes_{\ZZ_p}\ZZ_p[[H_{\ff p^\infty}]])\lra H^1(K_{\p},\ZZ_p(1)\otimes_{\ZZ_p}\ZZ_p[[H_{\ff p^\infty}]]) \right)$$
interpolates the Bloch--Kato Selmer groups of critical Hecke characters of conductor dividing $\ff p^\infty$. 

Since $L(\psi\mathbf{N}_K^j,0)$ is not critical for any $j$, the Selmer group $H^1_{0,\emptyset}(\psi_\p^{-1}\chi^{1-j})$ given as in \eqref{remark_two_Greenberg_Selmer_groups_comparison_MC_Rubin} need not coincide with the Bloch--Kato Selmer group. However, it still arises as a specialization of the Iwasawa theoretic Selmer group $ H^1_{0,\emptyset}(\ZZ_p(1)\otimes_{\ZZ_p}\ZZ_p[[H_{\ff p^\infty}]])$, and as such, the Iwasawa main conjectures (which were proved by Rubin under the additional hypothesis that $|H_{\ff}|$ is coprime to $p$) show that its size is controlled by the value $L_{\p,\ff}^{\rm Katz}(\psi_\p\chi^{j})$ of Katz' $p$-adic $L$-function outside its range of interpolation. In more precise terms,
\be\label{eqn_Rubin_H1alpha_vanishes}
L_{\p,\ff}^{\rm Katz}(\psi_\p\chi^{j})\neq 0 \hbox{ and } p\nmid |H_{\ff}| \qquad \implies \qquad  H^1_{0,\emptyset}(\psi_\p^{-1}\chi^{1-j})=0\,.
\ee
\end{remark}

We may use the observations in Remark~\ref{remark_two_Greenberg_Selmer_groups_comparison_MC_Rubin} to establish a vanishing criterion for $H^1_\alpha(V_f(j))=0$.
\begin{proposition}
 \label{prop_Rubin_H1alpha_vanishes}
 Suppose that the order of the ray class $H_{\ff}$ modulo $\ff$ is coprime to $p$ and that $L_{\p,\ff}^{\rm Katz}(\psi_\p\chi^j)\neq 0$. Then $H^1_\alpha(V_f(j))=0$.
\end{proposition}
\begin{proof}
Given an $E$-valued character $\xi$ of $G_K$, let us denote by $E(\xi)$ the one-dimensional $E$-vector space on which $G_K$ acts by $\xi$. We note that $E(\psi_\p^{-1}\chi^{1-j})$ is the Tate-dual of $E(\psi_\p\chi^{j})$ and the Selmer groups
$H^1_\alpha(V_f(j))\simeq H^1_{\emptyset,0}(\psi_\p\chi^j)$ and $H^1_{0,\emptyset}(\psi_\p^{-1}\chi^{1-j})$ are related by the Euler--Poincar\'e characteristic formula, which reads in this case
\be\label{Euler_Poincare_Hecke_char_Greenberg_Selmer}
\dim_E H^1_\alpha(V_f(j))={\rm dim}_E H^1_{\emptyset,0}(\psi_\p\chi^j)= \dim_E H^1_{0,\emptyset}(\psi_\p^{-1}\chi^{1-j})\,.
\ee
The proof of our proposition follows on combining \eqref{eqn_Rubin_H1alpha_vanishes} and \eqref{Euler_Poincare_Hecke_char_Greenberg_Selmer}.
\end{proof}

It would be tempting to use the conclusion of Proposition~\ref{prop_Rubin_H1alpha_vanishes} to deduce that $H^1_\alpha(V_f(j))=0$ for all but finitely many $j$ (paralleling our discussion in \S\ref{subsec:vanishing of Hbeta} regarding the Selmer group $H^1_\beta(V_f(j))$). However, this is not immediate (as in the case of $H^1_\beta(V_f(j))$) for the following reason. Since the function $\xi\mapsto L_{\p,\ff}^{\rm Katz}(\psi_\p\xi)$ on the characters $\xi$ of $\Gal(K(\mu_{p^\infty})/K)$ is an Iwasawa function, it follows that $L_{\p,\ff}^{\rm Katz}(\psi_\p\chi^{j_0})\neq 0$ for \emph{some} integer $j_0$ if and only if $L_{\p,\ff}^{\rm Katz}(\psi_\p\chi^{j})\neq 0$ for all but finitely many integers $j$. This observation combined with Proposition~\ref{prop_Rubin_H1alpha_vanishes} tells us that $H^1_\alpha(V_f(j))=0$ for all but finitely many $j$ provided that $L_{\p,\ff}^{\rm Katz}(\psi_\p\chi^{j_0})\neq 0$ for \emph{some} integer $j_0$. 

The problem is that, as we have noted above, no character of the form $\psi_\p\xi$ (where $\xi$ is a character of $\Gal(K(\mu_{p^\infty})/K)$) belongs to the interpolation range of the Katz $p$-adic $L$-function $L_{\p,\ff}^{\rm Katz}$ (since the Hodge--Tate weights of characters of this form at $\p$ and $\p^c$ are $(1-k+j,j)$ for some integer $j$, whereas the interpolation range for the Katz $p$-adic $L$-function consists of those characters with Hodge--Tate weights $(r,s)$ with $r\geq 1$ and $s\leq 0$) and we have no direct way of checking whether $L_{\p,\ff}^{\rm Katz}(\psi_\p\chi^{j_0})\neq 0$ for a single integer $j_0$.

The values $L_{\p,\ff}^{\rm Katz}(\psi_\p\chi^j)$ are governed by Perrin-Riou's $p$-adic Beilinson conjectures, in view of which, we conjecture that
\begin{itemize}
\item[\mylabel{item_pBCM}{\bf $p$Bei$'$})] $L_{\p,\ff}^{\rm Katz}(\psi_\p\chi^{j})\neq 0$ for $j\gg 0$.
\end{itemize}


\subsection{Critical $\mathscr L$--invariants}
\label{subsec:critical L-invariants}
Let us assume that \eqref{item_pB} is valid for a fixed integer $j$, so that the conditions \eqref{item_BK_j} and \eqref{item_J} hold for the eigenforms $f$ and $f^*$ for the given $j$ (recall that \eqref{item_pB} is indeed valid for all $j\gg 0$). Then  $H^1_{\{p\}}(V_f(j))$ is a one-dimensional $E$-vector space,
which injects into the two-dimensional $E$-vector space  
$H^1(\Qp, V_f(j))$ equipped with the canonical decomposition 
into the direct sum of one-dimensional spaces $H^1(\Qp, V^{(\alpha)}(j))$ and $H^1(\Qp, V^{(\beta)}(j)).$
This will allow us, in \S\ref{subsec:critical L-invariants},
to define analogues of Greenberg's $\mathscr L$-invariants in this scenario. On the analytic side, the interpolation formula in Theorem~\ref{thm_interpolative_properties}(i) shows that the critical
$p$-adic $L$-function vanishes on the critical strip. 
Our $\mathscr L$-invariants appear in leading term formulae for these critical $p$-adic $L$-functions (cf. Proposition~\ref{prop: comparision p-adic L-functions for alpha and beta}, whose proof we present in \S\ref{subsubsec_proof_prop_2_16} below),  as a justification to view this as an extremal incarnation of the exceptional zero phenomena.

\subsubsection{The case $j\geq k$}
\label{subsubsec_3231_2023_07_07_0911}
Let us consider the commutative diagram 
\begin{equation}
\begin{aligned}
\xymatrix{
\Dc (V_{f}^{(\beta)}(j)) \ar[r]^{\exp} & H^1(\Qp, V_{f}^{(\beta)}(j))\\
H^1_{\{p\}}(V_{f}(j)) \ar[u]^{\lambda_\beta}
\ar[d]_{\lambda_\alpha} \ar[r]^(.32){\res_p}  &H^1(\Qp, V_{f}^{(\alpha)}(j)) \oplus H^1(\Qp, V_{f}^{(\beta)}(j))
\ar[u]_{\pr_\beta} \ar[d]^{\pr_\alpha}\\
\Dc (V_{f}^{(\alpha)}(j))  \ar[r]^{\exp}  & H^1(\Qp, V_{f}^{(\alpha)}(j))
}
\end{aligned}
\end{equation}
where both exponential maps are isomorphisms (since $j\geq k$), and $\lambda_\alpha$ and $\lambda_\beta$ denote the unique maps making the diagram commute.  Note that $\lambda_\alpha$ (resp. $\lambda_\beta$) is an isomorphism
if and only if  $H^1_\beta(V_f(j))$ (resp. $H^1_\alpha (V_f(j))$) vanishes.

\begin{defn} For any  $j\geq k$ such that \eqref{item_pB} holds (so that $\lambda_\alpha$ is an isomorphism), we define the critical $\mathscr L$-invariant $\mathscr L^{\crit}(V_{f}(j))\in E$ as the unique quantity that validates the identity
\[
\lambda_\beta \circ \lambda_\alpha^{-1}\left (\eta^\alpha_f[j]\right )= {\mathscr L}^{\rm cr}(V_{f}(j)) \cdot \eta_f^\beta [j].
\]
\end{defn}
This definition can be recast in the following equivalent form. Let $z\in H^1_{\{p\}}(V_{f}(j))$ be a nonzero cohomology class. Let us write
\[
\res_p (x)= a\, \exp(\eta_f^{\alpha} [j])  + b\,\exp (\eta_f^\beta [j])\,.
\]
Then ${\mathscr L}^{\rm cr}(V_{f}(j))=\frac{b}{a} \in E$. Note again that $a\neq 0$ since \eqref{item_pB} holds. It is also clear from definitions that 
\be\label{eqn_2022_04_22_16_22}
{\mathscr L}^{\rm cr}(V_{f}(j))\neq 0 \quad\iff \quad b\neq 0 \quad \iff \quad H^1_\alpha(V_f(j))=0\,.
\ee

\subsubsection{The case $1\leq j\leq k-1$}
\label{subsubsec_critical_L_invariant_critical_range}
Let us now consider the commutative diagram 
\begin{equation}
\begin{aligned}
\xymatrix{
\Dc (V_{f}^{(\beta)}(j)) \ar[r]^{\exp} & H^1(\Qp, V_{f}^{(\beta)}(j))\\
H^1_{\{p\}}(V_{f}(j)) \ar[u]^{\lambda_\beta}
\ar[d]_{\lambda_\alpha} \ar[r]^(.34){\res_p}  &H^1(\Qp, V_{f}^{(\alpha)}(j)) \oplus H^1(\Qp, V_{f}^{(\beta)}(j))
\ar[u]_{\pr_\beta} \ar[d]^{\pr_\alpha}\\
\Dc (V_{f}^{(\alpha)}(j))   & H^1(\Qp, V_{f}^{(\alpha)}(j)) \ar[l]^{\exp^*}
}
\end{aligned}
\end{equation}
where the exponential and the dual exponential maps are isomorphisms, and $\lambda_\alpha$ and $\lambda_\beta$ denote the unique maps making the diagram commute. As in \S\ref{subsubsec_3231_2023_07_07_0911}, the map $\lambda_\alpha$ (resp. $\lambda_\beta$) is an isomorphism if and only if  $H^1_\beta(V_f(j))$ (resp. $H^1_\alpha (V_f(j))$) vanishes.

\begin{defn} For any  $1\leq j \leq k-1$ such that \eqref{item_pB} holds (so that $\lambda_\alpha$ is an isomorphism), we define the critical $\mathscr L$-invariant $\mathscr L^{\crit}(V_{f}(j))\in E$ as the unique quantity for which the following equality holds true:
\[
\lambda_\beta \circ \lambda_\alpha^{-1}\left (\eta^\alpha_f[j]\right )= {\mathscr L}^{\rm cr}(V_{f}(j)) \cdot \eta_f^\beta [j].
\]
\end{defn}
The interpolation formula for the slope-zero $p$-adic $L$-function shows that \eqref{item_pB} holds true if $1\leq j \leq k-1$ and $j\neq \frac{k}{2}$ or when $L(f^*,\frac{k}{2})\neq 0$. 

As before, this definition can be recast in the following equivalent form. Let $x\in H^1_{\{p\}}(V_{f}(j))$ be a nonzero cohomology class. If we write 
\[
\res_p (x)= a (\exp^*)^{-1}(\eta_f^{\alpha} [j])  +b\exp (\eta_f^\beta [j])\,,
\]
then ${\mathscr L}^{\rm cr}(V_{f}(j))=\frac{b}{a}$. As in \eqref{eqn_2022_04_22_16_22}, we have also in this scenario (i.e., when $1\leq j \leq k-1$ and $j\neq \frac{k}{2}$ or when $L(f^*,\frac{k}{2})\neq 0$)
\be\label{eqn_2022_04_26_14_24}
{\mathscr L}^{\rm cr}(V_{f}(j))\neq 0 \quad\iff \quad b\neq 0 \quad \iff \quad H^1_\alpha(V_f(j))=\{0\}\,.
\ee
by definitions.

\subsubsection{Iwasawa theoretic $\mathscr L$-invariants} 
\label{subsubsec_3233_2022_05_06_1354}
We next explain that the $\mathscr L$-invariants ${\mathscr L}^{\rm cr}(V_{f}(j))$ we have introduced above interpolate to an Iwasawa theoretic $\mathscr L$-invariant. 

Recall that by Kato's fundamental result \cite{kato04}, the $\Lambda [1/p]$-module $H^1_{\Iw}(V_{f}(j))$ is free  of rank one, and  for any $j\in \ZZ$ (note that we no longer require that the condition \eqref{item_pB} holds for this choice of $j$), the restriction map gives rise to an injective map
\[
\res_p \,:\, H^1_{\Iw}(V_{f}(j)) \hookrightarrow H^1_\Iw(\Qp, V_{f}^{(\alpha)}(j)) \oplus H^1_{\Iw}(\Qp, V_{f}^{(\beta)}(j))\,.
\]
Consider the exponential maps given as in \eqref{eqn: integral exponentials}:
\begin{equation}
\begin{aligned}
& \Exp_{V_{f}^{(\alpha)}(j),j-k+1 }\,:\, \mathfrak D(V_{f}^{(\alpha)}(j)) \lra H^1_\Iw (\Qp, V_{f}^{(\alpha)}(j))\,,\\
& \Exp_{V_{f}^{(\beta)}(j),j}\,:\, \mathfrak D(V_{f}^{(\beta)}(j)) \lra H^1_\Iw (\Qp, V_{f}^{(\beta)}(j))\,.
\end{aligned}
\end{equation}
To simplify the notation, we will denote them by $\Exp_{\alpha,j}$ and $\Exp_{\beta,j}$, respectively. One then has a commutative diagram
\begin{equation}
\begin{aligned}
\xymatrix{
\mathfrak D (V_{f}^{(\beta)}(j)) \ar[r]^{\Exp_{\beta,j}}_{\sim} & H^1_\Iw (\Qp, V_{f}^{(\beta)}(j))\\
H^1_{\Iw}(V_{f}(j)) \ar[u]^{\lambda^{\Iw}_\beta}
\ar[d]_{\lambda^{\Iw}_\alpha} \ar[r]^-{\res_p}  &H_\Iw^1(\Qp, V_{f}^{(\alpha)}(j)) \oplus H_\Iw^1(\Qp, V_{f}^{(\beta)}(j))
\ar[u]_{\pr_\beta} \ar[d]^{\pr_\alpha}\\
\mathfrak D (V_{f}^{(\alpha)}(j)) \ar[r]^{\Exp_{\alpha,j}}_{\sim}  & H^1_\Iw(\Qp, V_{f}^{(\alpha)}(j)) 
}
\end{aligned}
\end{equation}
where  $\lambda_\alpha^\Iw$
and $\lambda_\beta^\Iw$ denote the unique maps making the diagram commute. 
As we have noted above, the restriction map $\res_p$ is injective by \cite{kato04}.

Let  $\bz$ be a generator of $H^1_{\Iw}(V_{f}(j))$. Then 
\[
\res_p (\bz)= F_j \Exp_{\alpha,j} (\eta_f^{\alpha} [j] )+ G_j \Exp_{\beta,j} 
(\eta_f^\beta [j])
\]
for some $F_j,G_j\in \Lambda_E.$

\begin{defn} 
\label{def_2022_04_27_1457}
For any $j\in \ZZ,$ we define the Iwasawa theoretic $\mathscr L$-invariant $\mathscr L^{\crit}_{\Iw}(V_{f}(j))$ by
\[
\mathscr L^{\crit}_{\Iw}(V_{f}(j))=G_j/F_j \in \Frac (\Lambda_E).
\]
\end{defn}
Recall the Euler-like factors we have defined in Section~\ref{sec_new_2_3_2022_03_14}:
\[
e_{p,\alpha}(f,\mathds{1},j)= \left (1-\frac{p^{j-1}}{\alpha}\right )\cdot \left (1-\frac{\beta}{p^{j}}\right ),
\qquad
e_{p,\beta}(f,\mathds{1},j)= \left (1-\frac{p^{j-1}}{\beta}\right )\cdot \left (1-\frac{\alpha}{p^{j}}\right ).
\]

\begin{proposition} 
\label{prop:about L-invariants}
\item[i)] $\mathscr L_\Iw^{\rm cr}(V_{f}(j))$ is well-defined for every integer $j$ and we have 
$$\mathscr L_\Iw^{\rm cr}(V_{f}(j))=\Tw_j\circ  \mathscr L^{\crit}_\Iw(V_{f}).$$

\item[ii)] For all $j\gg 0$, we have $F_j(0)\neq 0$ and $\mathscr L^{\crit}_\Iw(V_{f}(j))$ is well-defined. Moreover, for all such $j$, 
\be\label{eqn_2022_04_08_1510}\mathscr L^{\crit}_\Iw(V_{f}(j))(0)=
\frac{(j-k)!}{ (j-1)!} \cdot \frac{e_{p,\alpha}(f,\mathds{1},j)}
{e_{p,\beta}(f,\mathds{1},j)}   \cdot  \mathscr L^{\crit}(V_{f}(j))\,.\ee

\item[iii)] If $1\leq j \leq k-1$ with $j\neq \frac{k}{2}$ (or if $j=\frac{k}{2}$ and $L(f^*,\frac{k}{2})\neq 0$), then we have
\be\label{eqn_2022_04_08_1511}
\mathscr L^{\crit}_\Iw(V_{f}(j))(0)=\frac{(-1)^{k-j}}{(j-1)! (k-j-1)!} \cdot 
 \frac{e_{p,\alpha}(f,\mathds{1},j)}{e_{p,\beta}(f,\mathds{1},j)}
\cdot 
\mathscr L^{\crit}(V_{f}(j))\,.\ee
\end{proposition}
\begin{proof} 
\item[i)] It follows from Proposition~\ref{prop:projection of beilinson-kato element} that $\res_p(\bz (f,\xi))\notin H^1_\Iw (\Qp, V_f^{(\beta)}(k))$ and therefore, $F_{k}\neq 0$.
Moreover, for any $m\in \ZZ$, we have, by definitions,
\be\label{eqn_twist_2022_04_08}
\res_p (\Tw_m(\bz)) = \Tw_m (F_{k}) \Exp_{\alpha,k+m} (\eta_f^{\alpha} [k+m] )+ \Tw_m(G_{k}) \Exp_{\beta,k+m} 
(\eta_f^\beta [k+m]).
\ee
Hence, $F_{k+m}=\Tw_m (F_{k})$ and therefore, $F_j\neq 0$ for any $j\in \ZZ$. Finally, it follows from \eqref{eqn_twist_2022_04_08} that
\begin{equation}
\label{eqn:twisting Iwasawa L-invariant}
\mathscr L_\Iw^{\rm cr}(V_{f}(k+m))=\Tw_m\circ  \mathscr L^{\crit}_\Iw(V_{f}(k))
\end{equation}
and the proof of Part (i) follows.

\item[ii)] Since the projection  map 
$$\pr_0\,:\, H^1_{\Iw}(V_f(j))_\Gamma \lra H^1_{\{p\}}(V_f(j))$$ 
is injective and  $H^1_{\Iw}(V_f(j))$ is $\Lambda_E$-free, it follows that the image $z:=\pr_0 (\bz)\in H^1_{\{p\}}(V_f(j))$ of the generator $\bz$ of $H^1_{\Iw}(V_f(j))$ is nonzero. The formulae \eqref{eqn:specialization of PR formulae} shows that for each $j \gg 0$ we have 
\begin{equation}
\label{eqn: decomposition of res (z)}
\res_p(z)= (j-k)! \left (\frac{1-p^{j-1}\alpha^{-1}}{1-p^{-j}\alpha}\right )F_j(0) \cdot  \exp \left (\eta_f^{\alpha} [j]\right  )+
(j-1)!  \left (\frac{1-p^{j-1}\beta^{-1}}{1-p^{-j}\beta}\right ) G_j(0) \cdot \exp \left (\eta_f^\beta [j] \right).
\end{equation}
For $j\gg 0$ the condition \eqref{item_pB} holds and for such $j$, it follows from the description of $\mathscr L^{\crit}(V_{f}(j))$ that $F_{j}(0)\neq 0$ and that \eqref{eqn_2022_04_08_1510} is valid.
 
\item[iii)] The proof of \eqref{eqn_2022_04_08_1511} is virtually identical to that of \eqref{eqn_2022_04_08_1510} and for that reason, it will be omitted.
\end{proof}

\begin{remark}
\label{remark_2023_01_05_1513}
Suppose that $f$ has CM. It follows from Proposition~\ref{prop_Rubin_H1alpha_vanishes} that we have the implication
$$\eqref{item_pBCM}\quad \implies{\quad \mathscr L_\Iw^{\mathrm{cr}}(V_f)\neq 0}\,.$$
As we have noted in Remark~\ref{remark_G_intro}(ii), this implication can be also deduced from \cite[Theorem 3.2]{LLZ_critical}.
\end{remark}

\subsubsection{Proof of Proposition~\ref{prop: comparision p-adic L-functions for alpha and beta}}
\label{subsubsec_proof_prop_2_16}
We are now ready to present a proof of Proposition~\ref{prop: comparision p-adic L-functions for alpha and beta}, where we compare the improved $p$-adic $L$-function $L_{\mathrm{K},\alpha}^{\mathrm{imp}}(f, \xi)$ and the slope-zero $p$-adic $L$-function $L_{\mathrm{K},\beta}(f,\xi)$ in terms of our Iwasawa theoretic $\mathscr L$-invariant $\mathscr L_{\Iw}^{\rm cr}$.
In order not to change our notation, we shall prove Proposition~\ref{prop: comparision p-adic L-functions for alpha and beta} for the dual eigenform $f^*$:
\[
L_{\mathrm{K},\alpha^*}^{\mathrm{imp}}(f^*,\xi^*)=(-1)^{k-1}\mathscr L_{\Iw}^{\rm cr}(V_f(k))
\cdot
L_{\mathrm{K},\beta^*}(f^*,\xi^*)\,.
\]
This statement is illustrative of the extreme exceptional zero phenomena that the $\theta$-critical $p$-adic $L$-functions exhibit.

\begin{proof}[Proof of Proposition~\ref{prop: comparision p-adic L-functions for alpha and beta}]
Let us write
\[
\res_p (\bz (f,\xi))=F_k^{\rm BK}\Exp_{V_f^{(\alpha)}(k),1}(\eta_f^{\alpha} [k])+ G_k^{\rm BK}\Exp_{V_f^{(\beta)}(k),k}(\eta_f^\beta [k])\,,
\]
where $F_k^{\rm BK}, G_k^{\rm BK} \in\LL [1/p].$
Then,
\be\label{eqn_2022_04_08_16_53}
L_{\mathrm{K},\alpha^*}^{\mathrm{imp}}(f^*,\xi^*)=G_k^{\rm BK}
\left <\Exp_{V_f^{(\beta)}(k),k}(\eta_f^\beta [k]),
 \Exp_{V_{f^*}^{(\alpha)},1-k} (\eta_{f^*}^{\alpha})^\iota\right >=(-1)^{k-1}G_k^{\rm BK}.
\ee
where the final equality follows from the explicit reciprocity law together with the fact that $V_f^{(\beta)}(k)$ is the Tate dual of $V_{f^*}^{(\alpha)}$. The same argument also yields
\be\label{eqn_2022_04_08_16_54}
L_{\mathrm{K},\beta^*}^{\pm}(f^*,\xi^*)=F_k^{\rm BK}
\left <\Exp_{V_f^{(\alpha)}(k),k}(\eta_f^{\alpha} [k]),
\Exp_{V_{f^*}^{(\beta)},0} (\eta_{f^*}^{\beta})^\iota\right >=F_k^{\rm BK}\,.
\ee
The proof of Proposition~\ref{prop: comparision p-adic L-functions for alpha and beta} follows on comparing \eqref{eqn_2022_04_08_16_53} and \eqref{eqn_2022_04_08_16_54}, and noting that 
$$G_k^{\rm BK}/F_k^{\rm BK}=G_k/F_k=\mathscr{L}^{\rm cr}_\Iw(V_f(k)).$$
\end{proof}

\subsection{Tate--Shafarevich groups, regulators and Tamagawa numbers}
\label{subsect: Tate-Shafarevich}

\subsubsection{} 
In this section, we use our $\mathscr L$-invariants to compare Tate--Shafarevich groups $\Sha_\star$ to one another, arising from local conditions associated to $\star\in\{0,\alpha,\beta\}.$ 
Fix a $\cO_E$-lattice $T_f$ of $V_f$ stable under the action of $G_\QQ.$

\begin{proposition}
\label{prop: Poitou-Tate for star groups}
For  $\star \in\{\alpha,\beta,{\rm f}\}$,  the Poitou--Tate exact sequence induces the  exact sequence
\[
0\lra  H^1_\star (T_f(j))\lra  H^1_{\{p\}} (T_f(j)) \lra 
\frac{H^1(\Qp, T_f(j))}{H^1_\star(\Qp, T_f(j))}
\lra \left (
\frac{H^1_\star (V_{f^*}/T_{f^*}(k-j))}{H^1_0 (V_{f^*}/T_{f^*}(k-j))}\right )^\vee
\lra 0\,,
\]
where we recall that $H^1_{\{p\}} (T_f(j))$ and  $H^1_0 (V_{f^*}/T_{f^*}(k-j))$ denote the relaxed and the strict Selmer groups, respectively. 
\end{proposition}

\begin{proof} This is a particular case of Proposition~\ref{prop: about general tate-shafarevich}, since for each $\star \in\{\alpha,\beta,{\rm f}\},$ the vector spaces $H^1_\star(\Qp, V_f(j))$ and $H^1_\star(\Qp, V_{f^*}(k-j))$ are orthogonal complements of one another under the local Tate duality pairing $H^1(\Qp, V_f(j)) \otimes H^1(\Qp, V_{f^*}(k-j)) \rightarrow E$.
\end{proof}

\subsubsection{}
\label{subsubsec_2022_04_27_1616}
In the remainder of \S\ref{subsect: Tate-Shafarevich}, we assume that $j\geqslant 1$
and $j\neq \frac{k}{2}$. The case $j\leqslant 0$ can be handled using duality and relying on the case $j\geq k$ (which we will investigate here). The central critical case $j=\frac{k}{2}$ will be studied in \S\ref{subsubsec_3246_2022_04_27_1609}.

Let us consider the following condition:
\begin{itemize}
\item[\mylabel{item_S}{\bf S})] $H^1_\alpha (V_f(j))=\{0\}$\,. 
\end{itemize}
Recall from \S\ref{subsec:critical L-invariants} that the critical $\mathcal L$-invariant $\mathscr{L}(V_f(j))$ is defined whenever the hypothesis \eqref{item_pB} is valid. The condition \eqref{item_pB} holds for all but finitely many $j$ (cf. \S\ref{subsubsec_3222_2022_04_28_1537}) and for all $j\neq\frac{k}{2}$ with $1\leq j\leq k-1$ (cf. \S\ref{subsubsec_critical_L_invariant_critical_range}). Whenever \eqref{item_pB} holds, recall from \eqref{eqn_2022_04_22_16_22} that the condition  \eqref{item_S} is equivalent to the non-vanishing of $\mathscr{L}^{\rm cr}(V_f(j))$. 

\begin{defn}
For any $\Zp$-module $M$, we denote by $M_{\tf}$ the quotient of $M$  by its maximal $p$-torsion submodule. 
\end{defn}

\begin{lemma}
\label{lemma:auxiliary Tate-Shafarevich}
Suppose $j\geqslant 1,$ $j\neq \frac{k}{2}$, and assume that   \eqref{item_S} holds true. If $j\geqslant k,$ we assume, in addition,
the validity of the hypothesis \eqref{item_pB}. Then the following assertions hold true:
\item[i)]
For $\star\in\{0, \alpha,\beta\}$, 
the group $H^1_\star (V_{f^*}/T_{f^*}(k-j))$
has finite cardinality, and $H^1_\star (T_f(j))_{\tf}=0.$
In particular, 
\[
\Sha_\star (T_{f^*}(k-j))=H^1_\star (V_{f^*}/T_{f^*}(k-j)),
\qquad  \star\in\{ 0,  \alpha,\beta\}.
\]

\item{ii)}   For $\star\in\{\alpha,\beta\},$ we have 
\[
\# \Sha_\star  (T_{f^*}(k-j))=\# \Sha_0  (T_{f^*}(k-j)) \cdot
\left [
H^1(\Qp, T_f(j))_\tf : H^1_{\{p\}}(T_f(j))_\tf +H^1_\star (\Qp,T_f(j))_\tf \right ]\,.
\]
\end{lemma}
\begin{proof} 
\item[i)] Fix $j\geqslant 1$ such that $j\neq \frac{k}{2}$. It follows from  \eqref{item_S} that $H^1_{\alpha} (T_f(j))_{\tf}=0$. If $1\leqslant j\leqslant k-1$, the vanishing of $H^1_{\beta} (T_f(j))_{\tf}$ follows from Kato's theorem recalled in Section~\ref{subsubsec_3221_2022_04_08}. If $j\geqslant k$, the same vanishing follows from \eqref{lemma_H1beta_zero_under_pB}.
Since $H^1_0(T_f(j))\subset H^1_{\alpha}(T_f(j),$ we deduce that $H^1_{\star}(T_f(j))_{\tf}=0$ for all $\star\in \{0,\alpha,\beta\}$.

We claim that  $H^1_0 (V_{f^*}(k-j))=\{0\}.$ For $1\leqslant j\leqslant k-1$ this follows again from Kato's theorem (applied to the form $f^*$), whereas for $j\geqslant k$, we use Lemma~\ref{lemma_padicBeilinson_implies_B_J}.
By the definition of the Tate--Shafarevich group $\Sha_0(T_{f^*}(k-j))$, we have an exact sequence
\[ 0\lra H^1_0 (V_{f^*}(k-j)) \lra H^1_0 (V_{f^*}/T_{f^*}(k-j)) \lra \Sha_0(T_{f_*}(k-j)) \lra 0\,.
\]
Since $\Sha_0(T_{f^*}(k-j))$ is finite by Proposition~\ref{prop: about general tate-shafarevich}, we conclude that $H^1_0 (V_{f^*}/T_{f^*}(k-j)) \simeq \Sha_0(T_{f^*}(k-j))$ is finite.

Let $\star \in \{\alpha,\beta\}.$ It follows from Lemma~\ref{lemma_padicBeilinson_implies_B_J} when $j\geq k$ and \S\ref{subsubsec_3221_2022_04_08} in the remaining cases that $H^1_{\{p\}} (T_f(j))$ is of rank one and the map
\be \label{eqn_2022_04_27_1606}
H^1_{\{p\}} (T_f(j))_{ \tf} \lra \frac{H^1(\Qp, T_f(j))_{ \tf}}{H^1_\star(\Qp, T_f(j))_{ \tf}}
\ee
is an injection (thanks to the vanishing $H^1_\star (T_f(j))_{\tf}=0$) with rank-one target. This shows that its cokernel has finite cardinality. Proposition~\ref{prop: Poitou-Tate for star groups} tells us that the cokernel of the map \eqref{eqn_2022_04_27_1606} is isomorphic to $ \left (
\dfrac{H^1_\star (V_{f^*}/T_{f^*}(k-j))}{H^1_0 (V_{f^*}/T_{f^*}(k-j))}\right )^\vee$, thence $$[{H^1_\star (V_{f^*}/T_{f^*}(k-j))}: {H^1_0 (V_{f^*}/T_{f^*}(k-j))}]<\infty\,.$$ 
Since we have already verified that $H^1_0 (V_{f^*}/T_{f^*}(k-j))$ has finite cardinality, this concludes the proof that $H^1_\star (V_{f^*}/T_{f^*}(k-j))$ has finite cardinality as well.
Hence, by the definition of the Tate-Shafarevich groups, we have $\Sha_\star (T_{f^*}(k-j))=H^1_\star (V_{f^*}/T_{f^*}(k-j))$
for $\star \in \{0,\alpha,\beta \}$.

\item[ii)]  This portion follows from Proposition~\ref{prop: Poitou-Tate for star groups}, on noting that the quotient $\dfrac{H^1(\Qp, T_f(j))}{H^1_\star(\Qp, T_f(j))}$ is torsion-free as per the definition of $H^1_\star(\Qp, T_f(j))$ (cf. Definition~\ref{defn_3_1_2022_29_04_11_09}).
\end{proof}

\subsubsection{} For $\star \in\{\alpha,\beta\},$ we denote by $N_\star [j] \subset \Dc (V_f^{(\star)}(j))$ the $\cO_E$-lattice generated by $\eta_f^{(\alpha)}[j]$. Recall that $\vert \,\cdot \,\vert$  denotes the absolute value on $E$ normalized so that $\vert p\vert =p^{-[E:\Qp]}$. Recall from \S\ref{subsubsec_1172_2023_07_07_0859} that we have put
\begin{equation}
\label{eqn: definition of Omega}
\Omega_p(T_f^{(\star)}, N_\star):=\vert a\vert, 
\end{equation}
where $a\in E^*$ is such that $\Dc(T_f^{(\star)})=aN_\star$. The quantities $\Omega_p(T_f^{(\star)}, N_\star)$ should be thought of as $p$-adic periods. We remark that  $\Omega_p(T_f^{(\star)}, N_\star)=\left [\Dc(T_f^{(\star)}):N_\star \right ].$
For any $\cO_E$-Galois module $T$ we set
\[
w(T):=\# H^0(\Qp, T\otimes_{\cO_E} E/\cO_E).
\]
The following proposition yields an interpretation of 
the critical $\mathscr L$-invariant in terms of Tate--Shafarevich groups.

\begin{proposition}
\label{prop: Tate-Shafarevich and L-invariant}
Suppose  $\frac{k}{2}\neq j\geqslant 1$, and $\star\in\{\alpha,\beta\}$. Assume the validity of the hypotheses \eqref{item_pB} and \eqref{item_S}. Then 

\[
\bigl \vert \mathscr L_\Iw^{\mathrm{cr}}(V_f(j))\bigr \vert \cdot \frac{\# \Sha_\alpha  (T_{f^*}(k-j))}{\# \Sha_\beta  (T_{f^*}(k-j))}
=
\frac{w(T_{f^*}^{(\alpha)}(k-j)) \cdot w(T_{f}^{(\alpha)}(j))}  
{w(T_{f^*}^{(\beta)}(k-j)) \cdot w(T_{f}^{(\beta)}(j))} 
\cdot \frac{\Omega_p(T_f^{(\beta)}, N_\beta)}{\Omega_p(T_f^{(\alpha)}, N_\alpha)}
\,.
\]
\end{proposition}
\begin{proof} It follows from Lemma~\ref{lemma:auxiliary Tate-Shafarevich} that
\[
\frac{\# \Sha_\alpha  (T_{f^*}(k-j))}{\# \Sha_\beta  (T_{f^*}(k-j))}
=
\frac {\left [H^1(\Qp,T_f^{(\beta)}(j))_\tf : H^1_{\{p\}}(T_f(j))_\tf \right ]}
{\left [H^1(\Qp,T_f^{(\alpha)}(j))_\tf : H^1_{\{p\}}(T_f(j))_\tf \right ]}.
\]
We have,
\[
\begin{aligned}
&\left [H^1(\Qp,T_f^{(\beta)}(j))_\tf : H^1_{\{p\}}(T_f(j))_\tf \right ]=
\left [N_\beta [j]: \lambda_\beta \left (H^1_{\{p\}}(T_f(j))_\tf\right) \right ]
\left [H^1(\Qp,T_f^{(\beta)}(j))_\tf :\exp \left (N_\beta [j]\right) \right ]
\\
&\left [H^1(\Qp,T_f^{(\alpha)}(j))_\tf : H^1_{\{p\}}(T_f(j))_\tf \right ]=
\left [N_\alpha [j]: \lambda_\alpha \left (H^1_{\{p\}}(T_f(j))_\tf\right) \right ]
\left [H^1(\Qp,T_f^{(\beta)}(j))_\tf :(\exp^*)^{-1}\left (N_\alpha [j]\right) \right ].
\end{aligned}
\]
Note that
\[
\frac{\left [N_\beta [j]: \lambda_\beta \left (H^1_{\{p\}}(T_f(j))_\tf\right) \right ]}{\left [N_\alpha [j]: \lambda_\alpha \left (H^1_{\{p\}}(T_f(j))_\tf\right) \right ]}=
\bigl \vert \mathscr L^{\mathrm{cr}}(V_f(j))\bigr \vert^{-1}.
\]
Using the formulae we have proved as part of Lemma~\ref{lemma: computation tamagawa}, it follows that
\[
\begin{aligned}
& \left [H^1(\Qp,T_f^{(\beta)}(j))_\tf :\exp \left (N_\beta [j]\right) \right ]= \vert (j-1)!\vert \cdot \left \vert \frac{1-\beta^{-1}p^{j-1}}{1-\beta p^{-j}}\right \vert \cdot \frac{w(T_{f^*}^{(\alpha)}(k-j))}{w(T_f^{(\beta)} (j))}
\cdot \Omega_p(T_f^{(\beta)},N_\beta),\\
&\left [H^1(\Qp,T_f^{(\alpha)}(j))_\tf :\left (\exp^*\right )^{-1} \left (N_\alpha [j]\right) \right ]= \vert (k-j-1)!\vert^{-1} \cdot \left \vert \frac{1-\alpha^{-1}p^{j-1}}{1-\alpha p^{-j}}\right \vert \cdot \frac{w(T_{f^*}^{(\beta)}(k-j))}{w(T_f^{(\alpha)} (j))} \cdot \Omega_p(T_f^{(\alpha)},N_\alpha).
\end{aligned}
\]
Combining this with Proposition~\ref{prop:about L-invariants}, we deduce that
\[
\bigl \vert \mathscr L_\Iw^{\mathrm{cr}}(V_f(j))\bigr \vert \cdot\frac{\# \Sha_\alpha  (T_{f^*}(k-j))}{\# \Sha_\beta  (T_{f^*}(k-j))}
=
\frac{w(T_{f}^{(\alpha)}(j)) w(T_{f^*}^{(\alpha)}(k-j))}{w(T_{f}^{(\beta)}(j))w(T_{f^*}^{(\beta)}(k-j))} \cdot \frac{\Omega_p(T_f^{(\beta)},N_\beta)}
{\Omega_p(T_f^{(\alpha)},N_\alpha)}\,,
\]
as required.
\end{proof}

\begin{defn} Assume that the conditions \eqref{item_pB} and \eqref{item_S} hold. 

\item[i)]{} Let $\star\in \{\alpha,\beta\}.$ For all $j\geqslant k,$ we define the $p$-adic regulator map as the natural isomorphism
\[
r_{V_f(j),\star} \,:\, H^1_{\{p\}}(V_f(j)) \lra \Dc (V_f(j))/\Dc (V_f^\star (j))
\]
induced from the composition 
$$H^1_{\{p\}}(V_f(j))\xrightarrow{\res_p} H^1(\Qp, V_f(j)){=}H^1_{\rm f}(\QQ_p,V_f(j))\xrightarrow{\log} \Dc (V_f(j))\,,$$ 
where $\log$ is the Bloch--Kato logarithm and the equality $H^1(\Qp, V_f(j)) = H^1_{\rm f}(\QQ_p,V_f(j))$ is because $j\geq k$.

\item[ii)]{} For $\star=\alpha$ (resp. $\star=\beta$), we put 
$$R_\star (T_f(j)):= \left \vert \det (r_{V_f(j),\star})\right \vert^{-1}\,,$$ 
where we compute determinant of $r_{V_f(j),\star}$ with respect to the lattice
$H^1_{\{p\}}(T_f(j))_\tf \subset H^1_{\{p\}}(V_f(j))$ and the lattice generated  by  $\eta_f^{\beta}[j]$ (resp. by $\eta_f^{\alpha}[j]$). 
\end{defn}

The following is an interpretation of the critical $\mathcal L$-invariant in terms of regulators when $j\geqslant k$.

\begin{proposition}
\label{prop: regulators and L-invariant}
 Assume that the conditions \eqref{item_pB} and \eqref{item_S} hold.
Then
\[
\mathscr L^\crit (V_f(j))= \frac{ \det (r_{V_f(j),\alpha})}{ \det (r_{V_f(j),\beta})}\,,
\]
and therefore, 
\[
\left \vert \mathscr L^\crit (V_f(j))\right \vert=\frac{R_\beta (T_f(j))}{R_\alpha (T_f(j))}\,.
\]
\end{proposition}

\subsubsection{}
In the remainder of this  paper, we will  consider the Tamagawa numbers of the $G_{\QQ_p}$-representations $V=V_f(j), V_f^{(\alpha)}(j), V_f^{(\beta)}(j)$ with respect to the 
the bases $\omega$ consisting of the vectors  $\eta_f^{\alpha}[j]$ and  $\eta_f^{\beta}[j]$. For this reason, we will omit $\omega$ from notation and write $\Tam^0_p(T_f(j)),$
$\Tam^0_p(T_f^{(\alpha)}(j)),$ $\Tam^0_p(T_f^{(\beta)}(j))$ for the corresponding Tamagawa numbers. We remark that
\[
\Tam^0_p(T_f^{(\star)}(j))=\Tam^0_{p,\mathrm{can}}(T_f^{(\star)}(j))
\cdot \Omega_p(T_f^{(\star)},N_\star)
\]
where the canonical Tamagawa numbers $\Tam^0_{p,\mathrm{can}}$ are as in \S\ref{subsubsec_1172_2023_07_07_0859}.

\begin{proposition}
\label{prop: comparision of sha in noncritical case} Under the assumptions \eqref{item_pB} and \eqref{item_S}, the following formulae hold true for all $j\geqslant k$:
\[
\begin{aligned}
&\#\Sha_\alpha (T_f(j))=
\frac{R_{\alpha} (T_f(j)) \cdot \Tam_{p} (T^{(\beta)}_f(j))}
{\# H^0(\Qp, V_f^{(\beta)}/T_f^{(\beta)}(j))} \cdot \# \Sha_{\rm f}(T_f(j)),\\
&\#\Sha_\beta (T_f(j))=
\frac{R_{\beta} (T_f(j)) \cdot \Tam_{p} (T^{(\alpha)}_f(j))}
{\# H^0(\Qp, V_f^{(\alpha)}/T_f^{(\alpha)}(j))} \cdot \# \Sha_{\rm f}(T_f(j)).
\end{aligned}
\] 
\end{proposition}
\begin{proof} We have, by the definitions of the relevant objects, that
\begin{equation}
\label{eqn: index formula noncritical case}
\left [
H^1(\Qp, T_f(j))_\tf : H^1_{\{p\}}(T_f(j))_\tf +H^1_\alpha (\Qp,T_f(j))_\tf \right ]=
\frac{R_{\alpha} (T_f(j)) \cdot \Tam_{p} (T^{(\beta)}_f(j))}
{\# H^0(\Qp, V_f^{(\beta)}/T_f^{(\beta)}(j))}.
\end{equation}
The first formula  follows from Lemma~\ref{lemma:auxiliary Tate-Shafarevich} and the proof of the second formula is entirely analogous. 
\end{proof}

\subsection{The central critical scenario}
\label{subsec: central critical scenario}

\subsubsection{The case $L(f^*,\frac{k}{2})\neq 0$}
 In this section, we assume that $k$ is even and study the Selmer groups $H^1_\star (V_f(k/2))$. Recall that  $H^1_\beta (V_f(k/2))=H^1_{\mathrm f} (V_f(k/2))$\,,
and that  
\[
\mathrm{ord}_{s=\frac{k}{2}}L(f^*,s) \geqslant 1 \quad \iff \quad z(f,\xi,k/2)\in H^1_{\mathrm{f}}(V_f(k/2)).
\]
Assume that  $L(f^*,\frac{k}{2})\neq 0.$ Then $z(f,\xi,k/2)\notin H^1_{\mathrm{f}}(V_f(k/2)),$
and  the situation is very similar with the non-central scenario
studied in Section~\ref{subsect: Tate-Shafarevich}. Namely, one has
\[
H^1_{\mathrm f} (V_f(k/2))=H^1_{\mathrm f} (V_{f^*}(k/2))=\{0\}.
\]
By Proposition~\ref{prop: Poitou-Tate for star groups},
we have an isomorphism
\[
H^1_{\{p\}}(V_f(k/2)) \simeq \frac{H^1(\Qp, V_f(k/2))}{H^1_{\mathrm f}(\Qp, V_f(k/2))}\,.
\]
Therefore $H^1_{\{p\}}(V_f(k/2))$ is the one-dimensional vector
space generated by $z (f,\xi, k/2).$ We expect that the condition (\refeq{item_S})
holds in this case, namely that $H^1_\alpha (V_f(k/2))=\{0\}$. Note that it is equivalent to the 
condition 
\[
z(f,\xi,k/2)\notin H^1_{\alpha}(V_f(k/2)).
\]
In particular, (\refeq{item_S}) implies that  $\mathscr L^{\crit}(V_f(k/2))\neq 0.$

\subsubsection{The case  $\mathrm{ord}_{s=\frac{k}{2}}L(f^*,s)=1$}
\label{subsubsec_crit_L_invariant_central_critical}
In the remainder of Section~\ref{subsec: central critical scenario} we assume that 
$L(f^*,\frac{k}{2})=0.$ First we turn our attention to the case $\mathrm{ord}_{s=\frac{k}{2}}L(f^*,s)=1.$
Throughout \S\ref{subsubsec_crit_L_invariant_central_critical}--\S\ref{subsect: poles of the l-invariant}, let us denote by $c_{\rm Iw}\in H^1_{\rm Iw}(V_f(\frac{k}{2}))$ either the Beilinson--Kato class if $f$ does not have CM, or the collection of (twisted) elliptic units\footnote{We note if $f$ has CM, the associated Galois representation in general fails the big image hypothesis (condition 12.5.2 in \cite{kato04}; see also Remark 12.8 in op. cit.) that is required to employ the Euler system machinery with the Beilinson--Kato classes. That is the reason why we need to  appeal to the elliptic unit Euler system in this case (as Kato also did in \cite{kato04}, \S15). See also~\cite[Theorem~2.28]{LLZ_critical} for a comparison of elliptic units and Beilinson--Kato classes.} along the cyclotomic tower if $f$ has CM. Recall that conjecturally, any eigenform that admits a $\theta$-critical $p$-stabilization has CM. Therefore, conjecturally, we are always in the latter scenario, but we will not rely on this expectation in what follows. Let us put ${\rm pr }_0(c_\Iw)=:c\in  H^1(\QQ,V_f(\frac{k}{2}))$.
The following assertion is often referred to as Perrin-Riou conjecture:
\begin{itemize}
\item[\mylabel{item_PR1}{\bf PR})]   $\res_p(c)\neq 0$\,.
\end{itemize}
Note that \eqref{item_PR1} holds in the scenario when $\ord_{s=\frac{k}{2}}L(f^*,s)=0$ thanks to Kato's work~\cite{kato04}.
\begin{remark}
\label{remark: about PR1 condition}
\item[i)] The  hypothesis \eqref{item_PR1} holds true when ${\rm ord}_{s=\frac{k}{2}}L(f^*,s)=1$ and $f=f_A$ is the newform attached to an elliptic curve $A/\QQ$ by Eichler and Shimura, thanks to the results \cite{BDV2022,BPSI} on Perrin-Riou's conjecture and the finiteness of $\Sha_{A/\QQ}[p^\infty]$ (thanks to Kolyvagin) in this scenario. 

\item[ii)] When $k>2$ (but still  ${\rm ord}_{s=\frac{k}{2}}L(f^*,s)=1$), the main results of \cite{BDV2022,BPSI} allow us to relate \eqref{item_PR1} to the non-vanishing of the image of a Heegner cycle (which is a non-trivial element of $H^1_{\rm f}(V_f(\frac{k}{2}))$, thanks to Gross--Zagier--Zhang) under the map $H^1_{\rm f}(V_f(\frac{k}{2}))\xrightarrow{\res_p}H^1_{\rm f}(\QQ_p,V_f(\frac{k}{2}))$. This problem can be recast in terms of the injectivity of a $p$-adic Abel--Jacobi map, which seems to be presently out of reach.

\item[iii)] Thanks to  Bertolini--Darmon--Prasanna~\cite{BDP_CM} and Kato~\cite[\S15]{kato04}, the condition \eqref{item_PR1} holds if $\ord_{s=\frac{k}{2}}L(f^*,s)\leq 1$ and
\begin{itemize}
\item[\mylabel{item_CM}{\bf CM})]  $f$ has complex multiplication (as it is conjectured by Greenberg to be always the case).
\end{itemize}
\item[iv)] When $k=2$, Zhang~\cite{ZhaoGreenbergConjwt2} shows (extending the earlier works of Serre~\cite{SerreGreenbergConjwt2}, Emerton~\cite{EmertonGreenbergConjwt2} and Ghate~\cite{GhateGreenbergConjwt2}) that $f$ has CM. In particular,  the condition \eqref{item_PR1} holds in this particular case whenever $\ord_{s=\frac{k}{2}}L(f^*,s)\leq 1$.
\end{remark}

\subsubsection{}
\label{subsect: poles of the l-invariant}
We denote by 
\begin{equation}
\label{eqn: p-adic heights <,>}
\langle\,,\,\rangle_\star :\quad  H^1_{\star}(V_{f^*}(k/2))\otimes 
H^1_{\star}(V_{f}(k/2)) \lra E,
\qquad \star \in \{\alpha,\beta\}
\end{equation}
the $p$-adic height pairing associated to the splitting submodules $\Dc (V_f^{(\star)}{(\frac{k}{2})})\subset \Dc (V_f(\frac{k}{2}))$ and 
$\Dc (V_{f^*}^{(\star)}{(\frac{k}{2})})\subset \Dc (V_{f^*}(\frac{k}{2}))$ 
 in the sense of \cite{benoisheights}.
Note that 
\[
\langle\,,\,\rangle_{\beta} :\quad   H^1_{\mathrm{f}}(V_{f^*}(k/2))\otimes H^1_{\mathrm{f}}(V_{f}(k/2)) \lra E
\]
coincides with the classical $p$-adic height pairing defined in \cite{nekovarheightpairings,
PerrinRiou92}.

We begin our analysis of the Iwasawa theoretic $\mathscr L$-invariant $\mathscr L^{\crit}_\Iw(V_{f}({k}/{2}))$ at the trivial character when ${\rm ord}_{s=\frac{k}{2}}L(f^*,s)=1$.

\begin{proposition}
 \label{prop_2022_04_26_13_00}
Suppose that ${\rm ord}_{s=\frac{k}{2}}L(f^*,s)=1$ and assume that the hypothesis \eqref{item_PR1} holds. Then:
 
\item[i)]   $\dim_E H^1_{\rm f}(V_{f}(\frac{k}{2})) =1$ and 
$H^1_{\{p\}}(V_{f}(\frac{k}{2}))=H^1_{\rm f}(V_{f}(\frac{k}{2})).$

\item[ii)] $H^1_0(V_{f}(\frac{k}{2}))=H^1_{\alpha}(V_{f}(\frac{k}{2}))=\{0\}.$ 
\item[iii)] $G_{\frac{k}{2}}(0)\neq 0$.  
\item[iv)]  Assume in addition that one of the conditions \eqref{item_SZ1} and \eqref{item_SZ2} holds true.
\begin{itemize}
\item[\mylabel{item_SZ1}{\bf SZ1})] $f$ has CM (conjecturally, this is always the case) by an imaginary quadratic field $K$, $\mu_{p^\infty}(K)=\{1\}$ and $p$ does not divide the order of the torsion subgroup of the ray group of $K$ modulo $\ff$, where $\ff$ is the conductor of the Hecke character attached to $f$.
\item[\mylabel{item_SZ2}{\bf SZ2})]
 $f$ does not have CM (conjecturally, this never happens), $k\equiv 2 \pmod{p-1}$, the residual representation $T_f/\m_ET_f$ is irreducible, and either there exists a prime $q|| N$ such that the inertia at $q$ acts non-trivially on $T_f/\m_ET_f$, or there exists a real quadratic field $F$ verifying the conditions of \cite[Theorem 4]{xinwanwanhilbert}.
\end{itemize}
Then $\mathscr L^{\crit}_\Iw(V_{f}({k}/{2}))$ has a simple pole at the trivial character if and only the $p$-adic height pairing 
$\langle\,,\,\rangle_\beta$ is nonzero.

\end{proposition}

We remark that the assumptions \eqref{item_SZ1} and \eqref{item_SZ1} are to ensure the validity of the \emph{slope-zero} main conjecture for $f$; cf. \cite{skinnerurbanmainconj, xinwanwanhilbert}, \cite[Theorem 2.5.2]{skinnerPasificJournal2016}, \cite[Theorem 15.2]{kato04}.

\begin{proof}
The assertion in (ii) that $H^1_0(V_{f}(\frac{k}{2}))=0$ follows from the Euler system machinery and our running assumption \eqref{item_PR1}. Moreover, since we have $H^1_0(V_{f}(\frac{k}{2}))=0$, it follows that the natural map 
\be \label{eqn_2022_04_26_1525}H^1_{\rm f}(V_f({k}/{2})) \xrightarrow{\res_p} H^1_{\rm f}(\QQ_p,V_f({k}/{2}))=H^1_{\beta}(\QQ_p,V_f({k}/{2}))
\ee
is an injection. Since $\dim_E H^1_{\beta}(\QQ_p,V_f({k}/{2}))=1$ and $c\in H^1_{\rm f}(V_f({k}/{2}))$ is a nonzero element, it follows that the map \eqref{eqn_2022_04_26_1525} is an isomorphism. This shows that $\dim_E H^1_{\rm f}(V_{f}(\frac{k}{2})) =1$. The fact that $H^1_0(V_{f}(\frac{k}{2}))=0$ and the Euler--Poincar\'e characteristic formula show that $\dim_E H^1_{\{p\}}(V_{f}(\frac{k}{2})) =1$ as well. Hence, the natural inclusion 
$$H^1_{\rm f}(V_{f}({k}/{2})) \hookrightarrow H^1_{\{p\}}(V_{f}({k}/{2})) $$
is an equality. Thence,
\begin{align}
\begin{aligned}
\label{eqn_long_chain_220427}
    H^1_{\alpha}(V_{f}({k}/{2}))&=\ker\left(H^1_{\{p\}}(V_{f}({k}/{2}))\lra \frac{H^1(\QQ_p,V_f({k}/{2}))}{H^1_{\alpha}(\QQ_p, V_{f}({k}/{2}))}\right)\\
    &=\ker\left(H^1_{\rm f}(V_{f}({k}/{2}))\lra \frac{H^1(\QQ_p,V_f({k}/{2}))}{H^1_{\alpha}(\QQ_p, V_{f}({k}/{2}))}\right)\\
    &=\ker\left(H^1_{\rm f}(V_{f}({k}/{2}))\lra \frac{H^1(\QQ_p,V_f({k}/{2}))}{H^1_{\rm f}(\QQ_p,V_{f}({k}/{2}))\cap H^1_{\alpha}(\QQ_p, V_{f}({k}/{2}))}\right)\\
    &=\ker\left(H^1_{\beta}(V_{f}({k}/{2}))\lra \frac{H^1(\QQ_p,V_f({k}/{2}))}{H^1_{\beta}(\QQ_p,V_{f}({k}/{2}))\cap H^1_{\alpha}(\QQ_p, V_{f}({k}/{2}))}\right)\\
    &=H^1_{0}(V_{f}({k}/{2}))=0\,,
\end{aligned}
\end{align}
and this completes the proof of Parts (i) and (ii).

We next prove (iii). Let ${\bf z}\in H^1_{\Iw}(V_{f}(\frac{k}{2})) \setminus (\gamma-1)\cdot H^1_{\Iw}(V_{f}(\frac{k}{2}))$ be any element and let $z:={\rm pr}_0({\bf z})\in H^1_{\{p\}}(V_{f}(\frac{k}{2}))$ denote its image. The formulae \eqref{eqn:specialization of PR formulae} show that 
\begin{align}
    \label{eqn: decomposition of res (z)_bis_26_04}
    \begin{aligned}
    \res_p(z)=(-1)^{\frac{k}{2}}(k/2-1)!\left (\frac{1-p^{\frac{k}{2}-1}\alpha^{-1}}{1-p^{-\frac{k}{2}}\alpha}\right )\,&\,F_{\frac{k}{2}}(0) \cdot  (\exp^*)^{-1} \left (\eta_f^{\alpha} [{k}/{2}]\right)\\
    &+({k}/{2}-1)!  \left (\frac{1-p^{\frac{k}{2}-1}\beta^{-1}}{1-p^{-\frac{k}{2}}\beta}\right ) G_{\frac{k}{2}}(0) \cdot \exp \left (\eta_f^\beta [{k}/{2}] \right).
    \end{aligned}
\end{align}
Using Equation~\eqref{eqn: decomposition of res (z)_bis_26_04} with $\mathbf{z}=c_\Iw$ and combining with the fact that $H^1_{\alpha}(V_{f}(\frac{k}{2}))=\{0\}$ from (ii) (which tells us that $\res_p(c)$ cannot fall inside $H^1_{\alpha}(\QQ_p,V_{f}(\frac{k}{2}))$), we infer that $G_{\frac{k}{2}}(0)\neq 0$, as required.

It remains to prove (iv). In view of Part (iii) and the definition of $\mathscr{L}^{\crit}_\Iw(V_{f}(\frac{k}{2}))$, we only need to check that $F_\frac{k}{2}$ has a simple zero at the trivial character. 

By \eqref{eqn_2022_04_08_16_54} and noting in our setting that $F_{\frac{k}{2}}=A\cdot F^{\rm BK}_{\frac{k}{2}}$ for some $A\in \LL$ coprime to $(\gamma-1)$ (where $F^{\rm BK}_{\frac{k}{2}}={\rm Tw}_{-\frac{k}{2}}F^{\rm BK}_k$ and $F^{\rm BK}_k$ is given as in \S\ref{subsubsec_proof_prop_2_16}), this is equivalent to checking that the slope-zero $p$-adic $L$-function $L_{\mathrm{K},\beta^*}(f^*,\xi^*)$ has a simple zero at $\chi^{\frac{k}{2}}$ under our running hypotheses.
Our argument below to check this is based on the theory we develop 
in this book, and we refer the reader to \cite{PerrinRiou92} for a more classical 
treatment. 

In \cite[Section~2.4.7]{perrinriou95}, it is explained that in the slope-zero case, the  Main Conjecture in the formulation of Perrin-Riou (cf. Conjecture~\ref{conj:ordinary MC} below) is equivalent to the classical Iwasawa Main Conjecture of Mazur and Greenberg. The latter conjecture is proved in \cite{rubinmainconj, kato04, skinnerurbanmainconj, skinnerPasificJournal2016, xinwanwanhilbert}. 
Therefore, the order of vanishing of $L_{\mathrm{K},\beta^*}(f^*,\xi^*)$ at $k/2$ is equal to the order $r(M)$ of the zero at $\gamma_1-1$ of the characteristic polynomial of the Iwasawa module $M:=\bR^2{\boldsymbol\Gamma}_{\Iw} (V_f(k/2),\beta)$ given as in Section~\ref{sec_modules_of_algebraic_padic_L_functions}. Recall that $\dim_EH^1_{\mathrm{f}}(V_f(k/2))=1$ under our assumptions. Combining 
Lemma~\ref{lemma: descent in the rank one case} and Theorem~\ref{thm: bockstein map in central critical case} below (which do not rely on the present proposition), we deduce that the height pairing is non-zero (in the setting of our proposition) if and only if $r(M)=1$, as required.

\end{proof}

\begin{corollary}
\label{cor_Iw_crit_L_inv_non_trivial_analytic_rank_one}
Suppose that ${\rm ord}_{s=\frac{k}{2}}L(f^*,s)=1$ and the condition \eqref{item_PR1} holds. Then the Iwasawa theoretic $\mathscr L$-invariant $\mathscr{L}^{\rm cr}_{\Iw}(V_f)$ is nonzero.
\end{corollary}

\begin{proof}
This is an immediate consequence of Proposition~\ref{prop_2022_04_26_13_00}(iii).
\end{proof}
 
\begin{corollary}
\label{cor_2022_07_01_1502}
Suppose that $f=f_A$ is the newform attached to an elliptic curve $A/\QQ$. If ${\rm ord}_{s=1}L(A,s)=1$, then the assertions (i)--(iii) of Proposition~\ref{prop_2022_04_26_13_00} hold unconditionally. Moreover, the Iwasawa theoretic $\mathscr L$-invariant $\mathscr{L}^{\rm cr}_{\Iw}(V_f(k/2))$ is nonzero and has a simple pole at the trivial character.
\end{corollary}  

\begin{proof} If $f=f_A$, then the condition \eqref{item_PR1} holds (see Remark~\refeq{remark: about PR1 condition}). Since $f$ admits a $\theta$-critical $p$-stabilization, then by a result of Serre~\cite{SerreGreenbergConjwt2}, the elliptic curve $A$ has complex multiplication. In this case, Bertrand in~\cite{Bertrand1984} proved that for any $P\in E(\QQ)$ of infinite order, its $p$-adic height $\langle P,P\rangle_\beta$ is nonzero. This shows that the hypothesis on the non-triviality of the $p$-adic height pairing $\langle \,,\,\rangle_\beta$ in Proposition~\ref{prop_2022_04_26_13_00}(iv) is readily verified in this scenario. We remark that the slope-zero main conjecture \ref{item_MCbeta} was proved by Rubin, cf. \cite[Theorem 12.4]{Rubin1992PRConj} when $p>2$.
\end{proof}




\subsubsection{The case ${\rm ord}_{s=\frac{k}{2}}L(f^*,s)>1$}
\label{subsubsec_3254_2022_08_12_1456}
We next study the Iwasawa theoretic $\mathscr L$-invariant $\mathscr L^{\crit}_\Iw(V_{f}({k}/{2}))$ at the trivial character when ${\rm ord}_{s=\frac{k}{2}}L(f^*,s)>1$. As we shall see in Proposition~\ref{prop_22_04_26_1150}, our conclusions in this scenario are conditional on various folklore (but very difficult) conjectures. This, of course, parallels the very gloomy state of affairs concerning the Bloch--Kato conjectures when the analytic rank is strictly greater than $1$. We will rely on the following notion of semi-simplicity (see also \S\ref{subsubsec_2023_08_28_1052} for an extensive discussion): 
\begin{defn}
\label{defn_semisimpleatI}
A $\LL(\Gamma_1)$-module $M$ is called semi-simple at $(\gamma_1-1)$ if its localization at $(\gamma_1-1)$ is semi-simple as a $\LL(\Gamma_1)_{(\gamma_1-1)}$-module.
\end{defn}

\begin{proposition}
 \label{prop_22_04_26_1150}
 Suppose that $r_{\mathrm{an}}:={\rm ord}_{s=\frac{k}{2}}L(f^*,s)>1$. Assume also that
 \begin{itemize}
\item[\mylabel{item_BK_cc}{\bf BK})] $\dim_E H^1_{\rm f}(V_f(\frac{k}{2}))=r_{\mathrm{an}}$ (Bloch--Kato conjecture)\,,
\item[\mylabel{item_Sha}{$\Sha$}\bf )] the map $H^1_{\rm f}(V_f(k/2))\xrightarrow{\res_p} H^1_{\rm f}(\QQ_p, V_f(k/2))$ is not the zero map\,, 
\item[\mylabel{item_slope_zero_ht}{\bf ht})] the $p$-adic height pairing 
$$\langle\,,\,\rangle_{\beta}:\quad  H^1_{\rm f}(V_{f^*}(k/2))\otimes 
H^1_{\rm f}(V_{f}(k/2)) \lra E $$
is non-degenerate.
\end{itemize}
Then:
\item[i)] We have 
$$ H^1_{\beta}(V_{f}({k}/{2})) = H^1_{\rm f}(V_f({k}/{2}))=H^1_{\{p\}}(V_f({k}/{2})), \quad \hbox{ and } \quad H^1_0(V_{f}({k}/{2}))=H^1_{\alpha}(V_{f}({k}/{2})).$$ 
Moreover, $H^1_0(V_{f}(\frac{k}{2}))=H^1_{\alpha}(V_{f}(\frac{k}{2}))$ have rank $r_{\mathrm{an}}-1$. 

Assume in addition that either \eqref{item_SZ1} or else \eqref{item_SZ2} holds true. Then:
\item[ii)]  $c_{\rm Iw}\in (\gamma_1-1)^{r_{\mathrm{an}}-1}\cdot H^1_{\rm Iw}(V_{f}(\frac{k}{2})) \setminus (\gamma_1-1)^{r_{\mathrm{an}}}\cdot H^1_{\rm Iw}(V_{f}(\frac{k}{2}))\,.$
\item[iii)] Let $d_{\Iw} \in H^1_{\rm Iw}(V_{f}(\frac{k}{2}))$ 
denote the unique element with $c_{\Iw}=(\gamma_1-1)^{r_{\mathrm{an}}-1}d_{\Iw}$ and let us put $\partial c:={\rm pr}_0(d_{\Iw})\in H^1_{\{p\}}(V_f(\frac{k}{2}))$ to denote the image of $d_\Iw$. Then $0\neq \partial c \in H^1_{\rm f}(V_f(k/2))$.
\end{proposition}

\begin{proof}
\item[i)] Since $\dim_E  H^1_{\rm f}(\QQ_p, V_f(k/2))=1$, our running assumption \eqref{item_Sha} shows that the map
$$H^1_{\rm f}(V_f(k/2))\xrightarrow{\res_p} H^1_{\rm f}(\QQ_p, V_f(k/2))$$
is surjective. Since $H^1_{0}(V_f(k/2))$ is the kernel of this map, it follows that 
\be\label{eqn_2022_04_27_1303}
\dim_E H^1_0(V_f(k/2))= \dim_E H^1_{\rm f}(V_f(k/2))-1\,.
\ee
We also have 
\be\label{eqn_2022_04_27_1304}
\dim_E H^1_0(V_f(k/2))= \dim_E H^1_{\{p\}}(V_f(k/2))-1\,.
\ee
by the global Euler--Poincar\'e characteristic formula. Combining \eqref{eqn_2022_04_27_1303} and \eqref{eqn_2022_04_27_1304}, we deduce that the natural containment $H^1_{\rm f}(V_f(\frac{k}{2}))\subseteq H^1_{\{p\}}(V_f(\frac{k}{2}))$ is indeed an equality. This proves the first asserted equality. We may now argue as in \eqref{eqn_long_chain_220427} and prove the second claimed equality. The final assertion follows from \eqref{eqn_2022_04_27_1303} and the second asserted equality, combined with our running hypothesis \eqref{item_BK_cc}.
\item[ii)] Let $a$ be the natural number such that 
$$c_{\rm Iw}\in (\gamma_1-1)^{a}\cdot H^1_{\rm Iw}(V_{f}({k}/{2})) \setminus (\gamma_1-1)^{a+1}\cdot H^1_{\rm Iw}(V_{f}({k}/{2}))\,.$$ 
Note that such $a$ uniquely exists since $c_{\rm Iw}$ is nonzero. It follows from the validity of the slope-zero main conjectures for $f$ (thanks to our assumption that either \eqref{item_SZ1} or \eqref{item_SZ2} holds; cf. \cite[Theorem 15.2]{kato04} if $f$ has CM, or \cite{skinnerurbanmainconj, xinwanwanhilbert}, \cite[Theorem 2.5.2]{skinnerPasificJournal2016} otherwise) that 
\be \label{eqn_22_04_27_1335}
{\rm char}\left(H^2_{\rm Iw}(V_{f}({k}/{2})) \right)\subset (\gamma_1-1)^a 
 \LL_E\setminus (\gamma_1-1)^{a+1}  \LL_E\,.
\ee
Moreover, since $H^2_{\rm Iw}(V_{f}(\frac{k}{2}))$ is a quotient of the Iwasawa theoretic Greenberg Selmer group, which is semi-simple at $(\gamma_1-1)$ thanks to our running assumption \eqref{item_slope_zero_ht} (cf. \cite{PerrinRiou92}; see also Theorem~\ref{thm: bockstein map in central critical case}(iii) below), it follows that $H^2_{\rm Iw}(V_{f}({k}/{2}))$ is semi-simple at $(\gamma_1-1)$ as well. This fact combined with \eqref{eqn_22_04_27_1335} shows that 
$$H^2(V_{f^*}({k}/{2}))\simeq H^2_{\rm Iw}(V_{f^*}({k}/{2}))_\Gamma\simeq E^{a}\,.$$
On the other hand, we have
$$a=\dim_E H^2(V_{f^*}({k}/{2}))=\dim_E H^1_0(V_{f}({k}/{2}))=r_{\mathrm{an}}-1$$
by global duality and Part (i). This shows that $a=r_{\rm an}-1$, as required.

\item[iii)] It follows from our running assumptions \eqref{item_BK_cc} and  \eqref{item_slope_zero_ht} combined with the truth of Iwasawa main conjectures that
\be\label{eqn_2022_04_27_1547}
e_{\mathds{1}} {\rm Tw}_{\frac{k}{2}} L_{\mathrm{S},\beta^*}(f^*,\xi^*)\in 
(\gamma_1-1)^{r_{\mathrm{an}}} \LL_E\setminus (\gamma_1-1)^{r_{\mathrm{an}}+1}
 \LL_E\,,\ee
where $e_{\mathds{1}}=\frac{1}{|\Delta|}\sum_{\delta\in \Delta}\delta$ is the idempotent corresponding to the trivial character of $\Delta$. It follows from 
Proposition~\ref{prop_2_17_2022_05_11_0842} and the definition of $\partial c$ that 
$$0\equiv e_{\mathds{1}} {\rm Tw}_{\frac{k}{2}}
{L_{\mathrm{S},\beta^*}^\pm(f^*,\xi^*)} \equiv C\cdot
{\left (  \exp^*_{V_f(\frac{k}{2})}(\res_p(\partial c)), \eta_{f^*}^{\beta}[k/2]
\right )_{V_f(\frac{k}{2})}}(\gamma_1-1)^{r_{\mathrm{an}}-1}\mod (\gamma_1-1)^{r_{\mathrm{an}}}$$
for some explicit $C\in E^\times$ whose exact form we do not need here. This shows that $\partial c\in H^1_{\rm f}(V_f(\frac{k}{2}))$, as required.

Moreover, since $c_\Iw\notin (\gamma_1-1)^{r_{\mathrm{an}}} H^1_\Iw(V_f(\frac{k}{2}))$, it also follows that $\partial c \neq 0$. 
\end{proof}

\begin{remark}
We comment on the hypotheses in Proposition~\ref{prop_22_04_26_1150}.
 When $f=f_E$ is the newform attached to an elliptic curve $E/\QQ$, then the hypothesis \eqref{item_Sha} follows from the finiteness of $\Sha_{E/\QQ}[p^\infty]$ (which is unknown in this level of generality) combined with \eqref{item_BK_cc}.
\end{remark}

\subsubsection{} We  consider the following higher-rank version of the 
Perrin-Riou conjecture: 
\begin{itemize}
\item[\mylabel{item_PR}{\bf h-PR})] $\res_p(\partial c)\neq 0$ (higher-rank Perrin-Riou conjecture)
\end{itemize}
Note that \eqref{item_PR} implies \eqref{item_Sha}.

\begin{proposition} 
\label{prop: consequences of higher PR}
Assume that the conditions  \eqref{item_BK_cc},
\eqref{item_slope_zero_ht} and \eqref{item_PR} hold. 
\item[i)] We have
\be\label{eqn_2022_04_27_1443}
H^1_{\rm f}(V_f({k}/{2}))=H^1_0(V_f({k}/{2}))\oplus E\cdot \partial c\,.
\ee
In particular,
\[
H^1_{\rm f}(V_f({k}/{2}))=H^1_0(V_f({k}/{2}))\oplus H^1_\Iw(V_f(k/2))_\Gamma .
\]
\item[ii)] $G_{\frac{k}{2}}(0)\neq 0$.
\item[iii)] $\mathscr L^{\crit}_\Iw(V_{f}({k}/{2}))$ has a simple pole at the trivial character.

\end{proposition}
\begin{proof}
\item[i)] The proof of \eqref{eqn_2022_04_27_1443} is immediate whenever \eqref{item_PR} holds, since $\partial c\in H^1_{\rm f}(V_f(\frac{k}{2})) \setminus H^1_0(V_f(\frac{k}{2}))$. As the first Iwasawa cohomology $H^1_\Iw (V_f(\frac{k}{2}))$ is a free $\LL_E$-module of rank one, the $E$-vector space  of coinvariants $H^1_\Iw(V_f(\frac{k}{2}))_\Gamma$ is one-dimensional, and therefore it is generated by $\partial c$ in that case.

\item[ii)] Let $\mathbf z \in H^1_\Iw(V_f(\frac{k}{2}))$ be a generator and let us put $z:=\pr_0({\bf z})\in H^1_{\{p\}}(V_f(\frac{k}{2}))$ as before. Suppose on the contrary that $G_{\frac{k}{2}}(0)=0$. Then $\res_p(z)\in H^1_\alpha(\QQ_p,V_f(\frac{k}{2}))$, or equivalently, $z\in H^1_\alpha(V_f(\frac{k}{2}))=H^1_0(V_f(\frac{k}{2}))$ where the second equality has been verified in Proposition~\ref{prop_22_04_26_1150}. This shows that 
\be\label{eqn_2022_04_27_1508}
\res_p(z)=0\,.
\ee
Let $d_\Iw$ be as in the statement of Proposition~\ref{prop_22_04_26_1150}. Since, by definitions, we have $d_\Iw=A\cdot \mathbf z$ for some $A\in \LL$ such that $e_{\mathds 1}A\in \LL(\Gamma_1)$ is coprime to $(\gamma_1-1)$, it follows from \eqref{eqn_2022_04_27_1508} that $\res_p(\partial z)=0$, contrary to our running assumption \eqref{item_PR}. This shows that $G_{\frac{k}{2}}(0)\neq 0$, as required.

\item[iii)] In view of Part (ii) and the definition of $\mathscr{L}^{\crit}_\Iw(V_{f}(\frac{k}{2}))$, we only need to check that $F_\frac{k}{2}$ has a simple zero at the trivial character. Let $\mathbf z \in H^1_\Iw(V_f(\frac{k}{2}))$ be a generator as above. Since $c_\Iw=A\cdot(\gamma_1-1)^{r_{\mathrm{an}}-1}\cdot \mathbf z $ for some $A\in \LL$ with $e_{\mathds 1}A$ coprime to $(\gamma_1-1)$, it follows from \eqref{eqn_2022_04_08_16_54} that this requirement is equivalent to \eqref{eqn_2022_04_27_1547}, which we have verified under our running assumptions.
\end{proof}

\subsubsection{Tate--Shafarevich groups}
\label{subsubsec_3246_2022_04_27_1609}
We next study the Tate--Shafarevich groups at the central critical point. The exact sequence in the statement of Proposition~\ref{prop: Poitou-Tate for star groups} reads
\be\label{eqn_2022_04_29_1012}
0\lra  H^1_\star (T_f({k}/{2}))\lra  H^1_{\{p\}} (T_f({k}/{2})) \lra \frac{H^1(\Qp, T_f({k}/{2}))}{H^1_\star(\Qp, T_f({k}/{2}))}\lra \left (\frac{H^1_\star (V_{f^*}/T_{f^*}({k}/{2}))}{H^1_0 (V_{f^*}/T_{f^*}({k}/{2}))}\right )^\vee \lra 0\,.
\ee
We first consider the case

{\bf a)} $H^1_\star (V_f(\frac{k}{2}))\neq H^1_{\{p\}} (V_f(\frac{k}{2}))$. In this scenario, by dimension considerations, the map
$$\displaystyle  H^1_{\{p\}} (V_f({k}/{2})) \lra \frac{H^1(\Qp, V_{f}({k}/{2}))}{H^1_\star(\Qp, V_f({k}/{2}))}$$
is surjective and it follows from the exact sequence \eqref{eqn_2022_04_29_1012} that 
\be\label{eqn_2022_04_29_1015}
{\rm rank}\,H^1_\star (V_{f^*}/T_{f^*}({k}/{2}))^\vee = {\rm rank}\, H^1_0 (V_{f^*}/T_{f^*}({k}/{2}))^\vee\,.
\ee
Moreover, it follows using \cite[Lemma 4.1.1]{mr02} and \eqref{eqn_2022_04_29_1015} that
\[
H^1_0(V_{f^*}({k}/{2}))=H^1_\star ( V_{f^*}({k}/{2}))\,.
\]
Therefore, 
\[
\Sha_\star (T_f^*({k}/{2}))=
\mathrm{coker} \left (
H^1_0 ( V_{f^*}({k}/{2}))
\lra H^1_\star (V_{f^*}/T_{f^*}({k}/{2}))
\right )
.
\]

This fact plugged in the exact sequence \eqref{eqn_2022_04_29_1012} yields (once again, relying on the fact that the quotient $\dfrac{H^1(\Qp, T_f(j))}{H^1_\star(\Qp, T_f(j))}$ is torsion-free):
\begin{equation}
\label{eqn: formula for shafarevich groups in the central case}
\#\Sha_\star (T_{f^*}({k}/{2})) =\#\Sha_0 (T_{f^*}({k}/{2}))\cdot \left[H^1(\Qp, T_f({k}/{2}))_\tf :
H^1_\star(\Qp, T_f({k}/{2}))_\tf  +\frac{H^1_{\{p\}} (T_f({k}/{2}))_\tf}{H^1_\star (T_f({k}/{2}))_\tf} \right ]^{-1}.
\end{equation}
The identity \eqref{eqn: formula for shafarevich groups in the central case} can be thought of as an  analogue of that in Lemma~\ref{lemma:auxiliary Tate-Shafarevich}(ii). 

We now consider the case

{\bf b)} $H^1_\star (V_f(\frac{k}{2}))=H^1_{\{p\}} (V_f(\frac{k}{2}))$. In this scenario, the exact sequence \eqref{eqn_2022_04_29_1012} reduces to an isomorphism
\be\label{eqn_2022_04_29_1154}
 \frac{H^1(\Qp, T_f({k}/{2}))}{H^1_\star(\Qp, T_f({k}/{2}))}\xrightarrow{\sim} \left (\frac{H^1_\star (V_{f^*}/T_{f^*}({k}/{2}))}{H^1_0 (V_{f^*}/T_{f^*}({k}/{2}))}\right )^\vee
\ee
of free $\cO_E$-modules of rank one. This in particular shows that 
\be\label{eqn_2022_04_29_1211}
{\rm rank}\,H^1_\star (V_{f^*}/T_{f^*}({k}/{2}))^\vee = {\rm rank}\, H^1_0 (V_{f^*}/T_{f^*}({k}/{2}))^\vee+1\,.
\ee
which in turn shows, using \cite[Lemma 4.1.1]{mr02}, that
\be\label{eqn_2022_04_29_1213}
\dim\,H^1_\star (V_{f^*}({k}/{2})) = \dim\, H^1_0 (V_{f^*}({k}/{2}))+1\,.
\ee
On passing to Pontryagin duals in \eqref{eqn_2022_04_29_1154} and using local Tate duality, we obtain the exact sequence
\be\label{eqn_2022_04_29_1158}
 0\lra H^1_0 (V_{f^*}/T_{f^*}({k}/{2})) \lra H^1_\star (V_{f^*}/T_{f^*}({k}/{2})) \lra H^1_\star(\Qp, V_{f^*}/T_{f^*}({k}/{2}) )\lra 0\,
\ee
and in turn the commutative diagram
\be\label{eqn_2022_04_29_1205}
\begin{aligned}
\xymatrix{
 0\ar[r]& H^1_0 (V_{f^*}({k}/{2}))\ar[r]\ar[d]& H^1_\star (V_{f^*}({k}/{2})) \ar[r]\ar[d]& H^1_\star(\Qp, V_{f^*}({k}/{2}) )\ar[r]\ar@{->>}[d]& 0\\
 0\ar[r]& H^1_0 (V_{f^*}/T_{f^*}({k}/{2}))\ar[r]& H^1_\star (V_{f^*}/T_{f^*}({k}/{2})) \ar[r]& H^1_\star(\Qp, V_{f^*}/T_{f^*}({k}/{2}) )\ar[r]& 0
}
\end{aligned}
\ee
where the surjection on the first row at the very right is a consequence of \eqref{eqn_2022_04_29_1213} together with the fact that $\dim H^1_\star(\Qp, V_{f^*}({k}/{2}) ) =1$, and the vertical surjection on the very right follows from the definitions.

Using the snake lemma with the diagram \eqref{eqn_2022_04_29_1205}, we deduce that
\be\label{eqn_2022_04_29_1219}
\#\Sha_\star(T_{f^*}(k/2))\,=\,\#\Sha_0(T_{f^*}(k/2))\cdot\left[H^1_\star (\QQ_p, T_{f^*}({k}/{2}))_{\rm tf}: \frac{ H^1_\star (T_{f^*}({k}/{2}))_{\rm tf}}{H^1_0 (T_{f^*}({k}/{2}))_{\rm tf}}\right]^{-1}\,.
\ee

\begin{remark}
In view of our analysis in Propositions~\ref{prop_2022_04_26_13_00}
and \ref{prop_22_04_26_1150}, we expect that
\begin{itemize}
    \item $H^1_\beta (V_f(\frac{k}{2}))= H^1_{\{p\}} (V_f(\frac{k}{2}))$ if and only if $L(f^*,\frac{k}{2})= 0$\,, 
\item $H^1_\alpha (V_f(\frac{k}{2}))\neq H^1_{\{p\}} (V_f(\frac{k}{2}))$ in all cases.
\end{itemize}
\end{remark}


\section{Modules of algebraic $p$-adic $L$-functions}
\label{sec_modules_of_algebraic_padic_L_functions}

\subsection{Selmer complexes for modular forms} 
In this subsection, we define Perrin-Riou's modules of algebraic $p$-adic  $L$-functions associated to the Hecke eigenvalues $\alpha$ and $\beta$, using the formalism of Selmer complexes as in \cite{benoisextracris}.  The reader will notice that the general constructions could be greatly simplified and made more explicit for the slope-$0$ eigenvalue $\beta$. However, for our further purposes, it is important to present a coherent picture allowing the unified treatment of both cases, working with abstract objects with finer functorial properties.

\subsubsection{}\label{subsubsec_3311_2022_05_06_1130} Let us fix an $\cO_E$-lattice $T_f$ of $V_f$ and for $\star \in \{\alpha,\beta\}$, let us set $D^{(\star)}[j]:=
\Dc (V_f^{(\star)}(j))$ and  denote by $N_\star [j]\subset D^{(\star)}[j]$ the $\cO_E$-lattice generated by $\eta_f^{\star}[j].$  We consider the following local conditions $\left (\RG_{\Iw}(\QQ_\ell, V_f(j),\star)\right )_{\ell\in S}$ on the level of complexes:
\begin{itemize}
\item{}  For all $\ell\in S\setminus \{p\}$, we shall take  the unramified local condition (cf. \eqref{eq: Iwasawa unramified local condition}): 
$$
\RG_{\Iw}(\QQ_\ell, V_f(j),\star) :=
\left [ (V_f(j)\otimes_{\Zp}\LL)^{I_\ell} \xrightarrow{\Fr_\ell-1} (V_f(j)\otimes_{\Zp}\LL)^{I_\ell}\right ]\,.
$$
Recall that 
\[
\bR^i\mathbf{\Gamma}_{\Iw} (\QQ_\ell, V_f(j),\star)=\begin{cases} 
H^1_\Iw (\QQ_\ell, V_f(j)) &\text{if $i\neq 1,$}\\
0 &\text{if $i\neq 1.$}
\end{cases}
\]
\item{} The local condition at $p$ is given by the complex
\[
\RG_{\Iw}(\Qp, V_f(j),\star)  :=D^{(\star)}[j] \otimes_{\cO_E} \LL [-1]
\]
concentrated in degree $1$, together with the morphism
 \eqref{eqn: integral exponentials}
\[
\bExp_{\star,j}\,:\,D^{(\star)} [j]\otimes \LL  \lra \RG_{\Iw}(\Qp, V_f(j))\,. 
\]
Here, we have used the quasi-isomorphism 
$$\RG_{\Iw}(\Qp, V_f(j)) \simeq \left [\bD^\dagger (V_f(j)) \xrightarrow{\psi-1}
\bD^\dagger (V_f(j))\right ]$$ 
to define the exponential map on the level of complexes.
\end{itemize}

We denote by  $\RG_{\Iw}^{\mathrm{imp}}(V_f(j),\alpha)$ 
and $\RG_{\Iw}(V_f(j),\beta)$  the Selmer complexes 
associated to the diagrams
\begin{equation}
\label{eqn: punctured Selmer complexes}
\begin{aligned}
&\xymatrix{
\RG_{{\Iw}}(V_f(j))
\ar[r] & \underset{\ell\in S}\bigoplus \RG_{\Iw} (\QQ_\ell,V_f(j)))\\
& \underset{\ell \in S}\bigoplus \RG_{\Iw}(\QQ_\ell, V_f(j),\alpha)
\,
 \ar[u]
}
&&
&&&\xymatrix{
\RG_{{\Iw}}(V_f(j))
\ar[r] & \underset{\ell\in S}\bigoplus \RG_{\Iw} (\QQ_\ell,V_f(j)))\\
& \underset{\ell \in S}\bigoplus \RG_{\Iw}(\QQ_\ell, V_f(j),\beta)
\,\,.
 \ar[u]
}
\end{aligned}
\end{equation}
\subsubsection{}
\label{subsect:non-improved Selmer complex}
 We remark that, from the point of view of  Perrin-Riou's theory \cite{perrinriou95}, the complex 
\linebreak 
$\RG_{\Iw}(V_f(j),\beta)$ is the algebaric counterpart of the $p$-adic $L$-function $L_{\mathrm{S},\beta^*}(f^*,\xi^*).$ In Section~\ref{sect: main conjecture for critical forms}, we will see that $\RG_{\Iw}^{\mathrm{imp}}(V_f(j),\alpha)$ 
is naturally an algebraic counterpart of the cyclotomic improvement $L_{\mathrm{S},\alpha^*}^{[0],\mathrm{imp}}(f^*,\xi^*)$ of the $p$-adic $L$-function $L_{\mathrm{S},\alpha^*}^{[0]}(f^*,\xi^*)$. One may also introduce ``a non-improved version'' of this complex on replacing $\bExp_{\alpha,j}$ by the map
\[
\bExp_{V_f^{(\alpha)}(j),j} \,:\, D^{(\alpha)}{[j]}\otimes \LL \lra  \RG_{\Iw}(\Qp, V_f(j))\otimes_{\LL_E}\CH(\Gamma)
\]
and considering the Selmer complex $\RG_{\Iw}(V_f(j),\alpha)$ of $\CH (\Gamma)$-modules
associated to the diagram
\begin{equation}
\nonumber
\xymatrix{
\RG_{{\Iw}}(V_f(j))\otimes_{\LL_E}\CH(\Gamma)
\ar[r] & \underset{\ell\in S}\bigoplus \bigl (\RG_{\Iw} (\QQ_\ell,V_f(j)))\otimes_{\LL_E}\CH(\Gamma)\bigr )\\
& \underset{\ell \in S}\bigoplus \bigl (\RG_{\Iw}(\QQ_\ell, V_f(j),\alpha)\otimes_{\LL_E}\CH(\Gamma)\bigr )
\,\,,
 \ar[u]
}
\end{equation}
in which the component at $p$ of the vertical map is  $\bExp_{V_f^{(\alpha)}(j),j}.$
Recall that the maps  $\bExp_{V_f^{(\alpha)}(j),j}$ and $\bExp_{\alpha, j}$
are related by the formula
\begin{equation}
\label{eqn: relation between exponentials}
\bExp_{V_f^{(\alpha)}(j),j}= \ell (V_f(j))\cdot \bExp_{\alpha, j}\,,
\qquad\qquad \ell (V_f(j))=\underset{i=1-k+j}{\overset{j-1}\prod} \ell_i.
\end{equation}

\subsubsection{} 
\label{subsect: compatible T_f and N}
For $\star\in \{\alpha,\beta\}$, set $\Omega_p (T_f^{(\star)}, N_\star)=\vert b\vert,$ where $b\in E^*$  is such that $\Dc (T_f^{(\star)})={b}N_\star$ (cf. Section~\ref{subsec_Tamagawa_numbers}). If we choose 
$T_f$ such that $\Omega_p (T_f^{(\star)}, N_\star)\in \cO_E^*$, then thanks to \eqref{eqn: integral exponentials}, we have an integral version of the diagrams
\eqref{eqn: punctured Selmer complexes}:
\begin{equation}
\nonumber
\begin{aligned}
&\xymatrix{
\RG_{{\Iw}}(T_f(j))
\ar[r] & \underset{\ell\in S}\bigoplus \RG_{\Iw} (\QQ_\ell,T_f(j)))\\
& \underset{\ell \in S}\bigoplus \RG_{\Iw}(\QQ_\ell, T_f(j),\alpha)
\,
 \ar[u]
}
&&
&&&\xymatrix{
\RG_{{\Iw}}(T_f(j))
\ar[r] & \underset{\ell\in S}\bigoplus \RG_{\Iw} (\QQ_\ell,T_f(j)))\\
& \underset{\ell \in S}\bigoplus \RG_{\Iw}(\QQ_\ell, T_f(j),\beta)
\,\,.
 \ar[u]
}
\end{aligned}
\end{equation}
Whenever $\Omega_p (T_f^{(\star)}, N_\star)\in \cO_E^*$, we let $\RG_{\Iw}^{\mathrm{imp}}(T_f(j),\alpha)$ and $\RG_{\Iw}(T_f(j),\beta)$ denote the Selmer complexes associated to these  diagrams.

\subsubsection{}
The basic properties of these complexes are summarized in the following:

\begin{proposition} 
\label{prop: selmer complex beta}
\item[i)]  The $\Lambda$-module
\begin{equation}
\nonumber
\Sha^2_\Iw (T_f(j)):=\ker \left (H^2_{\Iw,S}(T_f(j)) \lra \underset{\ell\in S}\oplus H^2_{\Iw}(\QQ_\ell, T_f(j))\right )
\end{equation}
is finitely generated and $\Lambda$-torsion. 

\item[]{} Assume that the condition \eqref{item_S} holds true for some integer $j=j_0$, and let $\star \in \{\alpha,\beta\}$. Then the following statements are valid: 

\item[ii)] For all $j\in\ZZ,$ the map 
\[
f_\star\,:\,H^1_{\Iw}(T_f(j)) \oplus H^1_{\Iw}(\Qp, T_f^{(\star)}(j)) \xrightarrow{\res_p\, +\, \iota_\star}  H^1_{\Iw}(\Qp, T_f(j)) 
\]
is injective. Here, $\iota_\star$ is the morphism induced from the natural injection $T_f^{(\star)}\to T_f$.

\item[iii)] $\mathbf R^i\mathbf \Gamma^{\mathrm{imp}}_{{\Iw}}(T_f(j),\alpha)=\{0\}=\mathbf R^i\mathbf \Gamma_{{\Iw}}(T_f(j),\beta)$ if $i\neq 2,3.$

\item[iv)]  The $\LL$-modules $\mathbf R^2\mathbf \Gamma^{\mathrm{imp}}_{{\Iw}}(T_f(j),\alpha)$ and  $\mathbf R^2\mathbf \Gamma_{{\Iw}}(T_f(j),\beta)$
sit in the exact sequences
\begin{equation}
\nonumber
\begin{aligned}
&0\lra \frac{ H^1_{\Iw}(\Qp, T_f(j))}{H^1_{\Iw}(T_f(j))\oplus \Exp_{\alpha,j}(N_\alpha[j]\otimes \Lambda)} \lra 
\mathbf R^2\mathbf \Gamma^{\mathrm{imp}}_{{\Iw}}(T_f(j),\alpha) \lra \Sha^2_\Iw (T_f(j)) \lra 0\,,
\\
&0\lra \frac{ H^1_{\Iw}(\Qp, T_f(j))}{H^1_{\Iw}(T_f(j))\oplus \Exp_{\beta,j}(N_\beta[j]\otimes \Lambda)} \lra 
\mathbf R^2\mathbf \Gamma_{{\Iw}}(T_f(j),\beta) \lra \Sha^2_\Iw (T_f(j)) \lra 0\,.
\end{aligned}
\end{equation} 

\item[v)] $\mathbf R^3\mathbf \Gamma^{\mathrm{imp}}_{{\Iw}}(T_f(j),\alpha)\simeq
H^0(\QQ (\zeta_{p^\infty}), V_{f^*}/T_{f^*}(k-j))^{\wedge}\simeq 
\mathbf R^3\mathbf \Gamma_{{\Iw}}(T_f(j),\beta)\,.$

\item[vi)] The analogous properties hold for the complexes  $\RG_{\Iw}^{\mathrm{imp}}(V_f(j),\alpha)$  and $\RG_{\Iw}(V_f(j),\beta).$ In addition,
\[
\mathbf R^3\mathbf \Gamma^{\mathrm{imp}}_{{\Iw}}(V_f(j),\alpha)=\{0\}=\mathbf R^3\mathbf \Gamma_{{\Iw}}(V_f(j),\beta)\,.
\]
\end{proposition}
\begin{proof} 
\item[i)] This statement is proved in \cite[Theorem~12.4]{kato04}.

\item[ii)] For each $\star\in \{\alpha,\beta\}$, we have 
\[
\res_p\left(H^1_{\Iw}(V_f(j))\right) \cap H^1_{\Iw}(\Qp, V_f^{(\star)}(j))=\{0\}
\]
inside $H^1_\Iw (\Qp, V_f(j))$. Indeed, this claim  holds unconditionally  if $\star=\beta$, and when $\star=\alpha$, it follows from  condition \eqref{item_S} (cf. Section~\ref{subsubsec_3233_2022_05_06_1354}).
By \cite[Theorem~12.4]{kato04}, $H^1_{\Iw}(T_f(j))$ is a torsion-free $\LL$-module. Since $H^1(\Qp, T_f^{(\star)}(j))$ is a free $\LL$-module of rank one, we deduce that  
\[
\res_p\left(H^1_{\Iw}(T_f(j))\right) \cap H^1_{\Iw}(\Qp, T_f^{(\star)}(j))=\{0\}.
\]
This is a restatement of (ii). 

\item[iii-vi)] These portions follow easily from earlier remarks; cf.  \cite[Theorem~4]{benoisextracris} for similar arguments. For the reader's convenience, we reproduce these arguments of op. cit. below.

It follows from the definition of a Selmer complex that we have an exact triangle
\[
\underset{\ell\in S}\bigoplus
\RG_\Iw (\QQ_\ell, T_f(j)) [-1] \lra 
\RG_{\Iw}^{\mathrm{imp}} (T_f(j), \alpha) \lra \RG_{\Iw} (T_f(j)) \oplus 
\left (\underset{\ell \in S}\bigoplus \RG_{\Iw}(\QQ_\ell, T_f(j),\alpha)\right ).
\]
The associated long exact sequence of cohomology reads
\begin{multline}
    \label{eqn: exact sequence for selmer complex}\begin{aligned}
        \ldots \lra 
\underset{\ell\in S}\oplus H^{i-1}_\Iw (\QQ_\ell, T_f(j)) 
\lra \bR^i\boldsymbol{\Gamma}_{\Iw}^{\mathrm{imp}} (T_f(j), \alpha) \lra 
H^i_{\Iw}(T_f(j)) \oplus
\underset{\ell\in S}\bigoplus\bR^{i}\boldsymbol{\Gamma}_\Iw (\QQ_\ell, T_f(j),\alpha) \\ 
\lra \underset{\ell\in S}\oplus H^{i}_\Iw (\QQ_\ell, T_f(j)) 
\lra
\ldots
    \end{aligned}
\end{multline}
Thanks to this exact sequence, it is clear that 
$\bR^0\boldsymbol{\Gamma}_{\Iw}^{\mathrm{imp}} (T_f(j), \alpha)=0$. Moreover,
\[
\bR^1\boldsymbol{\Gamma}_{\Iw}^{\mathrm{imp}} (T_f(j), \alpha)=
\ker \left (H^1_\Iw (T_f(j))\oplus (N_\alpha[j]\otimes\LL)
\rightarrow \underset{\ell\in S}\oplus H^1_\Iw (\QQ_\ell, T_f(j)) \right )\subset \ker (f_\alpha)=0
\]
by Part (ii) of our proposition. 

Note that for all  $\ell\neq p$, we have 
\[
H^1_\Iw (\QQ_\ell,T_f(j))= H^1 (\QQ_\ell (\zeta_{p^\infty})/\QQ_\ell, (\LL \otimes_{\Zp}T_f(j)^{I_\ell})^\iota)=:
\bR^1\boldsymbol{\Gamma}_\Iw(\QQ_\ell, T_f(j),\alpha),
\]
where $I_\ell$ denotes the inertia group at $\ell$ (cf. \cite{PerrinRiou92}, Section~2.2.4). Using the  sequence \eqref{eqn: exact sequence for selmer complex} with $i=2$, we obtain the exact sequence 
\[
H^1_{\Iw}(T_f(j))\oplus (N_\alpha[j]\otimes\LL) \lra 
H^1_\Iw (\Qp, T_f(j))  \lra 
\bR^2\boldsymbol{\Gamma}_{\Iw}^{\mathrm{imp}} (T_f(j), \alpha)
\lra H^2_{\Iw}(T_f(j)) \lra \underset{\ell\in S}\oplus H^2_\Iw (\QQ_\ell, T_f(j))\,.
\]
We therefore conclude that the induced sequence
\[
0\lra \frac{ H^1_{\Iw}(\Qp, T_f(j))}{H^1_{\Iw}(T_f(j))\oplus \Exp_{\alpha,j}(N_\alpha[j]\otimes \Lambda)} \lra 
\mathbf R^2\mathbf \Gamma^{\mathrm{imp}}_{{\Iw}}(T_f(j),\alpha) \lra \Sha^2_\Iw (T_f(j)) \lra 0
\]
is exact. Finally, using \eqref{eqn: exact sequence for selmer complex} with $i=3$, we have
\[
 H^2_{\Iw}(T_f(j)) \lra \underset{\ell\in S}\oplus H^2_\Iw (\QQ_\ell, T_f(j)) \lra
\bR^3\boldsymbol{\Gamma}_{\Iw}^{\mathrm{imp}} (T_f(j), \alpha)
\lra 0\,.
\]
On comparing this exact sequence with the Poitou--Tate exact sequence,
we deduce that 
\[
\bR^3\boldsymbol{\Gamma}_{\Iw}^{\mathrm{imp}} (T_f(j), \alpha) \simeq H^0(\QQ (\zeta_{p^\infty}), V_{f^*}/T_{f^*}(k-j))^{\wedge}.
\]
The proof of these facts for other complexes is completely analogous. 
We also note that
\[
\bR^3\boldsymbol{\Gamma}_{\Iw}^{\mathrm{imp}} (V_f(j), \star)=
H^0(\QQ (\zeta_{p^\infty}), V_{f^*}(k-j))^*=0
\]
because $V_{f^*}^{(\alpha^*)} \simeq E (\nu_{\alpha^*}\chi^{1-k})$ and $V_{f^*}^{(\beta^*)} \simeq E (\nu_{\beta^*})$ with {\it non-trivial} unramified characters $\nu_{\alpha^*}$ and $\nu_{\beta^*}.$ This completes the proof of our proposition.
\end{proof}

\subsection{Projective resolutions  and determinants}
\label{subsec: bezout rings}
\subsubsection{}
In this subsection, we recall some basic facts about
projective resolutions over the rings $\cO_E[[\Gamma_1]],$ $\cO_E[[\Gamma_1]] [1/p]$, and $\CH(\Gamma_1),$ which will be used in the rest of the book. 

Let $R$ be an integral domain. Throughout, projective $R$-module always means a {\it finitely generated} projective $R$-module. Recall that an $R$-module $M$ is perfect if there exists a finite projective resolution $P_\bullet \rightarrow M\rightarrow 0$ of $M$. If $P_\bullet=(P_i)_{i\geqslant 0}$ is finite, its length 
is defined as the largest $i$ such that $P_i\neq 0$. 

Every finitely generated module over a principal ideal domain is perfect and admits a  projective resolution of length $1$. 
If $R$ is a regular local ring of dimension $d,$ then, by a classical theorem of Serre (cf. \cite{Mat92}, \S19) every finitely generated $R$-module admits a projective resolution of length $d$. 
In particular, every finitely generated module over $\cO_E[[\Gamma_1]]$ admits a projective resolution of length $2$.

\subsubsection{} An integral domain is a B\'ezout ring if every finitely generated ideal of $R$ is principal. It follows immediately From this definition that a noetherian B\'ezout ring is principal. An archetypal example of a non-principal  B\'ezout ring is the ring of analytic functions on the open unit disc over a complete non-archimedean field (cf. \cite{lazard62}). In particular, $\CH(\Gamma_1)$ is a B\'ezout ring.

Some results about finitely generated modules over a principal ideal domain extend, with some modifications,  to modules over a B\'ezout ring. For example, every finitely generated torsion-free module over a B\'ezout ring $R$ is free (cf. \cite{kapl49}). Recall that an $R$-module is finitely presented if it is isomorphic to a quotient of the form $F/N$, where $F$ and $N$ are finitely generated and $F$ is free. As a result, we deduce that every finitely presented module over a B\'ezout ring admits a free resolution of length $1$. 
We need a very slight extension of this fact.

\begin{lemma}
\label{lemma: resolutions over Bezout}
Every  finitely presented  $\CH(\Gamma)$-module admits a projective resolution of length $1$. 
\end{lemma}
\begin{proof} Recall the  canonical decomposition \eqref{eqn: decomposition of LL}:
\[
\CH(\Gamma)=\underset{\eta\in X(\Delta)}{\oplus} \CH(\Gamma_1)_\eta\,.
\]
Note that  $\CH(\Gamma_1)_\eta=\CH(\Gamma_1)\delta_{\eta}$ are isomorphic to $\CH(\Gamma_1)$ as rings.  
Correspondingly, any $\CH(\Gamma)$-module $M$ decomposes into a direct sum
\[
M=\underset{\eta\in X(\Delta)}{\oplus} M^{(\eta)},
\]
where each  $M^{(\eta)}$ is an $\CH(\Gamma)_\eta$-module. Since an $\CH(\Gamma)$-module $P$ is projective if and only if $P^{(\eta)}$ is $\CH(\Gamma)_\eta$-projective for all $\eta\in X(\Delta),$ the proof of our lemma follows from the case of B\'ezout rings. 
\end{proof}

\begin{remark} The same argument proves the well-known fact that any finitely generated $\LL$-module admits a projective resolution of length $2$. 
\end{remark}

\subsubsection{}
\label{subsec: abstract trivialization}
In this subsection, we recall for the convenience of the reader  the trivialization of the determinant of a torsion module over an integral domain, and refer to \cite{knudsen-mumford} for the proofs and further details. 

Let $R$ be a domain and let $Q(R)$ denote its field of fractions. Let $C^\bullet$ be a perfect complex of $R$-modules. Assume that $C^\bullet\otimes_R Q(R)$ is acyclic.  Suppose that $P^\bullet$ is a complex of projective modules  quasi-isomorphic to the perfect complex $C^\bullet.$ Then $P^\bullet \otimes_R Q(R)$ is acyclic, and therefore has a canonical trivialization
\[
{\det}_{Q(R)}^{-1} \left (P^\bullet \otimes_R Q(R)\right ) \xrightarrow{\,\,\sim\,\,} Q(R).
\]
(We work with ${\det}^{-1}$ rather than with  $\det$. This is because ${\det}^{-1}$ is a  natural generalization of the characteristic polynomial of an Iwasawa module.)
Composing this isomorphism with the canonical map
\[
{\det}^{-1}_R (C^{\bullet}) \simeq {\det}^{-1}_{R} (P^{\bullet})
\xhookrightarrow{\quad}  {\det}^{-1}_{Q(R)} (P^{\bullet}\otimes Q(R))
\]
we obtain a trivialization
\begin{equation}
    \label{eqn: trivialization of abstract complex}
\vartheta_{C^\bullet} \,:\,{\det}^{-1}_R (C^{\bullet}) \xhookrightarrow{\quad} Q(R).
\end{equation}
Note that  if $P'^{\bullet}$ is another complex of projective modules
quasi-isomorphic to $C^\bullet,$ we then have a diagram
\[
\xymatrix{
& {\det}^{-1}_{Q(R)} \bigl (P^{\bullet}\otimes_R Q(R)\bigr ) \ar[dr] 
\ar[dd]^{{}\simeq}&\\
{\det}^{-1}_{R} (C^{\bullet}) \ar[ur] \ar[dr]
& & Q(R)\,.\\
&{\det}^{-1}_{Q(R)} \bigl (P'^{\bullet}\otimes_R Q(R)
\bigr ) \ar[ur] &
}
\]

The triangle on the left of this diagram commutes tautologically, and the commutativity of the triangle on the right follows easily from the definition
of the trivialization map \cite[pp.~32--34]{knudsen-mumford}.
Therefore, $\vartheta=\vartheta_{C^\bullet}$ does not depend on the choice of $P^\bullet$.

\subsubsection{} We need the following partial generalization of the 
previous construction to the case of an arbitrary ring $R$ (is not necessarily an integral domain). Let $M$ be  a torsion $R$-module that admits a projective resolution $P_{\bullet}= (P_1\xrightarrow{\phi} P_0)$ of length one by finitely generated {\it free} modules of rank $r$. The map $\det{(\phi)}\,:\, \wedge^r P_1\rightarrow \wedge^r  P_0$ induces a map
\[
{\det}_{R}^{-1}
\left ( P_1 \xrightarrow{\phi} P_0\right )=
{\det}_{R}(P_1)\otimes  {\det}_{R}^{-1}
 ({P}_0)
 \xrightarrow{\det (\phi)} {R},
\]
which we  also denote by $\det (\phi).$  The arguments of \cite{knudsen-mumford} apply verbatim and show  that the composition
\begin{equation}
\label{eqn: trivialization for length one resolution}
\vartheta_M\,:\, {\det}^{-1}_R(M) \simeq {\det}^{-1}_R(P_1\xrightarrow{\phi} P_0)
\xhookrightarrow{\quad} R
\end{equation}
does not depend on the choice of $P_\bullet$\,. If $R$ is an integral domain, 
$\vartheta_M$ coincides with the map \eqref{eqn: trivialization of abstract complex}
where $M$ is viewed as a complex concentrated in degree $0.$


\subsection{Modules of algebraic $p$-adic $L$-functions}

\subsubsection{}\label{subsubsec_3312_20220505} 
Let us assume that the hypothesis \eqref{item_S} holds true for some choice of $j=j_0$. Then by Proposition~\ref{prop: selmer complex beta}, the cohomology of the complexes $\RG^{\mathrm{imp}}_{\Iw}(V_f(j),\alpha )$ and $\RG_{\Iw}(V_f(j),\beta )$
are $\LL_E$-torsion, and  we have  injective maps  
\[
i_{\Iw,V_f(j)}^{(\alpha),\mathrm{imp}}\,:\, {\det}_{\Lambda_E}^{-1} \RG^{\mathrm{imp}}_{\Iw}(V_f(j),\alpha)\hookrightarrow \LL_E\,,
\qquad 
i_{\Iw,V_f(j)}^{(\beta)}\,:\, {\det}_{\Lambda_E}^{-1} \RG_{\Iw}(V_f(j),\beta)\hookrightarrow \LL_E
\] 
for every integer $j$.
The  Iwasawa theoretic version of the fundamental  line of $T_f(j)$ associated to 
local conditions $\left (\RG_{\Iw}(\QQ_\ell, T_f(j),\star)\right )_{\ell\in S}$
is given as follows, cf. \cite{perrinriou95,benoisextracris}:
$$
\Delta_{\Iw} (T_{f}(j),N_\star[j]):=
{\det}^{-1}_{\LL}\left (\RG_{\Iw}(T_{f}(j)) \bigoplus \left (\underset{\ell\in S}\bigoplus
\RG_{\Iw} (\QQ_\ell,T_{f}(j),\star ))\right )\right ) \otimes 
{\det}_{\LL} \left (\underset{\ell\in S}\bigoplus
\RG_{\Iw}(\QQ_\ell,T_{f}(j))\right ).
$$
It follows from  definitions  that we have canonical morphisms
\[
\begin{aligned}
&\Theta_{\Iw, V_f(j)}^{(\alpha),\mathrm{imp}}\,:\,\Delta_{\Iw} (T_{f}(j),N_\alpha[j]) \hookrightarrow  {\det}_{\Lambda_E}^{-1} \RG^{\mathrm{imp}}_{\Iw}(V_f(j),\alpha), \\
&
\Theta_{\Iw, V_f(j)}^{(\beta)}\,:\,\Delta_{\Iw} (T_{f}(j),N_\beta[j]) \hookrightarrow  {\det}_{\Lambda_E}^{-1} \RG_{\Iw}(V_f(j),\beta).
\end{aligned}
\]


\begin{defn} 
\label{defn_327_2022_08_18_1205}
Assume that the hypothesis \eqref{item_S} holds true for some $j=j_0$. Then for any integer $j$, the free $\LL_{\cO_E}$-modules 
\[
\begin{aligned}
&\mathbf{L}^{\mathrm{imp}}_{\Iw,\alpha} (T_{f}(j),N_\alpha[j]):=
i^{(\alpha),\mathrm{imp}}_{{\Iw,V_f(j)}}
\circ \Theta_{\Iw, V_f(j)}^{(\alpha),\mathrm{imp}}
\bigl (\Delta_{\Iw} (T_{f}(j),N_\alpha[j])\bigr ),
\\
&\mathbf{L}_{\Iw,\beta} (T_{f}(j),N_\beta[j]):=i_{\Iw,V_f(j)}^{(\beta)} \circ \Theta_{\Iw, V_f(j)}^{(\alpha)} \bigl (\Delta_{\Iw}T_{f}(j),N_\beta[j])\bigr )
\end{aligned}
\]
of rank one are called the modules of (algebraic) $p$-adic $L$-functions.
\end{defn}

We remark that the Iwasawa theoretic properties of $\mathbf{L}_{\Iw,\beta} (T_{f}(j),N_\beta[j])$ associated to the slope-zero eigenvalue $\beta$ are covered by  Perrin-Riou's general theory developed in~\cite{perrinriou95}. In contrast, the module $\mathbf{L}^{\mathrm{imp}}_{\Iw,\alpha} (T_{f}(j),N_\alpha[j])$ associated to the critical eigenvalue $\alpha$ does not fit in the framework of op. cit. and will be studied in the remainder of this section.

We record the first properties of these modules in the following:

\begin{proposition} 
\label{eqn: computation of the Euler-Poincare line for beta}
Assume that $\Omega_p(T_f^{(\star)},N_\star)\in \cO_E^*$ for
$\star \in \{\alpha,\beta\}.$
The following statements hold true for all $m\in \ZZ$, under the validity of \eqref{item_S} for some integer $j=j_0$: 

\item[i)] One has
$$
\begin{aligned}
&\mathbf{L}_{\Iw,\alpha}^{\mathrm{imp}} (T_{f}(m),N_\alpha[m])= 
\mathrm{char}_{\LL} \left ( \frac{ H^1_{\Iw}(\Qp, T_f(m))}
{H^1_{\Iw}(T_f(m))+  \Exp_{\alpha,m}(N_\alpha [j]\otimes \LL)}\right )\cdot 
\mathrm{char}_{\LL}(\Sha^2_\Iw (T_f(m)))\,, \\
&\mathbf{L}_{\Iw,\beta} (T_{f}(m),N_\beta[m])= 
\mathrm{char}_{\LL} \left ( \frac{ H^1_{\Iw}(\Qp, T_f(m))}
{H^1_{\Iw}(T_f(m))+  \Exp_{\beta,m}(N_\beta [j]\otimes \LL)}\right )\cdot 
\mathrm{char}_{\LL}(\Sha^2_\Iw (T_f(m)))\,.
\end{aligned}
$$

\item[ii)] $\mathbf{L}_{\Iw,\alpha}^{\mathrm{imp}} (T_{f}(m),N_\alpha[m])=\mathscr L_{\Iw}^{\rm cr}(V_f(m))
\cdot \mathbf{L}_{\Iw,\beta} (T_{f}(m),N_\beta[m]).$
\end{proposition}
\begin{proof} 
\item[i)] This follows immediately from Proposition~\ref{prop: selmer complex beta}. Note that 
$ H^0(\QQ (\zeta_{p^\infty}), V_{f^*}/T_{f^*}(k-m))$ is finite and hence 
\[
\mathrm{char}_{\Lambda}\left (\mathbf R^3\mathbf \Gamma^{\mathrm{imp}}_{{\Iw}}(T_f(m),\alpha)\right )=\LL=\mathrm{char}_{\Lambda}\left (\mathbf R^3\mathbf \Gamma_{{\Iw}}(T_f(m),\beta)\right ).
\] 

\item[ii)] It follows from the definition of the Iwasawa theoretic critical $\mathcal L$-invariant 
and the elementary properties of determinants that
\[
\mathrm{char}_{\LL} \left ( \frac{ H^1_{\Iw}(\Qp, T_f(m))}
{H^1_{\Iw}(T_f(m))+ \Exp_{\alpha,j}(N_\alpha [m]\otimes \LL)}\right )=
\mathscr L_{\Iw}^{\mathrm{cr}}(V_f(m))\cdot \mathrm{char}_{\LL} \left ( \frac{ H^1_{\Iw}(\Qp, T_f(m))}{H^1_{\Iw}(T_f(m))+ \Exp_{\beta,m}(N_\beta [m]\otimes \LL)}\right ).
\] 
This combined with Part (i) proves the asserted equality. 
\end{proof}

\section{Main Conjecture for $\theta$-critical forms}
\label{sect: main conjecture for critical forms}

\subsection{}
In \S\ref{sect: main conjecture for critical forms}, we will always assume that $T_f$ is chosen so that 
\[
\Omega_p(T_f^{(\star)},N_\star)\in \cO_E^*, \qquad \star \in \{\alpha,\beta\}\,.
\]
The slope-zero ($p$-ordinary) Iwasawa Main Conjecture in the  formulation of Perrin-Riou reads:

\begin{conj}[$p$-ordinary Iwasawa Main Conjecture] 
\label{conj:ordinary MC}$\,$
\begin{itemize} 
\item[\mylabel{item_MCbeta}{$\mathbf{MC} (f,\beta)$}] \hspace{2cm} 
$\mathbf{L}_{\Iw,\beta}(T_f(k),N_\beta[k])= (L_{\mathrm{S},\beta^*}(f^*,\xi^*)^\iota)\,.$
\end{itemize}
\end{conj}

Proposition~\ref{eqn: computation of the Euler-Poincare line for beta} allows us to translate \ref{item_MCbeta} to the classical version of the slope-zero Iwasawa main conjecture.  Furthermore, under the validity of \eqref{item_SZ1} or \eqref{item_SZ1}, the conjecture \ref{item_MCbeta} holds true; cf. \cite{skinnerurbanmainconj, xinwanwanhilbert}, \cite[Theorem 2.5.2]{skinnerPasificJournal2016}, \cite[Theorem 15.2]{kato04}.

Note that \ref{item_MCbeta} is invariant under cyclotomic twists, namely, it is equivalent to the assertion that 
\[
\nonumber
i_{\Iw,T_f(k-j)}^{(\beta)}\circ \Theta_{\Iw,V_f(k-j)}^{(\beta)}\left (\Delta_{\Iw} (T_{f}(k-j),\beta)\right )= (\Tw_{j}\circ L_{\mathrm{S},\beta^*}(f^*,\xi^*)^\iota), \qquad \forall j\in \ZZ\,.
\]
(We recall that  $\Tw_{j}\circ F(\gamma-1)^\iota:= F(\chi^{j}(\gamma)\gamma-1)^\iota$.)

We next formulate the critical Iwasawa Main Conjecture in the style of Perrin-Riou:
\begin{conj}[Critical Iwasawa Main Conjecture] 
We have $\mathscr L^{\crit}_\Iw (V_f)\neq 0$ and
\begin{itemize} 
\item[\mylabel{item_MCalpha}{$\mathbf{MC} (f,\alpha)$}] \hspace{2cm} 
$\mathbf{L}_{\Iw,\alpha}^{\mathrm{imp}}(T_f(k),N_\alpha[k])^\pm = 
\left (\lambda (f^*)^\pm L^{\mathrm{imp},\pm}_{\mathrm{S},\alpha^*}(f^*,\xi^*)^\iota
\right  )\,,
$
\end{itemize}
where $\lambda (f^*)^\pm$ is the constant defined in 
Corollary~\ref{cor_thm:comparision with Bellaiche's construction}.
\end{conj}
We may of course rewrite \ref{item_MCalpha} in the following equivalent form: 
\\
\ref{item_MCalpha} \qquad\qquad\qquad $  \mathbf{L}_{\Iw,\alpha}^{\mathrm{imp}}(T_f(k),N_\alpha[k])= 
\left (L^{\mathrm{imp}}_{\mathrm{K},\alpha^*}(f^*,\xi^*)^\iota
\right  ),$
\\
where $L^{\mathrm{imp}}_{\mathrm{K},\alpha^*}(f^*,\xi^*)$ is defined    in Theorem~\ref{thm_interpolative_properties}. It is in its original formulation where the $p$-adic $L$-function  $L^{[0],\mathrm{imp}}_{\mathrm{S},\alpha^*}(f^*,\xi^*)$, which is constructed in a purely analytic manner, makes an appearance.

Moreover,  Proposition~\ref{eqn: computation of the Euler-Poincare line for beta},
allows   to rewrite \ref{item_MCalpha} in the following explicit form:
$$
{\rm char}_{\LL} \left ( \frac{ H^1_{\Iw}(\Qp, T_f(k))}{H^1_{\Iw}(T_f(k))+
\Exp_{\alpha,k} (N_\alpha [k]\otimes \LL)}\right )^\pm \cdot 
\mathrm{char}_{\LL}(\Sha^2_\Iw (T_f(k)))^\pm = 
{ \left (\lambda (f^*)^\pm L^{[0],\mathrm{imp},\pm}_{\mathrm{S},\alpha^*}(f^*,\xi^*)
\right  ).
}
$$

We may link the main conjectures \ref{item_MCalpha} and \ref{item_MCbeta} with one another:

\begin{proposition}
\label{prop: equivalence of main conjectures}
 Assume that $\mathscr L_\Iw^{\rm cr} (V_f)\neq 0$. Then the  conjectures  \ref{item_MCalpha}  and  \ref{item_MCbeta}  are equivalent.
\end{proposition}
\begin{proof} 
This follows from Propositions~\ref{prop: comparision p-adic L-functions for alpha and beta} and ~\ref{eqn: computation of the Euler-Poincare line for beta}. 
\end{proof}
We remind the reader that the non-vanishing $\mathscr L_\Iw^{\rm cr} (V_f)\neq 0$ of the Iwasawa theoretic critical $\cL$-invariant is equivalent to the requirement that the conditions \eqref{item_pB} and \eqref{item_S} holds true for some integer $j=j_0$. See \S\ref{subsubsec_3227_2022_04_29} where we have an extensive discussion on the validity of \eqref{item_S} in the scenario when the eigenform $f$ has CM.

\begin{remark} Proposition~\ref{prop: equivalence of main conjectures}
tells us that the Iwasawa theory of the critical stabilization of $f^*$ is equivalent to that of its ordinary stabilization. In Chapter~\ref{chapter_main_conj_infinitesimal_deformation}, we shall formulate a refined version of  \ref{item_MCalpha} that relates the infinitesimal thickening $\widetilde{L}_{\mathrm{S},\alpha^*}(f^*,\xi^*)$ of the critical $p$-adic $L$-function to the Iwasawa theoretic invariants of the infinitesimal deformation of $V_f$ along the eigencurve.
\end{remark}

\subsection{}
\label{subsect: comparision between improved and nonimproved complexes}
We may formulate the critical Main Conjecture in an equivalent form which involves the non-improved $p$-adic $L$-function 
$L_{\mathrm{S},\alpha^*}^{{[0]}}(f^*,\xi^*)$ and the non-improved Selmer complex 
$\RG_\Iw (V_f(k),\alpha).$ Namely, it follows  from the identity \eqref{eqn: relation between exponentials} (where also the factor $\ell (V_{f}(j))$ is defined) that we have a commutative diagram 
\begin{equation}
\begin{aligned}
\label{eqn_3_38_2022_08_09_1139}
\xymatrix{
\RG_{\Iw}(\Qp, T_{f}(j), \alpha)\otimes_\LL \CH (\Gamma) \ar[rr]^-{\Exp_{V_{f}(j),j}} 
\ar[d]_{{\,}\cdot\,\ell (V_{f}(j))}& &\RG_\Iw (\Qp, T_{f}(j))\otimes_\LL \CH (\Gamma) \ar[d]^{\mathrm{id}}\\
\RG_{\Iw} (\Qp, T_{f}(j), \alpha)\otimes_\LL \CH (\Gamma)  \ar[rr]_{\Exp_{\alpha,j}\otimes 
\mathrm{id}} & &\RG_\Iw (\Qp, T_{f}(j))\otimes_\LL \CH (\Gamma)\,.
 } 
\end{aligned}
\end{equation}
for the local conditions at $p$. 

The diagram \eqref{eqn_3_38_2022_08_09_1139} induces a canonical morphism
\begin{equation}
\label{eqn:map between improved and nonimproved complexes}
\RG_{\Iw}(V_{f}(j), \alpha) \lra \RG_{\Iw}^{\mathrm{imp}}(T_{f}(j), \alpha)
\otimes_\LL \CH (\Gamma)\,.
\end{equation}
 
\begin{proposition}
\label{prop: properties of punctual R2Gamma}
Assume that  $\mathscr L_\Iw (V_{f})\neq 0$. 
Then  the following statements hold true:

\item{i)} We have
\begin{equation}
\label{item_RG}
\mathbf R^i\boldsymbol\Gamma_\Iw (V_{f}(j), \alpha)=0 \quad \textrm{ if $i\neq 2.$}
\end{equation}

\item{ii)} The $\CH(\Gamma)$-module $\mathbf R^2\boldsymbol\Gamma_\Iw (V_{f}(j), \alpha)$ is a torsion module. Moreover, it is perfect and admits a projective resolution of length $1$. 
\end{proposition}
\begin{proof} Assume that  $\mathscr L_\Iw (V_{f})\neq 0$. 
Mimicking the proof of Proposition~\ref{prop: selmer complex beta}, we deduce that
\linebreak
$\mathbf R^i\boldsymbol\Gamma_\Iw (V_{f}(j), \alpha)=0$  if $i\neq 2$ and  that $\mathbf R^2\boldsymbol\Gamma_\Iw (V_{f}(j), \alpha)$ is $\CH (\Gamma)$-torsion. Let us consider the exact sequence
\[
0\lra \displaystyle\frac{H^1_\Iw(\Qp,V_{f}(j))\otimes \CH}{H^1_\Iw(V_{f}(j))\otimes \CH+
\Exp_{V_{f}(j),j}(N_{\alpha}[j]\otimes  \CH)}
\lra 
\mathbf R^2\boldsymbol\Gamma_\Iw (V_{f}(j), \alpha)
\lra \Sha^2_\Iw (V_{f}(j))\otimes \CH
\lra 0,
\]
where we write $\otimes \CH$ in place of  $\otimes_{\LL_E} \CH(\Gamma)$ to save space. 

For each $\eta\in X(\Delta)$, the $\eta$-isotypic components  of  $H^1_\Iw(V_{f}(j))$, $H^1_\Iw(\Qp,V_{f}(j))$ and $\Sha^2_\Iw (V_{f}(j))$ are finitely generated modules over the principal ideal domain $\cO_E[[\Gamma_1]][1/p].$ Since $N_{\alpha}[j]\otimes \CH(\Gamma)$ is free of rank one, 
the $\eta$-components of the left and right terms of the exact sequence above are finitely presented and therefore admit projective resolutions of length 1. The behaviour of the projective dimension in short exact sequences (cf. \cite[Lemma 10.109.9(ii)]{stacks-project}) tells us that the projective dimension of $\mathbf R^2\boldsymbol\Gamma_\Iw (V_{f}(j), \alpha)^{(\eta)}$ is bounded by $1$ for each $\eta$. 
The proposition now follows from Lemma~\ref{lemma: resolutions over Bezout}.
\end{proof}

From Proposition~\ref{prop: properties of punctual R2Gamma}, we obtain a canonical trivialization
\[
i_{\Iw,V_f(j)}^{(\alpha)}\,:\, {\det}_{\CH (\Gamma)}^{-1} \RG_\Iw (V_f(j),\alpha) \lra \CH(\Gamma)
\]
and a canonical map
\[
\Theta_{\Iw,V_f(j)}^{(\alpha)}\,:\, \Delta_\Iw (T_f(j),N_\alpha [j])
\hookrightarrow  {\det}_{\CH (\Gamma)}^{-1} \RG_\Iw (V_f(j),\alpha).
\]

\begin{lemma}
\label{lemma_2022_08_18_1206}
 Set
\[
\mathbf{L}_{\Iw,\alpha}(T_f(j),N_\alpha[j]):= 
i_{\Iw,V_f(j)}^{(\alpha)} \circ \Theta_{\Iw,V_f(j)}^{(\alpha)}
\bigl ( \Delta_\Iw (T_f(j),N_\alpha [j])
\bigr ).
\]
Then 
\[
\mathbf{L}_{\Iw,\alpha}(T_f(j),N_\alpha[j])= \ell (V_f(j)) \cdot 
\mathbf{L}_{\Iw,\alpha}^{\mathrm{imp}}(T_f(j),N_\alpha[j]).
\]
\end{lemma}
\begin{proof}
Consider the following commutative diagram:
\small
\[
\xymatrix{
0\ar[r] &\displaystyle\frac{H^1_\Iw(\Qp,T_{f}(j))\otimes \CH}{
H^1_\Iw(T_{f}(j))\otimes \CH+
\Exp_{V_{f}(j),j}(N_{\alpha}[j]\otimes \CH)}
\ar[r] \ar[d]
&\mathbf R^2\boldsymbol\Gamma_\Iw (V_{f}(j), \alpha)
\ar[r] \ar[d] & \Sha^2_\Iw (T_{f}(j))\otimes\CH  \ar[d]^{\mathrm{id}}
\ar[r] &0 \\
0\ar[r] &\displaystyle \left (\frac{H^1_\Iw(\Qp,T_{f}(j))}{
H^1_\Iw(T_{f}(j))+
\Exp_{\alpha,j}(N_{\alpha}[j]\otimes \LL)}\right )\otimes \CH
\ar[r]
&\mathbf R^2\boldsymbol\Gamma^{\mathrm{imp}}_\Iw (V_{f}(j), \alpha)\otimes \CH
\ar[r] & \Sha^2_\Iw (T_{f}(j))\otimes\CH 
\ar[r] &0.
}
\]
\normalsize 
The second row of this diagram is the exact sequence computing the second cohomology of $\RG_{\Iw}^{\mathrm{imp}}(V_f(j),\alpha)$ (cf. Proposition~\ref{prop: selmer complex beta}) tensored with $\CH (\Gamma)$. Here, we have written $\otimes \CH$ in place of $\otimes_\LL \CH (\Gamma)$ to save space. The first row is the analogous exact sequence for $\RG_{\Iw}(V_f(j),\alpha),$ and the vertical maps are induced by the morphism \eqref{eqn:map between improved and nonimproved complexes}. Since
\[
\Exp_{V_{f}(j),j}(N_{\alpha}[j]\otimes_E \CH (\Gamma))= \ell (V_f(j)) \cdot 
\Exp_{\alpha,j}(N_{\alpha}[j]\otimes_E \CH(\Gamma)),
\]
the proof of the lemma follows by taking the determinants in the diagram above.
\end{proof}

Therefore, the conjecture \ref{item_MCalpha} can be written in the following equivalent form:
\begin{equation}
\label{eqn_2022_08_09_1142}
\mathbf{L}_{\Iw,\alpha}(T_f(k),N_\alpha[k])^\pm=
\left ( \lambda (f^*)^\pm L_{\mathrm{S},\alpha^*}^{[0],\pm}(f^*,\xi^*)^\iota\right )=\left(L_{\mathrm{K},\alpha^*}^{\pm}(f^*,\xi^*)^\iota\right),
\end{equation}
where $L_{\mathrm{K},\alpha^*}^{\pm}(f^*,\xi^*)$ is defined in Theorem~\ref{thm_interpolative_properties}.

\section{Iwasawa theoretic descent } 
\label{sec_Iwasawa_theoretic_descent}
In this final section of the present chapter, we develop Iwasawa descent for the modules of algebraic $p$-adic $L$-functions (cf. Definition~\ref{defn_327_2022_08_18_1205}), which we utilize to prove the leading term formulae for these. Our main results in this regard are Theorem~\ref{thm:descent thm: noncentral case}, which applies to both non-critical and non-central critical values, and Theorem~\ref{thm_331_2022_04_29_1629}, which focuses on central critical values. We would like to emphasize also Theorem~\ref{thm: bockstein map in central critical case}, which highlights a significant difference between the modules of $p$-adic $L$-functions associated with the $\theta$-critical $p$-stabilization and that attached to the slope-zero $p$-stabilization. This disparity becomes apparent, for example, in their orders of vanishing at the central critical point (cf. Theorem~\ref{thm: bockstein map in central critical case}(iv)). Considering the Iwasawa main conjectures \ref{item_MCalpha} and \ref{item_MCbeta} in conjunction with Proposition~\ref{prop: comparision p-adic L-functions for alpha and beta}, this difference can be accounted by a pole of our Iwasawa theoretic $\cL$-invariant at the central critical point.

\subsection{General formalism}
\label{subsect: general iwasawa descent}
\subsubsection{} 
\label{subsubsec_2023_08_28_1052}
We recall the formalism of Iwasawa descent. To simplify notation, let us set $z=\gamma_1-1$. Let $\mathscr K(\Gamma_1)$ denote the field of fractions of $\mathscr H(\Gamma_1)$.

Let $C^\bullet$ be a perfect complex of $\mathscr H(\Gamma_1)$-modules. 
Assume  that $C^{\bullet}\otimes_{\mathscr{H}(\Gamma_1) }\mathscr K(\Gamma_1)$ is acyclic and let us write 
$$
\vartheta_\Iw \,\,:\,\, {\det}^{-1}_{\mathscr H (\Gamma_1)} C^\bullet \lra  \mathscr K (\Gamma_1)
$$
for the associated trivialization map  \eqref{eqn: trivialization of abstract complex}.  We then have
$
\vartheta_\Iw ({\det}^{-1}_{\mathscr H (\Gamma_1)} C^\bullet)=f\mathscr H (\Gamma_1),
$
for some $f\in \mathscr K (\Gamma_1)$.  Let $r$ be the unique integer such that 
$z^{-r}f\in \mathscr H_0(\Gamma_1)^\times$ is a unit in the localization $\mathscr H_0(\Gamma_1)$ of $\mathscr H (\Gamma_1)$ with respect to the principal ideal $z\mathscr H(\Gamma_1)$. 

Let us define $C^\bullet_0:=C^\bullet\otimes^{\mathbf L}_{\mathscr H (\Gamma_1)}E$, where the tensor product is with respect to the augmentation map $\mathscr H (\Gamma_1)\to E$.
We have a natural distinguished triangle
\begin{equation*}
C^\bullet \xrightarrow{ [z]}C^\bullet \lra C^\bullet_0 \lra C^{\bullet}[1]\,.
\end{equation*}
In degree $n$, this triangle gives rise to a short exact sequence
\begin{equation}
\label{eqn_2022_05_06_1540}
0\lra H^n(C^\bullet)_{\Gamma_1}\lra H^n(C^\bullet_0)\lra H^{n+1}(C^\bullet)^{\Gamma_1}\lra 0.
\end{equation}

We say that $C^\bullet$ is semi-simple (at the augmentation ideal) if the composition of the natural map
\begin{equation}
\label{eqn: semisimplicity projection}
\mathrm{Bock}_n\,:\, H^n(C^\bullet)^{\Gamma_1}\lra H^n(C^\bullet)\lra H^n(C^\bullet)_{\Gamma_1}  
\end{equation}
is an isomorphism for all $n$. If $C^\bullet$ is semi-simple, then there exists a
canonical  trivialisation of ${\det}_{E}C^\bullet_0$ given by  
\begin{align}
\label{eqn: nekovar trivialization}
\begin{aligned}
    \vartheta_{ 0}\,\,:\,\,{\det}_{E}^{{-1}}C^\bullet_0 &= \underset{n\in \mathbb Z}{\otimes}
{\det}_{E}^{(-1)^{{n-1}}}H^n(C_0^\bullet) \\ 
&\xrightarrow[\eqref{eqn_2022_05_06_1540}]{\sim} 
\underset{n\in \ZZ}{\otimes} \left ({\det}^{(-1)^{{n-1}}}_{E}H^n(C^\bullet)_{\Gamma_1} \otimes 
{\det}^{(-1)^{{n-1}}}_{E}H^{n+1}(C^\bullet)^{\Gamma_1} \right )
  \\
&\xrightarrow{\,\,\,\sim\,\,\,}\underset{n\in  \ZZ}{\otimes} \left ({\det}^{(-1)^{{n-1}}}_{E}H^n(C^\bullet)_{\Gamma_1} \otimes 
{\det}^{(-1)^{{n}}}_{E}H^{n}(C^\bullet)^{\Gamma_1} \right )
\xrightarrow[\eqref{eqn: semisimplicity projection}]{\sim} E\,.
\end{aligned}
\end{align}
 Note that the complex $C^\bullet_0$ is not necessarily acyclic and that  $\vartheta_0$ 
differs from the trivialization \eqref{eqn: trivialization of abstract complex} in general.

\begin{lemma}
\label{lemma Burns-Greither}
Let $C^\bullet$ be a perfect complex of $\mathscr H(\Gamma_1)$-modules. Assume that $C^\bullet\otimes_{\mathscr H(\Gamma_1)} \mathscr K(\Gamma_1)$ is acyclic as well as that $C^\bullet$ is semi-simple. Then, 
\begin{equation*}
r=\displaystyle \underset{n\in  \ZZ}{\sum} (-1)^{n} \dim_{E}H^n(C^\bullet)^{\Gamma_1}\,,
\end{equation*}
and there exists a commutative diagram
\[
\xymatrix{
{\det}^{-1}_{\mathscr H (\Gamma_1)} C^\bullet \ar[rr]^-{z^{-r}\vartheta_\Iw} 
\ar[d]_{\otimes_{\mathscr H (\Gamma_1)}^{\mathbf L}E} &  &\mathscr H_0 (\Gamma_1) \ar[d]\\
{\det}^{-1}_{E}{C^\bullet_0} \ar[rr]_-{\vartheta_{ 0}} &&E
}
\]
in which the right vertical arrow is the augmentation map.
The same holds for complexes of $\LL_E$-modules.
\end{lemma}
\begin{proof} 
See \cite[Lemma 8.1]{BurnsGreither2003Inventiones}. This is a particular case of Nekov\'a\v r's
descent theory \cite{nekovar06}. Note that Burns and Greither consider complexes over $\Lambda_E$ but since $\mathscr H (\Gamma_1)$ is a B\'ezout ring,  all their arguments work in our case and are omitted here. 
\end{proof}

\subsubsection{} The following classical result is a particular case of 
Lemma~\ref{lemma Burns-Greither}. Let  $M$ is a finitely generated torsion $\cO_E [[\Gamma_1]]$-module. Then $M_E:=M\otimes_{\cO_E}E$ is of finite dimension over $E$, and the tautological exact sequence
\[
0\lra M_E^{\Gamma_1} \lra M_E \xrightarrow{{[z]}} M_E \lra (M_E)_{\Gamma_1} \lra 0
\] 
shows that $\dim_E(M_E^{\Gamma_1})=\dim_E (M_E)_{\Gamma_1}$. In particular, 
\be\label{eqn_2022_05_06_1205}
M_E^{\Gamma_1}=\{0\}\, \iff\, (M_E)_{\Gamma_1} =\{0\}\,.
\ee
Let $g=g({z})\in \LL{(\Gamma_1)}$ be a generator of the characteristic ideal of  $M$, and let $r(M)$ denote the order of vanishing of the formal power series $g$ at $z=0$. Assume that $M_E$ is semi-simple, i.e. that the natural map
\[
M_E^{\Gamma_1} \lra  (M_E)_{\Gamma_1}
\]
is an isomorphism. Then,
\begin{equation}
\label{eqn: descent dimension}
r(M)= \dim_E M_E^{\Gamma_1}\,,
\end{equation}
and
\begin{equation}
\label{eqn: descent coinvariants}
\lim_{s\to 0}\frac{g (\chi (\gamma_1)^s-1)}{s^{r(M)}} \sim_p
(\log \chi (\gamma_1))^{r(M)}\left [M_{\Gamma_1} :M^{\Gamma_1}\right ]\,,
\end{equation}
where ``$\sim_p$'' means equality up to multiplication by a $p$-adic unit.

We will repeatedly use these general observations in what follows. The following lemma was used in the proof of Proposition~\ref{prop_2022_04_26_13_00}.
\begin{lemma}
\label{lemma: descent in the rank one case}
Let $M$ be a finitely generated torsion $\LL(\Gamma_1)$-module. Assume that 
$M_E^{\Gamma_1}$ and $(M_E)_{\Gamma_1}$ are one-dimensional $E$-vector spaces. Then  the natural  map $M_E^{\Gamma_1}\rightarrow (M_E)_{\Gamma_1}$ is an isomorphism if and only if  $r(M)=1.$
\end{lemma}
\begin{proof}
Let $A$ denote the  $z$-primary component of $M_E.$
Since $\LL_E$ is principal, $M_E^{\Gamma_1}=
A^{\Gamma_1}$ and $(M_E)_{\Gamma_1}=A_{\Gamma_1}.$
Since $A$ is a direct sum of cyclic modules of the form $E[z]/(z^m),$
from our assumptions if follows immediately that $A\simeq E[z]/(z)$
if and only if the map  $M_E^{\Gamma_1}\rightarrow (M_E)_{\Gamma_1}$
is an isomorphism.
\end{proof}

\subsection{The non-central case}
\label{eqn: general descent non-central case}

\subsubsection{} In this section, we apply Lemma~\ref{lemma Burns-Greither}
to the trivializations 
\[
i_{\Iw,V_f(j)}^{(\alpha),\mathrm{imp}}\,:\,{\det}^{-1}_{\LL_E} \RG_\Iw^{\mathrm{imp}} (V_f(j),\alpha)
\lra \LL_E\,,
\qquad \qquad 
i_{\Iw,V_f(j)}^{(\beta)}\,:\,{\det}^{-1}_{\LL_E} \RG_\Iw (V_f(j),\beta)
\lra \LL_E,
\]
Let us set
\[
\RG^{\mathrm{imp}} (V_f (j),\alpha):=\RG_\Iw (V_f (j),\alpha)\otimes_{\LL_E}^{\mathbf L}E\,, \qquad \qquad 
\RG (V_f (j),\beta):=\RG_\Iw (V_f (j),\beta)\otimes_{\LL_E}^{\mathbf L}E\,.
\]
\subsubsection{}
Let us denote by $\mathbf{L}^{\mathrm{imp}}_{\Iw,\alpha} (T_{f}(j), N_{\alpha}[j],\mathds{1})$ 
(resp.  $\mathbf{L}_{\Iw,\beta} (T_{f}(j), N_{\beta}[j],\mathds{1})$)  
the  isotypic component
of the $\LL_{\cO_E}$-module $\mathbf{L}^{\mathrm{imp}}_{\Iw,\alpha} (T_{f}(j), N_{\alpha}[j])$ (resp.  $\mathbf{L}_{\Iw,\beta} (T_{f}(j), N_{\beta}[j])$) that corresponds to the trivial character  $\mathds{1}\in X(\Delta)$. 

Let $f_\alpha^{\mathrm{imp}} ({z})\in \mathrm{Frac}(\cO_E[[\Gamma_1]])$ (resp.  $f_\beta  ({z})\in \mathrm{Frac}(\cO_E[[\Gamma_1]])$) denote a generator of $\mathbf{L}^{\mathrm{imp}}_{\Iw,\alpha} (T_{f}(j), N_{\alpha}[j],\mathds{1})$ (resp.  $\mathbf{L}_{\Iw,\beta} (T_{f}(j), N_{\alpha}[j],\mathds{1})$) and  let  $r_\alpha(j)$ (resp. $r_\beta(j)$) denote  the order of vanishing of  $f_\alpha^{\mathrm{imp}}({z})$ (resp. of $f_\beta (z)$) at $z=0$. We set 
\[
\begin{aligned}
&\mathbf L_{\Iw,\alpha}^{\mathrm{imp},*}(T_f(j), N_\alpha[j],\mathds{1}, 0) =\lim_{s\to 0} s^{-r_\alpha(j)}f_\alpha^{\mathrm{imp}} (\chi (\gamma_1)^s-1),\\
&\mathbf L_{\Iw,\beta}^{*}(T_f(j), N_\beta[j],\mathds{1}, 0) =\lim_{s\to 0} s^{-r_\beta (j)}f_\beta (\chi (\gamma_1)^s-1).
\end{aligned}
\] 
These quantities depend on the choices of $f_\alpha^{\mathrm{imp}},$  $f_\beta,$  and $\gamma_1$ only up to multiplication by a unit in $\cO_E$. 

\subsubsection{} For $\star \in \{\alpha,\beta\}$, let us set
\begin{equation}
\label{eqn: definition of Delta_E}
\Delta_E(T_f(j),\star)=\Delta_\Iw (T_f(j),\star)\otimes_{\LL}\cO_E\,.
\end{equation}
Then, assuming that the Selmer complexes $\RG^{\mathrm{imp}}_\Iw (V_f(j),\alpha) $ and $ \RG_\Iw (V_f(j),\beta)$ are semi-simple,  Lemma~\ref{lemma Burns-Greither} provides us with a pair of commutative diagrams 
\begin{equation}
\begin{aligned}
\label{eqn: commutative diagrams for descent}
\large
\xymatrix{
{\det}_{\mathscr H (\Gamma)}^{-1} \RG^{\mathrm{imp}}_\Iw (V_f(j),\alpha) \ar[rr]^-{z^{-r_\alpha(j)}i_{\Iw,V_f(j)}^{(\alpha)}} 
\ar[d]_{\otimes_{\mathscr H (\Gamma)}^{\mathbf L}E} &  &\mathscr H (\Gamma) \ar[d]\\
{\det}_{E}^{-1}\RG^{\mathrm{imp}} (V_f(j),\alpha)  \ar[rr]_-{i_{V_f(j)}^{(\alpha),\mathrm{imp}}} &&E
}
\quad\,\,
\xymatrix{
{\det}_{\mathscr H (\Gamma)}^{-1} \RG_\Iw (V_f(j),\beta) \ar[rr]^-{z^{-r_\beta (j)}i_{\Iw,V_f(j)}^{(\beta)}} 
\ar[d]_{\otimes_{\mathscr H (\Gamma)}^{\mathbf L}E} &  &\mathscr H (\Gamma) \ar[d]\\
{\det}_{E}^{-1}\RG (V_f(j),\beta)  \ar[rr]_-{i_{V_f(j)}^{(\beta)}} &&E \,\,.
}
\end{aligned}
\end{equation}
\normalsize
This would in turn furnish us with the following interpretation of special values of algebraic $p$-adic $L$-functions:
\begin{equation}
\label{eqn: general formulas for special values}
\begin{aligned}
& \mathbf L_{\Iw,\alpha}^{\mathrm{imp},*}(T_f(j), N_\alpha[j],\mathds{1}, 0)=  i_{V_f(j)}^{(\alpha),\mathrm{imp}}\left (\Delta_E(T_f(j),\alpha)\right ),\\
&\mathbf L_{\Iw,\beta}^{*}(T_f(j), N_\beta[j],\mathds{1}, 0)=  i_{V_f(j)}^{(\beta)}\left (\Delta_E(T_f(j),\beta)\right ).
\end{aligned}
\end{equation}
We now employ this approach in variety of scenarios.

Proposition~\ref{prop: selmer complex beta} provides us with the following exact
sequences, which will play a key role in the remainder of this chapter:
\begin{align}
\label{eqn: first sequence for R-one-Iw}
&0\lra \mathrm{coker} (g_\alpha) \lra \bR^2\boldsymbol{\Gamma}^{\mathrm{imp}}_\Iw(T_f(j),\alpha) \lra \Sha_\Iw (T_f(j))\lra 0\,,\\
\label{eqn: second sequence for R-one-Iw}
&0\lra \mathrm{coker} (g_\beta) \lra \bR^2\boldsymbol{\Gamma}_\Iw
(T_f(j),\beta) \lra \Sha_\Iw (T_f(j))\lra 0\,,
\end{align}
where 
\begin{equation}
\label{eqn: the map fstar}
{g}_\star\,:\,H^1_{\Iw}(T_f(j)) \oplus N_\star[j]\otimes \LL \xrightarrow{\res_p\, +\, \Exp_{\star,j}}  H^1_{\Iw}(\Qp, T_f(j)),
\qquad\qquad \star \in \{\alpha,\beta\}\,.
\end{equation}

\begin{proposition}
\label{prop: descent non-critical general case}
 Suppose that $j\neq k/2$ is a strictly positive integer. 

\item[i)]{} Assume that we are in one of the following situations:
\begin{itemize}
\item[a)]{} $1\leqslant j\leqslant k-1$\,.
\item[b)]{} $j\geqslant k$ and the conjecture \eqref{item_pB} holds.  
\end{itemize}
Then the following statements  hold true: 
\begin{itemize}
\item[1)] $\RG (V_f(j),\beta)$ is acyclic and $r_\beta (j)=0.$ 

\item[2)]{} The cohomology groups of $\RG (T_f(j),\beta):=\RG_\Iw (T_f(j),\beta)\otimes_{\LL}\cO_E$ have finite cardinality, and we have
\[
\left \vert \mathbf L_{\Iw,\beta}^{*}(T_f(j), N_\beta[j],\mathds{1}, 0) \right \vert^{-1}
= \frac{\# \bR^2\boldsymbol{\Gamma} (T_f(j),\beta)}{\# \bR^1\boldsymbol{\Gamma} (T_f(j),\beta)}.
\]
\end{itemize} 

\item[ii)]{} Assume  that we are in one of the following scenarios:
\begin{itemize}
\item[a)]{} $1\leqslant j\leqslant k-1$ and the condition \eqref{item_S}  holds\,.
\item[b)]{} $j\geqslant k$ and the the both conditions \eqref{item_S} and \eqref{item_pB} hold.  
\end{itemize}
Then the  statements 1-2) above also hold for the complexes $\RG^{\mathrm{imp}} (V_f(j),\alpha),$ $\RG^{\mathrm{imp}} (T_f(j),\alpha)$ and the module 
$\mathbf{L}^{\mathrm{imp}}_{\Iw,\alpha} (T_{f}(j), N_{\alpha}[j]).$ 
In particular, $r_\alpha(j)=0$ and 
\[
\left \vert \mathbf L_{\Iw,\alpha}^{\mathrm{imp},*}(T_f(j), N_\alpha[j],\mathds{1}, 0) \right \vert^{-1}
= \frac{\# \bR^2\boldsymbol{\Gamma}^{\mathrm{imp}} (T_f(j),\alpha)}{\# \bR^1\boldsymbol{\Gamma}^{\mathrm{imp}}  (T_f(j),\alpha)}\,.
\]
 \end{proposition}

 \begin{proof} 
 We prove (i), treating the scenarios (a) and (b) separately. Our proof uses two statements well-known to the experts, Lemma~\ref{lemma:exact sequence with H_f,p} and  Lemma~\ref{lemma: finiteness of coinvariants of Iwasawa Sha}, whose proofs are deferred until \S\ref{subsubsec_202208091256}.
\item[a)] Consider the short exact sequence \eqref{eqn_2022_05_06_1540}, which reads
 \[
0 \lra \RG_\Iw^{i} (V_f(j),\beta)_{\Gamma} 
\lra \bR^i\boldsymbol{\Gamma} (V_f(j),\beta)
\lra \RG_\Iw^{i+1} (V_f(j),\beta)^{\Gamma} \lra 0\,.
 \]
in our case.
Under our running assumptions, $\bR^i \boldsymbol{\Gamma}_\Iw (V_f(j),\beta){=\{0\}}$ if $i\neq 2$. Therefore, we have natural identifications
\[
\bR^1\boldsymbol{\Gamma} (V_f(j),\beta)=\bR^2 \boldsymbol{\Gamma}_\Iw (V_f(j),\beta)^\Gamma\,,
\qquad \qquad
\bR^1\boldsymbol{\Gamma} (V_f(j),\beta)=\bR^2 \boldsymbol{\Gamma}_\Iw (V_f(j),\beta)_\Gamma.
\]
Let us apply the snake lemma to the diagram
$$\xymatrix{
0\ar[r]& \mathrm{coker} (g_\beta) \ar[r]\ar[d]_{[\gamma-1]}& \bR^2\boldsymbol{\Gamma}_\Iw(V_f(j),\beta) \ar[r]\ar[d]_{[\gamma-1]}& \Sha_\Iw (V_f(j))\ar[r]\ar[d]_{[\gamma-1]}& 0\\
0\ar[r]& \mathrm{coker} (g_\beta) \ar[r]& \bR^2\boldsymbol{\Gamma}_\Iw(V_f(j),\beta) \ar[r]& \Sha_\Iw (V_f(j))\ar[r]& 0\\
}$$
that we obtain from the exact sequence \eqref{eqn: second sequence for R-one-Iw} tensored with $\Qp$, where $\gamma$ is a topological generator and the vertical arrows are the multiplication-by-$(\gamma-1)$ maps. Noting that, under our assumptions, we have $\Sha_\Iw (V_f(j))^\Gamma=
{\{0\}}= \Sha_\Iw (V_f(j))_\Gamma$ by Lemma~\ref{lemma: finiteness of coinvariants of Iwasawa Sha} below, we deduce that
  \begin{equation}
 \label{eqn: formulas for coker (F)}
 \begin{aligned}
 &\bR^2 \boldsymbol{\Gamma}_\Iw (V_f(j),\beta)^\Gamma=
 \mathrm{coker} (g_\beta)^\Gamma  =
 \ker \left (
  H^1_\Iw (V_f(j))_\Gamma\oplus H^1_\Iw (\Qp, V_f^{(\beta)}(j))_\Gamma \lra 
H^1_\Iw (\Qp, V_f(j))_\Gamma 
 \right ),
 \\
 &\bR^2 \boldsymbol{\Gamma}_\Iw (V_f(j),\beta)_\Gamma=
 \mathrm{coker} (g_\beta)_\Gamma =
 \mathrm{coker} \left (
  H^1_\Iw (V_f(j))_\Gamma\oplus H^1_\Iw (\Qp, V_f^{(\beta)}(j))_\Gamma \lra 
H^1_\Iw (\Qp, V_f(j))_\Gamma 
 \right ).
 \end{aligned}
 \end{equation}
Using Lemma~\ref{lemma:exact sequence with H_f,p}, we infer that 
$$H^1_\Iw (V_f(j))_\Gamma=H^1_{\{p\}}(V_f(j))\,\,, \qquad H^1_\Iw (\Qp, V_f^{(\beta)}(j))_\Gamma=
H^1(\Qp, V_f^{(\beta)}(j))\,.$$
Hence,  \eqref{eqn: formulas for coker (F)} can be recast as
\begin{equation}
 \label{eqn: formulas for coker (F) bis}
 \begin{aligned}
 &\bR^2 \boldsymbol{\Gamma}_\Iw (V_f(j),\beta)^\Gamma=
  \mathrm{coker} (g_\beta)^\Gamma  =
 \ker \left (
  H^1_{\{p\}} (V_f(j))\oplus H^1(\Qp, V_f^{(\beta)}(j)) \lra 
H^1(\Qp, V_f(j)) 
 \right ),
\\
 &\bR^2 \boldsymbol{\Gamma}_\Iw (V_f(j),\beta)_\Gamma=
 \mathrm{coker} (g_\beta)_\Gamma =
 \mathrm{coker} \left (
  H^1_{\{p\}} (V_f(j))\oplus H^1(\Qp, V_f^{(\beta)}(j)) \lra 
H^1(\Qp, V_f(j)) 
 \right ).
 \end{aligned}
 \end{equation}
 Since we have $H^1_\beta (V_f(j))=0$ thanks to our running hypotheses, we deduce that 
 $$\bR^2 \boldsymbol{\Gamma}_\Iw(V_f(j),\beta)^\Gamma=\{0\}= \bR^2 \boldsymbol{\Gamma}_\Iw(V_f(j),\beta)_\Gamma\,.$$
 This completes the proof that the complex $\RG (V_f(j),\beta)$ is acyclic.
 
\item[b)] It follows from the acyclicity of $\RG (V_f(j),\beta)$ that $r_\beta(j)=0$ and the map $i_{V_f(j)}^{(\beta)}$ coincides with the composition of the isomorphisms
\[
{\det}_{E}^{-1} \RG (V_f(j),\beta)\xrightarrow{\sim} 
\bigotimes_{i\in \ZZ} {\det}_{E}^{(-1)^{i-1}} \bR^i\boldsymbol{\Gamma} (V_f(j),\beta)
\xrightarrow{\sim} 
 E.
\]
Since $\RG^i(T_f(j),{\beta})\otimes_{\cO_E}E=\RG^i (V_f(j),{\beta})=\{0\}$ for all $i\in \ZZ$, it follows that the cohomology groups of the complex  $\RG (T_f(j),{\beta})$ have finite cardinality. The part (i) of the proposition follows from the commutativity of the diagram \eqref{eqn: commutative diagrams for descent} and  the fact that
$\Delta_\Iw (T_f(j),\beta)\simeq {\det}_{\LL}^{-1} \RG_\Iw (T_f(j),\beta)$.

The proof of part (ii) is entirely analogous, and we omit it here.
\end{proof}

\subsubsection{} 
\label{subsubsec_3524_2022_08_16}
Let us apply Lemma~\ref{lemma Burns-Greither} with the complex $\RG_\Iw (V_f(j),\alpha)$. This computation will be used in Chapter~\ref{chapter_main_conj_infinitesimal_deformation}.

Let us put
\[
\RG (V_f (j),\alpha):=\RG_\Iw (V_f (j),\alpha)\otimes_{\CH(\Gamma)}^{\mathbf L}E.
\]

\begin{proposition} 
\label{prop: descent for non-improved complex}
Let  $1\leqslant j\leqslant k-1$ be an integer such that $j\neq k/2$. Assume that condition \eqref{item_S} holds. 
Then the following statements  hold true: 
\item[i)] We have  
\[
\bR^1\boldsymbol{\Gamma} (V_f(j),\alpha)\simeq D^{(\alpha)}\,,
\qquad \qquad 
\bR^2\boldsymbol{\Gamma} (V_f(j),\alpha)\simeq \frac{H^1(\Qp, V_f(j)}{\res_p(H^1_{\{p\}}(V_f(j)))}\,,
\]
and $\bR^i\boldsymbol{\Gamma} (V_f(j),\alpha)=0$ for $i\neq 1,2.$

\item[ii)] The complex $\RG_\Iw (V_f(j),\alpha)$ is semi-simple.  The Bockstein map together with the isomorphisms in Part (i) give rise to a commutative diagram
\begin{equation}
\label{eqn diagram non-improved descent}
\begin{aligned}
\xymatrix{
\bR^1\boldsymbol{\Gamma} (V_f(j),\alpha) \ar[d]_{\mathrm{Bock_2}}\ar[r] & D^{(\alpha)} \ar[d]^-{\delta}\\
\bR^2\boldsymbol{\Gamma} (V_f(j),\alpha)\ar[r] &\displaystyle\frac{H^1(\Qp, V_f(j)}{\res_p(H^1_{\{p\}}(V_f(j)))}\,,
}
\end{aligned}
\end{equation}
where the map $\delta$ is given by 
\[
\delta (d)=\ell^*(V_f(j))\cdot \left (\pr_0 \circ \Exp_{\alpha,j} (d)\right )\,, \qquad 
\ell^*(V_f(j))=\underset{i=1-k}{\overset{-1}\prod} (i+j)
\]
or, in more explicit form, 
\[
\delta (d)
=- (j-1)! \left ( \frac{1-p^{j-1}\alpha^{-1}}{1-p^{-j}\alpha}\right ) 
\left (\exp^*_{V^{(\alpha)}_{x_0}(j)}\right )^{-1}(d)\,.
\]

\item[iii)] Both squares in the following diagram commute:
\[
\xymatrix{
\Delta_E (T_f(j),\alpha) \ar@{^{(}->}[r] \ar@{=}[d] & {\det}^{-1}_E \RG^{\mathrm{imp}}(V_f(j),\alpha) 
\ar[rr]^-{i_{V_f(j)}^{(\alpha),\mathrm{imp}}} \ar[d] & &E \ar[d]^-{\ell^*(V_f(j))}\\
\Delta_E (T_f(j),\alpha) \ar@{^{(}->}[r]  &{\det}^{-1}_E\RG (V_f(j),\alpha) 
\ar[rr]_-{i_{V_f(j)}^{(\alpha)}}  & & E.
}
\]
\end{proposition}
\begin{proof} 
\item[i)] The proof is purely formal. We remark that 
the map 
$$\Exp_{V(j),j,0}\,:\, D^{(\alpha)} \lra H^1(\Qp, V_f(j))$$
is the zero-map and therefore, the kernel and the cokernel of the map
\[
H^1_{\Iw}(V_f(j))_\Gamma \oplus D^{(\alpha)} \xrightarrow{\res_p+\Exp_{V_f(j)j,0}}
H^1_\Iw (\Qp, V_f(j))_\Gamma
\]  
are isomorphic to $D^{(\alpha)}$ and $\displaystyle\frac{H^1(\Qp,V_f(j))}{\res_p(H^1_{\{p\}}(V_f(j)))}$, respectively. The proof of (i) can be completed mimicking the arguments used in the proof of Proposition~\ref{prop: descent non-critical general case}.

\item[ii)] It follows from the definition of the Bockstein map that 
\[
\delta (d)= \pr_0 (y) \mod H^1_{\{p\}}(V_f(j)), \quad \textrm{where $\Exp_{V_f(j),j}(y)=d.$}
\]
Taking \eqref{eqn: relation between exponentials} into account together with the formulae \eqref{eqn:specialization of PR formulae}, we conclude the proof of Part (ii). 

\item[iii)]  This assertion directly follows from definitions.
\end{proof}

\subsubsection{} 
\label{subsubsec_202208091256}
We prove give a proof of the two lemmas used in the proof of Proposition~\ref{prop: descent non-critical general case}. 
Both statements are well-known (cf. \cite[Proposition~3.4.2(i)]{perrinriou95}), but we record them here for the reader's convenience.

\begin{lemma}
\label{lemma:exact sequence with H_f,p}
For any $p$-adic representation $W$ of $G_{\QQ,S}$ with $H^0(\Qp, W^*(1))=0$, we have a short exact sequence
\[
0\lra H^1_{\Iw}(W)_{\Gamma} \lra H^1_{\{p\}}(W) \lra \Sha^2_{\Iw}(W)^\Gamma \lra 0\,.
\] 
\end{lemma}
\begin{proof} Consider the following commutative diagram
\[
\xymatrix{
& & 0\ar[d]&0\ar[d] & \\
& &H^1_{\{p\}}(W) \ar[d] \ar@{.>}[r]&\Sha_\Iw^2(W)^\Gamma \ar[d] & \\
0 \ar[r] &H^1_\Iw (W)_{\Gamma} \ar[r] \ar[d] & H^1 (W) \ar[r] \ar[d] 
&H^2_\Iw (W)^\Gamma \ar[r] \ar[d]   & 0 \\
0 \ar[r] &\underset{\ell \in S\setminus \{p\}}\oplus H^1_\Iw (\QQ_\ell, W)_\Gamma 
\ar[r]
& \underset{\ell \in S\setminus \{p\}}\oplus H^1 (\QQ_\ell, W)
\ar[r] 
&\underset{\ell \in S\setminus \{p\}}\oplus H^2_\Iw (\QQ_\ell, W)^\Gamma
\ar[r] 
&0\,.
}
\]
The rows and the middle column of this diagram are exact. We explain that the right-most column is also exact. We have 
\[
H^2_\Iw (\Qp, W)^\Gamma\simeq \left (H^0(\Qp (\zeta_{p^\infty}),W^*(1))_\Gamma\right )^*
\]
by local duality. It follows from \eqref{eqn_2022_05_06_1205} that
\[
\dim_E H^0(\Qp (\zeta_{p^\infty}),W^*(1))_\Gamma=
\dim_E H^0(\Qp (\zeta_{p^\infty}),W^*(1))^\Gamma =\dim_E H^0(\Qp, W^*(1))=0,
\]
where the final vanishing is our running assumption. Hence  $H^2_\Iw (\Qp, W)^\Gamma=0$. This shows that the right-most column in the diagram above is indeed exact, and in turn, allows us to define the dotted arrow. Moreover, we have 
$$H^1_\Iw (\QQ_\ell, W)_{\Gamma}\stackrel{\sim}{\lra} H^1_{\rm f}(\QQ_\ell, W)$$ 
for all $\ell\neq p$. The lemma follows from this fact via a diagram chase.
\end{proof}

\begin{lemma} 
\label{lemma: finiteness of coinvariants of Iwasawa Sha}

\item[i)]{} There exists a canonical isomorphism
\[
\Sha^2_\Iw (V_f(j))_\Gamma \simeq H^1_{\mathrm 0}(V_{f^*}(k-j))
\] 

\item[ii)]{} Assume that $H^1_{\mathrm 0}(V_{f^*}(k-j))=0.$ Then, 
\[
\Sha^2_\Iw (V_f(j))^\Gamma =\{0\}= \Sha^2_\Iw (V_f(j))_\Gamma 
\]
and $H^1_\Iw (V_f(j))_\Gamma =H^1_{\{p\}}(V_f(j)).$
\end{lemma} 
\begin{proof} Part i) is proved in  \cite[Section 3.3.4]{perrinriou95} (see also \cite{benoisextracris}, Section~6.1.3).
Part ii)  follows from this fact and Lemma~\ref{lemma:exact sequence with H_f,p}.
\end{proof}

\subsection{The central critical case} 
\label{subsec_353_2023_07_11_1619}
In this section, we assume that $k$ is even and consider the case $j=\frac{k}{2}$. Proposition~\ref{prop: descent for non-improved complex} applies also with $j=\frac{k}{2}$ when $r_{\rm an}(\frac{k}{2})=0$. We therefore assume $r_{\rm an}(\frac{k}{2})\geq 1$. 
\subsubsection{} 
\label{subsubsec_3531_2022_08_23_1647}
Recall the $p$-adic heights
\[
\left < \,,\,\right >_\star : H^1_\star (V_{f^*}({k}/{2})) \otimes H^1_\star  (V_{f}({k}/{2})) 
\lra E, \qquad\qquad \star \in\{\alpha,\beta\},
\]
associated to splitting submodules   $D_\star [k/2] \subset \Dc (V_f(k/2))$, cf. \cite{benoisheights}. Recall that $H^1_\beta (V_{f^*}({k}/{2}))=H^1_{\mathrm{f}} (V_{f^*}({k}/2))$
and  $\left < \,,\,\right >_\beta$ coincides with the classical $p$-adic height 
pairing introduced in \cite{nekovarheightpairings}. 

Let us define the associated regulator
\[
R_\star (T_f({k}/{2})):=\left \vert \det\left <e_i^*, e_j\right >_\star \right \vert^{-1}\,,
\]
(as above, the subscript ``$\star$'' stands for $\alpha$ or $\beta$) for any choice of bases $\{e_i^*\}$ of $H^1_\star (T_{f^*}({k}/{2}))_\tf$
and  $\{e_i\}$ of $H^1_\star (T_{f^*}({k}/{2}))_\tf$. It is clear that $R_\star (T_f({k}/{2}))$ depends on the choice of these bases only up to multiplication by a $p$-adic unit. 

In what follows, whenever $M$ and $N$ are free submodules of $H^1_\star (T_{f^*}({k}/{2}))_\tf$ and $H^1_\star (T_{f}({k}/{2}))_\tf$ of the same rank and with bases $\{m_i\}$ and $\{n_j\}$, respectively, we write $\left< M,N\right>_{ \star}$ as a shorthand for 
$ \det\left <m_i, n_j\right >_\star $.

\begin{theorem}
\label{thm: bockstein map in central critical case}
Assume that the conditions \eqref{item_BK_cc} and \eqref{item_PR} hold true. Then:

\item[i)]{} $ H^1_\alpha (V_f(k/2)) = H^1_0 (V_f(k/2))$ and 
the pairing $\left <\,,\,\right >_\alpha$ coincides with the restriction of 
$\left <\,,\,\right >_\beta$ on $H^1_0(V_f(k/2)).$

\item[ii)]{} There exist canonical isomorphisms 
\[
\begin{aligned}
&\bR^1\boldsymbol{\Gamma}^{\mathrm{imp}}(V_f(k/2),\alpha)\simeq H^1_0 (V_f(k/2))\,, 
&&\bR^2\boldsymbol{\Gamma}^{\mathrm{imp}}(V_f(k/2),\alpha) \simeq H^1_0 (V_{f^*}(k/2))^*
\\
&\bR^1\boldsymbol{\Gamma}(V_f(k/2),\beta)\simeq H^1_{\mathrm{f}} (V_f(k/2))\,, 
&&\bR^2\boldsymbol{\Gamma}(V_f(k/2),\beta) \simeq H^1_{\mathrm{f}} (V_{f^*}(k/2))^*.
\end{aligned}
\]
\item[iii)]{} The Bockstein maps are given by the diagrams
\[
\xymatrix{
\bR^1\boldsymbol{\Gamma}^{\mathrm{imp}}(V_f(k/2),\alpha)\ar[r]^-{\sim} \ar[d]_-{\mathrm{Bock}_2}
&H^1_0 (V_f(k/2)) \ar[d]^{\left <\,,\,\right>_\alpha} \\
\bR^2\boldsymbol{\Gamma}^{\mathrm{imp}}(V_f(k/2),\alpha) \ar[r]^-{\sim} &H^1_0 (V_{f^*}(k/2))^*
}
\qquad \qquad
\xymatrix{
\bR^1\boldsymbol{\Gamma}(V_f(k/2),\beta)\ar[r]^-{\sim} \ar[d]_{\mathrm{Bock}_2}
&H^1_{\mathrm{f}} (V_f(k/2)) \ar[d]^{\left <\,,\,\right >_\beta} \\
\bR^2\boldsymbol{\Gamma}(V_f(k/2),\beta) \ar[r]^-{\sim} &H^1_{\mathrm{f}} (V_{f^*}(k/2))^*
}
\]
where the right vertical maps in both diagrams are induced by the $p$-adic height pairing. 

In particular, the complex $\RG_\Iw^{\mathrm{imp}}(V_f({k/2}),\alpha)$ 
(resp. $\RG_\Iw(V_f(k/2),\beta)$) is semi-simple if and only if the $p$-adic height pairing $\left <\,,\,\right >_\alpha$ (resp. $\left <\,,\,\right >_\beta$)
is non-degenerate. 

\item[iv)]{} We have 
\[
r_\alpha (k/2)= \dim_EH^1_0 (V_f(k/2))\,,
\qquad \qquad 
r_\beta (k/2)=\dim_EH^1_{\rm f} (V_f(k/2))\,.
\]
If in addition \eqref{item_slope_zero_ht} holds, we have $r_\beta (k/2)=r_\alpha (k/2)+1$.
\end{theorem}
\begin{proof}
 We refer the reader to \cite[Proposition~3.4.2(ii) and Section~3.3.4]{perrinriou95} and  \cite{Colmez2000} for the proofs of the results \eqref{item_1X}--\eqref{item_3X} recorded below. 
\begin{itemize}
\item[\mylabel{item_1X}{\bf 1})] $H^1_0(V_{f^*}(k/2))$ and $H^1_{u}(V_f(k/2)):=H^1_{\Iw}(V_f(k/2))_\Gamma\cap 
H^1_{\mathrm{f}}(V_f(k/2))$ are orthogonal to one another under the pairing $\left <\,,\,\right >_\beta$. Therefore,
\begin{equation}
\label{eqn:decomposition p-adic height}
\left <H^1_{\mathrm{f}}(T_{f^*}(k/2))_{\tf},H^1_{\mathrm{f}}(T_{f}(k/2))_{\tf}\right >_\beta= 
\left <H^1_0(T_{f^*}(k/2))_{\tf}, 
\frac{H^1_{\mathrm{f}}(T_{f}(k/2))_\tf}
{H^1_{u}(T_f(k/2))_\tf} \right >_\beta 
\left <\frac{H^1_{\mathrm{f}}(T_{f^*}(k/2))_\tf}{H^1_0(T_{f^*}(k/2))_\tf}, 
H^1_{u}(T_f(k/2))_\tf \right >_\beta .
\end{equation}

\item[\mylabel{item_2X}{\bf 2})] The Iwasawa-theoretic version of the Poitou--Tate exact sequence
provides us with a map
\[
H^1(\QQ (\zeta_{p^\infty}), V_{f^*}(k/2))^* \lra
\Sha_\Iw (V_f(k/2))\,.
\]
Since $H^0(V_{f^*}(\frac{k}{2}))=0,$ we have $H^1(V_{f^*}(\frac{k}{2}))\simeq
H^1(\QQ (\zeta_{p^\infty}), V_{f^*}(\frac{k}{2}))^\Gamma.$
This gives rise to a commutative diagram
\begin{equation}
\label{additional diagram descent}
\begin{aligned}
\xymatrix{
H^1(\QQ (\zeta_{p^\infty}), V_{f^*}(k/2))^*_{\Gamma} \ar[r]^-{\sim}
\ar[d]  &H^1(V_{f^*}(k/2))^*\ar[d]\\
\Sha_\Iw (V_f(k/2))_{\Gamma} \ar[r]_{\sim} &H^1_0(V_{f^*}(k/2))^*,
}
\end{aligned}
\end{equation}
where the vertical morphisms are surjective and the horizontal ones are 
isomorphisms.

\item[\mylabel{item_3X}{\bf 3})] Consider the composition of morphisms
\[
\mu \,:\, H^1_{\mathrm{f}}(V_{f}(k/2)) \lra \Sha_\Iw (V_{f}(k/2))^\Gamma
\lra \Sha_\Iw (V_{f}(k/2))_\Gamma
\lra H^1_{0}(V_{f^*}(k/2))^*,
\]
where the first map is provided by Lemma~\ref{lemma:exact sequence with H_f,p},
the middle map is the natural projection and the 
last map is provided by diagram \eqref{additional diagram descent}. Then
$\mu$  
coincides with the pairing induced by $\log^{-1} \chi (\gamma)\left <\,\cdot\,,
\,\cdot \,\right >_\beta$. 
\end{itemize}

By Proposition~\ref{prop_22_04_26_1150}, we have
\[
H^1_\alpha (V_f(k/2))=H^1_0(V_f(k/2))\,,
\]
and by Proposition~\ref{prop: consequences of higher PR}, 
\[
H^1_\Iw (V_f(k/2))_\Gamma \cap H^1_0 (V_f(k/2))=\{0\}\,.
\]
We note that the proofs of these portions of Proposition~\ref{prop_22_04_26_1150} and Proposition~\ref{prop: consequences of higher PR} do not rely on the hypothesis \eqref{item_slope_zero_ht}. Therefore, Lemma~\ref{lemma:exact sequence with H_f,p} gives rise to an isomorphism
\[
H^1_0 (V_f(k/2)) \simeq \Sha_\Iw (V_f(k/2))^\Gamma\,.
\]

Consider the map \eqref{eqn: the map fstar} for $j=\frac{k}{2}.$
Since
\[
\res_p\left(H^1_\Iw (V_f(k/2))_\Gamma \right )\cap H^1_\alpha (\Qp,V_f(k/2))=\{0\},
\]
we infer that $\ker ({g}_\alpha)_\Gamma\otimes \Qp$ and $\mathrm{coker} ({g}_\alpha)_\Gamma\otimes \Qp$ vanish. 
Tensoring the exact sequence \eqref{eqn: first sequence for R-one-Iw} with $\Qp$
and applying the snake lemma, we obtain an exact sequence 
\begin{multline}
\nonumber
0\lra \mathrm{coker} ({g}_\alpha)^\Gamma\otimes \Qp \lra 
\bR^1\boldsymbol{\Gamma}^{\mathrm{imp}}(V_f(k/2),\alpha)\lra \Sha_\Iw (V_f(k/2))^\Gamma
\\
\lra
\mathrm{coker} ({g}_\alpha)_\Gamma\otimes \Qp \lra 
\bR^2\boldsymbol{\Gamma}^{\mathrm{imp}}(V_f(k/2),\alpha) \lra \Sha_\Iw (V_f(k/2))_\Gamma
\lra 0\,.
\end{multline}
Hence, 
\[
\begin{aligned}
&\bR^1\boldsymbol{\Gamma}^{\mathrm{imp}}(V_f(k/2),\alpha)\simeq \Sha_\Iw (V_f(k/2))^\Gamma \simeq H^1_0(V_f(k/2)),\\
&\bR^2\boldsymbol{\Gamma}^{\mathrm{imp}}(V_f(k/2),\alpha)\simeq \Sha_\Iw (V_f(k/2))_\Gamma \simeq H^1_0(V_{f^*}(k/2))^*.
\end{aligned}
\]
A diagram chase shows that the diagram
\[
\xymatrix{
\bR^1\boldsymbol{\Gamma}^{\mathrm{imp}}(V_f(k/2),\alpha)\ar[r]^{\sim} \ar[d]_{\mathrm{Bock}_2}
&\Sha_\Iw (V_f(k/2))^\Gamma \ar[d] \\
\bR^2\boldsymbol{\Gamma}^{\mathrm{imp}}(V_f(k/2),\alpha) \ar[r]^{\sim} &\Sha_\Iw (V_f(k/2))_\Gamma
}
\]
commutes. This fact, combined with \eqref{item_1X}--\eqref{item_3X}, concludes the proof of the portion of the theorem concerning the complex $\RG_\Iw^{\mathrm {imp}}(V_f(k/2),\alpha).$

We now turn our attention to the case of $\RG_\Iw(V_f(k/2),\beta).$ Recall that we have
\be\label{eqn_2023_08_28_1227}
H^1_\beta (V_f(k/2))=H^1_{\mathrm{f}} (V_f(k/2))=H^1_{\{p\}} (V_f(k/2))
\ee
by  Proposition~\ref{prop: consequences of higher PR}; we remark once again that the proof of this portion of Proposition~\ref{prop: consequences of higher PR} does not rely on the hypothesis \eqref{item_slope_zero_ht}. Since we have $H^1_\Iw (V_f(k/2))_\Gamma \subset H^1_{\{p\}}(V_f(k/2))$ by
Lemma~\ref{lemma:exact sequence with H_f,p}, it follows from \eqref{eqn_2023_08_28_1227} that
\begin{equation}
\label{eqn: equality for universal norms}
H^1_u(V_f(k/2))= H^1_\Iw (V_f(k/2))_\Gamma.
\end{equation}
Moreover, we have a canonical decomposition
\begin{equation}
\label{eqn: decomposition H^1 at central point case}
H^1_{\{p\}} (V_f(k/2))=H^1_\Iw (V_f(k/2))_\Gamma\oplus H^1_0(V_f(k/2))\,.
\end{equation}
 We therefore have canonical isomorphisms
\[
\mathrm{coker} (g_\beta)^\Gamma\otimes \Qp\simeq H^1_\Iw (V_f(k/2))_\Gamma\,,
\qquad \qquad 
\mathrm{coker} (g_\beta)_\Gamma\otimes \Qp\simeq 
\frac{H^1(\Qp,V_f(k/2))}{H^1_{\mathrm{f}}(\Qp,V_f(k/2))}.
\]
The Selmer complex $\RG (V_f(k/2),\beta)$ is associated to the diagram
\[
\xymatrix{
\RG_S(V_f(k/2)) \ar[r] & \underset{\ell\in S}\oplus \RG (\QQ_\ell, V_f(k/2))\\
&   \underset{\ell\in S}\oplus \RG (\QQ_\ell, V_f(k/2),\beta ),  \ar[u]
}
\]
where $\RG (\QQ_\ell, V_f(k/2),\beta ):=\RG_{\mathrm{f}}(\QQ_\ell, V_f(k/2))$
for $\ell \neq p$ and the local condition at $p$ is given by 
$$\RG (\Qp, V_f(k/2),\beta ):=D^{(\beta)}[k/2]$$ 
together with the map
\[
D^{(\beta)}[k/2][-1] \xrightarrow{\bExp_{\beta, k/2,0}} \RG (\Qp, V_f(k/2))\,.
\]
The map $\bExp_{\beta, k/2,0}$ induces 
$$\Exp_{\beta, k/2,0}\,:\,D^{(\beta)}[k/2]\lra H^1(\Qp, V_f(k/2))\,,$$ 
and the image of the map $\Exp_{\beta, k/2,0}$ is the Bloch--Kato subgroup $H^1_{\mathrm{f}}(\Qp, V_f(k/2))$. Therefore,
\[
\bR^1\boldsymbol{\Gamma}(V_f(k/2),\beta) \simeq H^1_{\mathrm{f}}(V_f(k/2))\,.
\]
It is not difficult to see that we have a commutative diagram
\[
\xymatrix{
0 \ar[r] &\mathrm{coker} (g_\beta)^\Gamma\otimes\Qp  \ar[r]\ar[d]^{\sim} &\bR^1 (V_f(k/2),\beta) \ar[r] \ar[d]^{\sim} \ar[r]
& \Sha_\Iw (V_f(k/2))^\Gamma \ar[r] \ar[d]^{\sim} &0\\
0 \ar[r] &H^1_\Iw (V_f(k/2))_\Gamma \ar[r] &H^1_{\{p\}}(V_f(k/2))
\ar[r] &H^1_0(V_f(k/2)) \ar[r] &0\,,
}
\]
where the bottom row is induced by the decomposition \eqref{eqn: decomposition H^1 at central point case} and all vertical maps are isomorphisms. Since the bottom row is exact, then so is the upper row. Therefore, we have the following commutative diagram
with exact rows:
\[
\xymatrix{
0 \ar[r] &\mathrm{coker} (g_\beta)^\Gamma\otimes \Qp  \ar[r]\ar[d] &\bR^1 (V_f(k/2),\beta) 
\ar[d]^{\mathrm{Bock_2}}\ar[r]
& \Sha_\Iw (V_f(k/2))^\Gamma \ar[r] \ar[d] &0\\
0 \ar[r] &\mathrm{coker} (g_\beta)_\Gamma \otimes \Qp  \ar[r] &\bR^2 (V_f(k/2),\beta) \ar[r]  \ar[r]
& \Sha_\Iw (V_f(k/2))_\Gamma \ar[r] &0.
}
\] 
Note that the right-most vertical map in this diagram admits an interpretation in terms of the $p$-adic height pairing, cf. \eqref{item_3X}. The left-most vertical map fits in the diagram
\[
\xymatrix{
\mathrm{coker} (g_\beta)^\Gamma \otimes\Qp \ar[d] \ar[r]^{\sim} & H^1_\Iw(V_f(k/2))_\Gamma \ar[d]^-{
\Upsilon}\\
\mathrm{coker} (g_\beta)_\Gamma \otimes \Qp \ar[r]^{\sim} &
\frac{H^1(\Qp,V_f(k/2))}{H^1_{\mathrm{f}}(\Qp,V_f(k/2))},
}
\]
where $\Upsilon$ is the unique map that ensures that this diagram commutes. We have the following interpretation of the right vertical map in this diagram 
in terms of $p$-adic height pairing, cf. \cite{perrinriou95}: The diagram
\[
\xymatrix{
\displaystyle\frac{H^1_{\mathrm{f}}(V_{f^*}(k/2))}{H^1_{0}(V_{f^*}(k/2))} \times 
H^1_\Iw(V_f(k/2))_\Gamma \ar[rr]^-{\left <\,,\,\right>_\beta} 
\ar[d]^{(\res_p,{ \Upsilon})}& &E\ar@{=}[d]\\
H^1_{\mathrm{f}}(\Qp, V_{f^*}(k/2)) \times \displaystyle\frac{H^1(\Qp,V_f(k/2))}{H^1_{\mathrm{f}}(\Qp,V_f(k/2))} \ar[rr]_-{\textrm{local duality}} & &E
}
\]
commutes.

Putting these facts together, we see that $\mathrm{Bock}_2$ is an isomorphism
if and only if the pairing $\left <\,,\,\right>_\beta$ is non-degenerate. 
This completes the proof of the portion of the theorem concerning the complex $\RG_\Iw (V_f(k/2),\beta)$.
\end{proof}

\subsection{Birch and Swinnerton-Dyer-style formulae}
\label{subsec_2022_08_09_1622}
In \S\ref{subsec_2022_08_09_1622}, we prove $p$-adic Birch and Swinnerton-Dyer style formulae for our algebraic $p$-adic $L$-functions. These results will not be used in the remainder of this paper, but we include them here for the sake of completeness. We remark that they were proved at a great level of generality in \cite{perrinriou95} in the non-$\theta$-critical case (see also \cite{benoisextracris}), and the proofs in the $\theta$-critical scenario (which we present below) are very similar granted the input from the earlier portions of \S\ref{sec_Iwasawa_theoretic_descent}.

\subsubsection{}
As before, let us set 
\[
e_{p,\alpha}(f,\mathds{1},j)= \left (1-\alpha^{-1}p^{j-1}\right )\cdot 
\left (1-\beta p^{-j}\right )\,,
\qquad 
e_{p,\beta}(f,\mathds{1},j)    =\left (1-\beta^{-1}p^{j-1}\right )\cdot 
\left (1-\alpha p^{-j}\right ).
\]
We denote by 
$$\Tam^0(T_f(j)):=\underset{\ell\in S}\prod \Tam_{\ell}(T_f(j))$$ 
the product of local Tamagawa numbers of $V_f(j)$ and by 
$$\Tam_{\{p\}}^0(T_f(j)):=\underset{\ell\in S\setminus\{p\}}\prod \Tam_{\ell}(T_f(j))$$ 
the product of local Tamagawa numbers away from $p$.
Recall that we have set $\Omega_p(T_f^{(\star)}, N_\star ):=\vert b\vert, $ where $b\in E^*$ is such that $\Dc(T_f^{(\star)})=bN_\star$. Throughout \S\ref{subsec_2022_08_09_1622}, we will assume that $\Omega_p(T_f^{(\star)}, N_\star )\in \cO_E^*$ for $\star\in \{\alpha,\beta\}.$

\begin{theorem}
\label{thm:descent thm: noncentral case}
\item[i)]{} Suppose $j\geqslant k$.   
Assume that the conditions \eqref{item_pB} and \eqref{item_S} hold.
Then $r_\alpha (j)=0$ and 
\[
\left \vert \mathbf L_{\Iw,\alpha}(T_f(j),N_\alpha[j],\mathds{1},0)\right \vert^{-1}=\left \vert (j-k)!\cdot 
e_{p,\alpha}(f,\mathds{1},j)
\right \vert^{-1} \cdot \frac{\# \Sha_{\rm f}(T_f(j))\cdot R_\alpha (T_f(j)) \cdot \Tam^0 (T_f(j)) }
{\# H^0(V_f/T_f(j)) \cdot \# H^0(V_{f^*}/T_{f^*}(k-j))}.
\]

\item[ii)]{} Suppose  $1\leqslant j\leqslant k-1$  and $j\neq \frac{k}{2}$. Assume that the condition\eqref{item_S} holds. Then   $r_\alpha (j)=0$ and 
\[
\nonumber
\left \vert \mathbf L_{\Iw,\alpha}^{\mathrm{imp}}(T_f(j), N_\alpha[j],\mathds{1}, 0)\right \vert^{-1}
=\frac{\# \Sha_\alpha (T_f(j))\cdot \Tam^0_{\{p\}}(T_f(j))}
{\# H^0(V_f/T_f(j))\cdot \# H^0(V_{f^*}/T_{f^*}(k-j))}\cdot w (T_f^{(\beta)}(j))  \cdot w (T^{(\beta)}_{f^*}(k-j)),
\]
where $w (T_f^{(\beta)}(j))= \# H^0(\Qp, V_f^{(\beta)}/T_f^{(\beta)}(j))$ and
$w (T^{(\beta)}_{f^*}(k-j))=  \# H^0(\Qp, V^{(\beta)}_{f^*}/T^{(\beta)}_{f^*}(k-j)).$

\item iii) Suppose $j\geqslant k.$ Assume that the condition\eqref{item_pB}
holds and $H^1_\beta (V_f(j))=0.$  Then  $r_\beta(j)=0$ and 
\[
\left \vert \mathbf L_{\Iw,\beta}(T_f(j),N_\beta[j],\mathds{1},0)\right \vert^{-1}=\left \vert (j-1)!\cdot e_{p,\beta}(f,\mathds{1},j )\right \vert^{-1} \cdot \frac{\# \Sha_{\rm f}(T_f(j))\cdot R_\beta (T_f(j)) \cdot \Tam^0 (T_f(j))
}
{\# H^0(V_f/T_f(j)) \cdot \# H^0(V_{f^*}/T_{f^*}(k-j))}.
\]
\item iv) Suppose  $1\leqslant j\leqslant k-1$  and $j\neq \frac{k}{2}$. Then  $r_\beta(j)=0$ and
\[
\left \vert \mathbf L_{\Iw,\beta}(T_f(j),N_\beta[j],\mathds{1},0)\right \vert^{-1}=\left \vert (j-1)!\cdot e_{p,\beta}(f,\mathds{1},j)\right \vert^{-1} \cdot \frac{\# \Sha_{\rm f}(T_f(j)) \cdot \Tam^0 (T_f(j))}
{\# H^0(V_f/T_f(j)) \cdot \# H^0(V_{f^*}/T_{f^*}(k-j))}.
\]

\end{theorem}
 
\begin{remark} The formulae for $\mathbf L_{\Iw,\alpha}^{\mathrm{imp}}(T_f(j),N_\alpha[j],\mathds{1},0)$
and $\mathbf L_{\Iw,\beta}(T_f(j),N_\beta[j],\mathds{1},0)$ in Theorem~\ref{thm:descent thm: noncentral case}
are concurrent with Propositions~\ref{prop: Tate-Shafarevich and L-invariant} and \ref{prop: regulators and L-invariant}, in view of \ref{item_MCalpha}, \ref{item_MCbeta} and Proposition~\ref{prop: comparision p-adic L-functions for alpha and beta}.
Note that, we can express $\left \vert \mathbf L_{\Iw,\alpha}(T_f(j),N_\alpha[j],\mathds{1},0)\right \vert^{-1}$ also in the form
\[
\left \vert (j-k)!\cdot 
e_{p,\alpha}(f,\mathds{1},j) \cdot \mathscr{L}^{\mathrm{cr}}(V_f(j))
\right \vert^{-1} \cdot \frac{\# \Sha_{\rm f}(T_f(j))\cdot R_\beta (T_f(j)) \cdot \Tam^0 (T_f(j)) }
{\# H^0(V_f/T_f(j)) \cdot \# H^0(V_{f^*}/T_{f^*}(k-j))} 
\]
in the situation of Theorem~\ref{thm:descent thm: noncentral case}(i), and
\[
\left \vert \mathscr{L}^{\mathrm{cr}}(V_f(j))
\right \vert^{-1}
\frac{\# \Sha_{\mathrm{f}} (T_f(j))\cdot \Tam^0_{\{p\}}(T_f(j))}
{\# H^0(V_f/T_f(j))\cdot \# H^0(V_{f^*}/T_{f^*}(k-j))}\cdot w (T_f^{(\alpha)}(j))  \cdot w (T^{(\alpha)}_{f^*}(k-j))
\]
in the situation of Theorem~\ref{thm:descent thm: noncentral case}(ii). These expressions are akin to those that appear in the $p$-adic Birch and Swinnerton-Dyer conjecture in the non-critical scenario, cf. \cite{MTT, benoisextracris}.
\end{remark}

\begin{proof}
 Applying the snake lemma to the commutative diagram 
$$\xymatrix{
0\ar[r]& \mathrm{coker} (g_\alpha) \ar[r]\ar[d]_{[\gamma-1]}& \bR^2\boldsymbol{\Gamma}^{\mathrm{imp}}_\Iw(T_f(j),\alpha) \ar[r]\ar[d]_{[\gamma-1]}& \Sha_\Iw (T_f(j))\ar[r]\ar[d]^{[\gamma-1]}  &0\\
0\ar[r]& \mathrm{coker} (g_\alpha) \ar[r]& \bR^2\boldsymbol{\Gamma}^{\mathrm{imp}}_\Iw(T_f(j),\alpha) \ar[r]& \Sha_\Iw (T_f(j))\ar[r]&0
}$$
with exact rows which we obtain from the exact sequence \eqref{eqn: first sequence for R-one-Iw}, we have an exact sequence
\begin{multline}
\nonumber
0\rightarrow \mathrm{coker} (g_\alpha)^\Gamma \rightarrow \bR^1\boldsymbol{\Gamma}^{\mathrm{imp}}(T_f(j),\alpha) \rightarrow \Sha_\Iw (T_f(j))^\Gamma \\
\rightarrow \mathrm{coker} (g_\alpha)_\Gamma \rightarrow \bR^2\boldsymbol{\Gamma}^{\mathrm{imp}}(T_f(j),\alpha) \rightarrow \Sha_\Iw (T_f(j))_\Gamma 
\rightarrow 0.
\end{multline}

By Proposition~\ref{prop: descent non-critical general case},  all terms of this sequence are finite groups under our running assumptions. Moreover, $r_\alpha(j)=0$ and 
\begin{equation}
\label{eqn: first explicit formula for L}
\left \vert \mathbf L_{\Iw,\alpha}^{\mathrm{imp}}(T_f(j),N_\alpha[j],\mathds{1}, 0)\right \vert^{-1}= 
\frac{\#\bR^2\boldsymbol{\Gamma}^{\mathrm{imp}}(T_f(j),\alpha)}{\#\bR^1\boldsymbol{\Gamma}^{\mathrm{imp}}(T_f(j),\alpha)}=
\frac{\# \mathrm{coker}({g}_\alpha)_{\Gamma}}{\#\mathrm{coker}({g}_\alpha)^{\Gamma}} \cdot 
\frac{\# \Sha_{\Iw}(T_f(j))_{\Gamma}}{\# \Sha_{\Iw}(T_f(j))^{\Gamma}}.
\end{equation}
The arguments of \cite[Section~3.3.4]{perrinriou95} (see also \cite{Colmez2000}, computations following Proposition~3.9,  pp. 34-35) 
show that 
\begin{equation}
\label{eqn: formula for the invariants/coinvariants of sha}
\frac{\# \Sha_{\Iw}(T_f(j))_{\Gamma}}{\# \Sha_{\Iw}(T_f(j))^{\Gamma}}=
\frac{\left [H^1(\Qp, T_f(j)):H^1_\Iw (\Qp, T_f(j))_\Gamma \right ]}
{\left [H^1_{\{p\}}(T_f(j)):H^1_\Iw (T_f(j))_{\Gamma}\right ]}
\cdot
\frac{{\#\Sha_0(T_{f^*}(k-j))} \cdot\Tam_{\{p\}}^0(T_f(j))}{\# 
H^0(V_{f^*}/T_{f^*}(k-j))
}\,.
\end{equation}

On the other hand, the snake lemma gives
\begin{equation}
\label{eqn: descent for f}
\frac{\# \mathrm{coker}({g}_\alpha)_{\Gamma}}{\#\mathrm{coker}({g}_\alpha)^{\Gamma}}=
\left [H^1_{\Iw}(\Qp, T_f(j))_{\Gamma}: H^1_{\Iw}(T_f(j))_{\Gamma}+ \Exp_{\alpha, j} (N_\alpha [j]\otimes \Lambda)_{\Gamma} \right ].
\end{equation}
Since $ \Exp_{\alpha, j} (N_\alpha[ j]\otimes \Lambda)_{\Gamma}$ is $\Zp$-free,
we have
\begin{align}
\label{eqn: proof of special values formula: noncritical case second formula}
\begin{aligned}
&\left [H^1(\Qp, T_f(j)):H^1_{\{p\}}(T_f(j)) + \Exp_{\alpha, j} (N_\alpha[ j]\otimes \Lambda)_{\Gamma} \right ]=\\
&\hspace{2cm}\left [H^1(\Qp, T_f(j))_\tf :H^1_{\{p\}}(T_f(j))_\tf + H^1_\star (\Qp, T_f(j))_\tf \right ]\\
&\hspace{2.5cm}\times \frac{\# H^0(\Qp, V_f/T_f(j))\left[H^1_\star (\Qp, T_f(j))_\tf : \Exp_{\star, j} (N_\star[ j]\otimes \Lambda)_{\Gamma} \right ]}{\# H^0(V_f/T_f(j))}.
\end{aligned}
\end{align}
Combining  formulae \eqref{eqn: first explicit formula for L}, 
\eqref{eqn: formula for the invariants/coinvariants of sha}
\eqref{eqn: descent for f}  and  (\ref{eqn: proof of special values formula: noncritical case second formula}) with 
Lemma~\ref{lemma:auxiliary Tate-Shafarevich}(ii), we deduce that
\begin{align}
\label{eqn: noncentral descent general step}
\begin{aligned}
&\left \vert \mathbf L_{\Iw,\alpha}^{\mathrm{imp}}(T_f(j),N_\alpha[j],\mathds{1}, 0)\right \vert^{-1}=\\
&\hspace{0.5cm}\frac{\# \Sha_\alpha (T_f(j))\cdot \# H^0(\Qp, V_f/T_f(j))\cdot \Tam^0_{\{p\}}(T_f(j))}
{\# H^0(V_f/T_f(j))\cdot \# H^0(V_{f^*}/T_{f^*}(k-j))} \cdot \left [H^1 (\Qp, T_f^{(\alpha)}(j))_\tf : \Exp_{\alpha, j} (N_\alpha[ j]\otimes \Lambda)_{\Gamma} \right ].
\end{aligned}
\end{align}
Assume first that $j\geqslant k$. It follows from the definition of the Tamagawa numbers and \eqref{eqn:specialization of PR formulae} that
\begin{align}
\label{eqn: proof of special values formula: noncritical case third formula}
\begin{aligned}
\left[H^1 (\Qp, T^{(\alpha)}_f(j))_\tf : \Exp_{\alpha, j} (N_\alpha [ j]\otimes \Lambda)_{\Gamma} \right ]&=
\vert (j-k)! \vert^{-1} \left \vert 1-\alpha^{-1}p^{j-1}   \right \vert^{-1} \frac{\Tam_{p}(T_f^{(\alpha)}(j))}{\#H^0(\Qp,V_f^{(\alpha)}/T_f^{(\alpha)}(j))}\,.
\end{aligned}
\end{align}
Putting together \eqref{eqn: noncentral descent general step} and \eqref{eqn: proof of special values formula: noncritical case third formula} with Proposition~\ref{prop: comparision of sha in noncritical case}, we infer that
\[
\left \vert \mathbf L_{\Iw,\alpha}^*(T_f(j),\mathds{1}, 0)\right \vert^{-1}=
\left \vert (j-k)! \cdot e_{p,\alpha}(f,\cdot \mathds{1},j) )
\right \vert^{-1} \frac{\# \Sha_{\rm f}(T_f(j))\cdot R_\alpha (T_f(j)) \cdot \Tam^0 (T_f(j))}
{\# H^0(V_f/T_f(j)) \cdot \# H^0(V_{f^*}/T_{f^*}(k-j))}\,.
\]
This concludes the proof of Theorem~\ref{thm:descent thm: noncentral case}(i).

Let us now assume that $1\leqslant j\leqslant k-1$ and $j\neq \frac{k}{2}.$
Since $ \Exp_{\alpha, j}$ establishes  an isomorphism between $\mathfrak D(T_f^{(\alpha)}(j))$ and $H^1_\Iw (\Qp, T_f^{(\alpha)}(j)),$ we have
\begin{equation}
\label {eqn: degenerated tamagawa number}
\left[H^1 (\Qp, T^{(\alpha)}_f(j)) : \Exp_{\alpha, j} (N_\alpha [ j]\otimes \Lambda)_{\Gamma} \right ]=\# H^0(\Qp, V^{(\beta)}_{f^*}/T^{(\beta)}_{f^*}(k-j)).
\end{equation}
Therefore,
\begin{multline}
\nonumber
\left \vert \mathbf L_{\Iw,\alpha}^*(T_f(j),\mathds{1}, 0)\right \vert^{-1}
=\\
\frac{\# \Sha_\alpha (T_f(j))\cdot \Tam^0_{\{p\}}(T_f(j))}
{\# H^0(V_f/T_f(j))\cdot \# H^0(V_{f^*}/T_{f^*}(k-j)) }\cdot  \# H^0(\Qp, V_f^{(\beta)}/T_f^{(\beta)}(j))\cdot 
\# H^0(\Qp, V^{(\beta)}_{f^*}/T^{(\beta)}_{f^*}(k-j)).
\end{multline}
This completes the proof of Theorem~\ref{thm:descent thm: noncentral case}(ii).

The assertions in (iii) and (iv)  concerning $\mathbf L_{\Iw,\beta}(T_f(j),N_\beta[j],\mathds{1}, 0)$ can be proved by the same argument (and, in any case, these statements are covered by Perrin-Riou's general descent formalism). 
\end{proof}

\subsubsection{}\label{subsubsec_2022_08_091648}
 In \S\ref{subsubsec_2022_08_091648}, we assume that $k$ is even and consider the scenario when $j=\frac{k}{2}$ is the central critical point for the Hecke $L$-function of $f$.
 
\begin{theorem} 
\label{thm_331_2022_04_29_1629}
Assume that $k\geqslant 2$ is even and that the conditions \eqref{item_BK_cc}, \eqref{item_slope_zero_ht} and \eqref{item_PR} hold true. Then: 

\item[i)]{} We have $r_\beta  (\frac{k}{2})= \dim_E H^1_{\rm f}(V_{f^*}(\frac{k}{2}))$
and 
\begin{align*}
&\left \vert \mathbf L_{\Iw,\beta}^*(T_f({k}/{2}),N_\alpha[k/2],\mathds{1}, 0)\right \vert^{-1}
\\
&\qquad\qquad\qquad\qquad=\vert ({k}/{2}-1)! \cdot e_{p,\beta}(f, \mathds{1},k/2))\vert^{-1}\cdot 
\frac{\# \Sha_{\rm f} (T_{f^*}({k}/{2})) \cdot R_\beta (T_f({k}/{2})) \cdot \Tam^0 (T_f({k}/{2}))}{\#H^0(V_{f}/T_f({k}/{2})) \cdot \#H^0(V_{f^*}/T_{f^*}({k}/{2}))}. 
\end{align*}

\item[ii)]{} 
We have  $r_\alpha  (\frac{k}{2})=\dim_E H^1_0 (V_{f^*}(\frac{k}{2})).$
The $p$-adic height pairing  
\[
\left < \,,\,\right >_\alpha :\quad H^1_0(V_{f^*}(k/2))\otimes H^1_0(V_{f}(k/2))
\lra E 
\]
is non-degenerate and coincides with the restriction of $\left < \,,\,\right >_\beta$
to $H^1_0(V_{f}(k/2))$. Moreover,
\begin{multline}
\nonumber
\left \vert \mathbf L_{\Iw,\alpha}^*(T_f({k}/{2}),N_\alpha[k/2],\mathds{1}, 0)\right \vert^{-1}\\
=\frac{\# \Sha_\alpha (T_f({k}/{2}))\cdot R_\alpha (T_f({k}/{2}))\cdot \Tam^0_{\{p\}}(T_f({k}/{2}))}
{\# H^0(V_f/T_f({k}/{2}))\cdot \#H^0(V_{f^*}/T_{f^*}({k}/{2}))} \cdot w(T_f^{(\beta)}({k}/{2})) \cdot  w(T^{(\beta)}_{f^*}({k}/{2})),
\end{multline}
where 
$$w(T_f^{(\beta)}({k}/{2})):= \# H^0(\Qp, V_f^{(\beta)}/T_f^{(\beta)}({k}/{2}))\quad \hbox{ and } \quad w(T^{(\beta)}_{f^*}({k}/{2})):= \# H^0(\Qp, V^{(\beta)}_{f^*}/T^{(\beta)}_{f^*}({k}/{2})).$$
\end{theorem}

\begin{proof}
\item[i)] This portion is covered by the general descent formalism of Perrin-Riou (which our argument below to prove (ii) parallels), and its proof will be omitted.

\item[ii)] We already proved as part of Theorem~\ref{thm: bockstein map in central critical case} all the assertions except for the leading term formula. 

It follows from Lemma~\ref{lemma Burns-Greither} and our explicit description of the Bockstein map in Theorem~\ref{thm: bockstein map in central critical case}(iii) that
\begin{equation}
\label{eqn:beginning central descent}
\left \vert \mathbf L_{\Iw,\alpha}^*(T_f({k}/{2}),N_\alpha[k/2], \mathds{1}, 0)\right \vert^{-1}
=\left [\Sha_\Iw (T_f({k}/{2}))_\Gamma : \Sha_\Iw (T_f({k}/{2}))^\Gamma \right ] \cdot \left  [\# \mathrm{coker} (g_\alpha)_\Gamma :\# \mathrm{coker} (g_\alpha)^\Gamma \right ].
\end{equation}

Recall that we have $H^1_u(T_f(k/2))= H^1_\Iw (T_f(k/2))_\Gamma$ under our assumptions; cf. \eqref{eqn: equality for universal norms}. We use the following computation from  \cite[Section~3.3.4]{perrinriou95} (see also  \cite{Colmez2000}, the discussion following Proposition~3.5):
\begin{align}
    \label{eqn:formula 2 descent central case}
    \begin{aligned}
    &\left [\Sha_\Iw (T_f({k}/{2}))_\Gamma : \Sha_\Iw (T_f({k}/{2}))^\Gamma \right ]=\\
&\hspace{.7cm}
(\log \chi (\gamma_1))^{-r_\alpha (k/2)} \cdot 
\frac{\# \Sha_0(T_{f^*}({k}/{2}))\Tam^0_{\{p\}}(T_f({k}/{2})) 
\left [H^1(\Qp,T_f({k}/{2})) : H^1_\Iw (\Qp, T_f({k}/{2}))_{\Gamma}\right ]}{\# H^0(V_{f^*}/T_{f^*}({k}/{2}))}
\\
&\hspace{8cm}\times
\left \vert\left <H^1_0(T_{f^*}({k}/{2}))_\tf , 
\frac{H^1_{\{p\}} (T_f ({k}/{2}))}{H^1(T_f(k/2))^{\rm u}}\right >_{\alpha} \right \vert^{-1}.
    \end{aligned}
\end{align}
Note that we have used here the fact that $H^1_{\{p\}}(T_f(\frac{k}{2}))=H^1_{\rm f}(T_f(\frac{k}{2}))$ under our running assumptions. Let us write 
\begin{multline}
\label{eqn:descent central change of lattice}
\left  <H^1_0(T_{f^*}({k}/{2}))_\tf , 
\frac{H^1_{\{p\}} (T_f ({k}/{2}))}{H^1(T_f(k/2))^{\rm u}}\right >_{\alpha}
=\\
\frac{\left  <H^1_0(T_{f^*}({k}/{2}))_\tf , H^1_0(T_{f}({k}/{2}))_\tf
\right >_\alpha }
{\left [H^1_{\{p\}} (T_f ({k}/{2}))_{\tf}: H^1(T_f(k/2))^{\rm u}_\tf+
H^1_0(T_f(k/2))_\tf \right ]} \cdot
\frac{\#H^1(T_f(k/2))^{\rm u}_{\mathrm{tor}}}{
\#H^0(V_f/T_f(k/2))
}\,.
\end{multline}
The formula \eqref{eqn: descent for f} applies  in our case, and
mimicking the computations from the proof of Theorem~\ref{thm:descent thm: noncentral case}, we deduce that
\begin{align}
\label{eqn: descent central critical second term}
\begin{aligned}
\frac{\# \mathrm{coker} (g_\alpha)_\Gamma}{\# \mathrm{coker} (g_\alpha)^\Gamma} \cdot& \left [H^1(\Qp,T_f({k}/{2})) : H^1_\Iw (\Qp, T_f({k}/{2}))_{\Gamma}\right ]= w (T_f^{(\beta)}(k/2))  w (T_{f^*}^{(\beta)}(k/2)) \\
&\hspace{3.75cm}\times \frac{\left [ H^1(\Qp, T_f(k/2))_\tf : H^1(T_f(k/2))^{\rm u}_\tf +
H^1_\alpha (\Qp, T_f(k/2))_\tf \right ]}{\# H^1(T_f(k/2))^{\rm u}_{\mathrm{tor}}}
\,.
\end{aligned}
\end{align}
It is easy to check that
\begin{align}
\label{eqn: descent in central case simplification formula}
\begin{aligned}
&\frac{\left [ H^1(\Qp, T_f(k/2))_\tf : H^1(T_f(k/2))^{\rm u}_\tf +
H^1_\alpha (\Qp, T_f(k/2))_\tf \right ]}
{\left [H^1_{\{p\}} (T_f ({k}/{2}))_{\tf}: H^1(T_f(k/2))^{\rm u}_\tf+
H^1_0(T_f(k/2))_\tf \right ]}=\\
&\hspace{4cm}\left [ H^1(\Qp, T_f(k/2))_\tf : 
H^1_\alpha (\Qp, T_f(k/2))_\tf +\frac{H^1_{\{p\}} (T_f ({k}/{2}))_{\tf}}
{H^1_0(T_f(k/2))_\tf} \right ]\,.
\end{aligned}
\end{align}
The formulae \eqref{eqn:formula 2 descent central case}--\eqref{eqn: descent in central case simplification formula} together with  
Proposition~\ref{eqn: formula for shafarevich groups in the central case}
show that
\begin{multline}
\nonumber
\left [\Sha_\Iw (T_f({k}/{2}))_\Gamma : \Sha_\Iw (T_f({k}/{2}))^\Gamma \right ] \cdot \left  [\# \mathrm{coker} (g_\alpha)_\Gamma :\# \mathrm{coker} (g_\alpha)^\Gamma \right ]
\\\qquad\quad\sim_p (\log \chi (\gamma_1))^{-r_\alpha (k/2)} \times\frac{\Sha_\alpha (T_{f^*}(k/2)) \cdot R_\alpha (T_f(k/2)) \cdot \Tam_{\{p\}}^0(T_f(k/2))}{\# H^0(V_f/T_f({k}/{2}))\cdot \#H^0(V_{f^*}/T_{f^*}({k}/{2}))} \cdot w(T_f^{(\beta)}({k}/{2})) \cdot  w(T^{(\beta)}_{f^*}({k}/{2})).
\end{multline} 
Part (ii) of our theorem follows from this computation and  \eqref{eqn:beginning central descent}. 
\end{proof}

\chapter{Main Conjectures for the infinitesimal deformation}
\label{chapter_main_conj_infinitesimal_deformation}
In this chapter, we develop Iwasawa theory for the infinitesimal deformation $V_{k}$ of Deligne's representation $V_{f}$ along the eigencurve. This involves the definition of the \emph{thick} fundamental line in \S\ref{sec_fundamental_line_20220505} in this context, which we then use to define the infinitesimal thickening of the module of algebraic $p$-adic $L$-functions; cf. Equation~\eqref{eqn_trivialization_thick_module_of_padic_L}. 

The infinitesimal thickening \ref{item_MCinf} of the main conjecture asserts that the infinitesimal thickening of the module of algebraic $p$-adic $L$-functions is generated by the thick $p$-adic $L$-function $\widetilde L_{\mathrm S,\alpha^*}^\pm (f^*, \xi^*)^\iota$ up to a canonical fudge factor. 

To prove the leading term formulae for the infinitesimal thickening of the module of algebraic $p$-adic $L$-functions (cf. Theorem~\ref{thm: descent for tilt complex}), we rely on the Iwasawa descent formalism for the infinitesimal thickening of the module of algebraic $p$-adic $L$-functions, which we develop in \S\ref{subsec_425_2022_08_19_1548}. Our descent formalism dwells on the general discussion in \S\ref{subsec_423_2022_0810_0849} and the integrality properties we establish in the opening section of this chapter (cf. Lemma~\ref{lemma_thick_lattice}). 

 As a further motivation for the contents of the present chapter, we recall from  Chapter~\ref{chapter_main_conjectures} the degenerate behaviour of the Selmer complex $\RG_\Iw( T_{f},\alpha)$  when $f_\alpha$ is $\theta$-critical.
The key insight in our study, as we shall illustrate in the present chapter as well as the next, is that the Selmer complex attached to the infinitesimal deformation $V_{k}$ (which we call the thick Selmer complex) encodes strictly more information. This is illustrated by Theorems~\ref{thm: descent for tilt complex} and~\ref{thm_414_2022_0810_1712}, which tell us that the thick Selmer complex ``sees''
also the special values of the secondary $p$-adic $L$-function $L^{[1]}_{\mathrm{K},\alpha^*}(f^*, \xi^*)^\pm$.

 Recall that we have $V_f \simeq V_{x_0}$, where $x_0$ is the critical point on the eigencurve which corresponds to $f_\alpha$. We  assume throughout this chapter that the ramification index of the weight map $\cX  \rightarrow \cW$ at $x_0$ is equal to $e=2$. Recall the tautological exact sequence
\begin{equation}
    \label{eqn_4_1_2022_03_16}
    0\lra V_f \lra V_{k}\lra  V_{f}\lra 0,
\end{equation}
where the copy of $V_f$ on the left arises from its identification with $V[x_0]$.

We note that $ V_{k}$ can be viewed as a Galois representation with coefficients in the ring $\widetilde E:=E[X]/(X^2)$. One may also forget the Hecke-module structure and consider $ V_{k}$ as a $4$-dimensional Galois representation over $E$. 

\section{The fundamental line}
 \label{sec_fundamental_line_20220505}
 \subsection{Lattices}  We recall various definitions from \S\ref{subsubsec_4_5_1_1_16_03_2022} for the convenience of the reader. Let us put $\widetilde E:= E[X]/X^2$. For each $n\geqslant 0,$ recall the unital $O_E$-subalgebra 
\[
\cO_E^{(n)}= O_E+\varpi_E^{-n} O_E X \,\subset E[X]/X^2
\]
of $\widetilde E$, where $\varpi_E$ is a uniformizer  of $E$.
 Set
\[
\cO_E^{(\infty)}:=\underset{n\geqslant 1}\cup \cO_E^{(n)}.
\]

Let $T_{f}$ be  an  $\cO_E$-lattice in $V_{f}$ which is stable under the action of $G_{\QQ,S}.$

\begin{lemma}  The following assertions hold true:
\label{lemma_thick_lattice}
\item[i)]{}
 For any  $n\gg 0$ there exists a free $\cO_E^{(n)}$-submodule $ T_{k}^{(n)}\subset V_{k}$ of rank two such that 
\begin{itemize}
    \item[\mylabel{item_L1}{\bf L1})] $T_{k}^{(n)}$ is stable under the action of $G_{\QQ,S}$\,,
\item[\mylabel{item_L2}{\bf L2})] $T_{k}^{(n)}\otimes_{\cO_E^{(n)}}\cO_E=T_{f}$.
\end{itemize}

\item[ii)]{} Suppose that the natural number $n$ verifies the conclusions of Part (i). Let $m$ be another natural number such that $T_{k}^{(m)}$ is a free $\cO_E^{(m)}$-submodule of $V_k$ satisfying the properties \eqref{item_L1}--\eqref{item_L2} with $n$ replaced by $m$. Then there exists $s\geqslant \max \{m,n\}$ such that 
\[
T_{k}^{(n)}\otimes_{\cO_E^{(n)}}\cO^{(s)}_E  =T_{k}^{(m)}\otimes_{\cO_E^{(m)}}\cO^{(s)}_E.
\]

\end{lemma}
\begin{proof} 
\item[i)] Let $\{v_1, v_2\}$ be a basis of $T_{f}.$ Fix arbitrary lifts
$\widetilde v_1, \widetilde v_2\in V_{k}$ of $v_1, v_2$ under the surjection $ V_{k}\twoheadrightarrow V_{f}$ and set $\widetilde T=\cO_E^{(0)}\widetilde v_1+ \cO_E^{(0)}\widetilde v_2$.    Note that we have 
$$g\cdot \widetilde v_i=a_{i,1}(g)\widetilde v_1+ a_{i,2}(g)\widetilde v_2+X(b_{i,1}(g)\widetilde v_1+ b_{i,2}(g)\widetilde v_2)$$ 
with $a_{i,j}(g)\in \cO_E$ for any $g\in G_{\QQ,S}$, since $\widetilde{v}_i\mapsto v_i$ under the $G_{\QQ,S}$-equivariant morphism $ V_{k}\twoheadrightarrow V_{f}$. This discussion shows that 
$$g\cdot \widetilde v_i\in \bigcup_{m\geq 0} (\widetilde T +\varpi_E^{-m}X\widetilde T)\,.$$
It follows from the compactness of the Galois group that there exists $n\geqslant 0$ such that 
\[
g\cdot  \widetilde v_i \in \widetilde T +{\varpi_E^{-n}}X\widetilde T, \qquad \forall g\in G_{\QQ,S},
\quad i=1,2. 
\]
Set $ T_{k}^{(n)}= \cO_E^{(n)}  \widetilde v_1+ \cO_E^{(n)}\widetilde v_2$. Thanks to our  choice of $n$,
the ${\cO}_{E}^{(n)}$-module $ T_{k}^{(n)}$ is $G_{\QQ,S}$-stable. Furthermore, we have $ T_{k}^{(n)}\otimes_{\cO_E^{(n)}}O_E=T_{f}$ by construction and our proof is complete.

 It is clear that   for each $m\geq n$, the $\ \cO_E^{(m)}$-module $ T_{k}^{(n)}\otimes_{ \cO_E^{(n)}}\cO_E^{(m)}$ verifies the properties  \eqref{item_L1}--\eqref{item_L2}. This proves the first part of the lemma.

\item[ii)] It follows from the compactness that there exists $s\geqslant \max \{m,n\}$
such that 
\[
\textrm{
$T_k^{(n)}\subset T_k^{(m)}\otimes_{\cO_E^{(m)}}\cO_E^{(s)}$ \quad and  \quad 
$T_k^{(m)}\subset T_k^{(n)}\otimes_{\cO_E^{(n)}}\cO_E^{(s)}.$
}
\]
Then 
\[
T_k^{(n)}\otimes_{\cO_E^{(n)}}\cO_E^{(s)} \subset  T_k^{(m)}\otimes_{\cO_E^{(m)}}\cO_E^{(s)}
\subset T_k^{(n)}\otimes_{\cO_E^{(n)}}\cO_E^{(s)},
\]
and the proposition is proved.
\end{proof} 

Recall that $\LL^{(n)}:= \cO_E^{(n)}\otimes_{\cO_E}\LL$ and let us set $\widetilde\CH (\Gamma)=\widetilde E \otimes_E\CH (\Gamma)$. We have decompositions of these algebras into $\Delta$-isotypic components: 
\[
\LL^{(n)}=\underset{i=1}{\overset{p-1}\oplus} \LL^{(n)}_{\omega^i},
\qquad 
\widetilde\CH (\Gamma)=\underset{i=1}{\overset{p-1}\oplus} \widetilde\CH_{\omega^i} (\Gamma).
\]
Let us put 
\[
\LL^{(\infty)}:= \cO_E^{(\infty)}\otimes_{\cO_E}\LL= \varinjlim_n \LL^{(n)}.
\]

\subsection{Complexes}  
\label{subsec_complexes_412}
\subsubsection{} Let us fix a natural number $n$ for which the properties \eqref{item_L1}--\eqref{item_L2} are verified. Recall that  $\iota$ denotes  the canonical involution on $\Lambda.$ 
We consider the global and local  Iwasawa cohomology complexes 
\begin{equation}
\label{eqn:complexes of Iwasawa cohomology}
\RG_{{\Iw},S} ( T_{k}^{(n)}(j))\,:=\RG(G_{\QQ,S},  T_{k}^{(n)}(j)\otimes_{\cO_E} \LL^\iota)\,, \qquad 
\RG_{\Iw}(\QQ_\ell, T_{k}^{(n)}(j))\,:=\,\RG ( G_{\QQ_\ell},  T_{k}^{(n)}(j)\otimes_{\cO_E} \LL^\iota)
\end{equation}
as objects of the derived category $\mathscr D(\LL^{ (n)})$ of  $\LL^{(n)}$-modules.
One then has
\[
\RG_{{\Iw},S} ( T_{k}^{(n)}(j))\otimes_{\LL^{(n)}}^{\mathbf L}\LL 
\simeq \RG_{{\Iw},S} (T_{f}(j))\,,
\qquad 
\RG_{{\Iw}} (\QQ_\ell,  T_{k}^{(n)}(j))\otimes_{\LL^{(n)}}^{\mathbf L}\LL 
\simeq \RG_{{\Iw}} (\QQ_\ell, T_{f}(j)).
\]

\subsubsection{}
Since the action of $G_{\QQ,S}$ (respectively $G_{\QQ_\ell}$) on $V_k$  factors through a quotient containing a pro-$p$-group of finite index,
the following properties follow from the general results of Flach \cite{flach}
(see especially Propositions 4.1, 4.2, 5.2 and Remark 3.1 in op. cit.):
\begin{itemize}
\item[$\bullet$]{} The complexes $\RG_{S} ( T_{k}^{(n)}(j))$
and $\RG (\QQ_\ell,  T_{k}^{(n)}(j))$ ($\ell \in S$) are perfect objects 
in the derived category $\mathscr D(\cO_E^{(n)})$ of $\cO_E^{(n)}$-modules. 
The complexes $\RG_{{\Iw},S} ( T_{k}^{(n)}(j))$ and $\RG_{{\Iw}} (\QQ_\ell,  T_{k}^{(n)}(j))$ are perfect objects 
in the derived category $\mathscr D(\LL^{(n)})$ of $\LL^{(n)}$-modules. 

\item[$\bullet$]{} For $m\geqslant n,$ we have 
\begin{equation}
\nonumber
\begin{aligned}
&\RG_{S} ( T_{k}^{(n)}(j))\otimes_{\cO_E^{(n)}}^{\mathbf L}\cO_E^{(m)}
\simeq \RG_{S} ( T_{k}^{(n)}(j)\otimes_{\cO_E^{(n)}}\cO_E^{(m)} ),\\
&\RG(\QQ_\ell, T_{k}^{(n)}(j))\otimes_{\cO_E^{(n)}}^{\mathbf L}\cO_E^{(m)}
\simeq \RG (\QQ_\ell, T_{k}^{(n)}(j)\otimes_{\cO_E^{(n)}}\cO_E^{(m)} ),\\
& \RG_{\Iw,S} ( T_{k}^{(n)}(j))\otimes_{\LL^{(n)}}^{\mathbf L}\LL^{(m)}
\simeq \RG_{\Iw,S} ( T_{k}^{(n)}(j)\otimes_{\LL^{(n)}}\LL^{(m)} ),\\ 
&\RG_\Iw(\QQ_\ell, T_{k}^{(n)}(j))\otimes_{\LL^{(n)}}^{\mathbf L}\LL^{(m)}
\simeq \RG_\Iw (\QQ_\ell, T_{k}^{(n)}(j)\otimes_{\LL^{(n)}}\LL^{(m)} ).
\end{aligned}
\end{equation}
\end{itemize} 

\subsubsection{}
\label{subsubsec_2022_08_24_1032}  Set $D_k:= \DCc \left(\bD_{k}\right).$  Recall that we denote by $N_\alpha$ the $\cO_E$-lattice
of $D^{(\alpha)}:=\Dc (V_{f}^{(\alpha)})$ generated by $\eta_f^{\alpha}$. Let us fix a free $\cO_E^{(n)}$-submodule 
$ N_{k}^{(n)} \subset D_k$ of rank one
such that 
\[
N_{k}^{(n)}\otimes_{\cO_E^{(n)}}\cO_E=N_\alpha.
\] 
For each non-archimedean
place $\ell\neq  p$, we consider the usual unramified  local condition at $\ell$:
$$
\RG_{\text{\rm Iw}}(\QQ_\ell,  T_{k}^{ (n)}(j), {N_{k}^{(n)}[j]}):= \RG_{{\Iw},\rm f}(\Bbb Q_\ell, T_{k}^{(n)}(j)):=\,
\left [(T_k^{ (n)}(j))^{I_\ell} \otimes \Lambda^{\iota} \xrightarrow{1-f_\ell} (T_k^{(n)}(j))^{I_\ell}\otimes \Lambda^{\iota} \right ],
$$
where $I_\ell$ is the inertia subgroup at $\ell$ and $f_\ell$ is the geometric Frobenius.

\subsubsection{}  At $\ell=p$, we shall consider the following Perrin-Riou--style local conditions, see \cite{benoisextracris} for a further discussion.
Let  $\bD^\dagger_E$ denote the functor which associates to a 
$p$-adic $G_{\Qp}$-representation its overconvergent $(\varphi,\Gamma)$-module \cite{CherbonnierColmez98}.  There exists a quasi-isomorphism
\[
\RG_{\Iw}(\Qp, T_{k}^{(n)}(j))\simeq \left [ \bD^\dagger_E ( T_{k}^{(n)}(j)) \xrightarrow{\psi -1}  \bD^\dagger_E ( T_{k}^{ (n)}(j))\right ],
\]
in the derived category of $\LL^{ (n)}$-modules, cf. \cite[Proposition~A5]{benoisextracris}.
Since $\CH (\Gamma)$ is flat over $\LL,$ we have 
\begin{multline}
\nonumber
\RG_{\Iw}(\Qp, T_{k}^{(n)}(j))\otimes_{\LL^{(n)}}^{\mathbf{L}}\widetilde\CH (\Gamma)
\simeq \RG_{\Iw}(\Qp, T_{k}^{(n)}(j))\otimes_{\LL}^{\mathbf{L}}\CH (\Gamma)
\\
\simeq \left [ \bD^\dagger_E ( T_{k}^{(n)}(j))\otimes_{\LL}\CH (\Gamma) \xrightarrow{\psi -1}  
\bD^\dagger_E ( T_{k}^{ (n)}(j))\otimes_{\LL}\CH (\Gamma) \right ].
\end{multline}
Moreover, as $ \bD^\dagger_E ( T_{k}^{ (n)}(j))\otimes_{\LL}\CH (\Gamma) \simeq \DdagrigE ( V_{k}(j))$
(cf. \cite{pottharst}),   we have the following isomorphism in the derived category $\mathscr D(\CH(\Gamma))$
of $\CH (\Gamma)$-modules:
\[
\RG_{\Iw}(\Qp, T_{k}^{ (n)}(j))\otimes_{\LL^{(n)}}^{\mathbf{L}}\widetilde\CH (\Gamma)\simeq \left [ \DdagrigE ( V_{k}(j)) \xrightarrow{\psi -1}  \DdagrigE  ( V_{k}(j))\right ].
\]
Let us set
$$ 
\RG_{\Iw}(\Qp , T_{k}^{ (n)}(j),{ N_{k}^{(n)}[j]}) := 
({ N_{k}^{(n)}}[j]\otimes_{\cO_E} \Lambda) [-1].
$$
We will denote the Perrin-Riou exponential map 
\[
\text{\rm Exp}_{\bD_{k}(j),j}\,:\, \mathfrak D(\bD_{k}[j]) \lra \DdagrigE ( V_{k}(j))^{\psi=1}
\] 
simply by $\text{\rm Exp}_{\bD_{k}(j)}.$ The map $\text{\rm Exp}_{\bD_{k}(j)}$ induces a morphism
$$
\RG_{\Iw} (\Qp, T_{k}^{(n)}(j),{ N_{k}^{(n)}[j]}) \lra \RG_{\Iw} (\Qp, T_{k}^{ (n)}(j))\otimes_{\LL^{(n)}}\widetilde \CH(\Gamma),
$$
which we  shall take as the local condition at $p.$

\subsubsection{}
\label{subsubsec_4125_2022_08_19_1511}
Consider the diagram
$$
\xymatrix{
\RG_{{\Iw},S}( T_{k}^{ (n)}(j))\otimes_{\Lambda^{ (n)}}\widetilde \CH(\Gamma)
\ar[r] & \underset{\ell\in S}\bigoplus \RG_{\Iw} (\QQ_\ell, T_{k}^{ (n)}(j))\otimes_{\Lambda^{ (n)}} 
\widetilde\CH (\Gamma)\\
 & \left (\underset{\ell \in S}\bigoplus \RG_{\Iw}(\QQ_\ell,  T_{k}^{(n)}(j), 
 { N_{k}^{(n)}[j]})
\right ) \otimes_{\Lambda^{(n)}} \widetilde\CH (\Gamma)\,\,.
 \ar[u]
}
$$
 Note that the objects of this diagram depend on $V_k$ and $D_k$ rather than $T_k^{(n)}$ and  $N_{k}^{(n)}$ (since $p$ is inverted), and we denote by $\RG_{\Iw} \left (V_{k}(j), \alpha\right )$ the Selmer complex associated to this diagram.
By definition, the complex $\RG_{\Iw} \left (V_{k}(j), \alpha\right )$ (to which we henceforth refer as the \emph{thick} Selmer complex) sits in the following distinguished triangle:
\begin{align}
\label{eqn_28_2021_06_02}
\begin{aligned}
\RG_{\Iw} \left (V_{k}(j), \alpha \right ) \lra &
\left (\RG_{\Iw,S} ( T_{k}^{ (n)}(j)) \bigoplus \left (\underset{\ell\in S}\bigoplus
\RG_{\Iw} \left (\QQ_\ell,  T_{k}^{(n)}(j),{ N_{k}^{(n)}[j]}  \right )\right )\right ) \otimes_{\LL^{ (n)}}\widetilde{\CH}(\Gamma)\\
&\qquad\qquad\qquad\qquad\qquad\lra \left (\underset{\ell\in S} \bigoplus
\RG_{\Iw} (\QQ_\ell, T_{k}^{ (n)}(j))\right )\otimes_{\LL^{(n)}}\widetilde{\CH}(\Gamma)\xrightarrow{[+1]}\, .
\end{aligned} 
\end{align}
Let us  define the fundamental line $\Delta_{\Iw} \left ( T_{k}^{(n)}(j), { N_{k}^{(n)}}[j] \right )$ on setting
\begin{align*}
\nonumber
\Delta_{\Iw}& \left ( T_{k}^{ (n)}(j), { N_{k}^{(n)}[j]} \right ):=\\
&{\det}^{-1}_{\LL^{(n)}}\left (\RG_{\Iw,S}( T_{k}^{ (n)}(j)) \bigoplus \left (\underset{\ell\in S}\bigoplus
\RG_{\Iw} \left (\QQ_\ell, T_{k}^{ (n)}(j),{ N_{k}^{(n)}[j]} \right )\right )\right ) \otimes 
{\det}_{\LL^{ (n)}} \left (\underset{\ell\in S}\bigoplus
\RG_{\Iw} \left (\QQ_\ell, T_{k}^{(n)}(j) \right )\right ).
\end{align*}
It follows from the exact triangle \eqref{eqn_28_2021_06_02} that we have a natural injection
\begin{equation}
\label{eqn: injection for trivialization}
 \Delta_{\Iw} \left ( T_{k}^{(n)}(j), { N_{k}^{(n)}[j]} \right )
 \lra  {\det}^{-1}_{\widetilde\CH (\Gamma)}\RG_{\Iw}\left (V_{k}(j),\alpha\right )\,.
\end{equation}

\subsubsection{} 
Suppose $m\geqslant n$ are natural numbers both verifying the conditions \eqref{item_L1} and \eqref{item_L2}.  We then have an isomorphism 
\begin{equation}
\label{eqn: extension of scalars for Delta}
\Delta_{\Iw} \left ( T_{k}^{(n)}(j),  N_{k}^{(n)}[j] \right )\otimes_{\cO_E^{(n)}}\cO_E^{(m)}
\simeq 
\Delta_{\Iw} \left ( T_{k}^{(n)}(j)\otimes_{\cO_E^{(n)}}\cO_E^{(m)} , N_{k}^{(m)}[j] \right ),
\end{equation}
which is compatible with the morphisms (\ref{eqn: injection for trivialization})
in the obvious sense. Set
\begin{equation}
\label{eqn: Definition Delta-infty}
\Delta_{\Iw} \left ( T_{k}^{(\infty)}(j),N_k^{(\infty)}[j]    \right ):=
\Delta_{\Iw} \left ( T_{k}^{(n)}(j),  N_{k}^{(n)}[j] \right )\otimes_
{\cO_E^{(n)}}{\cO_E^{(\infty)}}.
\end{equation}
Since the rings $\cO_E^{(n)}$ and $\cO_E^{(\infty)}$ are local, 
$\Delta_{\Iw} \left ( T_{k}^{(n)}(j),  N_{k}^{(n)}[j] \right )$ and $\Delta_{\Iw} \left ( T_{k}^{(\infty)}(j), N_k^{(\infty)}[j]   \right )$ 
are free modules of rank one over $\cO_E^{(n)}$ and $\cO_E^{(\infty)}$, respectively. It follows from \eqref{eqn: extension of scalars for Delta} and Lemma~\ref{lemma_thick_lattice}(ii) that the module $\Delta_{\Iw} \left ( T_{k}^{(\infty)}(j), N_k^{(\infty)}[j]   \right )$ (to which we henceforth refer to the \emph{thick} fundamental line) does not, up to canonical isomorphism, depend on the choice of $n$ and the lift $T_{k}^{(n)}$ of $T_{f}$. 
 By extension of scalars, the map (\ref{eqn: injection for trivialization}) induces a map
\begin{equation}
\label{eqn: injection for trivialization infinity}
\Theta_{\Iw, V_k(j)}^{(\alpha)} \,:\, \Delta_{\Iw} \left ( T_{k}^{(\infty)}(j), 
N_k^{(\infty)}[j] \right )
 \lra  {\det}^{-1}_{\widetilde\CH (\Gamma)}\RG_{\Iw}\left (V_{k}(j),\alpha  \right )\,.
\end{equation}

\subsubsection{} 
 Recall that
\[
N_{\alpha} \otimes_{\cO_E}E=D^{(\alpha)}.
\] 
Consider the Selmer complex $\RG_{\Iw}\left (V_{f}(j), \alpha\right )$
of $\CH_E(\Gamma)$-modules associated to the diagram 
$$
\xymatrix{
\RG_{{\Iw},S}( T_{f}(j))\otimes_{\Lambda}\CH(\Gamma)
\ar[r] & \underset{\ell\in S}\bigoplus \RG_{\Iw} (\QQ_\ell, T_{f}(j))\otimes_{\Lambda} 
\CH (\Gamma)\\
 & \left (\underset{\ell \in S}\bigoplus \RG_{\Iw}(\QQ_\ell, 
 T_{f}(j), \alpha)
\right ) \otimes_{\Lambda} \CH (\Gamma)\,\,,
 \ar[u]
 }
$$
which was introduced in Section~\ref{subsect:non-improved Selmer complex}. 

The tautological  exact sequence \eqref{eqn_4_1_2022_03_16}
induces a distinguished triangle
\begin{equation}
\label{eqn: distinguished triangle Selmer}
\RG_\Iw \left (V_{f}(j), \alpha \right ) \lra \RG_\Iw \left ( V_{k}(j),\alpha\right ) 
\lra \RG_\Iw \left (V_{f}(j),\alpha \right ) \lra \RG_\Iw \left (V_{f}(j),\alpha
\right )[1]\,.
\end{equation}


\subsection{Trivialization of the thick fundamental line} 

\subsubsection{} 
\label{subsubsec_aux_trivialization}

Our goal in \S\ref{subsubsec_aux_trivialization} is to establish an auxiliary statement (Proposition~\ref{prop:resolution} below) that we will use to trivialize the thick fundamental line $\Delta_{\Iw} (T_k^{(\infty)}(j),N_k^{(\infty)}[j])$. 

Let $R$ be a B\'ezout ring, which is simultaneously an  $E$-algebra. Set $\widetilde R:=R\otimes_E\widetilde E$. Then $\widetilde R/X\widetilde R=R.$  If $\widetilde M$
is a $\widetilde R$-module, we can and we will consider $M:=\widetilde M/X\widetilde M$ as an $R$-module. Note that $X^2\widetilde M=0.$ We shall consider $\widetilde R$-modules satisfying the following additional 
condition: 

\begin{itemize}
\item[\mylabel{item_M}{\bf (M)}]$\ker ({ \widetilde M}\xrightarrow{X} { \widetilde M})=X { \widetilde M}\,.$
\end{itemize}
Note that each free $\widetilde R$-module verifies \ref{item_M}. 
\begin{lemma} 
\label{lemma_length_2}
The following two conditions are equivalent:

\item[a)] $\widetilde M$ satisfies \ref{item_M}\,.

\item[b)] The sequence 
\[
{ 
0\lra M \lra \widetilde M \lra 
 M\lra 0,
}
\]
where the map $M:=\widetilde M/X\widetilde M \to \widetilde{M}$ is induced by multiplication 
by $X$, is exact. 
\end{lemma} 
\begin{proof} This is clear and the proof is omitted.
\end{proof} 
 
\begin{proposition} 
\label{prop:resolution}
Suppose that $\widetilde M$ is a finitely generated $\widetilde R$-module
satisfying \ref{item_M}. If  $M$ is finitely presented over $R$, then $\widetilde M$ admits  a free resolution of length $1$.
\end{proposition}
\begin{proof} 
Since $R$ is a B\'ezout ring and $M$ is finitely presented, it has a free resolution 
$$0\lra P_1\lra P_0 \lra M\lra 0$$ of length 1 (cf. \S\ref{subsec: bezout rings}).
We will construct a diagram of the form
\begin{equation}
\label{eqn:diagram resolution}
\xymatrix{
0\ar[r] &\widetilde P_1\ar[r]^{\widetilde f_1}\ar@{->>}[d]^{r_1}  &\widetilde P_0\ar[r]^{\widetilde f_0}\ar@{->>}[d]^{r_0}
&\widetilde M  \ar[r]\ar@{->>}[d] &0\\
0 \ar[r]  &P_1\ar[r]_{f_1}   &  P_0\ar[r]_{f_0}\ar[r] &M\ar[r] &0.
}
\end{equation}
 where $\widetilde P_i$ ($i=0,1$) are free $\widetilde R$-modules. Fix a basis $\{e_{0,i}\}_{i}$ of $P_0.$ Let  $\widetilde P_0=\underset{i}\oplus \widetilde  R\,\widetilde 
e_{0,i}$ be an arbitrary free $\widetilde R$-module of rank ${\rm rank}_R \, P_i$, generated by some set $\{\widetilde e_{0,i}\}_{i}.$ 
Let us denote by ${r_0 \,:\, \widetilde P_0 \rightarrow  P_0}$ the map given by
\[
\sum_i \widetilde a_i\widetilde e_{0,i} \longmapsto \sum_i a_i e_{0,i},
\]
where $a_i$ stands for the image of $\widetilde a_i$ under the projection $\widetilde {R}\twoheadrightarrow R$. 

Let $\widetilde m_i\in \widetilde M$ be any lift of $m_i:=f_0(e_{0,i})\in M$. We define $\widetilde f_0$ as the unique $\widetilde R$-linear map such that  
$\widetilde f_0(\widetilde e_{0,i})=\widetilde m_i$ for all $i.$ It is easy to see by Nakayama's lemma that  $\widetilde f_0$ is surjective and the square on the right in diagram \eqref{eqn:diagram resolution} commutes. 

Similarly, fix a basis  $\{e_{1,j}\}_{j}$ of $P_1$ and set $\widetilde P_1=\underset{j}\oplus \widetilde A\widetilde e_{1,j}.$ We define the map $r_1\,:\, \widetilde P_1 \rightarrow P_1$ as before, by setting $r_1(\widetilde e_{1,j})=e_{1,j}.$   Take any lifts $\widetilde n_j\in \widetilde P_0$
of the elements $n_j:=f_1(e_{1,j})$ and define the $\widetilde R$-linear map $g\,:\, \widetilde P_1 \rightarrow P_0$ setting $g (\widetilde e_{1,j})= \widetilde n_j.$ If we take $\widetilde f_1=g,$
then the square on the left in diagram \eqref{eqn:diagram resolution} commutes, but the upper row is not necessarily 
exact, and we therefore need to modify $g$. 

For any $j$, it follows from the constructions above that
$\widetilde f_0\circ g (\widetilde e_{1,j})\in X\widetilde M.$ Hence $\widetilde f_0\circ g (\widetilde e_{1,j})= X x_j$ for some $x_j\in \widetilde M.$ Take $y_j\in \widetilde P_0$ such that 
$\widetilde f_0(y_j)= x_j.$ Define the map $\widetilde f_1$ setting
\[
\widetilde f_1 (\widetilde e_{1,j})= g(\widetilde e_{1,j}) - Xy_j, \qquad \forall j\,.
\] 
Then $\widetilde f_0\circ \widetilde f_1=0.$ In addition, since $\widetilde f_1\equiv g\pmod {X},$
the diagram \eqref{eqn:diagram resolution} is commutative. 

We only need to check that the upper row of diagram \eqref{eqn:diagram resolution} is exact. To that end, let us first assume that  $x\in \ker (\widetilde f_1)$. Write $x=\sum_j \widetilde a_j \widetilde e_{1,j}.$
Then $r_1(x)\in \ker (f_1),$ and therefore $\sum_j a_j e_{1,j}=0,$
where $a_j\in \cO_\cW$ denotes the image of $\widetilde a_j$. We therefore infer that  $a_j=0$ and hence $\widetilde a_j\in X\widetilde R$ for all $j.$ Writing $\widetilde a_j=X\widetilde b_j$, we see that $x=Xy$ for  $y=\sum_j \widetilde b_j \widetilde e_{1,j} \in \widetilde P_1,$ and $X\widetilde f_1(y)=0$ in $\widetilde P_0.$
Since $\widetilde P_0$ satisfies the condition \ref{item_M}, this implies that 
$\widetilde f_1(y)\in X\widetilde{P_0},$ and therefore once again that $\widetilde b_j\in X\widetilde R.$
This shows that $\widetilde a_j=0$ for all $j$, and hence $x=0$.

We next check the exactness of the upper row at $\widetilde P_0$. Assume that $x\in \widetilde P_0$
is such that $\widetilde f_0 (x)=0$. It follows from the commutativity of \eqref{eqn:diagram resolution}
and the exactness of the bottom row that there exists $y\in \widetilde P_1$
such that $\widetilde f_1 (y)\equiv x \pmod{X\widetilde P_0}$. We may therefore assume without loss of generality that $x\in X\widetilde P_0.$ Write $x=Xx'$ with $x'\in \widetilde P_0.$ Then 
$X\widetilde f_0(x')=0$ in $\widetilde M.$ Since $\widetilde M$ satisfies \ref{item_M},
this implies that $\widetilde f_0 (x')\in X\widetilde M.$ Hence 
$f_0\circ r_0(x')=0$ and it follows from the commutativity of  \eqref{eqn:diagram resolution} that there exists $z\in \widetilde P_1$ such that $\widetilde f_1 (z)\equiv x'\pmod{X}.$
Therefore $\widetilde f_1(Xz)=Xx'=x,$ and the exactness is checked. 
\end{proof}


\subsubsection{} 
\label{subsubsec_30_05_2021_1_6_7} 
Suppose that $\mathscr L_\Iw^{\rm cr} (V_{f})\neq 0$. By \eqref{item_RG},  $\bR^i\boldsymbol{\Gamma}_\Iw(V_{f}(j),\alpha )=0$  for all  $i\neq 2$, 
and it follows from the global Euler--Poincar\'e characteristic formula that  $\bR^2\boldsymbol{\Gamma}_\Iw(V_{f}(j),\alpha)$ is torsion. The distinguished triangle \eqref{eqn: distinguished triangle Selmer} then reduces to an exact sequence
\[
0\lra \bR^2\boldsymbol{\Gamma}_\Iw (V_{f}(j),\alpha ) \xrightarrow{[X]} \bR^2\boldsymbol{\Gamma}_\Iw ( V_{k}(j), \alpha ) 
\lra \bR^2\boldsymbol{\Gamma}_\Iw (V_{f}(j),\alpha ) \lra 0.
\]
We infer using Lemma~\ref{lemma_length_2} that $\bR^2\boldsymbol{\Gamma}_\Iw ( V_{k}(j),\alpha )$ verifies the condition \ref{item_M}. 
By Proposition~\ref{prop: properties of punctual R2Gamma}, for each $\eta\in X(\Delta),$ the $\eta$-isotypic component $\bR^2\boldsymbol{\Gamma}_\Iw (V_{f}(j),\alpha )^{(\eta)}$ is a finitely presented torsion module over $\CH(\Gamma)_{\eta}$. 
Therefore, for every integer $j$, we have a  canonical  trivialization
\eqref{eqn: trivialization for length one resolution}
\begin{equation}
\label{eqn_trivialization_j}
i_{\Iw,V_{k}(j)}^{(\alpha)}\,:\, 
{\det}_{\widetilde\CH (\Gamma)}^{-1} \bR^2\boldsymbol{\Gamma}_\Iw \left ( V_{k}(j),\alpha\right )
\xhookrightarrow{\quad} \widetilde\CH(\Gamma)\,.
\end{equation} 
These maps are compatible with one another under twisting.

Paralleling the discussion in \S\ref{sec_modules_of_algebraic_padic_L_functions}, the  map \eqref{eqn: injection for trivialization} allows us to consider 
$\Delta_{\Iw} \left (T_{k}^{(\infty)}(j), N_k^{(\infty)}[j]\right )$ as a ${\LL}^{(\infty)}$-submodule of
${\det}_{\widetilde\CH (\Gamma)}^{-1} \bR^2\boldsymbol{\Gamma}_\Iw \left (V_{k}(j),\alpha\right ).$ 
\begin{defn}
\label{defn_eqn_trivialization_thick_module_of_padic_L}
We define the free $\LL^{(\infty)}$-module of rank one 
\be
\label{eqn_trivialization_thick_module_of_padic_L}
{\mathbf L}_{\Iw, \alpha } \left (T_{k}^{(\infty)}(j), N_k^{(\infty)}[j]\right ):=i_{\Iw,V_{k}(j)}^{(\alpha)}\left (\Delta_{\Iw} \left (T_{k}^{(\infty)}(j), N_k^{(\infty)}[j]\right )\right)
\ee
and call it the infinitesimal thickening of the module of algebraic $p$-adic $L$-functions associated with $T_{k}^{(\infty)}(j)$.
\end{defn}

As we have remarked at the start of Chapter~\ref{chapter_main_conjectures}, this terminology is inspired by that of Perrin-Riou~\cite[\S2]{perrinriou95}. The nomenclature can be justified by the Iwasawa main conjecture \ref{item_MCinf} below, which predicts that this module is indeed generated by the infinitesimal thickening of the critical  $p$-adic $L$-function. Its subtle relation with the modules of ``punctual'' $p$-adic $L$-functions (that are given by Definition~\ref{defn_327_2022_08_18_1205}) is the subject of Theorem~\ref{thm: descent for tilt complex}.

\section{Main Conjectures for the infinitesimal thickening of critical $p$-adic $L$-functions}
\label{sec_IMC}
\subsection{}
\label{subsec_IMC_thick}

We are now ready to formulate the infinitesimal thickening of the Iwasawa Main Conjecture for $\theta$-critical forms \ref{item_MCalpha}. 
Recall that we denote by $\widetilde \cE_N$ the image of $\cE_N=\underset{\ell \mid N}
\prod (1-a_\ell (x)\sigma_\ell^{-1})$ in $\LL^{(\infty)}.$
\begin{conj} 
\label{conj:infinitesimal Iwasawa conjecture}
We have $\mathscr L_\Iw^{\rm cr} (V_f)\neq 0$ and   
\begin{itemize} 
\item[\mylabel{item_MCinf}{$\widetilde{\mathbf{MC}} (f)$}] \hspace{2cm} 
${\mathbf L}_{\Iw, \alpha }\left (T_k^{(\infty)}(k), N_k^{(\infty)}[k]\right )^\pm  
=
\left (\lambda^\pm(f^*)\cdot \widetilde L_{\mathrm S,\alpha^*}^\pm (f^*, \xi^*)^\iota
\right )\,,$
\end{itemize}
where $\left (\lambda^\pm(f^*)\cdot \widetilde L_{\mathrm S,\alpha^*}^\pm (f^*, \xi^*)^\iota\right )$ denotes the $\LL^{(\infty),\pm}$-module generated by
$\lambda^\pm(f^*)\cdot \widetilde L_{\mathrm S,\alpha^*}^\pm (f^*, \xi^*)^\iota .$
\end{conj}
Using Proposition~\ref{prop_1_19_2022_16_03}, this conjecture can be formulated in terms of Kato's 
$p$-adic $L$-function:
\be
\nonumber
\label{conj_infinitesimal_Iwasawa_conjecture_Kato}
\widetilde \cE_N^{\iota} \cdot {\mathbf L}_{\Iw, \alpha}\left (T_k^{(\infty)}(k), N_k^{(\infty)}[k]\right ) =
\left (
\widetilde L_{\mathrm K,\alpha^*}(f^*, \xi^*)^\iota \right ).
\ee

\subsection{} In what follows, we shall explore various consequences of Conjecture~\ref{conj:infinitesimal Iwasawa conjecture}. 

\begin{proposition}
\label{infinitesimaMC implies MC} 
Conjecture~\ref{item_MCinf} implies \ref{item_MCalpha}.  
\end{proposition}
\begin{proof} This follows from  constructions of Section~\ref{subsect: comparision between improved and nonimproved complexes} and the following commutative diagram, where the vertical morphisms are induced by $\otimes^{\mathbf{L}}_{\widetilde
\CH (\Gamma)}\CH (\Gamma)$:
\[
\xymatrix{
\Delta_{\Iw}(T_k^{(\infty)}(k),N_k^{(\infty)}[k]) 
\ar [d] \ar[rr]^-{\Theta_{\Iw, V_k(k)}^{(\alpha)}} & &
{\det}_{\widetilde\CH (\Gamma)}^{-1}\RG_\Iw (V_k(k),\alpha) 
\ar[rr]^-{i_{\Iw,V_{k}(k)}^{(\alpha)}}  \ar[d] & &\widetilde \CH(\Gamma) \ar[d]\\
\Delta_{\Iw}(T_{f}(k),N_{\alpha}[k]) \ar[rr]_-{\Theta_{\Iw, V_{ f}(k)}^{(\alpha)}}&
& {\det}_{\CH (\Gamma)}^{-1}\RG_\Iw (V_{ f}(k),\alpha) \ar[rr]_-{i_{\Iw,V_{f}(k)}^{(\alpha)}} & &\CH(\Gamma).
}
\]

\end{proof}

\subsection{Descent formalism over $\widetilde \CH (\Gamma)$}
\label{subsec_423_2022_0810_0849}
\subsubsection{} Our main aim in the remainder of this chapter is to prove that Conjecture~\ref{item_MCinf} is consistent with the slope-zero Main Conjecture \ref{item_MCbeta} and the interpolation
properties of the secondary $p$-adic $L$-function $L^{[1]}_{\mathrm{K},\alpha^*}(f^*,\xi^*)$. We remark that the thick Selmer complex $\RG_\Iw (V_k(j),\alpha)$ is not semi-simple, namely the natural map
\[
\RG_\Iw (V_k(j),\alpha)^\Gamma \lra \RG_\Iw (V_k(j),\alpha)_\Gamma
\]
is not an isomorphism.

Our goal in \S\ref{subsec_423_2022_0810_0849} is to prove an analog of the descent formula \eqref{eqn: descent coinvariants} for $\widetilde\CH (\Gamma)$-modules. 

\subsubsection{}
\label{subsubsec_4232}
 To simplify notation, let us put $z:=\gamma_1-1.$ Let $\widetilde M$ and 
$\widetilde N$ be free $\widetilde \CH (\Gamma_1)$-modules 
of finite rank $h$ and let 
\[
\phi\,:\, \widetilde M\lra \widetilde N
\]
be an injective morphism. Then $\phi$ induces a map 
\[
\det (\phi)\,: \,{\det}_{\widetilde\CH (\Gamma_1)}(\widetilde M) \otimes {\det}^{-1}_{\widetilde\CH (\Gamma_1)}(\widetilde N)
\lra \widetilde\CH (\Gamma_1).
\]
Let $\widetilde M_{\Gamma_1}:=\widetilde M \otimes_{\widetilde\CH (\Gamma_1)}\widetilde E$ and 
$\widetilde N_{\Gamma_1}:=\widetilde N \otimes_{\widetilde\CH (\Gamma_1)}\widetilde E.$ We have 
an induced morphism of free $\widetilde E$-modules 
\[
\phi_{\Gamma} \,:\, \widetilde M_{\Gamma_1}\lra \widetilde N_{\Gamma_1}.
\]
Consider the tautological exact sequence
\begin{equation}
\label{eqn: exact sequence for f_Gamma}
0 \lra \ker (\phi_\Gamma) \lra \widetilde M_{\Gamma_1}\xrightarrow{\phi_{\Gamma}} \widetilde N_{\Gamma_1}
\lra \mathrm{coker} (\phi_\Gamma) \lra 0.
\end{equation}
of $\widetilde E$-modules. We assume that the following condition holds true:
\begin{itemize}
\item[\mylabel{item_D1}{\bf D1})]
$\dim_E \left (\ker (\phi_\Gamma)\right )=1.$
\end{itemize}
It follows from this condition that $\ker (\phi_\Gamma)\subset X\widetilde M_{\Gamma_1}$
and $X\widetilde N_{\Gamma_1} \subset \mathrm{im} (\phi_\Gamma).$

Let $m\in \ker (\phi_\Gamma)$ be any element. Let us write $m=Xm'$ for some $m'\in \widetilde M_{\Gamma_1}$
and take a lift $\widehat m'\in M$ of $m'$ under the canonical projection 
$\widetilde M\rightarrow \widetilde M_{\Gamma_1}$. Then 
\[
\phi(X\widehat m')\in X \widetilde N \cap z\widetilde N=zX\widetilde N.
\]
Write $\phi(X\widehat m')=X z\widehat n$, with $\widehat n\in \widetilde N$. Let $B(\phi)(m)$ denote the image of $\widehat n$ in $\mathrm{coker} (\phi_\Gamma)=\widetilde N_{\Gamma_1}/\mathrm{im} (\phi_\Gamma)$. It is easy to see that $B(\phi)(m)$ does not depend on the choice of $\widehat n$,  and we have a well defined map
\[
B(\phi)\,:\, \ker (\phi_\Gamma) \lra \mathrm{coker} (\phi_\Gamma)\,,
\]
which we view as a variant of the Bockstein morphism. We assume in addition that 
\begin{itemize}
\item[\mylabel{item_D2}{\bf D2})] The map $B(\phi)$ is an isomorphism. 
\end{itemize}

Then the exact sequence \eqref{eqn: exact sequence for f_Gamma} induces a trivialization
\begin{equation}
\label{eqn: trivialization map i_f}
i_\phi \,:\, {\det}_{E}(\widetilde M_{\Gamma_1}) \otimes {\det}^{-1}_{E}(\widetilde N_{\Gamma_1})\simeq 
{\det}_{E}(\ker (\phi_\Gamma)) \otimes {\det}_{E}^{-1}(\mathrm{coker} (\phi_\Gamma)) \xrightarrow{B(\phi)} E\,.
\end{equation}

\begin{example} In this paragraph, let us consider the special case when $\widetilde{M}=\widetilde{N}=\widetilde{\CH}(\Gamma_1).$ Any morphism 
$\phi\,:\, \widetilde{\CH}(\Gamma_1) \rightarrow \widetilde{\CH}(\Gamma_1)$ can be written in the form
\[
\phi (\widetilde{m})=(a(z)+Xb(z)) {\widetilde{m}}, \qquad 
\widetilde{m}\in \widetilde{\CH}(\Gamma_1),
\]
for some $a(z),b(z)\in {\CH}(\Gamma_1).$   We have $\widetilde{M}_{\Gamma_1}=\widetilde{N}_{\Gamma_1}=\widetilde{E},$
and $\phi_{\Gamma}\,:\,\widetilde{E} \rightarrow \widetilde{E}$ is the multiplication by $Xb(0)$ map. Therefore, condition \eqref{item_D1} holds if and only if $b(0)\neq 0.$
The map $B(\phi)\,:\, X\widetilde{E} \rightarrow \widetilde{E}/X\widetilde{E}$  sends $X$ to $a^*(0):=\left. z^{-1}a(z)\right \vert_{z=0},$ and therefore \eqref{item_D2} holds if and only if $a^*(0)\neq 0.$ 
\end{example}

We refer the reader to Lemma~\ref{prop: Bockstein for torsion tilde-modules} where we present criteria to check the properties \eqref{item_D1} and \eqref{item_D2}. We employ these criteria in Theorem~\ref{thm: semisimplicity of tilde Selmer} in the context of eigencurve, to check their validity (under suitable assumptions) for the module $\widetilde{M}=\bR^2\boldsymbol{\Gamma}_{\Iw} (V_k (j),\alpha)$. This is key to the Iwasawa decent formalism for thick Selmer complexes that we develop in \S\ref{subsec_425_2022_08_19_1548}.

\subsubsection{}  
\label{subsubsec_4233}
For any $\widetilde \CH(\Gamma_1)$-bases  $\{m_i\}_{i=1}^h$ and $\{n_i\}_{i=1}^h$
of $\widetilde M$ and $\widetilde N$, let us set 
\[
\widetilde{\mathbf m}=\bigwedge_{i=1}^h m_i \in {\det}_{\widetilde \CH(\Gamma_1)}\widetilde M,
\qquad 
\widetilde{\mathbf n}=\bigwedge_{i=1}^h n_i \in {\det}_{\widetilde \CH(\Gamma_1)}\widetilde N.
\]
Consider the families $\{m_i, Xm_i\}_{i=1}^h$ and $\{n_i, Xn_i\}_{i=1}^h$
as $\CH(\Gamma_1)$-bases of $\widetilde M$ and $\widetilde N$ and let us set
\[
\mathbf{m}= \bigwedge_{i=1}^h (m_i\wedge Xm_i) \in {\det}_{\CH (\Gamma_1)}(\widetilde M),
\qquad 
\mathbf{n}= \bigwedge_{i=1}^h (n_i\wedge Xn_i) \in {\det}_{\CH (\Gamma_1)}(\widetilde N).
\] 
Let $\mathbf{m}_\Gamma \in \det_{E}(\widetilde M_{\Gamma_1})$ and $\mathbf{n}_\Gamma \in \det_{E}(\widetilde N_{\Gamma_1})$ denote the images of $\mathbf{m}$ and $\mathbf{n}$.

The following statement is the analogue of the descent  formula  \eqref{eqn: descent coinvariants} in this setting:

\begin{proposition}
\label{prop: tilde descent formalism}
Assume that the map $\phi$ satisfies  conditions  \eqref{item_D1} and \eqref{item_D2}.  Retaining the notation of \S\ref{subsubsec_4232} and \S\ref{subsubsec_4233},  let us write 
\[
\det (\phi) (\widetilde{\mathbf{m}}\otimes {\widetilde{\mathbf{n}}}^{-1})= a(z)+b(z)X.
\]
Then $a(z)$ has a zero of order one at $0$ and 
\[
i_\phi (\mathbf{m}_\Gamma \otimes \mathbf{n}^{-1}_\Gamma)= a^*(0)\cdot b(0),
\]
where $a^*(0)=\left. z^{-1}a(z)\right \vert_{z=0}$.
\end{proposition} 

\begin{proof}  Since $\CH(\Gamma_1)$ is a B\'ezout ring, it is easy to see that we can choose an $\widetilde \CH(\Gamma_1)$-basis $\{m_i\}_{i=1}^h$ of $M$ such that $X\overline m_1$ is an $E$-basis of $\ker (\phi_\Gamma)$ (where $\overline m$ denotes the image of $m\in M$ under the projection $M\to M_\Gamma$). Fix an arbitrary $\widetilde \CH(\Gamma_1)$-basis $\{n_i\}_{i=1}^h$ of $N$. Then 
\[
\phi(m_i)= \underset{j=1}{\overset{h}\sum} (a_{ij}(z)+Xb_{ij}(z)) n_j,
\]
where $a_{1j}(0)=0$ for all $1\leqslant j\leqslant h.$
Set $a_{1j}^*(0):=\left.z^{-1}a_{1j}(z) \right \vert_{z=0}.$

Using the fact that $a_{1j}(0)=0$ for  $1\leqslant j\leqslant h,$
it is easy to see that 
\begin{equation}
\label{eqn: first equation in descent proposition}
a^*(0)\cdot b(0) =\det( A ) \cdot \det(B)
\end{equation} 
where
\[
A=
\left (
\begin{matrix}
a_{11}^*(0) & a_{12}^*(0) & \ldots &a_{1h}^*(0)\\
a_{11}(0) & a_{12}(0) & \ldots &a_{1h}(0)\\
\cdots &\cdots &\cdots &\ddots \\
a_{h1}(0) & a_{h2}(0) & \cdots &a_{hh}(0)
\end{matrix}
\right ),
\qquad 
B=
\left (
\begin{matrix}
b_{11}(0) & b_{12}(0) & \ldots &b_{1h}(0)\\
a_{11}(0) & a_{12}(0) & \ldots &a_{1h}(0)\\
\cdots &\cdots &\cdots &\ddots \\
a_{h1}(0) & a_{h2}(0) & \cdots &a_{hh}(0)
\end{matrix}
\right ).
\]

Then 
\[
B(\phi) (Xm_1)= \underset{j=1}{\overset{h}\sum}a_{ij}^*(0) n_j \mod{\mathrm{im}(\phi_\Gamma)}\,,
\]
and $i_\phi(\mathbf{m}_\Gamma\otimes \mathbf{n}_\Gamma^{-1})$ is equal to the determinant
\[
\det \left (
\begin{array}{@{}cc|cc|cc|cc@{}}
 0           & b_{11}(0) & 0           & b_{12}(0) &\cdots &\cdots & 0           & b_{1h}(0) \\
 a_{11}^*(0) & 0         & a_{12}^*(0) &0          &\cdots &\cdots  & a_{1h}^*(0) &0 \\\hline
 a_{21}(0)   &b_{21}(0)  & a_{22}(0)   &b_{22}(0)  &\cdots  &\cdots  & a_{2h}(0)   &b_{2h}(0)\\
  0           & a_{21}(0) & 0          & a_{22}(0) &\cdots &\cdots & 0  & a_{2h}(0) \\
 \hline
 \vdots           & \vdots  & \vdots          & \vdots  &\ddots  & & \vdots & \vdots \\
\vdots           & \vdots  & \vdots          & \vdots  &  &\ddots & \vdots & \vdots \\
\hline
 a_{21}(0)   &b_{21}(0)  & a_{22}(0)   &b_{22}(0)  &\cdots  &\cdots  & a_{hh}(0)   &b_{hh}(0)\\
 0           & a_{h1}(0) & 0           & a_{h2}(0) &\cdots &\cdots & 0  & a_{hh}(0)
 \end{array}
\right )
\]
Using elementary row and column operations, it is easy to see that this determinant is equal to a determinant of the form 
\[
-\det \left (
\begin{array}{@{}c|c@{}}
* &A\\
\hline
B &0
\end{array}
\right ) = \det(A)\det(B)\,.
\]
Comparing this computation with \eqref{eqn: first equation in descent proposition}, we obtain the lemma
for our particular choice of $\{m_i\}_{i=1}^h$. 

The general case can be reduced to this particular case by a change of basis.  
\end{proof}

\subsubsection{}
Let $\widetilde M$ be a  $\widetilde \CH(\Gamma_1)$-module satisfying the condition of Proposition~\ref{prop:resolution}. Recall that  $\widetilde M$ has a free resolution
\[
0\lra \widetilde P_1 \xrightarrow{\,\,\phi\,\,}  \widetilde P_0 \lra \widetilde M \lra 0.
\]
Then ${\det}_{\widetilde{\CH} (\Gamma_1)}(\widetilde M) \simeq 
{\det}_{\widetilde{\CH} (\Gamma_1)}(\widetilde P_0)
\otimes {\det}^{-1}_{\widetilde{\CH} (\Gamma_1)}(\widetilde P_1)$.

\begin{lemma}
\label{prop: Bockstein for torsion tilde-modules}
Assume in addition that $\dim_E({\widetilde M}^{\,\Gamma_1})=1$. Then:

\item{i)} $\widetilde M^{\Gamma_1}\subset X\widetilde M$.

\item{ii)} The map
\begin{equation}
\nonumber
\begin{aligned}
B\,:\, {\widetilde M}^{\,\Gamma_1} &\lra \widetilde M_{\Gamma_1},\\
m &\longmapsto y \mod{z\widetilde M}, \qquad \textrm{where $B(m)=Xy$}
\end{aligned}
\end{equation}
is well-defined (namely, it does not depend on the choice of $y$).

\item{iii)} The map $\phi\,:\,\widetilde P_1 \rightarrow \widetilde P_0$ satisfies the condition \eqref{item_D1} and we have a commutative diagram
\[
\xymatrix{
{\widetilde M}^{\,\Gamma_1} \ar[d]_-{B} \ar[r]^{\sim} &\ker (\phi_\Gamma) \ar[d]^{B(\phi)}\\
{\widetilde M}_{\Gamma_1}  \ar[r]_-{\sim} &\mathrm{coker} (\phi_\Gamma).
}
\] 
In particular, $\phi$ satisfies the condition \eqref{item_D2} if and only if $B$ is an isomorphism.
\end{lemma}
\begin{proof} The proof is straightforward and is omitted.
\end{proof}

\subsubsection{}
Let  $\widetilde{C}^\bullet$ be  a complex of $\widetilde{\CH}(\Gamma_1)$-modules  concentrated in a single {\it even} degree $i$
and such that $H^i(\widetilde{C}^\bullet)$ satisfies the conditions
of Proposition~\ref{prop:resolution}. Let 
\[
0\lra \widetilde{P}_{1}\xrightarrow{\,\,\phi\,\,} \widetilde{P}_{0} \lra H^i(\widetilde{C}^\bullet)
\lra 0
\]
be a free resolution of $H^i(\widetilde{C}^\bullet).$ There exists a  unique, up to a homotopy, lift of the resolution above to a quasi-isomorphism
\[
\xymatrix{
&  &\widetilde{P}_{1} \ar[r]^{\phi} \ar[d] &\widetilde{P}_0  \ar[d] & & \\
\cdots\ar[r]& \widetilde{C}^{i-2} \ar[r] &\widetilde{C}^{i-1} \ar[r] & 
\widetilde{C}^{i} \ar[r] & \widetilde{C}^{i+1} \ar[r]&\cdots.
}
\]
Set $\widetilde{C}^{\bullet}_0=\widetilde C^{\bullet}\otimes_{\widetilde{\CH}(\Gamma_1)}^{\mathbb{L}} E.$
The diagram above  provides us with a trivialization 
\begin{equation}
\label{eqn: abstract descent trivialization for tilde-complexes}
{\eta}_{0} \,:\, {\det}_{E}^{-1} (\widetilde{C}^{\bullet}_0) \simeq 
{\det}_{E}^{-1} ((\widetilde{P}_1)_{\Gamma_1}\xrightarrow{\phi_{\Gamma}} 
(\widetilde{P}_0)_{\Gamma_1}) \xrightarrow{i_\phi} E,
\end{equation}
where $i_{\phi}$ is the  map \eqref{eqn: trivialization map i_f}.
If $\widetilde{P}_{\bullet}'$ is another free resolution of $H^i(\widetilde{C}^\bullet),$
then there exists a unique, up to a homotopy, quasi-isomorphism $\widetilde{P}_{\bullet}'\simeq \widetilde{P}_\bullet.$ This implies that ${\eta}_{0}$ does not depend on the choice  of $\widetilde{P}_{\bullet}.$
The map ${\eta}_{0}$ should be seen as an analogue of the trivialization map \eqref{eqn: nekovar trivialization} in classical Iwasawa theory.

\begin{remark} Let us consider the special case when $h=2$ and consider the map
$\phi \,:\, \widetilde{M} \rightarrow \widetilde{N}$
between two free $\widetilde{\CH}(\Gamma_1)$-modules of rank $2$
given by the diagonal matrix of the form
\[
\left (
\begin{matrix}
a_1(z)+Xb_1(z) &0\\
0& a_2(z)+Xb_2(z)
\end{matrix}
\right )
\]
with $a_1(0)=a_2(0)=0$ and $b_1(0),b_2(0)\neq 0.$
Then $\ker (\phi_{\Gamma})=X\widetilde{M},$ and in particular,
$\dim_E( \ker(\phi_{\Gamma}))= \dim_E( \mathrm{coker}(\phi_{\Gamma}))=2.$ In this situation, we still can 
define the isomorphisms $B(\phi)\,:\,\ker (\phi_{\Gamma}) \rightarrow 
\mathrm{coker}(\phi_{\Gamma})$ and 
$i_\phi \,:\, {\det}_{E}(\widetilde M_{\Gamma_1}) \otimes {\det}^{-1}_{E}(\widetilde N_{\Gamma_1}) \rightarrow E$ via the same formulae. An easy computation gives
\[
i_\phi(m_{\Gamma}\otimes n_{\Gamma}^{-1})=a_1^*(0)a_2^*(0) b_1(0)b_2(0).
\]
On the other hand, $\det (\phi) (\widetilde{m}\otimes \widetilde{n}^{-1})=a_1(z)a_2(z) +X(a_1(z)b_2(z)+b_1(z)a_2(z)).$ This shows that Proposition~\ref{prop: tilde descent formalism} 
can not be directly generalized to the case when $\dim_E( \ker(\phi_{\Gamma}))>1.$ 
\end{remark}

\subsection{The Bockstein map for thick Selmer complexes}
\label{subsec_2023_07_17_1056}
In the present subsection, we study the Bockstein morphism $B(\phi)$ in the context of thick Selmer complexes. The main results of \S\ref{subsec_2023_07_17_1056}, which are recorded as Theorem~\ref{thm: semisimplicity of tilde Selmer} (where we explicitly describe the Bockstein morphism for thick Selmer complexes and its key properties), will be used in \S\ref{subsec_425_2022_08_19_1548} to prove various leading term formulae (cf. Theorems~\ref{thm: descent for tilt complex} and \ref{thm_414_2022_0810_1712}) employing our Iwasawa descent formalism for thick Selmer complexes.

The reader familiar with \cite{nekovar06} will realize that the overall structure of this portion of our work is akin to \S11.6 and \S11.7 in op. cit.: Nekov\'a\v{r} studies Bockstein morphisms and their basic properties in \cite[\S11.6]{nekovar06} to establish a descent formalism, and applies this with his Selmer complexes in \cite[\S11.7]{nekovar06} to prove what he calls ``formulas of the Birch and Swinnerton-Dyer type''. Note however that his formalism does not apply in the setting we have placed ourselves.
\label{subsec_thick_bockstein_424}
\subsubsection{}  We start with several 
auxiliary lemmas.   


\begin{lemma}
\label{lemma:vanishing invariants of sha}  
Let $1\leqslant j\leqslant k-1$ be an integer. Assume that one of the following conditions holds:
\item[a)]{} $j\neq \frac{k}{2}$.
\item[b)]{} $j= \frac{k}{2}$ and $H^1_0(V_{f^*}(k/2))=\{0\}$.

Then,
\[
\Sha_\Iw^2(V_k(j))^{\Gamma}=\{0\}=\Sha_\Iw^2(V_k(j))_{\Gamma}\,.
\]
\end{lemma} 
\begin{proof}
The commutative diagram with exact rows 
\[
\xymatrix{
H^2_\Iw (V_{f}(j)) \ar[r] \ar[d] & H^2_\Iw (V_k(j)) \ar[r] \ar[d] 
&H^2_\Iw (V_{f}(j))\ar[r] \ar[d]   & 0 \\
\underset{\ell \in S}\oplus H^2_\Iw (\QQ_\ell, V_{f}(j)) 
\ar[r]
& \underset{\ell \in S}\oplus H^2_\Iw (\QQ_\ell, V_k(j))
\ar[r] 
&\underset{\ell \in S}\oplus H^2_\Iw (\QQ_\ell, V_{f}(j))
\ar[r] 
&0
}
\]
induces the exact sequence
\[
\Sha_\Iw^2(V_{f}(j)) \lra \Sha_\Iw^2(V_k(j)) \lra \Sha_\Iw^2(V_{f}(j))\,.
\]
Let $A$ (resp.  $B\subset \Sha_\Iw^2(V_{f}(j))$) denote the coimage (resp. the image) of the first (resp. the second) map. Consider the exact sequence
\[
0 \lra A \lra  \Sha_\Iw^2(V_k(j)) \lra B \lra 0\,.
\]
By the snake lemma, we have an exact sequence
\[
0 \lra A^\Gamma \lra  \Sha_\Iw^2(V_k(j))^\Gamma \lra B^\Gamma \lra  A_\Gamma  \lra 
\Sha_\Iw^2(V_k(j))_\Gamma  \lra B_\Gamma  \lra 0\,.
\]
 Recall that  that if $M$ is a finitely generated torsion $\LL_E$-module,
then $M$ is of finite dimension over $E$, and 
\be
M^\Gamma=\{0\}\, \iff\, M_\Gamma =\{0\},
\ee
cf. \eqref{eqn_2022_05_06_1205}.
Since $\Sha_\Iw^2 (V_f(k/2))_\Gamma \simeq H^1_0(V_{f^*}(k/2))^*$, we have
$\Sha^2_\Iw (V_{f}(j))^\Gamma=\{0\}= \Sha^2_\Iw (V_{f}(j))_\Gamma$ in both cases (a) and (b).  This implies that the modules $A^\Gamma$, $B_\Gamma$, and therefore $A_\Gamma$ and $B^\Gamma$,  vanish. Thus
$$\Sha_\Iw^2(V_k(j))^\Gamma=\{0\}=\Sha_\Iw^2(V_k(j))_\Gamma,$$
as required.
\end{proof}

\begin{lemma}
\label{lemma: for semisimplicity}
 Let $1\leqslant j\leqslant k-1$ be an integer. Assume that one of the following conditions holds:
\item[a)]{} $j\neq \frac{k}{2}$.
\item[b)]{} $j= \frac{k}{2}$ and $H^1_0(V_{f}(k/2))=\{0\}=H^1_0(V_{f^*}(k/2))$.

Then the rows of the following commutative diagram are exact and the vertical morphisms are injective:
\be\label{eqn_2022_05_06_1150}
\begin{aligned}
\xymatrix{
0 \ar[r] &H^1_{\{p\}}(V_{f}(j)) \ar[r]^{\iota_k}  \ar[d]_{\res_p}
&H^1_{\{p\}}(V_k(j)) \ar[r] \ar[d]_{\res_p}  &H^1_{\{p\}}(V_{f}(j)) \ar[d]
\ar[r] &0 \\
0 \ar[r] &H^1 (\Qp, V_{f}(j)) \ar[r]_{\iota_k}  &H^1 (\Qp, V_k(j)) \ar[r] &H^1 (\Qp, V_{f}(j))
\ar[r] &0
}
\end{aligned}
\ee 
In particular, for every such $j$ we have 
\[
\res_p\circ\iota_k\left(H^1_{\{p\}}(V_{f}(j))\right)=\res_p\left(H^1_{\{p\}}(V_k(j))\right) \cap \iota_k\left(H^1 (\Qp, V_{f}(j))\right)=\iota_k\circ\res_p\left(H^1_{\{p\}}(V_{f}(j))\right)\,.
\]

\end{lemma}
\begin{proof} 
Since $H^0(\Qp, V_{f}(j))=\{0\}=H^0(V_{f}(j))$, the rows of the diagram are exact except possibly in the right-most surjection on the top.

We have an exact sequence
\[
0 \lra H^1_\Iw (V_f(j)) \lra H^1_\Iw (V_k(j)) \lra H^1_\Iw (V_f(j))  \lra H^2_\Iw (V_f(j))
\]
 of $\LL_E$-modules, where $H^1_\Iw (V_f(j))$ is free of rank one and $H^2_\Iw (V_f(j))$ is a $\LL_E$-torsion module (cf. \cite{kato04}, Theorem~13.4). Therefore,  $H^1_\Iw (V_k(j))$ is free of rank $2$ over $\LL_E$.
It follows from Lemma~\ref{lemma:exact sequence with H_f,p} and Lemma~\ref{lemma:vanishing invariants of sha} that 
$$H^1_{\{p\}}(V_{f}(j))=H^1_\Iw (V_{f}(j))_\Gamma \quad\hbox{ and }\quad H^1_{\{p\}}(V_k(j))=H^1_\Iw (V_k(j))_\Gamma\,.$$
Therefore, $\dim_E H^1_{\{p\}}(V_{f}(j))=1$ and $\dim_E H^1_{\{p\}}(V_{k}(j))=2.$
This shows that the map $H^1_{\{p\}}(V_k(j))\rightarrow H^1_{\{p\}}(V_{f}(j))$ is surjective, and we have proved that the rows of the diagram are exact.

It follows from the vanishing $H^1_{0}(V_{f}(j))=\{0\}$ that the left-most and the right-most vertical maps are injective. Therefore, the vertical map in the middle is injective as well. The final assertion can be proved by a diagram chase.



\end{proof}

\subsubsection{}  We now prove the main result of this section.
Let $g$ denote the map  
\begin{equation}
\label{eqn: the map g}
\left (H^1_\Iw (V_k(j))\otimes_{\widetilde \LL_E} \widetilde \CH (\Gamma)\right )
\oplus 
\left (D_k [j] \otimes_{\widetilde \LL_E} \widetilde \CH (\Gamma)
\right ) 
\xrightarrow{\res_p\,+\, \Exp_{\bD_k(\chi^j)}}
 H^1_\Iw (\Qp, V_k(j))\otimes_{\widetilde \LL_E} \widetilde \CH (\Gamma)\,,
\end{equation}
where we use the notation from Section~\ref{subsec: defortmation of exponential map}.
If $\mathscr L_\Iw^{\mathrm{cr}}(V_{f})\neq 0,$ the map  $g$ is injective,
and we regard  \eqref{eqn: the map g} as a resolution of 
$\mathrm{coker} (g)$. Since $H^0(\Qp (\zeta_{p^\infty}), V_k^*(1-j))=\{0\},$
we also have $H^2_\Iw (\Qp, V_k(j))=\{0\},$ and the sequence
\[
0 \lra H^1_\Iw (\Qp, V_{f}(j)) \lra H^1_\Iw (\Qp, V_k(j))
\lra H^1_\Iw (\Qp, V_{f}(j)) \lra 0
\]
is exact. Since $ H^1_\Iw (\Qp, V_{f}(j))$ is free over $\LL_E$, this implies that $H^1_\Iw (\Qp, V_k(j))$ is $\widetilde \LL_E$-free,
and therefore  $H^1_\Iw (\Qp, V_k(j))\otimes_{\widetilde \LL_E} \widetilde \CH (\Gamma)$ is $ \widetilde \CH (\Gamma)$-free. 
  
Let us denote by $\RG_{\Iw}^{(\mathds 1)} (V_k (j),\alpha)$ the isotypic component corresponding to the trivial character of $\Delta=\mathrm{Gal}(\QQ (\zeta_p)/\QQ)$. Note that the $\Gamma_1$-invariants (resp.  $\Gamma_1$-coinvariants) of this component coincide with the  $\Gamma$-invariants  (resp.  $\Gamma$-coinvariants) of $\RG_{\Iw} (V_k (j),\alpha)$. 


\begin{theorem}
\label{thm: semisimplicity of tilde Selmer}
Let $1\leqslant j\leqslant k-1$ be an integer such that either $j\neq \frac{k}{2}$, or else $j= \frac{k}{2}$ and $L(f^*,\frac{k}{2})\neq 0$. 

We have
\[
\bR^2\boldsymbol{\Gamma}_{\Iw} (V_k (j),\alpha)^\Gamma \simeq \mathrm{coker} (g)^\Gamma\,,
\qquad \bR^2\boldsymbol{\Gamma}_{\Iw} (V_k (j),\alpha)_\Gamma \simeq \mathrm{coker} (g)_\Gamma\,.
\]
Moreover:
\item[i)] Assume that the condition \eqref{item_C4} holds. There exists a commutative diagram
\begin{equation}
\label{eqn: statement semisimplicity selmer}
\begin{aligned}
\xymatrix{
\bR^2\boldsymbol{\Gamma}_{\Iw} (V_k(j),\alpha )^\Gamma \ar[r]^-{\sim} \ar[d]_{B} & XD_{k} [j] \ar[d]^-{\widetilde\delta} \\
\bR^2\boldsymbol{\Gamma}_{\Iw} (V_k(j), \alpha)_\Gamma  \ar[r]^-{\sim} &\dfrac{H^1(\Qp, V_{f}(j))}{\res_p(H^1_{\{p\}}(V_{f}(j)))},
}
\end{aligned}
\end{equation}
where the vertical map on the right $\widetilde\delta$ is given by
\[
\widetilde\delta (Xd)=- (j-1)! \left ( \frac{1-p^{j-1}\alpha^{-1}}{1-p^{-j}\alpha}\right ) 
\left (\exp^*_{V^{(\alpha)}_{f}(j)}\right )^{-1}({d}), \qquad d\in D_k[j].
\]

 \item[ii)] The following conditions are equivalent.
\begin{itemize}
\item[1)]{} $\mathscr L^{\mathrm{cr}}(V_{f}(j))\neq 0$\,.


\item[2)]{} 
The condition \eqref{item_C4} holds, $\bR^2\boldsymbol{\Gamma}_\Iw (V_k (j),\alpha)^\Gamma$ is a one dimensional $E$-vector space
and the map  

$$B\,:\,\bR^2\boldsymbol{\Gamma}_{\Iw} (V_k (j),\alpha)^\Gamma \lra \bR^2\boldsymbol{\Gamma}_{\Iw} (V_k (j),\alpha)_\Gamma$$
defined as in Lemma~\ref{prop: Bockstein for torsion tilde-modules} is an isomorphism.
\end{itemize}
\end{theorem}
\begin{proof} 
In all cases we consider,  $H^1_{\mathrm{f}}(V_f(\frac{k}{2})=
H^1_{\mathrm{f}}(V_{f^*}(\frac{k}{2}))=0.$ Therefore  we can 
apply Lemmas~\ref{lemma:vanishing invariants of sha} and \ref{lemma: for semisimplicity}. By the definition of the Selmer complex, we have an exact sequence
\[
0\lra \mathrm{coker} (g) \lra 
\bR^2{\boldsymbol\Gamma}_\Iw  (V_k (j),\alpha ) \lra \Sha_\Iw^2(V_k(j))\otimes_{\widetilde \LL_E} \widetilde \CH (\Gamma) \lra 0\,.
\]
Applying the snake lemma to the diagram
$$
\xymatrix{0\ar[r]& \mathrm{coker} (g) \ar[r]\ar[d]_{[\gamma-1]}& 
\bR^2\boldsymbol{\Gamma}_\Iw (V_k (j),\alpha) \ar[r]\ar[d]_{[\gamma-1]}& \Sha_\Iw^2(V_k(j))\otimes_{\widetilde \LL_E} \widetilde \CH (\Gamma) \ar[r]\ar[d]^{[\gamma-1]}& 0\\ 0\ar[r]& \mathrm{coker} (g) \ar[r]& 
\bR^2\boldsymbol\Gamma_\Iw (V_k (j),\alpha) \ar[r]& \Sha_\Iw^2(V_k(j))\otimes_{\widetilde \LL_E} \widetilde \CH (\Gamma) \ar[r]& 0
}
$$
obtained from this exact sequence and Lemma~\ref{lemma:vanishing invariants of sha}, we deduce that 
\be\label{eqn_2022_05_06_1421}\bR^2\boldsymbol\Gamma_\Iw  (V_k (j),\alpha)^\Gamma \simeq  \mathrm{coker}(g)^\Gamma
\quad \hbox{ and }\quad \bR^2\boldsymbol\Gamma_\Iw  (V_k (j),\alpha)_\Gamma \simeq  \mathrm{coker}(g)_\Gamma.\ee
This proves the very first assertion in our theorem.

\item[i)] Since  
$$\left (H^1_\Iw (\Qp, V_k(j))\otimes_{\widetilde \LL_E} \widetilde \CH (\Gamma)\right )^\Gamma=\{0\}\,,$$ 
it follows that
\[
\textrm{coker} (g)^\Gamma =
\ker (g_\Gamma):=\ker \left (
H^1_\Iw (V_k(j))_{\Gamma} \oplus  D_k[j]
\xrightarrow{\res_p\,+\, \Exp^{}_{\bD_k(\chi^j),0}}  H^1(\Qp, V_k(j))
\right ).
\]
Moreover, by Proposition~\ref{prop:definition of kappa} and Proposition~\ref{prop: comparison of exponentials for different eigenvalues}, we have 
\[
\ker \left (\Exp^{}_{\bD_{k} (\chi^j),0}\right )= X D_{k}[j] \subset  D_k[j]\quad,
\qquad \textrm{im} \left (\Exp^{}_{\bD_k (j),0}\right ) =H^1(\Qp, X V^{(\beta)}_{k}(j)).
\] 
Note that $ X D_{k}[j] \simeq D^{(\alpha)}[j]$ and $(X V_{k})^{(\beta)}(j)\simeq 
V^{(\beta)}_{f}(j).$
Since 
\[
\res_p\left(H^1_{\{p\}}(V_{f}(j))\right) \cap H^1(\Qp, V^{(\beta)}_{f}(j))=H^1_\beta (V_{f}(j))=\{0\}
\]
 and $H^1_\Iw (V_k(j))_{\Gamma} \hookrightarrow H^1_{\{p\}}(V_k(j))$, it follows from Lemma~\ref{lemma: for semisimplicity} that
$$\res_p\left(H^1_\Iw (V_k(j))_{\Gamma}\right) \cap \Exp^{}_{\bD_k(j),0}( D_k[j])=\{0\}\,.$$ 
This yields a canonical isomorphism 
\[
\textrm{coker} (g)^\Gamma \simeq \ker \left (\Exp^{}_{\bD_k (\chi^j),0}\right )=XD_k[j] \simeq D^{(\alpha)}[j]\,.
\]
On the other hand, it follows from Lemma~\ref{lemma: for semisimplicity} that 
$$\dim_E \left (H^1_{\{p\}}(V_k(j))\right )=2 \quad \hbox{ and } \quad  \res_p\left(H^1_{\{p\}}(V_k(j))\right)\cap \iota_k\left(H^1(\Qp, V_{f}^{(\beta)}(j))\right)=\{0\}\,.$$ 
Therefore,
\[
\textrm{coker} (g)_\Gamma \simeq 
\textrm{coker} (g_\Gamma):=
\frac{H^1(\Qp, V_k(j))}{\res_p\left(H^1_{\{p\}}(V_k(j))\right)+ \iota_k\left(H^1(\Qp, V_{f}^{(\beta)}(j)\right)} \xrightarrow{\sim}  \frac{H^1(\Qp, V_{f}(j))}{\res_p\left(H^1_{\{p\}}(V_{f}(j))\right)}\,,
\]
where the last isomorphism is induced by the projection $V_k(j) \rightarrow V_{f}(j)$. In particular, 
\[
\dim_E \mathbf{R}^2\boldsymbol{\Gamma}_\Iw (V_k(j),\alpha)_\Gamma=
\dim_E \left (\textrm{coker} (g)_\Gamma\right )=1.  
\]
Let $0\rightarrow \widetilde P_1 \xrightarrow{\phi} \widetilde P_0 \rightarrow \bR^2\boldsymbol{\Gamma}_\Iw (V_k(j),\alpha)\to 0$ be a free resolution. Since $H^1_\Iw (\Qp, V_k(j))\otimes_{\widetilde \LL_E}\widetilde \CH(\Gamma)$ is free,
there exists a commutative diagram
\[
\xymatrix{
0 \ar[r] &\left (H^1_\Iw (V_k(j))
\oplus D_k [j]\right )  \otimes_{\widetilde \LL_E} \widetilde \CH (\Gamma)
\ar[d] \ar[r]^-{g}
&
 H^1_\Iw (\Qp, V_k(j))\otimes_{\widetilde \LL_E} \widetilde \CH (\Gamma) \ar[d]
 \ar[r] & \mathrm{coker}(g)\ar[d] \ar[r] &0\\
0\ar[r] & \widetilde P_1 \ar[r]_{\phi} & \widetilde P_0 \ar[r] &
\mathbf{R}^2\boldsymbol{\Gamma}_\Iw (V_k(j),\alpha)  \ar[r] & 0
}
\] 
with exact rows and columns. Applying the snake lemma, we obtain the following cartesian squares, where all maps are 
isomorphisms:
\begin{equation}
\nonumber
\xymatrix{
\mathrm{coker} (g)^\Gamma \ar[r] \ar[d] & \ker (g_\Gamma) \ar[d] \\
\bR^2\boldsymbol{\Gamma}_\Iw (V_k(j),\alpha)^\Gamma \ar[r] & \ker (\phi_\Gamma)
}
\qquad  \qquad 
\xymatrix{
\mathrm{coker} (g)_\Gamma \ar[r] \ar[d] & \mathrm{coker} (g_\Gamma) \ar[d] \\
\bR^2\boldsymbol{\Gamma}_\Iw (V_k(j),\alpha)_\Gamma \ar[r] & \mathrm{coker} (\phi_\Gamma).
}  
\end{equation}
Therefore, we have a commutative diagram
\begin{equation}
\label{eqn: commutative diagram for bockstein maps}
\begin{aligned}
\xymatrix{
\bR^2\boldsymbol{\Gamma}_\Iw (V_k(j),\alpha)^\Gamma \ar[r]^-{\sim} \ar[d]_B  
&\ker (\phi_\Gamma)  \ar[r]^{\sim} \ar[d]_{B(\phi)}
&\ker (g_\Gamma) \ar[d]^{B(g)} \\
\bR^2\boldsymbol{\Gamma}_\Iw (V_k(j),\alpha)_\Gamma
\ar[r]^-{\sim}
&\mathrm{coker} ({ \phi}_\Gamma)
\ar[r]^-{\sim}
&\mathrm{coker} (g_\Gamma)\,.
}
\end{aligned}
\end{equation}
By the definitions of the isomorphisms above, we have a commutative diagram
\begin{equation}
\label{diagram semisimplicity of selmer}
\begin{aligned}
\xymatrix{
{\rm ker} (g_\Gamma) \ar[r]^{\sim} \ar[dd]_{B(g)} & XD_{k}[j] \ar[dd]^{\widetilde\delta 
} \\\\
\textrm{coker} (g_\Gamma)  \ar[r]_-{\sim} &\displaystyle\frac{H^1(\Qp, V_f(j))}{\res_p\left(H^1_{\{p\}}(V_f(j))\right)}
}
\end{aligned}
\end{equation}
where the vertical map on the right is given by
\[
\widetilde\delta (Xd)=  \pr_0 (y), \quad \textrm{where $d\in D_k[j]$ and $\Exp_{V^{(\alpha)}_{x_0}(j),j}(d)=(\gamma-1)y.$}
\]
The identity \eqref{eqn: relation between exponentials} allows us to write
\[
y= \ell (V_{f}(j)) \cdot
\Exp_{\alpha,j}(x),
\]
where  $\ell (V_{f}(j))=\underset{\substack{1-k+j\leqslant i\leqslant j-1\\i\neq 0}}\prod \ell_i.$
Therefore, using \eqref{eqn:specialization of PR formulae},
\begin{align*}
\widetilde\delta (Xd)=\left (\underset{\substack{1\leqslant i\leqslant k-1\\i\neq k-j}}\prod (i+j-k) \right )\pr_0\circ \Exp_{\alpha,j}(x)
=- (j-1)! \left ( \frac{1-p^{j-1}\alpha^{-1}}{1-p^{-j}\alpha}\right ) 
\left (\exp^*_{V^{(\alpha)}_{f}(j)}\right )^{-1}(d).
\end{align*}
Now the proof of (ii) follows from \eqref{eqn: commutative diagram for bockstein maps}
and \eqref{diagram semisimplicity of selmer}.

\item[ii)] The implication  
$$\mathscr L^{\mathrm{cr}}(V_{f}(j))\neq 0 \quad \implies{\quad \eqref{item_C4}}$$
is explained in Remark~\ref{remark_G_intro}(ii). We now use \eqref{eqn_2022_04_26_14_24} to see that the map $\widetilde\delta$ is an isomorphism  if and only if $\mathscr L^{\mathrm{cr}}(V_{f}(j))\neq 0$. This concludes the proof of (ii). 
\end{proof}

\subsection{Iwasawa descent}
\label{subsec_425_2022_08_19_1548}
The main results in this subsection are Theorems~\ref{thm: descent for tilt complex} and \ref{thm_414_2022_0810_1712}, where we prove various leading term formulae, employing Iwasawa descent for thick Selmer complexes which we established in the earlier portions (\S\ref{subsec_423_2022_0810_0849} and \S\ref{subsec_2023_07_17_1056}) of the present chapter. 
\subsubsection{}
\label{subsubsec4251_2022_08_16}
Let us set
\[
\RG (V_k(j), \alpha)
:= 
\RG_\Iw (V_k(j), \alpha)\otimes^{\mathbf L}_{\widetilde {\mathscr{H}}(\Gamma)}\widetilde E.
\]
Note that 
$$\RG (V_k(j), \alpha)\simeq 
\RG_\Iw (V_k(j), \alpha)\otimes^{\mathbf L}_{ {\mathscr{H}}(\Gamma)} E$$ 
as complexes of $E$-vector spaces. For an integer $n$ that verifies \eqref{item_L1} and \eqref{item_L2}, let us put  
\[
\RG \left (\QQ_\ell,T_k^{(n)}(j),N_k^{(n)}[j]\right ):=
\RG_\Iw \left (\QQ_\ell,T_k^{(n)}(j),N_k^{(n)}[j]\right ).
\]
More explicitly, 
\[
\RG \left (\QQ_\ell,T_k^{(n)}(j),N_k^{(n)}[j]
\right )=
\begin{cases}
\RG_{\mathrm{f}} \left (\QQ_\ell,T_k^{(n)}(j)\right ) & 
\textrm{if $\ell \neq p$}  \\
N_k[j] [-1] &\textrm{if $\ell=p$}.
\end{cases}
\]
We will consider the fundamental line 
\begin{align*}
\nonumber
&\Delta_{\widetilde E} \left  (T_k^{(n)}(j), N_k^{(n)}[j] \right ):=\\
&\,\,{\det}_{\cO_E^{(n)}}
\left (
\RG_S \left ( T_k^{(n)}(j) \right ) \oplus 
\left (\underset{\ell \in S}\bigoplus \RG \left (\QQ_\ell,T_k^{(n)}(j),N_k^{(n)}[j]
\right ) \right ) \right )
\otimes_{\cO_E^{(n)}}
{\det}_{\cO_E^{(n)}}
\left  (\underset{\ell \in S}\bigoplus \RG (\QQ_\ell,T_k^{(n)}(j)) \right ).
\end{align*}
over $\cO^{(n)}_E$. Note that
\[
\Delta_{\widetilde E} \left  (T_k^{(n)}(j), N_k^{(n)}[j] \right )\simeq
\Delta_{\Iw} \left  (T_k^{(n)}(j), N_k^{(n)}[j] \right )\otimes_{\LL^{(n)}}\cO_E^{(n)}\,.
\]
We will also consider a fundamental line over $\cO_E$, which we define on setting
\begin{multline}
\nonumber
\Delta_{E} \left  (T_k^{(n)}(j), N_k^{(n)}[j] \right ):=\\
{\det}_{\cO_E}
\left (
\RG_S \left ( T_k^{(n)}(j) \right ) \bigoplus 
\left (\underset{\ell \in S}\bigoplus \RG \left (\QQ_\ell,T_k^{(n)}(j),N_k^{(n)}[j]
\right ) \right ) \right )
\otimes_{\cO_E}
{\det}_{\cO_E}
\left  (\underset{\ell \in S}\bigoplus \RG (\QQ_\ell,T_k^{(n)}(j)) \right ).
\end{multline}
We remark that 
\begin{equation}
\nonumber
\Delta_{E} \left  (T_k^{(n)}(j), N_k^{(n)}[j] \right )\simeq 
{\det}_{\cO_E}\Delta_{\widetilde E} \left  (T_k^{(n)}(j), N_k^{(n)}[j] \right ).
\end{equation}

\subsubsection{} By the functoriality of $\otimes^{\mathbf{L}}_{\LL^{(n)}}\widetilde\cO_E,$
there exists a commutative diagram
\[
\xymatrix{
\Delta_{\Iw} \left  (T_k^{(n)}(j), N_k^{(n)}[j] \right ) \ar[rr]^-{\Theta^{(\alpha)}_{\Iw, V_k(j)}} 
\ar[d]_{\otimes_{\LL^{(n)}}^{\mathbf L}\widetilde\cO_E}& & {\det}^{-1}_{\widetilde \CH(\Gamma)}\RG (V_k(j),\alpha) \ar[d]^{\otimes^{\mathbf L}_{\mathrm{\widetilde \CH(\Gamma)}}\widetilde E} \\
\Delta_{\widetilde E} \left  (T_k^{(n)}(j), N_k^{(n)}[j] \right ) \ar[rr]_-{\widetilde \Theta_{ V_k(j)}^{(\alpha)}} & & {\det}^{-1}_{\widetilde E}\RG (V_k(j),\alpha).
}
\]
Applying the functor ${\det}_{\cO_E}(-)$ on the bottom row of this diagram,
we obtain a map
\begin{equation}
\nonumber
\Theta_{V_k(j)}^{(\alpha)}\,:\,
\Delta_{E} \left  (T_k^{(n)}(j), N_k^{(n)}[j] \right ) \lra {\det}^{-1}_{E}\RG (V_k(j),\alpha)\,.
\end{equation}

\subsubsection{}  
Let 
\[
i_{V_k(j)}\,:\, {\det}_E^{-1}\RG (V_k(j), \alpha)  \rightarrow  E,
\] 
denote the trivialization
map \eqref{eqn: abstract descent trivialization for tilde-complexes}
for the complex $\widetilde{C}^{\bullet}=\RG_\Iw (V_k(j), \alpha).$


\begin{theorem} 
\label{thm: descent for tilt complex}
Let $1\leqslant j\leqslant k-1$ be an integer such that either $j\neq \frac{k}{2}$, or else $j= \frac{k}{2}$ and $L(f^*,\frac{k}{2})\neq 0$. Assume that 
$\mathscr L^{\mathrm{cr}}(V_f(j))\neq 0$.
Then there exists a canonical isomorphism 
\[
\mu \,\,:\,\, {\det}^{-1}_E \RG (V_k(j), \alpha)
\xrightarrow{\,\sim\,} 
{\det}^{-1}_E \RG(V_{f}(j),\alpha )\otimes 
{\det}^{-1}_E \RG (V_{f}(j),\beta )
\]
with the following properties:
\begin{itemize}
\item[i)]{} The following diagram commutes:
\large
\be\label{eqn_20220506_1637}
\begin{aligned}
\xymatrix{
{\det}^{-1}_E \RG (V_k(j), \alpha) \ar[dd]_-{\mu}
\ar[rrr]^-{i_{V_k(j)}} &&& E\ar@{=}[dd]
\\\\
{\det}^{-1}_E \RG (V_{f}(j),\alpha) \otimes {\det}^{-1}_E \RG (V_{f}(j),\beta) \ar[rrr]_-{i_{V_{f}(j)}^{(\alpha)}\,\otimes\, i_{V_{f}(j)}^{(\beta)}} 
& && E,
}
\end{aligned}
\ee
\normalsize
where $i_{V_{f}(j)}^{(\alpha)}$ and $i_{V_{f}(j)}^{(\beta)}$ denote the trivialization maps provided by Lemma~\ref{lemma Burns-Greither} as in  Section~\ref{eqn: general descent non-central case}, and 
$\ell^*(V_{f}(j)):=\underset{\substack{i=1-k\\ i\neq -j}}{\overset{-1}\prod} (i+j)$\,.

\item[ii)]{}  The map $\mu$ sends 
$$\Theta_{V_k(j)}^{(\alpha)} \left ( \Delta_E ( T^{(n)}_{k}(j),N_k^{(n)}[j])\right )$$ 
onto 
$$
C_{\mathrm{K}}\cdot { \frac{b(\mathds 1,j)}{a (\mathds 1,j)}}\cdot \Theta_{ V_{x_0}(j)}^{(\alpha)} \left ( \Delta_E (T_{x_0}(j),\alpha)\right )\otimes_{\cO_E}
\Theta_{ V_{x_0}(j)}^{(\beta)} \left ( \Delta_E (T_{x_0}(j),\beta)\right )$$
where $C_{\mathrm{K}}$ is the constant defined in \eqref{eqn: the constant a}.

\item[iii)]{} We have
\[
i_{V_k(j)}\circ \Theta_{V_k(j)}^{(\alpha)} \left ( \Delta_E ( T^{(n)}_{k}(j),N_k^{(n)}[j])\right ) =\mathbf L_{\Iw,\alpha}^*(T_{f}(j), N_{\alpha}[j],\mathds{1}, 0 ) \cdot
\mathbf L_{\Iw,\beta}(T_{f}(j), N_\beta[j],\mathds{1}, 0).
\]
\end{itemize}
\end{theorem}
\begin{proof} It follows from the condition \eqref{item_C4} (which holds since we assume $\mathscr L^{\mathrm{cr}}(V_f(j))\neq 0$, cf. Remark~\ref{remark_G_intro}(ii)) that we have an exact sequence
\begin{equation}
\label{eqn: exact sequence with kappa_j}
0 \lra XD_k[j] \lra  D_k [j]
\xrightarrow{\kappa_j} D^{(\beta)}[j] \lra 0\,, \qquad XD_k\simeq D^{(\alpha)},
\end{equation}
where $\kappa_j$ is the map defined in Proposition~\ref{prop:definition of kappa}.
Let $\RG (XV_{k}(j), D_k[j])$ denote   the Selmer complex 
associated to the diagram
\[
\xymatrix{
\RG_S (XV_{k}(j)) \ar[r] &\underset{\ell \in S}\oplus \RG (\QQ_\ell, XV_{k}(j))\\
& \underset{\ell \in S}\oplus \RG (\QQ_\ell, XV_{k}(j), D_k[j])
\ar[u],
}
\]
where the local condition at $p$ is given by 
\[
\Exp_{\bD_k (\chi^j),j,0}\,:\,D_k[j] \lra H^1(\Qp, XV_{k}(j)), 
\]
and the local conditions at $p\neq \ell \in S$ by the unramified local conditions.
Note that we may naturally identify $XV_{k}\subset V_k$ with the representation $V_{f}$. The exact sequence (\ref{eqn: exact sequence with kappa_j})
and Proposition~\ref{prop: comparison of exponentials for different eigenvalues}
show that we have a commutative diagram 
\begin{equation}
\label{eqn:diagrm for local conditions}
\begin{aligned}
\xymatrix{
XD_k[j] [-1] \ar[r] 
&\RG (\Qp, XV_{k}(j), D_k[j]) \ar[rr]^-{\frac{b(\mathds 1,j)}{a(\mathds 1,j)}\cdot \kappa_j} \ar[dd]_{\Exp_{V_{f}^{(\alpha)} (j),0}} 
& &\RG (\Qp, V_{f}(j),\beta) \ar[dd]^{ \widetilde\Exp_{\beta,j,0}}\\\\
& H^1(\Qp, XV_{k}(j)) \ar@{=}[rr] & 
& H^1(\Qp, V_{f}(j))\,.
}
\end{aligned}
\end{equation}
Here  $\widetilde\Exp_{\beta,j,0}$ denotes the projection of the map 
$\widetilde\Exp_{\beta,j}$ defined in (\ref{eqn:definition of widetilde E beta}),
and $D^{(\alpha)}[j] [-1]$ means that the module $D^{(\alpha)}[j]$
is placed in degree one. This diagram  gives rise to an exact triangle  
\begin{equation}
\label{eqn: auxiliary exact sequence} 
XD_k[j] [-1] \lra \RG (XV_{k}(j),  D_k[j])\lra \RG (V_{f}(j),\beta)\,. 
\end{equation}

Let $\RG_0 (V_{k}/XV_k (j),\alpha)$ denote the Selmer complex associated to the diagram
\[
\xymatrix{
\RG_S (V_{k}/XV_k (j)) \ar[r] &\underset{\ell \in S}\bigoplus \RG (\QQ_\ell, 
V_{k}/XV_k(j))\\
& \underset{\ell \in S\setminus \{p\}}\bigoplus \RG_{\mathrm{f}}(\QQ_\ell, V_{k}/XV_k(j))
\ar[u]
}
\]
with strict local conditions at $p$. Note that we canonically have $V_k/XV_k\simeq V_{f}$. The exact sequence of local conditions 
\[
0\lra 
\underset{\ell \in S}\bigoplus\RG (\QQ_\ell, XV_{k}(j), D_k[j])) \lra 
\underset{\ell \in S}\bigoplus \RG (\QQ_\ell, V_k(j), D_k[j]) \lra
\underset{\ell \in S\setminus \{p\}}\bigoplus \RG_{\mathrm{f}}(\QQ_\ell, V_{k}/XV_k(j))
\lra 0
\]
provides us with an exact triangle
\[
\RG (XV_{k}(j),D_k[j]) \lra 
 \RG (V_k(j), \alpha)
 \lra \RG_0 (V_{k}/XV_k(j),\alpha) 
\lra \RG (XV_{k}(j), D_k[j])[1]\,.
\]
Combined with \eqref{eqn: auxiliary exact sequence}, we obtain canonical isomorphisms
\begin{align}
\label{eqn_20220506_1630}\begin{aligned}
{\det}^{-1}_E \RG (V_k(j), \alpha)&\simeq {\det}^{-1}_E \RG_0 (V_{f}(j),\alpha) \otimes {\det}^{-1}_E XD_k[j] \otimes  {\det}^{-1}_E \RG (V_{f}(j),\beta)  \\
&\simeq {\det}^{-1}_E \RG (V_{f}(j),\alpha) \otimes {\det}^{-1}_E \RG (V_{f}(j),\beta)\,,
\end{aligned}
\end{align}
where we use the isomorphism  $XD_k\simeq D^{(\alpha)}$ to write
$$
{\det}_E \RG_0 (V_{f}(j),\alpha) \otimes {\det}_E XD_k[j] \simeq
{\det}_E \RG_0 (V_{f}(j),\alpha) \otimes {\det}_E D^{(\alpha)}[j] \simeq {\det}_E \RG (V_{f}(j),\alpha)\,.$$
We define $\mu$ as the composition of the maps in \eqref{eqn_20220506_1630}. 
Let $0\rightarrow \widetilde P_1 \xrightarrow{\phi} \widetilde P_0 \rightarrow 0$
be a free resolution of $\bR^2\boldsymbol{\Gamma}_\Iw (V_k(j),\alpha).$
By Theorem~\ref{thm: semisimplicity of tilde Selmer} and Proposition~\ref{prop: descent for non-improved complex}, we have the following 
commutative diagram
\[
\xymatrix{
\ker (f_\Gamma) \ar[d]_-{ B(\phi)} \ar[r]^{\sim} &  XD_k[j] \ar[d]^-{\widetilde\delta} \ar[r]^{\sim} & D^{(\alpha)}[j] \ar[d]^-{\delta} \ar[r]^-{\sim} &\bR^1\boldsymbol{\Gamma} (V_{f}(j),\alpha) 
\ar[d]^-{\mathrm{Bock}_2}
\\
\textrm{coker} (\phi_\Gamma)  \ar[r]^-{\sim} 
&\displaystyle\frac{H^1(\Qp, V_{f}(j))}{\res_p(H^1_{\{p\}}(V_{f}(j)))} \ar@{=}[r] &
\displaystyle\frac{H^1(\Qp, V_{f}(j))}{\res_p(H^1_{\{p\}}(V_{f}(j)))}
\ar[r]^-{\sim}
&\bR^2\boldsymbol{\Gamma} (V_{f}(j),\alpha) \quad .
}
\]
Since the Bockstein map $\mathrm{Bock}_2$ trivializes ${\det}_{E}\RG (V_{f}(j),\alpha)$ 
and $ B(\phi)$ trivializes ${\det}_{E}\RG (V_{k}(j),\alpha),$ it is  immediate from the construction of $\mu$  that the diagram 
\eqref{eqn_20220506_1637} commutes.
Using  the diagram  \eqref{eqn:diagrm for local conditions}, it is now easy to check that the map $\mu$ sends 
\[
\Theta_{V_k(j)}^{(\alpha)}( \Delta_E (T^{(n)}_k(j)),N_k^{(n)}[j] ) 
\]
onto 
\[
a \cdot\Theta_{ V_{f}(j)}^{(\alpha)} \left ( \Delta_E (T_{f}(j),\alpha)\right )
\otimes_{\cO_E} 
\Theta_{ V_{f}(j)}^{(\beta)} \left ( \Delta_E (T_{f}(j),\beta)\right )\,,
\] 
where $a\in E$ is such that $aN_\beta={ \frac{b(\mathds 1,j)}{a (\mathds 1,j)}}\kappa_0 (N_\alpha),$ which is precisely Part (ii) in the statement of our theorem.

Finally, (iii) follows from (i), (ii) and Lemma~\ref{lemma Burns-Greither},
which tells us that 
\[
\begin{aligned}
&i_{V_{x_0}(j)}^{(\alpha)}\circ \Theta_{ V_{f}(j)}^{(\alpha)} \left ( \Delta_E (T_{f}(j),\alpha)\right ) =\mathbf L_{\Iw,\alpha}^*(T_{f}(j), N_{\alpha}[j],\mathds{1}, 0 ),\\
&i_{V_{x_0}(j)}^{(\beta)}\circ \Theta_{ V_{f}(j)}^{(\beta)} \left ( \Delta_E (T_{f}(j),\beta)\right ) =\mathbf L_{\Iw,\beta}(T_{f}(j), N_{\beta}[j],\mathds{1}, 0 ).
\end{aligned}
\]
This completes the proof of our theorem.
\end{proof}

\subsubsection{}
The following theorem is evidence for the infinitesimal thickening~\ref{item_MCinf} of the $\theta$-critical Main Conjecture, as it shows (among other things) that it is concurrent with the conclusion of Theorem~\ref{thm_interpolative_properties} and slope-zero Main Conjecture \ref{item_MCbeta}. 

\begin{theorem}
\label{thm_414_2022_0810_1712}
Let $1\leqslant j\leqslant k-1$ be an integer such that either $j\neq \frac{k}{2}$, or else $j= \frac{k}{2}$ and $L(f^*,\frac{k}{2})\neq 0.$
Assume that 

\item[a)] $\mathscr L^{\mathrm{cr}}(V_f(k-j))\neq 0$\,,


\item[b)] the infinitesimal thickening~\ref{item_MCinf} of the $\theta$-critical Main Conjecture is valid.


Then, 
\be\label{eqn_2023_08_14_1012}
 L^{[1]}_{\mathrm{K},\alpha^*}(f^*, \xi^*;\chi^j)^\pm 
\sim_p
C_{\mathrm{K}}\cdot { \mathcal E_N(f^*;\chi^j) \cdot \frac{b(\mathds 1,k-j)}{a (\mathds 1,k-j)}} \cdot \mathbf L_{\Iw,\beta}(T_{f}(k-j),N_\beta[k-j],\mathds 1,0)^\pm\,.
\ee
\end{theorem}
\begin{proof} 
\item[i)] The Infinitesimal Main Conjecture reads
\[
\widetilde {\mathcal E}_N^{ \iota} \cdot \mathbf L_{\Iw,\alpha}\left (T_k^{(n)}(k), N_k[k]\right )  =
\LL^{(n)}  \widetilde L_{\mathrm{K},\alpha^*}(f^*,\xi^*)^\iota,
\qquad \textrm{for $n\gg 0.$}
\]
Twisting both sides, we deduce that
\[
\left (\Tw_{j}\circ \widetilde {\mathcal E}_N^{ \iota}\right )
 \cdot
 \mathbf L_{\Iw,\alpha}\left (T_k^{(n)}(k-j), N_k[k-j]\right ) =
\LL^{(n)}  \left (\Tw_{-j}\circ \widetilde L_{\mathrm{K},\alpha^*}(f^*,\xi^*)\right )^\iota,
\qquad \textrm{for $n\gg 0.$}
\]
We can rewrite the final equality in the form
\[
\left (\Tw_{j}\circ \widetilde {\mathcal E}_N^{ \iota}\right )
 \cdot
 \mathbf L_{\Iw,\alpha}\left (T_k^{(n)}(k-j), N_k[k-j]\right ) =
\left (a+\frac{X}{\varpi_E^n}b \right )  \left (\Tw_{-j}\circ \widetilde L_{\mathrm{K},\alpha^*}(f^*,\xi^*)\right )^\iota,
\]
where $a+\frac{X}{\varpi_E^n}b\in \LL^{(n)}$ is a unit; this is equivalent to saying that $a\in \LL^{\times}$ and $b\in \LL.$

By Theorem~\ref{thm: semisimplicity of tilde Selmer}, $\RG_\Iw (V_k(k-j),\alpha)^\Gamma$
is one-dimensional, and we can apply our descent formalism with the lattice 
\[
\Theta_{\Iw,V_k({ k-}j)}^{(\alpha)}(\Delta_\Iw (T_k^{(n)}(k-j), N_k[k-j]) \subset 
{\det}^{-1}_{\widetilde \CH(\Gamma)}\RG_\Iw (V_k(k-j),\alpha).
\]
Recall that 
\[
 \widetilde L_{\mathrm{K},\alpha^*}(f^*,\xi^*)=
  L_{\mathrm{K},\alpha^*}^{[0]}(f^*,\xi^*) + X L_{\mathrm{K},\alpha^*}^{[1]}(f^*,\xi^*)\,.
\]
The function $L_{\mathrm{K},\alpha^*}^{[0]}(f^*,\xi^*)$ has a simple zero  at all characters $\chi^{j}$ (where $j$ is as in the theorem), and we denote by  $L_{\mathrm{K},\alpha^*}^{[0],*}(f^*,\xi^*,\chi^j)$ its special value at $\chi^j.$
Proposition~\ref{prop: tilde descent formalism} implies that 
the map $i_{V_k({ k-}j)}$ sends 
\[
{ \mathcal E_N(f^*;\chi^j)^2 \cdot }\Theta^{(\alpha)}_{V_k({ k-}j)}\left (\Delta_E(T_k^{(n)}(k-j),N_k^{(n)}[k-j])\right)
\] 
onto 
\begin{multline}
\nonumber
a(0)\cdot  L_{\mathrm{K},\alpha^*}^{[0],*}(f^*,\xi^*;\chi^j) \left (a \cdot L_{\mathrm{K},\alpha^*}^{[1]}(f^*,\xi^*;\chi^j)+{ \frac{b}{\varpi_E^n}} \cdot L_{\mathrm{K},\alpha^*}^{[0]}(f^*,\xi^*,\chi^j)\right )
\\
=
{
a(0)^2\cdot   L_{\mathrm{K},\alpha^*}^{[0],*}(f^*,\xi^*;\chi^j) \cdot 
 L_{\mathrm{K},\alpha^*}^{[1]}(f^*,\xi^*;\chi^j).
} 
\end{multline}
On the other hand, Theorem~\ref{thm: descent for tilt complex} yields
\[
i_{V_k(k-j)}\circ \Theta_{V_k({ k-}j)}^{(\alpha)} \left ( \Delta_E ( T^{(n)}_{k}({ k-}j),N_k^{(n)}[{ k-}j])\right ) =\mathbf L_{\Iw,\alpha}^*(T_{f}(j), N_{\alpha}[k-j],\mathds{1}, 0 ) \cdot
\mathbf L_{\Iw,\beta}(T_{f}(k-j), N_\beta[k-j],\mathds{1}, 0).
\]
Therefore, 
\begin{align*}
 L_{\mathrm{K},\alpha^*}^{[0],*}(f^*,\xi^*;\chi^j) &\,\cdot 
 L_{\mathrm{K},\alpha^*}^{[1]}(f^*,\xi^*;\chi^j) \\
 &\sim_p
C_{\mathrm{K}}\cdot \mathcal E_N(f^*;\chi^j)^2 \cdot \frac{b (\mathds 1,k-j)}{a (\mathds 1,k-j)}\cdot 
\mathbf L_{\Iw,\alpha}^*(T_{f}(k-j), N_{\alpha}[k-j],\mathds{1}, 0 ) \cdot
\mathbf L_{\Iw,\beta}(T_{f}(j), N_\beta[j],\mathds{1}, 0).
\end{align*}
Note that  Conjecture \ref{item_MCalpha} (which follows from \ref{item_MCinf} by Proposition~\ref{infinitesimaMC implies MC}) implies that 
\[
 L_{\mathrm{K},\alpha^*}^{[0],*}(f^*,\xi^*;\chi^j)
\sim_p
\mathcal E_N(f^*;\chi^j) \cdot
\mathbf L_{\Iw,\alpha}^*(T_{x_0}(k-j), N_{\alpha}[k-j],\mathds{1}, 0 )\,.
\]
Hence,
\[
 L_{\mathrm{K},\alpha^*}^{[1]}(f^*,\xi^*;\chi^j) \sim_p
C_{\mathrm{K}}\cdot \mathcal E_N({ f^*;\chi^j})\cdot 
{ \frac{b(\mathds 1,k-j)}{a (\mathds 1,k-j)}}\cdot 
 \mathbf L_{\Iw,\beta}(T_{f}(k-j), N_\beta[k-j],\mathds{1}, 0)\,,
 \]
 and the proof of our theorem is complete.
\end{proof}

\begin{corollary}
\label{cor: infinitesimal descent theorem}
 The formula \eqref{eqn_2023_08_14_1012} can be rewritten in the form  
 \[
 L^{[1]}_{\mathrm{K},\alpha^*}(f^*, \xi^*;\chi^j)^\pm 
\sim_p
C_{\mathrm{K}}\cdot \mathcal E_N(f^*;\chi^j) \cdot \frac{e_{p,\alpha^*}(f^*,\mathds 1,j)}{e_{p,\beta^*} (f^*,\mathds 1,j)} \cdot \mathbf L_{\Iw,\beta}(T_{f}(k-j),N_\beta[k-j],\mathds 1,0)^\pm\,.
\]
\end{corollary}
\begin{proof} This follows from the  equalities
\[
\frac{b(\mathds 1,k-j)}{a (\mathds 1,k-j)}=
\frac{e_{p,\alpha}(f,\mathds 1,k-j)}{e_{p,\beta} (f\mathds 1,k-j)}
=
\frac{e_{p,\alpha^*}(f^*,\mathds 1,j)}{e_{p,\beta^*} (f^*\mathds 1,j)}\,.
\]
\end{proof}

\begin{remark}
It follows from Theorem~\ref{thm_interpolative_properties}  (see especially Equations \eqref{eqn:condition for thm_interpolative_properties} and \eqref{formula:main theorem1}) that 
\[
L^{[1]}_{\mathrm{K},\alpha^*}(f^*, \xi^*;\chi^j)^\pm =C_K \cdot \mathcal E_N(f^*;\chi^j)\cdot \frac{e_{p,\alpha^*}(f^*,\mathds 1,j)}{e_{p,\beta^*} (f^*\mathds 1,j)}\cdot L_{\mathrm{K},\beta^*}(f^*, \xi^*;\chi^j)^\pm.
\]
Combining this  formula with Corollary~\ref{cor: infinitesimal descent theorem}, we deduce that
\[
L_{\mathrm{K},\beta^*}(f^*, \xi^*;\chi^j)^\pm \sim_p  \cdot \mathbf L_{\Iw,\beta}(T_{f}(k-j),N_\beta[k-j],\mathds 1,0)^\pm\,.
\]
This also follows from the slope-zero Main Conjecture \ref{conj:ordinary MC} via Iwasawa theoretic descent. 
This shows that the conjecture $\widetilde{\mathbf{MC}} (f)$
is compatible with  the slope-zero Main Conjecture $\mathbf{MC}(f,\beta)$.
\end{remark}

\chapter{Thick Selmer groups and derivatives of critical $p$-adic $L$-functions}
\label{chapter_critical_Selmer_padic_reg} 
Throughout \S\ref{chapter_critical_Selmer_padic_reg}, we shall work in the setting of \S\ref{subsubsec_221_04042022}. In particular, we assume throughout that the hypothesis \eqref{item_C4} holds true.

\section{Recap and motivation}
Before we describe our results in the present chapter, we summarize our results in the earlier portions and explain how our results in \S\ref{chapter_critical_Selmer_padic_reg} complement those.


Recall that in \S\ref{subsec_defn_critical_padic_L_eigencurve}, we have introduced a two-variable $p$-adic $L$-function $L_{\mathrm{K},\alpha}^\pm(\cX,\xi ) \in  \mathscr{H}_{\mathcal X}(\Gamma)$. We proved that we can write it in the form
\be
\label{eqn_taylor_expansion_1}
L_{\mathrm{K},\alpha}^{\pm}(\cX,\xi )=\ell (V_f(j))\,
  L_{\mathrm{K},\alpha}^{[0],\mathrm{imp},\pm}(f,\xi) +
X\cdot L_{\mathrm{K},\alpha}^{[1],\pm}(f,\xi)+  
\ldots\,,
\ee
where $\ell (V_f(j))= \left (\underset{i=1-k}{\overset{-1}\prod}\ell^\iota_i\right )$
(cf. Theorem~\ref{thm_interpolative_properties}(i)).

 The interpolation properties of $L_{\mathrm{K},\alpha}^{[1]}(f,\xi)$ in terms of the critical values of the Hecke $L$-function of $f$ (cf. Theorem~\ref{thm_interpolative_properties}(ii)) suggest that the \emph{thick $p$-adic $L$-function} $\widetilde{L}_{\mathrm{K},\alpha}^\pm(f,\xi ) :=L_{\mathrm{K},\alpha}^\pm (\cX,\xi )\pmod{X^2}$ also carries arithmetic information, that is not readily visible in $L_{\mathrm{K},\alpha}^{[0],\pm}(f,\xi ):=L_{\mathrm{K},\alpha}^\pm(\cX,\xi )\pmod X$. 

Based on this insight, we developed a theory of thick Selmer complexes in Chapter~\ref{chapter_main_conj_infinitesimal_deformation} and introduced what we call the thick fundamental line. We formulated the Main Conjecture \ref{item_MCinf} in terms of these to reflect the arithmetic properties of the thick $p$-adic $L$-function $L_{\mathrm{K},\alpha}^\pm (\cX,\xi )\pmod{X^2}$.

\subsection{Special values} 
\label{subsec_special_values_chapter_5}

{Let $j$ be an integer and $r$ a natural number. In what follows, we consider the function  ${L}_{\mathrm{K},\alpha}^\pm(\cX,\xi,\chi^s)$,
which is analytic on $j+p\Zp$, and we denote by 
$$
\left. \frac{\partial^r}{\partial s^r} {L}_{\mathrm{K},\alpha}^\pm(\cX,\xi,\chi^s)\right \vert_{ s=j}$$
its derivative at $s=j$.
}

\subsubsection{The non-central critical $j$}
\label{subsubsection_5111}
Let us now fix an integer $j\neq \frac{k}{2}$ with $1\leqslant j\leqslant k-1$. In view of \eqref{eqn_taylor_expansion_1}, we have the following linearization
of $L_{\mathrm{K},\alpha}^\pm(\cX,\xi ,\chi^s)$ about $X=0$, $s=j$:
\begin{align}\label{eqn_taylor_expansion_12}
\begin{aligned}
L_{\mathrm{K},\alpha}^{\pm}(\cX,\xi,\chi^s) \equiv\,
(s-j)\cdot \ell^*(V_f(j))\cdot L_{\mathrm{K},\alpha}^{[0],\mathrm{imp},\pm}(f,\xi,\chi^j) +X\cdot 
&L_{\mathrm{K},\alpha}^{[1],\pm}(f,\xi,\chi^j)\\
& \mod (X^2, X(s-j), (s-j)^2)\,.
\end{aligned}
\end{align}
Here, $\ell^*(V_f(j)):= \underset{{\substack{i=1-k \\ i\neq -j}}}{\overset{-1}\prod} 
(i+j)$.

 Theorem~\ref{thm_interpolative_properties} and Proposition~\ref{prop: comparision p-adic L-functions for alpha and beta} allows us to compute the linearization $\widetilde{L}_{\mathrm{K},\alpha}^{\pm}(f,\xi,\chi^s)$ of the thick $p$-adic $L$-function about $s=j$ in terms of the Hecke $L$-value $L (f,j)$, assuming that the $\cL$-invariant $\mathscr{L}^{\rm cr}(V_{f}(j))$ is nonzero.

\subsubsection{Non-critical $j$} 
\label{subsubsection_5112}
We may assume without loss of generality that $j\geq k$.  Granted the validity of the $\theta$-critical main conjecture \ref{item_MCalpha} (which is equivalent, assuming that  $\mathscr{L}^{\rm cr}_{\rm Iw}$ is nonzero, to the slope-zero main conjecture \ref{item_MCbeta}, cf. Proposition~\ref{prop: equivalence of main conjectures}), we have
$$L_{\mathrm{K},\alpha}^{[0],\pm}(f,\xi,\chi^j) \neq 0\,.$$
This value can be explained (up to a $p$-adic unit) in terms of fundamental arithmetic invariants; see Theorem~\ref{thm:descent thm: noncentral case}. 

In summary, assuming that $\mathscr{L}^{\rm cr}_{\rm Iw}(V_f)\neq 0$ and the slope-zero main conjecture \ref{item_MCbeta} holds true, the leading term of the thick $p$-adic $L$-function $\widetilde{L}_{\mathrm{K},\alpha}^{\pm}(f,\xi,\chi^s)$ about $s=j$ equals $L_{\mathrm{K},\alpha}^{[0],\pm}(f,\xi,\chi^j)$, whose arithmetic meaning is explained by Theorem~\ref{thm:descent thm: noncentral case}(i).

\subsubsection{Special value at the central critical point} 
\label{subsubsection_5113}
Let us assume that $k$ is an even integer. In view of our discussion in Sections \ref{subsubsection_5111}--\ref{subsubsection_5112}, the case that remains to understand the arithmetic meaning of $\widetilde{L}_{\mathrm{K},\alpha}^{\pm}(f,\xi)$ concerns the central critical point $j=\frac{k}{2}$.

Suppose first that $L (f,\frac{k}{2})\neq 0$. Then the treatment of $\widetilde{L}_{\mathrm{K},\alpha}^{\pm}(f,\xi,\chi^s)$ about $s=\frac{k}{2}$ is identical to the treatment of other critical values (cf. \S\ref{subsubsection_5111} and Theorem~\ref{thm_414_2022_0810_1712}). We will therefore assume until the end of \S\ref{subsubsection_5113} that $L (f,\frac{k}{2})=0$. Let us put $r:={\ord}_{s=\frac{k}{2}}L (f,s)$\,.

Let us assume that the conditions \eqref{item_BK_cc}, \eqref{item_slope_zero_ht} and \eqref{item_PR}  as well as the slope-zero main conjecture \ref{item_MCbeta} hold true. It follows from \eqref{eqn_taylor_expansion_1} combined with Theorem~\ref{thm_interpolative_properties}, Proposition~\ref{prop: comparision p-adic L-functions for alpha and beta}, Proposition~\ref{prop: consequences of higher PR}(iii), and the proof of Proposition~\ref{prop_22_04_26_1150} that
\begin{align}
\label{eqn_taylor_expansion_2}
\begin{aligned}
\widetilde{L}_{\mathrm{K},\alpha}^{\pm}(\cX,\xi,\chi^s) \equiv
&\left(s-\frac{k}{2}\right)^r
\cdot\dfrac{\ell^*(V_f(j))}{(r-1)!}\cdot  \frac{\partial^{r-1}}{\partial s^{r-1} } \left. L_{\mathrm{K},\alpha}^{[0],\mathrm{imp},\pm}(f,\xi,\chi^s)\right \vert_{s=\frac{k}{2}}\\
&\,\, + X\cdot \left(s-\frac{k}{2}\right)\cdot \frac{d}{d s }
\left. L_{\mathrm{K},\alpha}^{[1],\pm}(f,\xi,\chi^s)\right \vert_{s=\frac{k}{2}} \qquad \mod \left(X\left(s-\frac{k}{2}\right)^2, \left(s-\frac{k}{2}\right)^{r+1}\right).
\end{aligned}
\end{align}

Moreover, since the $\theta$-critical main conjecture \ref{item_MCalpha} holds true in this scenario, Theorem~\ref{thm_331_2022_04_29_1629}(ii) provides us with an interpretation of the partial derivative 
$$\frac{\partial^{r-1}}{\partial s^{r-1} }
\left. L_{\mathrm{K},\alpha}^{[0],\mathrm{imp},\pm}(f,\xi,\chi^s)  \right \vert_{s=\frac{k}{2}}$$ 
in terms of the fundamental arithmetic invariants (among which is the $p$-adic regulator $R_\alpha(T_{f}(k/2))$).

As the next task at hand, we need to compute the derivative
$$\frac{d}{d s } L_{\mathrm{K},\alpha}^{[1], \pm}(f,\xi,\chi^s)_{\vert_{s=\frac{k}{2}}}$$ 
in terms of $p$-adic heights. A strategy paralleling the discussion in the previous paragraph would require 
\begin{itemize} 
\item the truth of the thick Main Conjecture \ref{item_MCinf}, 
\item a descent formalism for the thick fundamental line, extending Theorem~\ref{thm_414_2022_0810_1712}(ii) to cover the case  when $j=\frac{k}{2}$ and $r>0$.
\end{itemize}
This doesn't seem tractable; see Remark~\ref{remark_4_15_2022_08_19_1420}. 

We will instead use a different method to compute $\frac{d}{d s }\widetilde{L}_{\mathrm{K},\alpha}^{\pm}(f,\xi,\chi^s)_{\vert_{s=\frac{k}{2}}}$ when $r=1$ in terms of $p$-adic heights on the thick Selmer complexes (cf. Theorem~\ref{thm_A_adic_regulator_formula_eigencurve}). In this regard, our results in the present chapter are supplementary to those we have established in Chapters~\ref{chapter_main_conjectures} and \ref{chapter_main_conj_infinitesimal_deformation}.

\begin{remark}
Proposition~\ref{prop_2022_04_26_13_00} and Corollary~\ref{cor_Iw_crit_L_inv_non_trivial_analytic_rank_one} shows that when $r=1$, \eqref{eqn_taylor_expansion_2} is valid under a weaker set of assumptions:
\begin{itemize}
\item The hypothesis \eqref{item_PR1} holds.
\item The height pairing $\langle\,,\,\rangle_{\beta}$ is nonzero.
\item Either \eqref{item_SZ1} or \eqref{item_SZ2} is valid, to guarantee that the slope-zero main conjecture \ref{item_MCbeta} holds true; cf. \cite{skinnerurbanmainconj, xinwanwanhilbert}, \cite[Theorem 2.5.2]{skinnerPasificJournal2016}, \cite[Theorem 15.2]{kato04}.
\end{itemize} 
As we have explained in the proof of Corollary~\ref{cor_2022_07_01_1502},  all these three conditions hold true when $f=f_A$ is the normalized newform attached to an elliptic curve $A/\QQ$. 
\end{remark}

\subsection{This chapter} Our main objective in \S\ref{chapter_critical_Selmer_padic_reg} is twofold. 

Our first goal is to illustrate the relation of thick Selmer groups and the Bloch--Kato Selmer groups in an explicit manner; see in particular Corollary~\ref{cor: computation selmer in rank one case}.  

Our second goal is to compute, when $r=1$, the height pairings we have introduced in \S\ref{subsec_height_cyclotomic} on thick Selmer groups; cf. Theorem~\ref{thm_A_adic_regulator_formula_eigencurve}.


\section{Selmer complexes revisited}
\label{sec_thick_Selmer_complexes_5_2}


\subsection{Notation} Throughout \S\ref{sec_thick_Selmer_complexes_5_2}, we work in the setting of Section~\ref{sect:eigenspace-transition}.
We study the arithmetic invariants associated to the following objects:
\begin{itemize}
\item A free $\widetilde{E}$-module 
$\tildeV$ of rank $2$
equipped with a continuous action of $G_{\QQ,S}.$ 


\item We set $V=\ker (\tildeV \xrightarrow{X}\tildeV)$ and   assume that the 
restriction of $V$ on the decomposition group at $p$ splits as 
$$
V=V^{(\alpha)} \oplus V^{(\beta)},
$$
where $V^{(\alpha)}$ and $V^{(\beta)}$  are one-dimensional crystalline representations with Hodge--Tate weights $1-k$ and $0$ respectively, with $\alpha$ and $\beta$ the $\varphi$-eigenvalues on $\Dc (V^{(\alpha)})$ and  $\Dc (V^{(\beta)}).$ We assume that  $\alpha\neq p^{1-k}$.  Set $\bD^{(\alpha)}:=\DdagrigE (V^{(\alpha)})$ and 
\begin{equation}
\label{eqn_202208121009}
\bD:=t^{k-1}\bD^{(\alpha)}.
\end{equation}

\item A non-saturated crystalline $(\varphi,\Gamma)$-submodule $\widetilde\bD\subset \bD_{\mathrm{rig},\widetilde E}^\dagger(\tildeV)$ of rank one over the relative Robba ring $\cR_{\widetilde E}$ such that 
\[
\ker (\widetilde\bD \xrightarrow{X}\widetilde\bD)\simeq \bD.
\]
 \end{itemize}
In this section, we assume that $k$ is even. We are interested
in the central twist of the representations $V$ and $\tildeV$ and set
\[
\begin{aligned}
&\Vcc:=V(\chi^{\frac{k}{2}}), &&\tildeVcc:=V(\chi^{\frac{k}{2}})
\\
&\Dcc:=\bD (\chi^{\frac{k}{2}}),   &&\tildeDcc:=\widetilde\bD (\chi^{\frac{k}{2}}).
\end{aligned}
\]
The $G_{\Qp}$-representation $\Vcc$ is split in our setting:
\be
\label{eqn_17_11_1123}
\left.\Vcc\right\vert_{G_{\Qp}}=\Vcc^{(\alpha)}\oplus \Vcc^{(\beta)}\,.
\ee
The  associated $(\varphi,\Gamma)$-modules can be explicitly described:
\begin{equation}
\DdagrigE ( \Vcc^{(\alpha)})=\CR_E(\chi^{1-\frac{k}{2}}\nu_\alpha)\qquad,  \qquad \DdagrigE ( \Vcc^{(\beta)})=\CR_E(\chi^{\frac{k}{2}}\nu_\beta)\,,
\nonumber
\end{equation}
where $\nu_\alpha$ and $\nu_\beta$ are unramified characters such that $\nu_\alpha(\Fr_p)=\alpha / p^{k-1}$ and $\nu_\beta(\Fr_p)=\beta.$ 

\subsection{Punctual Selmer complexes}
 \label{subsubsec_Degenerate_local_conditions_1}
 \subsubsection{}


Our main objective in \S\ref{subsubsec_Degenerate_local_conditions_1} is to 
examine the Selmer complexes  $\RG (\Vcc, \Dcc^{(\alpha)})$ and  $\RG (\Vcc, \Dcc)$ associated to the data $(\Vcc,\Dcc^{(\alpha)})$ and $(\Vcc,\Dcc)$ as in Section~\ref{subsec_selmer_complexes}. In particular, we  illustrate the degenerate properties of $\RG (\Vcc, \Dcc)$. This will serve as an added justification for our definition and extensive discussion in Section~\ref{subsec_thick_Selmer_groups_16_11} of the thick Selmer complexes. We first examine the Selmer complex  $\RG (\Vcc, \Dcc^{(\alpha)}).$   

\begin{proposition}
\label{prop:comparision improved Selmer}
There exists a canonical isomorphism
\be\label{eqn_2022_08_23_1156}
\RG (\Vcc, \Dcc^{(\alpha)})\simeq\RG^{\mathrm{imp}} (\Vcc,\alpha)
\ee
in the derived category. Here $ \RG^{\mathrm{imp}} (\Vcc,\alpha)$ is the Selmer complex of Perrin-Riou given as in \S\ref{subsubsec_3524_2022_08_16}, where we take $V$ in place of $V_f$.
\end{proposition}
\begin{proof} We have the following commutative diagram of local conditions at $p$:
\[
\xymatrix{\Dc (\Vcc^{(\alpha)})[-1] \ar[d]_{\bExp_{\alpha,k/2,0}} 
\ar[r] &\RG (\Qp,\Vcc)\ar@{=}[d]\\
\RG (\Qp, \Dcc^{(\alpha)}) \ar[r] &\RG (\Qp,\Vcc).
}
\]
where the left vertical map $\bExp_{\alpha,k/2,0}$ is induced by the 
map $\bExp_{\alpha,k/2}\,:\,\mathfrak D(\Dcc^{(\alpha)})[-1] \rightarrow \RG_\Iw (\Dcc^{(\alpha)}).$ In degree one, this map coincides with the exponential map 
\[
\Exp_{\alpha,k/2,0}\,:\,\Dc (\Dcc^{(\alpha)})\lra H^1(\Qp,\Dcc^{(\alpha)})\,.
\]
Equation \eqref{eqn:specialization of PR formulae} reads:
\[
\Exp_{\alpha,k/2,0}(x)=
\frac{(-1)^{k/2}}{(k/2-1)!} (\exp^*_{\Dcc^{(\alpha)}})^{-1} \left ((1-\alpha^{-1}p^{k/2-1})(1-\alpha p^{-k/2})^{-1}(x)\right ).
\]
Since $v_p(\alpha)=k-1$ and $\exp^*_{\Dcc^{(\alpha)}}\,:\,H^1(\Qp,\Dcc^{(\alpha)}) \rightarrow\Dc (\Dcc^{(\alpha)})$ is an isomorphism, the map $\Exp_{\alpha,k/2,0}$
is an isomorphism. Moreover $H^i(\Qp,\Dcc^{(\alpha)})=0$ for $i\neq 1.$ Thus 
$\bExp_{\alpha,k/2,0}$ is a quasi-isomorphism, and therefore  the corresponding Selmer complexes are quasi-isomorphic. 
\end{proof}

We turn our attention to the cohomology of $\RG (\Qp,\Dcc).$

\begin{lemma} 
\label{lemma: coboundary map}
We have $H^i(\Qp,\Dcc)=0$ (for $i\neq 1$) and $\dim_EH^1(\Qp,\Dcc)=1.$
Moreover, the natural map $H^1(\Qp,\Dcc)\rightarrow H^1(\Qp,\Dcc^{(\alpha)})$
induced by the inclusion $\Dcc \hookrightarrow \Dcc^{(\alpha)}$, is the zero map.
\end{lemma}
\begin{proof} The computation of $H^i(\Qp,\Dcc)$ follows easily from the Poincar\'e
duality and the Euler characteristic formula for the cohomology of $(\varphi,\Gamma)$-modules.

 Let us consider the following commutative diagram:
\[
\xymatrix{
\mathscr D_{\textrm{cris}}(\Dcc) \ar[d]_-{\exp} \ar[r]^-{\sim} & \Dc (\Vcc^{(\alpha)}) \ar[d]^-{\exp}\\
H^1(\Qp,\Dcc) \ar[r] &H^1(\Qp, \Vcc^{(\alpha)})
}
\]
The Hodge--Tate weights of $\Dcc$ are positive, and therefore 
the left vertical map is an isomorphism. The right vertical map is the zero map, since the Hodge--Tate weight of the $1$-dimensional $G_{\Qp}$-representation $\Vcc_{f}^{(\alpha)}$ is negative. Therefore the bottom row is the zero map.
\end{proof}

\begin{remark} Let us consider the tautological exact sequence
\[
0\lra \Dcc \lra \Dcc^{(\alpha)} \lra \,^c\DD \lra 0\,,
\]
where ${}^c\DD :=\Dcc^{(\alpha)}/\Dcc.$
It follows from Lemma~\ref{lemma: coboundary map} that the map $\partial_0 \,:\, H^0(\Qp, \,^c\DD) \rightarrow 
H^1(\Qp, \Dcc)$ is an isomorphism. It can be described explicitly as follows.
Let $ x\in H^0(\Qp, \,^c\DD)$ and let $d\in \Dcc^{(\alpha)}$ denote any lift of $x$.  Set $y=((\varphi-1)d,(\gamma-1)d).$ Then $y \in C^1_{\varphi,\gamma}({}^c\bD)$ and 
$
\partial_0(x)=y.
$

\end{remark}

\begin{proposition}
\label{prop_degenerate_local_conditions}  
The Selmer group $\mathbf R^1\boldsymbol{\Gamma}(\Vcc,\Dcc)$ splits canonically into the direct sum
\begin{equation}
\label{eqn_prop_degenerate_local_conditions} 
\begin{aligned}
\mathbf R^1\boldsymbol{\Gamma}(\Vcc,\Dcc) &\stackrel{\sim}{\lra} H^1_{0}(\Vcc) \oplus H^1(\Qp, \Dcc),\\
[(z,(z_\ell^+),(\lambda_\ell))] &\longmapsto  \quad\quad\left(\,[z]\,\,,\,\, [z_p^+]\, \right ). 
\end{aligned} 
\end{equation}
\end{proposition}

\begin{proof} 
This is proved in a manner identical to \cite[Proposition~7.1.2]{benoisheights}.  The exact sequence \eqref{exact sequence to compute Selmer groups} combined with the fact that
$
H^0(\Qp,\Vcc )=
\{0\}
$ 
shows that
\[
\mathbf R^1\boldsymbol{\Gamma}(\Vcc,\Dcc)=\ker \left ( H^1(\QQ,\Vcc) 
\oplus \left (\underset{\ell\in S\setminus \{p\}} \bigoplus 
H^1_{\rm f}(\QQ_\ell, \Vcc)\oplus H^1(\Qp,\Dcc) \right )\xrightarrow{f_S-g_S} \underset{\ell\in S} \bigoplus 
H^1(\QQ_\ell, \Vcc) \right ).
\]
Since $H^1_{\rm f}(\QQ_\ell, \Vcc) \xrightarrow{f_\ell} H^1(\QQ_\ell, \Vcc)$ is the natural inclusion for $\ell\in S\setminus\{p\}$ and $H^1(\Qp, \Dcc)\xrightarrow{g_p} H^1(\Qp, \Vcc)$ is the zero-map by Lemma~\ref{lemma: coboundary map},
we have a natural identification
\begin{align*}
\mathbf R^1\boldsymbol{\Gamma}(\Vcc,\Dcc)\,\,\,&= \,\ker \left ( H^1(\QQ,\Vcc)\lra \underset{\ell\in S\setminus \{p\}}\bigoplus \frac{H^1(\QQ_\ell,\Vcc)}{H^1_{\rm f}(\QQ_\ell, \Vcc)} \oplus H^1(\Qp,\Vcc)\right )\oplus H^1(\Qp,\Dcc) \\
&\simeq H^1_{0}(\Vcc)\oplus H^1(\Qp, \Dcc)\,, \\
[(z,(z_\ell^+),(\lambda_\ell))] &\longmapsto  \quad\quad\left(\,[z]\,\,,\,\, [z_p^+]\, \right )
\end{align*}
where the second equality follows from the definition of the fine Selmer group $H^1_{0}(\Vcc)$, cf. \S\ref{subsubsec_defn_fine_selmer_relaxed_Selmer}. This completes the proof of our proposition.
\end{proof}

\begin{remark}
\label{remark_prop_degenerate_local_conditions}
Let $f$ be an eigenform with the critical stabilisation $f^{\alpha},$
and let $V:=V_f$, $\bD:=t^{k-1}\bD^{(\alpha)}.$ Proposition~\ref{prop_degenerate_local_conditions} tells us that $R^1\boldsymbol{\Gamma}(\Vcc,\Dcc)$ extends the fine Selmer group $ H^1_{0}(\Vcc)$, rather than the Bloch--Kato Selmer group, by the local cohomology group $H^1(\Qp, \Dcc)$. This is an indication as to why the Selmer complex $\RG(\Vcc,\Dcc)$ may not capture some of the arithmetic invariants of $f$ (that the Bloch--Kato Selmer group encodes).
\end{remark}



\subsubsection{} 
\label{subsec: punctual heights}
We continue to work in the setting of \S\ref{subsubsec_Degenerate_local_conditions_1}. Assume that we have a 
Galois equivariant non-degenerate bilinear pairing 
\begin{equation}
\label{eqn: V' and V pairing}
V' \otimes V \lra E(1-k).
\end{equation}
It follows from the non-degeneracy of this pairing that the $G_{\Qp}$-representation $V'$ also splits into the direct sum 
$
V'=V'^{(\alpha)} \oplus V'^{(\beta)}, 
$
where $V'^{(\alpha)}$ (respectively $V'^{(\beta)}$) is the orthogonal complement of
$V^{(\beta)}$ (respectively $V^{(\alpha)}$). Set $\bD'^{(\alpha)}:=\DdagrigE (V'^{(\alpha)})$ and $\bD':=t^{k-1}\bD'^{(\alpha)}.$ 
The pairing \eqref{eqn: V' and V pairing}   induces 
\[
\Vcc' \otimes \Vcc \lra E(1)\,,
\]
which gives rise to the $p$-adic height pairing
\[
\left <\,,\,\right >_{\Dcc',\Dcc}=\left <\,,\,\right >^{\cyc}_{\Dcc',\Dcc}\,:\, \RG (\Vcc',\Dcc') \otimes \RG (\Vcc,\Dcc) \lra E[-2]\,.
\]
\begin{proposition}
\label{prop_factor_degenerate_height_pairing}
 We have a commutative diagram 
\begin{equation}
\nonumber
\xymatrix{
\mathbf R^1\boldsymbol{\Gamma}(\Vcc',\Dcc')\otimes \mathbf R^1\boldsymbol{\Gamma}(\Vcc,\Dcc)\ar[rr]^-{\left <\,\,,\,\,\right >_{\Dcc',\Dcc}}
\ar[d] & &E \ar@{=}[d]\\
H^1_{0}(\Vcc')\otimes H^1_{0}(\Vcc)\ar[rr]
_-{\left <\,\,,\,\,\right >_{\alpha}} 
& &E},
\end{equation}
where the left vertical map is induced by the decomposition \eqref{eqn_prop_degenerate_local_conditions}, and the $p$-adic height pairing $\left <\,\,,\,\,\right >_{\alpha}$ is induced by the composition
\[
H_0^1(\Vcc) \lra \Sha_\Iw (\Vcc)^\Gamma \lra
\Sha_\Iw(\Vcc)_\Gamma \lra H_0^1(\Vcc')^*\,,
\]
(cf. \S\ref{subsubsec_3531_2022_08_23_1647}).
\end{proposition}
\begin{proof}
Employing Part (iii) of Proposition~\ref{proposition_general_properties_of_heights} with $\bD_1'=0$ and taking $\bD_2'=\bD'$ to be the $(\varphi,\Gamma)$-module associated to any splitting module, we infer that the restriction of any $p$-adic height pairing (associated to any splitting module) to $H^1_0$ is independent of the choice of the said splitting module, and therefore coincides with $\left <\,\,,\,\,\right >_{\alpha}.$ 

Since the $p$-adic height pairing factors through the Bockstein morphisms by its definition (cf. Equation \eqref{eqn_defn_height_cyclo}),
the proof  follows from the commutative diagram
\[
\xymatrix{
\mathbf R^1\boldsymbol{\Gamma}(\Vcc',\Dcc') \ar[r] &\mathbf R^2\boldsymbol{\Gamma}(\Vcc',\Dcc')\\
H^1(\Qp, \Dcc') \ar@{^(->}[u] \ar[r] & H^2(\Qp,\Dcc') \ar[u]
}
\]
where the horizontal maps are the Bockstein maps, together with the fact that $H^2(\Qp,{ \Dcc'})=\{0\}$, which we have checked in Lemma~\ref{lemma: coboundary map}.
\end{proof}

\begin{corollary}
\label{cor_prop_factor_degenerate_height_pairing}
In the setting of Theorem~\ref{prop_factor_degenerate_height_pairing}, assume that $H^1_{0}(\Vcc)=\{0\}$. Then the height pairing 
$$\left <\,\,,\,\,\right >_{\Dcc',\Dcc}\,\,:\,\mathbf R^1\boldsymbol{\Gamma}(\Vcc',\Dcc')\otimes \mathbf R^1\boldsymbol{\Gamma}(\Vcc,\Dcc)\lra E$$ 
on the extended Selmer groups vanishes.
\end{corollary}

\begin{remark}
\label{remark_cor_prop_factor_degenerate_height_pairing}
Corollary~\ref{cor_prop_factor_degenerate_height_pairing} is further evidence as to why the Selmer complex attached to a $\theta$-critical modular form $f_\alpha$ does not capture all the arithmetic invariants associated with $f$. Together with Remark~\ref{remark_prop_degenerate_local_conditions}, they serve as a justification for the study of what we call thick Selmer complexes in \S\ref{subsec_thick_Selmer_groups_16_11} below.
\end{remark}

\subsection{Thick Selmer groups}

\subsubsection{}
In this section, we study various Selmer groups naturally attached to the representation $\tildeVcc=\tildeV (\chi^{k/2}).$ The relevant Selmer
complex will be introduced in the next section. 
\begin{defn}
The thick Greenberg Selmer group is defined by setting
$$
\Sel_{\alpha}(\tildeVcc) := 
\ker \left ({H^1_S( \tildeVcc)} 
\lra 
\underset{\ell\neq p}\bigoplus \frac{H^1(\QQ_\ell,\tildeVcc)}{ H^1_{\rm f}(\QQ_\ell,\tildeVcc)} \bigoplus
\frac{H^1(\QQ_p,\tildeVcc)}{ H^1_{\rm f}(\QQ_p,X\cdot \tildeVcc)}
\right ).$$
\end{defn}

 The sequence \eqref{eqn_4_1_2022_03_16} induces the exact sequence
\be
\label{eqn_Selmer_base_change}
0\lra {H^1_S(\Vcc)} \xrightarrow{[X]} {H^1_S(\tildeVcc}) \lra {H^1_S(\Vcc)}
\ee
where the injection on the left is because ${H^0_S(\Vcc)}=\{0\}$.

\begin{proposition}
\label{prop_Selmer_base_change}
Assume that $H^0(\QQ_\ell,\Vcc)=0$ for all $\ell\in S\setminus\{p\}.$
Then the sequence \eqref{eqn_Selmer_base_change} gives rise to an exact sequence
$$0\lra {H^1_{\rm f}(\Vcc)}\xrightarrow{[X]} \Sel_{\alpha}(\tildeVcc)\rightarrow  H^1_{0}(\Vcc)\,.$$
\end{proposition}

\begin{proof}
Let $\ell\neq p$ be a prime in $S$. Since $\dim_E H^1_{\mathrm{f}}(\QQ_\ell,\Vcc)= 
\dim_E H^0(\QQ_\ell,\Vcc),$ we have that 
\[
H^1_{\mathrm{f}}(\QQ_\ell,\Vcc)=0, \qquad \forall \ell\neq p.
\]
It then follows from the short exact sequence 
\[
H^1_{\mathrm{f}}(\QQ_\ell,\Vcc) \lra H^1_{\mathrm{f}}(\QQ_\ell,\tildeVcc)
\lra H^1_{\mathrm{f}}(\QQ_\ell,\Vcc)
\]
that $ H^1_{\mathrm{f}}(\QQ_\ell,\tildeVcc)=0$. For $\ell=p,$ the tautological exact sequence
\[
0\lra X \Vcc^{(\alpha)} \lra \tildeVcc/X\Vcc^{(\beta)}
\lra \Vcc \lra 0
\]
induces the exact sequence
\[
0\lra H^1(\Qp,  X \Vcc^{(\alpha)})\lra 
H^1(\Qp, \tildeVcc/X\Vcc^{(\beta)}) \lra 
H^1(\Qp, \Vcc),
\]
where $H^1(\Qp,  X \Vcc^{(\alpha)})\simeq H^1(\Qp,  X \Vcc)/H^1_{\mathrm{f}}(\Qp,  X \Vcc).$ Let us set
\[
\begin{aligned}
&U_{\mathrm{f}}:= \underset{\ell\in S\setminus\{p\}}\oplus H^1(\QQ_\ell,\Vcc)
\oplus H^1(\Qp,  X \Vcc^{(\alpha)}),\\
&\widetilde U:=\underset{\ell\in S\setminus\{p\}}\oplus H^1(\QQ_\ell,\tildeVcc)\oplus H^1(\Qp, \tildeVcc/X\Vcc^{(\beta)}),\\
&U_0:=\underset{\ell\in S\setminus\{p\}}\oplus H^1(\QQ_\ell,\Vcc)\oplus H^1(\Qp, \Vcc).
\end{aligned}
\]

We then have a commutative diagram with exact rows
\[
\xymatrix{0 \ar[r]& {H^1_S(\Vcc)} \ar[r] \ar[d]
& {H^1_S(\tildeVcc)}\ar[r] \ar[d] & {H_S^1(\Vcc)} \ar[d]\\
0\ar[r] & U_{\mathrm{f}}
\ar[r]
 &\widetilde U \ar[r]
  &U_0\,. 
}
\]
The exact sequence of the kernels of vertical maps reads
\[
0\lra {H^1_{\rm f}(\Vcc)} \xrightarrow{[X]} \Sel_{\alpha}(\tildeVcc)\xrightarrow{\,\pi\,}  H^1_{0}(\Vcc)\,,
\]
and the proposition is proved.

\end{proof}

\subsection{Thick Selmer complexes}
\label{subsec_thick_Selmer_groups_16_11}


In this Section, we are interested in Selmer complexes attached to the representation
$\tildeVcc:=\,\tildeV (\chi^{k/2}).$  
The complex $\RG(\tildeVcc,\tildeDcc)$ (defined following the formalism of \S\ref{subsec_selmer_complexes}) will be called the \emph{thick Selmer complex} and its cohomology groups the \emph{thick (extended) Selmer groups}. 

Our primary objective in \S\ref{subsec_thick_Selmer_groups_16_11} is to compute the thick Selmer group $\mathbf{R}^1\boldsymbol{\Gamma}(\tildeVcc,\tildeDcc)$ explicitly and show that it is indeed the extension of the true (Bloch--Kato) Selmer group $H^1_{\rm f}(\Vcc)$ in an important case of interest (e.g. when the Hecke $L$-function of the dual form $f^*$ has at most a simple zero at $\frac{k}{2}$). This will be the first indication that the thick Selmer complexes are the correct algebraic objects that capture arithmetic properties of the refinement $(f,\alpha)$ of $f$ (in the sense of Bella\"iche), including their extremal exceptional zeros as well as those encoded in the Bloch--Kato Selmer group (which conjecturally explain the leading terms of complex $L$-functions). 

We continue to work in the setting of \S\ref{subsubsec_Degenerate_local_conditions_1}. In particular, we assume that the conditions \eqref{item_C1}--\eqref{item_C3} and \eqref{item_C4} hold true. 

\begin{proposition}
\label{lemma_computation_of_heights_bis}
The following assertions are valid:

\item[i)] $H^i(\Qp,\tildeDcc)=0$ for $i\neq 1$, and 
the exponential map induces an isomorphism
\[
\exp\,:\, \DCc (\tildeDcc)\xrightarrow{\sim} H^i(\Qp,\tildeDcc).
\]

\item[ii)] We have a quasi-isomorphism 
$$\RG(\tildeVcc , \tildeDcc)\simeq \RG (\tildeVcc,\alpha)\,,$$ 
 where the complex $ \RG (\tildeVcc,\alpha)$ is  as in \S\ref{subsubsec4251_2022_08_16}, where we replace $V_k$ with $\tildeV.$ 
\end{proposition}

\begin{proof}
\item[i)] On taking the cohomology of the exact sequence
$$0\lra \Dcc\xrightarrow{[X]}  \tildeDcc \lra   \Dcc\lra 0$$
and taking into account that $H^i(\Qp, \Dcc)=0$ for $i\neq 1,$ we 
obtain that $H^i(\Qp,\tildeDcc)=0$ for $i\neq 1.$ Moreover, $\tildeDcc$ is crystalline of  weight $k/2 \geqslant 1$ and $\alpha\neq p^{k/2-1}.$ Therefore, the exponential map is an isomorphism, as claimed. 

\item[ii)] The proof is entirely analogous to the proof of Proposition~\ref{prop:comparision improved Selmer}. We have the following commutative diagram of local conditions at $p$:
\[
\xymatrix{\DCc (\tildeDcc)[-1] \ar[d]_{\bExp_{\tildeDcc,0}} 
\ar[r]^-{\widetilde{\mathrm{E}}_{\alpha,k/2}^{(\mathds{1})}} &\RG (\Qp,\tildeVcc)\ar@{=}[d]\\
\RG (\Qp, \tildeDcc) \ar[r] &\RG (\Qp,\tildeVcc)\,.
}
\]
Here, the horizontal maps correspond to the local conditions at $p$ for the complexes $\RG (\tildeVcc,\alpha)$ and $\RG(\tildeVcc , \tildeDcc)$, respectively. Recall that the $(\varphi,\Gamma)$-module $\tildeDcc$ is not saturated in $\DdagrigE(\tildeVcc)$ and has Hodge--Tate weights $(k/2,k/2).$ Therefore, the vertical map on the left is an isomorphism in degree one.
Moreover $\mathbf{R}\boldsymbol{\Gamma}^i (\Qp, \tildeDcc)=0$
for $i\neq 1.$ Thence, the vertical map on the left is a quasi-isomorphism. 
This concludes the proof of this portion.

\end{proof}

We will compute $ \RG (\tildeVcc,\alpha)$ (in turn also $\RG(\tildeVcc , \tildeDcc)$), assuming that 
{$H^1_0(\Vcc)=0.$}
In this respect, {the present section}
can be thought of as a supplement to \S\ref{chapter_main_conj_infinitesimal_deformation}.
Set $\widetilde D=\DCc (\widetilde{\bD})$ and let us put 
${}^c\widetilde D:=\widetilde D[k/2]$.

\begin{proposition}
\label{prop_exceptional_zero_like_extension}
Assume that the hypothesis \eqref{item_C4} holds true and that $H^1_0(\Vcc)=0$. Then:

\item[i)]{}
We have the following natural short exact sequences:
\begin{equation}
\label{eqn:computation of tilde Selmer}
0 \lra X{}^c\widetilde D   \lra \mathbf{R}^1\boldsymbol\Gamma (\tildeVcc,\alpha)
\lra \mathrm{Sel}_\alpha (\tildeVcc) \lra 0\,.
\end{equation}

\item[ii)]{} If in addition the condition \eqref{item_Sha} holds for $V$, i.e. if the map
\[
H^1_{\rm f}(V) \lra H^1_{\rm}(\Qp,V)
\]
is nonzero, then there exists an exact sequence
\begin{equation}
\label{eqn:computation of tilde Selmer-bis}
0\lra H^1_0(\tildeVcc) \lra \mathbf{R}^1\boldsymbol\Gamma 
(\tildeVcc,\alpha) \lra {}^c\widetilde D \lra 0\,.
\end{equation} 
\end{proposition}

\begin{proof}
\item[i)]  Since $\RG (\tildeVcc,\alpha):=\RG_{\Iw}(\tildeVcc,\alpha)\otimes^{\mathbf L}_{\CH(\Gamma)}E,$ we have
$$\mathbf{R}^1\boldsymbol{\Gamma}(\tildeVcc,\alpha)=
\ker \left ( {H^1_S(\tildeVcc)}
\oplus \left (\underset{\ell\in S\setminus \{p\}} \bigoplus 
H^1_{\rm f}(\QQ_\ell, \tildeVcc)\oplus {}^c\widetilde D \right )\xrightarrow{f_S-g_S} \underset{\ell\in S} \bigoplus 
H^1(\QQ_\ell, \tildeVcc) \right ),$$
where $g_p=\widetilde{\mathrm E}_{\alpha,k/2}^{(\mathds{1})}$ (cf. the proof of 
Proposition~\ref{lemma_computation_of_heights_bis}).
Thanks to this description of $\mathbf{R}^1\boldsymbol{\Gamma}(\tildeVcc  ,\alpha)$ and since $g_\ell$ is injective for $\ell\neq p$, the module $\mathbf{R}^1\boldsymbol{\Gamma}(\tildeVcc  ,\alpha)$ fits in an exact sequence
\begin{align}
\begin{aligned}
\label{seq_Sel_tildeVk_bis}
0\lra \ker&\left({}^c\widetilde D\xrightarrow{g_p}  H^1(\QQ_p, \tildeVcc) \right)\lra \mathbf{R}^1\boldsymbol{\Gamma}(\tildeVcc,\alpha)\\
&\qquad\qquad\qquad\lra  \ker\left({H^1_S(\tildeVcc)}\xrightarrow{\res_S}  \frac{\underset{\ell\in S} \bigoplus 
H^1(\QQ_\ell, \tildeVcc)}{{\rm im}(g_S)}\right)\lra 0\,.
\end{aligned}
\end{align}
We note that 
\begin{equation}
\label{eqn_image_of_g_l}
{\rm im}(g_\ell)=\begin{cases}  H^1_{{\rm f}}(\QQ_\ell,\tildeVcc)&  \hbox{ if }  \ell\neq p\,,\\
H^1_{{\rm f}}(\QQ_p,X\cdot \tildeVcc)& \hbox{ if } \ell=p\,,
\end{cases}
\end{equation}
where this follows from definitions if $\ell\neq p$ (cf. \S\ref{subsubsec_1211_2023_07_06_1328}), whereas in the case $\ell=p$, we refer to Corollary~\ref{cor: transition for exponentials}. Therefore, the rightmost term in \eqref{seq_Sel_tildeVk_bis} equals $\Sel_{\alpha}(\tildeVcc) $. 


Using \eqref{eqn_171_11_11_31} and Proposition~\ref{prop: comparison of exponentials for different eigenvalues}
we observe that
$$\ker\left({}^c\widetilde D \xrightarrow{\,g_p\,}  H^1(\QQ_p, \tildeVcc) \right)=
X{}^c\widetilde D =
\Dc ({}^cV^{(\alpha)})\,.$$
This combined with the discussion in the previous paragraph completes the proof of the exactness of the first sequence. 

\item[ii)] It follows from definitions that $H^1_0(\tildeVcc) \subset \mathbf{R}^1\boldsymbol{\Gamma}(\tildeVcc,\alpha).$
Set $H^1_{/0}(\tildeVcc):= H^1_S(\tildeVcc)/H^1_0(\tildeVcc).$  We then have an exact sequence
\begin{equation}
\label{eqn: auxiliary exact sequence thick selmer}
0\lra H^1_0(\tildeVcc) \lra \mathbf{R}^1\boldsymbol{\Gamma}(\tildeVcc,\alpha) \lra 
\ker \left (H^1_{/0}(\tildeVcc)\oplus {}^c\widetilde D \lra H^1(\Qp,\tildeVcc)
 \right )\lra 0\,.
\end{equation}
The localisation  map $H^1_{/0}(\tildeVcc) \rightarrow H^1(\Qp,\tildeVcc)$ is injective, and it follows from the exact sequence \eqref{eqn_Selmer_base_change} and the condition \eqref{item_Sha} that 
\[
\mathrm{im} \left (H^1_{/0}(\tildeVcc) \lra H^1(\Qp,\tildeVcc)\right )=
H^1_{\rm f} (\Qp, X\tildeVcc)\,.
\]
On the other hand, it follows from \eqref{item_C4} that 
\[
\mathrm{im} \left ({}^c\widetilde D \xrightarrow{\,g_p\,} H^1(\Qp, \tildeVcc)\right )=
H^1_{\rm f}(\Qp, X\tildeVcc).
\]
Therefore, the map
\[
\ker \left (H^1_{/0}(\tildeVcc)\oplus {}^c\widetilde D \lra H^1(\Qp,\tildeVcc)
 \right ) \lra {}^c\widetilde D, \qquad (x,y) \mapsto y
\]
is an isomorphism. Using this isomorphism, we can rewrite the exact sequence \eqref{eqn: auxiliary exact sequence thick selmer} as
\eqref{eqn:computation of tilde Selmer-bis}. This concludes the proof of our proposition.
\end{proof}

\subsubsection{} Using the isomorphisms $X{}^c\widetilde D\simeq  \DCc (X \tildeDcc) \overset{\exp}\simeq H^1(\Qp,X\tildeDcc)$ and 
$\mathbf{R}^1\boldsymbol{\Gamma}(\tildeVcc  ,\alpha)\simeq 
\mathbf{R}^1\boldsymbol{\Gamma}(\tildeVcc  ,\tildeDcc)$,
we can rewrite the exact sequence \eqref{eqn:computation of tilde Selmer} in the following form:
\[
0 \lra H^1(\Qp,X\tildeDcc)  \lra 
\mathbf{R}^1\boldsymbol{\Gamma}(\tildeVcc  ,\tildeDcc) 
\lra 
\Sel_{\alpha}(\tildeVcc) \lra 0\,.
\]
We record the following consequence of this fact:

\begin{corollary} 
\label{cor: computation selmer in rank one case}
Assume that $H^1_0(V)=0$ and the condition \eqref{item_C4}
holds. Then there exist an exact sequence 
\[
0 \lra H^1(\Qp,X\tildeDcc)  \lra 
\mathbf{R}^1\boldsymbol{\Gamma}(\tildeVcc  ,\tildeDcc) 
\xrightarrow{\,\Delta\,} 
H^1_{\rm f}(\Vcc) \lra 0\,.
\]
If, in addition, the condition \eqref{item_Sha} holds, then 
we have an isomorphism 
\begin{equation}
\nonumber 
\begin{aligned}
&\mathbf{R}^1\boldsymbol{\Gamma}(\tildeVcc  ,\tildeDcc)\xrightarrow{\,\sim\,}  {}^c\widetilde D,\\
&[x^\sel]:=[(x,(x^+_\ell)_{\ell\in S}, (\lambda_\ell)_{\ell\in S})]
\longmapsto \log_{\tildeDcc}(x^+_p)\,,
\end{aligned}
\end{equation}
and the following diagram commutes:
\[
\xymatrix{
0\ar[r] & H^1(\Qp,X\tildeDcc) \ar[r] \ar[d]_{\log} &\mathbf{R}^1\boldsymbol{\Gamma}(\tildeVcc  ,\tildeDcc) \ar[d]_{\log} \ar[r] &H^1_{\rm f}(\Vcc) \ar[d]^{\log} \ar[r] &0\\
0\ar[r] &X{}^c\widetilde D \ar[r] &{}^c\widetilde D \ar[r] &{}^cD \ar[r]
&0\,.
}
\]
\end{corollary}



The following proposition computes the Bockstein map on the complex $\RG (\tildeVcc,\alpha).$

\begin{proposition} 
\label{prop_4_14_2022_08_19_1401}
Suppose  that 
$H^1_0(\Vcc)=0$ and $H^1(\QQ_\ell,\Vcc)=0$ for all $\ell \in S\setminus\{p\}.$
Assume that the hypotheses \eqref{item_Sha} and \eqref{item_C4} hold true. Then:
\item[i)]{}  There exist natural isomorphisms 
\[
\mathbf{R}^1\boldsymbol{\Gamma} (\tildeVcc,\alpha) \simeq  {}^c\widetilde D\,\,,
\qquad \qquad 
\mathbf{R}^2\boldsymbol{\Gamma} (\tildeVcc,\alpha) \simeq 
\displaystyle\frac{H^1(\Qp, \tildeVcc)}{\res_p\left(H^1_{\{p\}}(\tildeVcc)\right)}\,.
\]
In particular, $\mathbf{R}^i\boldsymbol{\Gamma} (\tildeVcc,\alpha)$ is a free $\widetilde E$-module of rank one $(i=1,2)$.

\item[ii)] We have a commutative diagram 
\[
\xymatrix{
\mathbf{R}^1\boldsymbol{\Gamma} (\tildeVcc,\alpha)
\ar[r]^-{\sim} \ar[d]_-{\mathrm{Bock_2}} & {}^c\widetilde D\ar[d]^-{\vartheta}\\
\mathbf{R}^2\boldsymbol{\Gamma} (\tildeVcc,\alpha) \ar[r]^-{\sim}
&  \displaystyle\frac{H^1(\Qp, \tildeVcc)}{\res_p\left(H^1_{\{p\}}(\tildeVcc)\right)}\,,
}
\]
where the right vertical map $\vartheta$ is given by the following rule: Let 
$d\in {}^c\widetilde D$ and let $\widehat d\in {}^c\widetilde D\otimes \LL$ be any lift of
$d.$ Then there exist $x\in H^1_\Iw (X\tildeVcc)\otimes_{\LL_E} \CH(\Gamma)$ and
$y\in H^1_\Iw (\Qp, \tildeVcc)\otimes_{\LL_E} \CH(\Gamma)$ such that 
\[
\Exp_{\tildeVcc}(\widehat{d})= \res_p(x)+(\gamma-1)y\,.
\]
Then $\vartheta (d):=\pr_0 (y) \mod H^1_{\{p\}}(\tildeVcc)$. 
In particular, $\mathrm{im} (\vartheta) \subset H^1 (\Qp, X\tildeVcc)$\,.
\end{proposition}
\begin{proof} 
Since $H^1_0(\Vcc)=0,$ 
it follows from  Lemma~\ref{lemma:vanishing invariants of sha} 
that 
$\Sha^2_\Iw (\tildeVcc)^\Gamma=\{0\}=\Sha^2_\Iw (\tildeVcc)_\Gamma$, and 
\[
\mathbf{R}^1\boldsymbol{\Gamma} (\tildeVcc,\alpha)=\mathrm{coker} (g)^\Gamma\,,
\qquad \qquad 
\mathbf{R}^2\boldsymbol{\Gamma} (\tildeVcc,\alpha)=\mathrm{coker} (g)_\Gamma\,,
\]
where $g$ is the map \eqref{eqn: the map g} with $j=\frac{k}{2}$:
\begin{equation}
\nonumber
\left (H^1_\Iw (\tildeVcc)\otimes_{\widetilde \LL_E} \widetilde \CH (\Gamma)\right )
\oplus 
\left ({}^c\widetilde D  \otimes_{\widetilde \LL_E} \widetilde \CH (\Gamma)
\right ) 
\xrightarrow{\res_p\,+\, \Exp_{\tildeDcc}}
 H^1_\Iw (\Qp, \tildeVcc)\otimes_{\widetilde \LL_E} \widetilde \CH (\Gamma)\,.
\end{equation}
Therefore, as in the previous chapter, we infer that $\mathbf{R}^1\boldsymbol{\Gamma}(\tildeVcc,\alpha)$ and  $\mathbf{R}^2\boldsymbol{\Gamma}(\tildeVcc,\alpha)$ can be naturally identified with the kernel and the cokernel, respectively,  of the map
\begin{equation}
\nonumber
g_\Gamma \,:\, H^1_\Iw (\tildeVcc)_\Gamma 
\oplus {}^c\widetilde D  
\xrightarrow{\res_p\,+\, \Exp_{\tildeDcc,0}}
 H^1 (\Qp, \tildeVcc)\,.
\end{equation}
Recall that in  Corollary~\ref{cor: computation selmer in rank one case} (see also the proof of Proposition~\ref{prop_exceptional_zero_like_extension}),
we already proved that 
$$\mathbf{R}^1\boldsymbol{\Gamma}(\tildeVcc,\alpha)
\simeq \ker (g_\Gamma)\simeq {}^c\widetilde D.$$ 
This can be deduced also using the arguments  below, which utilize to compute 
 $\mathbf{R}^2\boldsymbol{\Gamma}(\tildeVcc,\alpha).$ 

Recall that $H^1_\Iw (\tildeDcc)_\Gamma \xrightarrow{\sim} H^1_{\{p\}}(\tildeDcc)=H^1_{\mathrm{f}}(\tildeDcc)$. In what follows, we shall use the indicated isomorphism to identify these spaces. It follows from Lemma~\ref{lemma: for semisimplicity}
that  $H^1_\Iw (\tildeVcc)_\Gamma=H^1_{\{p\}}(\tildeVcc)$ sits in the exact sequence
\[
0\lra H^1_{\{p\}}(X\tildeVcc) \lra H^1_\Iw (\tildeVcc)_\Gamma \lra H^1_{\{p\}}(\Vcc) \lra 0\,,
\]
which can be extended to the following commutative diagram with exact rows
$$
\xymatrix{
0\ar[r]& H^1_{\{p\}}(X\tildeVcc) \ar[r]\ar[d]_{\res_p}& H^1_\Iw (\tildeVcc)_\Gamma \ar[r]\ar[d]_{\res_p}& H^1_{\{p\}}(\Vcc) \ar[r]\ar[d]^{\res_p}& 0\\
&H^1_{\{p\}}(\QQ_p, X\tildeVcc) \ar[r]& H^1_{\{p\}}(\QQ_p, \tildeVcc) \ar[r]&H^1_{\{p\}}(\QQ_p, \Vcc)
}
$$
where the vertical arrows are injections.
Since the restriction map 
$H^1_{\{p\}}(X\tildeVcc) \xrightarrow{\res_p}  H^1_{\mathrm{f}}(\Qp, X\tildeVcc)$
 is an isomorphism, and 
 \be\label{eqn_2022_08_25_1052}\notag\mathrm{im} \left( \Exp_{\tildeDcc,0}\right)=H^1_{\mathrm{f}}(\Qp, X\tildeVcc),\ee
 we have
\be
\label{eqn_2022_08_19_1326}
\mathrm{im} \left( \Exp_{\tildeDcc,0}\right)=H^1_{\mathrm{f}}(\Qp, X\tildeVcc)= \res_p(H^1_{\{p\}}(X\tildeVcc)).
\ee
As a result, we have ${\rm im}(g_\Gamma)=\res_p(H^1_\Iw (\tildeVcc)_\Gamma)=\res_p(H^1_{\{p\}} (\tildeVcc))\,,$
and the second isomorphism in (i) follows.

The proof that $\mathbf{R}^1\boldsymbol{\Gamma} (\tildeVcc,\alpha)$ is a free $\widetilde{E}$-module of rank one follows from the first isomorphism in (i), whereas the fact that $\mathbf{R}^2\boldsymbol{\Gamma} (\tildeVcc,\alpha)$ is a free $\widetilde{E}$-module of rank one from the second isomorphism in (i) combined with Lemma~\ref{lemma: for semisimplicity}.


The formula for $\vartheta$ in (ii) follows directly from the definition of the Bockstein map and the proofs of the remaining claims in this portion are clear. 
\end{proof}

\begin{remark} 
\label{remark_4_15_2022_08_19_1420}
In this remark, we return back to the setting of Chapter~\ref{chapter_main_conj_infinitesimal_deformation} and  assume that $\widetilde V=V_k$ is the infinitesimal thickening of the $p$-adic Galois representation attached to a critical modular form of weight $k$ along the eigencurve. 
\item[i)]{} The discussion in \S\ref{sec_thick_Selmer_complexes_5_2} applies verbatim when the pair $({}^c\widetilde{V}, {}^c\widetilde{\bD})$ is replaced with $(\widetilde{V}(\chi^j),\widetilde{\bD}(\chi^{j}))$ for integers $1\leq j\leq k-1$, without demanding that $k$ is even. It is required that $k$ is even only when one is interested (as we are in the remainder of this chapter) in the interpretation of the central critical values in terms of $p$-adic height pairings. 
\item[ii)]{} Proposition~\ref{prop_4_14_2022_08_19_1401}(ii) tells us that the thick Selmer group $ \mathbf{R}^1\boldsymbol{\Gamma}(\Vcc_k  ,\Dcc_k)$ behaves in a manner reminiscent of the behaviour of the extended Selmer groups in the presence of exceptional zeros. The exceptional zero that $ \mathbf{R}^1\boldsymbol{\Gamma}(\Vcc_k  ,\Dcc_k)$  accounts for is that of thick $p$-adic $L$-function $\widetilde{L}_{\mathrm{K},\eta}^{\pm}(f,\xi) $, cf. \eqref{eqn_taylor_expansion_1}.

\item[iii)]{} The kernel of $\vartheta$ contains $X{}^cD_k$. We infer from the commutative diagram in Proposition~\ref{prop_4_14_2022_08_19_1401}(ii) that the complex $\RG_\Iw (\Vcc_k,\alpha)$ is not semi-simple.  It seems that the modification of the Bockstein map used in the earlier sections of Chapter \ref{chapter_main_conj_infinitesimal_deformation} can not be extended to this case.

\item[iv)]{} It is not clear whether or not the map $\vartheta$ is nonzero. The requirement that $\vartheta\neq 0$ appears to be a refinement of \eqref{item_C4}. 


\item[v)]{} Based on the absence of a duality theory over $\cO_{\mathcal X}$ combined with the previous remarks, it is not clear to us whether or not the derivative of the secondary $p$-adic $L$-function $L^{[1],\pm}_{K,\alpha^*}(f^*,\xi^*)$ at the central critical point can be computed in terms of $p$-adic heights. See, however, Theorem~\ref{thm_A_adic_regulator_formula_eigencurve}(ii) below where we compute this quantity in terms of the second-order derivative of an $\cO_\cX$-adic regulator (at $X=0$). Nonetheless, the arithmetic meaning of  $L^{[1],\pm}_{K,\alpha^*}(f^*,\xi^*)$ at the central critical point, even granted the  infinitesimal thickening~\ref{item_MCinf} of the $\theta$-critical Main Conjecture, remains mysterious.
\end{remark}

\section{$p$-adic regulators and leading term formulae}
\label{sec_padic_regulators}
Our goal in this section is to compute the leading terms of the critical $p$-adic $L$-functions given as in \S\ref{chapter:critical L-functions}, in terms of the regulators associated to the $p$-adic heights we have introduced in \S\ref{sec_padic_heights}. As in the earlier sections in this chapter, we shall solely concentrate on the central critical values.

\subsection{The Selmer complex $\RG (\Vcc_{\mathcal X},\Dcc_{\cX})$}

\subsubsection{Central critical twists}
\label{subsubsec_selmer_groups_18_11} 
As before, we let $V_\cX$ and $V'_\cX$ denote the free $\cO_\cX$-modules of rank $2$ as in \S\ref{subsubsec_221_24082022}, which are endowed with continuous actions of $G_{\QQ,S}$. Recall that the Galois representations $V_\cX$ and $V'_\cX$ interpolate Deligne's representations over the affinoid neighborhood $\cX$ contained in the eigencurve, about the point $x_0$ that corresponds to the $\theta$-critical eigenform $f_{\alpha}\in S_{k}(\Gamma_1(N)\cap\Gamma_0(p), \varepsilon_f)$.


\begin{defn}
\label{defn_univ_wt_cyclo_char}
The universal $\mathscr{H}_\cW$-valued cyclotomic character $\bbchi: \Gamma \to \cO_\cW^*$ is given by
$$\bbchi(\gamma):=\chi(\gamma)^{k}
\exp \left (\log_p (1+Y)\frac{\log_p (\left <\chi (\gamma)\right >)}{\log_p(1+p)}\right )\,,\qquad \gamma\in \Gamma\,.
$$
\end{defn}

We continue to assume that $k$ is even and we consider the \emph{central critical twist} $\Vcc_{\cX}:=V_{\cX}(\bfchi^{\frac{1}{2}})$ of $V_{\cX}$. Put $\Dcc_{\cX}:=\bD_{\cX}(\bfchi^{\frac{1}{2}})$.
The pairing $(\,,\,)$ given as in \eqref{subsubsec_221_24082022}
gives rise to an $\cO_\cW$-linear pairing 
\be
\label{self_duality_pairing_over_W}
\Vcc_{\cX}'\otimes \Vcc_{\cX} \lra \cO_\cW(\chi).
\ee

For each $x\in \cX^{\mathrm{cl}}(E)\setminus \{x_0\},$
the $\varphi$-module $\Dc (V_x')$ decomposes into the direct sum
\[
\Dc (V_x')=\Dc (V_x')^{\varphi=\alpha'(x)}\oplus \Dc (V_x')^{\varphi=\beta'(x)},\qquad 
\textrm{where 
$\alpha'(x)=\alpha (x)\varepsilon_f^{-1}(p),$
and 
$\beta'(x)=\beta (x)\varepsilon_f^{-1}(p),$
}
\]
and we denote by $\bD_{\cX}'$ the unique free $\CR_\cX$-submodule of rank one
of $\DdagrigX (V_\cX')$ such that $\bD'_x=\Dc (V_x')^{\varphi=\alpha'(x)}$
for all $x\in \cX^{\mathrm{cl}}(E)\setminus \{x_0\}.$ Set $\Dcc_\cX':=\Dcc_\cX (\bfchi^{\frac{1}{2}}).$ The modules $\Dcc'_\cX$ and 
$\Dcc_\cX$  are orthogonal to each other  under the canonical pairing induced from \eqref{self_duality_pairing_over_W}.

\subsubsection{} In this section, we study the Selmer complexes $\RG (\Vcc_\cX,\Dcc_\cX)$ 
and $\RG (\Vcc'_\cX,\Dcc'_\cX)$ (see  \S\ref{subsec_selmer_complexes} for general definitions).


\begin{lemma}
\label{lemma_2022_08_24_1452}
Let $\ell \in S\setminus\{p\}$ be a prime. Then the following statements hold true:

\item[i)]{} $H^1(\QQ_\ell,\Vcc_x)=H^1(\QQ_\ell,\Vcc'_x)=\{0\}$ at any classical point $x\in \cX (E),$ 

\item[ii)]{} $H^1(\QQ_\ell,\Vcc_\cX)=H^1(\QQ_\ell,\Vcc_\cX')=\{0\}$ whenever $\mathcal X$ is sufficiently small.

\item[iii)]{} For sufficiently small $\cX$, the $\cO_\cX$-modules $\mathbf{R}^1\boldsymbol{\Gamma}(\Vcc_\cX,\Dcc_\cX)$ and $\mathbf{R}^1\boldsymbol{\Gamma}(\Vcc'_\cX,\Dcc'_\cX)$ are free.
\end{lemma}
\begin{proof}
\item[i)] It follows from the local-global compatibility of Langlands correspondence that 
$$H^0(\QQ_\ell,\Vcc_x)=\{0\}=H^2(\QQ_\ell,\Vcc_x)$$
for all $x\in \cX^{\rm cl}(E)$ and $\ell\in S\setminus\{p\}$. By the local Euler characteristic formula, we then conclude that $H^1(\QQ_\ell,\Vcc_x)=\{0\}$ for all 
$x\in \cX^{\rm cl}(E)$. We deduce from the self-duality of $\Vcc_x$ that $H^1(\QQ_\ell,\Vcc_x')=\{0\}$.

\item[ii)] It follows from Part i) that $H^1(\QQ_\ell,\Vcc'_\cX)/\mathfrak{m}_xH^1(\QQ_\ell,\Vcc_\cX')=\{0\}$ for every $x\in \cX^{\rm cl}(E)$, thence also that $H^1(\QQ_\ell,\Vcc_\cX')$ is a torsion $\cO_\cX$-module with no support at any classical point in $\cX(E)$ (including $x=x_0$). On shrinking $\cX$ as necessary, we can ensure that it has no support anywhere on $\cX$.

\item[iii)]  Since $\cO_\cX$ is a PID, it suffices to prove that both modules are torsion-free. 

Since the $\cO_\cX$-module $\mathbf{R}^1\boldsymbol{\Gamma}(\Vcc_\cX,\Dcc_\cX)$ is finitely generated, we can ensure (on shrinking $\cX$) that the torsion submodule of $\mathbf{R}^1\boldsymbol{\Gamma}(\Vcc_\cX,\Dcc_\cX)$ is contained in 
$$\mathbf{R}^1\boldsymbol{\Gamma}(\Vcc_\cX,\Dcc_\cX)[X^\infty]:=\bigcup_n\ker\left(\mathbf{R}^1\boldsymbol{\Gamma}(\Vcc_\cX,\Dcc_\cX)\xrightarrow{[X^n]}\mathbf{R}^1\boldsymbol{\Gamma}(\Vcc_\cX,\Dcc_\cX)\right)\,.$$
It therefore suffices to check that 
$${\rm im}\left(\mathbf{R}^0\boldsymbol{\Gamma}(\Vcc_{x_0},\Dcc_{x_0})\lra  \mathbf{R}^1\boldsymbol{\Gamma}(\Vcc_\cX,\Dcc_\cX)\right) =\ker\left(\mathbf{R}^1\boldsymbol{\Gamma}(\Vcc_\cX,\Dcc_\cX)\xrightarrow{[X]}\mathbf{R}^1\boldsymbol{\Gamma}(\Vcc_\cX,\Dcc_\cX)\right)\,.$$
This is clear, since $\mathbf{R}^0\boldsymbol{\Gamma}(\Vcc_{x_0},\Dcc_{x_0})=\{0\}$.

The proof of the freeness of $\mathbf{R}^1\boldsymbol{\Gamma}(\Vcc'_\cX,\Dcc'_\cX)$ is identical and is omitted here.
\end{proof}

\subsubsection{} Let us set $\DDcc_{\cX}:=\DdagrigX (\Vcc_\cX)/\Dcc_{\cX},$ so that we have  the following tautological exact sequence:
\[
0 \lra \Dcc_{\cX} \lra \DdagrigX (\Vcc_\cX)
\lra \DDcc_{\cX} \lra 0\,.
\]
We also set $\DDcc_{\cX}':=\DdagrigX (\Vcc'_\cX)/\Dcc'_{\cX}$\,.

\begin{lemma}\label{lemma_tildeD_H0_is_zero}
The following assertions are valid for sufficiently small $\cX$:
\item[i)] $H^0(\QQ_p, \DDcc_\cX)=\{0\}.$
\item[ii)]  Assume that the condition \eqref{item_C4} holds. Then the support of the $\cO_\cX$-module $H^1(\QQ_p,\DDcc_\cX)_{\rm tor}$ is $\{x_0\}$, and 
$H^1(\QQ_p, \DDcc_\cX)_{\rm tor}\simeq \cO_\cX/X\cO_\cX$.

\item[iii)] The assertions i-ii) hold if we replace $\DDcc_\cX$ with $\DDcc'_\cX$.
\end{lemma}

\begin{proof}
\item[i)] 

Let us consider the short exact sequence
\begin{equation}
\label{eqn: short exact sequence for DDcc}
0 \lra \DDcc_{\cX} \xrightarrow{X-X(x)} \DDcc_{\cX}
\xrightarrow{\mathrm{sp}_{x}} \DDcc_{x} \lra 0\,.
\end{equation}
The exactness of the sequence
$$
0\lra H^0(\QQ_p,\DDcc_\cX)\xrightarrow{X-X(x)} H^0(\QQ_p, \DDcc_\cX) \xrightarrow{\mathrm{sp}_x} H^0(\QQ_p,\DDcc_x)$$
for each $x\in \cX^{\rm cl}(E)$ and the vanishing $H^0(\QQ_p,\DDcc_x)=0$ whenever 
$x\neq x_0$ shows (on shrinking $\cX$ as necessary) that  $H^0(\QQ_p,\DDcc_\cX)$ can only have support at $x_0$; in particular it is a torsion $\cO_\cX$-module.

On the other hand, the exact sequence
$$0\lra H^0(\QQ_p,{\Dcc}_\cX)\lra H^0(\QQ_p,\Vcc_\cX)\lra H^0(\QQ_p,\DDcc_\cX)\lra H^1(\QQ_p,{\Dcc}_\cX)$$
together with the fact that $H^0(\QQ_p,\Vcc_\cX)=0$ show that $H^0(\QQ_p,\DDcc_\cX)$ injects into $H^1(\QQ_p,{\Dcc}_\cX)_{\rm tor}$. Since we have
$$ 
0=H^0(\QQ_p,{\Dcc}_x)\stackrel{\sim}{\lra} H^1(\QQ_p,{\Dcc}_\cX)[\mathfrak{m}_x]
$$
for each $x\in \cX^{\rm cl}(E)$ {(including $x=x_0$)}, we can ensure that $H^1(\QQ_p,{\Dcc}_\cX)_{\rm tor}=0$ (shrinking $\cX$, if necessary). This concludes the proof that $H^0(\QQ_p,\DDcc_\cX)=0$. 

\item[ii)]  
We note that the vanishing of $H^0(\Qp, \DDcc_\cX)$ does not imply the vanishing of $H^0(\Qp, \DDcc_{x_0}).$  More precisely, using the long exact sequence in $G_{\QQ_p}$-cohomology that the short exact sequence \eqref{eqn: short exact sequence for DDcc} gives rise to, we obtain an isomorphism
\[
H^0(\Qp,\DDcc_{x}) \simeq H^1(\Qp, \DDcc_\cX) [\mathfrak m_x]\,
\]
for any $x$. The $\cO_\cX$-module  $H^1(\QQ_p,\DDcc_\cX)_{\rm tor}$  is finitely generated. Since $H^0(\Qp,\DDcc_{x})=0$ if $x\in \cX^{\mathrm{cl}}(E)\setminus\{x_0\}$ and $H^0(\Qp,\DDcc_{x_0})\simeq H^1(\Qp,\Dcc_{x_0})$ is of dimension one over $E$, we can assume, on shrinking $\cX$ as necessary, that the support of  $H^1(\QQ_p,\DDcc_\cX)_{\rm tor}$ is  $\{x_0\}$. 

Since $H^1(\Qp,\Dcc_{x_0})\simeq \cO_\cX/X\cO_\cX,$
we only need to prove that  $H^1(\QQ_p, \DDcc_\cX)[X^2]=H^1(\QQ_p, \DDcc_\cX)[X]$. 
By the exactness of the sequence
$$0\lra H^0(\QQ_p, \DDcc_{k})\lra H^1(\QQ_p, \DDcc_\cX)\xrightarrow{[X^2]} H^1(\QQ_p, \DDcc_\cX)$$
we reduce to proving that $\dim_E H^0(\QQ_p,  \DDcc_{k})=1.$ 
To that end, let us consider the exact sequence
\[
0 \lra \Dcc_k \lra  \bD^{\dagger}_{\mathrm{rig}}(\Vcc_k) \lra {\DDcc}_k \lra 0\,,
\]
which induces an exact sequence
\[
H^0(\Qp, {\DDcc}_k)\lra  H^1(\Qp,\Dcc_k) \lra H^1(\Qp, \Vcc_k)\,.
\]
Let $V_k^{(\alpha)}$ be  the maximal crystalline sub-representation of $V_k$ 
defined  in \S\ref{subsubsec_2113_2022_08_24_1732} and denoted there by $\widetilde V^{(\alpha)}.$ Then the quotient  
$\Vcc_k\Big{/} {}^c V_k^{(\alpha)}\simeq \Vcc_{x_0}^{(\beta)}$
has Hodge-Tate weight $-\frac{k}{2}$, and therefore the map 
$$H^1(\Qp, {}^cV_k^{(\alpha)}) \lra H^1(\Qp, \Vcc_k)$$ 
is injective. We have a commutative diagram
\[
\xymatrix{
0 \ar[r] & \mathscr D_{\mathrm{cris}}(\Dcc_{x_0}) \ar[r] & \mathscr D_{\mathrm{cris}}(\Dcc_k) \ar@{=}[d]^{\exp} \ar[r]
 &\Dc (  {}^cV_k^{(\alpha)})\big{/}\Fil^0\Dc ( {}^c V_k^{(\alpha)}) \ar@{^(->}[d] \\
0 \ar[r] & H^0(\Qp, {\DDcc}_k)\ar[r] &H^1(\Qp,\Dcc_k)
\ar[r]  &H^1(\Qp, {}^c V_k^{(\alpha)}),
}
\] 
where the first row is exact and the image of the right-most map on this row is nonzero thanks to the condition \eqref{item_C4}. This shows that  $H^0(\Qp, \DDcc_k)$ is of dimension one over $E$, as required. 
\end{proof}

To simplify the notation, we denote by  $H^1_{\Dcc}(\Vcc_\cX)$ the Greenberg-style Selmer group attached  to the $(\varphi,\Gamma)$-module
$\Dcc_\cX$, namely  
\[
H^1_{\Dcc}(\Vcc_\cX):=\ker\left(H^1_S(\Vcc_\cX)\xrightarrow{\bigoplus_{\ell \in S} 
{r_\ell}} H^1(\QQ_p, \DDcc_\cX)\oplus\bigoplus_{p\neq \ell\in S}H^1(\QQ_\ell^\ur,\Vcc_\cX)\right)\,.
\]
(cf. \S\ref{subsec_123_21_11_1607}). In view of Lemma~\ref{lemma_2022_08_24_1452}, we have
\[
H^1_{\Dcc}(\Vcc_\cX):=\ker\left(H^1_S(\Vcc_\cX)\xrightarrow{\,r_p\,} H^1(\QQ_p, \DDcc_\cX)\right)\,.
\]

\begin{proposition} There exists a canonical isomorphism
\[
 \mathbf{R}^1\boldsymbol{\Gamma}(\Vcc_\cX,\Dcc_\cX) \simeq 
 H^1_{\Dcc}(\Vcc_\cX).
\]
\end{proposition}
\begin{proof}
We have the following exact sequence thanks the definition of Selmer complexes:
\begin{equation}
\label{eqn_selmer_complex_big_defn}
0 \lra H^0(\QQ_p,\DDcc_\cX)\lra \mathbf{R}^1\boldsymbol{\Gamma}(\Vcc_\cX,\Dcc_\cX)\lra H^1_{\Dcc}(\Vcc_\cX)\lra 0.
\end{equation}
The proof of the proposition follows since $H^0(\QQ_p,\DDcc_\cX)=0$ by Lemma~\ref{lemma_tildeD_H0_is_zero}. 
\end{proof}

\subsection{The rank-one case}
\label{subsec: the rank one case}
In this section, we maintain the previous notations and conventions and assume in addition that
\be
\label{eqn_assume_analytic_rank_equals_1}
r_{\mathrm{an}}(k/2):=\ord_{s=\frac{k}{2}}L(f,s)=1\,.
\ee
as well as the hypothesis \eqref{item_PR1} hold true for the eigenform $f$. 

In what follows, thanks to our discussion in Remark~\ref{remark: about PR1 condition}, one may drop the assumption \eqref{item_PR1} from the statements below whenever \eqref{item_CM} holds true (and it does hold true when $k=2$).
\subsubsection{} 
We denote by 
${}^{c}\bz_0 (\cX,\xi)\in H^1_S(\Vcc_{\mathcal X}')$ the image of the big Beilinson--Kato element $\bz(\cX,\xi)\in H^1_S(V_{\mathcal X}'\,\widehat\otimes\,\LL^\iota(1))$ under the obvious projection, where the latter element is given as in \S\ref{subsubsec_2221_18_11_2021}. 
Note that \eqref{eqn_assume_analytic_rank_equals_1} implies (thanks to Kato's reciprocity laws) that 
\be
\label{eqn_assume_analytic_rank_equals_1_BK_1}
{}^{c}\bz_0 (x_0,\xi)\in H^1_{\rm f}(\Vcc_{x_0}')\,.
\ee


Recall that
the condition \eqref{item_PR1} implies (thanks to the Euler system machinery) that 
\be\label{eqn_373_20_11_14_52} 
\dim H^1_{0}(\Vcc_{x_0}')=0\,,\qquad\dim H^1_{\rm f}(\Vcc_{x_0}')=1\,
\ee
(cf. Proposition~\ref{prop_2022_04_26_13_00}).
This in turn shows that the parity conjecture asserting that
\be\label{eqn_parity_conj_x}
\ord_{s=\frac{k}{2}}L(f,s)\equiv \dim H^1_{\rm f}(\Vcc_{x_0}')\mod 2
\ee
holds true.


\begin{proposition}
\label{rk_rk_prime_implies_parity}
Suppose that \eqref{eqn_assume_analytic_rank_equals_1} and \eqref{item_PR1} hold. {Then on sufficiently small  $\cX$  we have
\[
\dim_E H^1_{\rm f}(\Vcc_x')=1
\] 
for all $x\in \cX^{\rm cl}(E)$. Moreover, 
${}^{c}\bz_0 (x,\xi)\in H^1_{\rm f}(\Vcc_x')$ 
and the truth of the parity conjecture is constant on $\cX^{\rm cl}(E)$.
}
\end{proposition}

\begin{proof}
The condition \eqref{item_PR1} shows that $\res_p({}^{c}\bz_0 (\cX,\xi))\in H^1(\QQ_p,\Vcc_\cX')$ is nonzero. Since $H^1(\QQ_p,\Vcc'_\cX)$ is a free $\cO_\cX$-module of rank $2$, it follows (on shrinking $\cX$ as necessary) that 
$$\res_p({}^{c}\bz_0 (x,\xi))=\res_p({}^{c}\bz_0 (\cX,\xi))(x)\neq 0$$ 
for all $x\in \cX^{\rm cl}(E)$.

The constancy of the sign of the functional equation in families (cf. \cite{PX_parity}, \S3.6) together with \eqref{eqn_assume_analytic_rank_equals_1} shows that $\ord_{s=\frac{w(x)}{2}} L(f_x,s)$ is odd for all $x\in \cX^{\rm cl}(E)$. In particular, $L(f_x,\frac{w(x)}{2})=0$. This together with Kato's reciprocity laws (and the conclusion of the preceding paragraph) show that 
$$
0\neq {}^{c}\bz_0 (x,\xi)\in  H^1_{\rm f}(\Vcc_x').
$$ 
for all $x$ as above. Now a standard application of the Euler system machinery (with the Beilinson--Kato Euler system for the eigenform $f_x$) shows that $\dim H^1_{\rm f}(\Vcc_x')=1$, which has the same parity as $\ord_{s=\frac{w(x)}{2}} L(f_x,s)$.
\end{proof}

Recall that we have an isomorphism
\begin{equation}
\nonumber
\mathbf{R}^1\boldsymbol{\Gamma}(\Vcc_\cX',\Dcc_\cX')\simeq 
H^1_{\bD'}(\Vcc_\cX').
\end{equation}

 
\begin{theorem}
\label{lemma_parity_rk_rk_prime_implies_sign}
Suppose that $r_{\mathrm{an}}(k/2)=1$  and  assume that the conditions \eqref{item_PR1} and 
\eqref{item_C4} hold. Then for sufficiently small  $\cX$, the following assertions are valid:
\item[i)]  Both $\mathbf{R}^1\boldsymbol{\Gamma}(\Vcc_\cX,\Dcc_\cX)$ and $\mathbf{R}^1\boldsymbol{\Gamma}(\Vcc'_\cX,\Dcc_\cX')$ are free of rank one as $\cO_\cX$-modules.

\item[ii)] The $\cO_\cX$-module $H^1_{\Dcc'}(\Vcc_\cX')$, and therefore also the isomorphic module $\mathbf{R}^1\boldsymbol{\Gamma}(\Vcc_\cX',\Dcc_\cX')$,
is  generated by the element $ X\cdot{}^{c}\bz_0 (\cX,\xi).$ 
\end{theorem}

\begin{proof}
\item[i)] Equation \eqref{eqn_373_20_11_14_52} together with \eqref{eqn_prop_degenerate_local_conditions} show that 
$$\dim\mathbf{R}^1\boldsymbol{\Gamma}(\Vcc_{x_0}',\Dcc_{x_0}')=1.$$ 
The control theorem for Selmer complexes yields an injection 
$$ \mathbf{R}^1\boldsymbol{\Gamma} (\Vcc_\cX',\Dcc_\cX')/\mathfrak{m}_{x_0} \mathbf{R}^1\boldsymbol{\Gamma} (\Vcc_\cX,\Dcc_\cX')\hookrightarrow  \mathbf{R}^1\boldsymbol{\Gamma}(\Vcc_{x_0}',\Dcc_{x_0}'),$$ which in turn shows that 
$${\rm rank}_{\cO_{\cX}}  \mathbf{R}^1\boldsymbol{\Gamma}(\Vcc_\cX',\Dcc_\cX')\leqslant 1.$$  

 Proposition~\ref{rk_rk_prime_implies_parity} tells us that, on shrinking $\cX$ as necessary, $H^1_{\rm f}(\Vcc_x')=\mathbf{R}^1\boldsymbol{\Gamma}(\Vcc'_x,\Dcc_x')$ is 1-dimensional for every $x\in \cX^{\rm cl}(E)\setminus\{x_0\}$. 
Consider the exact sequence 
$$
0 \lra \mathbf{R}^1\boldsymbol{\Gamma}(\Vcc_\cX',\Dcc_\cX')/\mathfrak{m}_x \mathbf{R}^1\boldsymbol{\Gamma}(\Vcc_\cX',\Dcc_\cX')\lra  \mathbf{R}^1\boldsymbol{\Gamma}(\Vcc'_x,\Dcc_x')
\lra  \mathbf{R}^2\boldsymbol{\Gamma}(\Vcc_\cX',\Dcc_\cX')[\mathfrak{m}_x]
\lra 0,
$$ 
induced from the control theorem. Since $\mathbf{R}^2\boldsymbol{\Gamma}(\Vcc_\cX',\Dcc_\cX')$ is a finitely generated $\cO_\cX$-module, 
on further shrinking $\cX$, one can ensure that this exact sequence induces an
isomorphism
\[
\mathbf{R}^1\boldsymbol{\Gamma}(\Vcc_\cX',\Dcc_\cX')/\mathfrak{m}_x \mathbf{R}^1\boldsymbol{\Gamma}(\Vcc_\cX',\Dcc_\cX') \simeq H^1_{\rm f}(\Vcc_x'),
\qquad \forall x\in \cX^{\rm cl}(E).
\]
This shows that $ \mathbf{R}^1\boldsymbol{\Gamma}(\Vcc_\cX',\Dcc_\cX')\neq \{0\}$. Recalling from Lemma~\ref{lemma_2022_08_24_1452} that $\mathbf{R}^1\boldsymbol{\Gamma}(\Vcc_\cX',\Dcc_\cX')$ is free, and combining with the conclusion of the previous paragraph, we deduce that 
$${\rm rank}_{\cO_{\cX}}  \mathbf{R}^1\boldsymbol{\Gamma}(\Vcc_\cX',\Dcc_\cX')= 1\,,$$ as required.

The argument to prove that ${\rm rank}_{\cO_{\cX}}    \mathbf{R}^1\boldsymbol{\Gamma}(\Vcc_\cX,\Dcc_\cX)= 1$ proceeds in a similar way.

\item[ii)] 
Let us put $P':=X\cdot {}^{c}\bz_0 (\cX,\xi).$   We shall first prove that $P'\in H^1_{\bD'}(\Vcc_\cX').$ Recall that
\[
H^1_{\Dcc'}(\Vcc_\cX')= \ker \left (
H^1_S(\Vcc_\cX') \xrightarrow{\,r_p\,} H^1(\DDcc_\cX')                  
\right ).
\]
We need to check that $P'$ verifies the relevant local condition under our running hypotheses. For each $x\in \cX^{\rm cl}(E),$ we have a commutative diagram
\[
\xymatrix{H^1_S(\Vcc_\cX')/\mathfrak m_x H^1_S(\Vcc_\cX') \ar[r]^{r_{p,x}} \ar@{^(->}[d] 
&H^1(\DDcc_\cX')/\mathfrak m_x H^1(\DDcc_\cX') \ar@{^(->}[d]\\
H^1_S(\Vcc_x') \ar[r]& H^1(\DDcc'_x),
}
\]
where the vertical maps are injective. If $x\neq x_0,$ then 
\[
\ker \left (H^1_S(\Vcc_x') \lra  H^1(\DDcc'_x) \right ))\simeq H^1_{\rm f}(\Vcc'_x),
\]
and it follows from Proposition~\ref{rk_rk_prime_implies_parity} that
$r_{p,x}(P')=0.$ Moreover, we also have $r_{p,x_0}(P')=0.$ Since the support of 
$H^1(\Qp, \DDcc_{\cX}')_{\rm tor}=\{x_0\},$ this implies that 
$P'\in \ker (r_p)= H^1_{\Dcc'}(\Vcc_\cX').$

Let $Q'$ be a generator of $H^1_{\Dcc'}(\Vcc_\cX').$  Then $P'=a(X)Q'$ for some $a(X)\in \cO_\cX.$ To prove that $P'$ generates $H^1_{\bD'}(\Vcc_\cX')$  for sufficiently small $\cX$, we only need to check that $a(0)\neq 0.$
Suppose on the contrary that $a$ is divisible by $X$ and put $a(X)=Xb(X)$.
Then by definitions, we have 
$${}^{c}\bz_0 (\cX,\xi)-bP'\in {H}^1_{\Dcc'}(\Vcc_\cX')[X]=\{0\}\,$$
where the vanishing is because the $\cO_\cX$-module $\mathbf{R}^1\boldsymbol{\Gamma}(\Vcc_\cX',\Dcc_\cX')= {H}^1_{\Dcc'}(\Vcc_\cX')$ is torsion-free (cf. Lemma~\ref{lemma_2022_08_24_1452}). This in turn shows that ${}^{c}\bz_0 (\cX,\xi)\in   {H}^1_{\Dcc'}(\Vcc_\cX')$ and hence 
$$\res_p({}^{c}\bz_0 (x_0,\xi))\in \ker\left(H^1(\QQ_p,\Vcc_{x_0}')\lra H^1(\QQ_p, \DDcc'_{x_0})\right)={\rm im} \left(H^1(\QQ_p,\Dcc'_{x_0})\lra H^1(\QQ_p,\Vcc_{x_0}')\right)\stackrel{\rm Lemma~\ref{lemma: coboundary map}}{=}\{0\},$$ 
contradicting the hypothesis \eqref{item_PR1}.

\end{proof}

\subsection{$\cO_\cX$-adic height pairing}
\label{subsec_results_main_eigeincurve}

\subsubsection{}

Let $Q\in \mathbf{R}^1\boldsymbol{\Gamma}(\Vcc_\cX,\Dcc_\cX)$ be any element. For each $x \in \mathcal X ^{\rm cl}(E)$ (resp. $w\in \cW^{\rm cl}(E)$), we let $Q_x\in \mathbf{R}^1\boldsymbol{\Gamma}(\Vcc_x,\Dcc_x)$ (resp. $Q_w\in \mathbf{R}^1\boldsymbol{\Gamma}(\Vcc_w,\Dcc_w)$) denote the image of $Q$ under the appropriate morphism. We similarly define, for a given $Q'\in  \mathbf{R}^1\boldsymbol{\Gamma}(\Vcc_\cX',\Dcc_\cX')$, $Q'_?$ for $?=y,w$.

Following our notation in Section~\ref{subsec_alg_prelim}, let us set
\[
\begin{aligned}
&\mathbf{R}^1\boldsymbol{\Gamma}(\Vcc_\cX,\Dcc_\cX)^\circ:= 
\mathbf{R}^1\boldsymbol{\Gamma}(\Vcc_\cX,\Dcc_\cX)\otimes_{\cO_\cW}\cO_\cX,\\
&\mathbf{R}^1\boldsymbol{\Gamma}(\Vcc_\cX',\Dcc_\cX')^\circ:= 
\mathbf{R}^1\boldsymbol{\Gamma}(\Vcc_\cX',\Dcc_\cX')\otimes_{\cO_\cW}\cO_\cX.
\end{aligned}
\]
The $\cO_\cW$-bilinear  pairing $\Vcc_\cX'\otimes_{\cO_\cW}\Vcc_\cX \rightarrow \cO_\cW$ induces the cyclotomic $p$-adic height pairing

\[
\langle\,,\,\rangle
_{\Dcc',\Dcc}\,:\,  \mathbf{R}^1\boldsymbol{\Gamma}(\Vcc_\cX',\Dcc_\cX') \,\otimes_{\cO_\cW} \mathbf{R}^1\boldsymbol{\Gamma}(\Vcc_\cX,\Dcc_\cX)\lra \cO_\cW.
\]
We extend this pairing by linearity to a $\cO_\cX$-bilinear map 
$$\langle\,,\,\rangle_{\bD}\,:\,  \mathbf{R}^1\boldsymbol{\Gamma}(\Vcc_\cX',\Dcc_\cX')^\circ \,\otimes \mathbf{R}^1\boldsymbol{\Gamma}(\Vcc_\cX,\Dcc_\cX)^\circ \lra \cO_\cX\,,$$
\[
\left\langle \sum_{i} a_i'\otimes z_i', \sum_{j} a_j\otimes z_j \right\rangle_{\bD} :=\sum_{i,j} a_i' a_j\cdot \langle z_i',z_j\rangle_{\Dcc',\Dcc}.
\]

\begin{defn}
\label{def_regulator_cyc}
We define the pairing
$$\langle\,,\,\rangle_{\cX}\,:\,  \mathbf{R}^1\boldsymbol{\Gamma}(\Vcc_\cX',\Dcc_\cX') \,\otimes \mathbf{R}^1\boldsymbol{\Gamma}(\Vcc_\cX,\Dcc_\cX) \lra \cO_\cX$$
on setting
\[
\left\langle  Q', Q \right\rangle_{\cX}:=\frac{1}{2}
\langle Q', X\otimes Q+1\otimes XQ\rangle_{\bD}
\]
and call it the $\cO_\cX$-adic height pairing.
\end{defn}
We remark that 
\[
\left\langle  Q', Q \right\rangle_{\cX}:=\frac{1}{2}
\langle X\otimes Q'+1\otimes XQ',Q\rangle_{\bD}.
\]
thank to the property~\eqref{item_Adj}.
Moreover, for any $x\in \cX^{\mathrm{cl}} (E)\setminus \{x_0\}$
we have
\[
X\otimes Q'+1\otimes XQ' =1\otimes  2X(x)Q'+(X-X(x))\otimes Q'
+1\otimes (X-X(x))Q'.
\]
Thence, recalling from \eqref{eqn: abstract specialization} the specialization morphism ${\rm sp}_x$, we have
\[
\mathrm{sp}_x(X\otimes Q'+1\otimes XQ')=2X(x)Q'_x.
\]
This formula shows that the specialization of our  $\cO_\cX$-adic height pairing to a classical point $x\neq x_0$ is given by
\be\label{eqn_2022_12_16_1213}
\left\langle  Q', Q \right\rangle_{\cX}(x)=
X(x) \cdot \left\langle  Q_x', Q_x \right\rangle_x, \qquad \forall 
x\in \cX^{\mathrm{cl}} (E)\setminus \{x_0\},
\ee
where $\langle\,  ,\,\rangle_x$ denotes  Nekov\'a\v r's $p$-adic height on $V_x.$


\begin{lemma}
\label{reg_at_x_explicit}
Suppose that $r_{\mathrm{an}}(k/2)=1$ and condition \eqref{item_PR1}  holds true. Then the specialization of the pairing $\left\langle \,,\, \right\rangle_{\cX}$ at $x=x_0$ vanishes.
\end{lemma}

\begin{proof}
At  $x=x_0$ we have 
\[
\mathrm{sp}_{x_0}(X\otimes Q +1\otimes XQ)= XQ_{k} \in \mathbf{R}^1\boldsymbol{\Gamma}(\Vcc[x_0],\Dcc[x_0])=\mathbf{R}^1\boldsymbol{\Gamma}(\Vcc_{k},\Dcc_{k})[X].
\]
We can identify $\mathbf{R}^1\boldsymbol{\Gamma}(\Vcc_{k},\Dcc_{k})[X]$
with  $\mathbf{R}^1\boldsymbol{\Gamma}(\Vcc_{x_0},\Dcc_{x_0})$ and 
$XQ_{k}$ with $Q_{x_0}.$ Then $ \left\langle  Q_{x_0}', Q_{x_0} \right\rangle_{x_0}$ can be computed  as  $\langle\,Q'_{x_0},Q_{x_0}\,  \rangle
_{\Dcc',\Dcc},$ where
\[
\langle\,,\,\rangle
_{\Dcc'_{x_0},\Dcc_{x_0}}\,:\, 
\mathbf{R}^1\boldsymbol{\Gamma}(\Vcc_{x_0}',\Dcc_{x_0}')
\otimes \mathbf{R}^1\boldsymbol{\Gamma}(\Vcc_{x_0},\Dcc_{x_0})
\rightarrow E
\]
is the $p$-adic height pairing given as in Section~\ref{subsec: punctual heights}. Since $H^1_0 (V_{x_0})=0$ in the particular scenario we have placed ourselves in, the asserted vanishing follows from  Corollary~\ref{cor_prop_factor_degenerate_height_pairing}.
\end{proof}

\subsubsection{}

 
Let $x\in \cX^{\mathrm{cl}}(E)$ be any $E$-valued classical point. In \S\ref{sec_2_4_2022_05_11_0809}, we have fixed a canonical element $\omega_{x}\in \Fil^0\Dc (V_x)$ and defined a canonical basis $\{\eta_x^{\alpha}, \eta_x^{\beta}\}$ of eigenvectors 
where $\eta_x^{\alpha}\in \Dc (V_x)^{\varphi =\alpha (x)}$ and 
$\eta_x^{\beta}\in \Dc (V_x)^{\varphi =\beta(x)}.$ Recall that for 
$x=x_0$, we have  $\eta_{x_0}^{\alpha}:=\eta_{f}^{\alpha}=\omega_{f}.$

Let us choose $\eta \in \DCc (\bD_\cX)$ such that $\eta_{x_0}=\eta_{x_0}^{\alpha}.$
We note that, in general, $\eta_{x}\neq \eta_{x}^{\alpha}$ if $x\neq x_0$. We recall from \S\ref{subsubsec_2143_18_11} that we have put
$$
\bbeta:=1\otimes X\eta +X\otimes \eta \in \cO_{\cX}\otimes_{\cO_\cW}\DCc(\bD_\cX)\,,
$$ 
where $\eta\in \DCc(\bD_\cX)$ is a generator that we fix throughout. We also define ${}^c\eta_?\in \DCc(\Dcc_?)$ (where $?=y,w(y)$) as $\eta_?\otimes d_{\frac{w(x)}{2}}$ where $\eta_?$ is the image of $\eta$ under the appropriate morphism, and $d_{j}$ is the canonical generator of 
$\Dc (\Qp (j))$ for any integer $j$ (cf. \S\ref{subsubsec_1151_2028_08_25_1200}). We remark that 
$${}^c {\rm sp}_x(\bbeta)=(X-X(x))\,{}^c\eta_{w(x)}+
{2}X(x)\,{}^c\eta_{w(x)}\,,$$ 
so that the image of ${}^c{\rm sp}_x(\bbeta)$ under the morphism
$$\pi_x: \DCc(\Dcc_{w(x)})\lra \DCc(\Dcc_x)$$
equals ${2 }X(x)\, {}^c\eta_x$.


\subsection{Cyclotomic twist of the large exponential map }  
\label{subsubsec_meromorphic_log}
\subsubsection{} 
Our main objective in this  section is to study the central twist of the large exponential map. Recall that $\mathfrak D(\bD_\cX)=\cO_E[[\pi]]^{\psi=0}\otimes_{\cO_E}\DCc(\bD_\cX).$ The map $d\mapsto d\otimes (1+\pi)$ identifies $\DCc (\bD_\cX)$ with a $\cO_\cX$-submodule of rank one of $\mathfrak D(\bD_\cX)$ (note, however, that this map is not $\Gamma$-equivariant). The Iwasawa cohomology module $H^1_\Iw (\bD_\cX)$ is canonically isomorphic to $H^1(\Qp,\bD_\cX\widehat{\otimes}_{\cO_\cX}\CH_\cX(\Gamma)^\iota),$ and we have 
the specialization map
\[
H^1_\Iw (\bD_\cX)\simeq H^1(\Qp,\bD_\cX\widehat{\otimes}_{\cO_\cX}\CH_\cX(\Gamma)^\iota) \xrightarrow{\gamma\mapsto \bbchi^{\frac{1}{2}}(\gamma)} H^1(\QQ_p,\Dcc_\cX).
\]
Consider the map

\[
\EXP_{\Dcc}\,:  \, \DCc(\bD_\cX) \lra  H^1(\QQ_p,\Dcc_\cX)
\]
 given as the composition
\begin{align}
\begin{aligned}
\label{defn_EXP_twisted_1918}
\DCc(\bD_\cX)\lra \mathfrak D(\bD_\cX) \xrightarrow{\Exp_{\bD_\cX,0}} H^1_\Iw(\Qp,  \bD_\cX) \xrightarrow{\gamma\mapsto \bbchi^{\frac{1}{2}}(\gamma)} H^1(\QQ_p,\Dcc_\cX)\,.
\end{aligned}
\end{align}


\begin{lemma}
\label{lemma_EXP_twisted_still_isomorphism}
For sufficiently small $\cX$, the map $\EXP_{\Dcc}$
is an isomorphism.
\end{lemma}

\begin{proof}
Note that the source of $\EXP_{\Dcc}$ is a free $\cO_\cX$-module of rank one.
The exact sequence
\[
0\lra H^0(\Qp, \Dcc_x) \lra H^1(\Qp, \Dcc_\cX) 
\xrightarrow{X-X(x)} H^1(\Qp,\Dcc_\cX)
\]
together with the vanishing of $H^0(\Qp,\Dcc_x)$ in a sufficiently small neighbourhood
of $x_0$ show that $H^1(\Qp,\Dcc_\cX)$ is free of rank one over 
$\cO_\cX.$ 

By the definition of $\EXP_{\Dcc}$ in terms of the large exponential map,  for each $x\in \cX^{\mathrm{cl}}(E),$  the specialization 
\be\label{eqn_2022_08_24_1935} 
\EXP_{\Dcc,x}\,:\, 
\DCc(\bD_x)\lra H^1(\QQ_p,\Dcc_x)\ee
 coincides with the  map 
\begin{equation}
\label{eqn: specialization of EXP}
(-1)^{w(x)/2}\Gamma \left (\frac{w(x)}{2}\right )
  \left(1-\frac{p^{\frac{w(x)}{2}-1}}{a_p(f_x)}\right)\left(1-
  \frac{a_p(f_x)}{p^{\frac{w(x)}{2}}}\right)^{-1}\,\exp_{\Dcc_x }\,.
\end{equation}
In particular, it is an isomorphism at $x=x_0.$ This shows that $\EXP_{\Dcc}$ is necessarily bijective on a sufficiently small $\cX.$



\end{proof}



\subsubsection{} Each element of $z$ of $\mathbf{R}^1\boldsymbol{\Gamma}(\Vcc_?,\Dcc_?)$  (where ?=$\cX,x$) can be written uniquely as a pair $z=(a,b)$, where $a\in H^1(\Vcc_?)$ and $b\in H^1(\Qp, \Dcc_?)$. We define the map
$$\res_p:  \mathbf{R}^1\boldsymbol{\Gamma}(\Vcc_?,\Dcc_?)\lra H^1(\QQ_p,{\Dcc}_?)\,, \qquad z=(a,b)\mapsto b\,.$$
When $?=x\neq x_0$, we shall denote the image of $\res_p(z)$ (where $z\in \mathbf{R}^1\boldsymbol{\Gamma}(\Vcc_x,\Dcc_x)$) under the canonical isomorphism $H^1(\QQ_p,{\Dcc}_x)\xrightarrow{\sim}
H^1_{\rm f}(\QQ_p,\Vcc_x)$ also by $\res_p(z)$.


\begin{defn}
\label{defn_big_LOG_X}
Recall the generator $\eta\in \DCc(\Dcc_\cX)$. For each $z\in  \mathbf{R}^1\boldsymbol{\Gamma}(\Vcc_\cX,\Dcc_\cX)$, we let ${\rm LOG}_{\eta}(z)\in \cO_\cX$ denote the unique element that validates the identity
\be
\label{eqn_defn_big_merom_log}
\res_p(z)={\rm LOG}_{\eta}(z)\cdot \EXP_{\Dcc}(\eta)
\ee
inside the free $\cO_\cX$-module $H^1(\QQ_p,{\Dcc_\cX})$ of rank one.
\end{defn}
We note that the map ${\rm LOG}_{\eta}:  \mathbf{R}^1\boldsymbol{\Gamma}(\Vcc_\cX,\Dcc_\cX)\to \cO_\cX$ is well-defined thank to Lemma~\ref{lemma_EXP_twisted_still_isomorphism}.

\subsubsection{}
\label{subsubsec_3233_20_11_1408} 

For each  $x\in \cX^{\rm cl}(E)\setminus \{x_0\}$, we  have a commutative diagram
$$\xymatrix{
\DCc(\Dcc_x)\ar[r]^-{\exp}_-{\sim}\ar[d]_{\sim} &H^1(\QQ_p,\Dcc_x)\ar[d]^{\sim}
\\
\Dc(\Vcc_x)/{\rm Fil}^0\Dc(\Vcc_x)\ar [r]_-{\exp}^-{\sim}&H^1_{\rm f}(\QQ_p,\Vcc_x)\,.
}
$$
Recall that we have put 
$${}^c\eta_x:=\eta_x[{w(x)}/{2}] \in 
\DCc(\bD_x(\chi^{\frac{w(x)}{2}})).$$ 
For each $x\in \cX^{\rm cl}(E)\setminus \{x_0\}$,
we define ${}^c\omega_x'\in \Fil^0\Dc(\Vcc_x')$ as the unique vector such that 
\begin{equation}
\label{eqn: definition of omega'}
[{}^c\eta_x, {}^c\omega_x']=1\,.
\end{equation}
If $x=x_0,$ we set 
$
{}^c\omega_{x_0}':={}^c\eta_{f^*}^\alpha,
$
where  $\eta_{f^*}^\alpha \in \Fil^0\Dc ({}\Vcc_{x_0}')$ is the vector defined in Section~\ref{subsec311_2022_08_24_0934}. 

For all $x\in \cX^{\rm cl}(E)$, we put
\[
\nonumber
\begin{aligned}
&\log_{{}^c\omega_x'}\,:\,H^1_{\rm f}(\QQ_p,\Vcc_x) \lra E,\\
&\log_{{}^c\omega_x'}(z):=(\log_{\Vcc_x}(z), {}^c\omega_x').
\end{aligned}
\]

 
\begin{proposition}
\label{prop_properties_LOG}
For sufficiently small $\cX$ the following holds true:

\item[i)] For any  $x \in \cX^{\rm cl}(E)\setminus \{x_0\}$ we have
 
$$(-1)^{w(x)/2}\, \Gamma\left(\frac{w(x)}{2}\right) \left(1-\frac{p^{\frac{w(x)}{2}-1}}{a_p(f_x)}\right)\left(1-\frac{a_p(f_x)}{p^{\frac{w(x)}{2}}}\right)^{-1}{\rm LOG}_{\eta}(P)(x)=\log_{{}^c\omega_x'}(\res_p(P_x)).$$ 

\item[ii)] Assume that $r_{\mathrm{an}}(k/2)=1$ and  the conditions \eqref{item_PR1} and \eqref{item_C4} hold. Then 
${\rm LOG}_{\eta}(P)\in \cO_\cX^\times$ for any generator $P$ 
of $\mathbf{R}^1\boldsymbol{\Gamma}(\Vcc_\cX,\Dcc_\cX).$
Moreover,
\begin{equation}
\label{eq: LOG at 0}
(-1)^{k/2}\, C_{\mathrm{K}}\,  \Gamma\left(\frac{k}{2}\right) \left(1-\frac{p^{\frac{k}{2}-1}}{a_p(f_{x_0})}\right)\left(1-\frac{a_p(f_{x_0})}{p^{\frac{k}{2}}}\right)^{-1}{\rm LOG}_{\eta}(P)(x_0)=\log_{{}^c\omega_x'}(\res_p(\Delta (P_{x_0}))),
\end{equation}
where $C_{\mathrm{K}}$ is the constant defined in \eqref{eqn: the constant a}
and $\Delta \,:\,\mathbf{R}^1\boldsymbol\Gamma (\Vcc_k,\Dcc_k) \rightarrow 
H^1_{\mathrm{f}}(\Vcc_{x_0})$ is the morphism from Corollary~\ref{cor: computation selmer in rank one case}.
\end{proposition}

\begin{proof}
\item[i)]  On specializing \eqref{eqn_defn_big_merom_log} to $x\in \cX^{\rm cl}(E)\setminus \{x_0\}$ and using \eqref{eqn: specialization of EXP}, we have
$$\res_p(P_x)= (-1)^{w(x)/2}{\rm LOG}_{\eta}(P)(x) \, \Gamma\left(\frac{w(x)}{2}\right) \left(1-\frac{p^{\frac{w(x)}{2}-1}}{a_p(f_x)}\right)\left(1-\frac{a_p(f_x)}{p^{\frac{w(x)}{2}}}\right)^{-1}\exp_{\Dcc_x}({}^c\eta_x).$$
The asserted equality in Part (i) follows from the definition of $\log_{{}^c\omega_x'}$.

\item[ii)] We recall from Theorem~\ref{lemma_parity_rk_rk_prime_implies_sign} and  
\eqref{eqn_prop_degenerate_local_conditions} that, under our assumptions, the $\cO_\cW$-module $\mathbf{R}^1\boldsymbol{\Gamma}(\Vcc_{\cX},\Dcc_{\cX})$ is of rank one and  $\mathbf{R}^1\boldsymbol{\Gamma}(\Vcc_{x_0},\Dcc_{x_0})$ is a one-dimensional vector space over $E$. Therefore, the natural morphism
$$\mathbf{R}^1\boldsymbol{\Gamma}(\Vcc_\cX,\Dcc_\cX)/\mathfrak{m}_{x_0} \mathbf{R}^1\boldsymbol{\Gamma}(\Vcc_\cX,\Dcc_\cX) \hookrightarrow \mathbf{R}^1\boldsymbol{\Gamma}(\Vcc_{x_0},\Dcc_{x_0})$$ 
induced from the base change property of Selmer complexes is an isomorphism.
The image $P_{x_0}$ of $P$ is a generator of 
$$\mathbf{R}^1\boldsymbol{\Gamma}(\Vcc_{x_0},\Dcc_{x_0})\xrightarrow[\sim]{\eqref{eqn_prop_degenerate_local_conditions}}H^1(\Qp, \Dcc_{x_0} ).$$ This implies that $
\res_p(P)_{x_0}=\res_p(P_{x_0})\neq 0$, which concludes the proof that ${\rm LOG}_{\eta}(P)(x_0)\neq 0$ (hence ${\rm LOG}_{\eta}(P)\in \cO_\cX^\times$ on shrinking $\cX$ as necessary). 

Reducing  \eqref{eqn_defn_big_merom_log} modulo $X^2$  and using Corollary~\ref{cor: transition for exponentials}, we infer that
\begin{multline}
\nonumber
X \res_p(\Delta(P_{x_0})) \\
\equiv X (-1)^{k/2}{\rm LOG}_{\eta}(P)(x_0) \, \Gamma\left(\frac{k}{2}\right) \left(1-\frac{p^{\frac{k}{2}-1}}{a_p(f_{x_0})}\right)\left(1-\frac{a_p(f_{x_0})}{p^{\frac{k}{2}}}\right)^{-1}\exp_{\Vcc_{x_0}^{(\beta)}}({}^c\kappa_0(\eta_{x_0}))
\pmod{X^2}.
\end{multline}
Recalling that
\[
\left [{}^c\kappa_0(\eta_{x_0}), {}^c\omega_{x_0}'\right ] =C_{\mathrm{K}},
\]
we conclude with the proof of \eqref{eq: LOG at 0} and our proposition.

\end{proof}

\subsection{$\cO_\cX$-adic leading term formula}
\label{subsec_535_2022_12_16_16_37}
In this subsection, we work in the setting of \S\ref{subsec: the rank one case}. In particular, the conditions
$r_{\mathrm{an}}(k/2)=1$, \eqref{item_PR1} 
and \eqref{item_C4} are  enforced. In this case, the $E$-vector space $H^1_{\rm f}(\Vcc_{x_0})$ is one-dimensional, whereas $\mathbf{R}^1\boldsymbol{\Gamma}(\Vcc_{k}  ,\Dcc_{k})$ is  two-dimensional. We recall from Remark~\ref{remark: about PR1 condition} that one may drop the assumption \eqref{item_PR1} whenever \eqref{item_CM} holds true (in particular, when $k=2$).

\subsubsection{} 
\label{subsubsec_5351_2022_12_19_1351}

Let $P'= X\cdot  {}^c \bz_0(\cX,\xi)$ be the generator of $\mathbf{R}^1\boldsymbol{\Gamma}(\Vcc_\cX',\Dcc_\cX')$ constructed in Theorem~\ref{lemma_parity_rk_rk_prime_implies_sign}. We also fix a generator 
$P$ of  $\mathbf{R}^1\boldsymbol{\Gamma}(\Vcc_\cX,\Dcc_\cX)$ and set
\begin{equation}
\label{eqn_2023_07_10_1022}
     R_{\cX}:=\langle\,P',P\,\rangle_\cX \in \cO_\cX\,.
\end{equation}
Recall the two-variable $p$-adic $L$-function $L_{\mathrm{K},\eta}(\cX,\xi)$
constructed in Section~\ref{subsec_defn_critical_padic_L_eigencurve}.

We have the following leading term formula for our $p$-adic $L$-function at classical points $x\neq x_0$:
\begin{proposition}
\label{prop_regulator_away_from_x} Assume that $r_{\mathrm{an}}(k/2)=1$, and that the conditions  \eqref{item_PR1} and \eqref{item_C4} hold.
Suppose that $x \in \cX^{\rm cl}(E)\setminus\{x_0\}$ is not in the support of $\mathbf{R}^2\boldsymbol{\Gamma}(\Vcc_\cX',\Dcc_\cX')_{\rm tor}$. 
Then, 
\begin{multline}
X(x)\log_{{}^c\omega_x'}(P_x)\, \frac{d}{ds}L_{\mathrm{K},\eta}^{\pm}(\cX,\xi; x,\chi^{s} )\big{\vert}_{s=\frac{w(x)}{2}}
\\
={(-1)^{\frac{w(x)}{2}-1} \cdot 2}\cdot \Gamma\left(\frac{w(x)}{2}\right) \left(1-\frac{p^{\frac{w(x)}{2}-1}}{a_p(f_x)}\right)\left(1-a_p(f_x)p^{-\frac{w(x)}{2}}\right)^{-1}
R_{\cX}(y)\,,
\end{multline}
where $L_{\mathrm{K},\eta}^{\pm}(\cX,\xi;x):=L_{\mathrm{K},\eta}^{\pm}(\cX,\xi)(x)$ is the specialization of the two-variable $p$-adic $L$-function $L_{\mathrm{K},\eta}$ to $x$.
\end{proposition}
\begin{proof}
It follows from the definitions and the $\cO_\cW$-linearity of $\langle\,,\,\rangle_{\Dcc',\Dcc}$ that 
$$R_{\cX}(x)=\langle P'_x, X(x)P_x\rangle_{\Dcc_x,\Dcc_x'}\,.$$
Plugging in $X(x){}^{c}\bz_0 (x,\xi)$ in place of $P'_x$ and using the Rubin-style formula (Theorem~\ref{thm_RSformula_cyclo_height}), we have
\begin{align}
\begin{aligned}
\label{eqn_2022_08_25_1257}
R_{\cX}(x)=X(x)\left\langle\bz_0 (x,\xi), X(x)P_x\right\rangle_{\Dcc_x,\Dcc_x'}&=-X(x)\,\left (\mathfrak{d}_\cyc {}^{c}\bz (y,\xi), X(y)\res_p(P_x) \right )\\
&=-X(x)\log_{{}^c\omega_x'}(P_x)\,\left ( \mathfrak{d}_\cyc {}^{c}\bz (x,\xi), X(x)\exp({}^c\eta_x) \right )\,.
\end{aligned}
\end{align}
Here, the cyclotomic derivative $\mathfrak{d}_\cyc$ of a class is given in the statement of Theorem~\ref{thm_RSformula_cyclo_height} and in the second equality, we have used the definition of $\log_{{}^c\omega_x'}(P_x)$.

The definition of the $p$-adic $L$-function $L_{\mathrm{K},\eta}^{\pm}(\cX,\xi)$ in terms of the large exponential map (cf. \S\ref{subsubsec_2235_1204}) alongside the interpolation properties of the large exponential map (cf. \S\ref{subsubsec_1151_2028_08_25_1200}),  and the fact that ${\rm sp}_x(\bbeta)\otimes d_{\frac{w(x)}{2}}$ maps to ${2} X(x){}^c\eta_x$ under the natural isomorphism $\Dcc_{w(x)}[x]\xrightarrow[\pi_x]{\sim}\Dcc_x$ combined together yields
\begin{multline}
\frac{d}{ds}L_{\mathrm{K},\eta}(\cX,\xi; x,\chi^{s} )\big{\vert}_{s=\frac{w(x)}{2}}\\
={(-1)^{\frac{w(x)}{2}} \cdot 2}\cdot \Gamma\left(\frac{w(x)}{2}\right) \left(1-\frac{p^{\frac{w(x)}{2}-1}}{a_p(f_x)}\right)\left(1-a_p(f_x)p^{-\frac{w(x)}{2}}\right)^{-1}\left (\mathfrak{d}_\cyc {}^{c}\bz (x,\xi), X(x)\exp({}^c\eta_x) \right )\,.
\end{multline}
This identity combined with \eqref{eqn_2022_08_25_1257} concludes the proof.
\end{proof}


\begin{theorem}
\label{thm_A_adic_regulator_formula_eigencurve}
Assume that $r_{\mathrm{an}}(k/2)=1$ and the conditions \eqref{item_PR1} and \eqref{item_C4} hold.
\item[i)] ($\cO_{\cX}$-adic leading term formula) For sufficiently small $\cX$, we have the following identity in $\cO_{\cX}$:
$$X\cdot {\rm LOG}_{\eta}(P) \cdot \frac{\partial}{\partial s}L_{\mathrm{K},\eta}^{\pm}(\cX,\xi )\ \Big{\vert}_{
{s=\frac{w(x)}{2}}
}=-  R_{\cX}\,.$$

\item[ii)] The identity
\begin{align}
\label{eqn_main_thm_dsdX_eigencurve}
\begin{aligned}
\log_{\omega_{{x_0}}'}&\,(P_{x_0})\,\frac{\partial}{\partial s} \widetilde{L}_{\mathrm{K},\alpha}^{\pm}(f,\xi)\Big{\vert}_{s=\frac{k}{2}} \\
&= {(-1)^{\frac{k}{2}-1}  \,C_{\mathrm{K}}}\,{\Gamma\left(\frac{k}{2}\right) \left(1-\frac{p^{\frac{k}{2}-1}}{\alpha}\right)\left(1-\alpha p^{-\frac{k}{2}}\right)^{-1}}\left(\frac{1}{2}\left\langle P_{k}',P_{k}\right\rangle_{\Dcc_{k}',\Dcc_{k}} +\frac{X}{2}\frac{d^2 R_{\cX}}{dX^2}\big\vert_{X=0}\right)
\end{aligned}
\end{align}
holds in $\cO_\cX/X^2\cO_\cX$, where 
$$\mathbf{R}^1\boldsymbol{\Gamma}(\Vcc_{k}',\Dcc_{k}')\otimes \mathbf{R}^1\boldsymbol{\Gamma}(\Vcc_{k},\Dcc_{k})\xrightarrow{\langle \,,\,\rangle_{\Dcc_{k}',\Dcc_{k}}} E$$ 
is the specialization of the $\cO_\cW$-valued height pairing $\langle \,,\,\rangle_{\Dcc',\Dcc}$ to weight $k$.
\end{theorem}

\begin{proof}
\item[i)] This is an immediate consequence of Proposition~\ref{prop_regulator_away_from_x} and Proposition~\ref{prop_properties_LOG}(i), which together tell us that the difference of the two sides agree at every $x \in \cX^{\rm cl}(E)\setminus\{x_0\}$.


\item[ii)] The asserted identity follows from Part (i)
once we carry out the following tasks: \begin{itemize}
\item[a)]Prove that the expression $\dfrac{X}{2}\left\langle P_{k}',P_{k}\right\rangle_{\Dcc_{k}',\Dcc_{k}}$ is the linearization of the regulator $R_{\cX}$ about $X=0$:
\be
\nonumber 
\label{eqn_reg_mod_X2}
R_{\cX}(k)=\frac{X}{2}\left\langle P_{k}',P_{k}\right\rangle_{\Dcc_{k}',\Dcc_{k}} \mod X^2\,.
\ee 
\item[b)] Compare the coefficient of $X^2$ in the Taylor expansions of both sides in \eqref{eqn_main_thm_dsdX_eigencurve} (which manifestly agree in view of Theorem~\ref{thm_A_adic_regulator_formula_eigencurve}).
\end{itemize}

Unravelling the definitions, we see that
\begin{align*}
R_{\cX}(k)&=\frac{X}{2}\left\langle P_{k}',P_{k}\right\rangle_{\Dcc_{k}',\Dcc_{k}}+\frac{1}{2}\left\langle [X]P_{k}',P_{k}\right\rangle_{\Dcc_{k}',\Dcc_{k}}\\
&=\frac{X}{2}\left\langle P_{k}',P_{k}\right\rangle_{\Dcc_{k}',\Dcc_k}+\frac{1}{2}\left\langle {X}P_{x_0}',P_{x_0}\right\rangle_{\Dcc_{x_0}',\Dcc_{x_0}}=\frac{X}{2}\left\langle P_{k}',P_{k}\right\rangle_{\Dcc_{k}',\Dcc_{k}}
\end{align*}
where the second identity follows from the discussion in the paragraph following Theorem~\ref{thm_A_adic_regulator_formula_eigencurve} (where we explain that $[X] P_{k}'=XP_{x_0}'$) and the property \eqref{item_Adj}, whereas the final equality (i.e. the vanishing of the pairing $\left\langle {X}P_{x_0}',P_{x_0}\right\rangle_{\Dcc_{x_0}',\Dcc_{x_0}}$) has been observed as part of the proof of Proposition~\ref{reg_at_x_explicit}. This concludes our proof. 
\end{proof}

\subsubsection{The na\"ive regulator}
Let $P,P'$ be as in \S\ref{subsubsec_5351_2022_12_19_1351} and recall that $P_k\in \mathbf{R}^1\boldsymbol{\Gamma}(\Vcc_{k},\Dcc_{k})$ and $P_k'\in \mathbf{R}^1\boldsymbol{\Gamma}(\Vcc_{k}',\Dcc_{k}')$ denote their respective images. Recall that $P_k$  (resp. $P_k'$) generates the Selmer group $\mathbf{R}^1\boldsymbol{\Gamma}(\Vcc_{k},\Dcc_{k})$ (resp. $\mathbf{R}^1\boldsymbol{\Gamma}(\Vcc_{k}',\Dcc_{k}')$) as $\widetilde{E}$-modules. Therefore, $\{P_k,[X] P_k\}$ and $\{P_{k}',[X] P_k'\}$ are $E$-bases of the Selmer groups $ \mathbf{R}^1\boldsymbol{\Gamma}(\Vcc_{k},\Dcc_{k})$ and  $\mathbf{R}^1\boldsymbol{\Gamma}(\Vcc_{k}',\Dcc_{k}')$, respectively.

Attached to this data, we may define the \emph{na\"ive} regulator in terms of the $p$-adic height pairing $\left\langle \,,\,\right\rangle_{\Dcc_{k}',\Dcc_{k}}$ evaluated on these $E$-bases:
$$R_k^{\hbox{\tiny na{\"i}ve}}:=\det\begin{pmatrix} \left\langle [X] P_k',[X] P_k\right\rangle_{\Dcc_{k}',\Dcc_k} & \left\langle  [X] P_k' , P_k\right\rangle_{\Dcc_{k}',\Dcc_k}\\\\
\left\langle P_k',  [X] P_{k}\right\rangle_{\Dcc_{k}',\Dcc_k} &\left \langle P_k' ,P_k\right\rangle_{\Dcc_{k}',\Dcc_k}
\end{pmatrix}\,.
$$
One might expect this regulator to show up in our leading term formulae. We will explain below why this cannot be the case, by showing that $R_k^{\hbox{\tiny na{\"i}ve}}$ vanishes.

Note that we have, 
$$ \left\langle [X] P_k',[X] P_k\right\rangle_{\Dcc_{k}',\Dcc_k}= \left\langle [X^2] P_k',P_k\right\rangle_{\Dcc_{k}',\Dcc_k}=0$$
thanks to \eqref{item_Adj}. Moreover, as we have explained in the proof of Theorem~\ref{thm_A_adic_regulator_formula_eigencurve}(ii), we have 
$$\left\langle P_k',  [X] P_{k}\right\rangle_{\Dcc_{k}',\Dcc_k} = \left\langle P_{x_0}',  X P_{x_0}\right\rangle_{\Dcc_{x_0}',\Dcc_{x_0}} \,\,\stackrel{{\rm Cor. } \ref{cor_prop_factor_degenerate_height_pairing}}{=}\,\, 0$$
and similarly, $\left\langle  [X] P_k',  P_{k}\right\rangle_{\Dcc_{k}',\Dcc_k}=0$ as well. This shows that the na\"ive regulator $R_k^{\hbox{\tiny na{\"i}ve}}$ on the thick Selmer groups vanishes.

\subsubsection{} 
\label{subsubsec_3125_18_11}
We will conclude \S\ref{subsec_535_2022_12_16_16_37} with an observation concerning the torsion submodule of $\mathbf{R}^2\boldsymbol{\Gamma}(\Vcc_\cX,\Dcc_\cX)$. 

The base change properties of Selmer complexes give rise to the exact sequences
$$0\lra \mathbf{R}^1\boldsymbol{\Gamma}(\Vcc_\cX,\Dcc_\cX)/X\mathbf{R}^1\boldsymbol{\Gamma}(\Vcc_\cX,\Dcc_\cX) \lra \mathbf{R}^1\boldsymbol{\Gamma}(\Vcc_{x_0},\Dcc_{x_0})\lra \mathbf{R}^2\boldsymbol{\Gamma}(\Vcc_\cX,\Dcc_\cX)[X]\lra 0$$
$$0\lra \mathbf{R}^1\boldsymbol{\Gamma}(\Vcc_\cX,\Dcc_\cX)/X^2\mathbf{R}^1\boldsymbol{\Gamma}(\Vcc_\cX,\Dcc_\cX) \lra \mathbf{R}^1\boldsymbol{\Gamma}(\Vcc_{k},\Dcc_{k})\lra \mathbf{R}^2\boldsymbol{\Gamma}(\Vcc_\cX,\Dcc_\cX)[X^2]\lra 0\,.$$
This in turn yields
\begin{equation}
\label{eqn_rankof_H2_at_x}
\begin{aligned}
\dim_E \mathbf{R}^2\boldsymbol{\Gamma}(\Vcc_\cX,\Dcc_\cX)[X]= \dim_E\mathbf{R}^1\boldsymbol{\Gamma}(\Vcc_{x_0},\Dcc_{x_0})- {\rm rank}_{\cO_\cX}\,  \mathbf{R}^1\boldsymbol{\Gamma}(\Vcc_\cX,\Dcc_\cX)\\
\dim_E \mathbf{R}^2\boldsymbol{\Gamma}(\Vcc_\cX,\Dcc_\cX)[X^2]=\dim_E \mathbf{R}^1\boldsymbol{\Gamma}(\Vcc_{k},\Dcc_{k})-2\cdot {\rm rank}_{\cO_\cX}\,  \mathbf{R}^1\boldsymbol{\Gamma}(\Vcc_\cX,\Dcc_\cX)\,.
\end{aligned}
\end{equation}

Our considerations above leads to the following curious disparity concerning the behaviour of the torsion submodule of $\mathbf{R}^2\boldsymbol{\Gamma}(\Vcc_\cX,\Dcc_\cX)$:

\begin{proposition}
Suppose that the conditions \eqref{item_PR1} and  \eqref{item_C4} hold. Then for sufficiently small  $\cX$ we have
\begin{align}
\label{eqn_rankof_H2_at_x_k_ord_leq_1}
\begin{aligned}
\dim_E \mathbf{R}^2\boldsymbol{\Gamma}(\Vcc_\cX,\Dcc_\cX)[X]\,&=\dim_E \mathbf{R}^2\boldsymbol{\Gamma}(\Vcc_\cX,\Dcc_\cX)[X^2] \\
\,&= \begin{cases}
1& \hbox{ if }\quad  L(f,\frac{k}{2})\neq 0,\\
0&\hbox{ if } \quad \res_p \left({}^{c}\bz_0 (x_0,\xi)\right)\in H^1_{\rm f}(\QQ_p,\Vcc_{f}').\\
\end{cases}
\end{aligned}
\end{align}
\end{proposition}

\begin{proof}
The theory of Euler systems shows that 
\be
\label{eqn_eqn_PR_condition_x_0_implies_fine_Selmer_vanishes}
\eqref{item_PR1} \,\implies\, H^1_{0}(\Vcc_{f})=0\,,
\ee
cf. Proposition~\ref{prop_2022_04_26_13_00}.  The proof of our proposition follows on combining Proposition~\ref{prop_4_14_2022_08_19_1401}(ii), Theorem~\ref{lemma_parity_rk_rk_prime_implies_sign}(i), \eqref{eqn_rankof_H2_at_x}, and \eqref{eqn_eqn_PR_condition_x_0_implies_fine_Selmer_vanishes}.
\end{proof}

Based on these observations, we propose the following:
\begin{conj}
\label{conj_H2_X_power_torsion}
For sufficiently small $\cX$,
\begin{equation}
\label{eqn_conj_H2_tor}
\mathbf{R}^2\boldsymbol{\Gamma}(\Vcc_\cX,\Dcc_\cX)_{\rm tor}\simeq \begin{cases}
\cO_\cX/X\cO_\cX\simeq E&\qquad \hbox{ if } \quad L(f,\frac{k}{2})\neq 0,\\
\{0\}& \qquad\hbox{ if }\quad \ord_{s=\frac{k}{2}} L(f,s)= 1\,.\\
\end{cases}
\end{equation}
\end{conj}

We remark that Conjecture~\ref{conj_H2_X_power_torsion} indeed holds whenever the implication (extension of Perrin-Riou's conjecture to higher weight forms)
$$r_{\rm an}(k/2)= 1 \quad {}\implies{} \quad \, \eqref{item_PR1}$$
holds true.

\subsection{Central critical twists and a conjecture of Greenberg}
\label{subsec_greenbergs_conjecture_for_self_dual_families_1}
In the previous section, we have defined the $\cO_\cX$-adic regulator $R_\cX$ under the assumption that $r_{\mathrm{an}}(k/2)=1$, as well as the hypotheses \eqref{item_PR1} and \eqref{item_C4}, which were used (cf. Theorem~\ref{lemma_parity_rk_rk_prime_implies_sign}) to ensure that both $\cO_\cX$-modules $\mathbf{R}^1\boldsymbol{\Gamma}(\Vcc_{\mathcal X},\Dcc_{\cX})$ and $\mathbf{R}^1\boldsymbol{\Gamma}(\Vcc_{\mathcal X}',\Dcc_{\cX}')$ are free of rank one. An extension of a conjecture of Greenberg (cf. Conjecture~\ref{conj_greenberg}) below predicts that this is still the case under the weaker assumption that the nebentype of the family over $\cX$ is trivial and the global root number of some (every) classical member of the family equals $-1$.

Our goal in \S\ref{subsec_greenbergs_conjecture_for_self_dual_families_1} and \S\ref{subsec_padic_regulators_families} is to establish criteria for the validity of Conjecture~\ref{conj_greenberg}, so that we can define the $\cO_\cX$-adic regulator $\cR_\cX$ under hypotheses that are less stringent then those in the previous section.

\subsubsection{} 
\label{subsubsec_greenbergs_conjecture_for_self_dual_families}
Let $W(f)$ denote the global root number of $f$ at the central critical point. A conjecture of Greenberg (cf. \cite{Greenberg_1994_families}, see also \cite{TrevorArnold_Greenberg_Conj}, Conjectures 1.3 and 1.4) extended to positive-slope families\footnote{Greenberg's original conjecture concerns slope-zero (Hida) families.} predicts that
\begin{equation}
\label{eqn_2023_07_06_1137}
    \hbox{for all but finitely many $x\in \cX^{\rm cl}(E)$, we have } \,\,{\rm ord}_{s=\frac{w(x)}{2}} L(f_x,s)=\begin{cases} 1 &\hbox{ if } W(f)=-1\\ 
0& \hbox{ otherwise\,. } \end{cases}
\end{equation}
\begin{remark}
A prototypical result towards \eqref{eqn_2023_07_06_1137} is due to Rohrlich~\cite{rohrlich84}. To be more precise, let us consider the unique CM family $\bf{f}$ that admits the slope-zero $p$-stabilization of the newform $f_E$ associated with a CM elliptic curve $E$ as a weight-$2$ specialization. Rohrlich's generic non-vanishing results in op. cit. (both when $W(f_E)=1$ and $W(f_E)=-1$) imply a variant of \eqref{eqn_2023_07_06_1137} for the CM family $\f$.     
\end{remark}

In view of the control theorems for the Selmer groups $\mathbf{R}^1\boldsymbol{\Gamma}(\Vcc_{\mathcal X},\Dcc_{\cX})$ and $\mathbf{R}^1\boldsymbol{\Gamma}(\Vcc_{\mathcal X}',\Dcc_{\cX}')$ (cf. the exact sequence \eqref{eqn_2023_07_06_1136} below) together with Bloch--Kato conjectures concerning the $L$-values in \eqref{eqn_2023_07_06_1137}, we propose the following algebraic variant of this conjecture:}
\begin{conj}
\label{conj_greenberg} Assuming that $\cX$ is sufficiently small, we have
$${\rm rank}_{\cO_\cX}\,  \mathbf{R}^1\boldsymbol{\Gamma}(\Vcc_{\mathcal X}',\Dcc_{\cX}')={\rm rank}_{\cO_\cX}\,  \mathbf{R}^1\boldsymbol{\Gamma}(\Vcc_{\mathcal X},\Dcc_{\cX})=\begin{cases} 1 &\hbox{ if } W(f)=-1\\ 
0& \hbox{ otherwise\,. } \end{cases}$$
\end{conj}


In the next lemma, we give a sufficient condition for the validity of Conjecture~\ref{conj_greenberg} in the scenario when $W(f)=-1$. See also Remark~\ref{remark_arnold_condition_on_greenberg} and \S\ref{subsec_padic_regulators_families} below for further criteria.
\begin{lemma}
\label{lem_Greenberg_conj_sufficient_condition}
Let us denote by ${}^{c}\bz_0 (\cX,\xi)\in H^1_S(\Vcc_{\mathcal X}')$ the image of the big Beilinson--Kato element 
$$\bz (\cX,\xi)\in H^1(V_{\mathcal X}'\,\widehat\otimes\,\LL^\iota(1))
$$ 
over $\cX$ (cf. \S\ref{subsubsec_2221_18_11_2021}) under the map induced from the natural projection $V_{\mathcal X}'\,\widehat\otimes\,\LL^\iota(1)\to \Vcc_{\mathcal X}'$. Suppose that the global root number of some (every) classical member of the family $\mathcal X$ equals to $-1$ and  that ${}^{c}\bz_0 (\cX,\xi)\neq 0$. Then Conjecture~\ref{conj_greenberg} is true. 
\end{lemma}

\begin{proof}
This is very similar to the proof of Theorem~\ref{lemma_parity_rk_rk_prime_implies_sign}(i). By the assumptions of this lemma, we can find infinitely many $x_0\neq x\in \mathcal X^{\rm cl}(E)$ such that the specialization ${}^{c}\bz_0 (x,\xi)\in H^1_S(\Vcc_x')$ of the element ${}^{c}\bz_0 (\cX,\xi)$ is nonzero. This, and the fact that $L(f_x,\frac{w(x)}{2})=0$ together with the standard Euler system argument (along with the fact that $ \mathbf{R}^1\boldsymbol{\Gamma}(\Vcc_x',\Dcc'_x)$ coincides with the Bloch--Kato Selmer group for every classical point $x\neq x_0$) show that
$$\dim_E  \mathbf{R}^1\boldsymbol{\Gamma}(\Vcc_x',\Dcc_x')=1\,.$$
Since $\Vcc_x'\simeq \Vcc_x$ as $G_{\QQ}$-representations and $\Dcc_x'\simeq \Dcc_x$, this also shows that $\dim_E \mathbf{R}^1\boldsymbol{\Gamma}(\Vcc_x,\Dcc_x)=1$.

The control theorems for Selmer complexes tell us that the sequence 
\begin{equation}
\label{eqn_2023_07_06_1136}
    0\lra  \mathbf{R}^1\boldsymbol{\Gamma}(\Vcc_\cX',\Dcc_\cX')/\mathfrak{m}_x  \mathbf{R}^1\boldsymbol{\Gamma}(\Vcc_\cX',\Dcc_\cX')\lra \mathbf{R}^1\boldsymbol{\Gamma}(\Vcc_x', \Dcc_x')\lra   \mathbf{R}^2\boldsymbol{\Gamma}(\Vcc_\cX',\Dcc_\cX')[\m_x]\lra 0
\end{equation}
is exact. Since $ \mathbf{R}^2\boldsymbol{\Gamma}(\Vcc_\cX',\Dcc_\cX')$ is a finitely generated $\cO_\cX$-module,  $\mathbf{R}^2\boldsymbol{\Gamma}(\Vcc_\cX',\Dcc_\cX')[\m_x]=0$ for all but finitely many $x$. This, together with the short exact sequence above, shows that $ \mathbf{R}^1\boldsymbol{\Gamma}(\Vcc_\cX',\Dcc_\cX')/\mathfrak{m}_x  \mathbf{R}^1\boldsymbol{\Gamma}(\Vcc_\cX',\Dcc_\cX')$ is a one-dimensional $E$-vector space for infinitely many $x\in \mathcal X^{\rm cl}(E)$, which in turn implies (recalling that $\cO_\cX$ is a PID and $  \mathbf{R}^1\boldsymbol{\Gamma}(\Vcc_\cX',\Dcc_\cX')$ is a torsion-free $\cO_\cX$-module by Lemma~\ref{lemma_2022_08_24_1452}) that the $\cO_\cX$-module $  \mathbf{R}^1\boldsymbol{\Gamma}(\Vcc_\cX',\Dcc_\cX')$ is free of rank one.
\end{proof}

\begin{remark}
\label{remark_arnold_condition_on_greenberg}
Arguing as in the proof\footnote{The argument in op. cit. primarily dwells on a Rubin-style formula for $p$-adic heights. This, in our context, has been proved in \S\ref{subsec_Rubin_style_formulae}.} of \cite[Theorem 4.7]{TrevorArnold_Greenberg_Conj}, one may replace the running assumptions of Lemma~\ref{lem_Greenberg_conj_sufficient_condition} with the following condition: For some $x\in \mathcal X^{\rm cl}$, we have
\item[a)] $\ord_{s=\frac{w(x)}{2}}\, {L}_{\mathrm{K},\eta}^{\pm}(x,\xi) =1$,
\item[b)] Nekov\'a\v{r}'s $p$-adic height pairing
$$\langle\,,\,\rangle_{\rm Nek}\,:\,\,H^1_{\rm f}(\Vcc_{f_x}')\otimes H^1_{\rm f}(\Vcc_{f_x})\lra E$$
associated to the splitting module $\Dc (V_{f_x})^{\varphi=\alpha(x)}\subset \Dc (V_{f_x})$ is non-degenerate. 
This set of conditions in turn is equivalent to the hypotheses that 
\begin{itemize}
\item[\mylabel{item_S1}{{\bf S1}}$)$] $L_{\mathrm{K},\eta}^{\pm}(\cX,\xi )\in  (\gamma-\bbchi^{\frac{1}{2}}(\gamma))\cdot {\mathscr{H}}_\cX(\Gamma) \setminus (\gamma-\bbchi^{\frac{1}{2}}(\gamma))^2\cdot {\mathscr{H}}_\cX(\Gamma)$, and
\item[\mylabel{item_non_deg}{{\bf ht}}$)$]  the height pairing
$$\langle\,,\,\rangle_{\Dcc',\Dcc}\,:\,\,\mathbf{R}^1\boldsymbol{\Gamma}(\Vcc_\cX',\Dcc'_\cX)\otimes   \mathbf{R}^1\boldsymbol{\Gamma}(\Vcc_\cX,\Dcc_\cX)\lra \cO_{\cW}$$
given as in \eqref{eqn_defn_height_cyclo} is non-degenerate.
\end{itemize}
\end{remark}


\subsection{}
\label{subsec_padic_regulators_families} 
We conclude our paper with further criteria for the validity of Conjecture~\ref{conj_greenberg}. We will only study the portion of the conjecture that concerns the Galois representation $V_\cX'$, as the complementary statements that concern $V_\cX$ can be deduced from this case.

\subsubsection{}
Henceforth, we assume that $\cX$ is small enough to ensure that the conclusion of Lemma~\ref{lemma_2022_08_24_1452}(ii) is valid, so that $H^1(\QQ_\ell,\Vcc_\cX')=\{0\}$ for each $\ell\in S\setminus\{p\}$.

 
\subsubsection{}  As before, let us also denote by ${}^{c}\bz_0 (x,\xi)\in  H^1(\QQ,\Vcc_{x}')$ the specialization of ${}^{c}\bz_0 (\cX,\xi)$ to $x\in \cX^{\rm cl}(E)$. 
The following conditions will be relevant to our study. One expects them to hold true whenever $W(f)=-1$. 
We consider the following variant of the condition \eqref{item_S1}:
\begin{itemize}
\item[\mylabel{item_S2}{{\bf S2}})] We have
\begin{itemize}
\item[\mylabel{item_S2_1}{{\bf S2.1}})] $\res_p\left({}^{c}\bz_0 (x,\xi) \right)\in  H^1_{\rm f}(\QQ_p,\Vcc'_x)$ for infinitely many $x\in \cX^{\rm cl}(E)$,
\item[\mylabel{item_S2_2}{{\bf S2.2}})] $\res_p\left({}^{c}\bz_0 (y_0,\xi)\right)\neq 0$ for some $y_0 \in \cX^{\rm cl}(E)$ verifying \eqref{item_S2_1}. 
\end{itemize}
\end{itemize}

\begin{remark}

The condition \eqref{item_S2} implies \eqref{item_S1} provided that there exists $y_0\in \cX^{\rm cl}(E)$ verifying \eqref{item_S2_2} with the following additional properties:
\begin{itemize}
\item[\mylabel{item_IMC}{${\bf IMC}_{y_0}$}$)$] Iwasawa main conjecture for $y_0$ holds true: $L_{\mathrm{K},\eta}^{\pm}( y_0,\xi)\cdot \mathscr{H}(\Gamma)={\rm char}_{\mathscr{H}(\Gamma)}\,\mathbf{R}^2\boldsymbol{\Gamma}(\Vcc_{y_0}',\Dcc_{y_0}')$.
\item[\mylabel{item_non_deg_y_0}{${\bf ht}_{y_0}$}$)$] Nekov\'a\v{r}'s $p$-adic height pairing  
$$H^1_{\rm f}(\Vcc'_{y_0})\,\otimes\, H^1_{\rm f}(\Vcc_{y_0})\lra E$$ 
associated to the splitting module $\Dc(V_{y_0})^{\varphi=\alpha(y_0)}\subset \Dc(V_{y_0})$ is non-degenerate.
\end{itemize}
One expects that the conditions \eqref{item_IMC} and \eqref{item_non_deg_y_0} are always valid (in which case, the Greenberg Conjecture~\ref{conj_greenberg} also holds). 

\end{remark}

The following implications summarize the relation between various conditions and statements:
$$
\xymatrix{ 
 \eqref{item_S2} \quad + \quad \eqref{item_IMC} \quad+ \quad\eqref{item_non_deg_y_0}  \quad\ar@{=>}[rr]&& \quad   \eqref{item_S1} + \eqref{item_non_deg}\quad \ar@{=>}[rrr]^-{\rm Remark~\ref{remark_arnold_condition_on_greenberg}}&&& \quad {\rm Conjecture}~\ref{conj_greenberg}\,,}
 $$
 $$
\xymatrix{  \eqref{item_S2} \quad \ar@{=>}[rrr]^-{{\rm Lemma \,}~\ref{lem_Greenberg_conj_sufficient_condition}}&&&\quad  {\rm Conjecture}~\ref{conj_greenberg}\,.}
$$

\subsubsection{}  As our final task, we will explain that 
\be\label{eqn_2022_12_19_11_21}
\xymatrix{ 
r_{\rm an}(k/2)=1 \quad + \quad \eqref{item_PR1}\quad+ \quad \eqref{item_C4}\quad \ar@{=>}[r]& \quad \eqref{item_S2}\,,
}
\ee
namely, that the key running assumptions of \S\ref{subsec: the rank one case}--\S\ref{subsec_535_2022_12_16_16_37} imply the validity of \eqref{item_S2}. We begin with a reformulation of the condition \eqref{item_S2_1}. 

In view of Lemma~\ref{lemma_tildeD_H0_is_zero}(iii), \emph{we will assume without loss of generality that the support of the $\cO_\cX$-module $H^1(\QQ_p,{}^c\mathbb{D}'_\cX)_{\rm tor}$} \emph{is $\{x_0\}$.} Note that \eqref{item_C4} is already required in  Lemma~\ref{lemma_tildeD_H0_is_zero}(iii).

\begin{lemma}\label{lemma_S21_family}
Suppose that \eqref{item_C4} holds. Then the condition \eqref{item_S2_1} is equivalent to the requirement that
\be
\label{eqn_19_11_12_07}X\cdot {}^{c}\bz_0 (\cX,\xi)\in \ker\left(  H^1_S(\Vcc_\cX') \stackrel{r_p}{\lra} H^1(\QQ_p, {}^c\mathbb{D}'_\cX)\right).
\ee
\end{lemma}

\begin{proof}
It follows from the definition of ${L}_{\mathrm{K},\eta}^{\pm}(\cX,\xi)$ and its interpolation properties (cf. Eqn.~\eqref{eqn:comparision with Manin-Vishik}) that the condition  \eqref{item_S2_1} is equivalent to the requirement that
\be\label{eqn_19_11_12_06}
r_p\left({}^{c}\bz_0 (\cX,\xi)\right)\in \ker\left(H^1(\QQ_p,{}^c\mathbb{D}'_\cX)\xrightarrow{\left(\,\,\, \,\, \,\,\,,\, \EXP_\cX(\bbeta[\frac{k}{2}]) \right)_{\cX}} \cO_\cX\right),
\ee
where $r_p$ is the obvious map (cf. the proof of Theorem~\ref{lemma_parity_rk_rk_prime_implies_sign}(ii)), and
\begin{itemize}
\item $\EXP_\cX: \cO_{\cX}\otimes_{\cO_\cW}  \DCc(\bD_\cX) \xrightarrow{{\rm id}\,\otimes\,\EXP_{\Dcc}}  \cO_{\cX}\otimes_{\cO_\cW} H^1(\QQ_p,\Dcc_\cX)$ with $\EXP_{\Dcc}$ as in \ref{defn_EXP_twisted_1918}\,,
\item the pairing $\left(\,,\, \right)_{\cX}$ is given by 
$$H^1(\QQ_p,{}^c\mathbb{D}'_\cX)\otimes  (\cO_\cX\otimes_{\cO_\cW}\, H^1(\QQ_p,\Dcc_\cX))\xrightarrow{z'\otimes (a\otimes z)\mapsto a(z',z )_{\cW}} \cO_\cX\,.$$ 
\end{itemize}
The requirement \eqref{eqn_19_11_12_06} is in turn equivalent to the condition that $r_p\left({}^{c}\bz_0 (\cX,\xi)\right)\in H^1(\QQ_p,^{c}\mathbb{D}_\cX')_{\rm tor}$, which is the same as the condition \eqref{eqn_19_11_12_07}.
\end{proof}

We remark that we have verified as part of Theorem~\ref{lemma_parity_rk_rk_prime_implies_sign} the validity of \eqref{eqn_19_11_12_07}, assuming that $r_{\rm an}(k/2)=1$ as well as the condition \eqref{item_PR1}. In other words, it follows from the proof of Theorem~\ref{lemma_parity_rk_rk_prime_implies_sign} and Lemma~\ref{lemma_S21_family} that
$$\xymatrix{
r_{\rm an}(k/2)=1 \quad + \quad \eqref{item_PR1}\quad+ \quad \eqref{item_C4}\quad \ar@{=>}[r]& \quad \eqref{item_S2_1}\,.
}$$
In relation to the implication \eqref{eqn_2022_12_19_11_21}, it  therefore remains to explain that the property  \eqref{item_S2_2} holds. The condition \eqref{item_PR1} tells us that $\res_p(^c\bz_0 (x_0,\xi))\neq 0$. This shows that $\res_p(^c\bz_0 (\cX,\xi))\neq 0$, and in turn also that \eqref{item_S2_2} holds true.

\backmatter

\bibliographystyle{amsalpha}
\bibliography{references}

\end{document}